\documentclass[a4paper,11pt]{report}

\usepackage[utf8]{inputenc}
\usepackage{amsmath}
\usepackage{amssymb}
\usepackage{graphicx}
\usepackage{pgfplots,tikz}
\usetikzlibrary{arrows,shapes,trees,backgrounds}
\usepackage{hyperref}
\usepackage[notref,notcite]{}
\usepackage{float}
\usepackage{geometry}

\usepackage{caption}
\usepackage{subcaption}

\newenvironment{proof}{\noindent{\textsc{Proof.}}}
{$\hfill\Box$\vspace{0.1 cm}\\}

\newtheorem{theorem}{Theorem}[chapter]
\newtheorem{obs}{Observation}[chapter]
\newtheorem{mydef}{Definition}[chapter]
\newtheorem{remark}{Remark}[chapter]
\newtheorem{cor}{Corollary}[chapter]
\newtheorem{prop}{Proposition}[chapter]
\newtheorem{lemma}{Lemma}[chapter]
\newtheorem{exmp}{Example}[section]
\makeindex
\begin{document}
\thispagestyle{empty}

\vspace{3cm}

\begin{center}
\begin{LARGE}
Università degli Studi di Milano-Bicocca
\end{LARGE}
\end{center}

\begin{center}
\begin{large}
Scuola di Scienze\\
Corso di Laurea Magistrale in Matematica\\
\end{large}

\vspace{2cm}
\begin{figure}[h]
\centering
\includegraphics[scale=0.7]{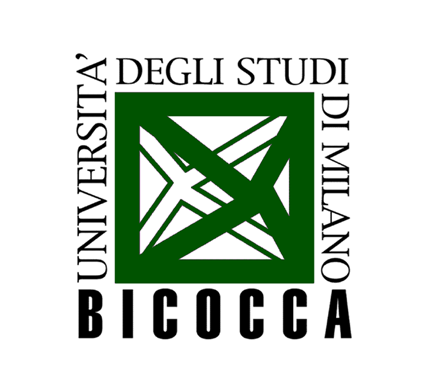} 
\end{figure}

\vspace{2cm}
{\Large \textbf{The Aw-Rascle-Zhang model with constraints}\\}
\vspace{0.7cm}

\begin{large}
Tesi di Laurea Magistrale in Matematica
\end{large}

\end{center}

\vspace{2cm}

\begin{flushleft}
Relatore:\\
\textbf{Dott. Mauro Garavello}\\
Correlatrice:\\
\textbf{Dott.ssa Paola Goatin}\\
\end{flushleft}

\begin{flushright}
Candidato:\\
\textbf{Stefano Villa\\
Matricola: 726465}
\end{flushright}

\vspace{2cm}

\begin{center}
Anno Accademico $2014/2015$
\end{center}

\vfill
\newpage
\pagenumbering{roman}
\tableofcontents
\newpage
\pagenumbering{arabic}
\chapter*{Introduction}
\addcontentsline{toc}{chapter}{Introduction}
The thesis deals with the Aw-Rascle-Zhang model for traffic, which was proposed by Aw and Rascle (see \cite{Aw-Rascle}) and, independently, by Zhang (see \cite{zhang}) in 2002. It is a hyperbolic system of two conservation laws and describes the traffic from a macroscopic point of view: it considers the evolution of macroscopic variables such as the average density or speed of the vehicles on a road.

During the '50s the first macroscopic model for traffic flows was introduced by Lighthill and Whitham and, independently, by Richards (LWR model; see \cite{LW,Richards}): the model is given by a single conservation law stating the conservation of the number of vehicles on a road:
\begin{equation}\label{LWR_equation}
\partial_t \rho+\partial_x (\rho\,v(\rho))=0,
\end{equation}
where $\rho$ is the density of the vehicles and $v(\rho)$ is a known decreasing function of $\rho$ which gives the velocity of the vehicles with respect to the density. A typical choice for this law is
$$v(\rho)=V\,\left(1-\dfrac{\rho}{\rho_{\max}}\right),$$
where $V$ and $\rho_{\max}$ are respectively the maximal speed and the maximal density of the vehicles allowed on the road.\\
The LWR model is effective to describe a situation of free traffic, i.e. a road where there is a small number of cars that can travel freely. When the density of the vehicles is over a certain threshold, many configurations are observed experimentally and the LWR model is not sufficient to capture them all; see \cite{colombo}.

To overcome this problem, second order models have been proposed. The Aw-Rascle-Zhang (ARZ) system is one of these models: it is a system of two partial differential equations in conservation form. Inspired by fluid dynamics models, the first equation states the conservation of the number of the vehicles on the road, while the second one imposes the conservation of a generalized momentum $z$:
$$\begin{cases}
\partial_t \rho + \partial_x (\rho\, v)=0,\\
\partial_t z + \partial_x (z\,v)=0.
\end{cases}
$$

Systems of conservation laws are widely applied to describe physical systems. It is well known that the Cauchy problem for a system of conservation laws with an integrable initial datum having sufficiently small bounded variation, admits a unique entropy-admissible solution. A possible tool to obtain the solution is the wave-front tracking method which is based on the Riemann problem, which is a Cauchy problem with a piecewise constant initial datum having only one jump discontinuity. A map which gives the solution to the Riemann problem is called ``Riemann solver''.

In 2011, Garavello and Goatin (see \cite{garavello_goatin}) introduced a model based on the ARZ system to describe the presence of a fixed constraint at some point of the road, corresponding for example to a toll gate or a traffic light. Two Riemann solvers have been proposed for the constrained Riemann problem. The solutions correspond to the real expected situation: the density is ``high'' before the constraint and ``low'' after it and, correspondingly, the velocity is reversed. The first solution conserves both density and momentum of the vehicles, while the second conserves only the density. These Riemann solvers are denoted by $\mathcal{RS}^q_1$ and $\mathcal{RS}^q_2$.\\

The present work is divided in three parts.

In the first part, we introduce the main concepts about systems of conservation laws and the ARZ model. In Chapter 1 we present the general theory of systems of conservation laws and we give the general solution to the Riemann problem, which consists in a combination of shocks, rarefaction waves and contact discontinuities. The solution is self-similar, i.e. constant on every line passing through the origin in the $(t,x)$ plane. In Chapter 2 we specialize the study of Chapter 1 to the ARZ model, showing the main properties of the system and defining its standard solution. We also define the invariant domains and give their characterization for the classical Riemann solver of the ARZ system.

In the second part we generalize the paper \cite{garavello_goatin} by Garavello and Goatin to the situation of a moving constraint corresponding for example to the presence of a large and slow vehicle on the road. The slow vehicle can be in turn influenced by the previous cars when a traffic jam is present. In Chapter 3 we give the mathematical model for this situation: we obtain a strongly coupled PDE-ODE system in which the main traffic is described with the ARZ system, while the trajectory of the slow vehicle is given by an ordinary differential equation. The presence of the constraint is traduced in a condition on the first component of the flux function. Two Riemann solvers corresponding to the ones proposed in \cite{garavello_goatin} are introduced, denoted by $\mathcal{RS}^\alpha_1$ and $\mathcal{RS}^\alpha_2$. We also characterize their invariant domains. In Chapter 3 we have applied numerical methods based on the Godunov's scheme to capture the solution given by $\mathcal{RS}^\alpha_1$ and $\mathcal{RS}^\alpha_2$ and to track the bus trajectory. The method for the first Riemann solver is based on its globally conservative character and captures exactly the solution to the Riemann problem for a general initial datum. For the second Riemann solver we have applied two numerical methods: the first is the same method used for $\mathcal{RS}^\alpha_1$ and the second method is based on a non-uniform mesh. Both methods succeed in computing the solution for special initial data. For general initial data the conservative method overestimates the density (and, correspondingly, underestimates the speed) after the constraint, while with the non-uniform mesh method the solution is exactly captured at least in the first cell after the constraint, but an oscillation appears.

The last part is discussed in Chapter 5 and contains the proof of the existence of the solution to a Cauchy problem for the Riemann solver $\mathcal{RS}^q_2$ in the case of an integrable initial datum with bounded variation, belonging to an invariant domain in which the characteristic waves of the first family have negative speed. The solution is obtained by applying the wave-front tracking method.
\chapter{Systems of Conservation Laws}
In this chapter we introduce the general theory of the systems of conservation laws and the general solution to the Riemann problem. We follow \cite{bressan}.\\
\section{Mathematical preliminaries}
Let $n$, $m$ and $p$ be positive integers and let us denote $\mathcal{C}^k$ the set of the $k$ times continuously differentiable functions.\\
First, we recall the implicit function theorem which gives a sufficient condition to traduce a relation having the form $F(x,y)=0$, in the graph of a function $y=\varphi(x)$, where $x$ and $y$ are vectors in $\mathbb{R}^n$ and $\mathbb{R}^m$.
\begin{theorem}[Implicit Function Theorem]
Let $U \subseteq \mathbb{R}^n$ and $V \subseteq \mathbb{R}^m$ be open sets and let $F:U\times V \to \mathbb{R}^m$ be a $\mathcal{C}^k$ function, with $k\geq 1$. If there exists a point $(\bar{x},\bar{y})\in U\times V$ such that $F(\bar{x},\bar{y})=0$ and the Jacobian matrix $D_y F(\bar{x},\bar{y})$ is invertible, then there exist a neighbourhood $\mathcal{N}\subseteq U$ of $\bar{x}$ and a $\mathcal{C}^k$ function $\varphi:\mathcal{N}\to V$ such that
$$\varphi(\bar{x}) = \bar{y} \; \text{ and } \; F(x,\varphi(x))=0 \; \text{ for every } \; x\in \mathcal{N}.$$
The derivative of $\varphi$ at the point $\bar{x}$ is the $m\times n$ Jacobian matrix
$$D\varphi (\bar{x}) = -[D_y F(\bar{x},\bar{y})]^{-1}\cdot D_x F(\bar{x},\bar{y}).$$
\end{theorem}
Assume that the function $F$ depends smoothly on a parameter $\eta$ defined on a neighbourhood of a compact $K$. In this case the neighbourhood $\mathcal{N}$ given by the implicit function theorem can be chosen uniformly with respect to $\eta \in K$.
\begin{theorem}\label{implicit_function_theorem_parametrized}
Let $U\subseteq \mathbb{R}^n$, $V\subseteq \mathbb{R}^m$ and $W\subseteq \mathbb{R}^p$ be open sets. Let $(x,y,\eta) \to F^\eta(x,y)$ be a $\mathcal{C}^k$ map from $U\times V \times W$ into $\mathbb{R}^m$, with $k\geq 1$. Let $\eta \to (x_\eta,y_\eta)$ be a $\mathcal{C}^k$ function from $W$ to $U\times V$ such that $F^\eta(x_\eta,y_\eta)=0$ for every $\eta$. If the Jacobian matrix $D_y F^\eta(x_\eta,y_\eta)$ is invertible for every $\eta$ in a compact set $K \subset W$, then there exist $\delta>0$ and a $\mathcal{C}^k$ function $(\eta,x) \to \phi^\eta(x)$ such that
$$\phi^\eta(x_\eta)=y_\eta \; \text{ and } \; F^\eta(x,\phi^\eta(x))=0 \; \text{ whenever } \; \eta \in K \; \text{ and } \; |x-x_\eta|\leq \delta.$$
\end{theorem}
Let $A$ be a $n\times n$ matrix with $n$ real distinct eigenvalues $\lambda_1<...<\lambda_n$. Let $r_1,...,r_n$ be the $n$ linearly independent eigenvectors of $A$, defined by the relations
$$A\, r_i =\lambda_i \, r_i \; \text{ for } \; i = 1,...,n.$$
The basis $\lbrace r_1,...,r_n \rbrace$ determines a basis of left eigenvectors $\lbrace l_1,...,l_n\rbrace$, i.e. vectors satisfying the relations
$$l_i\, A =\lambda_i\, l_i \; \text{ for every }\; i=1,...,n.$$
These vectors are defined by
$$l_i\, \cdot r_j = \begin{cases}
1 &\text{if } i=j,\\
0 & \text{if } i\neq j,
\end{cases} \; \text{ for every }\; i,j \in \lbrace 1,...,n \rbrace.$$ 
\begin{prop}\label{spectral_properties_continuity}
Let $A=[a_{ij}]_{i,j=1}^n$ be a $n\times n$ matrix. Suppose that for every $i$ and $j$ in $\lbrace 1,...,n \rbrace$, the entry $a_{ij}$ is a $\mathcal{C}^k$ function of a parameter $\eta\in \mathbb{R}^m$ and suppose that there exists a value $\bar{\eta}$ for which the matrix $A(\bar{\eta})$ has $n$ distinct real eigenvalues
$$\lambda_1(\bar{\eta})<...<\lambda_n(\bar{\eta}).$$ Then there exists a neighbourhood $\mathcal{N}$ of $\bar{\eta}$ such that for all $\eta\in \mathcal{N}$ the matrix $A(\eta)$ has distinct real eigenvalues
$$\lambda_1(\eta)<...<\lambda_n(\eta).$$
Moreover, for every $i=1,...,n$, we have:
\begin{enumerate}
\item the function $\eta \to \lambda_i(\eta)$ is $\mathcal{C}^k$;
\item there is a $\mathcal{C}^k$ function $\eta \to r_i(\eta)$ defined in $\mathcal{N}$ such that $r_i(\eta)$ is a right eigenvector of $A(\eta)$;
\item there is a $\mathcal{C}^k$ function $\eta \to l_i(\eta)$ defined in $\mathcal{N}$ such that $l_i(\eta)$ is a left eigenvector of $A(\eta)$.
\end{enumerate}
\end{prop}
\begin{proof}
Consider the polynomial
$$P(\eta,\lambda):=\det[\lambda\,\mathbf{I}-A(\eta)],$$
where $\mathbf{I}$ is the $n\times n$ identity matrix. For every $i \in \lbrace 1,...,n\rbrace$, we have
$$P(\bar{\eta},\lambda_i(\bar{\eta}))=0.$$
Moreover
$$\dfrac{\partial P}{\partial \lambda}(\bar{\eta},\lambda_i(\bar{\eta}))\neq 0,$$
indeed, since $\lambda_1(\bar{\eta})<...<\lambda_n(\bar{\eta})$ are eigenvalues of $A(\bar{\eta})$, we can write
$$P(\bar{\eta},\lambda)=(\lambda-\lambda_1(\bar{\eta}))\cdot ...\cdot (\lambda-\lambda_n(\bar{\eta})).$$
Hence
$$\dfrac{\partial P}{\partial \lambda}(\bar{\eta},\lambda)=\sum_{k=1}^n \, \prod_{\shortstack{$\scriptstyle{j=1}$\\ $\scriptstyle{j \neq k}$}}^n(\lambda-\lambda_j(\bar{\eta})),$$
which implies
$$\dfrac{\partial P}{\partial \lambda}(\bar{\eta},\lambda_i(\bar{\eta}))=\prod_{\shortstack{$\scriptstyle{j=1}$\\ $\scriptstyle{j \neq i}$}}^n(\lambda_i(\bar{\eta})-\lambda_j(\bar{\eta})) \neq 0.$$
We find the thesis applying the implicit function theorem: there exists a neighbourhood $\mathcal{N}_1$ of $\bar{\eta}$ and a $\mathcal{C}^k$ function $\eta \to \lambda_i(\eta)$, such that
$$P(\eta,\lambda_i(\eta))=0,$$
which implies that $\lambda_i(\eta)$ is an eigenvector for the matrix $A(\eta)$.\\
Let $r_i(\bar{\eta})$ be the normalized eigenvector of $A(\bar{\eta})$ corresponding to $\lambda_i(\bar{\eta})$. For every $i$ the matrix $A(\bar{\eta}) - \lambda_i(\bar{\eta})\mathbf{I}$ has rank $n-1$. Indeed its kernel is one-dimensional, otherwise there would exist at least two linearly independent eigenvectors $u$ and $v$ of $A(\bar{\eta})$ for the eigenvalue $\lambda_i(\bar{\eta})$. Therefore the set $$\lbrace r_1(\bar{\eta}),... r_{i-1}(\bar{\eta}),u,v,r_{i+1}(\bar{\eta}),...,r_n(\bar{\eta}) \rbrace$$
would be a set of $n+1$ linearly independent vectors and this is absurd. Hence $n-1$ row vectors of the matrix  $A(\bar{\eta}) - \lambda_i(\bar{\eta})\mathbf{I}$ are linearly independent, namely:
$$v_j(\bar{\eta})=(a_{j1}(\bar{\eta}),...,a_{jj}(\bar{\eta})-\lambda_i(\bar{\eta}),...,a_{jn}(\bar{\eta})) \; \text{ for every } \; j\in \lbrace 1,...,n \rbrace, \; j\neq j^*.$$
For $\eta=\bar{\eta}$, the vector $r_i(\eta)$ is defined by the system of $n$ equations
\begin{equation}\label{proof_spectral_properties_system_def_r_i_eta_bar}
\begin{cases}
r_i(\eta)\cdot r_i(\eta)=1,\\
v_j(\eta)\cdot r_i(\eta)=\sum_{k=1}^n a_{jk}(\eta)r_{i,k}(\eta)-\lambda_i(\eta)r_{i,j}(\eta)=0 \; \text{ for every } \; j \neq j^*,
\end{cases}
\end{equation}
where $r_{i,k}$ and $v_{j,k}$ are the $k$-th component of the vectors $r_i$ and $v_j$.
Consider the function
$$(\eta, r_i)\to F(\eta,r_i)= \begin{bmatrix}
|r_i|^2-1\\
v_j(\eta)\cdot r_i
\end{bmatrix}=\begin{bmatrix}
 \sum_{k=1}^n(r_{i,k}^2)-1\\
 \sum_{k=1}^n\, v_{j,k}(\eta)\,r_{i,k}
\end{bmatrix}\; \text{ for } \; j \neq j^*.
$$
By the system (\ref{proof_spectral_properties_system_def_r_i_eta_bar}), we have
$$F(\bar{\eta},r_i(\bar{\eta}))=0.$$
Moreover the Jacobian matrix
$$
D_{r_i} F(\eta,r_i)=\begin{bmatrix}
\partial_{r_{i,1}} F_1 & \cdots & \partial_{r_{i,n}} F_1\\
\partial_{r_{i,1}} F_2 & \cdots & \partial_{r_{i,n}} F_2\\
\vdots & \ddots & \vdots \\
\partial_{r_{i,1}} F_n & \cdots & \partial_{r_{i,n}} F_n
\end{bmatrix}= \begin{bmatrix}
2\,r_{i,1} & \cdots & 2\, r_{i,n}\\
a_{11}(\eta)-\lambda_i(\eta) & \cdots & a_{1n}(\eta)\\
\vdots & \ddots & \vdots\\
a_{n1}(\eta) & \cdots & a_{nn}(\eta)-\lambda_i(\eta)
\end{bmatrix}
$$
computed in $(\bar{\eta},r_i(\bar{\eta}))$ has rank $n$, because the first row is orthogonal to each of the following $(n-1)$ lines by the system (\ref{proof_spectral_properties_system_def_r_i_eta_bar}) and the lines $v_j(\bar{\eta})$ for $j\neq j^*$ are linearly independent. Therefore
$$\det\, D_{r_i} F(\bar{\eta},r_i(\bar{\eta})) \neq 0$$
and we can apply the implicit function theorem: there exists a neighbourhood $\mathcal{N}_2$ of $\bar{\eta}$ and a $\mathcal{C}^k$ function $\eta \to r_i(\eta)$, such that
$$F(\eta, r_i(\eta))=0 \; \text{ for every } \; \eta \in \mathcal{N}_2.$$
The equation $F(\eta, r_i(\eta))=0$ defines the $i$-th right eigenvector for the value $\eta$ and the relations
$$l_i(\eta)\cdot r_j(\eta)=\begin{cases}
1 & \text{if } i=j,\\
0 & \text{if } i \neq j,
\end{cases} \; \text{ for every } i,j \in \lbrace 1,...,n\rbrace,
$$
define the $i$-th left eigenvector. We have the thesis taking $\mathcal{N}= \mathcal{N}_1 \cap \mathcal{N}_2$.
\end{proof}
\section{Basic definitions and results}
\begin{mydef}
Let $f: \mathbb{R}^n \to \mathbb{R}^n$ be a $\mathcal{C}^2$ function and let $u:\mathbb{R}^+ \times \mathbb{R}\to \mathbb{R}^n$ be a locally integrable function.\\
A system of conservation laws is the following partial differential equation:
\begin{equation}\label{system_of_cons_laws}
\partial_t u + \partial_x [f(u)]=0.
\end{equation}
The function $f$ is called flux function and $u$ is the conserved quantity. 
\end{mydef}
Suppose that $u=u(t,x)$ is a smooth function. Let $(a,b)$ be a bounded interval in $\mathbb{R}$.\\
Integrating the equation (\ref{system_of_cons_laws}) over $(a,b)$, we find
\begin{equation*}
\begin{split}
& \int_a^b \left[\partial_t u + \partial_x [f(u)]\right] \, dx=0 \; \Longleftrightarrow \;\dfrac{d}{dt} \int_a^b u \,dx = f(a)-f(b),
\end{split}
\end{equation*}
which means that the variation in time of the quantity of $u$ inside the interval $(a,b)$ is equal to the flux of $u$ passing through $a$ and $b$. Hence there is no production of $u$ inside $(a,b)$. This justifies the name of conservation law given to the system (\ref{system_of_cons_laws}).
\begin{mydef}
Let $\bar{u}\in L^1_\text{loc}(\mathbb{R},\mathbb{R}^n)$ be a fixed function. The Cauchy problem for the system (\ref{system_of_cons_laws}) is 
\begin{equation}\label{Cauchy_problem_general}
\begin{cases}
\partial_t u + \partial_x[f(u)]=0,\\
u(0,x)=\bar{u}(x).
\end{cases}
\end{equation}
\end{mydef}
\begin{prop}
A function $u \in \mathcal{C}^1(\mathbb{R}^+\times \mathbb{R},\mathbb{R}^n)$ is a solution to (\ref{Cauchy_problem_general}) if and only if it solves the problem
\begin{equation}\label{Cauchy_problem_quasilinear_form}
\begin{cases}
\partial_t u + A(u) \, \partial_x u=0,\\
u(0,x)=\bar{u}(x),
\end{cases}
\end{equation}
where $A(u)=Df(u)$ is the Jacobian matrix of the flux function, i.e.
$$[A(u)]_{ij}= \dfrac{\partial f_i(u)}{\partial u_j}.$$  
\end{prop}
\begin{proof}
The functions $f$ and $u$ are respectively of class $\mathcal{C}^2$ and $\mathcal{C}^1$. Therefore we can apply the chain rule and we find
\begin{equation*}
0=\partial_t u_i + \partial_x[f(u)]_i=\partial_t u_i +\sum_{j=1}^n\dfrac{\partial f_i}{\partial u_j}(u)\partial_x u_j \; \text{ for every } \; i \in \lbrace 1, ... , n \rbrace.
\end{equation*}
Hence
\begin{equation*}
\begin{pmatrix}
\partial_t u_1\\
\vdots\\
\partial_t u_n
\end{pmatrix}+ \begin{bmatrix}
\partial_{u_1} f_1(u) & \cdots & \partial_{u_n} f_1(u)\\
\vdots & \ddots & \vdots\\
\partial_{u_1} f_n(u) & \cdots & \partial_{u_n} f_n (u)
\end{bmatrix}\, \begin{pmatrix}
\partial_x u_1\\
\vdots\\
\partial_x u_n
\end{pmatrix}=0,
\end{equation*}
which is the first equation in (\ref{Cauchy_problem_quasilinear_form}).
\end{proof}
The solution to the Cauchy problem (\ref{Cauchy_problem_general}) is in general discontinuous.
\begin{exmp}
Let us consider the Cauchy problem for the Burgers' equation:
\begin{equation}\label{burgers_equation_cauchy_problem}
\begin{cases}
\partial_t u +u\, \partial_x u =0,\\
u(0,x)=\dfrac{1}{1+x^2}.
\end{cases}
\end{equation}
We are going to apply the method of characteristics to find a solution to the Cauchy problem. This method discovers the curves $\gamma(s)=(\gamma_1(s),\gamma_2(s))$ in the plane $(t,x)$ along which the partial differential equation reduces to a system of ordinary differential equations; see \cite{evans}.\\
Let us denote
$$z(s)=u(\gamma(s))\; \text{ and } \; p(s)=(p_1(s),p_2(s))=Du(\gamma(s)).$$
First rewrite the equation $\partial_t u+u \,\partial_x u =0$ in the form
$$F(\gamma,z,p)=0$$
where
$$F(\gamma,z,p)=F((t,x),z,(p_1,p_2))=p_1+z\,p_2.$$
Then we have to solve the system:
$$
\begin{cases}
\dot{\gamma}(s)=D_p F(\gamma(s),z(s),p(s)),\\
\dot{z}(s)=p(s)\cdot D_p F(\gamma(s),z(s),p(s)),\\
\dot{p}(s)=-D_x F(\gamma(s),z(s),p(s))- \partial_u F(\gamma(s),z(s),p(s))\, p(s).
\end{cases}
$$
By the second equation, we find
$$\dot{z}(s)=\begin{pmatrix}
p_1(s)\\
p_2(s)
\end{pmatrix}\cdot
\begin{pmatrix}
1\\
z(s)
\end{pmatrix}=
p_1(s)+z(s)\,p_2(s)=0.$$
Therefore we obtain $u(\gamma(s))=z(s)=z(0)$ for every $s$. For the curve $\gamma(s)$ we find:
$$\dot{\gamma}(s)=\begin{pmatrix}
\dot{\gamma}_1(s)\\
\dot{\gamma}_2(s)
\end{pmatrix}=\begin{pmatrix}
1\\
z(s)
\end{pmatrix}.$$
Integrating the equations $\dot{\gamma}_1(s)=1$ and $\dot{\gamma}_2(s)=z(s)=z(0)$, we obtain
$$\gamma_1(s)=s +c_1\; \text{ and } \; \gamma_2(s)=z(0)\, s + c_2.$$
Let $(0,\bar{x})$ be the initial point for a characteristic curve, i.e.
$$\gamma(0)=(0,\bar{x}).$$
Substituting in the equations for $\gamma_1$ and $\gamma_2$, we find $c_1=0$ and $c_2=\bar{x}$. Moreover, since $u(0,x)=(1+x^2)^{-1}$, we have
$$z(0)=u(\gamma(0))=\dfrac{1}{1+\bar{x}^2}.$$
Hence we find
$$\begin{cases}
\gamma_1(s)=s,\\
\gamma_2(s)=\dfrac{1}{1+\bar{x}^2}s+\bar{x},
\end{cases}
$$
which implies
$$\gamma_2(s)=\dfrac{1}{1+\bar{x}^2}\gamma_1(s)+\bar{x}.$$
Therefore the points $(t,x)$ of the characteristic curve passing through the point $(0,\bar{x})$ satisfy the equation
$$x=\bar{x}+\dfrac{1}{1+\bar{x}^2}\,t \Longleftrightarrow t=(x-\bar{x})(1+\bar{x}^2).$$
In the plane $(t,x)$ these are lines passing through $(0,\bar{x})$ with slope $(1+\bar{x}^2)^{-1}$. If we fix two points $\bar{x}_2>\bar{x}_1\geq 0$, then the slope of the line passing through $(0,\bar{x}_2)$ is minor than the one of the line passing through $(0,\bar{x}_1)$. Therefore these lines have an intersection.\\
Since the value of $u$ on the line passing through a point $(0,\bar{x})$ is constant, i.e.
$$u(\gamma(s))=u(\gamma(0))=\dfrac{1}{1+\bar{x}^2},$$
in the point of intersection we obtain a double-valued function. Hence the solution $u$ must be discontinuous.
\end{exmp}
Hence we have to introduce the weak solution to the problem.
\begin{mydef}
Let $\Omega$ be a subset of $\mathbb{R}^+\times \mathbb{R}$. Assume that $u:\Omega \to \mathbb{R}^n$ and $f(u):\Omega \to \mathbb{R}^n$ are in $L^1_\text{loc}(\Omega)$. The function $u$ is a weak solution to the system of conservation laws (\ref{system_of_cons_laws}), if for every $\phi \in \mathcal{C}^1_c(\Omega,\mathbb{R}^n)$, we have
\begin{equation}\label{weak_solution}
\int_\Omega [u \cdot \partial_t \phi +f(u) \cdot \partial_x \phi] \, dt\,dx=0. 
\end{equation}
\end{mydef}
\begin{prop}
Let $\Omega\subseteq \mathbb{R}^+\times \mathbb{R}$ be a bounded set with smooth boundary $\partial \Omega$. A function $u \in \mathcal{C}^1(\Omega)$ is a classical solution to (\ref{system_of_cons_laws}) if and only if it is a weak solution. 
\end{prop}
\begin{proof}
Consider $\phi \in \mathcal{C}^1_c(\Omega)$ and the vector field $(\phi\, u, \phi\, f(u))$. Applying the divergence theorem, we find
\begin{equation*}
\int_\Omega \text{div}(\phi\, u, \phi \, f(u)) \, dt\, dx= \int_\Omega [\partial_t (\phi\, u)+ \partial_x (\phi \, f(u))]\,dt\,dx = \int_{\partial \Omega}(\phi\, u, \phi \, f(u))\cdot \nu \, d\sigma,
\end{equation*} 
where $\nu$ is the outer normal of $\partial \Omega$ and $d\sigma$ is the surface measure. Since $\phi \in \mathcal{C}^1_c(\Omega)$, we have
$$ \int_{\partial \Omega}(\phi\, u, \phi \, f(u))\cdot \nu \, d\sigma=0,$$
which implies
$$\int_\Omega [\partial_t (\phi\, u)+ \partial_x (\phi \, f(u))]\,dt\,dx =0.$$
Therefore the thesis follows from the equality:
$$ \int_\Omega\phi\, [\partial_t u + \partial_x[f(u)]]\, dt\,dx= -\int_\Omega[u \, \partial_t \phi +f(u) \, \partial_x \phi]\, dt\, dx.$$
\end{proof}
\begin{mydef}
Fix $\bar{u}\in L^1_\text{loc}(\mathbb{R},\mathbb{R}^n)$. A function $u \in L^1_\text{loc}([0,T]\times \mathbb{R},\mathbb{R}^n)$ is a weak solution for the Cauchy problem
\begin{equation*}
\begin{cases}
\partial_t u +\partial_x[f(u)]=0,\\
u(0,x)=\bar{u}(x),
\end{cases}
\end{equation*}
if $u$ is a weak solution to (\ref{system_of_cons_laws}) and $u(0,x)=\bar{u}(x)$ a.e. on $\mathbb{R}$.
\end{mydef}
The next theorem gives the conditions which a discontinuous function must satisfy to be a solution to the system (\ref{system_of_cons_laws}).
\begin{theorem}[Rankine-Hugoniot conditions]
Let $\Omega$ be an open set in $\mathbb{R}^+\times \mathbb{R}$ and let $u:\Omega \to \mathbb{R}^n$ be a piecewise $\mathcal{C}^1$ function, with jumps along a finite number $N\in \mathbb{N}$ of differentiable curves $x=\xi_i(t)$ for $i=1,...,N$. Let us define the limits
$$u^+(t,\xi_i)=\lim_{x \to \xi_i(t)^+}u(t,x) \; \text{ and } \; u^-(t,\xi_i) =\lim_{x \to \xi_i(t)^-}u(t,x)$$
for $i=1,...,N$.\\
Then $u$ is a weak solution to the system (\ref{system_of_cons_laws}) if and only if:
\begin{enumerate}
\item $\partial_t u+ \partial_x [f(u)]=0\;$ for every $(t,x)$ such that $x\neq \xi_i(t)$ for every $i=1,...,N$;
\item for every $t\in \mathbb{R}^+$ and $i=1,...,N$ the following condition holds:
\begin{equation}\label{Rankine-Hugoniot_conditions}
f(u^+(t,\xi_i))-f(u^-(t,\xi_i))=\dot{\xi}_i(t)\left[u^+(t,\xi_i)-u^-(t,\xi_i)\right].
\end{equation}
\end{enumerate}
\end{theorem}
\begin{proof}
Let us define the sets
\begin{equation*}
\begin{split}
& \Omega_0=\lbrace (t,x)\in\Omega : x<\xi_1(t) \rbrace,\; \; \; \Omega_{N}=\lbrace (t,x) \in \Omega : x \geq \xi_N(t)\rbrace \; \text{ and}\\
& \Omega_i= \lbrace (t,x)\in \Omega : \xi_i(t) \leq x < \xi_{i+1}(t) \rbrace \; \text{ for } i=1,...,N-1.
\end{split}
\end{equation*}
Let us consider the vector field $(\phi\, u, \phi\, f(u))$, where $\phi \in \mathcal{C}^1_c(\Omega)$. Applying the divergence theorem to each set $\Omega_i$, we find
\begin{equation*}
\begin{split}
\int_\Omega \text{div}(\phi\,u,\phi\, f(u) & )\,dt\,dx =  \sum_{i=0}^N \int_{\Omega_i} \text{div}(\phi\,u,\phi\, f(u))\,dt\,dx=\\
& = \sum_{i=0}^N\left[ \int_{\Omega_i} [\phi\, (\partial_t u+\partial_x f(u))] \, dt\,dx +\int_{\Omega_i} [u\,\partial_t \phi + f(u) \, \partial_x \phi]\, dt\, dx\right]=\\
& =\sum_{i=0}^N \int_{\partial \Omega_i}(\phi\,u,\phi\, f(u))\cdot\nu_i \, d\sigma,
\end{split}
\end{equation*}
where $\nu_i$ is the outer normal of the set $\Omega_i$ and $d\sigma$ is the surface measure.\\
The function $\phi$ is zero on $\partial \Omega_i$ except eventually for the points $(t,\xi_i(t))$ for $i=1,...,N$. Fix $i\in \lbrace 1,...,N \rbrace$ and take a parametrization
$$t\in[a,b]\to \xi_i(t)\in \Omega$$
of the points of the curve $\xi_i$ inside the domain $\Omega$. The curve $\xi_i$ is the intersection between $\Omega_{i-1}$ and $\Omega_i$ and its tangent vector is
$$\tau_i(t)= (1,\dot{\xi}_i(t)).$$
The outer normal of $\Omega_i$ on the curve $x=\xi_i(t)$ is
$$\nu_i(t)=\dfrac{1}{\sqrt{1+|\dot{\xi}_i(t)|^2}}(-\dot{\xi}_i(t),1),$$
while the outer normal of $\Omega_{i-1}$ on the same curve is
$$\nu_{i-1}(t)= \dfrac{1}{\sqrt{1+|\dot{\xi}_i(t)|^2}}(\dot{\xi}_i(t),-1).$$
Hence we find
\begin{equation*}
\begin{split}
\sum_{i=0}^N \int_{\partial \Omega_i} & (\phi\,u,\phi\, f(u))\cdot\nu_i \, d\sigma =\\ 
& = \sum_{i=1}^N \int_a^b \phi(s,\xi_i(s))\,[u_i^-\,\dot{\xi}_i(s)-f(u_i^-)-u_i^+\,\dot{\xi}_i(s)+f(u_i^+)]\,ds,
\end{split}
\end{equation*}
where the normalization factor $(1+|\dot{\xi}_i(t)|^2)^{-1/2}$ has disappeared in the product with the arclength. Therefore the thesis follows from the equation:
\begin{equation*}
\begin{split}
& \int_{\Omega} [u\,\partial_t \phi + f(u) \, \partial_x \phi]\, dt\, dx =\\
& = - \int_{\Omega} [\phi\, (\partial_t u+\partial_x f(u))] \, dt\,dx +\sum_{i=1}^N \int_a^b \phi(s,\xi_i)\,[(u_i^- -u_i^+)\,\dot{\xi}_i(s)+f(u_i^+)-f(u_i^-)]\,ds.
\end{split}
\end{equation*}
\end{proof}
\section{The Riemann problem}
In this section we introduce the Riemann problem for a system of conservation laws and we define its standard solution.
\begin{mydef}
Let $\Omega\in \mathbb{R}^n$ be an open set and let $f:\Omega \to \mathbb{R}^n$ be a vector field of class $\mathcal{C}^2$. Fix two points $u^l$ and $u^r$ in $\Omega$. The Riemann problem for the system of conservation laws (\ref{system_of_cons_laws}) is
\begin{equation}\label{Riemann_problem_general}
\begin{cases}
\partial_t u + \partial_x [f(u)]=0,\\
u(0,x)=\begin{cases}
u^l &\text{if } x\leq 0,\\
u^r & \text{if } x>0.
\end{cases}
\end{cases}
\end{equation}
\end{mydef}
\begin{mydef}
The system (\ref{system_of_cons_laws}) is strictly hyperbolic if for every $u \in \Omega$ the Jacobian matrix $Df(u)$ has $n$ real distinct eigenvalues $\lambda_1(u)< ... <\lambda_n(u)$.
\end{mydef}
If the system (\ref{system_of_cons_laws}) is strictly hyperbolic, then for every $u\in \Omega$ we can find two bases $\lbrace r_1(u),...,r_n(u) \rbrace$ and $\lbrace l_1(u),...,l_n(u) \rbrace$ respectively of right and left eigenvalues of $Df(u)$, i.e.
$$ l_i(u) \cdot Df(u)=\lambda_i(u) \, l_i(u) \; \text{ and } \;  Df(u) \cdot r_i(u) = \lambda_i(u)\,  r_i(u) \; \text{ for every } \; i=1,...,n.$$
Moreover we can choose the eigenvectors so that for every $u \in \Omega$ the following conditions are satisfied:
$$|r_i(u)|=1 \; \text{ and } \; l_i(u)\cdot r_j(u) = \begin{cases}
1 & \text{if } i=j,\\
0 & \text{if } i \neq j,
\end{cases} \; \text{ for every } \; i,j\in \lbrace 1,...,n \rbrace.$$
By Proposition \ref{spectral_properties_continuity}, since the entries $a_{ij}(u)$ of the matrix $A(u)$ are $\mathcal{C}^1$, the functions $u\to \lambda_i(u)$, $u\to r_i(u)$ and $u \to l_i(u)$ are of class $\mathcal{C}^1$.
\begin{mydef}
Fix $i \in \lbrace 1,...,n \rbrace $. The $i$-th characteristic field is the vector field
$$r_i(u)\Big|_{u\in\Omega}.$$
The $i$-th characteristic field is genuinely non-linear if
$$\nabla \lambda_i(u)\cdot r_i(u) \neq 0 \; \text{ for every } \;  u \in \Omega,$$
while it is linearly degenerate if
$$\nabla \lambda_i(u)\cdot r_i(u) =0 \; \text{ for every } \; u \in \Omega.$$
\end{mydef}
If the $i$-th characteristic field is genuinely non-linear, we can choose the sign of $r_i$ so that
\begin{equation}\label{sign_of_characteristic_field}
\nabla \lambda_i(u)\cdot r_i(u) >0 \; \text{ for every }\; u \in \Omega.
\end{equation}
Indeed, both functions $u \to r_i(u)$ and $u \to \nabla \lambda_i(u)$ are continuous. Hence $\lambda_i(u)$ is strictly decreasing or increasing along the direction of $r_i$, otherwise there would exist a point $u_0$ such that
$$\nabla \lambda_i(u_0)\cdot r_i(u_0)=0$$
and this is a contradiction of the hypothesis on the characteristic field.\\
The solution that we are going to define is self-similar, i.e. there exists a function $\psi:\mathbb{R}\to \mathbb{R}^n$ possibly discontinuous, such that
$$u(t,x)=\psi\left( \dfrac{x}{t}\right).$$
\subsection{Rarefaction waves}
Fix $i \in \lbrace 1,...,n \rbrace$. Suppose that the $i$-th characteristic field is genuinely non-linear.\\
Let us consider the Cauchy problem
\begin{equation}\label{Cauchy_problem_rarefaction_wave}
\begin{cases}
\dot{w}=r_i(w),\\
w(0)=w_0.
\end{cases}
\end{equation}
The function $w \to r_i(w)$ is $\mathcal{C}^1(\Omega)$, hence it is locally Lipschitz continuous. Therefore the system (\ref{Cauchy_problem_rarefaction_wave}) has a local unique solution $ \sigma \to R_i(\sigma)(w_0)$ which is the integral curve of the vector field $r_i(w)$ passing through $w_0$.
\begin{prop}\label{rarefaction_wave_properties}
Suppose that the system (\ref{system_of_cons_laws}) is strictly hyperbolic with smooth coefficients defined in an open set $\Omega\in \mathbb{R}^n$. Let $i \in \lbrace 1,...,n \rbrace$ be fixed and assume that the $i$-th characteristic field is genuinely non-linear. Fix two points $u^l$ and $u^r$ in $\Omega$. Suppose that there exists $\bar{\sigma}\geq 0$ such that
$$u^r=R_i(\bar{\sigma})(u^l),$$
where $\sigma \to R_i(\sigma)(u^l)$ is the solution to the Cauchy problem (\ref{Cauchy_problem_rarefaction_wave}) with initial datum $w(0)=u^l$.\\
Then the function $\sigma \to \lambda_i(R_i(\sigma)(u^l))$ is strictly increasing and
\begin{equation}\label{rarefaction_wave_def}
u(t,x)=\begin{cases}
u^l & \text{if } x<t\, \lambda_i(u^l),\\
R_i(\sigma)(u^l) & \text{if } x= t\, \lambda_i(R_i(\sigma)(u^l)) \; \text{ for } \; \sigma \in [0,\bar{\sigma}],\\
u^r & \text{if } x> t\, \lambda_i(u^r),
\end{cases}
\end{equation}
is a weak solution to the Riemann problem (\ref{Riemann_problem_general}).
\end{prop}
\begin{proof}
\textit{\textbf{Part 1.}} Consider the function
$$\sigma \in [0,\bar{\sigma}] \to \lambda_i(\sigma)=\lambda_i(R_i(\sigma)(u^l)).$$
By the chain rule and the definition of $R_i(\sigma)(u^l)$, we find
$$\dfrac{d\lambda_i(\sigma)}{d\sigma}= \nabla \lambda_i(R_i(\sigma)(u^l))\cdot r_i(R_i(\sigma)(u^l)).$$
By the condition (\ref{sign_of_characteristic_field}), we obtain
$$\dfrac{d\lambda_i(\sigma)}{d\sigma}>0,$$
which implies that the function $\sigma \to \lambda_i(\sigma)$ is strictly increasing, i.e.
$$\lambda_i(u^r)=\lambda_i(\bar{\sigma})>\lambda_i(\sigma)>\lambda_i(0)=\lambda_i(u^l) \; \text{ for every } \; \sigma \in (0,\bar{\sigma}).$$
Moreover there exists an inverse function
$$\lambda_i \to \sigma(\lambda_i).$$
\textit{\textbf{Part 2.}} Let us prove that
$$ \lim_{t\to 0^+} \parallel u(t,\cdot)-u(0,\cdot) \parallel_{L^1(\mathbb{R})}=0.$$
Indeed $t\,\lambda_i(u^l) \xrightarrow{t\to 0^+} 0$ and $t \, \lambda_i(u^r) \xrightarrow{t\to0^+} 0$. Moreover for $x=t\,\lambda_i(R_i(\sigma)(u^l))$, we find
\begin{equation*}
|u(t,x)-u(0,x)|=|R_i(\sigma)(u^l) -u(0,x)| \leq \sup_{\sigma\in[0,\bar{\sigma}]} |R_i(\sigma)(u^l)|+\max(|u^l|,|u^r|).
\end{equation*}
The function $\sigma \to R_i(\sigma)(u^l)$ is differentiable and defined in a compact, then it is bounded. Therefore
\begin{equation*}
\int_\mathbb{R}|u(t,x)-u(0,x)|\, dx = \int_{t\,\lambda_i(u^l)}^{t\,\lambda_i(u^r)}|R_i(\sigma)(u^l)-u(0,x)|\,dx \xrightarrow{t\to 0^+} 0,
\end{equation*}
because
\begin{equation*}
\begin{split}
|R_i(\sigma)(u^l) -u(0,x)| <+\infty.
\end{split}
\end{equation*}
Then, for a.e. $x \in \mathbb{R}$, we have
$$\lim_{t \to 0^+} u(t,x)=u(0,x).$$
\textit{\textbf{Part 3.}} Now, let us show that the function $u(t,x)$ defined in (\ref{rarefaction_wave_def}) satisfies the equation (\ref{system_of_cons_laws}).\\
The equation $\partial_t u+\partial_x[f(u)]=0$ is trivially satisfied in the sets
$$\lbrace (t,x) \in \mathbb{R}^+\times \mathbb{R}: x<t\, \lambda_i(u^l)\rbrace \; \text{ and } \; \lbrace (t,x) \in \mathbb{R}^+\times \mathbb{R}: x>t\, \lambda_i(u^r)\rbrace,$$
where the value of $u(t,x)$ is respectively $u^l$ and $u^r$. For every $\sigma \in [0,\bar{\sigma}]$, the function $u$ is constant in the set
$$\Lambda^\sigma=\lbrace (t,x)\in \mathbb{R}^+\times \mathbb{R}: x=t\,\lambda_i(\sigma)\rbrace.$$
Hence its directional derivative on the line $(t,t\,\lambda_i(\sigma))$ is zero, i.e.
$$
\nabla u \cdot \begin{pmatrix}
1\\
\lambda_i(\sigma)
\end{pmatrix} =\partial_t u +\lambda_i(\sigma)\, \partial_x u=0.
$$
The derivative $\partial_x u$ is parallel to $r_i(u)$ in the points $(t,x)$ such that $\lambda_i(R_i(\sigma)(u^l))=x/t$. Indeed by the chain rule we obtain
\begin{equation*}
\begin{split}
\dfrac{\partial u(t,x)}{\partial x} & = \dfrac{\partial R_i(\sigma)(u^l)}{\partial x} = \dfrac{d R_i(\sigma)(u^l)}{d \sigma} \, \dfrac{d\sigma(\lambda_i)}{d\lambda_i} \, \dfrac{d\lambda_i}{d x}= r_i(R_i(\sigma)(u^l)) \, \dfrac{d\sigma}{d\lambda_i} \,  \dfrac{d(x/t)}{d x}=\\
& = \dfrac{d\sigma}{d\lambda_i}\, t^{-1} \, r_i(R_i(\sigma)(u^l)).
\end{split}
\end{equation*}
Then $\partial_x u$ is an eigenvector of $Df(u)$ with eigenvalue $\lambda_i(\sigma)$, i.e.
$$\lambda_i(\sigma)\partial_x u =Df(u) \partial_x u.$$
Therefore
$$0=\partial_t u +\lambda_i(\sigma)\, \partial_x u=\partial_t u +Df(u) \partial_x u.$$
\end{proof}
\begin{mydef}\label{definition_rarefaction_wave}
The solution defined in (\ref{rarefaction_wave_def}) to the Riemann problem (\ref{Riemann_problem_general}) is called centred rarefaction wave. The solution $R_i(\sigma)(u^l)$ to the Cauchy problem (\ref{Cauchy_problem_rarefaction_wave}) with initial datum $w_0=u^l$ is called $i$-th rarefaction curve passing through $u^l$.
\end{mydef}
\begin{remark}
If $u^r=R_i(\bar{\sigma})(u^l)$ for $\bar{\sigma} <0$, then we find
$$\lambda_i(u^l)>\lambda_i(u^r).$$
Therefore the solution (\ref{rarefaction_wave_def}) is not admissible, because when $t\,\lambda_i(u^r)<x<t\,\lambda_i(u^l)$, we have a triple-valued function.
\end{remark}
\subsection{Shock waves}
\begin{prop}\label{properties_shock_waves}
Suppose that the system (\ref{system_of_cons_laws}) is strictly hyperbolic with smooth coefficients defined in an open set $\Omega\in \mathbb{R}^n$.\\
Fix a point $u^l$ in an open set $\Omega\in \mathbb{R}^n$. There exists a number $\bar{\sigma}>0$ and $n$ smooth curves $S_i(\cdot)(u^l):[-\bar{\sigma},\bar{\sigma}]\to \Omega$ together with $n$ scalar functions $\lambda_i:[-\bar{\sigma},\bar{\sigma}]\to \mathbb{R}$ such that
\begin{equation}\label{RH_conditions_shock_curve}
f(S_i(\sigma)(u^l))-f(u^l)=\lambda_i(\sigma)(S_i(\sigma)(u^l)-u^l) \; \text{ for every } \; \sigma \in [-\bar{\sigma},\bar{\sigma}].
\end{equation}
Moreover:
\begin{enumerate}
\item[(i)] the function
\begin{equation}\label{def_shock_wave}
u(t,x)=\begin{cases}
u^l & \text{if } x\leq t\, \lambda_i(\sigma),\\
S_i(\sigma)(u^l) & \text{if } x > t\, \lambda_i(\sigma),
\end{cases}
\end{equation}
is a weak solution to the system (\ref{system_of_cons_laws});
\item[(ii)] $\Big| \dfrac{dS_i(\sigma)(u^l)}{d\sigma}\Big|=1 \;$ for every $\sigma \in [-\bar{\sigma},\bar{\sigma}]$;
\item[(iii)] at $\sigma=0$ we have
\begin{equation*}
\begin{split}
& \lambda_i(0)=\lambda_i(u^l), \;\; \; \dfrac{d\lambda_i(\sigma)}{d\sigma}\Big|_{\sigma=0} = \dfrac{1}{2}\nabla \lambda_i(u^l)\cdot r_i(u^l), \;\;\; S_i(0)(u^l)=u^l \; \text{ and }\\
& \dfrac{dS_i(\sigma)(u^l)}{\sigma}\Big|_{\sigma=0}=r_i(u^l).
\end{split}
\end{equation*}
\end{enumerate}
\end{prop}
\begin{proof}\textbf{\textit{Part 1.}} Fix two points $u$ and $ \bar{u}$ in $\Omega$ such that for every $\xi \in [0,1]$, we have
$$\xi \, u+(1-\xi)\, \bar{u} \in \Omega.$$
Let us consider the matrix
$$A(u,\bar{u}) = \int_0^1 A(\xi \, u+(1-\xi)\, \bar{u})\, d\xi.$$
We claim that there exists a neighbourhood $\mathcal{N}$ of $\bar{u}$ for which, for every $u \in \mathcal{N}$, the matrix $A(u,\bar{u})$ has $n$ distinct eigenvalues
$$\lambda_1(u,\bar{u})<...<\lambda_n(u,\bar{u}).$$
Moreover the function $u \to \lambda_i(u,\bar{u})$ is a $\mathcal{C}^1$ function of $u$. We postpone the proof of this claim.\\
Call
$$\lbrace l_1(u,\bar{u}),...,l_n(u,\bar{u}) \rbrace \; \text{ and } \; \lbrace r_1(u,\bar{u}),..., r_n(u,\bar{u}) \rbrace$$
the left and right eigenvectors corresponding to $\lambda_1(u,\bar{u}),...,\lambda_n(u,\bar{u})$.\\
We can choose the eigenvectors normalized so that they satisfy
$$|r_i(u,\bar{u})|=1 \; \text{ and } \; l_i(u,\bar{u})\cdot r_j(u,\bar{u}) = \begin{cases}
1 & \text{if } i=j,\\
0 & \text{if } i \neq j,
\end{cases}\; \text{ for every } \, i,j=1,...,n.$$
By the fundamental theorem of calculus, we have
$$f(u)-f(u^l)=\int_0^1 Df(\sigma\,u+(1-\sigma)\,u^l)(u-u^l)\,d\sigma =A(u,u^l)(u-u^l),$$
which implies that the Rankine-Hugoniot conditions 
$$f(u)-f(u^l)=\lambda\,(u-u^l)$$
hold whenever $(u-u^l)$ is a right eigenvector of $A(u,u^l)$ with eigenvalue $\lambda=\lambda_i(u,u^l)$ for some $i\in \lbrace 1,...,n\rbrace$.\\
\textbf{\textit{Part 2.}} Fix $i\in \lbrace 1,...,n\rbrace$. The $i$-th eigenvector of $A(u,u^l)$ is the solution of
\begin{equation}\label{eigenvector_as_solution_orthogonal_to_left_eigenvectors}
\phi_j(u):=l_j(u,u^l)\cdot (u-u^l)=0 \; \text{ for every } \; j \neq i,
\end{equation}
which is a system of $n-1$ equations for the $n$ variables $u_1,...,u_n$ (the components of $u$).\\
We find:
\begin{equation*}
\begin{split}
& \dfrac{\partial \phi_j}{\partial u_k}(u)=\dfrac{\partial l_j}{\partial u_k}(u,u^l)\cdot (u-u^l) + l_{j,k}(u,u^l) \Longrightarrow \\
& \Longrightarrow \dfrac{\partial \phi_j}{\partial u_k}(u^l)=l_{j,k}(u^l).
\end{split}
\end{equation*}
The vectors $\lbrace l_j(u^l)\rbrace_{j=1,...,n}$ are linearly independent. Therefore the Jacobian matrix
$$\left[\dfrac{\partial \phi_j}{\partial u_k}\right]^{j=1,...,n-1}_{k=1,...,n}$$
has rank $n-1$.\\
Hence there exist $n-1$ linearly independent lines and, without loss of generality, we can assume that they are the first $n-1$. Therefore the Jacobian matrix
$$\left[\dfrac{\partial \phi_j}{\partial u_k}\right]_{j,k=1,...,n-1}$$
is invertible and we can apply the implicit function theorem: there exists a curve $\sigma \to S_i(\sigma)(u^l)$ defined in a neighbourhood $[-\bar{\sigma},\bar{\sigma}]$ of $\sigma=0$ such that
$$S_i(0)(u^l)=u^l \; \text{ and } \; \phi(S_i(\sigma)(u^l))=0 \; \text{ for every }\; \sigma \in [-\bar{\sigma},\bar{\sigma}].$$
Hence $S_i(\sigma)(u^l)$ is the right eigenvector of $A(S_i(\sigma)(u^l),u^l)$ with eigenvalue $\lambda_i(S_i(\sigma)(u^l),u^l)$. Calling $\lambda_i(\sigma)=\lambda_i(S_i(\sigma)(u^l),u^l)$, we obtain
$$f(S_i(\sigma)(u^l))-f(u^l)=\lambda_i(\sigma)\,(S_i(\sigma)(u^l)-u^l).$$
Therefore the function (\ref{def_shock_wave}) is a weak solution of the Riemann problem (\ref{Riemann_problem_general}), because $S_i(\sigma)(u^l)$ satisfies the Rankine-Hugoniot conditions with the propagation speed $\lambda_i(\sigma)$. This proves $(i)$.\\
\textbf{\textit{Part 3.}} To prove  $(ii)$, we can re-parametrize the curve $S_i(\sigma)$ by its arclength, because it is regular and we find the thesis.\\
For $(iii)$ we have already shown that
$$S_i(0)(u^l)=u^l \; \text{ and } \;\lambda_i(0)=\lambda_i(u^l,u^l)=\lambda_i(u^l).$$
Let us show that
\begin{equation}\label{length_of_the_shock_curve_derivative}
\dfrac{d S_i(\sigma)(u^l)}{d\sigma}\Big|_{\sigma=0}=r_i(u^l).
\end{equation}
By the equations (\ref{eigenvector_as_solution_orthogonal_to_left_eigenvectors}), at $\sigma=0$ the curve $S_i(\sigma)(u^l)$ must be orthogonal to all the vectors $l_j(u^l)$ for $j\neq i$. This is the case of $r_i(u^l)$. Hence we obtain the equation (\ref{length_of_the_shock_curve_derivative}).\\
Finally, let us denote $\dot{g}$ the derivative w.r.t. $\sigma$ of a function $g(\sigma)$, $A_i(\sigma)=A(S_i(\sigma))$ and $S_i(\sigma)(u^l)=S_i(\sigma)$.\\
Note that the $i$-th eigenvalue $\lambda_i(S_i(\sigma))$ of the matrix $A(S_i(\sigma))$ and the propagation speed $\lambda_i(\sigma)$ of the shock wave are in general different.\\
Differentiating the equation (\ref{RH_conditions_shock_curve}) two times w.r.t. $\sigma$, we obtain
\begin{equation*}
\begin{split}
& A_i(\sigma) \,\dot{S}_i(\sigma)= \dot{\lambda}_i(\sigma)\,(S_i(\sigma)-u^l)+\lambda_i(\sigma)\, \dot{S}_i(\sigma) \longrightarrow \\
& \longrightarrow \dot{A}_i(\sigma)\,\dot{S}_i(\sigma) + A_i(\sigma)\, \ddot{S}_i(\sigma)= \ddot{\lambda}_i(\sigma)\,(S_i(\sigma)-u^l)+2\, \dot{\lambda}_i(\sigma)\,\dot{S}_i(\sigma)+\lambda_i(\sigma)\, \ddot{S}_i(\sigma).
\end{split}
\end{equation*}
Since $S_i(0)=u^l$ and $\dot{S}_i(0)=r_i(u^l)$, at $\sigma=0$ we obtain
\begin{equation}\label{derivative_of_the_matrix_A_i_1}
\dot{A}_i(0)\, r_i(u^l) +A_i(0)\, \ddot{S}_i(0)=2 \, \dot{\lambda}_i(0)\, r_i(u^l)+\lambda_i(0)\,\ddot{S}_i(0).
\end{equation}
Let us now differentiate the relation
$$A_i(\sigma)\, r_i(S_i(\sigma)) =\lambda_i(S_i(\sigma))\, r_i(S_i(\sigma)).$$
We find
\begin{equation*}
\begin{split}
& \dot{A}_i(\sigma) \, r_i(S_i(\sigma)) +A_i(\sigma)\, \dot{r}_i(S_i(\sigma))=\dot{\lambda}_i(S_i(\sigma))\,r_i(S_i(\sigma))+\lambda_i(S_i(\sigma))\, \dot{r}_i(S_i(\sigma)) \Longrightarrow \\
& \begin{split} \dot{A}_i(\sigma)\, r_i(S_i(\sigma)) +A_i(\sigma)\, \nabla r_i(S_i(\sigma))\cdot \dot{S}_i(\sigma) = & \nabla \lambda_i(S_i(\sigma))\cdot \dot{S}_i(\sigma)\, r_i(S_i(\sigma))+\\
& + \lambda_i(S_i(\sigma)) \, \nabla r_i(S_i(\sigma))\cdot \dot{S}_i(\sigma),
\end{split}
\end{split}
\end{equation*}
which in $\sigma=0$ gives
\begin{equation}\label{derivative_of_the_matrix_A_i_2}
\dot{A}_i(0)\, r_i(u^l) = \nabla \lambda_i(u^l) \cdot r_i(u^l)\,r_i(u^l) + \lambda_i(u^l) \, \nabla r_i(u^l) \cdot r_i(u^l)  -  A_i(0) \, \dot{r}_i(u^l).
\end{equation}
Using this expression in the equation (\ref{derivative_of_the_matrix_A_i_1}), we obtain
\begin{equation*}
\begin{split}
A_i(0) \, \ddot{S}_i(0)+\nabla \lambda_i(u^l) \cdot r_i(u^l)\, r_i(u^l) & + \lambda_i(u^l)\, \nabla r_i(u^l) \cdot r_i(u^l) -  A_i(0)\, \dot{r}_i(u^l) =\\
& = 2\, \dot{\lambda}_i(0)\, r_i(u^l) +\lambda_i(0) \, \ddot{S}_i(0).
\end{split}
\end{equation*}
Multiplying on the left for $l_i(u^l)$, we find the thesis, indeed
\begin{equation*}
\begin{split}
l_i(u^l) \,A(u^l) \ddot{S}_i(0) & +\nabla \lambda_i(u^l) \, \cdot r_i(u^l)\,l_i(u^l)\,r_i(u^l) + l_i(u^l) (\lambda_i(u^l)-A(u^l))\, \nabla r_i(u^l) \cdot r_i(u^l) = \\
& = 2 \, l_i(u^l)\, \dot{\lambda}_i(0)\, r_i(u^l)+ l_i(u^l)\, \lambda_i(u^l)\, \ddot{S}_i(0),
\end{split}
\end{equation*}
which implies
$$\nabla \lambda_i \cdot r_i = 2\, \dot{\lambda}_i,$$
because $l_i(u^l)\,A(u^l)=\lambda_i(u^l)\,l_i(u^l)$ and $l_i(u^l) r_i(u^l)=1$.\\

\textit{\textbf{Proof of the claim.}}\\
\textit{Claim:} For every $u$ in a neighbourhood $\mathcal{N}$ of $\bar{u}$, the matrix $A(u,\bar{u})$ has $n$ real distinct eigenvalues $\lambda_1(u,\bar{u})$, ..., $\lambda_n(u,\bar{u})$. Moreover the function $u\to \lambda_i(u,\bar{u})$ is smooth for every $i=1,...,n$.\\
\textit{Proof.}
Fix $i \in \lbrace 1,...,n \rbrace$ and let us consider the polynomial
$$P(u,\lambda)=\det(A(u,\bar{u})-\lambda\, \mathbf{I}).$$
Since $A(\bar{u},\bar{u})=A(\bar{u})$ and $\lambda_i(\bar{u})$ is an eigenvalue of $A(\bar{u})$, we have
$$P(\bar{u},\lambda_i(\bar{u}))=0.$$
Moreover, we can write
$$P(\bar{u},\lambda)=(\lambda-\lambda_1(\bar{u}))\cdot ... \cdot (\lambda-\lambda_n(\bar{u})).$$
Hence
$$\dfrac{\partial P}{\partial \lambda}(\bar{u},\lambda_i(\bar{u}))=\prod_{\shortstack{$\scriptstyle{j=1}$\\ $\scriptstyle{j \neq i}$}}^n(\lambda_i(\bar{u})-\lambda_j(\bar{u})) \neq 0,$$
because the eigenvalues of $A(\bar{u})$ are distinct.\\
Therefore by the implicit function theorem, there exists a neighbourhood $\mathcal{N}$ of $\bar{u}$ and a smooth function $u \to \lambda_i(u,\bar{u})$ such that
$$P(u,\lambda_i(u,\bar{u}))=0,$$
which means that $\lambda_i(u,\bar{u})$ is an eigenvalue of the matrix $A(u,\bar{u})$.
\end{proof}
\begin{mydef}
The solution (\ref{def_shock_wave}) to the Riemann problem (\ref{Riemann_problem_general}) is called centred shock wave. The function $\sigma \to S_i(\sigma)(u^l)$ is called $i$-th shock curve passing through $u^l$.
\end{mydef}
\subsection{Contact discontinuities}
\begin{prop}
Suppose that the system (\ref{system_of_cons_laws}) is strictly hyperbolic with smooth coefficients defined in an open set $\Omega\in \mathbb{R}^n$. Fix $i \in \lbrace 1,...,n \rbrace$ and a point $u^l$ in an open set $\Omega\in \mathbb{R}^n$. Suppose that the $i$-th characteristic field is linearly degenerate. Then the $i$-th shock curve and the $i$-th rarefaction curve passing through $u^l$ coincide, i.e.
\begin{equation}\label{shock_and_rar_curve_coincide}
S_i(\sigma)(u^l)=R_i(\sigma)(u^l).
\end{equation}
Moreover, if $u^r= S_i(\bar{\sigma})(u^l)$ for some number $\bar{\sigma}\in \mathbb{R}$, then the function
\begin{equation}\label{def_contact_discontinuity}
u(t,x)=\begin{cases}
u^l & \text{if } x\leq t\, \lambda_i(u^l),\\
u^r & \text{if } x>t\, \lambda_i(u^r),
\end{cases}
\end{equation}
is a solution to the Riemann problem (\ref{Riemann_problem_general}).
\end{prop}
\begin{proof}
To prove the condition (\ref{shock_and_rar_curve_coincide}), let us fix a number $\bar{\sigma}$. By the fundamental theorem of calculus and the definition of $\lambda_i$ and $r_i$, we have
\begin{equation*}
\begin{split}
f(R_i(\bar{\sigma})(u^l))-f(u^l) & = \int_0^{\bar{\sigma}}Df(R_i(\sigma)(u^l))\dfrac{dR_i(\sigma)(u^l)}{d \sigma}\, d\sigma =\\
& = \int_0^{\bar{\sigma}}Df(R_i(\sigma)(u^l)) \, r_i(R_i(\sigma)(u^l)) \, d\sigma=\\
& = \int_0^{\bar{\sigma}}\lambda_i(R_i(\sigma)(u^l)) \, r_i(R_i(\sigma)(u^l)) \, d\sigma.
\end{split}
\end{equation*}
Since the $i$-th characteristic field is linearly degenerate, $\lambda_i(u)$ is constant along the integral curve of $r_i$. Therefore
\begin{equation*}
\begin{split}
f(R_i(\sigma)(u^l))-f(u^l) & = \lambda_i(u^l)\int_0^{\bar{\sigma}}\dfrac{dR_i(\sigma)(u^l)}{d\sigma}\, d\sigma = \\
& = \lambda_i(u^l) \, [R_i(\bar{\sigma})(u^l)-u^l],
\end{split}
\end{equation*}
which means that $R_i(\bar{\sigma})(u^l)$ satisfies the Rankine-Hugoniot conditions with propagation speed $\lambda_i(u^l)$. Hence
$$S_i(\sigma)(u^l) = R_i(\sigma)(u^l)$$
and the function (\ref{def_contact_discontinuity}) is a weak solution to the Riemann problem (\ref{Riemann_problem_general}).
\end{proof}
\begin{mydef}
The solution (\ref{def_contact_discontinuity}) to the Riemann problem (\ref{Riemann_problem_general}) is called contact discontinuity.
\end{mydef}
\subsection{The general solution to the Riemann problem}
In general the weak solution to the Riemann problem (\ref{Riemann_problem_general}) is not unique. Therefore we have to introduce additional admissibility conditions. There are several approaches. We use the Lax entropy condition; see \cite{lax}.
\begin{mydef}
A shock joining two points $u^l$ and $u^r$ in an open set $\Omega\subseteq \mathbb{R}^n$, with propagation speed $\lambda$ is Lax-admissible if there exists an index $k \in \lbrace 1,...,n \rbrace$ such that
\begin{equation}\label{Lax_entropy_condition}
\lambda_k(u^l) \geq \lambda \geq \lambda_k(u^r).
\end{equation}
\end{mydef}
\begin{prop}[Admissibility of a shock] Fix a point $u^l$ in an open set $\Omega\subseteq \mathbb{R}^n$. There exists a positive number $\bar{\sigma}$ for which a shock joining $u^l$ to a point $u(\sigma)=S_i(\sigma)(u^l)$ is Lax-admissible if and only if $-\bar{\sigma} \leq \sigma \leq 0$.
\end{prop}
\begin{proof}
By Proposition \ref{properties_shock_waves}, the function $\sigma \to \lambda_i(\sigma)$, which gives the propagation speed of the shock, is $\mathcal{C}^1$ and by the condition (\ref{sign_of_characteristic_field}), we obtain
$$\dfrac{d\lambda_i(\sigma)}{d\sigma}\Big|_{\sigma=0}=\dfrac{1}{2}\nabla \lambda_i(u^l) \cdot r_i(u^l)>0.$$
Therefore $\sigma \to \lambda_i(\sigma)$ is increasing in a neighbourhood $[-\bar{\sigma},\bar{\sigma}]$ of $\sigma=0$. Therefore
$$\lambda_i(\bar{\sigma})\leq \lambda_i(\sigma) \leq \lambda_i(0)=\lambda_i(u^l) \; \text{ if } -\bar{\sigma}\leq \sigma \leq 0.$$
The thesis follows observing that by the Lax-entropy condition, a shock of the $i$-th family is admissible if its propagation speed $\lambda_i(\sigma)$ is lower than $\lambda_i(u^l)$.
\end{proof}
We are now ready to introduce the Lax-admissible solution to the Riemann problem (\ref{Riemann_problem_general}).\\
Fix $i \in \lbrace 1,...,n \rbrace$ and a point $u_0\in \mathbb{R}^n$. Either the $i$-th characteristic field is genuinely non-linear or linearly degenerate, we can define the map
\begin{equation}\label{Lax_curves_general}
\psi_i(\sigma)(u_0)=\begin{cases}
R_i(\sigma)(u_0) & \text{if } \sigma \geq 0,\\
S_i(\sigma)(u_0) & \text{if } \sigma <0.
\end{cases}
\end{equation}
\begin{mydef}
The curve defined in (\ref{Lax_curves_general}) is called $i$-th Lax curve.
\end{mydef}
Let us take $(\sigma_1,...,\sigma_n)\in \mathbb{R}^n$ in a neighbourhood of $0\in \mathbb{R}^n$ and let us define the points
\begin{equation}
w_0=u^l, \; \; \; w_i=\psi_i(\sigma_i)(w_{i-1})=\begin{cases}
R_i(\sigma_i)(w_{i-1}) & \text{if } \sigma_i \geq 0,\\
S_i(\sigma_i)(w_{i-1}) & \text{if } \sigma_i <0,
\end{cases}\; \text{ for every } i=1,...,n,
\end{equation}
so that
$$w_n=\psi_n(\sigma_n)\circ \cdots \circ \psi_1(\sigma_1)(u^l).$$
Assume that $u^r=w_n$.\\
Each Riemann problem
\begin{equation}\label{ith_Riemann_problem_def_general_solution}
\begin{cases}
\partial_t u+\partial_x[f(u)]=0,\\
u(0,x)=\begin{cases}
w_{i-1} & \text{if } x \leq 0,\\
w_i & \text{if } x>0,
\end{cases}
\end{cases}
\end{equation}
has a unique Lax-admissible solution consisting of a simple wave of the $i$-th family, i.e.:
\begin{enumerate}
\item[(i)] if the $i$-th characteristic field is genuinely non-linear and $\sigma_i\geq 0$, the solution to (\ref{ith_Riemann_problem_def_general_solution}) is a rarefaction wave propagating with speed ranging over the interval $[\lambda_i(w_{i-1}),\lambda_i(w_i)]$. In this case, let us call
$$\lambda_i^-:=\lambda_i(w_{i-1}) \; \text{ and }\; \lambda_i^+:=\lambda_i(w_i);$$
\item[(ii)] if the $i$-th characteristic field is genuinely non-linear and $\sigma_i<0$ or if it is linearly degenerate, the solution is respectively a shock or a contact discontinuity with propagation speed
$$\lambda_i(w_{i-1},w_i),$$
which is the eigenvalue of the average matrix
$$A(w_{i-1},w_i)=\int_0^1 A(\xi \, w_{i-1}+(1-\xi)\, w_i) \, d\xi.$$
By the Lax entropy condition we have
$$\lambda_i(w_{i-1},w_i)\in [\lambda_i(w_{i}),\lambda_i(w_{i-1})].$$
In this case let us define
$$\lambda_i^-=\lambda_i^+:=\lambda_i(w_{i-1},w_i).$$
\end{enumerate}
If $\sigma_1$, ..., $\sigma_n$ are small enough, by the continuity of $u\to \lambda_i(u)$ for every $i$, we have
$$\lambda_1^-\leq \lambda_1^+<\lambda_2^-\leq \lambda_2^+<...<\lambda_n^-\leq \lambda_n^+.$$
Therefore a piecewise smooth function $u:\mathbb{R}^+\times \mathbb{R}\to \mathbb{R}^n$ is well defined by the assignment (see Figure \ref{fig_example_general_solution_Riemann_problem}):
\begin{equation}\label{general_solution_Riemann_problem}
u(t,x)=\begin{cases}
u^- & \text{if } \frac{x}{t}\in (-\infty,\lambda_i^-),\\
R_i(\sigma)(w_{i-1}) & \text{if } \frac{x}{t}=\lambda_i(R_i(\sigma)(w_{i-1}))\in [\lambda_i^-,\lambda_i^+),\\
w_i & \text{if } \frac{x}{t}\in [\lambda_i^+,\lambda_{i+1}^-),\\
u^+ & \text{if } \frac{x}{t}\in [\lambda_n^+,+\infty).
\end{cases}
\end{equation}
\begin{figure}[hbtp]
\centering
\definecolor{fftttt}{rgb}{1.,0.2,0.2}
\definecolor{ttttff}{rgb}{0.2,0.2,1.}
\begin{tikzpicture}[line cap=round,line join=round,>=triangle 45,x=1.0cm,y=1.0cm]
\draw[->,color=black] (-4.5,0.) -- (5.,0.);
\draw[->,color=black] (0.,-0.4) -- (0.,6.4);
\clip(-4.5,-0.4) rectangle (5.,6.4);
\draw [line width=1.6pt,color=ttttff,domain=-4.5:0.0] plot(\x,{(-0.--5.14*\x)/-4.82});
\draw [line width=1.6pt,color=ttttff,domain=-4.5:0.0] plot(\x,{(-0.--6.62*\x)/-3.76});
\draw [line width=1.6pt,color=fftttt,domain=-4.5:0.0] plot(\x,{(-0.--6.32*\x)/-0.96});
\draw [line width=1.6pt,color=ttttff,domain=0.0:5.0] plot(\x,{(-0.--6.46*\x)/0.7});
\draw [line width=1.6pt,color=ttttff,domain=0.0:5.0] plot(\x,{(-0.--6.66*\x)/2.78});
\draw [line width=1.6pt,color=fftttt,domain=0.0:5.0] plot(\x,{(-0.--5.18*\x)/4.86});
\draw [color=ttttff,domain=-4.5:0.0] plot(\x,{(-0.--6.28*\x)/-5.3});
\draw [color=ttttff,domain=-4.5:0.0] plot(\x,{(-0.--6.28*\x)/-4.86});
\draw [color=ttttff,domain=-4.5:0.0] plot(\x,{(-0.--5.14*\x)/-3.56});
\draw [color=ttttff,domain=-4.5:0.0] plot(\x,{(-0.--7.4*\x)/-4.64});
\draw [color=ttttff,domain=0.0:5.0] plot(\x,{(-0.--7.36*\x)/1.12});
\draw [color=ttttff,domain=0.0:5.0] plot(\x,{(-0.--5.6*\x)/1.08});
\draw [color=ttttff,domain=0.0:5.0] plot(\x,{(-0.--4.48*\x)/1.04});
\draw [color=ttttff,domain=0.0:5.0] plot(\x,{(-0.--5.86*\x)/1.68});
\draw [color=ttttff,domain=0.0:5.0] plot(\x,{(-0.--5.3*\x)/1.84});
\draw (-3.82,2.18) node[anchor=north west] {$w_0=u^l$};
\draw (-1.8,4.35) node[anchor=north west] {$w_1$};
\draw (-0.55,4.63) node[anchor=north west] {$w_2$};
\draw (2.28,4.13) node[anchor=north west] {$w_3$};
\draw (3.33,1.86) node[anchor=north west] {$w_4=u^r$};
\draw (-4.4,4.17) node[anchor=north west] {$\lambda_1^-$};
\draw (-2.8,5.4) node[anchor=north west] {$\lambda_1^+$};
\draw (-1.9,5.96) node[anchor=north west] {$\lambda_2^-=\lambda_2^+$};
\draw (-0.05,5.59) node[anchor=north west] {$\lambda_3^-$};
\draw (2.1,5.3) node[anchor=north west] {$\lambda_3^+$};
\draw (3.1,5.46) node[anchor=north west] {$\lambda_4^-=\lambda_4^+$};
\draw (-0.4,6.52) node[anchor=north west] {$t$};
\draw (4.03,-0.) node[anchor=north west] {$x$};
\end{tikzpicture}
\caption{Example of solution to the Riemann problem (\ref{Riemann_problem_general}).}\label{fig_example_general_solution_Riemann_problem}
\end{figure}
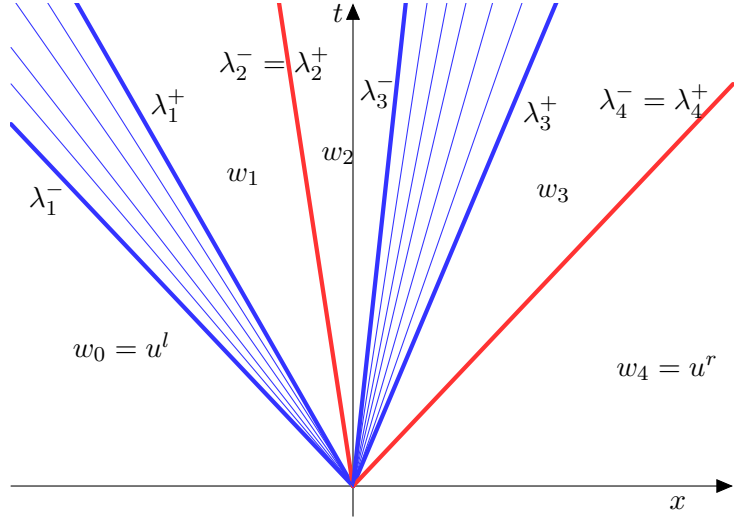

\begin{theorem}\label{theorem_general_solution_to_RP}
Let (\ref{system_of_cons_laws}) be a strictly hyperbolic system of conservation laws with smooth coefficients defined on an open set $\Omega\subseteq \mathbb{R}^n$. Suppose that for each $i \in \lbrace 1,...,n \rbrace$, the $i$-th characteristic field is linearly degenerate or genuinely non-linear.\\
Then for every compact $K \subset \Omega$, there exists $\delta>0$ such that the Riemann problem (\ref{Riemann_problem_general}) has a unique Lax-admissible solution of the form (\ref{general_solution_Riemann_problem}), whenever $u^l \in K$ and $|u^r-u^l|\leq \delta$.
\end{theorem}
\begin{proof}
By the discussion in the previous sections, the function (\ref{general_solution_Riemann_problem}) is a Lax-admissible weak solution for the Riemann problem (\ref{Riemann_problem_general}).\\
Let $u^l\in \Omega$ be fixed. Let us define the $\mathcal{C}^1$ maps
$$\Psi(\sigma_1,...,\sigma_n)(u^l)=\psi_n(\sigma_n)\circ \cdots \circ \psi_1(\sigma_1)(u^l)$$
and
$$\Phi^{u^l}((\sigma_1,...,\sigma_n),u)=\Psi(\sigma_1,...,\sigma_n)(u^l)-u.$$
We find
$$\Phi^{u^l}(0,u^l)=\Psi(0,...,0)(u^l)-u^l=0.$$
Moreover
$$\dfrac{\partial \Phi^{u^l}}{\partial \sigma_i}(0,u^l)=\dfrac{\partial \Psi}{\partial \sigma_i}(\sigma_1,...,\sigma_n)\Big|_{\sigma_1=...=\sigma_n=0}=r_i(u^l),$$
indeed, by Propositions \ref{rarefaction_wave_properties} and \ref{properties_shock_waves}, we have $\partial_{\sigma_i} \psi_i(0)(u^l)=r_i(u^l)$. Therefore
\begin{equation*}
\begin{split}
\dfrac{\partial \Psi}{\partial \sigma_i}(\sigma_1,...,\sigma_n)\Big|_{\sigma_1=...=\sigma_n=0} & = \lim_{\varepsilon \to 0} \dfrac{\Psi(0,...,0,\varepsilon,0,...,0)-\Psi(0,...,0)}{\varepsilon}=\\
&= \lim_{\varepsilon \to 0} \dfrac{\psi_i(\varepsilon)(u^l)-u^l}{\varepsilon}=\\
&= r_i(u^l).
\end{split}
\end{equation*} 
The vectors $r_1(u^l)$, ..., $r_n(u^l)$ are linearly independent, hence the rank of the Jacobian matrix $D_{(\sigma_1,...,\sigma_n)} \Phi^{u^l}$ is $n$. By Theorem \ref{implicit_function_theorem_parametrized}, for every compact set $K \subset \Omega$ there exists $\zeta >0$ and a $\mathcal{C}^1$ function
$$(\sigma_1,...,\sigma_n) \to u^r(\sigma_1,...,\sigma_n)$$
such that
$$\Phi^{u^l}((\sigma_1,...,\sigma_n), u^r(\sigma_1,...,\sigma_n))=0 \Longleftrightarrow u^r(\sigma_1,...,\sigma_n)=\Psi(\sigma_1,...,\sigma_n)(u^l)$$
for every $u^l\in K$ and $|(\sigma_1,...,\sigma_n)|\leq \zeta$. Since the function $(\sigma_1,...,\sigma_n) \to \Psi(\sigma_1,...,\sigma_n)$ is an homomorphism of a neighbourhood of $0\in\mathbb{R}^n$ onto a neighbourhood $\mathcal{N}$ of $u^-$, the condition $|(\sigma_1,...,\sigma_n)|\leq \zeta$ is equivalent to $|u^r(\sigma_1,...,\sigma_n)-u^l|\leq \delta$ for some $\delta>0$.
\end{proof}
\chapter{The Aw-Rascle-Zhang model}
In this chapter we are going to introduce the Aw-Rascle-Zhang (ARZ) model and its main properties.\\
The Aw-Rascle-Zhang model (see \cite{Aw-Rascle, zhang}) is
\begin{equation}\label{ARZ_system}
\begin{cases}
\partial_t \rho+\partial_x  (\rho\,v)=0,\\
\partial_t z+\partial_x (z\,v)=0,
\end{cases}
\end{equation}
where $\rho$ and $v$ are respectively the density and the velocity of the vehicles on the road and $z=\rho\,(v+p(\rho))$ is the generalized momentum.\\
The first equation states the conservation of the density, while the second states the conservation of the momentum. The function $p\in \mathcal{C}^2([0,+\infty),[0,+\infty))$ plays the role of the pressure. We assume that $p$ satisfies the following hypotheses:
\begin{equation}\label{ipotesi_pressione}
\begin{cases}
p(0)=0,\\
p'(\rho)>0 & \text{for every } \rho >0,\\
\rho \to \rho\, p(\rho) & \text{is strictly convex.}
\end{cases}
\end{equation} 
The pressure function describes how a typical driver reacts to a spatial variation of the concentration of cars in front of him.
The next propositions state the main properties of the ARZ system respectively in the $(\rho,z)$ and in the $(\rho,v)$ plane.
\begin{prop}\label{ARZ_properties_rho_z}
The ARZ system
$$
\begin{cases}
\partial_t \rho+\partial_x  (\rho\,v)=0,\\
\partial_t z+\partial_x (z\,v)=0,
\end{cases}
$$
in the conserved variables $(\rho,z)=(\rho,\rho\,(v+p(\rho))$ has the following properties.
\begin{enumerate}
\item The flux function is
$$f(\rho,z) = \begin{pmatrix}
f_1(\rho,z)\\
f_2(\rho,z)
\end{pmatrix}= \begin{pmatrix}
\rho\, v\\
z\,v
\end{pmatrix}= \begin{pmatrix}
z-\rho\, p(\rho)\\
\dfrac{z^2}{\rho}-z\,p(\rho)
\end{pmatrix}.$$
\item The eigenvalues of the Jacobian matrix of $f$ are
\begin{equation}\label{eigenvalues_rho_z}
\lambda_1(\rho,z)=-p(\rho)+\dfrac{z}{\rho}-\rho\,p'(\rho) \; \text{ and } \; \lambda_2(\rho,z)=-p(\rho)+\dfrac{z}{\rho}
\end{equation}
and the corresponding eigenvectors are 
\begin{equation}\label{eigenvectors_rho_z}
r_1(\rho,z)=\begin{pmatrix}
-1\\
-\frac{z}{\rho}
\end{pmatrix} \; \text{ and } \; r_2(\rho,z)=\begin{pmatrix}
1\\
\frac{z}{\rho}+\rho\,p'(\rho)
\end{pmatrix}.
\end{equation}
Moreover, the first characteristic field is genuinely non-linear and the second is linearly degenerate.
\item The first shock and rarefaction curves coincide and the Lax curves passing through a point $(\rho_0,z_0)\in \mathbb{R}^+\times \mathbb{R}^+$ are
\begin{equation}\label{Lax_curves_rho_z}
L_1(\rho,\rho_0,z_0)=\dfrac{z_0}{\rho_0}\, \rho \; \text{ and } \; L_2(\rho,\rho_0,z_0)=\dfrac{z_0}{\rho_0}\, \rho+\rho\,(p(\rho)-p(\rho_0));
\end{equation}
see Figure \ref{fig_Lax_curves_a}.
\item The Riemann invariants are
\begin{equation}\label{Riemann_invariants_rho_z}
s=\dfrac{z}{\rho}-p(\rho) \; \text{ and } \; w=\dfrac{z}{\rho}.
\end{equation}
\end{enumerate}
\end{prop}
\begin{proof}
Since $z= \rho\,(v+p(\rho))$, we find
$$v=\dfrac{z}{\rho}-p(\rho).$$
Therefore
$$\rho\, v= z-\rho\,p(\rho)\; \text{ and } \; z\, v = \dfrac{z^2}{\rho}-z\, p(\rho).$$
The Jacobian matrix for the flux function is
\begin{equation}
A(\rho,z)=Df(\rho,z)=\begin{bmatrix}
-p(\rho)-\rho\, p'(\rho) & 1\\[0.2cm]
-z\,p'(\rho)-\frac{z^2}{\rho^2} & \frac{2z}{\rho}-p(\rho)
\end{bmatrix}.
\end{equation}
Let $\mathbf{I}$ be the $2\times 2$ identity matrix. We have
\begin{equation*}
\begin{split}
\det[Df & -\lambda \mathbf{I}] =\det \begin{bmatrix}
-p(\rho)-\rho p'(\rho)-\lambda & 1\\[0.2cm]
-z\,p'(\rho)-\frac{z^2}{\rho^2} & \frac{2z}{\rho}-p(\rho)-\lambda
\end{bmatrix}=\\
&= (p(\rho)+\rho \, p'(\rho)+\lambda)\left(\lambda+p(\rho)-\frac{2z}{\rho}\right)+z\,p'(\rho)+\frac{z^2}{\rho^2}=\\
&=\lambda^2+\lambda\left(2p(\rho)+\rho\, p'(\rho)-\frac{2z}{\rho}\right)+p(\rho)\left(p(\rho)+\rho \, p'(\rho)-\frac{2z}{\rho}\right)-z\,p'(\rho)+\frac{z^2}{\rho^2}.
\end{split}
\end{equation*}
Let $\Delta$ be the discriminant for the characteristic equation
$$\det[Df -\lambda \mathbf{I}]=0.$$
We obtain $\Delta=\rho^2\, [p'(\rho)]^2$ and $\det[Df -\lambda \mathbf{I}] =0$ if and only if
$$\lambda_{1,2}=\dfrac{1}{2}\left(-2\, p(\rho)-\rho\, p'(\rho)+\frac{2z}{\rho}\pm \rho\, p'(\rho)\right)$$
which are the eigenvalues.\\
Let us denote $r_1=(v_1,v_2)^T$ the eigenvector corresponding to $\lambda_1=\frac{z}{\rho}-\rho\, p'(\rho)-p(\rho)$. We have
\begin{equation*}
\begin{split}
& Df\begin{pmatrix}
v_1\\[0.2cm]
v_2
\end{pmatrix}=\lambda_1\begin{pmatrix}
v_1\\[0.2cm]
v_2
\end{pmatrix}\Longleftrightarrow \begin{bmatrix}
-p(\rho)-\rho\, p'(\rho) & 1\\[0.2cm]
-z\,p'(\rho)-\frac{z^2}{\rho^2} & \frac{2z}{\rho}-p(\rho)
\end{bmatrix}\begin{pmatrix}
v_1\\[0.2cm]
v_2
\end{pmatrix}=\lambda_1\begin{pmatrix}
v_1\\[0.2cm]
v_2
\end{pmatrix}\Longleftrightarrow\\
& \Longleftrightarrow \begin{cases}
-p(\rho)\,v_1-\rho\, p'(\rho)\,v_1+v_2=\frac{z}{\rho}\,v_1-\rho\, p'(\rho)\,v_1-p(\rho)\,v_1,\\
-z\, p'(\rho)\,v_1- \frac{z^2}{\rho^2}\,v_1+\frac{2\,z}{\rho}\,v_2-p(\rho)\,v_2=\frac{z}{\rho}\,v_2-\rho\, p'(\rho)\,v_2-p(\rho)\,v_2.
\end{cases}\Longleftrightarrow\\
&\Longleftrightarrow \begin{cases}
v_2=\frac{z}{\rho}\,v_1,\\
-z\,p'(\rho)\,v_1-\frac{z^2}{\rho^2}\, v_1+\frac{z}{\rho}\, v_2=-\rho \, p'(\rho)\, v_2.
\end{cases}
\end{split}
\end{equation*}
These two equations are linearly dependent. Hence, choosing $v_1=-1$, we obtain
$$r_1=\begin{pmatrix}
-1\\
-\frac{z}{\rho}
\end{pmatrix}.$$
For the eigenvector $r_2=(v_1,v_2)^T$ corresponding to the eigenvalue $\lambda_2=\frac{z}{\rho}-p(\rho)$, we find
\begin{equation*}
\begin{split}
& Df\begin{pmatrix}
v_1\\[0.2cm]
v_2
\end{pmatrix}=\lambda_2\begin{pmatrix}
v_1\\[0.2cm]
v_2
\end{pmatrix}\Longleftrightarrow \begin{bmatrix}
-p(\rho)-\rho\, p'(\rho) & 1\\[0.2cm]
-z\,p'(\rho)-\frac{z^2}{\rho^2} & \frac{2\,z}{\rho}-p(\rho)
\end{bmatrix}\begin{pmatrix}
v_1\\[0.2cm]
v_2
\end{pmatrix}=\lambda_2\begin{pmatrix}
v_1\\[0.2cm]
v_2
\end{pmatrix}\Longleftrightarrow\\
& \Longleftrightarrow \begin{cases}
-p(\rho)\,v_1-\rho \, p'(\rho)\,v_1+v_2=\frac{z}{\rho}\, v_1-p(\rho)\,v_1,\\
-z\,p'(\rho)\,v_1-\frac{z^2}{\rho^2}\,v_1+\frac{2\,z}{\rho}\,v_2-p(\rho)\,v_2=\frac{z}{\rho}\, v_2-p(\rho)\,v_2.
\end{cases}\Longleftrightarrow\\
& \Longleftrightarrow \begin{cases}
v_2=(\frac{z}{\rho}+\rho \, p'(\rho))\,v_1,\\
\frac{z}{\rho}\, v_2=z\,p'(\rho)\,v_1+\frac{z^2}{\rho^2}\,v_1.
\end{cases}
\end{split}
\end{equation*}
The equations are linearly dependent. Therefore, choosing $v_1=1$, we obtain
$$r_2=\begin{pmatrix}
1\\
\frac{z}{\rho}+\rho\, p'(\rho)
\end{pmatrix}.$$
By the hypotheses (\ref{ipotesi_pressione}), we have
$$\nabla\lambda_1\cdot r_1=\begin{pmatrix}
-\frac{z}{\rho^2}-2\,p'(\rho)-\rho\, p''(\rho)\\[0.2cm]
\frac{1}{\rho}
\end{pmatrix}\cdot\begin{pmatrix}
-1\\[0.2cm]
-\frac{z}{\rho}
\end{pmatrix}=\,2p'(\rho)+\rho \, p''(\rho)=\dfrac{d^2}{\rho^2}(\rho \, p(\rho))>0$$
Therefore the first characteristic field is genuinely non-linear. For the second characteristic field, we find
$$\nabla\lambda_2\cdot r_2=\begin{pmatrix}
-\frac{z}{\rho^2}-p'(\rho)\\
\frac{1}{\rho}
\end{pmatrix}\cdot \begin{pmatrix}
1\\
\frac{z}{\rho}+\rho\,p'(\rho)
\end{pmatrix}=0.$$
Hence it is linearly degenerate.\\
Let us call $u=(\rho,z)$ the vector of the conserved variables.\\
Consider first the second characteristic field. Since it is linearly degenerate, the second rarefaction and shock curves coincide. Hence the second Lax curve consists of the points which satisfy the Rankine-Hugoniot conditions with propagation speed $\lambda_2$, i.e.
\begin{equation*}
\begin{split}
&f(\rho,z)-f(\rho_0,z_0) = \lambda_2(\rho,z)\left[(\rho,z)-(\rho_0,z_0)\right] \Longleftrightarrow\\
&\Longleftrightarrow \begin{cases}
z-\rho \,p(\rho) -z_0+\rho_0\,p(\rho_0)=(-p(\rho)+\frac{z}{\rho})(\rho-\rho_0),\\
\frac{z^2}{\rho}-z\,p(\rho)-\frac{z_0^2}{\rho_0}+z_0\,p(\rho_0)=(-p(\rho)+\frac{z}{\rho})(z-z_0).
\end{cases}
\end{split}
\end{equation*}
These two equations are linearly dependent and the solutions are the points $(\rho,z)$ such that
$$z=L_2(\rho,\rho_0,z_0)=\dfrac{z_0}{\rho_0}\,\rho+\rho\,(p(\rho)-p(\rho_0)).$$
Observe that a point $(\rho,z)$ connected to $(\rho_0,z_0)$ by a Lax curve of the second family satisfies
\begin{equation}\label{caracterization_contact_discontinuities_Riemann_invariant}
\dfrac{z}{\rho}-\rho\,p(\rho)=\dfrac{z_0}{\rho_0}-\rho_0\,p(\rho_0)=\lambda_2(\rho,z).
\end{equation}
The first characteristic field is genuinely non-linear. By Definition \ref{definition_rarefaction_wave} the first rarefaction curve passing through the point $(\rho_0,z_0)$ is the solution to the Cauchy problem
\begin{equation*}
\begin{cases}
\dot{u}=r_1,\\
u(0)=(\rho_0,z_0).
\end{cases}
\end{equation*}
Then
\begin{equation*}
\begin{cases}
\dot{\rho}=-1,\\
\dot{z}=-z/\rho.
\end{cases}\Longrightarrow \begin{cases}
\rho(\tau)=\rho_0-\tau,\\
\dot{z}(t)=z/(\tau-\rho_0).
\end{cases}\Longrightarrow
\begin{cases}
\tau=\rho_0-\rho,\\
z(\tau)=z_0\left(1-\tau/\rho_0\right).
\end{cases}
\end{equation*}
Substituting the first equation in the second, we find
$$L_1(\rho,\rho_0,z_0)=\dfrac{z_0}{\rho_0}\,\rho.$$
The shock curve of the first family is formed by the points $(\rho,z)$ which are the solutions to the system given by the Rankine-Hugoniot conditions:
\begin{equation*}
\begin{cases}
z-\rho\,p(\rho)-z_0+\rho_0\,p(\rho_0)=\lambda\,(\rho-\rho_0),\\
\frac{z^2}{\rho}-z\,p(\rho)-\frac{z_0^2}{\rho_0}+z_0\,p(\rho_0)=\lambda\,(z-z_0).
\end{cases}
\end{equation*} 
The first equation gives
$$Q:=z-\rho\,(p(\rho)+\lambda)=z_0-\rho_0\,(p(\rho_0)+\lambda).$$
We have to distinguish two cases:
\begin{itemize}
\item if $Q=0$, then
$$\lambda=\dfrac{z}{\rho}-p(\rho)=\dfrac{z_0}{\rho_0}-p(\rho_0).$$
By the equation (\ref{caracterization_contact_discontinuities_Riemann_invariant}), these are the points $(\rho,z)$ joined to $(\rho_0,z_0)$ by a contact discontinuity of the second family.
\item if $Q \neq 0$, then by the second equation we find
\begin{equation*}
\begin{split}
& \dfrac{z}{\rho}(z-\rho\,(p(\rho)+\lambda))=\dfrac{z_0}{\rho_0}\,(z_0-\rho_0\,(p(\rho_0)+\lambda)) \Longrightarrow \\
& \dfrac{z}{\rho}\, Q = \dfrac{z_0}{\rho_0}\,Q \Longrightarrow z=\dfrac{z_0}{\rho_0}\,\rho=L_1(\rho,\rho_0,z_0).
\end{split}
\end{equation*}
\end{itemize}
Hence the first rarefaction and shock curves coincide.\\
The Riemann invariants are the curves that are constant along the characteristic fields, i.e.
$$\nabla w\cdot r_1=0 \; \text{ and } \nabla s \cdot r_2 =0.$$
By the first equation, we find
$$-\partial_\rho  w-\dfrac{z}{\rho}\,\partial_z w=0 \Longrightarrow \rho \, \partial_\rho w+z \, \partial_z \, w=0 \Longrightarrow w=\dfrac{z}{\rho},$$
while the second equation implies
$$\partial_\rho s+\dfrac{z}{\rho}\, \partial_z s+\rho\,p'(\rho)\partial_z s=0\Longrightarrow s=\dfrac{z}{\rho}-p(\rho).$$
This completes the proof.
\end{proof}
\begin{lemma}\label{lemma_ARZ_equivalent_spectral}
Let us suppose that $\rho$ and $v$ are smooth. The ARZ system
\begin{equation}\label{ARZ_system_explicit}
\begin{cases}
\partial_t \rho +\partial_x (\rho \,v)=0,\\
\partial_t [\rho\,(v+p(\rho))]+\partial_x [\rho\,v\,(v+p(\rho))]=0,
\end{cases}
\end{equation}
is equivalent to the system
\begin{equation}\label{ARZ_spectral_properties}
\begin{cases}
\partial_t  \rho + \partial_x(\rho\,v)=0,\\
\partial_t v + (v-\rho\,p'(\rho))\, \partial_x v =0.
\end{cases}
\end{equation}
\end{lemma}
\begin{proof}
Since $\rho$ and $v$ are smooth, the second equation of (\ref{ARZ_system_explicit}) is equivalent to
\begin{equation}\label{conservation_momentum_equivalent_form}
\partial_t (v+p(\rho))+v\,\partial_x (v+p(\rho))=0,
\end{equation}
indeed
\begin{equation*}
\begin{split}
& (v+p(\rho)) \, \partial_t \rho +\rho\,\partial_t (v+p(\rho))+(v+p(\rho))\,\partial_x(\rho\,v)+\rho\,v \, \partial_x (v+p(\rho))=0 \Longleftrightarrow \\
& (v+p(\rho))\,[\partial_t \rho +\partial_x (\rho\,v)]+\rho\,[\partial_t (v+p(\rho))+v\,\partial_x (v+p(\rho))]=0 \Longleftrightarrow \\
& \partial_t (v+p(\rho))+v\,\partial_x (v+p(\rho))=0,
\end{split}
\end{equation*}
because, by the conservation of the density, we have
\begin{equation}\label{conservation_density}
\partial_t \rho +\partial_x (\rho\,v)=0.
\end{equation}
Let us multiply the factor $-p'(\rho)$ to the equation (\ref{conservation_density}) and add the result to the equation (\ref{conservation_momentum_equivalent_form}). Since
$$\partial_t (p(\rho))=p'(\rho)\,\partial_t \rho \; \text{ and } \; \partial_x (p(\rho))=p'(\rho)\,\partial_x \rho,$$
we obtain
\begin{equation*}
\begin{split}
& \partial_t (v+p(\rho))+v\,\partial_x (v+p(\rho)) - p'(\rho)[\partial_t \rho +\partial_x (\rho\,v)]=0 \Longleftrightarrow\\
& \partial_t v +p'(\rho)\,\partial_t \rho +v \,(\partial_x v +p'(\rho)\,\partial_x \rho)-p'(\rho)\,\partial_t \rho-p'(\rho)\,(v \, \partial_x \rho+\rho\,\partial_x v)=0 \Longleftrightarrow \\
& \partial_t v + (v-\rho\,p'(\rho))\,\partial_x v=0,
\end{split}
\end{equation*}
which is the second equation in (\ref{ARZ_spectral_properties}).
\end{proof}
\begin{prop}\label{ARZ_properties_rho_v}
The ARZ system
$$
\begin{cases}
\partial_t \rho +\partial_x (\rho \,v)=0,\\
\partial_t [\rho\,(v+p(\rho))]+\partial_x [\rho\,v\,(v+p(\rho))]=0,
\end{cases}
$$
in the non-conserved variables $(\rho,v)$ has the following properties.
\begin{enumerate}
\item The representation of the flux function in the variables $(\rho,v)$ is
$$f(\rho,v)=\begin{pmatrix}
f_1(\rho,v)\\
f_2(\rho,v)
\end{pmatrix}=\begin{pmatrix}
\rho\,v\\
\rho\,v \, (v+p(\rho))
\end{pmatrix}$$
\item The eigenvalues of the Jacobian matrix of the flux function are
\begin{equation}\label{eigenvalues_rho_v}
\lambda_1(\rho,v)=v-\rho\,p'(\rho) \; \text{ and }\; \lambda_2(\rho,v)=v
\end{equation}
and the corresponding eigenvectors are
\begin{equation}\label{eigenvectors_rho_v}
r_1(\rho,v)=\begin{pmatrix}
-1\\
p'(\rho)
\end{pmatrix} \; \text{ and } \; r_2(\rho,v)=\begin{pmatrix}
1\\
0
\end{pmatrix}.
\end{equation}
Moreover, the first characteristic field is genuinely non-linear, while the second is linearly degenerate.
\item The first rarefaction and shock curves coincide and the Lax curves passing through a point $(\rho_0,v_0)\in \mathbb{R}^+\times \mathbb{R}^+$ are
\begin{equation}\label{Lax_curves_rho_v}
L_1(\rho,\rho_0,v_0)=v_0+p(\rho_0)-p(\rho) \; \text{ and } \; L_2(\rho.\rho_0,v_0)=v_0;
\end{equation}
see Figure \ref{fig_Lax_curves_b}.
\item The Riemann invariants are 
$$s=v\; \text{ and }\; w=v+p(\rho).$$
\end{enumerate}
\end{prop}
\begin{proof}
By Lemma (\ref{lemma_ARZ_equivalent_spectral}), the ARZ system is equivalent to the system
\begin{equation*}
\begin{cases}
\partial_t \rho + v\, \partial_x \rho+\rho\,\partial_x v=0,\\
\partial_t v + (v-\rho\,p'(\rho))\, \partial_x v =0,
\end{cases}
\end{equation*}
which can be written in the form
$$\begin{pmatrix}
\partial_t \rho\\
\partial_t v
\end{pmatrix}+
\begin{bmatrix}
v & \rho \\
0 & v-\rho\, p'(\rho)
\end{bmatrix}\, \begin{pmatrix}
\partial_x \rho\\
\partial_x v
\end{pmatrix}=0.$$
For this system, we have
$$Df(\rho,v)=\begin{bmatrix}
v & \rho\\
0 & v-\rho\,p'(\rho)
\end{bmatrix}.$$
Hence the eigenvalues are
$$\lambda_1(\rho,v)=v-\rho\,p'(\rho) \; \text{ and } \; \lambda_2(\rho,v)=v.$$
If $r_1=(\mu_1,\mu_2)^T$ is the eigenvector corresponding to $\lambda_1(\rho,v)$, we find
$$
\begin{cases}
v\,\mu_1+\rho\, \mu_2=(v-\rho\,p'(\rho))\,\mu_1.\\
0+(v-\rho\,p'(\rho))\,\mu_2=(v-\rho\,p'(\rho))\,\mu_2.
\end{cases}
$$ 
The second equation is trivial, while by the first equation we obtain
$$\mu_2= -p'(\rho)\mu_1 \Longrightarrow r_1(\rho,v)= \begin{pmatrix}
-1\\
p'(\rho)
\end{pmatrix}.$$
Similarly, if $r_2= (\mu_1,\mu_2)^T$ is the eigenvector corresponding to $\lambda_2(\rho,v)$, we find
$$ \begin{cases}
v\,\mu_1+\rho\, \mu_2=v\,\mu_1.\\
0+(v-\rho\,p'(\rho))\,\mu_2=v\,\mu_2.
\end{cases}
$$
The first equation gives
$$\mu_2=0,$$
then the second equation is satisfied and
$$r_2(\rho,v)= \begin{pmatrix}
1\\
0
\end{pmatrix}.$$
The shock and rarefaction curves of the first family coincide by Proposition \ref{ARZ_properties_rho_z}. Hence the Lax curve of the first family passing through the point $(\rho_0,v_0)$ is the solution to the Cauchy problem
$$
\begin{cases}
\dot{\rho}= -1,\\
\dot{v}= p'(\rho),\\
(\rho,v)(0)=(\rho_0,v_0).
\end{cases}
$$
By the first equation and the initial condition, we find
$$\rho(\tau) = \rho_0-\tau.$$
Substituting the result in the second equation, we find:
$$\dot{v}(\tau)=\dfrac{d}{d\rho}p(\rho(\tau))=\dfrac{d p(\rho(\tau))}{d\tau}\dfrac{d \tau}{d\rho}=-\dot{p}(\rho(\tau)).$$
Therefore
$$v(\tau)=c-p(\rho(\tau)).$$
By the initial condition, we obtain
$$v_0 = c-p(\rho_0) \Longrightarrow L_1(\rho,\rho_0,v_0) = v_0+p(\rho_0)-p(\rho).$$
For the Lax curve of the second family, we have:
$$
\begin{cases}
\dot{\rho}=1,\\
\dot{v}=0,\\
(\rho,v)(0)=(\rho_0,v_0).
\end{cases} \Longleftrightarrow \begin{cases}
\rho(\tau) = \rho_0-\tau,\\
v(\tau)=v_0.
\end{cases}
$$
Hence $L_2(\rho,\rho_0,v_0)=v_0$.\\
Substituting $z=\rho\,(v+p(\rho))$ in the expressions (\ref{Riemann_invariants_rho_z}), we find
$$s= \dfrac{z}{\rho}-p(\rho) = v \; \text{ and } \;  w = \dfrac{z}{\rho}= v+p(\rho).$$
\end{proof}
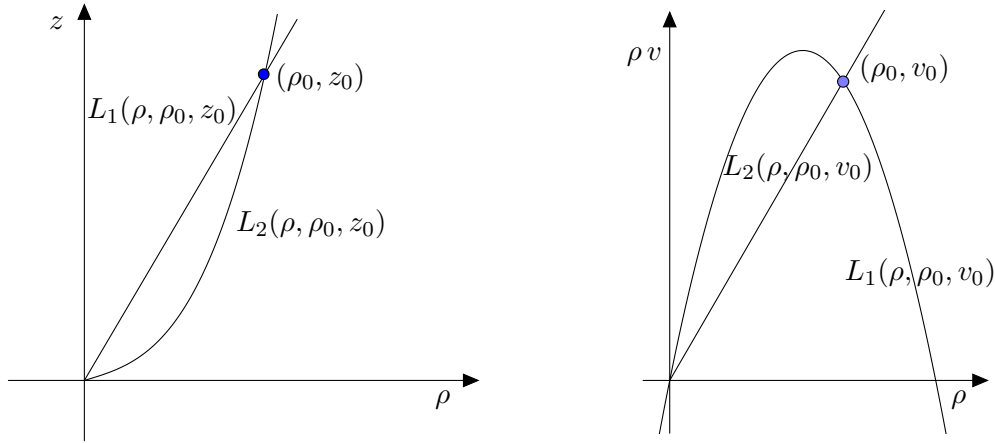
\begin{figure}[hbtp]
\centering
\begin{subfigure}[hbtp]{0.47\linewidth}
\centering
\definecolor{qqqqff}{rgb}{0.,0.,1.}
\begin{tikzpicture}[scale = 2., line cap=round,line join=round,>=triangle 45,x=1.0cm,y=1.0cm]
\draw[->,color=black] (-0.5,0.) -- (2.6,0.);
\draw[->,color=black] (0.,-0.4) -- (0.,2.5);
\clip(-0.5,-0.4) rectangle (2.6,2.5);
\draw (1.19,2.17) node[anchor=north west] {$(\rho_0,z_0)$};
\draw (-0.3,2.48) node[anchor=north west] {$z$};
\draw (2.24,0) node[anchor=north west] {$\rho$};
\draw (-0.05,1.95) node[anchor=north west] {$L_1(\rho,\rho_0,z_0)$};
\draw (0.93,1.2) node[anchor=north west] {$L_2(\rho,\rho_0,z_0)$};
\draw[smooth,samples=50,domain=0:1.27] plot(\x,{0.3*(\x)+(\x)^(3.0)});
\draw [domain=0:1.4] plot(\x,{(-0.--1.71*\x)/1.});
\begin{scriptsize}
\draw [fill=qqqqff] (1.18,2.03) circle (1.pt);
\end{scriptsize}
\end{tikzpicture}
\caption{Lax curves passing through the point $(\rho_0,z_0)$ in the $(\rho,z)$ plane.}\label{fig_Lax_curves_a}
\end{subfigure}
\quad
\begin{subfigure}[hbtp]{0.47\linewidth}
\centering
\definecolor{xdxdff}{rgb}{0.49,0.49,1.}
\begin{tikzpicture}[scale = 0.7,line cap=round,line join=round,>=triangle 45,x=1.0cm,y=1.0cm]
\draw[->,color=black] (-0.5,0.) -- (6.,0.);
\draw[->,color=black] (0.,-1.) -- (0.,7.);
\clip(-1.3,-1.) rectangle (6.5,7.);
\draw (3.37,6.38) node[anchor=north west] {$(\rho_0,v_0)$};
\draw (-1,6.5) node[anchor=north west] {$\rho\,v$};
\draw (5.1,0.) node[anchor=north west] {$\rho$};
\draw (3.1,2.5) node[anchor=north west] {$L_1(\rho,\rho_0,v_0)$};
\draw (0.8,4.5) node[anchor=north west] {$L_2(\rho,\rho_0,v_0)$};
\draw[smooth,samples=50,domain=-0.5:6.0] plot(\x,{5.0*(\x)-(\x)^(2.0)});
\draw [domain=0.0:6.0] plot(\x,{(-0.--5.66*\x)/3.26});
\begin{scriptsize}
\draw [fill=xdxdff] (3.26,5.66) circle (3pt);
\end{scriptsize}
\end{tikzpicture}
\caption{Lax curves passing through the point $(\rho_0,v_0)$ in the $(\rho,\rho\, v)$ plane.}\label{fig_Lax_curves_b}
\end{subfigure}
\caption{Representation of the Lax curves.}\label{fig_Lax_curves}
\end{figure}
\begin{remark}
As we have said, the conserved variables are the density $\rho$ and the generalized momentum $z$. Since $z$ has not a clear physical interpretation, whenever it is possible, we choose to work in the more significantly variables $(\rho,v)$.
\end{remark}
\section{The Riemann problem}
Let $(\rho^l,v^l)$ and $(\rho^r,v^r)$ be two points in $\mathbb{R}^+\times \mathbb{R}^+$. The Riemann problem for the ARZ model is
\begin{equation}\label{Riemann_problem_ARZ}
\begin{cases}
\partial_t\rho+\partial_x(\rho v)=0,\\
\partial_t[\rho\, (v+p(\rho))]+\partial_x[\rho \, v \, (v+p(\rho))]=0,\\
(\rho,v)(0,x)=\begin{cases}
(\rho^l,v^l) & \text{if } x\leq 0,\\
(\rho^r,v^r) & \text{if } x>0.
\end{cases}
\end{cases}
\end{equation}

\begin{prop}\label{eigenvalue_as_slope}
Fix a point $(\rho_0,v_0)\in\mathbb{R}^+\times \mathbb{R}^+$ and let $(\rho_1,v_1)$ be a point of the Lax curve of the first family passing through $(\rho_0,v_0)$, i.e.
$$v_1+p(\rho_1)=v_0+p(\rho_0).$$
The eigenvalue $\lambda_1(\rho_1,v_1)$ is the slope of the function $\rho \to \rho\, L_1(\rho,\rho_0,v_0)$ in the point $(\rho_1,v_1)$, i.e.
\begin{equation}
\lambda_1(\rho_1,v_1)=\dfrac{d}{d\rho}(\rho\, L_1(\rho,\rho_0,v_0))\Big|_{\rho=\rho_1}.
\end{equation}
Moreover the function $\rho \to \lambda_1(\rho,L_1(\rho,\rho^l,v^l))$ is strictly decreasing.
\end{prop}
\begin{proof}
We find:
\begin{equation*}
\begin{split}
\dfrac{d}{d\rho}(\rho \, L_1(\rho,\rho_0,v_0))\Big|_{\rho=\rho_1} & = v_0+p(\rho_0)-p(\rho_1)-\rho_1\,p'(\rho_1)=L_1(\rho_1,\rho_0,v_0)-\rho_1\,p'(\rho_1)\\
& =v_1-\rho_1\,p'(\rho_1)=\lambda_1(\rho_1,v_1).
\end{split}
\end{equation*}
By the hypotheses (\ref{ipotesi_pressione}), the function $\rho \to \rho\,p(\rho)$ is strictly convex, hence the function
$$\rho \to \rho\, L_1(\rho,\rho^l,v^l)=\rho\,(v^l+p(\rho)^l-p(\rho))$$
is strictly concave. Therefore its derivative is strictly decreasing.
\end{proof}
The next proposition gives a characterization of (Lax admissible) rarefaction waves, shock waves and contact discontinuities for the ARZ system.
\begin{prop}\label{prop_waves_conditions}
Fix $(\rho^l,v^l)$ and $(\rho^r,v^r)$ in $\mathbb{R}^+\times \mathbb{R}^+$ such that $\rho^l\neq 0$ and $\rho^r\neq 0$. Let us consider the Riemann problem (\ref{Riemann_problem_ARZ}). The following statements hold.
\begin{itemize}
\item[(i)] The points $(\rho^l,v^l)$ and $(\rho^r,v^r)$ are joined by a rarefaction wave if and only if
\begin{equation}\label{rarefaction_wave_conditions}
v^r+p(\rho^r)=v^l+p(\rho^l) \; \text{ and } \; \rho^l\geq \rho^r.
\end{equation}
\item[(ii)] The points $(\rho^l,v^l)$ and $(\rho^r,v^r)$ are joined by a shock wave if and only if
\begin{equation}\label{shock_wave_conditions}
v^r+p(\rho^r)=v^l+p(\rho^l) \; \text{ and } \; \rho^r>\rho^l.
\end{equation}
Moreover the shock speed is
\begin{equation}\label{shock_speed}
\lambda = \dfrac{\rho^r\,v^r-\rho^l\,v^l}{\rho^r-\rho^l}.
\end{equation}
\item[(iii)] The points $(\rho^l,v^l)$ and $(\rho^r,v^r)$ are joined by a contact discontinuity if and only if
\begin{equation}\label{contact_discontinuity_conditions}
v^r=v^l.
\end{equation}
\end{itemize}
\end{prop}
\begin{proof}
Since the first characteristic field is genuinely non-linear and the first shock and rarefaction curves coincide, the solution to the Riemann problem is a shock or a rarefaction wave if and only if
$$v^r=L_1(\rho^r,\rho^l,v^l) \Longleftrightarrow v^r+p(\rho^r)=v^l+p(\rho^l).$$
On the contrary, since the second characteristic field is linearly degenerate, the solution to (\ref{Riemann_problem_ARZ}) is a contact discontinuity if and only if
$$v^r=L_2(\rho^r,\rho^l,v^l) \Longleftrightarrow v^r=v^l.$$
Let us consider the case $v^r+p(\rho^r)=v^l+p(\rho^l)$.\\
By Proposition \ref{rarefaction_wave_properties}, for a rarefaction wave we have
$$\lambda_1(\rho^l,v^l)\leq \lambda_1(\rho^r,v^r).$$
By Proposition \ref{eigenvalue_as_slope}, the condition
\begin{equation*}
\lambda_1(\rho^l,v^l)=\lambda_1(\rho^l,L_1(\rho^l,\rho^l,v^l))\leq \lambda_1(\rho^r,L_1(\rho^r,\rho^l,v^l))= \lambda_1(\rho^r,v^r)
\end{equation*}
holds if and only if
\begin{equation*}
\rho^l \geq \rho^r.
\end{equation*}
Let us now consider a shock wave with propagation speed $\lambda$, i.e.
\begin{equation}
(\rho,v)(t,x)=\begin{cases}
(\rho^l,v^l) & \text{if } x\leq \lambda\,t,\\
(\rho^r,v^r) & \text{if } x> \lambda\, t.
\end{cases}
\end{equation}
The speed $\lambda$ must satisfy the Rankine-Hugoniot condition, i.e.
\begin{equation}\label{Rankine_Hugoniot_for_ARZ}
\begin{cases}
\rho^r\,v^r-\rho^l\,v^l=\lambda\,(\rho^r-\rho^l),\\
\rho^r\,v^r \,(v^r+p(\rho^r))-\rho^l\,v^l\,(v^l+p(\rho^l))=\lambda\,[\rho^r\,(v^r+p(\rho^r))-\rho^l\,(v^l+p(\rho^l))].
\end{cases}
\end{equation}
Since $v^r+p(\rho^r)=v^l+p(\rho^l)$, the equations are linearly dependent and solving the first equation w.r.t. $\lambda$, we find
$$\lambda = \dfrac{\rho^r\,v^r-\rho^l\,v^l}{\rho^r-\rho^l}.$$
By the Lax entropy condition (\ref{Lax_entropy_condition}), the shock waves of the first family are Lax-admissible if and only if
$$\lambda_1(\rho^r,v^r)\leq \lambda \leq \lambda_1(\rho^l,v^l),$$
which implies
$$\lambda_1(\rho^r,v^r) \leq \lambda_1(\rho^l,v^l) \Longleftrightarrow \rho^r \geq\rho^l,$$
because the function $\rho \to \lambda_1(\rho,L_1(\rho,\rho^l,v^l))$ is strictly decreasing by Proposition \ref{eigenvalue_as_slope}.
\end{proof}
\begin{remark}\label{speed_of_the_waves}
By Definition \ref{definition_rarefaction_wave}, the propagation speed of a rarefaction wave varies between $\lambda_1(\rho^l,v^l)$ and $\lambda_1(\rho^r,v^r)$. The speed of a contact discontinuity is $\lambda_2(\rho^l,v^l)=\lambda_2(\rho^r,v^r)=v^r$. Finally for a shock wave the speed is given by the Rankine-Hugoniot conditions and by Proposition \ref{shock_wave_conditions}, we have
$$\lambda = \dfrac{\rho^r\,v^r-\rho^l\,v^l}{\rho^r-\rho^l}.$$
Geometrically:
\begin{itemize}
\item by Proposition \ref{eigenvalue_as_slope} the propagation speed of a rarefaction wave in a point $(\rho^\sigma,v^\sigma)$ of the rarefaction in the $(\rho,\rho\,v)$ plane is the slope of the tangent line to the function $\rho \to \rho\, L_1(\rho,\rho^l,v^l)$ in $(\rho^\sigma,v^\sigma)$;
\item the speed of a shock is the slope of the line passing through $(\rho^l,v^l)$ and $(\rho^r,v^r)$ in the $(\rho,\rho\,v)$ plane;
\item the speed of a contact discontinuity is the slope of the line passing through the origin and the point $(\rho^r,v^r)$ in the $(\rho,\rho\,v)$ plane.
\end{itemize}
\end{remark}
\begin{mydef}\label{def_riemann_solver}
A Riemann solver for the system (\ref{ARZ_system}) is a map
$$\mathcal{RS}:(\mathbb{R}^+\times\mathbb{R}^+)^2 \to L^1(\mathbb{R},\mathbb{R}^+\times\mathbb{R}^+)$$
that for every couple $((\rho^l,v^l),(\rho^r,v^r))\in(\mathbb{R}^+\times\mathbb{R}^+)^2$ gives a solution to (\ref{Riemann_problem_ARZ}): 
\begin{equation*}
\mathcal{RS}((\rho^l,v^l),(\rho^r,v^r))(\cdot): \; \mathbb{R}\to \mathbb{R}^+\times\mathbb{R}^+,
\end{equation*}
\begin{equation*}
\lambda\to (\rho,v)(\lambda).
\end{equation*} 
\end{mydef}
\begin{remark}
By the self-similarity of the solution to a Riemann problem, setting its value for some $\lambda\in\mathbb{R}$, we give the solution for every point $(t,x)$ such that $x/t=\lambda$.
\end{remark}
The next proposition gives the standard solution for a Riemann problem with a general initial datum.
\begin{theorem}\label{prop_standard_solution}
Fix $(\rho^l,v^l)$ and $(\rho^r,v^r)$ in $\mathbb{R}^+\times \mathbb{R}^+$ such that $\rho^l\neq 0$ and $\rho^r\neq 0$. Let us consider the Riemann problem (\ref{Riemann_problem_ARZ}). Let us define the set
$$L=\lbrace \rho \in \mathbb{R}^+:v^r=L_1(\rho,\rho^l,v^l)\rbrace.$$
Let $(\rho^m,v^m)\in \mathbb{R}^+\times \mathbb{R}^+$  be the point such that
\begin{equation}\label{middle_state}
\rho^m = \max L \; \text{ and } \; v^m=v^r.
\end{equation}
The standard solution to the Riemann problem (\ref{Riemann_problem_ARZ}) is given by a rarefaction or shock wave of the first family joining $(\rho^l,v^l)$ to $(\rho^m,v^m)$, followed by a contact discontinuity of the second family which connects $(\rho^m,v^m)$ to $(\rho^r,v^r)$; see Figures \ref{fig_example_standard_solution_1} and \ref{fig_example_standard_solution_2}.\\
Therefore the standard Riemann solver for the Riemann problem (\ref{Riemann_problem_ARZ}) is:
\begin{itemize}
\item[(i)] if $\rho^l\geq \rho^m$, then
\begin{equation}\label{ARZ_classic_solution_1}
\mathcal{RS}((\rho^l,v^l),(\rho^r,v^r))\left(\dfrac{x}{t}\right)= \begin{cases}
(\rho^l,v^l) & \text{if } \frac{x}{t}< \lambda_1(\rho^l,v^l),\\
(\rho^\sigma,v^\sigma) & \text{if } \frac{x}{t} = \lambda_1(\rho^\sigma,v^\sigma) \; \text{ for } \; \sigma\in [0,1],\\
(\rho^m,v^m) & \text{if } \lambda_1(\rho^m,v^m) < \frac{x}{t} \leq v^r,\\
(\rho^r,v^r) & \text{if } \frac{x}{t}>v^r,
\end{cases}
\end{equation}
where $(\rho^\sigma,v^\sigma)$ is a point of the rarefaction joining $(\rho^l,v^l)$ to $(\rho^m,v^m)$, i.e.
$$\rho^\sigma \in [\rho^m,\rho^l] \; \text{ and } \; v^\sigma=L_1(\rho^\sigma,\rho^l,v^l);$$
\item[(ii)] if $\rho^l <\rho^m$, then
\begin{equation}\label{ARZ_classic_solution_2}
\mathcal{RS}((\rho^l,v^l),(\rho^r,v^r))\left(\dfrac{x}{t}\right)= \begin{cases}
(\rho^l,v^l) & \text{if } \frac{x}{t}\leq \lambda,\\
(\rho^m,v^m) & \text{if } \lambda <\frac{x}{t}\leq v^r,\\
(\rho^r,v^r) & \text{if } \frac{x}{t} >v^r,
\end{cases}
\end{equation}
where $\lambda$ is the propagation speed of the shock joining $(\rho^l,v^l)$ to $(\rho^m,v^m)$, i.e.
$$\lambda = \dfrac{\rho^m\,v^m-\rho^l\,v^l}{\rho^m-\rho^l}.$$
\end{itemize}
\end{theorem}
\begin{proof}
Theorem \ref{theorem_general_solution_to_RP} ensures that the solution to the Riemann problem (\ref{Riemann_problem_ARZ}) has the form (\ref{ARZ_classic_solution_1}) or (\ref{ARZ_classic_solution_2}).\\
By definition $v^m+p(\rho^m)=v^l+p(\rho^l)$. Then by Proposition \ref{prop_waves_conditions}, a rarefaction or a shock wave joins $(\rho^l,v^l)$ to $(\rho^m,v^m)$. Similarly, since $v^m=v^r$, a contact discontinuity connects $(\rho^m,v^m)$ to $(\rho^r,v^r)$.
\end{proof}
\begin{figure}[hbtp]
\centering
\begin{subfigure}[hbtp]{0.48\linewidth}
\centering
\definecolor{xdxdff}{rgb}{0.49019607843137253,0.49019607843137253,1.}
\begin{tikzpicture}[scale = 0.9, line cap=round,line join=round,>=triangle 45,x=1.0cm,y=0.7cm]
\draw[->,color=black] (-0.5,0.) -- (5.5,0.);
\draw[->,color=black] (0.,-0.7) -- (0.,7.);
\clip(-1,-0.7) rectangle (6.5,7.);
\draw[smooth,samples=50,domain=-0.5:5.5] plot(\x,{5.0*(\x)-(\x)^(2.0)});
\draw [domain=0.0:5.5] plot(\x,{(-0.--4.48*\x)/3.82});
\draw (-0.8,6.5) node[anchor=north west] {$\rho\,v$};
\draw (5,0.) node[anchor=north west] {$\rho$};
\draw (4.7,2.31) node[anchor=north west] {$(\rho^l,v^l)$};
\draw (2.01,2.64) node[anchor=north west] {$(\rho^r,v^r)$};
\draw (3.9,4.9) node[anchor=north west] {$(\rho^m,v^m)$};
\draw (1,7.14) node[anchor=north west] {$L_1(\rho,\rho^l,v^l)$};
\draw (4.1,6.65) node[anchor=north west] {$L_2(\rho,\rho^r,v^r)$};
\begin{scriptsize}
\draw [fill=xdxdff] (3.82,4.49) circle (1.5pt);
\draw [fill=xdxdff] (4.63,1.7) circle (1.5pt);
\draw [fill=xdxdff] (1.93,2.27) circle (1.5pt);
\end{scriptsize}
\end{tikzpicture}
\end{subfigure}
\quad
\begin{subfigure}[hbtp]{0.48\linewidth}
\centering
\begin{tikzpicture}[scale = 0.4,line cap=round,line join=round,>=triangle 45,x=1.0cm,y=1.0cm]
\draw[->,color=black] (-7.5,0.) -- (8.,0.);
\draw[->,color=black] (0.,-0.5) -- (0.,11.5);
\clip(-7.5,-1) rectangle (8.,11.5);
\draw [domain=-7.5:0.0] plot(\x,{(-0.--7.3*\x)/-2.58});
\draw [domain=0.0:8.0] plot(\x,{(-0.--7.28*\x)/0.74});
\draw [domain=-7.5:0.0] plot(\x,{(-0.--7.24*\x)/-1.86});
\draw [domain=-7.5:0.0] plot(\x,{(-0.--7.2*\x)/-1.36});
\draw [domain=-7.5:0.0] plot(\x,{(-0.--7.2*\x)/-0.92});
\draw [domain=-7.5:0.0] plot(\x,{(-0.--7.16*\x)/-0.42});
\draw [domain=0.0:8.0] plot(\x,{(-0.--7.22*\x)/0.32});
\draw [domain=0.0:8.0] plot(\x,{(-0.--5.5*\x)/3.52});
\draw (-5.236,5) node[anchor=north west] {$(\rho^l,v^l)$};
\draw (0.4,8.154) node[anchor=north west] {$(\rho^m,v^m)$};
\draw (3,3) node[anchor=north west] {$(\rho^r,v^r)$};
\draw (6.558,0.) node[anchor=north west] {$x$};
\draw (-0.9,11.) node[anchor=north west] {$t$};
\draw (-7,9) node[anchor=north west] {$\lambda_1(\rho^l,v^l)$};
\draw (0.7,10.794) node[anchor=north west] {$\lambda_1(\rho^m,v^m)$};
\draw (3,6) node[anchor=north west] {$\lambda_2(\rho^r,v^r)$};
\end{tikzpicture}
\end{subfigure}
\caption{Example of classic solution to the Riemann problem (\ref{Riemann_problem_ARZ}). In the represented case $(\rho^l,v^l)$ and $(\rho^m,v^m)$ are joined by a rarefaction wave with propagations speed between $\lambda_1(\rho^l,v^l)$ (negative) and $\lambda_1(\rho^m,v^m)$ (positive). A contact discontinuity with speed $\lambda_2(\rho^r,v^r)=\lambda_2(\rho^m,v^m)=v^r$ connects $(\rho^m,v^m)$ to $(\rho^r,v^r)$.}\label{fig_example_standard_solution_1}
\end{figure}
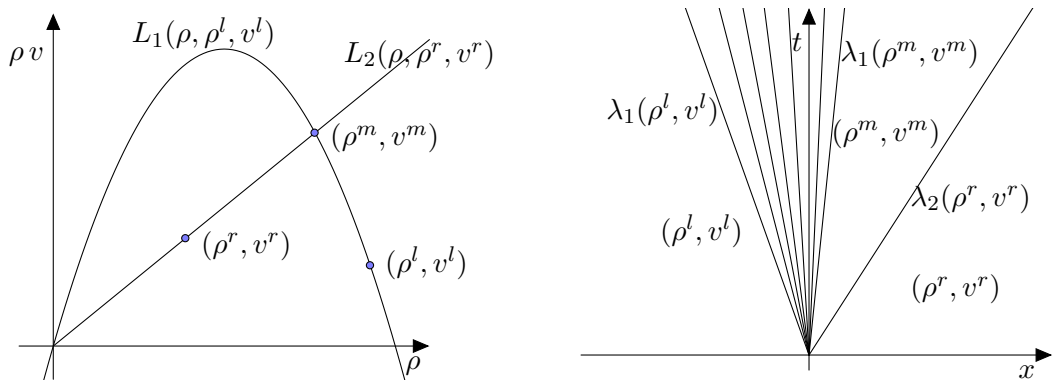
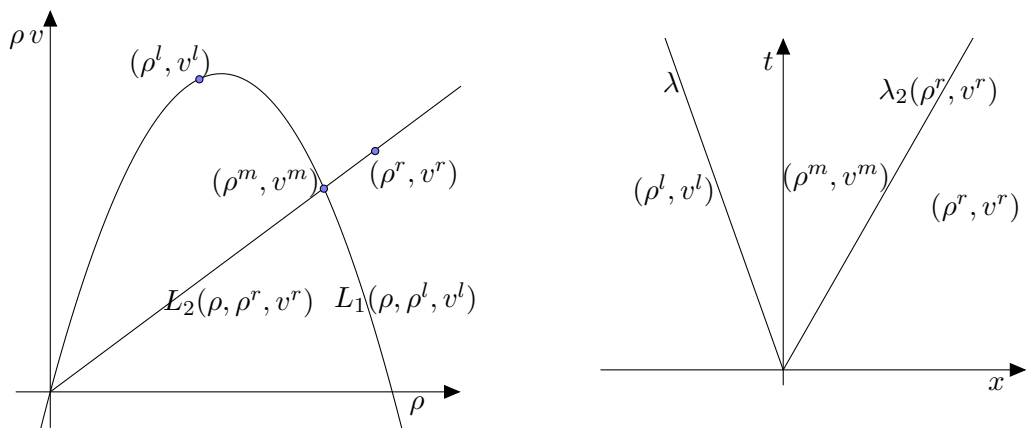
\begin{figure}[hbtp]
\centering
\begin{subfigure}[hbtp]{0.48\linewidth}
\centering
\definecolor{xdxdff}{rgb}{0.49019607843137253,0.49019607843137253,1.}
\begin{tikzpicture}[scale=0.9,line cap=round,line join=round,>=triangle 45,x=1.0cm,y=0.75cm]
\draw[->,color=black] (-0.5,0.) -- (6.,0.);
\draw[->,color=black] (0.,-0.7) -- (0.,7.5);
\clip(-1,-0.7) rectangle (6.5,7.5);
\draw[smooth,samples=50,domain=-0.5:6.0] plot(\x,{5.0*(\x)-(\x)^(2.0)});
\draw [domain=0.0:6.0] plot(\x,{(-0.--4*\x)/4.});
\draw (-0.75,7.3) node[anchor=north west] {$\rho\,v$};
\draw (5.1,0.1) node[anchor=north west] {$\rho$};
\draw (1.,7.) node[anchor=north west] {$(\rho^l,v^l)$};
\draw (4.5,4.75) node[anchor=north west] {$(\rho^r,v^r)$};
\draw (2.2,4.6) node[anchor=north west] {$(\rho^m,v^m)$};
\draw (4.0,2.36) node[anchor=north west] {$L_1(\rho,\rho^l,v^l)$};
\draw (1.5,2.23) node[anchor=north west] {$L_2(\rho,\rho^r,v^r)$};
\begin{scriptsize}
\draw [fill=xdxdff] (4.,3.991382190745906) circle (1.5pt);
\draw [fill=xdxdff] (2.18,6.14) circle (1.5pt);
\draw [fill=xdxdff] (4.75,4.73) circle (1.5pt);
\end{scriptsize}
\end{tikzpicture}
\end{subfigure}
\quad
\begin{subfigure}[hbtp]{0.48\linewidth}
\centering
\begin{tikzpicture}[scale = 0.4, line cap=round,line join=round,>=triangle 45,x=1.0cm,y=1.0cm]
\draw[->,color=black] (-6.,0.) -- (8.,0.);
\draw[->,color=black] (0.,-0.5) -- (0.,11.);
\clip(-6.,-1) rectangle (8.,11.);
\draw [domain=-6.0:0.0] plot(\x,{(-0.--7.3*\x)/-2.58});
\draw [domain=0.0:8.0] plot(\x,{(-0.--6.2*\x)/3.52});
\draw (-5.3,6.83) node[anchor=north west] {$(\rho^l,v^l)$};
\draw (-0.4,7.2) node[anchor=north west] {$(\rho^m,v^m)$};
\draw (4.5,6.21) node[anchor=north west] {$(\rho^r,v^r)$};
\draw (6.4,0.1) node[anchor=north west] {$x$};
\draw (-1,10.86) node[anchor=north west] {$t$};
\draw (-4.3,10.22) node[anchor=north west] {$\lambda$};
\draw (2.8,10.05) node[anchor=north west] {$\lambda_2(\rho^r,v^r)$};
\end{tikzpicture}
\end{subfigure}
\caption{Example of classic solution to the Riemann problem (\ref{Riemann_problem_ARZ}). In the represented case $(\rho^l,v^l)$ and $(\rho^m,v^m)$ are joined by a shock wave with negative propagations speed $\lambda=(\rho^m\,v^m-\rho^l\,v^l)/(\rho^m-\rho^l)$. A contact discontinuity with speed $\lambda_2(\rho^r,v^r)=\lambda_2(\rho^m,v^m)=v^r$ connects $(\rho^m,v^m)$ to $(\rho^r,v^r)$.}\label{fig_example_standard_solution_2}
\end{figure}
\section{The invariant domain}
\begin{mydef}
A set $\mathcal{D}\subseteq \mathbb{R}^+\times \mathbb{R}^+$ is an invariant domain (see \cite{hoff}) for the Riemann solver $\mathcal{RS}$ if for every $(\rho^l,v^l)$ and $(\rho^r,v^r)$ in $\mathcal{D}$, the solution $\mathcal{RS}((\rho^l,v^l),(\rho^r,v^r))(\lambda)$ to the Riemann problem (\ref{Riemann_problem_ARZ}) is in $\mathcal{D}$ for every $\lambda \in\mathbb{R}$. 
\end{mydef}
\begin{theorem}\label{standard_invariant_domain}
Fix $v_1$, $v_2$, $w_1$ and $w_2$ in $\mathbb{R}$ such that $0<v_1<v_2$, $0<w_1<w_2$ and $v_2<w_2$. The set
$$\mathcal{D}_{v_1,v_2,w_1,w_2}:=\lbrace (\rho,v) \in \mathbb{R}^+\times \mathbb{R}^+ : v_1\leq v \leq v_2, \; w_1 \leq v+p(\rho)\leq w_2 \rbrace$$
is invariant for the standard Riemann solver $\mathcal{RS}$ of the ARZ system; see Figure \ref{invariant_domain_standard}.
\end{theorem}
\begin{proof}
Consider $(\rho^l,v^l)$ and $(\rho^r,v^r)$ in the domain $\mathcal{D}_{v_1,v_2,w_1,w_2}$. Each point $(\bar{\rho},\bar{v})$ connected to $(\rho^l,v^l)$ by a rarefaction or a shock wave satisfies
$$\bar{v}+p(\bar{\rho})=v^l+p(\rho^l).$$
Therefore, since $(\rho^l,v^l)\in \mathcal{D}_{v_1,v_2,w_1,w_2}$, we have
$$w_1\leq \bar{v}+p(\bar{\rho})\leq w_2.$$
Similarly, each point $(\tilde{\rho},\tilde{v})$ joined to $(\rho^r,v^r)$ by a contact discontinuity is such that, $$\tilde{v}=v^r.$$
Hence $v_1\leq \tilde{v}\leq v_2$.\\
In particular, the middle state $(\rho^m,v^m)$ defined in (\ref{middle_state}) satisfies both conditions
$$v_1\leq v^m\leq v_2 \; \text{ and } \; w_1\leq v^m+p(\rho^m)\leq w_2,$$
because $v^m+p(\rho^m)=v^l+p(\rho^l)$ and $v^m=v^r$. Therefore $(\rho^m,v^m) \in \mathcal{D}_{v_1,v_2,w_1,w_2}$.\\
By Theorem \ref{prop_standard_solution}, the points $(\rho^l,v^l)$ and $(\rho^m,v^m)$ are joined by a rarefaction or by a shock wave. In the second case no intermediate states appear between $(\rho^l,v^l)$ and $(\rho^m,v^m)$. Otherwise, let $(\rho^\sigma,v^\sigma)$ be a point of the rarefaction connecting $(\rho^l,v^l)$ to $(\rho^m,v^m)$. By the hypotheses (\ref{ipotesi_pressione}), we have
$$\rho^l>\rho^\sigma>\rho^m \Longrightarrow p(\rho^l)>p(\rho^\sigma)>p(\rho^m)$$
and since $v^l+p(\rho^l)=v^\sigma+p(\rho^\sigma)=v^m+p(\rho^m)$, we find
$$v^l<v^\sigma<v^m.$$
Since $(\rho^l,v^l)$ and $(\rho^m,v^m)$ are in $\mathcal{D}_{v_1,v_2,w_1,w_2}$, we obtain
$$v_1\leq v^\sigma\leq v_2.$$
Therefore each point $(\rho^\sigma,v^\sigma)$ of the rarefaction belongs to $\mathcal{D}_{v_1,v_2,w_1,w_2}$.\\
If the points $(\rho^m,v^m)$ and $(\rho^r,v^r)$ are connected by a contact discontinuity, then there are no intermediate states between them.
\end{proof}
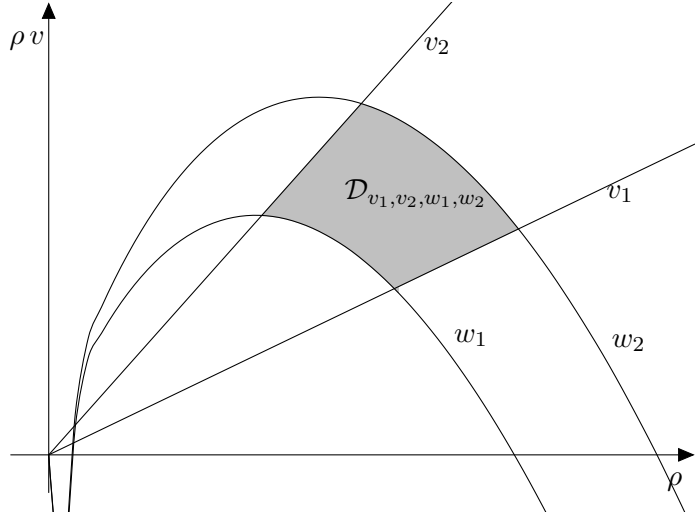
\begin{figure}[hbtp]
\centering
\definecolor{cqcqcq}{rgb}{0.7529411764705882,0.7529411764705882,0.7529411764705882}
\begin{tikzpicture}[scale = 0.5, line cap=round,line join=round,>=triangle 45,x=1.0cm,y=1.0cm]
\draw[->,color=black] (-1.,0.) -- (17.,0.);
\draw[->,color=black] (0.,-1.) -- (0.,12.);
\clip(-2.,-1.5) rectangle (17.,12.);
\draw[line width=0.pt,color=cqcqcq,fill=cqcqcq,fill opacity=1.0] {[smooth,samples=50,domain=5.61:8.22] plot(\x,{0-(-7.7/6.8)*\x-0.0/6.8})} -- (8.22,5.2) {[smooth,samples=50,domain=8.22:5.61] -- plot(\x,{3.5*\x-\x^(1.5)})} -- (5.61,6.3525) -- cycle;
\draw[line width=0.pt,color=cqcqcq,fill=cqcqcq,fill opacity=1.0] {[smooth,samples=50,domain=9.1:12.36] plot(\x,{4.0*\x-\x^(1.5)})} -- (12.36,6) {[smooth,samples=50,domain=12.36:9.1] -- plot(\x,{0-(-5.06/10.44)*\x-0.0/10.44})} -- (9.1,8.94) -- cycle;
\draw[line width=0.pt,color=cqcqcq,fill=cqcqcq,fill opacity=1.0] {[smooth,samples=50,domain=8.22:9.2] plot(\x,{4.0*\x-\x^(1.5)})} -- (9.2,4.3){[smooth,samples=50,domain=9.2:8.22] -- plot(\x,{3.5*\x-\x^(1.5)})} -- (8.22,9.3) -- cycle;
\draw[smooth,samples=50,domain=0.001:17.0] plot(\x,{4.0*(\x)-(\x)^(1.5)});
\draw[smooth,samples=50,domain=0.001:17.0] plot(\x,{3.5*(\x)-(\x)^(1.5)});
\draw [domain=0.0:17.0] plot(\x,{(-0.--7.7*\x)/6.8});
\draw [domain=0.0:17.0] plot(\x,{(-0.--5.06*\x)/10.44});
\draw (-1.3,11.5) node[anchor=north west] {$\rho\,v$};
\draw (16.,-0.2) node[anchor=north west] {$\rho$};
\draw (7.52,7.57) node[anchor=north west] {$\mathcal{D}_{v_1,v_2,w_1,w_2}$};
\draw (9.6,11.35) node[anchor=north west] {$v_2$};
\draw (14.41,7.41) node[anchor=north west] {$v_1$};
\draw (14.55,3.47) node[anchor=north west] {$w_2$};
\draw (10.4,3.6) node[anchor=north west] {$w_1$};
\end{tikzpicture}
\caption{The coloured area is an example of invariant domain for the standard Riemann solver.}\label{invariant_domain_standard}
\end{figure}
Let $(\rho_\text{min},v_\text{min})\in \mathcal{D}_{v_1,v_2,w_1,w_2}$ be the solution to the system
\begin{equation}\label{def_rho_v_min_1}
\begin{cases}
v+p(\rho)=w_1,\\
v=v_2.
\end{cases}
\end{equation}
Similarly, let $(\rho_\text{max},v_\text{max})\in \mathcal{D}_{v_1,v_2,w_1,w_2}$ be the solution to the system
\begin{equation}\label{def_rho_v_max_1}
\begin{cases}
v+p(\rho)=w_2,\\
v=v_1.
\end{cases}
\end{equation}
See Figure \ref{fig_lemma_curve_lipschitz_1}
\begin{prop}\label{prop_min_and_max_density_invariant_domain_RS_q_2}
The points $(\rho_\text{min},v_\text{min})$ and $(\rho_\text{max},v_\text{max})$ defined in (\ref{def_rho_v_min_1}) and (\ref{def_rho_v_max_1}) are the points of the invariant domain $\mathcal{D}_{v_1,v_2,w_1,w_2}$ respectively with minimal and maximal density.
\end{prop}
\begin{proof}
By definition $(\rho_\text{min},v_\text{min})$ and $(\rho_\text{max},v_\text{max})$ belong to $\mathcal{D}_{v_1,v_2,w_1,w_2}$. Let $(\rho,v)$ be another point of the domain $\mathcal{D}_{v_1,v_2,w_1,w_2}$. Since $v+p(\rho)\geq w_1$ and $v \leq v_2$, we have
$$p(\rho)\geq w_1-v \geq w_1 - v_2 =p(\rho_\text{min}).$$
By the hypotheses (\ref{ipotesi_pressione}), the function $\rho \rightarrow p(\rho)$ is (strictly) increasing. Therefore we find
$$\rho \geq \rho_{min}.$$
Similarly, since $(\rho,v)\in \mathcal{D}_{v_1,v_2,w_1,w_2}$, we have $v \geq v_1$ and $v+p(\rho)\leq w_2$. Therefore
$$p(\rho) \leq w_2-v_1= p(\rho_{max}).$$
Hence $\rho \leq \rho_{max}$.
\end{proof}
The next lemma states the conditions for the function $\psi(\rho)= \rho \, L_1(\rho,\rho_0,v_0)$ to be Lipschitz continuous inside the domain $\mathcal{D}_{v_1,v_2,w_1,w_2}$, when the pressure is a convex function.
\begin{lemma}\label{lemma_curve_lipschitz_1}
Let $(\rho_0,v_0)$ be a point in $\mathcal{D}_{v_1,v_2,w_1,w_2}$ and let $(\rho_{min},v_{min})$ and $(\rho_{max},v_{max})$ be the points defined in (\ref{def_rho_v_min_1}) and (\ref{def_rho_v_max_1}). If
\begin{equation}\label{ipotesi_lemma_lipschitz_1}
p''(\rho)\geq 0  \; \text{ for every } \rho \geq 0 \; \text{ and } \; \lambda_1(\rho_\text{min},v_\text{min})<0,
\end{equation}
then the function $\rho \to \rho\,L_1(\rho,\rho_0,v_0)$ is bi-Lipschitz in the domain $\mathcal{D}_{v_1,v_2,w_1,w_2}$, i.e
\begin{equation}\label{Lipschitz_inequality_1}
\lambda_1(\rho_\text{max},v_\text{max})\leq \dfrac{d}{d\rho}(\rho L_1(\rho,\rho_0,v_0)) \leq \lambda_1(\rho_\text{min},v_\text{min})
\end{equation}
for every $(\rho,v) = \left(\rho,L_1(\rho,\rho_0,v_0)\right)\in\mathcal{D}_{v_1,v_2,w_1,w_2}$.
\end{lemma}
\begin{proof}
Fix $\tilde{v}\in \mathbb{R}^+$. The function $\rho \to \lambda_1(\rho,\tilde{v})$ is strictly decreasing, indeed:
$$\dfrac{d}{d\rho}\lambda_1(\rho,\tilde{v}) = \dfrac{d}{d\rho}\left(\tilde{v}-\rho\,p'(\rho) \right)= -p'(\rho)-\rho\,p''(\rho)<0,$$
because by the hypotheses (\ref{ipotesi_pressione}) and (\ref{ipotesi_lemma_lipschitz_1}) we have $p'(\rho)>0$ and $p''(\rho)\geq 0$. Therefore
\begin{equation*}
\begin{split}
& \lambda_1(\rho_\text{min},v_2)\geq \lambda_1(\rho,v_2) \; \text{ for every } \; \rho \geq \rho_{\min} \; \text{ and}\\
& \lambda_1(\rho_{\max},v_1)\leq \lambda_1(\rho,v_1) \; \text{ for every } \; \rho\leq \rho_{\max}.
\end{split}
\end{equation*}
By Proposition \ref{eigenvalue_as_slope}, the function
$$\rho \to \lambda_1(\rho, L_1(\rho,\rho_0,v_0)) = \dfrac{d}{d\rho}\left(\rho\,L_1(\rho,\rho_0,v_0) \right)$$
is strictly decreasing.\\
Let us fix $(\rho,v)\in\mathcal{D}_{v_1,v_2,w_1,w_2}$. By the definition of $\mathcal{D}_{v_1,v_2,w_1,w_2}$ and Proposition \ref{prop_min_and_max_density_invariant_domain_RS_q_2}, we have
$$v_1\leq v \leq v_2 \; \text{ and } \; \rho_{\min}\leq \rho \leq \rho_{\max}.$$
Let $(\rho^*,v^*)\in\mathcal{D}_{v_1,v_2,w_1,w_2}$ be such that (see Figure \ref{fig_lemma_curve_lipschitz_1})
$$v^* = v_2 \; \text{ and } \; v^*+p(\rho^*)= v+p(\rho).$$
Then we have
$$\lambda_1(\rho,v)\leq \lambda_1(\rho^*,v^*)\leq \lambda_1(\rho_\text{min},v_\text{min}).$$
Similarly, let $(\rho^{**},v^{**})\in\mathcal{D}_{v_1,v_2,w_1,w_2}$ be such that (see Figure \ref{fig_lemma_curve_lipschitz_1})
$$v^{**} = v_2 \; \text{ and } \; v^{**}+p(\rho^{**})= v+p(\rho).$$
Then we have
$$\lambda_1(\rho,v)\geq \lambda_1(\rho^{**},v^{**})\geq \lambda_1(\rho_\text{max},v_\text{max}).$$
\end{proof}
\begin{figure}[hbtp]
\centering
\definecolor{cqcqcq}{rgb}{0.7529411764705882,0.7529411764705882,0.7529411764705882}
\definecolor{xdxdff}{rgb}{0.49019607843137253,0.49019607843137253,1.}
\definecolor{uququq}{rgb}{0.25098039215686274,0.25098039215686274,0.25098039215686274}
\begin{tikzpicture}[scale=0.5,line cap=round,line join=round,>=triangle 45,x=1.0cm,y=1.0cm]
\draw[->,color=black] (-1.,0.) -- (17.,0.);
\draw[->,color=black] (0.,-1.) -- (0.,12.);
\clip(-1.5,-1.5) rectangle (17.5,12.);
\draw[line width=0.pt,color=cqcqcq,fill=cqcqcq,fill opacity=1.0] {[smooth,samples=50,domain=4.45:6.75] plot(\x,{0-(-9.34/10.48)*\x-0.0/10.48})} -- (6.75,2.7) {[smooth,samples=50,domain=6.75:4.45] -- plot(\x,{3.0*\x-\x^(1.5)})} -- (4.45,3.96) -- cycle;
\draw[line width=0.pt,color=cqcqcq,fill=cqcqcq,fill opacity=1.0] {[smooth,samples=50,domain=6.75:9.66] plot(\x,{0-(-9.34/10.48)*\x-0.0/10.48})} -- (9.66,3.89) {[smooth,samples=50,domain=9.66:6.75] -- plot(\x,{0-(-4.5/11.18)*\x-0.0/11.18})} -- (6.75,6.) -- cycle;
\draw[line width=0.pt,color=cqcqcq,fill=cqcqcq,fill opacity=1.0] {[smooth,samples=50,domain=9.6:12.94] plot(\x,{4.0*\x-\x^(1.5)})} -- (12.94,5.2) {[smooth,samples=50,domain=12.94:9.6] -- plot(\x,{0-(-4.5/11.18)*\x-0.0/11.18})} -- (9.6,8.655487901127039) -- cycle;
\draw[smooth,samples=50,domain=0.001:17.0] plot(\x,{3.0*(\x)-(\x)^(1.5)});
\draw[smooth,samples=50,domain=0.001:17.0] plot(\x,{4.0*(\x)-(\x)^(1.5)});
\draw [domain=0.0:17.0] plot(\x,{(-0.--4.5*\x)/11.18});
\draw [domain=0.0:17.0] plot(\x,{(-0.--9.34*\x)/10.48});
\draw (13.,5.7449) node[anchor=north west] {$(\rho_{\max},v_{\max})$};
\draw (2.21218,5.10602) node[anchor=north west] {$(\rho_\text{min},v_\text{min})$};
\draw (11.1565,11.52144) node[anchor=north west] {$v_2$};
\draw (15.9481,7.66154) node[anchor=north west] {$v_1$};
\draw (4.5015,10.) node[anchor=north west] {$w_2$};
\draw (7.6,2.23106) node[anchor=north west] {$w_1$};
\draw (-1.2,11.4682) node[anchor=north west] {$\rho\,v$};
\draw (16.,0.1) node[anchor=north west] {$\rho$};
\draw[smooth,samples=50,domain=0.001:17.0] plot(\x,{3.7*(\x)-(\x)^(1.5)});
\draw (8.14844,7.82126) node[anchor=north west] {$(\rho^*,v^*)$};
\draw (9.50606,6.78308) node[anchor=north west] {$(\rho,v)$};
\draw (10.85,4.8) node[anchor=north west] {$(\rho^{**},v^{**})$};
\begin{scriptsize}
\draw [fill=uququq] (4.44,3.96) circle (2.5pt);
\draw [fill=uququq] (12.94,5.2) circle (2.5pt);
\draw [fill=xdxdff] (9.5,5.87) circle (2.5pt);
\draw [fill=uququq] (7.89,7.03) circle (2.5pt);
\draw [fill=uququq] (10.87,4.37) circle (2.5pt);
\end{scriptsize}
\end{tikzpicture}
\caption{Notations used in the proof of Lemma \ref{lemma_curve_lipschitz_1}.}\label{fig_lemma_curve_lipschitz_1}
\end{figure}
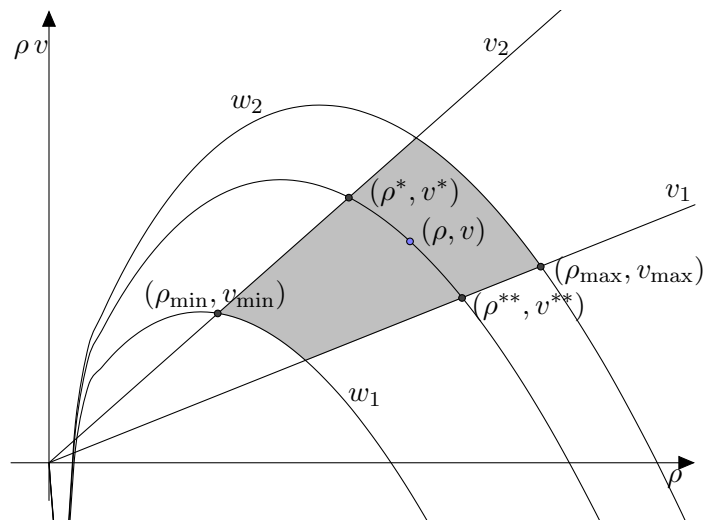
\chapter{Moving constraint}
A bus or a truck travelling on a road influences the traffic of the following vehicles and sometimes it is in turn influenced by the cars in front of it: the bus acts as a moving constraint on the flux of the vehicles. We want to describe this situation with the Aw-Rascle-Zhang system. See \cite{delle_monache_goatin} for the scalar case.\\
Let us denote
$$y(t)  \; \text{ and } \; \dot{y}(t) =\omega(\rho(t,y(t)+),v(t,y(t)+))$$
respectively the bus trajectory and speed, with $\omega:\mathbb{R}^+\times\mathbb{R}^+\to \mathbb{R}^+$ known. In the bus reference frame, the previous situation reduces to the case of a fixed constraint at $x=0$.
\begin{prop}\label{bus_reference_frame}
In the bus reference frame, the ARZ system is
\begin{equation}\label{Aw-Rascle_traslato}
\begin{cases}
\partial_{t}\rho+\partial_x \,(\rho(v-\dot{y}))=0,\\
\partial_t(\rho\, w)+\partial_x(\rho \,w\, (v-\dot{y}))=0.
\end{cases}
\end{equation}
\end{prop}
\begin{proof}
We have only to apply the chain rule.\\
Let $(\tilde{t},\tilde{x})$ be the coordinates in the bus reference frame, namely
\begin{equation*}
\begin{cases}
\tilde{t}=t,\\
\tilde{x}=x-y(t).
\end{cases}
\end{equation*}
Let $\tilde{\rho}(\tilde{t},\tilde{x})=\rho(t(\tilde{t},\tilde{x}),x(\tilde{t},\tilde{x}))$ and $\tilde{v}(\tilde{t},\tilde{x})=v(t(\tilde{t},\tilde{x}),x(\tilde{t},\tilde{x}))$ be respectively the density and the speed function in the new coordinates. First, we observe that
\begin{equation*}
\begin{split}
\dfrac{\partial \rho(t,x)}{\partial t} & =\dfrac{\partial \tilde{\rho}(\tilde{t},\tilde{x})}{\partial t}=\dfrac{\partial \tilde{\rho}(\tilde{t},\tilde{x})}{\partial \tilde{t}}\dfrac{\partial \tilde{t}}{\partial t}+\dfrac{\partial \tilde{\rho}(\tilde{t},\tilde{x})}{\partial \tilde{x}}\dfrac{\partial \tilde{x}}{\partial t}=\dfrac{\partial
\tilde{\rho}(\tilde{t},\tilde{x})}{\partial \tilde{t}}+(-\dot{y}(t(\tilde{t},\tilde{x}))\dfrac{\partial \tilde{\rho}(\tilde{t},\tilde{x})}{\partial \tilde{x}}.
\end{split}
\end{equation*}
Hence we have $\partial_{t} \rho = \partial_{\tilde{t}} \tilde{\rho}-\dot{y}\, \partial_{\tilde{x}} \tilde{\rho}$.\\
We find $\partial_x(\rho v)=\partial_{\tilde{x}} (\tilde{\rho} \tilde{v})$, indeed:
\begin{equation*}
\begin{split}
\dfrac{\partial (\rho v)}{\partial x}(t,x) & = \dfrac{\partial (\rho v)(t(\tilde{t},\tilde{x}),x(\tilde{t},\tilde{x}))}{\partial x}= \\
& = \dfrac{\partial (\rho v)(t(\tilde{t},\tilde{x}),x(\tilde{t},\tilde{x}))}{\partial \tilde{t}}\dfrac{\partial \tilde{t}}{\partial x}+ \dfrac{\partial (\rho v)(t(\tilde{t},\tilde{x}),x(\tilde{t},\tilde{x}))}{\partial \tilde{x}}\dfrac{\partial \tilde{x}}{\partial x}=\\
& = 0+\dfrac{\partial (\rho v)(t(\tilde{t},\tilde{x}),x(\tilde{t},\tilde{x}))}{\partial \tilde{x}}.
\end{split}
\end{equation*}
Finally,
$$ \dfrac{\partial \dot{y}(t)}{\partial \tilde{x}}=\dfrac{\partial \dot{y}(t(\tilde{t},\tilde{x}))}{\partial \tilde{x}}=\dfrac{\partial \dot{y}(\tilde{t})}{\partial \tilde{x}}=0.$$
Assembling these parts, we find
\begin{equation*}
\begin{split}
\partial_t \rho+\partial_x(\rho v) & =\partial_{\tilde{t}}\tilde{\rho}-\dot{y} \, \partial_{\tilde{x}}\tilde{\rho}+\partial_{\tilde{x}} (\tilde{\rho} \tilde{v})= \partial_{\tilde{t}}\tilde{\rho}+\partial_{\tilde{x}} (\tilde{\rho} \tilde{v})-\dot{y} \, \partial_{\tilde{x}}\tilde{\rho}-\tilde{\rho} \, \partial_{\tilde{x}}\dot{y}=\\
&=\partial_{\tilde{t}}\tilde{\rho}+\partial_{\tilde{x}}(\tilde{\rho} \, \tilde{v})-\partial_{\tilde{x}}(\dot{y} \, \tilde{\rho})
\end{split}
\end{equation*}
which is the first equation in (\ref{Aw-Rascle_traslato}).\\
Similarly for the second equation.
\end{proof}
By Proposition \ref{bus_reference_frame}, the flux function in the bus reference frame is
$$f(\rho,v)=\begin{pmatrix}
\rho (v-\dot{y})\\
\rho w (v-\dot{y})
\end{pmatrix}=\begin{pmatrix}
\rho(v-\dot{y})\\
\rho(v+p(\rho))(v-\dot{y})
\end{pmatrix}.$$
Let $V$ and $R$ be respectively the maximal speed and the maximal density of the vehicles in the considered road and let $\alpha \in (0,1)$ be the reduction rate of the road capacity at the bus position, i.e. $\alpha R$ is the density for which the velocity of the vehicles is $v=0$ in $x=y(t)$. Let $w_\alpha$ be the value of the Riemann invariant $w=v+p(\rho)$ at the point $(\rho,v)=(\alpha R,0)$, i.e.
$$w_\alpha := p(\alpha R).$$
\begin{prop}
The constraint on the first component of the flux at the bus position is given by
\begin{equation}\label{vincolo_sul_flusso}
\rho(t,y(t)) (v(t,y(t))-\dot{y}(t))\leq \rho_\alpha^2 \, p'(\rho_\alpha)=:F_\alpha,
\end{equation}
where $\rho_\alpha$ is the unique value of the density for which we have
$$p(\rho_\alpha)+\rho_\alpha \, p'(\rho_\alpha)=w_\alpha-\dot{y}(t).$$
\end{prop}
\begin{proof}
The maximal density in $x=y(t)$ is $\alpha R$ and for this value the vehicles near the bus are not moving, so that we have $v=0$. Let us consider the set
$$U_\alpha=\lbrace (\rho,v): v+p(\rho)\leq w_\alpha\rbrace.$$
Given a density $\bar{\rho}\in [0,\alpha R]$, the maximum of the Lax curves of the first family passing through the point $(\bar{\rho},v)\in U_\alpha$ is obtained for 
$$\bar{v}=L_1(\bar{\rho},\alpha R,0)=w_\alpha-p(\bar{\rho}).$$
We are looking for a constant $F_\alpha>0$ such that for every $(\rho,v)=(\rho, w_\alpha - p(\rho))$, the inequality
$$\phi(\rho):=\rho v-\rho \dot{y}=\rho(w_\alpha-p(\rho))-\rho \dot{y}\leq F_\alpha$$
holds. We observe that the derivative
$$\phi'(\rho)=w_\alpha-p(\rho)-\rho\,  p'(\rho)-\dot{y}$$
is zero if
$$p(\rho)+\rho\, p'(\rho)=w_\alpha-\dot{y}$$
and it is positive, if
$$p(\rho)+\rho\, p'(\rho)\leq w_\alpha-\dot{y}.$$
Hence $\phi$ has its maximum in the point $\rho_\alpha$ such that
\begin{equation}\label{rho_alpha}
p(\rho_\alpha)+\rho_\alpha p'(\rho_\alpha)=w_\alpha-\dot{y}(t).
\end{equation}
In this point we have
$$\phi(\rho_\alpha)=\rho_\alpha(w_\alpha-p(\rho_\alpha))-\dot{y}\, \rho_\alpha=\rho_\alpha( \dot{y}+\rho_\alpha \, p'(\rho_\alpha))-\dot{y}\, \rho_\alpha,$$
where we have used the equation (\ref{rho_alpha}). Hence at the bus position, we find
$$\rho(t,y(t)) (v(t,y(t))-\dot{y}(t))\leq \rho_\alpha^2 \, p'(\rho_\alpha)=F_\alpha.$$
\end{proof}
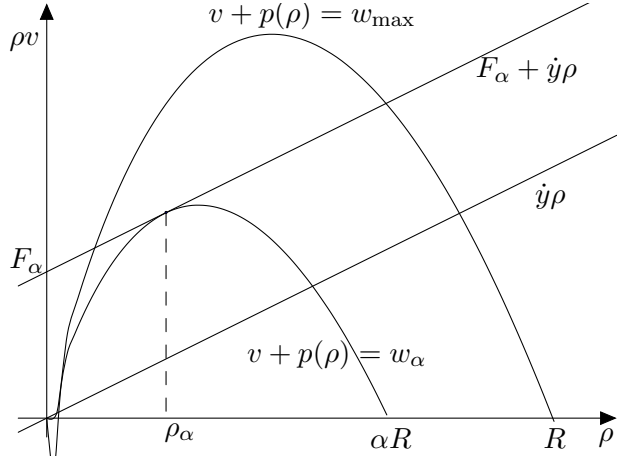
\begin{figure}[h]
\centering
\definecolor{qqqqff}{rgb}{0.,0.,1.}
\begin{tikzpicture}[scale = 0.25, line cap=round,line join=round,>=triangle 45,x=1.0cm,y=1.0cm]
\draw[->,color=black] (-1.5,0.) -- (30.,0.);
\draw[->,color=black] (0.,-1.) -- (0.,22.);
\clip(-2,-2.) rectangle (30.,22.);
\draw[smooth,samples=50,domain=0.001:26.7] plot(\x,{5.16*(\x)-(\x)^(1.5)});
\draw[smooth,samples=50,domain=0.001:17.9] plot(\x,{4.24*(\x)-(\x)^(1.5)});
\draw [domain=-1.5:30.] plot(\x,{(-0.--0.5*\x)/1.});
\draw [domain=-1.5:30.] plot(\x,{(--7.76--0.5*\x)/1.});
\draw [dash pattern=on 5pt off 5pt] (6.28,10.91)-- (6.28,0.);
\draw (5.65,0.27) node[anchor=north west] {$\rho_\alpha$};
\draw (16.49,0.) node[anchor=north west] {$\alpha R$};
\draw (25.60,0.) node[anchor=north west] {$R$};
\draw (28.5,0.) node[anchor=north west] {$\rho$};
\draw (25.16,13.) node[anchor=north west] {$\dot{y}\rho$};
\draw (22.05,19.67) node[anchor=north west] {$F_\alpha+\dot{y}\rho$};
\draw (-2.5,21.) node[anchor=north west] {$\rho v$};
\draw (10,4.7) node[anchor=north west] {$v+p(\rho)=w_\alpha$};
\draw (7.97,22.7) node[anchor=north west] {$v+p(\rho)=w_{\max}$};
\draw (-2.5,9.53) node[anchor=north west] {$F_\alpha$};
\begin{scriptsize}
\draw [fill=qqqqff] (6.28,10.91) circle (1.5pt);
\end{scriptsize}
\end{tikzpicture}
\caption{Representation of the flux in the fixed reference frame.}
\end{figure}
\begin{figure}[h]
\centering
\definecolor{qqqqff}{rgb}{0.,0.,1.}
\begin{tikzpicture}[scale = 0.25, line cap=round,line join=round,>=triangle 45,x=1.0cm,y=1.0cm]
\draw[->,color=black] (-5.,0.) -- (35.,0.);
\draw[->,color=black] (0.,-17.) -- (0.,19.);
\clip(-6.,-17.) rectangle (35.,19.);
\draw[smooth,samples=50,domain=0.001:35.0] plot(\x,{5.0*(\x)-(\x)^(1.5)});
\draw[smooth,samples=50,domain=0.001:35.0] plot(\x,{4.0*(\x)-(\x)^(1.5)});
\draw [domain=0.0:35.0] plot(\x,{(-0.-18.31*\x)/36.63});
\draw [domain=-5.:35.] plot(\x,{(--9.460137231507026-0.*\x)/1.});
\draw [dash pattern=on 8pt off 8pt] (7.02,9.46)-- (1.83,-0.91);
\draw (-2.5,9.5) node[anchor=north west] {$F_\alpha$};
\draw (-6,18) node[anchor=north west] {$\rho v -\dot y \rho$};
\draw (21,7) node[anchor=north west] {$v+p(\rho) = w_{\max}$};
\draw (6,7.2) node[anchor=north west] {$v+p(\rho) = w_{\alpha}$};
\draw (31,-0.) node[anchor=north west] {$\dot y \rho$};
\draw (32.73,-14.33) node[anchor=north west] {$\rho$};
\draw (28.,-14.5) node[anchor=north west] {$R$};
\draw (17.5,-10) node[anchor=north west] {$\alpha R$};
\draw (0.5,-0.7) node[anchor=north west] {$\rho_ \alpha$};
\begin{scriptsize}
\draw [fill=qqqqff] (7.02,9.46) circle (1.0pt);
\end{scriptsize}
\end{tikzpicture}
\caption{Representation of the flux in the bus reference frame.}
\end{figure}
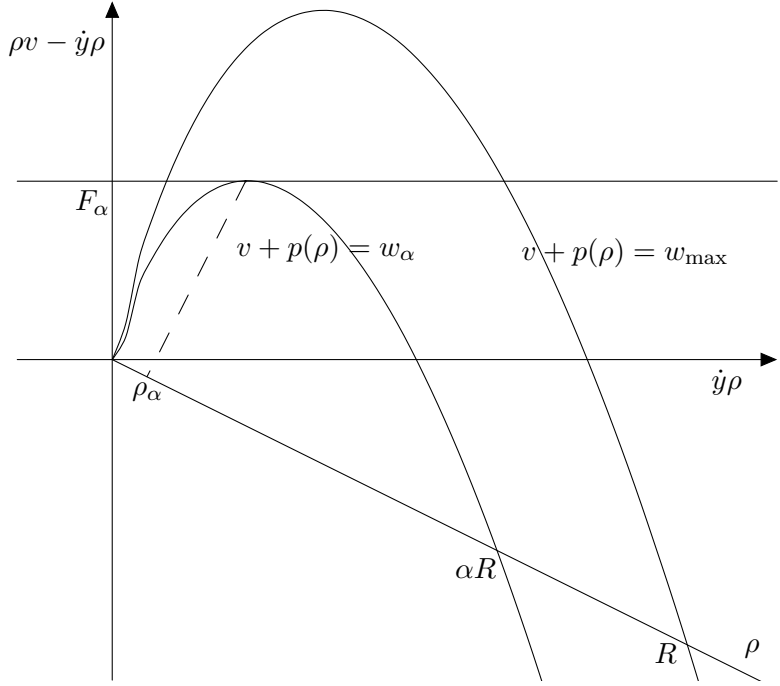
Let us consider a bus travelling in an empty road, where it can keep its own unperturbed velocity. If we represent the stops with a sequence $\lbrace x_i \rbrace _{i=1}^N$, the bus will proceed at its maximal speed $V_b$ ``far'' from each point $x_i$, i.e. its velocity profile is given by a sufficiently regular function $V(x)$ such that
\begin{equation}\label{velocita_autobus}
V(x)=\begin{cases}
V_b & \text{if } |x-x_i|>\delta \text{ for } i=1,...,N,\\
0   & \text{if } x=x_i \text{ for } i=1,...,N,
\end{cases}
\end{equation}
where $\delta$ is the space needed by the bus to stop, starting from its maximal speed.
Now, let us suppose that the bus remains at each stop for a constant time $\tau\in \mathbb{R}^+$ and let $t_i=\inf\lbrace t\in \mathbb{R}^+\,:\, y(t)=x_i\rbrace $ be the $i$-th stop instant, for $i=1,...,N$. The bus speed without traffic is a function $\dot{y}_{F}:\mathbb{R}^+\to [0,V_b]$ defined by
\begin{equation*}
\dot{y}_F(t)=\begin{cases}
V(y(t)) & \text{if } t \notin [t_i,t_i+\tau) \;\text{ for } \; i=1,...,N,\\
0 & \text{if } t\in[t_i,t_i+\tau) \;\text{ for } \; i=1,...,N.
\end{cases}
\end{equation*}
If we now introduce the traffic, the bus will travel with velocity $V(y(t))$ when there are no vehicles in front of him or when their speed $v(t,y(t)+)$ is higher than $V(y(t))$, otherwise it will adapt its velocity to the one of the traffic, namely
\begin{equation}
\dot{y}(t)=\begin{cases}
\dot{y}_F(t) & \text{if } \dot{y}_F(t)<v(t,y(t)+),\\
v(t,y(t)+) & \text{if } \dot{y}_F(t)\geq v(t,y(t)+).
\end{cases}
\end{equation}
\begin{remark}
If $V$ is the maximal speed of the cars on the street, we impose $v(t,y(t)+)=V$ if there are no vehicles in front of the bus, i.e. if $\rho(t,y(t)+)=0$. In this case the first condition is always satisfied, then we have $\dot{y}(t)=\dot{y}_F(t)$.
\end{remark}
The bus velocity is then given by a law depending on its position and on the speed of the preceding vehicles, namely
\begin{equation}
\dot{y}(t)=\omega(y(t),v(t,y(t)+)).
\end{equation}
Hence our model is given by the system
\begin{equation}\label{problema_vincolato}
\begin{cases}
\begin{cases}
\partial_{t}\rho+\partial_x (\rho(v-\dot{y}))=0,\\
\partial_t(\rho w)+\partial_x(\rho w(v-\dot{y}))=0,\\
\end{cases}\\
(\rho,v)(0,x)=(\rho_0,v_0)(x),\\
\rho(t,y(t)) (v(t,y(t))-\dot{y}(t))\leq \rho_\alpha^2 \, p'(\rho_\alpha),\\
\dot{y}(t)=\omega(y(t),v(t,y(t)+)),\\
y(0)=y_0,
\end{cases}
\end{equation}
where $(\rho_0,v_0)$ and $y_0$ are respectively the initial configuration of the density and of the speed of the cars and the initial bus position.
\section{The constrained Riemann problem}
Let $w_{max}:=p(R)$ be the value of the invariant $w=v+p(\rho)$ in the point $(R,0)$ of maximal density and velocity equal to zero.\\
Let $y_0=0$ be the initial position of the bus.\\
Let $(\rho^l,v^l)$ and $(\rho^r,v^r)$ be two data in the domain $\mathbb{R}^+\times\mathbb{R}^+$. Consider the Riemann problem
\begin{equation}\label{problema_riemann_bus_reference_frame}
\begin{cases}
\partial_{t}\rho+\partial_x (\rho(v-\dot{y}))=0,\\
\partial_t(\rho w)+\partial_x(\rho w(v-\dot{y}))=0,\\
(\rho,v)(0,x)=\begin{cases}
(\rho^l,v^l) & \text{ if } x\leq 0,\\
(\rho^r,v^r) & \text{ if } x>0,
\end{cases}
\end{cases}
\end{equation}
with the constraint
\begin{equation}\label{vincolo_1}
\rho(t,y(t)) (v(t,y(t))-\bar{V})\leq \rho_\alpha^2 \, p'(\rho_\alpha),
\end{equation}
where we assume that the speed $\dot{y}(t)$ is constant and its value is $\dot{y}(t)=\bar{V}\in[0,V_b]$ for every $t\in\mathbb{R}^+$.\\
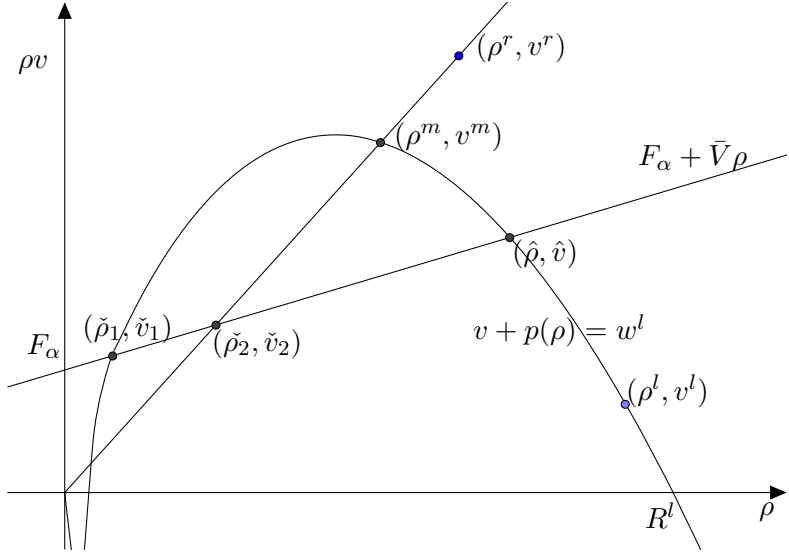
\begin{figure}[h]
\centering
\definecolor{qqqqff}{rgb}{0.,0.,1.}
\definecolor{xdxdff}{rgb}{0.49019607843137253,0.49019607843137253,1.}
\definecolor{uququq}{rgb}{0.25098039215686274,0.25098039215686274,0.25098039215686274}
\begin{tikzpicture}[scale = 0.5,line cap=round,line join=round,>=triangle 45,x=1.0cm,y=1.0cm]
\draw[->,color=black] (-1.5,0.) -- (19.,0.);
\draw[->,color=black] (0.,-1.5) -- (0.,13.);
\clip(-2,-1.5) rectangle (19.,13.);
\draw[smooth,samples=50,domain=0.001:19.0] plot(\x,{4.0*(\x)-(\x)^(1.5)});
\draw [domain=-1.5:19.] plot(\x,{(--3.25--0.3*\x)/1.});
\draw (15.,0.) node[anchor=north west] {$R^l$};
\draw (18.,0.) node[anchor=north west] {$\rho$};
\draw (14.82,9.58) node[anchor=north west] {$F_\alpha+\bar{V}\rho$};
\draw (-1.5,11.86) node[anchor=north west] {$\rho v$};
\draw (-1.27,4.48) node[anchor=north west] {$F_\alpha$};
\draw [domain=0.0:19.0] plot(\x,{(-0.--11.58*\x)/10.37});
\draw (10.55,12.45) node[anchor=north west] {$(\rho^r,v^r)$};
\draw (8.4,10.1) node[anchor=north west] {$(\rho^m,v^m)$};
\draw (3.6,4.6) node[anchor=north west] {$(\check{\rho_2},\check{v}_2)$};
\draw (0.2,5.) node[anchor=north west] {$(\check{\rho}_1,\check{v}_1)$};
\draw (11.5,7.) node[anchor=north west] {$(\hat{\rho},\hat{v})$};
\draw (14.5,3.4) node[anchor=north west] {$(\rho^l,v^l)$};
\draw (10.48,5.) node[anchor=north west] {$v+p(\rho)=w^l$};
\begin{scriptsize}
\draw [fill=uququq] (1.26,3.62) circle (3pt);
\draw [fill=uququq] (11.71,6.76) circle (3pt);
\draw [fill=xdxdff] (14.75,2.34) circle (3pt);
\draw [fill=qqqqff] (10.37,11.58) circle (3pt);
\draw [fill=uququq] (8.31,9.28) circle (3pt);
\draw [fill=uququq] (3.98,4.44) circle (3pt);
\end{scriptsize}
\end{tikzpicture}
\caption{Notations used: $(\hat{\rho},\hat{v})$ and $(\check{\rho}_1,\check{v}_1)$ are the points of the Lax curve of the first family passing through $(\rho^l,v^l)$ for which the constraint is satisfied with the equal; $(\check{\rho}_2,\check{v}_2)$ is the point of the Lax curve of the second family passing through $(\rho^r,v^r)$ for which the constraint is satisfied with the equal.}\label{notazioni_utilizzate_problema_Riemann}
\end{figure}
Let $I$ be the set
\begin{equation*}
\begin{split}
I & =\lbrace \rho \in [0,R]: \, \rho L_1(\rho,\rho^l,v^l)=\rho (v^l+p(\rho^l)-p(\rho))=F_\alpha+\rho \bar{V}\rbrace=\\
&=\lbrace \rho\in [0,R]:\, \rho(L_1(\rho,\rho^l,v^l)-\bar{V})=F_\alpha \rbrace.
\end{split}
\end{equation*}
If $I$ is not empty, let $(\hat{\rho},\hat{v})$ and $(\check{\rho}_1,\check{v}_1)$ be the points defined by 
$$\hat{\rho}=\max I, \; \; \hat{v}=\dfrac{F_\alpha}{\hat{\rho}}+\bar{V}, \; \; \check{\rho}_1 = \min I \; \mbox{ and } \; \check{v}_1=\dfrac{F_\alpha}{\check{\rho}_1}+\bar{V}.$$
These are respectively the points with maximal and minimal density of the Lax curve of the first family passing through $(\rho^l,v^l)$ for which the condition (\ref{vincolo_1}) on the flux is satisfied with the equal. Moreover, we define the point $(\check{\rho}_2,\check{v}_2)$ as
$$\check{\rho}_2=\dfrac{F_\alpha}{v^r-\bar{V}} \; \text{ and } \; \check{v}_2=v^r.$$
$(\check{\rho}_2,\check{v}_2)$ is the point of maximal density of the Lax curve of the second family passing through $(\rho^r,v^r)$ for which the condition (\ref{vincolo_1}) is satisfied with the equal. Finally we recall the definition (\ref{middle_state}) of the intermediate state $(\rho^m,v^m)$ of the classical solution: define the set
$$L=\lbrace \rho \in \mathbb{R}^+:v^r=L_1(\rho,\rho^l,v^l)\rbrace.$$
The middle state $(\rho^m,v^m)\in \mathbb{R}^+\times \mathbb{R}^+$ is the point defined by:
\begin{equation*}
\rho^m = \max L \; \text{ and } \; v^m=v^r;
\end{equation*}
see Figure (\ref{notazioni_utilizzate_problema_Riemann}).
\begin{lemma}\label{esistenza_rho_hat_rho_check_1}
Fix $(\rho^l,v^l)\in \mathbb{R}^+\times \mathbb{R}^+$. Let us suppose that the hypotheses (\ref{ipotesi_pressione}) hold. If there exists a point $(\rho^m,v^m)$ such that $v^m=L_1(\rho^m,\rho^l,v^l)$ and $\rho^m v^m>F_\alpha+\rho^m \bar{V}$, then $I=\lbrace \check{\rho}_1, \hat{\rho}\rbrace$.
\end{lemma}
\begin{proof}
Let us consider the plane $(\rho,\rho v)$ and let $w^l=v^l+p(\rho^l)$ be the value of the invariant $w$ in $(\rho^l,v^l)$. In these coordinates the set of points $(\rho, v)$ such that $v+p(\rho)=w^l$ is the graph of the function
$$\psi(\rho):=\rho w^l-\rho p(\rho).$$
This function is strictly concave, because the function $\rho \rightarrow \rho p(\rho)$ is strictly convex by the hypotheses (\ref{ipotesi_pressione}). Then the cardinality of the set $I$ is at most 2.\\
Since $\rho^m v^m>F_\alpha+\rho^m \bar{V}$, there will be exactly two points $(\hat{\rho},\hat{v})$ and $(\check{\rho}_1,\check{v}_1)$ belonging to the curve $\psi$ and such that
$$\check{\rho}_1<\rho^m, \;\; \hat{\rho}>\rho^m, \;\; \psi(\check{\rho}_1)=F_\alpha+\bar{V}\check{\rho}_1\; \text{ and } \; \psi(\hat{\rho})=F_\alpha+\bar{V}\hat{\rho}.$$
\end{proof}
\begin{lemma}\label{funzione_concava_e_retta}
Let $w^l\in\mathbb{R}^+$ be fixed and let $\rho v=F_\alpha+\bar{V}\rho$ be the constraint. Let us consider the function $\rho \to \psi(\rho)=\rho w^l-\rho p(\rho)$. Let $(\rho^\sigma,v^\sigma)$ be a point such that $\rho^\sigma v^\sigma=\psi(\rho^\sigma)$. Under the hypotheses of Lemma \ref{esistenza_rho_hat_rho_check_1}, we have
\begin{equation}
\rho^\sigma>\hat{\rho} \; \text{ or }\; \rho^\sigma<\check{\rho}_1 \; \text{ if and only if } \;\rho^\sigma v^\sigma< F_\alpha+\bar{V}\rho^\sigma.
\end{equation}
Equivalently
\begin{equation}\label{condizione_vincolo_soddisfatto_velocita}
v^\sigma<\hat{v} \; \text{ or }\; v^\sigma>\check{v}_1 \; \text{ if and only if } \;\rho^\sigma v^\sigma< F_\alpha+\bar{V}\rho^\sigma.
\end{equation}
\end{lemma}
\begin{proof}
By the hypotheses of Lemma \ref{esistenza_rho_hat_rho_check_1}, there is a point $(\rho^m,v^m)$ such that $v^m=L_1(\rho^m,\rho^l,v^l)$ and $\rho^m v^m>F_\alpha+\bar{V}\rho^m$. Hence we have $I=\lbrace \check{\rho}_1, \hat{\rho}\rbrace$.\\ 
If $\rho^\sigma>\hat{\rho}$, there exists $\gamma\in(0,1)$ such that
$$\hat{\rho}=\gamma \rho^\sigma+(1-\gamma)\check{\rho}_1.$$
If it was $\psi(\rho^\sigma)\geq F_\alpha+\bar{V}\rho^\sigma$, then by the strict concavity of $\psi$, we would have
\begin{equation*}
\begin{split}
\psi(\hat{\rho})&=\psi(\gamma\rho^\sigma+(1-\gamma)\check{\rho}_1)>\gamma\psi(\rho^\sigma)+(1-\gamma)\psi(\check{\rho}_1)\geq\\
&\geq \gamma(F_\alpha+\bar{V}\rho^\sigma)+(1-\gamma)(F_\alpha+\bar{V}\check{\rho}_1)=\\
& = \gamma F_\alpha+\gamma \bar{V}\rho^\sigma+F_\alpha+\bar{V}\check{\rho}_1-\gamma F_\alpha-\gamma\bar{V}\check{\rho}_1=\\
&=F_\alpha+\bar{V}(\gamma\rho^\sigma+(1-\gamma)\check{\rho}_1)=\\
& =F_\alpha+\bar{V}\hat{\rho} \Longrightarrow\\
&\Longrightarrow \psi(\hat{\rho})>F_\alpha+\bar{V}\hat{\rho}.
\end{split}
\end{equation*}
This is a contradiction of the definition of $(\hat{\rho},\hat{v})$. Similarly for $\rho^\sigma<\check{\rho}_1$.\\
Conversely, if $\rho^\sigma v^\sigma<F_\alpha+\bar{V}\rho^\sigma$ and $\rho^\sigma$ was in $[\check{\rho}_1,\hat{\rho}]$, there would exist $\gamma\in [0,1]$ such that $\rho^\sigma= \gamma\check{\rho}_1+(1-\gamma)\hat{\rho}$. Hence
\begin{equation*}
\begin{split}
\psi(\rho^\sigma)&> \gamma \psi(\check{\rho}_1)+(1-\gamma)\psi(\hat{\rho})=\gamma(F_\alpha+\check{\rho}_1\bar{V})+(1-\gamma)(F_\alpha+\hat{\rho}\bar{V})=\\
&=F_\alpha+\bar{V}(\gamma\check{\rho}_1+(1-\gamma)\hat{\rho})=F_\alpha+\bar{V}\rho^\sigma,
\end{split}
\end{equation*}
which is absurd.\\
The condition $\rho^\sigma >\hat{\rho}$ is equivalent to $v^\sigma <\hat{v}$. Indeed,
$$\psi(\rho^\sigma)=\rho^\sigma\,v^\sigma \Longleftrightarrow v^\sigma +p(\rho^\sigma)=v^l+p(\rho^l)=\hat{v}+p(\hat{\rho}).$$
Hence, by the hypotheses (\ref{ipotesi_pressione}), we find
\begin{equation*}
\rho^\sigma >\hat{\rho} \Longleftrightarrow p(\rho^\sigma) >p(\hat{\rho}) \Longleftrightarrow v^\sigma = \hat{v}+p(\hat{\rho})-p(\rho^\sigma) <\hat{v}.
\end{equation*}
\end{proof}
Let $(\rho^l,v^l)$ and $(\rho^r,v^r)$ be two fixed points in the domain $\mathbb{R}^+\times\mathbb{R}^+$. Let $\mathcal{RS}$ be the standard Riemann solver for the Riemann problem (\ref{Riemann_problem_ARZ}); see Proposition \ref{prop_standard_solution}.\\
Let us denote
$$\bar{\rho}((\rho^l,v^l),(\rho^r,v^r))(\cdot) \; \text{ and } \; \bar{v}((\rho^l,v^l),(\rho^r,v^r))(\cdot)$$
respectively the $\rho$ and $v$ components of the classical solution $\mathcal{RS}((\rho^l,v^l),(\rho^r,v^r))(\cdot)$.
\subsection{The first Riemann solver $\mathcal{RS}^\alpha_1$}
Let us introduce the first Riemann solver for the constrained Riemann problem (\ref{problema_riemann_bus_reference_frame}).\\
The Riemann solver
$$\mathcal{RS}^\alpha_1:(\mathbb{R}^+\times\mathbb{R}^+)^2\rightarrow L^1(\mathbb{R},\mathbb{R}^+\times\mathbb{R}^+)$$
is defined as follows.
\begin{enumerate}
\item If $f_1(\mathcal{RS}((\rho^l,v^l),(\rho^r,v^r))(\bar{V}))>F_\alpha+\bar{V} \bar{\rho}((\rho^l,v^l),(\rho^r,v^r))(\bar{V})$, then
\begin{equation*}
\begin{split}
&\mathcal{RS}^\alpha_1((\rho^l,v^l),(\rho^r,v^r))(x/t)=\begin{cases}
\mathcal{RS}((\rho^l,v^l),(\hat{\rho},\hat{v}))(x/t) & \text{ if } x < y(t),\\
\mathcal{RS}((\check{\rho}_1,\check{v}_1),(\rho^r,v^r))(x/t) & \text{ if } x\geq y(t),
\end{cases}\\
&\text{and } \; y(t)=\bar{V}t.
\end{split}
\end{equation*}
\item If $ f_1(\mathcal{RS}((\rho^l,v^l),(\rho^r,v^r))(\bar{V}))\leq F_\alpha+ \bar{V}\bar{\rho}((\rho^l,v^l),(\rho^r,v^r))(\bar{V})$ and\\ $\bar{V}<\bar{v}((\rho^l,v^l),(\rho^r,v^r))(\bar{V})$, then
\begin{equation*}
\mathcal{RS}^\alpha_1((\rho^l,v^l),(\rho^r,v^r))(x/t) = \mathcal{RS}((\rho^l,v^l),(\rho^r,v^r))(x/t) \; \text{ and } \; y(t)=\bar{V}t.
\end{equation*}
\item If $\bar{v}((\rho^l,v^l),(\rho^r,v^r))(\bar{V})\leq \bar{V}$, then
$$\mathcal{RS}^\alpha_1((\rho^l,v^l),(\rho^r,v^r))(x/t) = \mathcal{RS}((\rho^l,v^l),(\rho^r,v^r))(x/t) \; \text{ and } \; y(t)=v(t,y(t)+)t.$$
\end{enumerate}
The first case refers to a situation in which the traffic is influenced by the bus and the bus travels with its own velocity; in the second case the bus and the traffic do not influence each other; the third case represents a road where the traffic is congested and the bus travels with the speed of the previous cars.
\begin{remark}\label{nonclassical_shock}
In the first case, the solution given by $\mathcal{RS}^\alpha_1$ at $x=y(t)$ does not satisfy the Lax-entropy condition between the states $(\hat{\rho},\hat{v})$ and $(\check{\rho}_1,\check{v}_1)$, because
$$\check{\rho}_1<\hat{\rho} \Longleftrightarrow \lambda_1(\check{\rho}_1,\check{v}_1) >\lambda_1(\hat{\rho},\hat{v})$$
and the condition (\ref{Lax_entropy_condition}) for an entropy-admissible shock is the reverse. Therefore, we say that $(\hat{\rho},\hat{v})$ and $(\check{\rho}_1,\check{v}_1)$ are connected by a non-classical shock.
\end{remark}
\begin{remark}
Let us consider the bus reference frame. The representation of the flux function in the $(\rho,v)$ coordinates is
$$f(\rho,v)=\begin{pmatrix}
\rho(v-\bar{V})\\
\rho (v-\bar{V}) w
\end{pmatrix}.
$$
In this reference the non-classical shock travels with propagation speed $\lambda=0$.\\
The solution $\mathcal{RS}^\alpha_1$ is conservative for both density and momentum of the vehicles. Indeed, for the first component of the flux, we observe that
\begin{equation*}
\begin{split}
&\hat{\rho}(\hat{v}-\bar{V})-\check{\rho}_1(\check{v}_1-\bar{V})=0(\hat{\rho}-\check{\rho}_1) \Longleftrightarrow\\
&\Longleftrightarrow \hat{\rho}\hat{v}-\check{\rho}_1\check{v}_1-\bar{V}(\hat{\rho}-\check{\rho}_1)=0 \Longleftrightarrow\\
&\Longleftrightarrow \bar{V}= \dfrac{\hat{\rho}\hat{v}-\check{\rho}_1\check{v}_1}{\hat{\rho}-\check{\rho}_1}.
\end{split}
\end{equation*}
Hence the Rankine-Hugoniot condition holds for the first component if and only if $\bar{V}$ is the slope of the line passing through the points $(\hat{\rho},\hat{\rho}\hat{v})$ and $(\check{\rho}_1,\check{\rho}_1\check{v}_1)$ in the $(\rho,\rho v)$ coordinates which is true.\\
For the second component, let us denote $\hat{w}=\hat{v}+p(\hat{\rho})$ and $\check{w}_1=\check{v}_1+p(\check{\rho}_1)$. Since $\hat{w}=\check{w}_1=v^l+p(\rho^l)$, we find
\begin{equation*}
\begin{split}
&\hat{\rho}(\hat{v}-\bar{V})\hat{w}-\check{\rho}_1(\check{v}_1-\bar{V})\check{w}_1=0(\hat{v}-\check{v}_1) \Longleftrightarrow\\
&\Longleftrightarrow (v^l+p(\rho^l))(\hat{\rho}(\hat{v}-\bar{V})-\check{\rho}_1(\check{v}_1-\bar{V}))=0
\end{split}
\end{equation*}
and we can conclude as for the first component.\\
Therefore the Rankine-Hugoniot conditions hold along the line $x-y(t)=0$.
\end{remark}
\subsection{The second Riemann solver $\mathcal{RS}^\alpha_2$}
Let us introduce the second Riemann solver for the constrained Riemann problem (\ref{problema_riemann_bus_reference_frame}).\\
The Riemann solver 
$$\mathcal{RS}^\alpha_2:(\mathbb{R}^+\times\mathbb{R}^+)^2\rightarrow L^1(\mathbb{R},\mathbb{R}^+\times\mathbb{R}^+)$$
is defined as follows.
\begin{enumerate}
\item If $f_1(\mathcal{RS}((\rho^l,v^l),(\rho^r,v^r))(\bar{V}))>F_\alpha+\bar{V} \bar{\rho}((\rho^l,v^l),(\rho^r,v^r))(\bar{V})$, then
\begin{equation*}
\begin{split}
&\mathcal{RS}^\alpha_2((\rho^l,v^l),(\rho^r,v^r))(x/t)=\begin{cases}
\mathcal{RS}((\rho^l,v^l),(\hat{\rho},\hat{v}))(x/t) & \text{ if } x < y(t),\\
\mathcal{RS}((\check{\rho}_2,\check{v}_2),(\rho^r,v^r))(x/t) & \text{ if } x\geq y(t),
\end{cases}\\
&\text{and } \; y(t)=\bar{V}t.
\end{split}
\end{equation*}
\item If $ f_1(\mathcal{RS}((\rho^l,v^l),(\rho^r,v^r))(\bar{V}))\leq F_\alpha+ \bar{V}\bar{\rho}((\rho^l,v^l),(\rho^r,v^r))(\bar{V})$ and\\ $\bar{V}<\bar{v}((\rho^l,v^l),(\rho^r,v^r))(\bar{V})$, then
\begin{equation*}
\mathcal{RS}^\alpha_2((\rho^l,v^l),(\rho^r,v^r))(x/t) = \mathcal{RS}((\rho^l,v^l),(\rho^r,v^r))(x/t) \; \text{ and } \; y(t)=\bar{V}t.
\end{equation*}
\item If $\bar{v}((\rho^l,v^l),(\rho^r,v^r))(\bar{V})\leq \bar{V}$, then
$$\mathcal{RS}^\alpha_2((\rho^l,v^l),(\rho^r,v^r))(x/t) = \mathcal{RS}((\rho^l,v^l),(\rho^r,v^r))(x/t) \; \text{ and } \; y(t)=v(t,y(t)+)t.$$
\end{enumerate}
\begin{remark}
In the first case the points $(\check{\rho}_2,\check{v}_2)$ and $(\hat{\rho},\hat{v})$ are connected by a non-classical shock. Indeed, fix $(\rho^l,v^l)$ and $(\rho^r,v^r)$, let $(\rho^m,v^m)$ be the middle state of the classical solution and assume that the inequality $\check{\rho}_2>\hat{\rho}$ holds. This implies $\check{v}_2 <\hat{v}$, because $\hat{\rho}(\hat{v}-\bar{V})=\check{\rho}_2(\check{v}_2-\bar{V})$. Since $v^m=\check{v}_2$ and $v^m+p(\rho^m)=\hat{v}+p(\hat{\rho})$, we find $\rho^m>\hat{\rho}$. Therefore, by Lemma \ref{funzione_concava_e_retta}, we have
\begin{equation}\label{remark_RS_2_non_classical_shock_middle_state}
\rho^m\,v^m \leq F_\alpha+\rho^m \, \bar{V}.
\end{equation}
Consider the classical solution. Since $(\rho^m,v^m)$ and $(\rho^r,v^r)$ are connected with a contact discontinuity with propagation speed $v^r=v^m$, the solution $\mathcal{RS}((\rho^l,v^l),(\rho^r,v^r))$ in $\bar{V}$ is $(\rho^r,v^r)$ if and only if $v^r <\bar{V}$. In this case $(\rho^r,v^r)$ satisfies the constraint and the solution is classical. Let us assume $v^r \geq \bar{V}$. The classical solution in $\bar{V}$ is $(\rho^l,v^l)$, $(\rho^m,v^m)$ or a point $(\rho^\sigma,v^\sigma)$ of the rarefaction wave joining them. If $(\rho^l,v^l)$ and $(\rho^m,v^m)$ are connected by a rarefaction wave, all the points of the rarefaction satisfy the constraint, because $\hat{\rho}>\rho^m \geq \rho^\sigma$ for every $\sigma$. Finally, suppose that $(\rho^l,v^l)$ and $(\rho^m,v^m)$ are connected by a shock and that $(\rho^l,v^l)$ does not satisfy the constraint, i.e.
$$\rho^l\,(v^l-\bar{V})>F_\alpha.$$
The speed of the shock is
$$s =\dfrac{\rho^m\,v^m-\rho^l\,v^l}{\rho^m-\rho^l}.$$
If $\mathcal{RS}((\rho^l,v^l),(\rho^r,v^r))$ in $\bar{V}$ is $(\rho^m,v^m)$, then the solution given by $\mathcal{RS}^\alpha_2$ is classic. Otherwise we find
\begin{equation*}
\begin{split}
& \dfrac{\rho^m\,v^m-\rho^l\,v^l}{\rho^m-\rho^l}\geq \bar{V} \Longleftrightarrow\\
& \rho^m\,(v^m -\bar{V}) \geq \rho^l\,(v^l-\bar{V})>F_\alpha.
\end{split}
\end{equation*} 
This is a contradiction of the inequality (\ref{remark_RS_2_non_classical_shock_middle_state}).\\
Therefore if $\check{\rho}_2>\hat{\rho}$, the solution $\mathcal{RS}^\alpha_2((\rho^l,v^l),(\rho^r,v^r))(\lambda)$ satisfies the constraint for every $\lambda\in \mathbb{R}$ and the non-classical shock appears only when $\check{\rho}_2\geq \hat{\rho}$ which is against the Lax-entropy condition (\ref{Lax_entropy_condition}).
\end{remark}
\begin{remark}\label{the_second_Riemann_solver_is_not_conservative}
The Riemann solver $\mathcal{RS}^\alpha_2$ conserves only the density of the vehicles. Indeed, along the line $x-y(t)=0$ the Rankine-Hugoniot condition holds for the first component of the flux, because both $(\hat{\rho},\hat{v})$ and $(\check{\rho}_2,\check{v}_2)$ are on the line $\rho v=F_\alpha+\bar{V}\rho$, i.e.
\begin{equation*}
\begin{split}
&\hat{\rho}(\hat{v}-\bar{V})-\check{\rho}_2(\check{v}_2-\bar{V})=0 \Longleftrightarrow\\
&\Longleftrightarrow \bar{V}=\dfrac{\hat{\rho}\hat{v}-\check{\rho}_2\check{v}_2}{\hat{\rho}-\check{\rho}_2},
\end{split}
\end{equation*}
Calling $\hat{w}=\hat{v}+p(\hat{\rho})$ and $\check{w}_2=\check{v}_2+p(\check{v}_2)$, for the second component we have
\begin{equation*}
\begin{split}
&\check{\rho}_2(\check{v}_2-\bar{V})\check{w}_2-\hat{\rho}(\hat{v}-\bar{V})\hat{w}=0 \Longleftrightarrow\\
&\check{\rho}_2(\check{v}_2-\bar{V})-\dfrac{\hat{w}}{\check{w}_2}\hat{\rho}(\hat{v}-\bar{V})=0 \Longleftrightarrow\\
&\Longleftrightarrow \bar{V}=\dfrac{\check{\rho}_2\check{v}_2-\frac{\hat{w}}{\check{w}_2}\hat{\rho}\hat{v}}{\check{\rho}_2-\frac{\hat{w}}{\check{w}_2}\hat{\rho}}.
\end{split}
\end{equation*}
This condition is satisfied only if $\hat{w}=\check{w}_2$, but this is false in general. 
\end{remark}
\section{Invariant domains}
Let us modify the results of \cite{garavello_goatin} to characterize the invariant domains of the Riemann solvers $\mathcal{RS}^\alpha_1$ and $\mathcal{RS}^\alpha_2$.\\
Fixed $0<v_1<v_2$ and $0<w_1<w_2$ with $v_2<w_2$, by Proposition \ref{standard_invariant_domain}, the domain
\begin{equation}
\mathcal{D}_{v_1,v_2,w_1,w_2}=\lbrace (\rho,v)\in \mathbb{R}^+\times \mathbb{R}^+\, : \, v_1\leq v\leq v_2, \, w_1\leq v+p(\rho)\leq w_2 \rbrace
\end{equation}
is invariant for the standard Riemann solver $\mathcal{RS}$ for the Riemann problem (\ref{Riemann_problem_ARZ}); see Figure \ref{dominio_invariante}.\\
Let $\bar{V}\in[0,V_b]$ be the constant speed of the bus and let $\alpha\in(0,1)$ be fixed.\\
Let us define the function
\begin{equation}
\begin{split}
& h_\alpha:[\bar{V},+\infty) \rightarrow \mathbb{R}^+\cup \lbrace +\infty \rbrace,\\
& h_\alpha(v)=\begin{cases}
v+p\left(\dfrac{F_\alpha}{v-\bar{V}}\right) & \text{if } v>\bar{V},\\
+\infty & \text{if } v=\bar{V}.
\end{cases}
\end{split}
\end{equation}
Let $\underline{w}\in \mathbb{R}^+$ be fixed. We note that $h_\alpha(\underline{v})=\underline{w}$ if and only if
$$\underline{w}=\underline{v}+p\left(\underline{\rho}\right),$$
where the point $\left( \underline{\rho},\underline{v}\right)$ satisfies $\underline{\rho}(\underline{v}-\bar{V})=F_\alpha$. Therefore $h_\alpha$ gives the value of the Riemann invariant $w$ in the point $(\underline{\rho},\underline{v})\in(0,+\infty)\times(\bar{V},+\infty)$ for which we have $\underline{\rho} (\underline{v}-\bar{V})=F_\alpha$; see Figure \ref{dominio_invariante}.
\begin{figure}[hbtp]
\centering
\definecolor{cqcqcq}{rgb}{0.7529411764705882,0.7529411764705882,0.7529411764705882}
\definecolor{qqqqff}{rgb}{0.,0.,1.}
\begin{tikzpicture}[scale = 0.3, line cap=round,line join=round,>=triangle 45,x=1.0cm,y=1.0cm]
\draw[->,color=black] (-2.,0.) -- (30.,0.);
\draw[->,color=black] (0.,-1.) -- (0.,25.);
\clip(-2.,-1.5) rectangle (30.,25.);
\draw[line width=0.pt,color=cqcqcq,fill=cqcqcq,fill opacity=1.0] {[smooth,samples=50,domain=4.49:9.25] plot(\x,{0-(-17.3/8.15)*\x-0.0/8.15})} -- (9.25,11.11) {[smooth,samples=50,domain=9.25:4.49] -- plot(\x,{4.24*\x-\x^(1.5)})} -- (4.49,9.53) -- cycle;
\draw[line width=0.pt,color=cqcqcq,fill=cqcqcq,fill opacity=1.0] {[smooth,samples=50,domain=9.23:12.55] plot(\x,{5.16*\x-\x^(1.5)})} -- (12.55,8.78) {[smooth,samples=50,domain=12.55:9.23] -- plot(\x,{4.24*\x-\x^(1.5)})} -- (9.23,19.6) -- cycle;
\draw[line width=0.pt,color=cqcqcq,fill=cqcqcq,fill opacity=1.0] {[smooth,samples=50,domain=12.5:19.84] plot(\x,{5.16*\x-\x^(1.5)})} -- (19.84,14.03) {[smooth,samples=50,domain=19.84:12.5] -- plot(\x,{0-(-11.8/16.67)*\x-0.0/16.67})} -- (12.5,20.33) -- cycle;
\draw[smooth,samples=50,domain=0.001:30.0] plot(\x,{5.16227766*(\x)-(\x)^(1.5)});
\draw[smooth,samples=50,domain=0.001:30.0] plot(\x,{4.24*(\x)-(\x)^(1.5)});
\draw [domain=-2.:30.] plot(\x,{(--9.76--0.5*\x)/1.});
\draw (28.,0.) node[anchor=north west] {$\rho$};
\draw (20.86,21.2) node[anchor=north west] {$F_\alpha+\bar{V}\rho$};
\draw (-1.76,23.75) node[anchor=north west] {$\rho v$};
\draw (13.8,5.08) node[anchor=north west] {$w_1$};
\draw (23.2,9.7) node[anchor=north west] {$w_2$};
\draw (-1.93,11.75) node[anchor=north west] {$F_\alpha$};
\draw [domain=0.0:30.0] plot(\x,{(-0.--17.3*\x)/8.15});
\draw [domain=0.0:30.0] plot(\x,{(-0.--11.8*\x)/16.67});
\draw (10.74,23.81) node[anchor=north west] {$v_2$};
\draw (22.6,16.6) node[anchor=north west] {$v_1$};
\draw [domain=0.0:30.0] plot(\x,{(-0.--16.68*\x)/13.8});
\draw[smooth,samples=50,domain=0.001:30.0] plot(\x,{4.92*(\x)-(\x)^(1.5)});
\draw (17.5,6.08) node[anchor=north west] {$\underline{w}=h_\alpha(\underline{v})$};
\draw (14.,17.5) node[anchor=north west] {$(\underline{\rho},\underline{v})$};
\draw (17.75,22.3) node[anchor=north west] {$\underline{v}$};
\begin{scriptsize}
\draw [fill=qqqqff] (13.8,16.68) circle (5pt);
\end{scriptsize}
\end{tikzpicture}
\caption{The coloured area is the invariant domain $\mathcal{D}_{v_1,v_2,w_1,w_2}$. The point $\left(\underline{\rho},\underline{v}\right)$ satisfies $\underline{\rho} (\underline{v}-\bar{V})=F_\alpha$ and $\underline{v}+p\left(\underline{\rho}\right)=h_\alpha\left(\underline{v}\right)=\underline{w}$.}\label{dominio_invariante}
\end{figure}
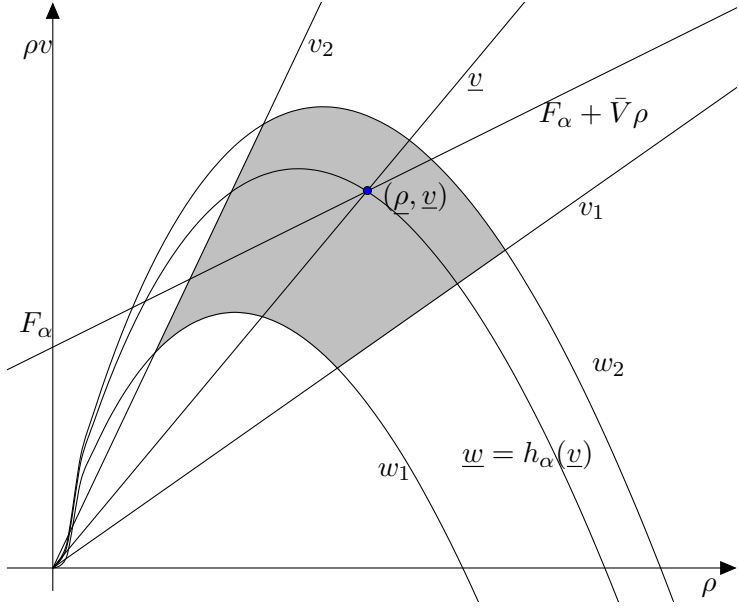
\begin{lemma}\label{lemma_monotonia_di_h_alfa}
Let us suppose that the hypotheses (\ref{ipotesi_pressione}) hold. Fixed $\alpha$, there exists $\bar{v}=\bar{v}(\alpha) \in [\bar{V},+\infty)$ such that $h_\alpha(v)$ is strictly decreasing in $[\bar{V},\bar{v})$ and strictly increasing in $(\bar{v},+\infty)$.
\end{lemma}
\begin{proof}
The function $\rho \rightarrow \rho p(\rho)$ is strictly convex, then its second derivative $2{p} '(\rho)+\rho {p}''(\rho)$ is strictly positive for every $\rho\in(0,+\infty)$. Hence the second derivative of ${h}_\alpha$, which is 
$$h''_\alpha(v)=\dfrac{F_\alpha}{(v-\bar{V})^3}\left[2{p} '\left(\dfrac{F_\alpha}{v-\bar{V}}\right)+\dfrac{F_\alpha}{v-\bar{V}}{p}''\left(\dfrac{F_\alpha}{v-\bar{V}}\right)\right],$$
is strictly positive for every $v>\bar{V}$ and the function 
\begin{equation*}
v \rightarrow {h} '_\alpha(v)=1-\dfrac{F_\alpha}{(v-\bar{V})^2}\, {p} '\left(\dfrac{F_\alpha}{v-\bar{V}}\right) \end{equation*}
is strictly increasing for every $v>\bar{V}$.\\
We have
$$\lim_{\rho\to +\infty}p(\rho)=+\infty,$$
otherwise there would be $M>0$ such that for every $\rho\in[0,+\infty)$ we have $p(\rho)<M$. Then $\rho p(\rho)< \rho M$ which is absurd, because the strictly convex function $\rho \rightarrow \rho p(\rho)$ cannot be less than a linear function. Therefore
\begin{equation*}
\begin{split}
& \lim_{v\to \bar{V}+}\left[ v+p\left(\dfrac{F_\alpha}{v-\bar{V}}\right)\right]=+\infty \; \text{ and }\\ & \lim_{v\to+\infty}h'_\alpha(v)=\lim_{v\to+\infty}\left[1-\dfrac{F_\alpha}{(v-\bar{V})^2}\, p'\left(\dfrac{F_\alpha}{v-\bar{V}}\right)\right]=1,
\end{split}
\end{equation*}
indeed, by the hypotheses (\ref{ipotesi_pressione}), we obtain
\begin{equation*}
\begin{split}
& p(\rho)+\rho \, p'(\rho)\geq p(0)+0=0 \Longrightarrow \rho \, p'(\rho)\geq -p(\rho) \Longrightarrow \\
& \Longrightarrow -\rho \, p'(\rho)\leq p(\rho) \Longrightarrow\\
&\Longrightarrow 0=\lim_{\rho\to 0}(-\rho\, p'(\rho))\leq \lim_{\rho \to 0} p(\rho) = 0. 
\end{split}
\end{equation*}
Hence the function $v\to h'_\alpha(v)$ is strictly increasing, when $v\to \bar{V}+$ its value is $-\infty$ and for $v\to +\infty$ is $1$.\\
On the other hand when $v\to \bar{V}+$ or $v \to +\infty$, the function $h_\alpha$ goes to $+\infty$. Therefore there must be a value $\bar{v}$ for which $h_\alpha$ is strictly increasing in $[\bar{V},\bar{v})$ and strictly decreasing in $(\bar{v},+\infty)$.
\end{proof}
\begin{prop}\label{coincidenza_risolutori}
Let $\bar{V}\leq v_1<v_2$, $0<w_1<w_2$, $v_2<w_2$ and $\alpha\in(0,1)$ be fixed. If $h_\alpha(v)\geq w_2$ for every $v\in [v_1,v_2]$, then the Riemann solvers $\mathcal{RS}^\alpha_1$ and $\mathcal{RS}^\alpha_2$ coincide with the Riemann solver standard $\mathcal{RS}$.
\end{prop}
\begin{proof}
Let $(\rho,v)$ be a point in $\mathcal{D}_{v_1,v_2,w_1,w_2}$. Hence $v_1\leq v\leq v_2$ and $w_1\leq v+p(\rho)\leq w_2$. By the hypotheses, we have
\begin{equation*}
\begin{split}
& v+p(\rho)\leq w_2 \leq h_\alpha(v) = v+p\left(\dfrac{F_\alpha}{v-\bar{V}}\right) \Longrightarrow p(\rho)\leq p\left(\dfrac{F_\alpha}{v-\bar{V}}\right)\Longrightarrow\\
&\Longrightarrow \rho \leq \dfrac{F_\alpha}{v-\bar{V}}\Longrightarrow \rho(v-\bar{V})\leq F_\alpha.
\end{split}
\end{equation*}
By the arbitrary choice of $(\rho,v)$ in the domain, we obtain
$$\sup\lbrace \rho (v-\bar{V}) :\, (\rho,v)\in\mathcal{D}_{v_1,v_2,w_1,w_2}\rbrace\leq F_\alpha.$$
Hence, if we choose $(\rho^l,v^l)$ and $(\rho^r,v^r)$ in  $\mathcal{D}_{v_1,v_2,w_1,w_2}$, the classical solution
$$\mathcal{RS}((\rho^l,v^l),(\rho^r,v^r))$$
satisfies the constraint for every $\lambda \in[v_1,v_2]$. Therefore the solutions given by the two Riemann solvers $\mathcal{RS}^\alpha_1$ and $\mathcal{RS}^\alpha_2$ coincide with the classical solution.
\end{proof}
\begin{cor}
Let $\bar{V}<v_1<v_2$, $0<w_1<w_2$, $v_2<w_2$ and $\alpha\in(0,1)$ be fixed. If $h_\alpha(v)\geq w_2$ for every $v\in [v_1,v_2]$, the domain $\mathcal{D}_{v_1,v_2,w_1,w_2}$ is invariant for both $\mathcal{RS}^\alpha_1$ and $\mathcal{RS}^\alpha_2$.
\end{cor}
\begin{proof}
It is a direct application of Propositions \ref{coincidenza_risolutori} and \ref{invariant_domain_standard}.
\end{proof}
The next theorems characterize the invariant domains of the Riemann solvers $\mathcal{RS}^\alpha_1$ and $\mathcal{RS}^\alpha_2$ when they are different from the standard Riemann solver. Figures \ref{figura_dominio_invariante_primo_risolutore} and \ref{figura_dominio_invariante_secondo_risolutore} are two examples of these domains.
\begin{theorem}\label{Proposizione_Dominio_invariante_R_alfa1_definitiva}
Let $0<v_1<v_2$, $0<w_1<w_2$, $v_2<w_2$ and $\alpha \in (0,1)$ be fixed. Let us assume that there exists $\bar{v}\in[\bar{V},v_2]$ for which $h_\alpha(\bar{v})<w_2$.
\begin{enumerate}
\item[(i)] If $v_2\leq\bar{V}$, then the set $\mathcal{D}_{v_1,v_2,w_1,w_2}$ is invariant for $\mathcal{RS}^\alpha_1$.
\item[(ii)] If $v_1\geq\bar{V}$, then the set $\mathcal{D}_{v_1,v_2,w_1,w_2}$ is invariant for $\mathcal{RS}^\alpha_1$ if and only if
\begin{equation}\label{condizioni_dominio_invariante1}
h_\alpha(v_1)\geq w_2 \; \text{ and } \; h_\alpha(v_2)\geq w_2.
\end{equation}
\item[(iii)] If $v_1<\bar{V}<v_2$, then the set $\mathcal{D}_{v_1,v_2,w_1,w_2}$ is invariant for $\mathcal{RS}^\alpha_1$ if and only if
\begin{equation}\label{condizioni_dominio_invariante1}
h_\alpha(v_2)\geq w_2.
\end{equation}
\end{enumerate}
\end{theorem}
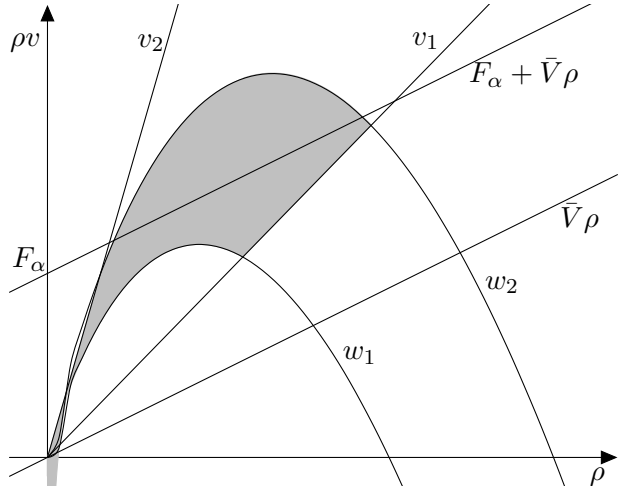
\begin{figure}[h]
\centering
\definecolor{cqcqcq}{rgb}{0.7529411764705882,0.7529411764705882,0.7529411764705882}
\begin{tikzpicture}[scale = 0.25,line cap=round,line join=round,>=triangle 45,x=1.0cm,y=1.0cm]
\draw[->,color=black] (-2.,0.) -- (30.,0.);
\draw[->,color=black] (0.,-1.) -- (0.,24.);
\clip(-2.5,-1.5) rectangle (30.,24.);
\draw[line width=0.pt,color=cqcqcq,fill=cqcqcq,fill opacity=1.0] {[smooth,samples=50,domain=2.81:10.32] plot(\x,{5.16*\x-\x^(1.5)})} -- (10.32,10.63) {[smooth,samples=50,domain=10.32:2.81] -- plot(\x,{4.24*\x-\x^(1.5)})} -- (2.81,9.8) -- cycle;
\draw[line width=0.pt,color=cqcqcq,fill=cqcqcq,fill opacity=1.0] {[smooth,samples=50,domain=10.32:17.05] plot(\x,{5.16227766*\x-\x^(1.5)})} -- (17.05,17.61) {[smooth,samples=50,domain=17.05:10.32] -- plot(\x,{0-(-14.6/14.12)*\x-0.0/14.12})} -- (10.32,20.12) -- cycle;
\draw[line width=0.pt,color=cqcqcq,fill=cqcqcq,fill opacity=1.0] {[smooth,samples=50,domain=0.0:2.83] plot(\x,{0-(-17.68/5.06)*\x-0.0/5.06})} -- (2.83,7.24) {[smooth,samples=50,domain=2.83:0.0] -- plot(\x,{4.242640687*\x-\x^(1.5)})} -- (0.,0.) -- cycle;
\draw[smooth,samples=50,domain=0.001:30.0] plot(\x,{5.16*(\x)-(\x)^(1.5)});
\draw[smooth,samples=50,domain=0.001:30.0] plot(\x,{4.24*(\x)-(\x)^(1.5)});
\draw [domain=-2.:30.] plot(\x,{(-0.--0.5*\x)/1.});
\draw [domain=-2.:30.] plot(\x,{(--9.76--0.5*\x)/1.});
\draw (28.,0.) node[anchor=north west] {$\rho$};
\draw (21.65,21.7) node[anchor=north west] {$F_\alpha+\bar{V}\rho$};
\draw (-2.5,23.) node[anchor=north west] {$\rho v$};
\draw (15.,6.4) node[anchor=north west] {$w_1$};
\draw (22.5,10.15) node[anchor=north west] {$w_2$};
\draw (-2.4,11.7) node[anchor=north west] {$F_\alpha$};
\draw [domain=0.0:30.0] plot(\x,{(-0.--17.68*\x)/5.06});
\draw [domain=0.0:30.0] plot(\x,{(-0.--14.6*\x)/14.12});
\draw (4.2,22.98) node[anchor=north west] {$v_2$};
\draw (18.66,23.08) node[anchor=north west] {$v_1$};
\draw (26.44,14.) node[anchor=north west] {$\bar{V}\rho$};
\end{tikzpicture}
\caption{Example of an invariant domain (the coloured area) for $\mathcal{RS}^\alpha_1$ for $v_1>\bar{V}$.}\label{figura_dominio_invariante_primo_risolutore}
\end{figure}
\begin{theorem}\label{Proposizione_Dominio_invariante_R_alfa2_definitiva}
Let $0<v_1<v_2$, $0<w_1<w_2$, $v_2<w_2$ and $\alpha\in(0,1)$ be fixed. Let us suppose that there exists $\bar{v}\in [\bar{V},v_2]$, such that $h_\alpha(\bar{v})<w_2$. 
\begin{enumerate}
\item[(i)] If $v_2\leq\bar{V}$, then the set $\mathcal{D}_{v_1,v_2,w_1,w_2}$ is invariant for $\mathcal{RS}^\alpha_2$.
\item[(ii)] If $v_1\geq \bar{V}$, then the set $\mathcal{D}_{v_1,v_2,w_1,w_2}$ is invariant for  $\mathcal{RS}^\alpha_2$ if and only if
\begin{equation}\label{condizioni_dominio_invariante_R_alfa_2}
h_\alpha(v_1)\geq w_2, \;\; h_\alpha(v_2)\leq w_2 \; \text{ and } \; h_\alpha(v)\geq w_1
\end{equation}
for every $v\in[v_1,v_2]$.
\item[(iii)] If $v_1<\bar{V}<v_2$, then the set $\mathcal{D}_{v_1,v_2,w_1,w_2}$ is invariant for  $\mathcal{RS}^\alpha_2$ if and only if
\begin{equation}\label{condizioni_dominio_invariante_R_alfa_2}
h_\alpha(v_2)\leq w_2 \; \text{ and } \; h_\alpha(v)\geq w_1
\end{equation}
for every $v\in[\bar{V},v_2]$.
\end{enumerate}
\end{theorem}
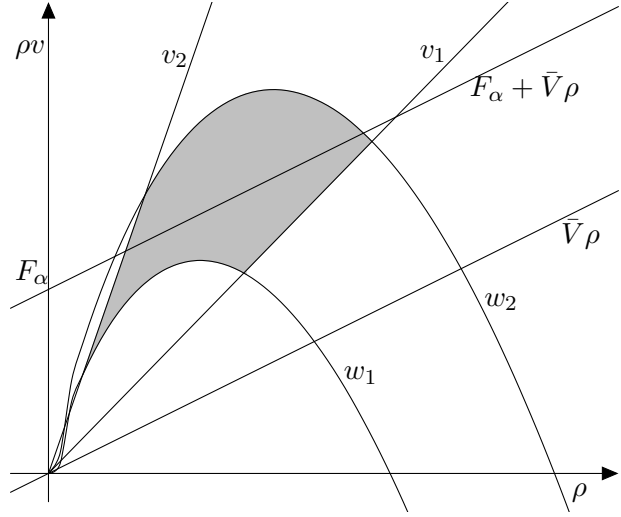
\begin{figure}[h]
\centering
\definecolor{cqcqcq}{rgb}{0.7529411764705882,0.7529411764705882,0.7529411764705882}
\begin{tikzpicture}[scale = 0.25,line cap=round,line join=round,>=triangle 45,x=1.0cm,y=1.0cm]
\draw[->,color=black] (-2.,0.) -- (30.,0.);
\draw[->,color=black] (0.,-1.5) -- (0.,25.);
\clip(-2.5,-2) rectangle (30.,25.);
\draw[line width=0.pt,color=cqcqcq,fill=cqcqcq,fill opacity=1.0] {[smooth,samples=50,domain=10.3:17.05] plot(\x,{5.16*\x-\x^(1.5)})} -- (17.05,17.61) {[smooth,samples=50,domain=17.05:10.3] -- plot(\x,{0-(-14.6/14.12)*\x-0.0/14.12})} -- (10.3,20.11) -- cycle;
\draw[line width=0.pt,color=cqcqcq,fill=cqcqcq,fill opacity=1.0] {[smooth,samples=50,domain=1.79:5.2] plot(\x,{0-(-17.72/6.1)*\x-0.0/6.1})} -- (5.2,10.2) {[smooth,samples=50,domain=5.2:1.79] -- plot(\x,{4.24*\x-\x^(1.5)})} -- (1.79,5.2) -- cycle;
\draw[line width=0.pt,color=cqcqcq,fill=cqcqcq,fill opacity=1.0] {[smooth,samples=50,domain=5.1:10.4] plot(\x,{5.16*\x-\x^(1.5)})} -- (10.4,10.58) {[smooth,samples=50,domain=10.4:5.1] -- plot(\x,{4.24*\x-\x^(1.5)})} -- (5.1,14.81) -- cycle;
\draw[smooth,samples=50,domain=0.001:30.0] plot(\x,{5.16*(\x)-(\x)^(1.5)});
\draw[smooth,samples=50,domain=0.001:30.0] plot(\x,{4.24*(\x)-(\x)^(1.5)});
\draw [domain=-2.:30.] plot(\x,{(-0.--0.5*\x)/1.});
\draw [domain=-2.:30.] plot(\x,{(--9.76--0.5*\x)/1.});
\draw (27.,0) node[anchor=north west] {$\rho$};
\draw (21.64,21.81) node[anchor=north west] {$F_\alpha+\bar{V}\rho$};
\draw (-2.3,23.5) node[anchor=north west] {$\rho v$};
\draw (15.,6.2) node[anchor=north west] {$w_1$};
\draw (22.4,10.1) node[anchor=north west] {$w_2$};
\draw (-2.3,11.87) node[anchor=north west] {$F_\alpha$};
\draw [domain=0.0:30.0] plot(\x,{(-0.--17.72*\x)/6.1});
\draw [domain=0.0:30.0] plot(\x,{(-0.--14.6*\x)/14.12});
\draw (5.4,22.98) node[anchor=north west] {$v_2$};
\draw (19.,23.24) node[anchor=north west] {$v_1$};
\draw (26.45,14.) node[anchor=north west] {$\bar{V}\rho$};
\end{tikzpicture}
\caption{Example of an invariant domain (the coloured area) for the Riemann solver $\mathcal{RS}^\alpha_2$ when $v_1>\bar{V}$.}\label{figura_dominio_invariante_secondo_risolutore}
\end{figure}
Let us define the sets $\mathcal{U}$ and $\mathcal{V}$ as
\begin{equation*}
\begin{split}
&\mathcal{U}:=\lbrace (\rho,v)\in (0,+\infty)\times(0,+\infty)\, : \, v \leq  \bar{V}\rbrace \; \text{ and}\\
& \mathcal{V}:= \lbrace (\rho,v)\in (0,+\infty)\times(0,+\infty)\, : \, v \geq \bar{V}\rbrace.
\end{split}
\end{equation*}
The proofs of Theorem \ref{Proposizione_Dominio_invariante_R_alfa1_definitiva} and Theorem \ref{Proposizione_Dominio_invariante_R_alfa2_definitiva} will be developed in the next sections and is divided in the following three parts:
\begin{enumerate}
\item the invariant domain is contained in $\mathcal{U}$;
\item the invariant domain is contained in $\mathcal{V}$;
\item the invariant domain has points in common with both $\mathcal{U}$ and $\mathcal{V}$.
\end{enumerate}
\subsection{The invariant domain is contained in $\mathcal{U}$}
If $v_2\leq \bar{V}$, then $\mathcal{D}_{v_1,v_2,w_1,w_2}\subseteq \mathcal{U}$; see Figure \ref{figura_dominio_invariante_caso_uno}. Indeed if $(\rho,v)$ belongs to $\mathcal{D}_{v_1,v_2,w_1,w_2}$, since $v\leq v_2$ and $ v_2 \leq \bar{V}$, we have $v \leq \bar{V}$.\\
Let $(\rho^l,v^l)$ and $(\rho^r,v^r)$ be two points in the domain $\mathcal{D}_{v_1,v_2,w_1,w_2}$. Since $\mathcal{D}_{v_1,v_2,w_1,w_2}$ is invariant for the standard Riemann solver $\mathcal{RS}$, we have 
$$f_1(\mathcal{RS}((\rho^l,v^l),(\rho^r,v^r))(\bar{V}))\leq \bar{V}\bar{\rho}((\rho^l,v^l),(\rho^r,v^r))(\bar{V}).$$
Therefore we are in the case in which the two Riemann solvers $\mathcal{RS}^\alpha_1$ and $\mathcal{RS}^\alpha_2$ coincide with $\mathcal{RS}$. Hence the domain $\mathcal{D}_{v_1,v_2,w_1,w_2}$ is invariant for both of them and this proves the points $(i)$ of Theorems \ref{Proposizione_Dominio_invariante_R_alfa1_definitiva} and \ref{Proposizione_Dominio_invariante_R_alfa2_definitiva}.
\begin{figure}[h]
\centering
\definecolor{cqcqcq}{rgb}{0.7529411764705882,0.7529411764705882,0.7529411764705882}
\begin{tikzpicture}[scale =0.25,line cap=round,line join=round,>=triangle 45,x=1.0cm,y=1.0cm]
\draw[->,color=black] (-2.,0.) -- (30.,0.);
\draw[->,color=black] (0.,-1.) -- (0.,23.);
\clip(-2.,-1.5) rectangle (30.,23.);
\draw[line width=0.pt,color=cqcqcq,fill=cqcqcq,fill opacity=1.0] {[smooth,samples=50,domain=15.16:17.0] plot(\x,{0-(-7.83/22.43)*\x-0.0/22.43})} -- (17.,2.03) {[smooth,samples=50,domain=17.0:15.16] -- plot(\x,{4.24*\x-\x^(1.5)})} -- (15.16,5.29) -- cycle;
\draw[line width=0.pt,color=cqcqcq,fill=cqcqcq,fill opacity=1.0] {[smooth,samples=50,domain=16.99:23.17] plot(\x,{0-(-7.83/22.43)*\x-0.0/22.43})} -- (23.17,2.76) {[smooth,samples=50,domain=23.17:16.99] -- plot(\x,{0-(-2.73/22.94)*\x-0.0/22.94})} -- (16.99,5.92) -- cycle;
\draw[line width=0.pt,color=cqcqcq,fill=cqcqcq,fill opacity=1.0] {[smooth,samples=50,domain=23.15:25.43] plot(\x,{5.16*\x-\x^(1.5)})} -- (25.43,3.03) {[smooth,samples=50,domain=25.43:23.15] -- plot(\x,{0-(-2.73/22.94)*\x-0.0/22.94})} -- (23.15,8.12) -- cycle;
\draw[smooth,samples=50,domain=0.001:30.0] plot(\x,{5.16*(\x)-(\x)^(1.5)});
\draw[smooth,samples=50,domain=0.001:30.0] plot(\x,{4.24*(\x)-(\x)^(1.5)});
\draw [domain=-2.:30.] plot(\x,{(-0.--0.5*\x)/1.});
\draw [domain=-2.:30.] plot(\x,{(--9.76--0.5*\x)/1.});
\draw (27.5,0.) node[anchor=north west] {$\rho$};
\draw (21.66,21.43) node[anchor=north west] {$F_\alpha+\bar{V}\rho$};
\draw (-2.3,22.5) node[anchor=north west] {$\rho v$};
\draw (9.49,12.4) node[anchor=north west] {$w_1$};
\draw (10.82,22.26) node[anchor=north west] {$w_2$};
\draw (-2.3,11.51) node[anchor=north west] {$F_\alpha$};
\draw [domain=0.0:30.0] plot(\x,{(-0.--7.83*\x)/22.43});
\draw [domain=0.0:30.0] plot(\x,{(-0.--2.73*\x)/22.94});
\draw (28.1,10.4) node[anchor=north west] {$v_2$};
\draw (27.68,3.7) node[anchor=north west] {$v_1$};
\draw (26.44,14.) node[anchor=north west] {$\bar{V}\rho$};
\end{tikzpicture}
\caption{Example of an invariant domain (the coloured area) for the Riemann solvers $\mathcal{RS}^\alpha_1$ and $\mathcal{RS}^\alpha_2$ when $\mathcal{D}_{v_1,v_2,w_1,w_2}$ is contained in $\mathcal{U}$.}\label{figura_dominio_invariante_caso_uno}
\end{figure}
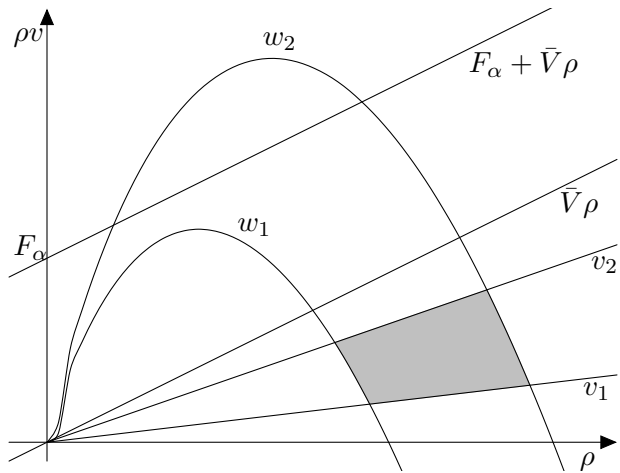
\subsection{The invariant domain is contained in $\mathcal{V}$}
If $v_1\geq\bar{V}$, then $\mathcal{D}_{v_1,v_2,w_1,w_2}\subseteq \mathcal{V}$. Indeed if $(\rho,v)$ belongs to $\mathcal{D}_{v_1,v_2,w_1,w_2}$, we have $v \geq \bar{V}$, because $v\geq v_1$ and $ v_1 \geq\bar{V}$; see Figures \ref{figura_dominio_invariante_primo_risolutore} and \ref{figura_dominio_invariante_secondo_risolutore}.
\subsubsection{The invariant domain for $\mathcal{RS}^\alpha_1$}
\begin{prop}\label{Proposizione_Dominio_invariante_R_alfa1}
Let $\bar{V}\leq v_1<v_2$, $0<w_1<w_2$, $v_2<w_2$ and $\alpha \in (0,1)$ be fixed. Let us assume that there exists $\bar{v}\in[v_1,v_2]$ such that $h_\alpha(\bar{v})<w_2$. The set $\mathcal{D}_{v_1,v_2,w_1,w_2}$ is invariant for $\mathcal{RS}^\alpha_1$ if and only if
\begin{equation}\label{condizioni_dominio_invariante1}
h_\alpha(v_1)\geq w_2 \; \text{ and } \; h_\alpha(v_2)\geq w_2.
\end{equation}
\end{prop}
\begin{proof} We split the proof in three parts.
\begin{enumerate}
\item Let us suppose that $h_\alpha(v_1)\geq w_2$ and $h_\alpha(v_2)\geq w_2$.\\
Since $\mathcal{D}_{v_1,v_2,w_1,w_2}$ is invariant for $\mathcal{RS}$, to prove that it is invariant for $\mathcal{RS}^\alpha_1$ we have to show that, for every initial data $(\rho^l,v^l)$ and $(\rho^r,v^r)$ for the constrained Riemann problem (\ref{problema_riemann_bus_reference_frame}) in $\mathcal{D}_{v_1,v_2,w_1,w_2}$, the points $(\hat{\rho},\hat{v})$ and $(\check{\rho}_1,\check{v}_1)$ are in $\mathcal{D}_{v_1,v_2,w_1,w_2}$.\\
By definition, we have
$$\hat{\rho}\hat{v}=F_\alpha +\bar{V}\hat{\rho}\; \text{ and } \; \hat{v}+p(\hat{\rho})=v^l+p(\rho^l).$$
We claim that $(\hat{\rho},\hat{v})\in \mathcal{D}_{v_1,v_2,w_1,w_2}$. Indeed, since $\hat{v}+p(\hat{\rho})=v^l+p(\rho^l)$, we find
\begin{equation}\label{dim_dom_inv_RS1_cond}
w_1\leq \hat{v}+p(\hat{\rho})\leq w_2.
\end{equation}
We have only to show that $v_1\leq \hat{v}\leq v_2$. Since $\hat{\rho}\hat{v}=F_\alpha +\bar{V}\hat{\rho}>\bar{V}\hat{\rho}$, the inequality $\hat{v}>\bar{V}$ holds. Therefore the function $h_\alpha$ is well defined in $\hat{v}$ and its value is
$$h_\alpha(\hat{v}) =\hat{v}+p\left(\dfrac{F_\alpha}{\hat{v}-\bar{V}}\right)=\hat{v}+p(\hat{\rho}).$$
Let $\tilde{v}\in(\bar{V},+\infty)$ be the minimum of the function $h_\alpha$, which exists by Lemma \ref{lemma_monotonia_di_h_alfa}. Hence $h_\alpha$ is strictly decreasing before $\tilde{v}$ and strictly increasing after $\tilde{v}$.\\
If $v_1<v_2\leq\tilde{v}$, then $h_\alpha(v_1)>h_\alpha(v_2)\geq h_\alpha(\tilde{v})$. By hypothesis, there exists $\bar{v}\in[v_1,v_2]$ such that $h_\alpha(\bar{v})<w_2$, but, since $\bar{v}\leq v_2$, we find $w_2>h_\alpha(\bar{v})\geq h_\alpha(v_2)$ which is a contradiction of the hypotheses (\ref{condizioni_dominio_invariante1}).\\
Similarly, if $\tilde{v}<v_1<v_2$, $h_\alpha$ is increasing in $[v_1,v_2]$ and hence, since $\bar{v}\geq v_1$, we find $w_2>h_\alpha(\bar{v})\geq h_\alpha(v_1)$ which is absurd.\\
Then it must be $v_1\leq \tilde{v}\leq v_2$. We know that $w_1\leq h_\alpha(\hat{v}) \leq w_2$. If it was $\hat{v}<v_1\leq \tilde{v}$, then we would have $h_\alpha(\hat{v})>h_\alpha(v_1)\geq w_2$ which is a contradiction of (\ref{dim_dom_inv_RS1_cond}). Similarly, if $\hat{v}>v_2\geq \tilde{v}$, then $h_\alpha(\hat{v})>h_\alpha(v_2)\geq w_2$ which is absurd.\\
Hence $v_1 \leq \hat{v} \leq v_2$ and this proves the claim.\\
Since  $(\check{\rho}_1,\check{v}_1)$ satisfies the same hypotheses of $(\hat{\rho},\hat{v})$, the proof is similar.
\item Let us suppose that $\mathcal{D}_{v_1,v_2,w_1,w_2}$ is invariant for $\mathcal{RS}^\alpha_1$ and, by contradiction, that $h_\alpha(v_1)<w_2$. Let $(\rho^*,v^*)\in \mathcal{D}_{v_1,v_2,w_1,w_2}$ be the solution to the system
\begin{equation*}
\begin{cases}
v+p(\rho)=w_2,\\
v=v_1.
\end{cases}
\end{equation*}
The point $(\rho^*,v^*)$ satisfies the inequality $\rho^* v^*>F_\alpha+\bar{V}\rho^*$ (see Figure \ref{fig_dim_primo_ris_caso_a}), indeed by the hypotheses (\ref{ipotesi_pressione}) we find
\begin{equation*}
\begin{split}
& h_\alpha(v_1)=v_1+p\left(\dfrac{F_\alpha}{v_1-\bar{V}}\right)<w_2=v^*+p(\rho^*)=v_1+p(\rho^*) \Longrightarrow\\
& \Longrightarrow \dfrac{F_\alpha}{v_1-\bar{V}}<\rho^* \Longrightarrow \rho^*(v^*-\bar{V})>F_\alpha.
\end{split}
\end{equation*}
Therefore the left trace of $\mathcal{RS}^\alpha_1((\rho^*,v^*),(\rho^*,v^*))$ in $\lambda=\bar{V}$ is $(\hat{\rho},\hat{v})$. By definition 
\begin{equation*}
\begin{split}
&\hat{v}=v^*+p(\rho^*)-p(\hat{\rho}) \Longrightarrow \hat{v}+p(\hat{\rho})=v^*+p(\rho^*) \; \text{ and }\\
&\hat{\rho}\hat{v}=F_\alpha+\bar{V}\hat{\rho}\Longrightarrow h_\alpha(\hat{v})=\hat{v}+p(\hat{\rho})=w_2.
\end{split}
\end{equation*}
Since $(\rho^*,v^*) \in \mathcal{D}_{v_1,v_2,w_1,w_2}$, even the point $(\hat{\rho},\hat{v})$ is in $\mathcal{D}_{v_1,v_2,w_1,w_2}$, because the domain is invariant. Hence we have $\hat{v}\geq v_1=v^*$ and we find
$$\hat{v}+p(\hat{\rho})=v^*+p(\rho^*) \Longrightarrow \rho^*\geq\hat{\rho}.$$
The equal cannot hold, because $(\rho^*,v^*)$ does not satisfy the constraint while $(\hat{\rho},\hat{v})$ does. Since the curve $\rho \to \rho L_1(\rho,\rho^*,v^*)$ is strictly concave and $(\hat{\rho},\hat{v})$ is its point with maximal density for which it is satisfied the condition $\hat{\rho}\hat{v}=F_\alpha+\bar{V}\hat{\rho}$, by Lemma \ref{funzione_concava_e_retta} all the points $(\rho,v)$ with density bigger then $\hat{\rho}$ should satisfy $\rho v<F_\alpha+\bar{V}\rho$, but this is a contradiction of the conditions satisfied by $(\rho^*,v^*)$.\\
Therefore the inequality $h_\alpha(v_1)\geq w_2$ must hold.
\item Let us suppose that $\mathcal{D}_{v_1,v_2,w_1,w_2}$ is invariant for $\mathcal{RS}^\alpha_1$ and let us suppose by contradiction that $h_\alpha(v_2)<w_2$. Let $(\rho^*,v^*)\in \mathcal{D}_{v_1,v_2,w_1,w_2}$ be the solution to the system
\begin{equation*}
\begin{cases}
v+p(\rho)=w_2,\\
v=v_2.
\end{cases}
\end{equation*}
The point $(\rho^*,v^*)$ is such that $\rho^* v^*>F_\alpha+\bar{V}\rho^*$ (see Figure \ref{fig_dim_primo_ris_caso_b}), indeed by the hypotheses (\ref{ipotesi_pressione}) we find
\begin{equation*}
\begin{split}
& h_\alpha(v_2)=v_2+p\left(\dfrac{F_\alpha}{v_2-\bar{V}}\right)<w_2=v^*+p(\rho^*)=v_2+p(\rho^*) \Longrightarrow\\
& \Longrightarrow \dfrac{F_\alpha}{v_2-\bar{V}}<\rho^* \Longrightarrow \rho^*(v^*-\bar{V})>F_\alpha.
\end{split}
\end{equation*}
Hence the right trace of $\mathcal{RS}^\alpha_1((\rho^*,v^*),(\rho^*,v^*))$ in $\lambda=\bar{V}$ is $(\check{\rho}_1,\check{v}_1)$. We note that 
\begin{equation*}
\begin{split}
&\check{v}_1+p(\check{\rho}_1)=v^*+p(\rho^*) \; \text{ and } \; \check{\rho}_1\check{v}_1=F_\alpha+\bar{V}\check{\rho}_1 \Longrightarrow\\
&\Longrightarrow h_\alpha(\check{v}_1)=\check{v}_1+p(\check{\rho}_1)=w_2.
\end{split}
\end{equation*}
The point $(\check{\rho}_1,\check{v}_1)$ belongs to the invariant domain $\mathcal{D}_{v_1,v_2,w_1,w_2}$, then $\check{v}_1\leq v_2=v^*$. Therefore the relations
\begin{equation*} 
\check{v}_1\leq v^* \; \text{ and } \; \check{v}_1+p(\check{\rho}_1)=v^*+p(\rho^*)
\end{equation*}
imply $\rho^*\leq\check{\rho}_1$.\\
The equal is not possible, because $(\rho^*,v^*)$ does not satisfy the constraint, while $(\check{\rho}_1,\check{v}_1)$ does. Since the curve $\rho \to \rho L_1(\rho,\rho^*,v^*)$ is strictly concave and $(\check{\rho}_1,\check{v}_1)$ is its point with minimal density for which it is satisfied the condition $\check{\rho}_1\check{v}_1=F_\alpha+\bar{V}\check{\rho}_1$, by Lemma \ref{funzione_concava_e_retta} all the points $(\rho,v)$ with a lower density then $\check{\rho}_1$ should be such that $\rho v<F_\alpha+\bar{V}\rho$, but this is a contradiction of the conditions satisfied by $(\rho^*,v^*)$.\\
Hence the inequality $h_\alpha(v_2)\geq w_2$ must hold.
\end{enumerate}
\end{proof}
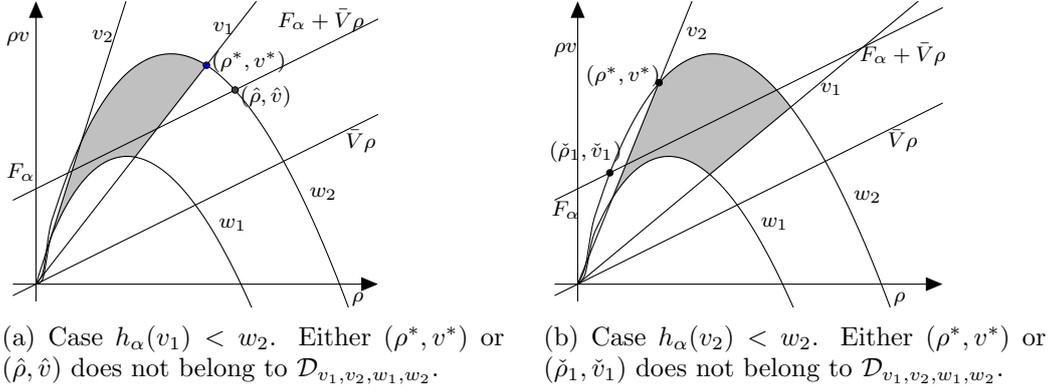
\begin{figure}
\centering
\begin{subfigure}[h]{0.45\linewidth}
\definecolor{cqcqcq}{rgb}{0.7529411764705882,0.7529411764705882,0.7529411764705882}
\definecolor{uququq}{rgb}{0.25098039215686274,0.25098039215686274,0.25098039215686274}
\definecolor{qqqqff}{rgb}{0.,0.,1.}
\begin{tikzpicture}[scale=0.15,line cap=round,line join=round,>=triangle 45,x=1.0cm,y=1.0cm]
\draw[->,color=black] (-2.,0.) -- (30.,0.);
\draw[->,color=black] (0.,-1.5) -- (0.,25.);
\clip(-3,-2) rectangle (30.,25.);
\draw[line width=0.pt,color=cqcqcq,fill=cqcqcq,fill opacity=1.0] {[smooth,samples=50,domain=8.7:14.96] plot(\x,{5.16*\x-\x^(1.5)})} -- (14.96,19.35) {[smooth,samples=50,domain=14.96:8.7] -- plot(\x,{0-(-19.35/14.96)*\x-0.0/14.96})} -- (8.7,19.25) -- cycle;
\draw[line width=0.pt,color=cqcqcq,fill=cqcqcq,fill opacity=1.0] {[smooth,samples=50,domain=1.16:4.1] plot(\x,{0-(-17.67/5.58)*\x-0.0/5.58})} -- (4.1,9.1) {[smooth,samples=50,domain=4.1:1.16] -- plot(\x,{4.24*\x-\x^(1.5)})} -- (1.16,3.67) -- cycle;
\draw[line width=0.pt,color=cqcqcq,fill=cqcqcq,fill opacity=1.0] {[smooth,samples=50,domain=3.99:8.8] plot(\x,{5.16*\x-\x^(1.5)})} -- (8.8,11.23) {[smooth,samples=50,domain=8.8:3.99] -- plot(\x,{4.24*\x-\x^(1.5)})} -- (3.99,12.62) -- cycle;
\draw[smooth,samples=50,domain=0.001:30.0] plot(\x,{5.16*(\x)-(\x)^(1.5)});
\draw[smooth,samples=50,domain=0.001:30.0] plot(\x,{4.24*(\x)-(\x)^(1.5)});
\draw [domain=-2.:30.] plot(\x,{(-0.--0.5*\x)/1.});
\draw [domain=-2.:30.] plot(\x,{(--8.43--0.5*\x)/1.});
\draw [domain=0.0:30.0] plot(\x,{(-0.--17.67*\x)/5.58});
\draw [domain=0.0:30.0] plot(\x,{(-0.--19.35*\x)/14.96});
\begin{scriptsize}
\draw (27.,0.) node[anchor=north west] {$\rho$};
\draw (20.5,25) node[anchor=north west] {$F_\alpha+\bar{V}\rho$};
\draw (-3.3,23.) node[anchor=north west] {$\rho v$};
\draw (15.35,6.52) node[anchor=north west] {$w_1$};
\draw (23.3,9.24) node[anchor=north west] {$w_2$};
\draw (-3.3,11.4) node[anchor=north west] {$F_\alpha$};
\draw (4.2,23.13) node[anchor=north west] {$v_2$};
\draw (14.8,23.9) node[anchor=north west] {$v_1$};
\draw (26.46,14.34) node[anchor=north west] {$\bar{V}\rho$};
\draw (14.7,21.32) node[anchor=north west] {$(\rho^*,v^*)$};
\draw (17.2,18.2) node[anchor=north west] {$(\hat{\rho},\hat{v})$};
\end{scriptsize}
\begin{scriptsize}
\draw [fill=qqqqff] (14.96,19.35) circle (8pt);
\draw [fill=uququq] (17.47,17.16) circle (8pt);
\end{scriptsize}
\end{tikzpicture}
\caption{Case $h_\alpha(v_1)<w_2$. Either $(\rho^*,v^*)$ or $(\hat{\rho},\hat{v})$ does not belong to $\mathcal{D}_{v_1,v_2,w_1,w_2}$.}\label{fig_dim_primo_ris_caso_a}
\end{subfigure}
\quad
\begin{subfigure}[h]{0.45\linewidth}
\definecolor{cqcqcq}{rgb}{0.7529411764705882,0.7529411764705882,0.7529411764705882}
\definecolor{qqqqff}{rgb}{0.,0.,1.}
\begin{tikzpicture}[scale =0.15,line cap=round,line join=round,>=triangle 45,x=1.0cm,y=1.0cm]
\draw[->,color=black] (-2.,0.) -- (32.,0.);
\draw[->,color=black] (0.,-1.5) -- (0.,25.);
\clip(-3.,-2) rectangle (35.,25.);
\draw[line width=0.pt,color=cqcqcq,fill=cqcqcq,fill opacity=1.0] {[smooth,samples=50,domain=11.6:18.68] plot(\x,{5.16*\x-\x^(1.5)})} -- (18.68,15.64) {[smooth,samples=50,domain=18.68:11.6] -- plot(\x,{0-(-15.63/18.67)*\x-0.0/18.67})} -- (11.6,20.37) -- cycle;
\draw[line width=0.pt,color=cqcqcq,fill=cqcqcq,fill opacity=1.0] {[smooth,samples=50,domain=3.07:7.25] plot(\x,{0-(-17.81/7.14)*\x-0.0/7.14})} -- (7.25,11.23) {[smooth,samples=50,domain=7.25:3.07] -- plot(\x,{4.24*\x-\x^(1.5)})} -- (3.07,7.64) -- cycle;
\draw[line width=0.pt,color=cqcqcq,fill=cqcqcq,fill opacity=1.0] {[smooth,samples=50,domain=7.15:11.7] plot(\x,{5.16*\x-\x^(1.5)})} -- (11.7,9.61) {[smooth,samples=50,domain=11.7:7.15] -- plot(\x,{4.24*\x-\x^(1.5)})} -- (7.15,17.79) -- cycle;
\draw[smooth,samples=50,domain=0.001:32.0] plot(\x,{5.16*(\x)-(\x)^(1.5)});
\draw[smooth,samples=50,domain=0.001:32.0] plot(\x,{4.24*(\x)-(\x)^(1.5)});
\draw [domain=-2.:32.] plot(\x,{(-0.--0.5*\x)/1.});
\draw [domain=-2.:32.] plot(\x,{(--8.43--0.5*\x)/1.});
\draw [domain=0.0:32.0] plot(\x,{(-0.--17.81*\x)/7.14});
\draw [domain=0.0:32.0] plot(\x,{(-0.--15.63*\x)/18.67});
\begin{scriptsize}
\draw (27,0.) node[anchor=north west] {$\rho$};
\draw (24,22) node[anchor=north west] {$F_\alpha+\bar{V}\rho$};
\draw (-2.8,22) node[anchor=north west] {$\rho v$};
\draw (15.36,6.63) node[anchor=north west] {$w_1$};
\draw (23.51,8.95) node[anchor=north west] {$w_2$};
\draw (-3,8) node[anchor=north west] {$F_\alpha$};
\draw (8.82,23.43) node[anchor=north west] {$v_2$};
\draw (20.56,18.51) node[anchor=north west] {$v_1$};
\draw (26.47,14.36) node[anchor=north west] {$\bar{V}\rho$};
\draw (0.,20.05) node[anchor=north west] {$(\rho^*,v^*)$};
\draw (-3.2,13.38) node[anchor=north west] {$(\check{\rho}_1,\check{v}_1)$};
\end{scriptsize}
\begin{scriptsize}
\draw [fill=black] (7.14,17.81) circle (8pt);
\draw [fill=black] (2.82,9.84) circle (8pt);
\end{scriptsize}
\end{tikzpicture}
\caption{Case $h_\alpha(v_2)<w_2$. Either $(\rho^*,v^*)$ or $(\check{\rho}_1,\check{v}_1)$ does not belong to $\mathcal{D}_{v_1,v_2,w_1,w_2}$.}\label{fig_dim_primo_ris_caso_b}
\end{subfigure}
\caption{Representation of the points used in the proof of Proposition \ref{Proposizione_Dominio_invariante_R_alfa1}. The invariant domain is the coloured area.}
\end{figure}
This proves the point $(ii)$ of Theorem \ref{Proposizione_Dominio_invariante_R_alfa1_definitiva}.
\subsubsection{The invariant domain of $\mathcal{RS}^\alpha_2$}
\begin{lemma}\label{lemma_1_dominio_invariante_RS_2}
Let $\bar{V}\leq v_1<v_2$, $0<w_1<w_2$, $v_2<w_2$ and $\alpha\in(0,1)$ be fixed. Let us suppose that there exists $\bar{v}\in[v_1,v_2]$ such that $h_\alpha(\bar{v})<w_2$. If the set $\mathcal{D}_{v_1,v_2,w_1,w_2}$ is invariant for $\mathcal{RS}^\alpha_2$, then $h_\alpha(v_1)\geq w_2$.
\end{lemma}
The proof is the same of part 2 of Proposition \ref{Proposizione_Dominio_invariante_R_alfa1}.
\begin{lemma}\label{lemma_2_dominio_invariante_RS2}
Let $\bar{V}\leq v_1<v_2$, $0<w_1<w_2$, $v_2<w_2$ and $\alpha\in(0,1)$ be fixed. Let us suppose that there exists $\bar{v}\in[v_1,v_2]$ such that $h_\alpha(\bar{v})<w_2$. If the set $\mathcal{D}_{v_1,v_2,w_1,w_2}$ is invariant for $\mathcal{RS}^\alpha_2$, then $h_\alpha(v)\geq w_1$ for every $v\in[v_1,v_2]$.
\end{lemma}
\begin{proof}
Assume by contradiction that there exists $\tilde{v}\in[v_1,v_2]$ such that $h_\alpha(\tilde{v})<w_1$. Let $(\rho^*,v^*)\in\mathcal{D}_{v_1,v_2,w_1,w_2}$ be the solution to the system
\begin{equation*}
\begin{cases}
v+p(\rho)=w_2,\\
v=\tilde{v}.
\end{cases}
\end{equation*}
We note that
\begin{equation*}
\begin{split}
& h_\alpha(\tilde{v})=\tilde{v}+p\left(\dfrac{F_\alpha}{\tilde{v}-\bar{V}}\right)<w_1<w_2=v^*+p(\rho^*) \Longrightarrow\\
&\Longrightarrow \tilde{v}+p\left(\dfrac{F_\alpha}{\tilde{v}-\bar{V}}\right)<v^*+p(\rho^*) \Longrightarrow \rho^*(v^*-\bar{V})>F_\alpha.
\end{split}
\end{equation*}
Hence the right trace of $\mathcal{RS}^\alpha_2((\rho^*,v^*),(\rho^*,v^*))$ in $\lambda=\bar{V}$ is $(\check{\rho}_2,\check{v}_2)$ which does not belong to $\mathcal{D}_{v_1,v_2,w_1,w_2}$; see Figure \ref{fig_dim_lemma_2_sec_ris}. Indeed, we have
$$\check{v}_2 =v^*=\tilde{v}\Longrightarrow h_\alpha(\check{v}_2)=h_\alpha(\tilde{v})<w_1.$$
This condition is a contradiction of the hypothesis of invariance of the domain $\mathcal{D}_{v_1,v_2,w_1,w_2}$ for $\mathcal{RS}^\alpha_2$, hence the inequality $h_\alpha(v)\geq w_1$ must hold for every $v\in[v_1,v_2]$.
\end{proof}
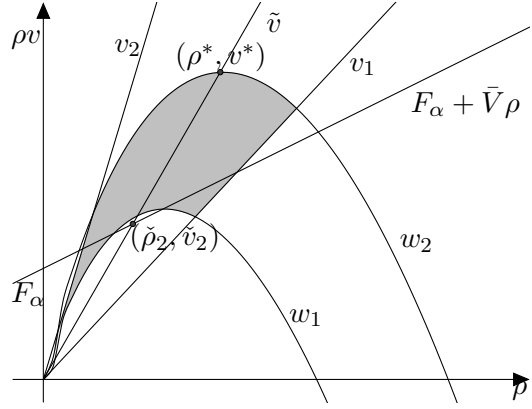
\begin{figure}[h]
\centering
\definecolor{cqcqcq}{rgb}{0.7529411764705882,0.7529411764705882,0.7529411764705882}
\definecolor{uququq}{rgb}{0.25098039215686274,0.25098039215686274,0.25098039215686274}
\begin{tikzpicture}[scale = 0.2, line cap=round,line join=round,>=triangle 45,x=1.0cm,y=1.0cm]
\draw[->,color=black] (-2.,0.) -- (32.,0.);
\draw[->,color=black] (0.,-1.5) -- (0.,25.);
\clip(-3.,-1.5) rectangle (32.,25.);
\draw[line width=0.pt,color=cqcqcq,fill=cqcqcq,fill opacity=1.0] {[smooth,samples=50,domain=10.01:16.68] plot(\x,{5.16*\x-\x^(1.5)})} -- (16.68,17.99) {[smooth,samples=50,domain=16.68:10.01] -- plot(\x,{0-(-16.76/15.54)*\x-0.0/15.54})} -- (10.01,20.) -- cycle;
\draw[line width=0.pt,color=cqcqcq,fill=cqcqcq,fill opacity=1.0] {[smooth,samples=50,domain=3.29:10.1] plot(\x,{5.16*\x-\x^(1.5)})} -- (10.1,10.75) {[smooth,samples=50,domain=10.1:3.29] -- plot(\x,{4.24*\x-\x^(1.5)})} -- (3.29,11.01) -- cycle;
\draw[line width=0.pt,color=cqcqcq,fill=cqcqcq,fill opacity=1.0] {[smooth,samples=50,domain=0.8:3.35] plot(\x,{0-(-17.22/5.14)*\x-0.0/5.14})} -- (3.35,8.08) {[smooth,samples=50,domain=3.35:0.8] -- plot(\x,{4.24*\x-\x^(1.5)})} -- (0.8,2.67) -- cycle;
\draw[smooth,samples=50,domain=0.001:32.0] plot(\x,{5.16*(\x)-(\x)^(1.5)});
\draw[smooth,samples=50,domain=0.001:32.0] plot(\x,{4.24*(\x)-(\x)^(1.5)});
\draw [domain=-2.:32.] plot(\x,{(--7.36--0.5*\x)/1.});
\draw (30.17,0.54) node[anchor=north west] {$\rho$};
\draw (23.5,20) node[anchor=north west] {$F_\alpha+\bar{V}\rho$};
\draw (-2.8,23.82) node[anchor=north west] {$\rho v$};
\draw (15.5,5.32) node[anchor=north west] {$w_1$};
\draw (22.77,10.26) node[anchor=north west] {$w_2$};
\draw (-2.8,7.12) node[anchor=north west] {$F_\alpha$};
\draw [domain=0.0:32.0] plot(\x,{(-0.--17.22*\x)/5.14});
\draw [domain=0.0:32.0] plot(\x,{(-0.--16.76*\x)/15.54});
\draw (4,23.24) node[anchor=north west] {$v_2$};
\draw (19.46,21.96) node[anchor=north west] {$v_1$};
\draw [domain=0.0:32.0] plot(\x,{(-0.--14.*\x)/8.});
\draw (14,25) node[anchor=north west] {$\tilde{v}$};
\draw (8.,23) node[anchor=north west] {$(\rho^*,v^*)$};
\draw (5.,11.) node[anchor=north west] {$(\check{\rho}_2,\check{v}_2)$};
\begin{scriptsize}
\draw [fill=uququq] (11.64,20.37) circle (5pt);
\draw [fill=uququq] (5.89,10.31) circle (5pt);
\end{scriptsize}
\end{tikzpicture}
\caption{Representation of the points used in the proof of Lemma \ref{lemma_2_dominio_invariante_RS2}. The invariant domain is the coloured area. If, by contradiction, there was $\tilde{v}\in[v_1,v_2]$ such that $h_\alpha(\tilde{v})<w_1$, the point $(\check{\rho}_2,\check{v}_2)$ would not be in the domain $\mathcal{D}_{v_1,v_2,w_1,w_2}$.}\label{fig_dim_lemma_2_sec_ris} 
\end{figure}
\begin{lemma}\label{lemma_3_dominio_invariante_2}
Let $\bar{V}\leq v_1<v_2$, $0<w_1<w_2$, $v_2<w_2$ and $\alpha\in(0,1)$ be fixed. Let us suppose that there exists $\bar{v}\in[v_1,v_2]$ such that $h_\alpha(\bar{v})<w_2$. If the set $\mathcal{D}_{v_1,v_2,w_1,w_2}$ is invariant for $\mathcal{RS}^\alpha_2$, then $h_\alpha(v_2)\leq w_2$.
\end{lemma}
\begin{proof}
Assume by contradiction that $h_\alpha(v_2)>w_2$. Let $(\rho^r,v^r)\in\mathcal{D}_{v_1,v_2,w_1,w_2}$ be the solution to the system (see Figure \ref{fig_dim_lemma_3_sec_ris})
\begin{equation*}
\begin{cases}
v+p(\rho)=w_2,\\
v=v_2,
\end{cases}
\end{equation*}
and let $(\rho^l,v^l)\in \mathcal{D}_{v_1,v_2,w_1,w_2}$ be the solution to
\begin{equation*}
\begin{cases}
v+p(\rho)=w_2,\\
v=v_1.
\end{cases}
\end{equation*}
The points $(\rho^l,v^l)$ and $(\rho^r,v^r)$ are connected by the standard Riemann solver with a rarefaction, because $v^l=v_1<v_2=v^r$.\\
We claim that $\rho^l \geq \hat{\rho}$ and $\rho^r<\check{\rho}_1$.\\
Indeed, let $\tilde{v}$ be the minimum of the function $h_\alpha$. We have $v^l\leq \tilde{v}$, otherwise, by Lemma \ref{lemma_2_dominio_invariante_RS2}, for every $v\in [v_1,v_2]$ we would have $h_\alpha(v)\geq h_\alpha(v_1)=w_2$, but this contradicts the hypothesis on $\bar{v}$.\\
Let us show that $\hat{v}\geq v^l$.\\
If $\hat{v}\geq \tilde{v}$, then we have $v^l\leq \hat{v}$. By Lemma \ref{lemma_1_dominio_invariante_RS_2}, the inequality $h_\alpha(v_1)\geq w_2$ holds and by the definition of $(\rho^l,v^l)$, we have $w_2=v^l+p(\rho^l)=\hat{v}+p(\hat{\rho})=h_\alpha(\hat{v})$. Hence if $\hat{v}<\tilde{v}$, we find
$$h_\alpha(\hat{v})=w_2\leq h_\alpha(v_1)=h_\alpha(v^l)\Longrightarrow \hat{v}\geq v^l.$$
In both these cases we have $\hat{v}\geq v^l$, which is equivalent to $\rho^l \geq \hat{\rho}$.\\
Similarly, if it were $v_2\leq\tilde{v}$, for every $v\in[v_1,v_2]$ we would have $h_\alpha(v)\geq h_\alpha(v_2)>w_2$ and this is a contradiction of the hypothesis on $\bar{v}$. Therefore $v_2>\tilde{v}$. If $\check{v}_1\leq \tilde{v}$, then clearly $v_2>\check{v}_1$, while if $\check{v}_1 >\tilde{v}$, since
$$h_\alpha(\check{v}_1)=\check{v}_1+p(\check{\rho}_1)=v^l+p(\rho^l)=w_2<h_\alpha(v_2),$$
we find $v^r=v_2 >\check{v}_1$ which is equivalent to $\rho^r<\check{\rho}_1$, because
$$v^r+p(\rho^r)=w_2=v^l+p(\rho^l)=\check{v}_1+p(\check{\rho}_1).$$
This proves the claim.\\
The line passing through the points $(\hat{\rho},\hat{v})$ and $(\check{\rho}_1,\check{v}_1)$ has slope $\bar{V}$. Therefore the trace of $\mathcal{RS}((\rho^l,v^l),(\rho^r,v^r))$ in $\lambda=\bar{V}$ is one of the points $(\rho^\sigma,v^\sigma)$ of the rarefaction linking $(\rho^l,v^l)$ with $(\rho^r,v^r)$, because by Lagrange Theorem, there is  a value $\rho^\sigma\in (\check{\rho}_1,\hat{\rho})\subset (\rho^r,\rho^l)$ such that
$$\dfrac{d}{d\rho}(\rho L_1(\rho,\rho^l,v^l))\big|_{\rho=\rho^\sigma}=\bar{V}.$$
By Lemma \ref{funzione_concava_e_retta}, this point does not satisfy the constraint. Hence we find that the right trace of $\mathcal{RS}^\alpha_2((\rho^l,v^l),(\rho^r,v^r))$ in $\lambda=\bar{V}$ is $(\check{\rho}_2,\check{v}_2)$.\\
Since $\check{v}_2=v^r=v_2$, we have
$$\check{v}_2+p(\check{\rho}_2)=h_\alpha(\check{v}_2)=h_\alpha(v_2)>w_2.$$
Therefore $(\check{\rho}_2,\check{v}_2)$ does not belong to $\mathcal{D}_{v_1,v_2,w_1,w_2}$, but this is a contradiction with the hypothesis of invariance of the domain for the Riemann solver $\mathcal{RS}^\alpha_2$.
\end{proof}
\begin{figure}[h]
\centering
\definecolor{cqcqcq}{rgb}{0.7529411764705882,0.7529411764705882,0.7529411764705882}
\definecolor{xdxdff}{rgb}{0.49019607843137253,0.49019607843137253,1.}
\definecolor{qqqqff}{rgb}{0.,0.,1.}
\definecolor{uququq}{rgb}{0.25098039215686274,0.25098039215686274,0.25098039215686274}
\begin{tikzpicture}[scale =0.25,line cap=round,line join=round,>=triangle 45,x=1.0cm,y=1.0cm]
\draw[->,color=black] (-2.,0.) -- (30.,0.);
\draw[->,color=black] (0.,-1.5) -- (0.,25.);
\clip(-3.,-2) rectangle (30.,25.);
\draw[line width=0.pt,color=cqcqcq,fill=cqcqcq,fill opacity=1.0] {[smooth,samples=50,domain=0.0:2.01] plot(\x,{0-(-17.4/4.72)*\x-0.0/4.72})} -- (2.01,5.68) {[smooth,samples=50,domain=2.01:0.0] -- plot(\x,{4.24*\x-\x^(1.5)})} -- (0.,0.) -- cycle;
\draw[line width=0.pt,color=cqcqcq,fill=cqcqcq,fill opacity=1.0] {[smooth,samples=50,domain=11.11:18.09] plot(\x,{5.16*\x-\x^(1.5)})} -- (18.09,16.44) {[smooth,samples=50,domain=18.09:11.11] -- plot(\x,{0-(-15.3/16.84)*\x-0.0/16.84})} -- (11.11,20.32) -- cycle;
\draw[line width=0.pt,color=cqcqcq,fill=cqcqcq,fill opacity=1.0] {[smooth,samples=50,domain=2.01:11.2] plot(\x,{5.16*\x-\x^(1.5)})} -- (11.2,10.03) {[smooth,samples=50,domain=11.2:2.01] -- plot(\x,{4.24*\x-\x^(1.5)})} -- (2.01,7.52) -- cycle;
\draw[smooth,samples=50,domain=0.001:30.0] plot(\x,{5.16*(\x)-(\x)^(1.5)});
\draw[smooth,samples=50,domain=0.001:30.0] plot(\x,{4.24*(\x)-(\x)^(1.5)});
\draw [domain=-2.:30.] plot(\x,{(--13.9--0.2*\x)/1.});
\draw (27,0.) node[anchor=north west] {$\rho$};
\draw (23.81,18.8) node[anchor=north west] {$F_\alpha+\bar{V}\rho$};
\draw (-2.5,23.92) node[anchor=north west] {$\rho v$};
\draw (13.2,5.6) node[anchor=north west] {$w_1$};
\draw (22.74,10.36) node[anchor=north west] {$w_2$};
\draw (-2.5,14.) node[anchor=north west] {$F_\alpha$};
\draw [domain=0.0:30.0] plot(\x,{(-0.--17.4*\x)/4.72});
\draw [domain=0.0:30.0] plot(\x,{(-0.--15.30*\x)/16.84});
\draw (4.1,23.39) node[anchor=north west] {$v_2$};
\draw (24.21,23.12) node[anchor=north west] {$v_1$};
\draw (18.,17.8) node[anchor=north west] {$(\rho^l,v^l)$};
\draw (16.1,19.9) node[anchor=north west] {$(\hat{\rho},\hat{v})$};
\draw (4.6,15.1) node[anchor=north west] {$(\check{\rho}_1,\check{v}_1)$};
\draw (2.05,8.95) node[anchor=north west] {$(\rho^r,v^r)$};
\draw (-0.5,17.14) node[anchor=north west] {$(\check{\rho}_2,\check{v}_2)$};
\draw (8.63,23.) node[anchor=north west] {$(\rho^\sigma,v^\sigma)$};
\begin{scriptsize}
\draw [fill=uququq] (18.09,16.44) circle (5pt);
\draw [fill=uququq] (17.3,17.35) circle (5pt);
\draw [fill=uququq] (5.17,14.93) circle (5pt);
\draw [fill=qqqqff] (2.01,7.37) circle (5pt);
\draw [fill=xdxdff] (3.98,14.7) circle (5pt);
\draw [fill=xdxdff] (11.20,20.33) circle (5pt);
\end{scriptsize}
\end{tikzpicture}
\caption{Representation of the points used in the proof of Lemma \ref{lemma_3_dominio_invariante_2}. The invariant domain is the coloured area. If, by contradiction, the inequality $h_\alpha(v_2)>w_2$ held, the point $(\check{\rho}_2,\check{v}_2)$ could be outside the set $\mathcal{D}_{v_1,v_2,w_1,w_2}$.}\label{fig_dim_lemma_3_sec_ris} 
\end{figure}
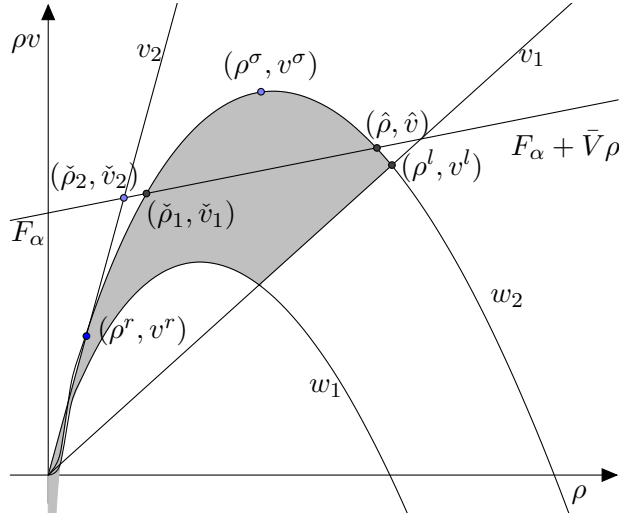
\begin{prop}\label{proposizione_dominio_invariante_R_alfa_2}
Let $\bar{V}\leq v_1<v_2$, $0<w_1<w_2$, $v_2<w_2$ and $\alpha\in(0,1)$ be fixed. Let us suppose that there exists $\bar{v}\in [v_1,v_2]$, such that $h_\alpha(\bar{v})<w_2$. The set $\mathcal{D}_{v_1,v_2,w_1,w_2}$ is invariant for the Riemann solver $\mathcal{RS}^\alpha_2$ if and only if
\begin{equation}\label{condizioni_dominio_invariante_R_alfa_2}
h_\alpha(v_1)\geq w_2, \; \; h_\alpha(v_2)\leq w_2 \; \text{ and } \; h_\alpha(v)\geq w_1\; \text{ for every } \; v\in[v_1,v_2].
\end{equation}
\end{prop}
\begin{proof}
The three Lemmas \ref{lemma_1_dominio_invariante_RS_2}, \ref{lemma_2_dominio_invariante_RS2} and \ref{lemma_3_dominio_invariante_2} prove that if the domain $\mathcal{D}_{v_1,v_2,w_1,w_2}$ is invariant for $\mathcal{RS}^\alpha_2$, then the inequalities (\ref{condizioni_dominio_invariante_R_alfa_2}) hold.\\
We have to show the vice versa. We split the proof in two parts and postpone at the end the proofs of the claims that we will do.\\
Let $(\rho^l,v^l)$ and $(\rho^r,v^r)$ be two fixed points in $\mathcal{D}_{v_1,v_2,w_1,w_2}$ and let $(\rho^m,v^m)$ be the intermediate state of the classical solution. Since $\mathcal{D}_{v_1,v_2,w_1,w_2}$ is invariant for $\mathcal{RS}$, we have only to show that the left and right traces of $\mathcal{RS}^\alpha_2((\rho^l,v^l),(\rho^r,v^r))$ in $\lambda=\bar{V}$ belong to $\mathcal{D}_{v_1,v_2,w_1,w_2}$.\\
If $\mathcal{RS}^\alpha_2((\rho^l,v^l),(\rho^r,v^r))$ coincides with $\mathcal{RS}((\rho^l,v^l),(\rho^r,v^r))$, there is nothing to prove. Hence we assume that $\mathcal{RS}^\alpha_2((\rho^l,v^l),(\rho^r,v^r))$ is not the classical solution. In this case $(\hat{\rho},\hat{v})$ and $(\check{\rho}_2,\check{v}_2)$ are respectively the left and the right trace of the solution in $\bar{V}$.
\begin{enumerate}
\item If $(\hat{\rho},\hat{v})$ does not belong to $\mathcal{D}_{v_1,v_2,w_1,w_2}$, let $(\rho^*,v^*)\in  \mathcal{D}_{v_1,v_2,w_1,w_2}$ be a point of $L_1(\rho,\rho^l,v^l)$. We claim (\textit{Claim} (i)) that
\begin{equation}\label{prop_dom_inv_RS2_claim_1}
\rho^* v^*\leq F_\alpha+\bar{V}\rho^*.
\end{equation}
Since for $(\rho^l,v^l)$ and $(\rho^m,v^m)$ the same hypotheses of $(\rho^*,v^*)$ hold, they satisfy the constraint.\\
The points $(\rho^m,v^m)$ and $(\rho^r,v^r)$ are joined with the classical Riemann solver by a contact discontinuity which propagates with speed $v^m=v^r$. Since $\bar{V}\leq v_1\leq v^r$, we find $v^r\geq\bar{V}$. If the equal holds, the left and right traces of $\mathcal{RS}((\rho^l,v^l),(\rho^r,v^r))$ in $\lambda=\bar{V}$ are respectively $(\rho^m,v^m)$ and $(\rho^r,v^r)$. Both these points satisfy the constraint, because $v^m=v^r=\bar{V}$, then for example for $(\rho^r,v^r)$ we have
$$\rho^r v^r=\rho^r \bar{V}<\rho^r \bar{V}+F_\alpha.$$
If the propagation speed is higher then the bus velocity and the classical solution connects $(\rho^l,v^l)$ and $(\rho^m,v^m)$ with a shock, then the trace of $\mathcal{RS}((\rho^l,v^l),(\rho^r,v^r))$ in $\lambda=\bar{V}$ is $(\rho^l,v^l)$ or $(\rho^m,v^m)$ (see Figure \ref{fig_dim_prop_dom_inv_sec_ris_caso_due}). Since both these points satisfy the constraint, the solution given by $\mathcal{RS}^\alpha_2$ coincides with the classical one.\\
If the propagation speed is higher then the bus velocity and the standard Riemann solver links $(\rho^l,v^l)$ and $(\rho^m,v^m)$ with a rarefaction (see Figure \ref{fig_dim_prop_dom_inv_sec_ris_caso_tre}), the trace in $\bar{V}$ is one of the points of the rarefaction. All these points satisfy the constraint, indeed, like $(\rho^*,v^*)$, they belong to the Lax curve of the first family passing through $(\rho^l,v^l)$ and they are inside the set $\mathcal{D}_{v_1,v_2,w_1,w_2}$, because it is invariant for the standard Riemann solver. Hence, even in this case the solution obtained with $\mathcal{RS}^\alpha_2$ coincides with the classical solution. This contradicts the assumption that we have made at the beginning of the proof.
\item If $(\check{\rho}_2,\check{v}_2)$ is not in $\mathcal{D}_{v_1,v_2,w_1,w_2}$, we claim (\textit{Claim} (ii)) that each point $(\rho^*,v^*)\in \mathcal{D}_{v_1,v_2,w_1,w_2}$ belonging to the curve $\rho \to \rho L_2(\rho, \rho^r,v^r)$, is such that
\begin{equation}
\rho^* v^*\leq F_\alpha+\bar{V}\rho^*.
\end{equation}
If $v^r=\bar{V}$, the left and right traces of $\mathcal{RS}((\rho^l,v^l), (\rho^r,v^r))$ in $\lambda=\bar{V}$ are respectively $(\rho^m,v^m)$ and $(\rho^r,v^r)$. Both these points satisfy the constraint, because $v^m=v^r=\bar{V}.$\\
If $v^l=v^r>\bar{V}$, then $\mathcal{RS}$ connects the points $(\rho^l,v^l)$ and $(\rho^r,v^r)$ with a contact discontinuity and there are no intermediate states. Since $v^r>\bar{V}$, the propagation speed of the wave is higher then the bus velocity. Therefore the classical solution in $\bar{V}$ coincides with $(\rho^l,v^l)$ which satisfies the constraint by \textit{Claim} (ii), because it satisfies the same hypotheses of $(\rho^*,v^*)$.\\
If $v^l>v^r$ (see Figure \ref{fig_dim_prop_sec_ris_caso_quattro}), then a shock of the first family propagating with speed
$$s=\dfrac{\rho^l v^l-\rho^m v^m}{\rho^l-\rho^m}$$
appears between $(\rho^l,v^l)$ and $(\rho^m,v^m)$.\\
The trace of $\mathcal{RS}((\rho^l,v^l),(\rho^r,v^r))$ in $\bar{V}$ is given by $(\rho^m,v^m)$ if $\bar{V}>s$, i.e. if the bus is travelling faster then the shock. In this case the solution given by $\mathcal{RS}^\alpha_2$ coincides with the classical solution by \textit{Claim} (ii) applied to $(\rho^m,v^m)$.\\
If $\bar{V}\leq s$, the trace of $\mathcal{RS}((\rho^l,v^l),(\rho^r,v^r))$ is $(\rho^l,v^l)$ (or $(\rho^m,v^m)$ and $(\rho^l,v^l)$ respectively on the right and on the left when the equal holds). The condition $\bar{V}\leq s$, is verified if and only if
$$\dfrac{\rho^m v^m-\rho^l v^l}{\rho^m-\rho^l}\geq \bar{V} \Longleftrightarrow \rho^l v^l \leq \rho^m v^m +\bar{V}(\rho^l-\rho^m)<F_\alpha+\bar{V}\rho^l,$$
because $(\rho^m,v^m)$ satisfies the constraint. Hence $\mathcal{RS}^\alpha_2$ gives the classical solution.\\
If $v^l<v^r$ (see Figure \ref{fig_dim_prop_sec_ris_caso_cinque}), the standard Riemann solver joins $(\rho^l,v^l)$ and $(\rho^m,v^m)$ with a rarefaction. The trace of $\mathcal{RS}((\rho^l,v^l),(\rho^r,v^r))$ in $\bar{V}$ can be one of the points $(\rho^\sigma,v^\sigma)$ connecting $(\rho^l,v^l)$ and $(\rho^m,v^m)$. These points are such that $\rho^m\leq \rho^\sigma \leq \rho^l$. By Lemma \ref{funzione_concava_e_retta} and since $(\rho^m,v^m)$ satisfies the constraint, we have $\hat{\rho}<\rho^m$ or $\rho^m<\check{\rho}_1$. In the first case, even $\rho^\sigma >\hat{\rho}$, then $(\rho^\sigma,v^\sigma)$ satisfies the constraint by Lemma \ref{funzione_concava_e_retta}.\\
The last case is $\rho^m<\check{\rho}_1$. We claim (\textit{Claim} (iii)) that this case never happens.\\
Therefore $(\check{\rho}_2,\check{v}_2)$ belongs to $\mathcal{D}_{v_1,v_2,w_1,w_2}$ which is consequently invariant.
\end{enumerate}
\textit{Proofs of the Claims.}
\begin{enumerate}
\item[(i)] Let us assume that $(\hat{\rho},\hat{v})$ does not belong to $\mathcal{D}_{v_1,v_2,w_1,w_2}$. Let $(\rho^*,v^*) \in \mathcal{D}_{v_1,v_2,w_1,w_2}$ be a point of $L_1(\rho,\rho^l,v^l)$. Then $$\rho^* v^*\leq F_\alpha+\bar{V}\rho^*.$$
\textit{Proof.} By the choice of $(\rho^*,v^*)$, the conditions $v_1\leq v^*\leq v_2$, $w_1\leq v^*+p(\rho^*)\leq w_2$ and $v^*+p(\rho^*)=v^l+p(\rho^l)$ hold. By definition, the last condition is true also for $(\hat{\rho},\hat{v})$, namely $\hat{v}+p(\hat{\rho})=v^l+p(\rho^l)$. Hence $\hat{v}+p(\hat{\rho}) \in [w_1,w_2]$. We are supposing that $(\hat{\rho},\hat{v})$ does not belong to $\mathcal{D}_{v_1,v_2,w_1,w_2}$, hence it must be $\hat{v}\notin [v_1,v_2]$.\\
If $\hat{v}<v_1$, then let $\tilde{v}$ be the minimum of the function $h_\alpha$ and let us assume that $v_1>\tilde{v}$. In this case $h_\alpha$ is increasing in $[v_1,v_2]$, thus $h_\alpha(v)\geq h_\alpha(v_1)\geq w_2$ for every $v\in[v_1,v_2]$ which is a contradiction of the hypothesis on $\bar{v}$. If $v_1<\tilde{v}$, then $h_\alpha(\hat{v})>h_\alpha(v_1)\geq w_2$ which is absurd because $h_\alpha(\hat{v})\in[w_1,w_2]$.\\
The only case remaining is $\hat{v}>v_2$. By definition $v^*$ is in $[v_1,v_2]$.Therefore we have
$$\hat{v}>v^* \Longrightarrow p(\rho^*)=p(\hat{\rho})+\hat{v}-v^*>p(\hat{\rho}) \Longrightarrow \rho^* > \hat{\rho}.$$
Since $(\hat{\rho},\hat{v})$ is the point with maximal density of the strictly concave function $\rho \to \rho L_1(\rho,\rho^l,v^l)$ which satisfies the constraint, by Lemma \ref{funzione_concava_e_retta}, for every $\rho\geq\hat{\rho}$ we find $\rho L_1(\rho,\rho^l,v^l)\leq F_\alpha +\bar{V}\rho$ and this is true in particular for $(\rho^*,v^*)$.
\item[(ii)] Let us assume that $(\check{\rho}_2,\check{v}_2)$ does not belong to $\mathcal{D}_{v_1,v_2,w_1,w_2}$. Let $(\rho^*,v^*)\in \mathcal{D}_{v_1,v_2,w_1,w_2}$ be a point of the curve $\rho \to \rho L_2(\rho, \rho^r,v^r)$. Then
\begin{equation}
\rho^* v^*\leq F_\alpha+\bar{V}\rho^*.
\end{equation}
\textit{Proof.} By the choice of $(\rho^*,v^*)$, we have $v^*=v^r$, $v_1\leq v^* \leq v_2$ and $w_1\leq v^*+p(\rho^*)\leq w_2$. By definition, $\check{v}_2=v^r$ and hence we have $v_1\leq \check{v}_2 \leq v_2$. Since $(\check{\rho}_2,\check{v}_2)$ does not belong to $\mathcal{D}_{v_1,v_2,w_1,w_2}$, it must be $h_\alpha(\check{v}_2)=\check{v}_2+p(\check{\rho}_2)<w_1$ or $h_\alpha(\check{v}_2)>w_2$. By hypothesis $h_\alpha(v)\geq w_1$ for every $v$ in $[v_1,v_2]$. Hence the only possible case is $h_\alpha(\check{v}_2)>w_2$. Recalling the definition of $(\rho^*,v^*)$, we find
\begin{equation*}
\begin{split}
&v^*+p(\rho^*)\leq w_2 \Longrightarrow \check{v}_2+p(\check{\rho}_2)>v^*+p(\rho^*)=v^r+p(\rho^*) \Longrightarrow \check{\rho}_2 >\rho^* \Longrightarrow\\
&\Longrightarrow \dfrac{F_\alpha}{\check{v}_2-\bar{V}}>\rho^*\Longrightarrow \rho^* v^* <F_\alpha+\bar{V}\rho^*.
\end{split}
\end{equation*}
\item[(iii)] If $v^l<v^r$ and $(\check{\rho}_2,\check{v}_2)$ does not belong to $\mathcal{D}_{v_1,v_2,w_1,w_2}$, then the case $\rho^m<\check{\rho}_1$ never appears.\\
\textit{Proof.} Since $\check{v}_2=v^r$ and $v^r \in [v_1,v_2]$, the condition $h_\alpha(\check{v}_2)= \check{v}_2+p(\check{\rho}_2)\notin [w_1,w_2]$ must hold to have $(\check{\rho}_2,\check{v}_2)\notin \mathcal{D}_{v_1,v_2,w_1,w_2}$.\\
The condition $\check{v}_2 \in [v_1,v_2]$ and the hypotheses (\ref{condizioni_dominio_invariante_R_alfa_2}) imply that $h_\alpha(\check{v}_2)\geq w_1$. Therefore we have
\begin{equation}\label{claim_3}
h_\alpha(\check{v}_2)>w_2.
\end{equation}
Let $\tilde{v}$ be the minimum of the function $h_\alpha$.\\
If $\check{v}_2\geq \tilde{v}$, then both $v_2$ and $\check{v}_2$ are in the interval where $h_\alpha$ is increasing. Hence
$$\check{v}_2=v^r\leq v_2 \Longrightarrow h_\alpha(\check{v}_2)\leq h_\alpha(v_2).$$
By the hypotheses (\ref{condizioni_dominio_invariante_R_alfa_2}), $h_\alpha(v_2)\leq w_2$, but this is a contradiction of the condition (\ref{claim_3}), because we would have $h_\alpha(\check{v}_2)\leq w_2$.\\
Let us show that $v^m<\hat{v}$ whenever $\check{v}_2< \tilde{v}$.\\
If $\hat{v}>\tilde{v}$ holds, then we have
$$v^m=\check{v}_2<\hat{v}.$$
If $\hat{v}<\tilde{v}$, then both $\hat{v}$ and $\check{v}_2$ are in the interval where $h_\alpha$ is decreasing. For the part 1 of the proposition, the point $(\hat{\rho},\hat{v})$ belongs to $\mathcal{D}_{v_1,v_2,w_1,w_2}$, hence
$$h_\alpha(\hat{v})=\hat{v}+p(\hat{\rho})\leq w_2<h_\alpha(\check{v}_2)\Longrightarrow v^m=\check{v}_2< \hat{v}.$$
Both these cases are in contradiction with $\rho^m<\check{\rho}_1$, because $\hat{\rho} \geq \check{\rho}_1$ and
$$v^m \leq \hat{v} \Longleftrightarrow \rho^m \geq \hat{\rho}.$$
\end{enumerate}
\end{proof}
\begin{figure}
\centering
\begin{subfigure}[h]{0.7\textwidth}
\definecolor{cqcqcq}{rgb}{0.7529411764705882,0.7529411764705882,0.7529411764705882}
\definecolor{xdxdff}{rgb}{0.49019607843137253,0.49019607843137253,1.}
\definecolor{uququq}{rgb}{0.25098039215686274,0.25098039215686274,0.25098039215686274}
\begin{tikzpicture}[scale =0.25,line cap=round,line join=round,>=triangle 45,x=1.0cm,y=1.0cm]
\draw[->,color=black] (-2.,0.) -- (30.,0.);
\draw[->,color=black] (0.,-1.5) -- (0.,23.5);
\clip(-3,-2) rectangle (30.,23.5);
\draw[line width=0.pt,color=cqcqcq,fill=cqcqcq,fill opacity=1.0] {[smooth,samples=50,domain=7.75:13.73] plot(\x,{0-(-18.41/12.62)*\x-0.0/12.62})} -- (13.73,7.37) {[smooth,samples=50,domain=13.73:7.75] -- plot(\x,{4.24*\x-\x^(1.5)})} -- (7.75,11.3) -- cycle;
\draw[line width=0.pt,color=cqcqcq,fill=cqcqcq,fill opacity=1.0] {[smooth,samples=50,domain=13.64:21.28] plot(\x,{5.16*\x-\x^(1.5)})} -- (21.28,11.68) {[smooth,samples=50,domain=21.28:13.64] -- plot(\x,{0-(-11.8/21.47)*\x-0.0/21.47})} -- (13.64,20.03) -- cycle;
\draw[smooth,samples=50,domain=0.001:30.0] plot(\x,{5.16*(\x)-(\x)^(1.5)});
\draw[smooth,samples=50,domain=0.001:30.0] plot(\x,{4.24*(\x)-(\x)^(1.5)});
\draw [domain=0.0:29.0] plot(\x,{(-0.--18.41*\x)/12.62});
\draw [domain=0.0:29.0] plot(\x,{(-0.--11.8*\x)/21.47});
\draw [domain=-2.:30.] plot(\x,{(--12.63--0.5*\x)/1.});
\draw [domain=0.0:30.0] plot(\x,{(-0.--14.1*\x)/17.71});
\draw[smooth,samples=50,domain=0.001:30.0] plot(\x,{5.*(\x)-(\x)^(1.5)});
\begin{scriptsize}
\draw (27,0.1) node[anchor=north west] {$\rho$};
\draw (20.08,23.2) node[anchor=north west] {$F_\alpha+\bar{V}\rho$};
\draw (-2.,22) node[anchor=north west] {$\rho v$};
\draw (13.5,4.31) node[anchor=north west] {$w_1$};
\draw (22.78,10.14) node[anchor=north west] {$w_2$};
\draw (-2.,14.3) node[anchor=north west] {$F_\alpha$};
\draw (13.,23.47) node[anchor=north west] {$v_2$};
\draw (24.22,13.94) node[anchor=north west] {$v_1$};
\draw (9.7,20.2) node[anchor=north west] {$(\hat{\rho},\hat{v})$};
\draw (13.01,11.4) node[anchor=north west] {$(\rho^r,v^r)$};
\draw (17.5,14.9) node[anchor=north west] {$(\rho^m,v^m)$};
\draw (11,17.) node[anchor=north west] {$(\rho^l,v^l)$};
\end{scriptsize}
\begin{scriptsize}
\draw [fill=uququq] (11.77,18.52) circle (5pt);
\draw [fill=xdxdff] (14.96,17.00) circle (5pt);
\draw [fill=xdxdff] (17.71,14.09) circle (5pt);
\draw [fill=xdxdff] (13.17,10.48) circle (5pt);
\end{scriptsize}
\end{tikzpicture}
\caption{The standard Riemann solver links $(\rho^l,v^l)$ and $(\rho^m,v^m)$ with a shock.}\label{fig_dim_prop_dom_inv_sec_ris_caso_due}
\end{subfigure}
\\
\vspace{0.5cm}
\begin{subfigure}[h]{0.7\textwidth}
\definecolor{cqcqcq}{rgb}{0.7529411764705882,0.7529411764705882,0.7529411764705882}
\definecolor{xdxdff}{rgb}{0.49019607843137253,0.49019607843137253,1.}
\definecolor{uququq}{rgb}{0.25098039215686274,0.25098039215686274,0.25098039215686274}
\begin{tikzpicture}[scale=0.25,line cap=round,line join=round,>=triangle 45,x=1.0cm,y=1.0cm]
\draw[->,color=black] (-2.,0.) -- (30.,0.);
\draw[->,color=black] (0.,-1.5) -- (0.,25.);
\clip(-3.,-2) rectangle (30.,25.);
\draw[line width=0.pt,color=cqcqcq,fill=cqcqcq,fill opacity=1.0] {[smooth,samples=50,domain=8.25:13.64] plot(\x,{0-(-18.11/13.22)*\x-0.0/13.22})} -- (13.64,7.5) {[smooth,samples=50,domain=13.64:8.25] -- plot(\x,{4.24*\x-\x^(1.5)})} -- (8.25,11.3) -- cycle;
\draw[line width=0.pt,color=cqcqcq,fill=cqcqcq,fill opacity=1.0] {[smooth,samples=50,domain=13.62:14.38] plot(\x,{0-(-18.11/13.22)*\x-0.0/13.22})} -- (14.38,7.9) {[smooth,samples=50,domain=14.38:13.62] -- plot(\x,{0-(-11.8/21.47)*\x-0.0/21.47})} -- (13.62,18.65) -- cycle;
\draw[line width=0.pt,color=cqcqcq,fill=cqcqcq,fill opacity=1.0] {[smooth,samples=50,domain=14.35:21.28] plot(\x,{5.16*\x-\x^(1.5)})} -- (21.28,11.68) {[smooth,samples=50,domain=21.28:14.35] -- plot(\x,{0-(-11.8/21.47)*\x-0.0/21.47})} -- (14.35,19.71) -- cycle;
\draw[smooth,samples=50,domain=0.001:30.0] plot(\x,{5.16*(\x)-(\x)^(1.5)});
\draw[smooth,samples=50,domain=0.001:30.0] plot(\x,{4.24*(\x)-(\x)^(1.5)});
\draw [domain=-2.:30.] plot(\x,{(--11.87--0.5*\x)/1.});
\draw [domain=0.0:30.0] plot(\x,{(-0.--18.1*\x)/13.22});
\draw [domain=0.0:30.0] plot(\x,{(-0.--11.8*\x)/21.47});
\draw[smooth,samples=50,domain=0.001:30.0] plot(\x,{5.*(\x)-(\x)^(1.5)});
\draw [domain=0.0:30.0] plot(\x,{(-0.--17.35*\x)/14.5});
\begin{scriptsize}
\draw (27.,0.1) node[anchor=north west] {$\rho$};
\draw (19.9,22.7) node[anchor=north west] {$F_\alpha+\bar{V}\rho$};
\draw (-2.,23.) node[anchor=north west] {$\rho v$};
\draw (14.,4.18) node[anchor=north west] {$w_1$};
\draw (23.75,7.73) node[anchor=north west] {$w_2$};
\draw (-2.,13.7) node[anchor=north west] {$F_\alpha$};
\draw (15.6,24.72) node[anchor=north west] {$v_2$};
\draw (24.81,14.2) node[anchor=north west] {$v_1$};
\draw (10.3,20.3) node[anchor=north west] {$(\hat{\rho},\hat{v})$};
\draw (10.5,13.5) node[anchor=north west] {$(\rho^r,v^r)$};
\draw (14.555,18.2) node[anchor=north west] {$(\rho^m,v^m)$};
\draw (15,12.6) node[anchor=north west] {$(\rho^l,v^l)$};
\draw (17.85,14.5) node[anchor=north west] {$(\rho^*,v^*)$};
\end{scriptsize}
\begin{scriptsize}
\draw [fill=uququq] (12.8,18.26) circle (5pt);
\draw [fill=xdxdff] (18.97,12.3) circle (5pt);
\draw [fill=xdxdff] (14.5,17.35) circle (5pt);
\draw [fill=xdxdff] (10.78,12.9) circle (5pt);
\draw [fill=xdxdff] (17.92,13.81) circle (5pt);
\end{scriptsize}
\end{tikzpicture}
\caption{The standard Riemann solver links $(\rho^l,v^l)$ and $(\rho^m,v^m)$ with a rarefaction.}\label{fig_dim_prop_dom_inv_sec_ris_caso_tre}
\end{subfigure}
\caption{Notations used in the proof of Proposition \ref{proposizione_dominio_invariante_R_alfa_2}. The invariant domain is the coloured region. In both cases the Riemann solver $\mathcal{RS}^\alpha_2$ coincides with $\mathcal{RS}$.}
\end{figure}
\begin{figure}
\centering
\begin{subfigure}[h]{0.7\textwidth}
\definecolor{cqcqcq}{rgb}{0.7529411764705882,0.7529411764705882,0.7529411764705882}
\definecolor{uququq}{rgb}{0.25098039215686274,0.25098039215686274,0.25098039215686274}
\definecolor{xdxdff}{rgb}{0.49019607843137253,0.49019607843137253,1.}
\definecolor{qqqqff}{rgb}{0.,0.,1.}
\begin{tikzpicture}[scale = 0.25,line cap=round,line join=round,>=triangle 45,x=1.0cm,y=1.0cm]
\draw[->,color=black] (-3.,0.) -- (30.,0.);
\draw[->,color=black] (0.,-2.) -- (0.,25.);
\clip(-3.,-2.) rectangle (30.,25.);
\draw[line width=0.pt,color=cqcqcq,fill=cqcqcq,fill opacity=1.0] {[smooth,samples=50,domain=2.98:7.17] plot(\x,{0-(-18.52/7.35)*\x-0.0/7.35})} -- (7.17,11.22) {[smooth,samples=50,domain=7.17:2.98] -- plot(\x,{4.24*\x-\x^(1.5)})} -- (2.98,7.50) -- cycle;
\draw[line width=0.pt,color=cqcqcq,fill=cqcqcq,fill opacity=1.0] {[smooth,samples=50,domain=7.15:12.22] plot(\x,{5.16*\x-\x^(1.5)})} -- (12.22,9.12) {[smooth,samples=50,domain=12.22:7.15] -- plot(\x,{4.24*\x-\x^(1.5)})} -- (7.15,17.79) -- cycle;
\draw[line width=0.pt,color=cqcqcq,fill=cqcqcq,fill opacity=1.0] {[smooth,samples=50,domain=12.2:19.5] plot(\x,{5.16*\x-\x^(1.5)})} -- (19.5,14.55) {[smooth,samples=50,domain=19.5:12.2] -- plot(\x,{0-(-14.84/19.89)*\x-0.0/19.89})} -- (12.2,20.36) -- cycle;
\draw[smooth,samples=50,domain=0.001:30.0] plot(\x,{5.16*(\x)-(\x)^(1.5)});
\draw[smooth,samples=50,domain=0.001:30.0] plot(\x,{4.24*(\x)-(\x)^(1.5)});
\draw [domain=-3.:30.] plot(\x,{(--11.84--0.5*\x)/1.});
\draw[smooth,samples=50,domain=0.001:30.0] plot(\x,{5.00*(\x)-(\x)^(1.5)});
\draw [domain=0.0:30.0] plot(\x,{(-0.--17.36*\x)/14.48});
\draw [domain=0.0:30.0] plot(\x,{(-0.--18.52*\x)/7.35});
\draw [domain=0.0:30.0] plot(\x,{(-0.--14.84*\x)/19.89});
\begin{scriptsize}
\draw (27.5,0.1) node[anchor=north west] {$\rho$};
\draw (22.08,23.75) node[anchor=north west] {$F_\alpha+\bar{V}\rho$};
\draw (-2,23.) node[anchor=north west] {$\rho v$};
\draw (13.6,5.04) node[anchor=north west] {$w_1$};
\draw (22.78,10.19) node[anchor=north west] {$w_2$};
\draw (-2,13.5) node[anchor=north west] {$F_\alpha$};
\draw (8.90,23.64) node[anchor=north west] {$v_2$};
\draw (23.06,17.89) node[anchor=north west] {$v_1$};
\draw (5.7,13.4) node[anchor=north west] {$(\rho^r,v^r)$};
\draw (14.3,18.2) node[anchor=north west] {$(\rho^m,v^m)$};
\draw (7.5,18.5) node[anchor=north west] {$(\rho^l,v^l)$};
\draw (17.41,21.31) node[anchor=north west] {$(\check{\rho}_2,\check{v}_2)$};
\draw (12.10,15.56) node[anchor=north west] {$(\rho^*,v^*)$};
\draw (11.,20.2) node[anchor=north west] {$(\hat{\rho},\hat{v})$};
\end{scriptsize}
\begin{scriptsize}
\draw [fill=qqqqff] (12.86,18.27) circle (5pt);
\draw [fill=xdxdff] (9.55,18.28) circle (5pt);
\draw [fill=xdxdff] (14.48,17.36) circle (5pt);
\draw [fill=xdxdff] (10.01,11.99) circle (5pt);
\draw [fill=uququq] (16.96,20.32) circle (5pt);
\draw [fill=xdxdff] (12.04,14.43) circle (5pt);
\end{scriptsize}
\end{tikzpicture}
\caption{The standard Riemann solver links $(\rho^l,v^l)$ and $(\rho^m,v^m)$ with a shock.}\label{fig_dim_prop_sec_ris_caso_quattro}
\end{subfigure}
\\
\vspace{0.5cm}
\begin{subfigure}[h]{0.7\textwidth}
\definecolor{cqcqcq}{rgb}{0.7529411764705882,0.7529411764705882,0.7529411764705882}
\definecolor{uququq}{rgb}{0.25098039215686274,0.25098039215686274,0.25098039215686274}
\definecolor{xdxdff}{rgb}{0.49019607843137253,0.49019607843137253,1.}
\begin{tikzpicture}[scale=0.25,line cap=round,line join=round,>=triangle 45,x=1.0cm,y=1.0cm]
\draw[->,color=black] (-2.,0.) -- (30.,0.);
\draw[->,color=black] (0.,-1.5) -- (0.,25.);
\clip(-2.,-1.5) rectangle (30.,25.);
\draw[line width=0.pt,color=cqcqcq,fill=cqcqcq,fill opacity=1.0] {[smooth,samples=50,domain=2.42:6.13] plot(\x,{0-(-18.70/6.96)*\x-0.0/6.96})} -- (6.13,10.83) {[smooth,samples=50,domain=6.13:2.42] -- plot(\x,{4.24*\x-\x^(1.5)})} -- (2.42,6.50) -- cycle;
\draw[line width=0.pt,color=cqcqcq,fill=cqcqcq,fill opacity=1.0] {[smooth,samples=50,domain=6.1:12.22] plot(\x,{5.16*\x-\x^(1.5)})} -- (12.22,9.12) {[smooth,samples=50,domain=12.22:6.1] -- plot(\x,{4.24*\x-\x^(1.5)})} -- (6.1,16.42) -- cycle;
\draw[line width=0.pt,color=cqcqcq,fill=cqcqcq,fill opacity=1.0] {[smooth,samples=50,domain=12.2:19.5] plot(\x,{5.16*\x-\x^(1.5)})} -- (19.5,14.55) {[smooth,samples=50,domain=19.5:12.2] -- plot(\x,{0-(-14.84/19.89)*\x-0.0/19.89})} -- (12.2,20.36) -- cycle;
\draw[smooth,samples=50,domain=0.001:30.0] plot(\x,{5.16*(\x)-(\x)^(1.5)});
\draw[smooth,samples=50,domain=0.001:30.0] plot(\x,{4.24*(\x)-(\x)^(1.5)});
\draw [domain=-2.:30.] plot(\x,{(--11.5--0.65*\x)/1.});
\draw [domain=0.0:30.0] plot(\x,{(-0.--18.70*\x)/6.96});
\draw [domain=0.0:30.0] plot(\x,{(-0.--14.84*\x)/19.9});
\draw[smooth,samples=50,domain=0.001:30.0] plot(\x,{5.*(\x)-(\x)^(1.5)});
\draw [domain=0.0:30.0] plot(\x,{(-0.--18.03*\x)/13.36});
\begin{scriptsize}
\draw (27.,0.1) node[anchor=north west] {$\rho$};
\draw (19.51,25.) node[anchor=north west] {$F_\alpha+\bar{V}\rho$};
\draw (-2,23.84) node[anchor=north west] {$\rho v$};
\draw (13.6,4.39) node[anchor=north west] {$w_1$};
\draw (22.79,10.14) node[anchor=north west] {$w_2$};
\draw (-2,12.7) node[anchor=north west] {$F_\alpha$};
\draw (8.42,23.94) node[anchor=north west] {$v_2$};
\draw (23.34,18.15) node[anchor=north west] {$v_1$};
\draw (5.,14.) node[anchor=north west] {$(\rho^r,v^r)$};
\draw (13.6,19.1) node[anchor=north west] {$(\rho^m,v^m)$};
\draw (13.4,15.28) node[anchor=north west] {$(\rho^l,v^l)$};
\draw (16.03,22.78) node[anchor=north west] {$(\check{\rho}_2,\check{v}_2)$};
\draw (15.7,17.) node[anchor=north west] {$(\rho^\sigma,v^\sigma)$};
\draw (9.03,20.42) node[anchor=north west] {$(\hat{\rho},\hat{v})$};
\end{scriptsize}
\begin{scriptsize}
\draw [fill=xdxdff] (17.64,14.2) circle (5pt);
\draw [fill=xdxdff] (13.36,18.03) circle (5pt);
\draw [fill=xdxdff] (9.23,12.45) circle (5pt);
\draw [fill=uququq] (16.44,22.18) circle (5pt);
\draw [fill=xdxdff] (16.,16.07) circle (5pt);
\draw [fill=uququq] (10.86,18.56) circle (5pt);
\end{scriptsize}
\end{tikzpicture}
\caption{The standard Riemann solver links $(\rho^l,v^l)$ and $(\rho^m,v^m)$ with a rarefaction and $\rho^m>\hat{\rho}$.}\label{fig_dim_prop_sec_ris_caso_cinque}
\end{subfigure}
\caption{Notations used in the proof of Proposition \ref{proposizione_dominio_invariante_R_alfa_2}. The invariant domain is the coloured region. In both cases the solution given by the Riemann solver $\mathcal{RS}^\alpha_2$ coincides with the classical one.}
\end{figure}
We have proved point $(ii)$ of Theorem \ref{Proposizione_Dominio_invariante_R_alfa2_definitiva}.
\subsection{The invariant domain intersects both $\mathcal{U}$ and $\mathcal{V}$}
If $v_1 < \bar{V}< v_2$, the domain $\mathcal{D}_{v_1,v_2,w_1,w_2}$ intersects both sets $\mathcal{U}=\lbrace (\rho,v)\in (0,+\infty)\times(0,+\infty)\, : \, v \leq \bar{V}\rbrace$ and $\mathcal{V}=\lbrace (\rho,v)\in (0,+\infty)\times(0,+\infty)\, : \, v \geq \bar{V}\rbrace$. Let us suppose that there exists $\bar{v}\in[\bar{V},v_2]$ such that $h_\alpha(\bar{v})<w_2$.\\
We can summarize the results obtained in the previous sections as follows.
\begin{enumerate}
\item The domain $\mathcal{D}_{v_1,\bar{V},w_1,w_2}$ is invariant for $\mathcal{RS}^\alpha_1$ and $\mathcal{RS}^\alpha_2$.
\item $h_\alpha(v_2)\geq w_2$ if and only if $\mathcal{D}_{\bar{V},v_2,w_1,w_2}$ is invariant for $\mathcal{RS}^\alpha_1$.
\item The conditions $h_\alpha(v) \geq w_1$ for every $v\in [\bar{V},v_2]$ and $h_\alpha(v_2)\leq w_2$ hold if and only if $\mathcal{D}_{\bar{V},v_2,w_1,w_2}$ is invariant for  $\mathcal{RS}^\alpha_2$.
\end{enumerate}
\begin{remark}
The inequality $h_\alpha(\bar{V})\geq w_2$, corresponding to the conditions on $v_1$ of Propositions \ref{Proposizione_Dominio_invariante_R_alfa1} and \ref{proposizione_dominio_invariante_R_alfa_2}, is always satisfied by the definition of $h_\alpha$.
\end{remark}
The next proposition states that if $v_1<\bar{V}<v_2$, then $\mathcal{D}_{v_1,v_2,w_1,w_2}$ is invariant for both $\mathcal{RS}^\alpha_1$ and $\mathcal{RS}^\alpha_2$.
\begin{prop}\label{domini_invarianti_ultima_prop}
Let $0<v_1<\bar{V}<v_2$, $0<w_1<w_2$, $v_2<w_2$ and $\alpha\in(0,1)$ be fixed. If the domains $\mathcal{D}_{v_1,\bar{V},w_1,w_2}$ and $\mathcal{D}_{\bar{V},v_2,w_1,w_2}$ are invariant for $\mathcal{RS}^\alpha_1$ and $\mathcal{RS}^\alpha_2$, then also $\mathcal{D}_{v_1,v_2,w_1,w_2}$ is invariant for both Riemann solvers.
\end{prop}
\begin{proof}
The domain $\mathcal{D}_{v_1,\bar{V},w_1,w_2}$ and $\mathcal{D}_{\bar{V},v_2,w_1,w_2}$ coincide respectively with
$$\mathcal{D}^\mathcal{U}_{v_1,v_2,w_1,w_2}:=\mathcal{D}_{v_1,v_2,w_1,w_2}\cap \mathcal{U} \; \text{ and } \; \mathcal{D}^\mathcal{V}_{v_1,v_2,w_1,w_2}:=\mathcal{D}_{v_1,v_2,w_1,w_2}\cap \mathcal{V}.$$
If the solutions given by the Riemann solvers $\mathcal{RS}^\alpha_1$ and $\mathcal{RS}^\alpha_2$ coincide with the classic one, then the domain $\mathcal{D}_{v_1,v_2,w_1,w_2}$ is invariant by Proposition \ref{standard_invariant_domain}.\\
We note that, for every $\alpha \in(0,1)$ and for every initial datum $((\rho^l,v^l),(\rho^r,v^r))\in(\mathbb{R}^+\times\mathbb{R}^+)^2$ of the Riemann problem (\ref{problema_vincolato}), the points $(\check{\rho}_1,\check{v}_1)$, $(\check{\rho}_2,\check{v}_2)$ and $(\hat{\rho},\hat{v})$ belong to $\mathcal{D}^\mathcal{V}_{v_1,v_2,w_1,w_2}$. Indeed for $(\hat{\rho},\hat{v})$, since $F_\alpha>0$, we find
$$\hat{\rho}\hat{v}=F_\alpha+\bar{V}\hat{\rho}>\bar{V}\hat{\rho}\Longrightarrow \hat{v}>\bar{V}.$$
Similarly for $(\check{\rho}_1,\check{v}_1)$ and $(\check{\rho}_2,\check{v}_2)$.\\
Therefore if the solutions given by the Riemann solvers $\mathcal{RS}^\alpha_1$ and $\mathcal{RS}^\alpha_2$ are not classic and the points $(\rho^l,v^l)$ and $(\rho^r,v^r)$ are in $\mathcal{D}^\mathcal{V}_{v_1,v_2,w_1,w_2}$, then $\mathcal{RS}^\alpha_1((\rho^l,v^l),(\rho^r,v^r))(\lambda)$ and $\mathcal{RS}^\alpha_2((\rho^l,v^l),(\rho^r,v^r))(\lambda)$ belong to $\mathcal{D}^\mathcal{V}_{v_1,v_2,w_1,w_2}$ for every $\lambda \in \mathbb{R}$.\\
If $v^r<\bar{V}$, the trace of the standard solution $\mathcal{RS}((\rho^l,v^l),(\rho^r,v^r))$ in $\bar{V}$ is $(\rho^r,v^r)$. Indeed the propagation speed of the contact discontinuity joining $(\rho^m,v^m)$ and $(\rho^r,v^r)$ is less than the bus speed.\\
In this case $(\rho^r,v^r)$ belongs to $\mathcal{U}$ and since $v^r=v^m$, the same holds for $(\rho^m,v^m)$. Since the points in $\mathcal{U}$ satisfy the constraint, the solutions given by the Riemann solvers $\mathcal{RS}^\alpha_1$ and $\mathcal{RS}^\alpha_2$ coincide with the classical one.\\
If $v^r=\bar{V}$, then both points $(\rho^l,v^l)$ and $(\rho^r,v^r)$ are either in $\mathcal{D}^\mathcal{U}_{v_1,v_2,w_1,w_2}$ or in $\mathcal{D}^\mathcal{V}_{v_1,v_2,w_1,w_2}$. Hence the solutions $\mathcal{RS}^\alpha_1((\rho^l,v^l),(\rho^r,v^r))$ and $\mathcal{RS}^\alpha_2((\rho^l,v^l),(\rho^r,v^r))$ are entirely contained in one of the two domains and consequently in $\mathcal{D}_{v_1,v_2,w_1,w_2}$.\\
If $v^r>\bar{V}$, the classical solution in $\bar{V}$ can be $(\rho^m,v^m)$, $(\rho^l,v^l)$ or, if $\rho^l\geq \rho^m$, one of the points of the rarefaction joining them. In this case both $(\rho^m,v^m)$ and $(\rho^r,v^r)$ are in $\mathcal{D}^\mathcal{V}_{v_1,v_2,w_1,w_2}$. If also $(\rho^l,v^l)$ belongs to $\mathcal{D}^\mathcal{V}_{v_1,v_2,w_1,w_2}$, the solutions $\mathcal{RS}^\alpha_1((\rho^l,v^l),(\rho^r,v^r))(\lambda)$ and $\mathcal{RS}^\alpha_2((\rho^l,v^l),(\rho^r,v^r))(\lambda)$ are in $\mathcal{D}^\mathcal{V}_{v_1,v_2,w_1,w_2}$ for every $\lambda \in \mathbb{R}$ by Propositions \ref{Proposizione_Dominio_invariante_R_alfa1} and \ref{proposizione_dominio_invariante_R_alfa_2}.\\
If $(\rho^l,v^l)$ belongs to $\mathcal{D}^\mathcal{U}_{v_1,v_2,w_1,w_2}$, we must have $\rho^l\geq\rho^m$, because $v^l\leq \bar{V}\leq v^m$ and $p(\rho^l)=p(\rho^m)+v^m-v^l\geq p(\rho^m)$. In this case the standard Riemann solver connects these two points with a rarefaction. If $(\rho^m,v^m)$ satisfies the constraint, then the same holds for each point $(\rho^\sigma,v^\sigma)$ of the rarefaction, because $\rho^\sigma\geq \rho^m$ and by Lemma \ref{funzione_concava_e_retta}. The trace of $\mathcal{RS}((\rho^l,v^l),(\rho^r,v^r))$ in $\bar{V}$ is one of the point of the rarefaction. Hence $\mathcal{RS}^\alpha_1$ and $\mathcal{RS}^\alpha_2$ coincide with the standard Riemann solver and the solutions are contained in $\mathcal{D}_{v_1,v_2,w_1,w_2}$.\\
If $(\rho^m,v^m)$ does not satisfy the constraint, let $(\rho^*,v^*)$ be the solution to the system
\begin{equation*}
\begin{cases}
v+p(\rho)=v^l+p(\rho^l),\\
v=\bar{V}.
\end{cases}
\end{equation*}
We can obtain the solution given by $\mathcal{RS}^\alpha_1$ to the Riemann problem (\ref{problema_vincolato}), as
\begin{equation}\label{soluzione_da_congiungere_1}
\mathcal{RS}^\alpha_1((\rho^l,v^l),(\rho^r,v^r))(x/t)=\begin{cases}
\mathcal{RS}^\alpha_1((\rho^l,v^l),(\rho^*,v^*))(x/t) & \text{if } x\leq\lambda_1(\rho^*,v^*)t,\\
\mathcal{RS}^\alpha_1((\rho^*,v^*),(\rho^r,v^r))(x/t) & \text{if } x>\lambda_1(\rho^*,v^*)t,
\end{cases}
\end{equation}
and similarly for the second Riemann solver
\begin{equation}\label{soluzione_da_congiungere_2}
\mathcal{RS}^\alpha_2((\rho^l,v^l),(\rho^r,v^r))(x/t)=\begin{cases}
\mathcal{RS}^\alpha_2((\rho^l,v^l),(\rho^*,v^*))(x/t) & \text{if } x\leq\lambda_1(\rho^*,v^*)t,\\
\mathcal{RS}^\alpha_2((\rho^*,v^*),(\rho^r,v^r))(x/t) & \text{if } x>\lambda_1(\rho^*,v^*)t.
\end{cases}
\end{equation}
If the classical solution in $\bar{V}$ does not satisfy the constraint, then the left trace of $\mathcal{RS}^\alpha_i$ for $i=1,2$ in $\lambda=\bar{V}$ is $(\hat{\rho},\hat{v})$. The solutions given by the Riemann solvers $\mathcal{RS}^\alpha_i$ for $i=1,2$ present a rarefaction between $(\rho^l,v^l)$ and $(\rho^*,v^*)$ and another rarefaction joining $(\rho^*,v^*)$ to $(\hat{\rho},\hat{v})$. The right and left propagation speeds along the line $x=\lambda_1(\rho^*,v^*)t$ are both $\lambda_1(\rho^*,v^*)$, hence the solution is well defined. Moreover, since $\mathcal{D}^\mathcal{U}_{v_1,v_2,w_1,w_2}$ and $\mathcal{D}^\mathcal{V}_{v_1,v_2,w_1,w_2}$ are invariant for both Riemann solvers, the point $(\rho^*,v^*)$ belongs to the intersection of the two domains and since $\mathcal{D}_{v_1,v_2,w_1,w_2}=\mathcal{D}^\mathcal{U}_{v_1,v_2,w_1,w_2}\cup \mathcal{D}^\mathcal{V}_{v_1,v_2,w_1,w_2}$, the solution obtained joining (\ref{soluzione_da_congiungere_1}) and (\ref{soluzione_da_congiungere_2}) is entirely contained in the domain $\mathcal{D}_{v_1,v_2,w_1,w_2}$.\\
Therefore the domain $\mathcal{D}_{v_1,v_2,w_1,w_2}$ is invariant.
\end{proof}
This concludes the proofs of Theorems \ref{Proposizione_Dominio_invariante_R_alfa1_definitiva} and \ref{Proposizione_Dominio_invariante_R_alfa2_definitiva}.
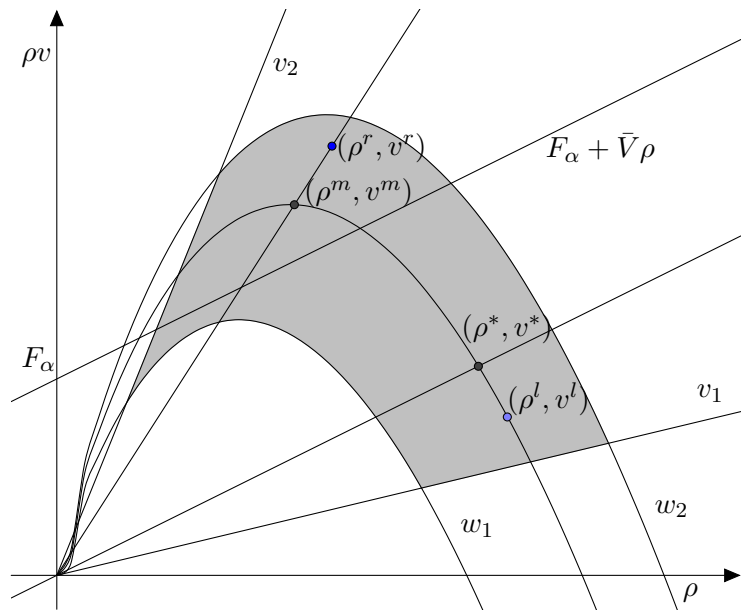
\begin{figure}[h]
\centering
\definecolor{cqcqcq}{rgb}{0.7529411764705882,0.7529411764705882,0.7529411764705882}
\definecolor{qqqqff}{rgb}{0.,0.,1.}
\definecolor{uququq}{rgb}{0.25098039215686274,0.25098039215686274,0.25098039215686274}
\definecolor{xdxdff}{rgb}{0.49019607843137253,0.49019607843137253,1.}
\begin{tikzpicture}[scale =0.3, line cap=round,line join=round,>=triangle 45,x=1.0cm,y=1.0cm]
\draw[->,color=black] (-2.,0.) -- (30.,0.);
\draw[->,color=black] (0.,-1.5) -- (0.,25.);
\clip(-2.,-1.5) rectangle (30.,25.);
\draw[line width=0.pt,color=cqcqcq,fill=cqcqcq,fill opacity=1.0] {[smooth,samples=50,domain=3.07:7.14] plot(\x,{0-(-16.76/6.73)*\x-0.0/6.73})} -- (7.14,11.21) {[smooth,samples=50,domain=7.14:3.07] -- plot(\x,{4.24*\x-\x^(1.5)})} -- (3.07,7.64) -- cycle;
\draw[line width=0.pt,color=cqcqcq,fill=cqcqcq,fill opacity=1.0] {[smooth,samples=50,domain=7.12:16.01] plot(\x,{5.16*\x-\x^(1.5)})} -- (16.01,3.86) {[smooth,samples=50,domain=16.01:7.12] -- plot(\x,{4.24*\x-\x^(1.5)})} -- (7.12,17.75) -- cycle;
\draw[line width=0.pt,color=cqcqcq,fill=cqcqcq,fill opacity=1.0] {[smooth,samples=50,domain=15.97:24.22] plot(\x,{5.16*\x-\x^(1.5)})} -- (24.22,5.84) {[smooth,samples=50,domain=24.22:15.97] -- plot(\x,{0-(-5.03/20.85)*\x-0.0/20.85})} -- (15.97,18.62) -- cycle;
\draw[smooth,samples=50,domain=0.001:30.0] plot(\x,{5.16*(\x)-(\x)^(1.5)});
\draw[smooth,samples=50,domain=0.001:30.0] plot(\x,{4.24*(\x)-(\x)^(1.5)});
\draw [domain=-2.:30.] plot(\x,{(--8.68--0.5*\x)/1.});
\draw (27,0.1) node[anchor=north west] {$\rho$};
\draw (21.,19.98) node[anchor=north west] {$F_\alpha+\bar{V}\rho$};
\draw (-2,23.72) node[anchor=north west] {$\rho v$};
\draw (17.26,2.90) node[anchor=north west] {$w_1$};
\draw (25.78,3.77) node[anchor=north west] {$w_2$};
\draw (-2,10.55) node[anchor=north west] {$F_\alpha$};
\draw [domain=0.0:30.0] plot(\x,{(-0.--16.76*\x)/6.73});
\draw [domain=0.0:30.0] plot(\x,{(-0.--5.03*\x)/20.85});
\draw (9.02,23.3) node[anchor=north west] {$v_2$};
\draw (27.62,8.87) node[anchor=north west] {$v_1$};
\draw [domain=-2.:30.] plot(\x,{(-0.--0.5*\x)/1.});
\draw[smooth,samples=50,domain=0.001:30.0] plot(\x,{4.8*(\x)-(\x)^(1.5)});
\draw [domain=0.0:30.0] plot(\x,{(-0.--18.97*\x)/12.07});
\draw (11.8,20.) node[anchor=north west] {$(\rho^r,v^r)$};
\draw (10.40,18.) node[anchor=north west] {$(\rho^m,v^m)$};
\draw (17.3,12.) node[anchor=north west] {$(\rho^*,v^*)$};
\draw (19.3,8.97) node[anchor=north west] {$(\rho^l,v^l)$};
\begin{scriptsize}
\draw [fill=xdxdff] (19.76,7.00) circle (5pt);
\draw [fill=uququq] (18.49,9.24) circle (5pt);
\draw [fill=qqqqff] (12.07,18.97) circle (5pt);
\draw [fill=uququq] (10.42,16.38) circle (5pt);
\end{scriptsize}
\end{tikzpicture}
\caption{Representation of the points used in the proof of Proposition \ref{domini_invarianti_ultima_prop}. In the represented case the domain is invariant for $\mathcal{RS}^\alpha_2$.}\label{fig_dom_inv_ultima_prop}
\end{figure}
\chapter{Numerical methods}
In this chapter we are going to introduce the Godunov's method; see \cite{randall_leveque}. We will modify it to compute numerical solutions corresponding to the two Riemann solvers $\mathcal{RS}^\alpha_1$ and $\mathcal{RS}^\alpha_2$.
\section{Godunov's method}
Let $h$ and $k$ be two fixed positive constants representing the increments for space and time discretizations and let us define the mesh points $(t^n,x_{j+1/2})$ as
\begin{equation}
t^n=nk \; \text{ for } \; n\in \mathbb{N}\; \text{ and } \; x_{j+1/2}=j h \; \text{ for } \; j\in \mathbb{Z}.
\end{equation}
We divide the $x$-axis in a sequence $\lbrace C_j\rbrace_{j\in\mathbb{Z}}$ of cells such that $C_j=[x_{j-1/2},x_{j+1/2})$.
and we define the centres of the cells:
\begin{equation}
x_j= \left(j-\dfrac{1}{2}\right)h\in C_j \; \text{ for every } \; j\in\mathbb{Z}.
\end{equation}
Fix $N\in\mathbb{N}$. Let us consider the Cauchy problem
\begin{equation}\label{num_methods_Cauchy_problem_general}
\begin{cases}
\partial_t u +\partial_x[f(u)] =0,\\
u(0,x) = u^0(x),
\end{cases}
\end{equation}
where the flux $f:\mathbb{R}^N \to \mathbb{R}^N$ is smooth, the function $u:\mathbb{R}^+\times \mathbb{R} \to \mathbb{R}^N$ is unknown and the initial datum $u^0:\mathbb{R}\to \mathbb{R}^N$ is locally integrable.\\ 
Our aim is to find a sequence of functions $\lbrace \bar{u}^n \rbrace_{n\in \mathbb{N}}$, where $\bar{u}^n$ is an approximation of the exact solution to the Cauchy problem (\ref{num_methods_Cauchy_problem_general}) at time $t^n$.\\
For every $j \in \mathbb{Z}$, we define the approximation $\bar{u}^0_j$ of $u^0$ on the cell $C_j$ as the average
\begin{equation}
\bar{u}^0_j:= \dfrac{1}{h}\int_{x_{j-1/2}}^{x_{j+1/2}}u(0,\xi)\, d\xi=\dfrac{1}{h}\int_{x_{j-1/2}}^{x_{j+1/2}}u^0(\xi)\, d\xi. 
\end{equation}
Moreover, we define the piecewise constant function
\begin{equation}\label{initial_datum}
x\to \bar{u}^0(x) = \sum_{j\in \mathbb{Z}} \left[\bar{u}^0_j\, \mathbf{1}_{C_j}(x)\right],
\end{equation}
where $\mathbf{1}_{C_j}$ is the  characteristic function of the set $C_j$, namely
$$\mathbf{1}_{C_j}(x)=\begin{cases}
1 & \text{if } \; x\in C_j,\\
0 & \text{if } \; x\notin C_j,
\end{cases}\; \text{ for } \; j \in \mathbb{Z}.$$
Since $\bar{u}^0$ is piecewise constant, we have infinitely many Riemann problems centred in $x_{j+1/2}$ for $j\in\mathbb{Z}$, i.e. 
\begin{equation}\label{num_methods_Riemann_problem_general}
\begin{cases}
\partial_t u + \partial_x[f(u)] = 0,\\
u(t^0,x) = \begin{cases}
u^l:= \bar{u}^0_j & \text{if } x\leq x_{j+1/2},\\
u^r:= \bar{u}^0_{j+1} & \text{if } x>x_{j+1/2},
\end{cases}
\end{cases}
\end{equation}
which can be solved exactly over the interval $[0,t^1]$, provided that $t^1$ is ``sufficiently small'' (in a sense that will be clarified later). Once we have the exact solution $\tilde{u}^0(t,x)$ in $[0,t^1]$, we can define a new piecewise constant approximate profile $x\to \bar{u}^1(x)$ at time $t^1$ as
\begin{equation}
\bar{u}^1(x)=\sum_{j\in\mathbb{Z}} \left( \bar{u}^1_j\, \mathbf{1}_{C_j}(x)\right)=\sum_{j\in\mathbb{Z}}\left[\dfrac{1}{h}\int_{x_{j-1/2}}^{x_{j+1/2}}\tilde{u}^0(t^1,\xi)\, d\xi \, \mathbf{1}_{C_j}(x)\right],
\end{equation}
where
$$\bar{u}^1_j= \dfrac{1}{h} \int_{x_{j-1/2}}^{x_{j+1/2}}\tilde{u}^0(t^1,\xi)\, d\xi.$$
The process can be repeated obtaining the sequence $\lbrace \bar{u}^n \rbrace_{n\in \mathbb{N}}$ (see Figure \ref{Godunov_griglia}). Explicitly, given the piecewise constant approximation $\bar{u}^n$ at time $t^n$ for $n\geq 0$, if the time step is small enough, we will obtain an exact entropy weak solution $\tilde{u}^n$ for the Cauchy problem (\ref{num_methods_Cauchy_problem_general}) with initial datum
$$u(t^n,x)= \bar{u}^n(x)$$
between the time steps $t^n$ and $t^{n+1}$. We then define the new approximate solution at time $t^{n+1}$ as
\begin{equation}\label{Godunov_averaged_solution}
\bar{u}^{n+1}(x)=\sum_{j\in\mathbb{Z}} \left( \bar{u}^{n+1}_j\, \mathbf{1}_{C_j}(x)\right)=\sum_{j\in\mathbb{Z}}\left[\dfrac{1}{h}\int_{x_{j-1/2}}^{x_{j+1/2}}\tilde{u}^n(t^{n+1},\xi)\, d\xi\, \mathbf{1}_{C_j}(x)\right],
\end{equation}
where
\begin{equation}\label{mean_process}
\bar{u}_j^{n+1}:=\dfrac{1}{h}\int_{x_{j-1/2}}^{x_{j+1/2}}\tilde{u}^n(t^{n+1},\xi)\, d\xi.
\end{equation}
\begin{figure}[h]
\centering
\definecolor{xdxdff}{rgb}{0.49019607843137253,0.49019607843137253,1.}
\definecolor{uququq}{rgb}{0.25098039215686274,0.25098039215686274,0.25098039215686274}
\definecolor{qqqqff}{rgb}{0.,0.,1.}
\begin{tikzpicture}[line cap=round,line join=round,>=triangle 45,x=1.0cm,y=1.0cm]
\draw[->,color=black] (-4.5,0.) -- (4.5,0.);
\draw[->,color=black] (0.,-0.1) -- (0.,5.5);
\clip(-4.5,-0.5) rectangle (4.5,5.5);
\draw [dash pattern=on 1pt off 1pt,domain=-4.5:4.5] plot(\x,{(-6.-0.*\x)/-1.});
\draw [dash pattern=on 1pt off 1pt,domain=-4.5:4.5] plot(\x,{(-4.-0.*\x)/-1.});
\draw [dash pattern=on 1pt off 1pt,domain=-4.5:4.5] plot(\x,{(-2.-0.*\x)/-1.});
\draw (-6.31,0.10) node[anchor=north west] {$x_{-\frac{5}{2}}$};
\draw [dash pattern=on 1pt off 1pt] (-6.,0.) -- (-6.,5.5);
\draw [dash pattern=on 1pt off 1pt] (-4.,0.) -- (-4.,5.5);
\draw [dash pattern=on 1pt off 1pt] (-2.,0.) -- (-2.,5.5);
\draw [dash pattern=on 1pt off 1pt] (2.,0.) -- (2.,5.5);
\draw [dash pattern=on 1pt off 1pt] (4.,0.) -- (4.,5.5);
\draw [dash pattern=on 1pt off 1pt] (6.,0.) -- (6.,5.5);
\draw (-4.39,0.0) node[anchor=north west] {$x_{-\frac{3}{2}}$};
\draw (-2.37,0.) node[anchor=north west] {$x_{-\frac{1}{2}}$};
\draw (-0.24,0.) node[anchor=north west] {$x_{\frac{1}{2}}$};
\draw (1.71,0.) node[anchor=north west] {$x_{\frac{3}{2}}$};
\draw (6.71,0.21) node[anchor=north west] {$x$};
\draw (-0.5,2.52) node[anchor=north west] {$t^1$};
\draw (-0.5,4.55) node[anchor=north west] {$t^2$};
\draw (-1.20,0.) node[anchor=north west] {$x_0$};
\draw (-3.24,0.) node[anchor=north west] {$x_{-1}$};
\draw (0.85,0.) node[anchor=north west] {$x_1$};
\draw (2.82,0.) node[anchor=north west] {$x_2$};
\draw (-0.4,5.5) node[anchor=north west] {$t$};
\draw (-1.45,0.65) node[anchor=north west] {$\bar{u}^0_0$};
\draw (0.66,0.65) node[anchor=north west] {$\bar{u}^0_1$};
\draw (2.59,0.65) node[anchor=north west] {$\bar{u}^0_2$};
\draw (-3.45,0.65) node[anchor=north west] {$\bar{u}^0_{-1}$};
\draw (-1.23,2.55) node[anchor=north west] {$\bar{u}^1_0$};
\draw (-3.22,2.55) node[anchor=north west] {$\bar{u}^1_{-1}$};
\draw (0.94,2.55) node[anchor=north west] {$\bar{u}^1_1$};
\draw (2.78,2.55) node[anchor=north west] {$\bar{u}^1_2$};
\draw (-1.31,4.6) node[anchor=north west] {$\bar{u}^2_0$};
\draw (-3.24,4.6) node[anchor=north west] {$\bar{u}^2_{-1}$};
\draw (0.83,4.6) node[anchor=north west] {$\bar{u}^2_1$};
\draw (2.85,4.6) node[anchor=north west] {$\bar{u}^2_2$};
\draw (-3.33,1.5) node[anchor=north west] {$\tilde{u}^0_{-1}$};
\draw (-1.24,1.50) node[anchor=north west] {$\tilde{u}^0_0$};
\draw (0.70,1.50) node[anchor=north west] {$\tilde{u}^0_1$};
\draw (2.70,1.49) node[anchor=north west] {$\tilde{u}^0_2$};
\draw (-1.24,3.33) node[anchor=north west] {$\tilde{u}^1_0$};
\draw (-3.20,3.26) node[anchor=north west] {$\tilde{u}^1_{-1}$};
\draw (0.98,3.40) node[anchor=north west] {$\tilde{u}^1_1$};
\draw (2.75,3.33) node[anchor=north west] {$\tilde{u}^1_2$};
\draw (-1.30,5.19) node[anchor=north west] {$\tilde{u}^2_0$};
\draw (-3.20,5.16) node[anchor=north west] {$\tilde{u}^2_{-1}$};
\draw (0.62,5.20) node[anchor=north west] {$\tilde{u}^2_1$};
\draw (2.77,5.25) node[anchor=north west] {$\tilde{u}^2_2$};
\begin{scriptsize}
\draw [color=qqqqff] (-4.,0.)-- ++(-1.5pt,-1.5pt) -- ++(3.0pt,3.0pt) ++(-3.0pt,0) -- ++(3.0pt,-3.0pt);
\draw [color=qqqqff] (-2.,0.)-- ++(-1.5pt,-1.5pt) -- ++(3.0pt,3.0pt) ++(-3.0pt,0) -- ++(3.0pt,-3.0pt);
\draw [fill=uququq] (-3.,0.) circle (1.5pt);
\draw [color=qqqqff] (0.,0.)-- ++(-1.5pt,-1.5pt) -- ++(3.0pt,3.0pt) ++(-3.0pt,0) -- ++(3.0pt,-3.0pt);
\draw [fill=uququq] (-1.,0.) circle (1.5pt);
\draw [color=qqqqff] (2.,0.)-- ++(-1.5pt,-1.5pt) -- ++(3.0pt,3.0pt) ++(-3.0pt,0) -- ++(3.0pt,-3.0pt);
\draw [fill=uququq] (1.,0.) circle (1.5pt);
\draw [color=qqqqff] (4.,0.)-- ++(-1.5pt,-1.5pt) -- ++(3.0pt,3.0pt) ++(-3.0pt,0) -- ++(3.0pt,-3.0pt);
\draw [fill=uququq] (3.,0.) circle (1.5pt);
\draw [color=qqqqff] (6.,0.)-- ++(-1.5pt,-1.5pt) -- ++(3.0pt,3.0pt) ++(-3.0pt,0) -- ++(3.0pt,-3.0pt);
\draw [fill=uququq] (5.,0.) circle (1.5pt);
\draw [color=qqqqff] (-6.,0.)-- ++(-1.5pt,-1.5pt) -- ++(3.0pt,3.0pt) ++(-3.0pt,0) -- ++(3.0pt,-3.0pt);
\draw [fill=uququq] (-5.,0.) circle (1.5pt);
\draw [fill=xdxdff] (0.,2.) circle (1.5pt);
\draw [fill=xdxdff] (0.,4.) circle (1.5pt);
\draw [fill=xdxdff] (0.,6.) circle (1.5pt);
\end{scriptsize}
\end{tikzpicture}
\caption{Representation of a part of the mesh used in Godunov's scheme and of the functions $\bar{u}^n_j$ and $\tilde{u}^{n}_j$ for $n=0,1,2$ and $j=-1,0,1,2$. The functions $\bar{u}^n$ and $\tilde{u}^n$ are defined respectively at $t=t^n$ and in the interval $[t^n,t^{n+1}]$.}\label{Godunov_griglia}
\end{figure}
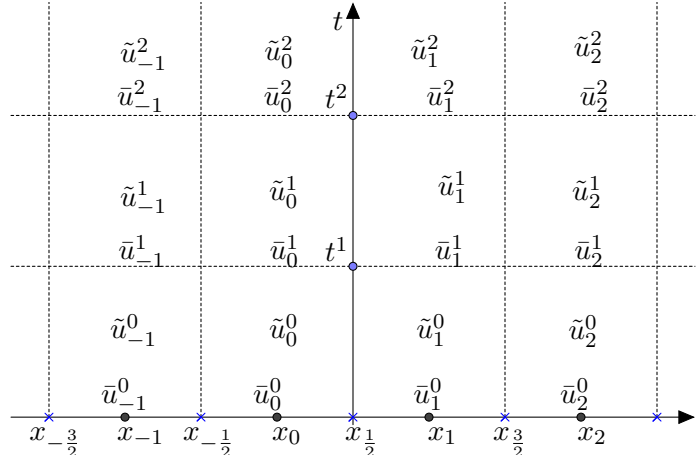
By the integral form of the conservation law, we have
\begin{equation}
\begin{split}
\int_{x_{j-1/2}}^{x_{j+1/2}}\tilde{u}^n(t^{n+1},\xi)\, d\xi & = \int_{x_{j-1/2}}^{x_{j+1/2}}\tilde{u}^n(t^n,\xi)\, d\xi + \int_{t^n}^{t^{n+1}}f(\tilde{u}^n(\tau,x_{j-1/2}))\, d\tau\, +\\
& -\int_{t^n}^{t^{n+1}}f(\tilde{u}^n(\tau,x_{j+1/2}))\,d\tau.
\end{split}
\end{equation}
Dividing both sides of the previous equation by $h$, we obtain
\begin{equation}\label{Godunov_scheme_integral_form_of_cl}
\begin{split}
\dfrac{1}{h}\int_{x_{j-1/2}}^{x_{j+1/2}}\tilde{u}^n(t^{n+1},x)\, dx & = \dfrac{1}{h}\int_{x_{j-1/2}}^{x_{j+1/2}}\tilde{u}^n(t^n,x)\, dx + \dfrac{k}{h}\left[\dfrac{1}{k}\int_{t^n}^{t^{n+1}}f(\tilde{u}^n(\tau,x_{j-1/2}))\, d\tau\right]+\\
& -	\dfrac{k}{h}\left[\dfrac{1}{k}\int_{t^n}^{t^{n+1}}f(\tilde{u}^n(\tau,x_{j+1/2}))\,d\tau\right].
\end{split}
\end{equation}
Recalling the definitions of $\bar{u}_j^n$ and $\bar{u}_j^{n+1}$ and introducing the numerical flux function $F$ defined by
\begin{equation}\label{numerical flux}
F(\bar{u}^n_j,\bar{u}^n_{j+1}):=\dfrac{1}{k}\int_{t^n}^{t^{n+1}}f(\tilde{u}^n(\tau,x_{j+1/2}))\, d\tau,
\end{equation}
the equation (\ref{Godunov_scheme_integral_form_of_cl}) becomes
\begin{equation}\label{Godunov_in_conservation_form}
\begin{split}
\bar{u}^{n+1}_j=\bar{u}^n_j-\dfrac{k}{h}[F(\bar{u}^n_j,\bar{u}^n_{j+1})-F(\bar{u}^n_{j-1},\bar{u}^n_j)].
\end{split}
\end{equation}
Note that the expression (\ref{numerical flux}) is consistent, whether the value of $\tilde{u}(t,x_{j+1/2})$ depends only by the left and right data of the local Riemann problem in $x=x_{j+1/2}$, which are respectively $\bar{u}^n_j$ and $\bar{u}^n_{j+1}$. This happens when the time step $k$ is sufficiently small.\\
The following proposition states a condition on the ratio $k/h$ to avoid that waves arising in the interactions between the two neighbouring Riemann problems centred in $x_{j-1/2}$ and $x_{j+1/2}$ cross the edges of the cell $C_j$ for every $j\in\mathbb{Z}$, so that they do not influence the value of $\tilde{u}$ in the interval $[t^n,t^{n+1}]$ at the points $\lbrace x_{j+1/2}\rbrace_{j\in\mathbb{Z}}$.
\begin{prop}\label{prop_CFL_cond}
Let $n\in \mathbb{N}$ be fixed. The value $\tilde{u}(x_{j+1/2},t)$ is constant over $[t^n,t^{n+1}]$, provided that the following CFL condition holds:
\begin{equation}\label{CFL_condiftion}
\left| \dfrac{k}{h}\lambda^n\right|\leq 1,
\end{equation}
where $\lambda^n=\max\lbrace |\lambda_i(\bar{u}^n_j)|\, : \, i=1,...,N \; \text{ and } \; j\in \mathbb{Z}\rbrace$ and $\lbrace\lambda_i\rbrace_{i=1}^N$ are the eigenvalues of the Jacobian matrix $Df$ of the flux function $f$.
\end{prop}
\begin{proof}
Let $j\in\mathbb{Z}$ be fixed and let us consider the Riemann problem (see Figure \ref{Godunov_esempio_soluzione})
\begin{equation}\label{Riemann_problem_dim_cfl_condition}
\begin{cases}
\partial_t u + \partial_x[f(u)] = 0,\\
u(t^n,x) = \begin{cases}
u^l:= \bar{u}^n_j & \text{if } x\leq x_{j+1/2},\\
u^r:= \bar{u}^n_{j+1} & \text{if } x>x_{j+1/2}.
\end{cases}
\end{cases}
\end{equation}
Let us call $\mu_0=u^l$, $\mu_{N}=u^r$ and $\lbrace \mu_i \rbrace_{i=1}^{N-1}$ the middle states given by the standard Riemann solver for the Riemann problem (\ref{Riemann_problem_dim_cfl_condition}).\\
Fix $i\in\lbrace 1,...,N \rbrace$. If the $i$-th characteristic field is genuinely non-linear, then the standard Riemann solver connects $\mu_{i-1}$ and $\mu_{i}$ with a rarefaction wave or with a shock. We know by Proposition \ref{rarefaction_wave_properties} that the propagation speeds of the rarefaction range between $\lambda_i(\mu_{i-1})$ and $\lambda_i(\mu_{i})$, while if a shock appears, its speed is $\lambda_i(\mu_{i-1},\mu_{i})$, where $\lambda_i(\mu_{i-1},\mu_{i})$ is the eigenvalue of the averaged matrix
$$Df(\mu_{i-1},\mu_{i})= \int_0^1 Df(\xi\,  \mu_{i-1}+(1-\xi)\mu_i)\, d\xi.$$
By the Lax entropy condition (\ref{Lax_entropy_condition}), we have
\begin{equation}
\lambda_i(\mu_i)\leq \lambda_i(\mu_{i-1},\mu_{i})\leq \lambda_{i}(\mu_{i-1}).
\end{equation}
If the $i$-th characteristic field is linearly degenerate, a contact discontinuity connecting $\mu_{i-1}$ to $\mu_{i}$ appears. Its propagation speed is $\lambda_i(\mu_{i-1})$.\\
Hence the propagation speed of a wave arising in the Cauchy problem (\ref{num_methods_Cauchy_problem_general}) at the time $t^n$ with initial datum
$$x\to \bar{u}^n(x) = \sum_{j\in \mathbb{Z}} \left[\bar{u}^n_j\, \mathbf{1}_{C_j}(x)\right],$$
is at most equal to
$$\lambda^n=\max\lbrace |\lambda_i(\bar{u}^n_j)|\, : \, i=1,...,N \; \text{ and } \; j\in \mathbb{Z}\rbrace.$$
For every $j$, the two points $x_{j-1/2}$ and $x_{j+1/2}$ satisfy $x_{j-1/2}-x_{j+1/2}= h$ and if the condition (\ref{CFL_condiftion}) holds, we have
$$\left|k\lambda^n\right|\leq h.$$
Therefore the interaction between two waves arising in the local Riemann problems centred in $x=x_{j-1/2}$ and $x=x_{j+1/2}$ with initial datum $\bar{u}^n$, remains in the interval $[x_{j-1/2},x_{j+1/2}]$ for every $t$ in $[t^n,t^{n+1}]$ (Figure \ref{Godunov_esempio_soluzione}).
\end{proof}
\begin{figure}[h]
\centering
\definecolor{uququq}{rgb}{0.25098039215686274,0.25098039215686274,0.25098039215686274}
\definecolor{xdxdff}{rgb}{0.49019607843137253,0.49019607843137253,1.}
\begin{tikzpicture}[scale =0.2,line cap=round,line join=round,>=triangle 45,x=1.0cm,y=1.0cm]
\clip(-4.5,-2.) rectangle (40.,17.);
\draw [dash pattern=on 4pt off 4pt,domain=-3.:40.] plot(\x,{(-14.-0.*\x)/-1.});
\draw [dash pattern=on 4pt off 4pt] (12.,-0.2) -- (12.,17.);
\draw [dash pattern=on 4pt off 4pt] (24.,-0.2) -- (24.,17.);
\draw [dash pattern=on 4pt off 4pt] (36.,-0.2) -- (36.,17.);
\draw [dash pattern=on 4pt off 4pt] (0.,-0.2) -- (0.,17.);
\draw (-4.2,14.2) node[anchor=north west] {$t^{n+1}$};
\draw [dash pattern=on 4pt off 4pt,domain=-3.:40.] plot(\x,{(-0.-0.*\x)/-1.});
\draw (-4.2,2.10) node[anchor=north west] {$t^n$};
\draw (8.89,0.) node[anchor=north west] {$x_{j-1/2}$};
\draw (20.84,0.) node[anchor=north west] {$x_{j+1/2}$};
\draw (33.13,0.) node[anchor=north west] {$x_{j+3/2}$};
\draw (4.31,6.58) node[anchor=north west] {$\bar{u}^n_{j-1}$};
\draw (16.50,6.58) node[anchor=north west] {$\bar{u}^n_{j}$};
\draw (29.10,5.98) node[anchor=north west] {$\bar{u}^n_{j+1}$};
\draw (12.,0.)-- (8.68,14.);
\draw (12.,0.)-- (6.71,14.);
\draw (12.,0.)-- (7.72,14.);
\draw (12.,0.)-- (9.41,14.);
\draw (12.,0.)-- (10.37,14.);
\draw (12.,0.)-- (14.18,14.);
\draw (24.,0.)-- (20.96,14.);
\draw (24.,0.)-- (30.35,10.23);
\draw (36.,0.)-- (30.35,10.23);
\draw (30.35,10.23)-- (32.78,14.);
\draw (30.35,10.23)-- (28.70,14.);
\begin{scriptsize}
\draw [fill=xdxdff] (12.,0.) circle (1.5pt);
\draw [fill=xdxdff] (24.,0.) circle (1.5pt);
\draw [fill=xdxdff] (36.,0.) circle (1.5pt);
\draw [fill=uququq] (0.,0.) circle (1.5pt);
\draw [fill=xdxdff] (0.,14.) circle (1.5pt);
\end{scriptsize}
\end{tikzpicture}
\caption{Example of solutions of neighbouring Riemann problems. For the CFL condition (\ref{CFL_condiftion}), the waves that arise from the problems centred in $x_{j-{1/2}}$ and $x_{j+1/2}$ can interact but they remain in the cell $C_j$. The solution is then constant along each line $x_{j+1/2}$.}\label{Godunov_esempio_soluzione}
\end{figure}
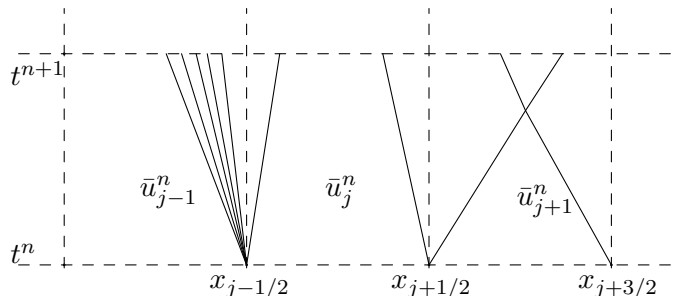
Finally, if the CFL condition holds and by the self-similarity of the solution of the Riemann problem, the integral
$$\int_{t^n}^{t^{n+1}}f(\tilde{u}^n(\tau,x_{j+1/2}))\, d\tau$$
is trivial because  $\tilde{u}^n$ is constant along the line $(t,x_{j+1/2})$ for $t\in [t^n,t^{n+1}]$. Therefore, if we denote this value with $u^*(\bar{u}^n_j,\bar{u}^n_{j+1})$, we find
$$F(\bar{u}^n_j,\bar{u}^n_{j+1})= f(u^*(\bar{u}^n_j,\bar{u}^n_{j+1}))$$
and the Godunov's method becomes
\begin{equation}\label{Godunov_method_final_form}
\bar{u}^{n+1}_j=\bar{u}^n_j-\dfrac{k}{h}[f(u^*(\bar{u}^n_j,\bar{u}^n_{j+1}))-f(u^*(\bar{u}^n_{j-1},\bar{u}^n_{j}))].
\end{equation}

\section{Godunov's method for the Aw-Rascle-Zhang system}
In this section we are going to specialize the Godunov's method for the ARZ system.\\
Let $u^0=(\rho^0,z^0):\mathbb{R}\to \mathbb{R}^2$ be a known function, let
$$ u=(\rho,z):=(\rho,\rho(v+p(\rho)))$$
be the vector of the conserved variables and let $u^n(x)=u(t^n,x)$ be the exact solution profile of the Cauchy problem
\begin{equation}\label{ARZ_Cauchy_probelem_rho_y}
\begin{cases}
\begin{cases}
\partial_t\rho+\partial_x(\rho v)=0,\\
\partial_t z+\partial_x(z v)=0,\\
\end{cases}\\
(\rho,z)(0,x)=(\rho^0,z^0)(x).
\end{cases}
\end{equation}
at time $t^n$. By Proposition \ref{ARZ_properties_rho_z}, the flux function in the coordinates $(\rho,z)$ is
\begin{equation}
f(\rho,z)= \begin{pmatrix}
f_1(\rho,z)\\
f_2(\rho,z)
\end{pmatrix}=
\begin{pmatrix}
\rho v\\
z v
\end{pmatrix}= \begin{pmatrix}
z-\rho\,p(\rho)\\
\frac{z^2}{\rho}-z\,p(\rho)
\end{pmatrix}
\end{equation}
and the eigenvalues of the Jacobian matrix $Df$ are
\begin{equation}
\lambda_1(\rho,z)=-p(\rho)+\dfrac{z}{\rho}-\rho p'(\rho) \; \text{ and } \; \lambda_2(\rho,z) = -p(\rho)+\dfrac{z}{\rho}.
\end{equation}
We define the piecewise constant averaged initial datum as
\begin{equation}
x\to \bar{u}^0(x)=(\bar{\rho}^0,\bar{z}^0)(x) = \sum_{j\in \mathbb{Z}}\left[(\bar{\rho}^0_j,\bar{z}^0_j) \,  \mathbf{1}_{C_j}(x) \right],
\end{equation}
where
\begin{equation}
\bar{\rho}^0_j:= \dfrac{1}{h}\int_{x_{j-1/2}}^{x_{j+1/2}}\rho^0(\xi)\, d\xi \;\; \text{ and } \;\; \bar{z}^0_j:= \dfrac{1}{h}\int_{x_{j-1/2}}^{x_{j+1/2}}z^0(\xi)\, d\xi.
\end{equation}
If the CFL condition
\begin{equation}
\left| \dfrac{k}{h}\lambda^0\right|\leq 1 \; \text{ with } \; \lambda^0=\max\lbrace |\lambda_i(\bar{u}^0_j)|\, : \, i=1,2 \; \text{ and } \; j\in \mathbb{Z}\rbrace
\end{equation}
holds, we can apply the equation (\ref{Godunov_method_final_form}) to find the approximate solution $\bar{u}^1$ of the profile $u^1$ of the exact solution at the new time step, i.e. for every $j\in \mathbb{Z}$ we have
\begin{equation}
\begin{split}
& \bar{\rho}^{1}_j=\bar{\rho}^0_j-\dfrac{k}{h}\left[f_1(u^*(\bar{u}^0_j,\bar{u}^0_{j+1}))-f_1(u^*(\bar{u}^0_{j-1},\bar{u}^0_{j})) \right] \; \text{ and}\\
& \bar{z}^{1}_j=\bar{z}^0_j-\dfrac{k}{h}\left[f_2(u^*(\bar{u}^0_j,\bar{u}^0_{j+1}))-f_2(u^*(\bar{u}^0_{j-1},\bar{u}^0_{j})) \right],
\end{split}
\end{equation}
where $u^*(\bar{u}^0_k,\bar{u}^0_{k+1})$ is the value of the exact solution at $x=x_{k+1/2}$ and for $t\in [t^0,t^1]$ to the local Riemann problem centred in $x_{k+1/2}$, which is
\begin{equation}
\begin{cases}
\begin{cases}
\partial_t\rho+\partial_x(\rho v)=0,\\
\partial_t z+\partial_x(z v)=0,
\end{cases}\\
(\rho,z)(t^0,x)=\begin{cases}
\bar{u}^0_k & \text{if } x\leq x_{k+1/2},\\
\bar{u}^0_{k+1} & \text{if } x>x_{k+1/2}.
\end{cases}
\end{cases}
\end{equation}
Repeating this process, we can compute the approximate solution $\bar{u}^n=(\bar{\rho}^n,\bar{z}^n)$ for every $n\in \mathbb{N}$, provided that the CFL condition 
\begin{equation}
\left| \dfrac{k}{h}\lambda^n\right|\leq 1 \; \text{ with } \; \lambda^n=\max\lbrace |\lambda_i(\bar{u}^n_j)|\, : \, i=1,2 \; \text{ and } \; j\in \mathbb{Z}\rbrace
\end{equation}
holds.\\
Fix $j\in \mathbb{Z}$. Let us consider the Riemann problem at time $t^n$ and centred in $x_{j+1/2}$, namely
\begin{equation}\label{ARZ_riemann_probelem_rho_y}
\begin{cases}
\begin{cases}
\partial_t\rho+\partial_x(\rho v)=0,\\
\partial_t z+\partial_x(z v)=0,
\end{cases}\\
(\rho,z)(t^n,x)=\begin{cases}
\bar{u}^n_j & \text{if } x\leq x_{j+1/2},\\
\bar{u}^n_{j+1} & \text{if } x>x_{j+1/2}.
\end{cases}
\end{cases}
\end{equation}
We can give explicitly the value $\tilde{u}^{n}_{j+1/2}:=u^*(\bar{u}^0_j,\bar{u}^0_{j+1})$ of the solution to the Riemann problem (\ref{ARZ_riemann_probelem_rho_y}) at $x_{j+1/2}$ in the interval $[t^n,t^{n+1}]$.\\
Let us call $u^l = \bar{u}^n_j$ and $u^r =\bar{u}^n_{j+1}$ and let $u^m$ be the middle state of the classical solution. As usual, it is convenient to use the $(\rho,v)$ coordinates, hence we define $(\rho^l,v^l)$, $(\rho^r,v^r)$ and $(\rho^*,v^*)$ respectively as the corresponding points of $u^l$, $u^r$ and $\tilde{u}^{n}_{j+1/2}$ in the $(\rho,v)$ plane. We recall (see Proposition \ref{ARZ_properties_rho_v}) that the eigenvalues in these coordinates are
$$\lambda_1(\rho,v)=v-\rho p'(\rho) \; \text{ and } \; \lambda_2(\rho,v)=v$$
and that the Lax curves passing through a point $(\rho_0,v_0)$ are
$$L_1(\rho,\rho_0,v_0) = v_0+p(\rho_0)-p(\rho) \; \text{ and }\; L_2(\rho,\rho_0,v_0)=v_0.$$

\subsection{First case: $v^r = L_1(\rho^l,v^l,\rho^r)$}
Whenever the case $v^l+p(\rho^l)= v^r+p(\rho^r)$ occurs, there are no intermediate states between $(\rho^l,v^l)$ and $(\rho^r,v^r)$. We have to distinguish two situations.
\begin{enumerate}
\item If $\rho^l<\rho^r$, then $(\rho^l,v^l)$ and $(\rho^r,v^r)$ are linked by the standard Riemann solver with a shock. The propagation speed is given by the Rankine-Hugoniot condition, i.e.
\begin{equation}
\lambda = \dfrac{\rho^l v^l-\rho^r v^r}{\rho^l-\rho^r}.
\end{equation}
If $\lambda \geq 0$, we have $(\rho^*,v^*)=(\rho^l,v^l)$. Otherwise, we have $(\rho^*,v^*)=(\rho^r,v^r)$.
\item If $\rho^l \geq \rho^r$, then a rarefaction joins $(\rho^l,v^l)$ and $(\rho^r,v^r)$. Let $(\rho^M,v^M)$ be the maximum of the function $\rho \to \psi(\rho):= \rho L_1(\rho,\rho^l,v^l)$, which exists because the function is strictly concave and
$$\lim_{\rho\to +\infty}\psi(\rho)=-\infty.$$
Since $\lambda_1(\rho,L_1(\rho,\rho^l,v^l))$ is the derivative of the function $\psi$ in the $(\rho,\rho v)$ plane (see Proposition \ref{eigenvalue_as_slope}), we have
$$\lambda_1(\rho^M,v^M)=0.$$
Let $(\rho^\sigma,v^\sigma)$ be a point of the rarefaction. By Proposition \ref{rarefaction_wave_properties}, the propagation speed of the rarefaction in $(\rho^\sigma,v^\sigma)$ is $\lambda_1(\rho^\sigma,v^\sigma)$.\\
If $\rho^r>\rho^M$, then $\lambda_1(\rho^\sigma,v^\sigma)<0$ for every $\sigma$. Hence $(\rho^*,v^*)=(\rho^r,v^r)$.\\
Else if  $\rho^l<\rho^M$, then $\lambda_1(\rho^\sigma,v^\sigma)>0$ for every $\sigma$. Then $(\rho^*,v^*)=(\rho^l,v^l)$.\\
Otherwise, since the wave on the line $x_{j+1/2}$ travels with propagation speed equal to zero, $(\rho^*,v^*)$ is the point of the rarefaction such that
$$\lambda_1(\rho^*,v^*)=0,$$
which is $(\rho^*,v^*)=(\rho^M,v^M)$.
\end{enumerate}

\subsection{Second case: $v^r \neq L_1(\rho^l,v^l,\rho^r)$}
When $v^r+p(\rho^r) \neq v^l+p(\rho^l)$, we have the following two cases.
\begin{enumerate}
\item If $v^l=v^r$, then $(\rho^l,v^l)$ and $(\rho^r,v^r)$ are connected by a contact discontinuity and no intermediate states appear. The solution is $(\rho^*,v^*)=(\rho^l,v^l)$, because the discontinuity propagates with speed $v^r>0$.
\item If $v^l \neq v^r$, the intermediate state $(\rho^m,v^m)$ appears. We can repeat the discussion of the first case ($v^r = L_1(\rho^l,v^l,\rho^r)$) with $(\rho^m,v^m)$ instead of $(\rho^r,v^r)$.
\end{enumerate}

\section{Numerical method for the Riemann solver $\mathcal{RS}^\alpha_1$}
Our aim is to modify the Godunov's method to find the numerical solutions for the moving constraint problem which correspond to the Riemann solver $\mathcal{RS}^\alpha_1$. See \cite{chalons_delle_monache_goatin} for the scalar case.\\
Assume that the speed of the bus $\dot{y}(t)$ is constant and its value is $\dot{y}(t)=\bar{V}\in[0,V_b]$ for every $t\in\mathbb{R}^+$.
Let us consider the Riemann problem (\ref{ARZ_riemann_probelem_rho_y}) with the constraint
\begin{equation}\label{vincolo}
\rho(t,y(t)) (v(t,y(t))-\bar{V})\leq \rho_\alpha^2 \, p'(\rho_\alpha)=F_\alpha,
\end{equation}
where $\alpha \in (0,1)$ is the reduction rate in the road capacity caused by the bus, $\rho_\alpha$ is the solution of the equation
$$p(\rho_\alpha)+\rho_\alpha\,p'(\rho_\alpha)=\omega_\alpha-\bar{V}$$
and $\omega_\alpha=p(\alpha R)$. Let $y_0$ be the initial position of the bus.\\
We recall the definitions of the points $(\hat{\rho},\hat{v})$, $(\check{\rho}_1,\check{v}_1)$ and $(\rho^m,v^m)$. Let $I$ be the set
\begin{equation*}
\begin{split}
I & =\lbrace \rho \in [0,R]: \, \rho L_1(\rho,\rho^l,v^l)=\rho (v^l+p(\rho^l)-p(\rho))=F_\alpha+\rho \bar{V}\rbrace=\\
&=\lbrace \rho\in [0,R]:\, \rho(L_1(\rho,\rho^l,v^l)-\bar{V})=F_\alpha \rbrace.
\end{split}
\end{equation*}
If $I\neq \emptyset$, $(\hat{\rho},\hat{v})$ and $(\check{\rho}_1,\check{v}_1)$ are the points defined by 
\begin{equation}\label{def_u_hat_u_check_1}
\hat{\rho}=\max I, \; \; \hat{v}=\dfrac{F_\alpha}{\hat{\rho}}+\bar{V}, \; \; \check{\rho}_1 = \min I \; \mbox{ and } \; \check{v}_1=\dfrac{F_\alpha}{\check{\rho}_1}+\bar{V}.
\end{equation}
Finally, let $(\rho^m,v^m)$ be the intermediate state of the classical solution, defined by 
$$v^m=v^r, \; \; \; L_1(\rho^m,\rho^l,v^l)=v^m \; \text{ and } \; \rho^m>0.$$
The corresponding points in the $(\rho,z)$ plane are
\begin{equation}\label{def_punti_risolutori}
\begin{split}
&\hat{u}=(\hat{\rho},\hat{z})= (\hat{\rho},\hat{\rho}(\hat{v}+p(\hat{\rho}))), \; \; \check{u}_1=(\check{\rho}_1,\check{z}_1)= (\check{\rho}_1,\check{\rho}_1(\check{v}_1+p(\check{\rho}_1)))\\
& \text{and } \; u^m=(\rho^m,z^m)= (\rho^m,\rho^m(v^m+p(\rho^m))).
\end{split}
\end{equation}
For the Riemann solver $\mathcal{RS}^\alpha_1$, we recall that both components $(\rho,z)$ are conserved.\\
Let $\mathcal{RS}$ be the classical Riemann solver and let
$$\bar{\rho}(u^l,u^r)(\cdot)$$
be the $\rho$ component of the classical solution $\mathcal{RS}(u^l,u^r)(\cdot)$.\\
\begin{remark}
Let $\mathcal{Z}:\mathbb{R^+}\times \mathbb{R}^+ \to \mathbb{R}^+\times \mathbb{R}^+$ be the map
$$\mathcal{Z}:(\rho,v)\to (\rho,z)=(\rho,\rho\,(v+p(\rho))).$$
The Riemann solvers $\mathcal{RS}^\alpha_1$ and $\mathcal{RS}^\alpha_2$ are defined for the non-conserved variables $(\rho,v)$. We still denote $\mathcal{RS}^\alpha_1$ and $\mathcal{RS}^\alpha_2$ the Riemann solvers
$$\mathcal{Z}\circ \mathcal{RS}^\alpha_1 \; \text{ and } \; \mathcal{Z}\circ \mathcal{RS}^\alpha_2$$
for the conserved variables $(\rho,z)$.
\end{remark}
\subsection{The bus is not influenced by the preceding vehicles}
Fix $n\in\mathbb{N}$. Let us suppose that we have computed the piecewise constant approximate solution $\bar{u}^n$ at the time $t^n$ with the Godunov's method. We follow the ideas of \cite{chalons_delle_monache_goatin, boutin_chalons, aguillon_system}.\\
First, consider a bus not influenced by the preceding vehicles. Its position at the time $t^n$ is $y^n:=y(t^n)=y_0+\bar{V}t^n\in C_m$ for some $m \in \mathbb{Z}$ and the value of the approximate solution at $y^n$ is $\bar{u}^n_m$. Let 
\begin{equation}
\bar{\rho}^n_m \; \text{ and } \; \bar{z}^n_m
\end{equation}
be respectively the $\rho$ and the $z$ component of the approximate solution $\bar{u}^n$ in the $m$-th cell.\\
If the Riemann solver $\mathcal{RS}^\alpha_1$ does not give the classical solution, a non-classical shock appears in $x=y(t)$. Following \cite{chalons_delle_monache_goatin}, a first idea to detect the non-classical shock, could be to check if the inequality
\begin{equation}\label{first_inequality}
f_1(\bar{u}^n_m) > F_\alpha+ \bar{V} \bar{\rho}^n_m
\end{equation}
holds. Since the non-classical shock arises as the solution given by $\mathcal{RS}^\alpha_1$ to the Riemann problem with initial datum
\begin{equation}\label{initial_datum_discontinuity_reconstruction}
u(0,x)=\begin{cases}
\bar{u}^n_{m-1} & \text{if } x\leq x_{m-1/2},\\
\bar{u}^n_{m+1} & \text{if } x>x_{m-1/2},
\end{cases}
\end{equation}
we will make a reconstruction of the discontinuity due to the presence of the non-classical shock, if also the inequality
\begin{equation}\label{second_inequality}
f_1(\mathcal{RS}(\bar{u}^n_{m-1},\bar{u}^n_{m+1})(\bar{V}))>F_\alpha+\bar{V}\bar{\rho}(\bar{u}^n_{m-1},\bar{u}^n_{m+1})(\bar{V})
\end{equation}
holds. In this case we modify the Godunov's scheme as follows.\\
We introduce in the $m$-th cell one left state $u^n_{m,l}=(\rho^n_{m,l},z^n_{m,l})$ and one right state $u^n_{m,r}=(\rho^n_{m,r},z^n_{m,r})$ as
$$u^n_{m,l}= \hat{u} \; \text{ and } \; u^n_{m,r}=\check{u}_1,$$
where $\hat{u}$ and $\check{u}_1$ are the points defined by the relations (\ref{def_punti_risolutori}) for the Riemann problem (\ref{ARZ_riemann_probelem_rho_y}) with initial datum (\ref{initial_datum_discontinuity_reconstruction}). Then we replace the solution $\bar{u}^n_m$ obtained with the Godunov's method in the $m$-th cell, with the function $u^n_\text{rec}=(\rho^n_\text{rec},z^n_\text{rec})$ defined by
\begin{equation}
\begin{split}
&\rho^n_\text{rec}=\rho^n_{m,l}\mathbf{1}_{(x_{m-1/2},\bar{x}^\rho_m)}+\rho^n_{m,r}\mathbf{1}_{(\bar{x}^\rho_m,x_{m+1/2})} \; \text{ and}\\
&z^n_\text{rec}=z^n_{m,l}\mathbf{1}_{(x_{m-1/2},\bar{x}^z_m)}+z^n_{m,r}\mathbf{1}_{(\bar{x}^z_m,x_{m+1/2})},
\end{split}
\end{equation}
where we have used the two points
$$\bar{x}^\rho_m=x_{m-1/2}+h\, d^{n,\rho}_m \; \text{ and } \; \bar{x}^z_m=x_{m-1/2}+h \, d^{n,z}_m$$
defined for two suitable constants $d^{n,\rho}_m$ and $d^{n,z}_m$ in $[0,1]$; see Figure \ref{fig_ricostruzione_discontinuita}.\\
In agreement with the first Riemann solver character, our aim is to preserve conservation. Therefore we require 
\begin{equation}
\begin{split}
&\rho^n_{m,l}d^{n,\rho}_m+\rho^n_{m,r}(1-d^{n,\rho}_m)=\bar{\rho}^n_m \; \text{ and}\\
&z^n_{m,l}d^{n,z}_m+z^n_{m,r}(1-d^{n,z}_m)=\bar{z}^n_m.
\end{split}
\end{equation}
Solving these two equations w.r.t. $d^{n,\rho}_m$ and $d^{n,z}_m$, we find
\begin{equation}\label{costanti_ricostruzione_discontinuita}
d^{n,\rho}_m=\dfrac{\bar{\rho}^n_m-\rho^n_{m,r}}{\rho^n_{m,l}-\rho^n_{m,r}} \; \text{ and } \; d^{n,z}_m=\dfrac{\bar{z}^n_m-z^n_{m,r}}{z^n_{m,l}-z^n_{m,r}}.
\end{equation}
Clearly, the conditions $d^{n,\rho}_m\in [0,1]$ and $d^{n,z}_m\in [0,1]$ are necessary to reconstruct the discontinuity in the cell $C_m$. These two constants are in general different.\\
By Remark \ref{nonclassical_shock}, the non-classical shock travels with the bus speed $\bar{V}$. Assuming that the discontinuity travels with the same speed of the non-classical shock, if we denote
$$\Delta t^\rho_{m+1/2}\; \text{ and } \; \Delta t^z_{m+1/2}$$
respectively the time needed by the $\rho$ and the $z$ component of the discontinuity to reach the interface $x_{m+1/2}$, we have
\begin{equation}
\begin{split}
& \Delta t^\rho_{m+1/2}\bar{V}=h(1-d^{n,\rho}_m) \Longleftrightarrow \Delta t^\rho_{m+1/2}=h\dfrac{1-d^{n,\rho}_m}{\bar{V}} \; \text{ and }\\
& \Delta t^z_{m+1/2}\bar{V}=h(1-d^{n,z}_m) \Longleftrightarrow \Delta t^z_{m+1/2}=h\dfrac{1-d^{n,z}_m}{\bar{V}}.
\end{split}
\end{equation}
Let
$$F_1(\bar{u}^n_m,\bar{u}^n_{m+1}) \; \text{ and } \; F_2(\bar{u}^n_m,\bar{u}^n_{m+1})$$
be the two components of the numerical flux in $x_{m+1/2}$. We have
\begin{equation}
F_1(\bar{u}^n_m,\bar{u}^n_{m+1})=\begin{cases}
f_1(u^n_{m,r}) & \text{if } t\in [t^n,t^n+\Delta t^\rho_{m+1/2}],\\
f_1(u^n_{m,l}) & \text{if } \Delta t^\rho_{m+1/2}<k \; \text{ and } \; t\in (t^n+\Delta t^\rho_{m+1/2},t^{n+1}].
\end{cases}
\end{equation}
Similarly
\begin{equation}
F_2(\bar{u}^n_m,\bar{u}^n_{m+1})=\begin{cases}
f_2(u^n_{m,r}) & \text{if } t\in [t^n,t^n+\Delta t^z_{m+1/2}],\\
f_2(u^n_{m,l}) & \text{if } \Delta t^z_{m+1/2}<k \; \text{ and } \; t\in (t^n+\Delta t^z_{m+1/2},t^{n+1}].
\end{cases}
\end{equation}
We can rewrite these expressions in the compact form
\begin{equation}\label{flux_redefinition}
\begin{split}
& F_1(\bar{u}^n_m,\bar{u}^n_{m+1}) = \dfrac{1}{k}\left[ \min(\Delta t^\rho_{m+1/2},k)f_1(u^n_{m,r}) + \max(k-\Delta t^\rho_{m+1/2},0)f_1(u^n_{m,l})\right] \; \text{ and}\\
& F_2(\bar{u}^n_m,\bar{u}^n_{m+1}) = \dfrac{1}{k}\left[ \min(\Delta t^z_{m+1/2},k)f_2(u^n_{m,r}) + \max(k-\Delta t^z_{m+1/2},0)f_2(u^n_{m,l})\right].
\end{split}
\end{equation}
\begin{remark}
Implementing this method, we also modify the numerical flux in $x_{m-1/2}$ for the left state $\bar{u}^n_{m-1}$ and the new right state $\hat{u}$, i.e.
$$F(\bar{u}^n_{m-1},\hat{u})= f(u^*(\bar{u}^n,\hat{u})).$$
This is done to preserve the consistency of the Godunov's scheme.
\end{remark}
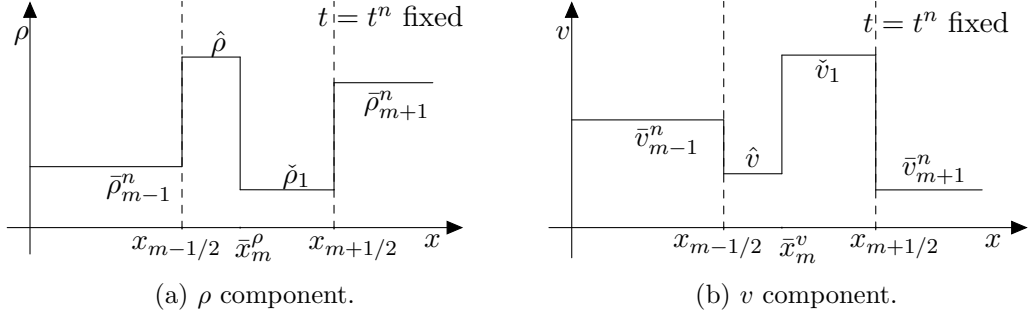
\begin{figure}[h]
\begin{subfigure}[h]{0.45\textwidth}
\definecolor{xdxdff}{rgb}{0.6588235294117647,0.6588235294117647,0.6588235294117647}
\begin{tikzpicture}[scale =0.1,line cap=round,line join=round,>=triangle 45,x=1.0cm,y=1.0cm]
\draw[->,color=black] (-3.,0.) -- (57.,0.);
\draw[->,color=black] (0.,-3.5) -- (0.,30.);
\clip(-3.5,-5) rectangle (57.,30.);
\draw (12,0.2) node[anchor=north west] {$x_{m-1/2}$};
\draw (35.28,0.2) node[anchor=north west] {$x_{m+1/2}$};
\draw (-3.5,28) node[anchor=north west] {$\rho$};
\draw (50.5,0.4) node[anchor=north west] {$x$};
\draw (36.5,30.7) node[anchor=north west] {$t=t^n \text{ fixed}$};
\draw (25.6,0.5) node[anchor=north west] {$\bar{x}^\rho_m$};
\draw [dash pattern=on 3pt off 3pt] (20.,-0.2) -- (20.,30.);
\draw [dash pattern=on 3pt off 3pt] (40.,-0.2) -- (40.,30.);
\draw (0.,8.08)-- (20.,8.08);
\draw (20.,22.61)-- (27.65,22.61);
\draw (20.,22.61)-- (20.,8.08);
\draw (27.65,22.61)-- (27.65,5.02);
\draw (27.65,5.02)-- (40.,5.);
\draw (40.,5.)-- (40.,19.21);
\draw (40.,19.21)-- (53,19.21);
\draw (8.68,8.5) node[anchor=north west] {$\bar{\rho}^n_{m-1}$};
\draw (31.88,9.7) node[anchor=north west] {$\check{\rho}_1$};
\draw (22.45,27.5) node[anchor=north west] {$\hat{\rho}$};
\draw (42.59,19.5) node[anchor=north west] {$\bar{\rho}^n_{m+1}$};
\begin{scriptsize}
\draw [fill=xdxdff] (20.,0.) circle (1.5pt);
\draw [fill=xdxdff] (40.,0.) circle (1.5pt);
\draw [fill=xdxdff] (27.65,0.) circle (1.5pt);
\end{scriptsize}
\end{tikzpicture}
\caption{$\rho$ component.}\label{fig_ricostruzione_discontinuita_rho}
\end{subfigure}
\quad
\begin{subfigure}[h]{0.45\textwidth}
\definecolor{xdxdff}{rgb}{0.6588235294117647,0.6588235294117647,0.6588235294117647}
\begin{tikzpicture}[scale =0.1,line cap=round,line join=round,>=triangle 45,x=1.0cm,y=1.0cm]
\draw[->,color=black] (-3.,0.) -- (60.,0.);
\draw[->,color=black] (0.,-4.) -- (0.,30.);
\clip(-3.5,-5.) rectangle (60.,30.);
\draw (12,0.4) node[anchor=north west] {$x_{m-1/2}$};
\draw (35.22,0.4) node[anchor=north west] {$x_{m+1/2}$};
\draw (-3.5,28) node[anchor=north west] {$v$};
\draw (53,0.4) node[anchor=north west] {$x$};
\draw (37.,30.) node[anchor=north west] {$t=t^n \text{ fixed}$};
\draw (25.95,0.2) node[anchor=north west] {$\bar{x}^v_m$};
\draw [dash pattern=on 3pt off 3pt] (20.,-0.3) -- (20.,30.);
\draw [dash pattern=on 3pt off 3pt] (40.,-0.3) -- (40.,30.);
\draw (0.,14.28)-- (20.,14.28);
\draw (20.,7.14)-- (27.65,7.14);
\draw (20.,7.14)-- (20.,14.28);
\draw (27.65,7.14)-- (27.65,22.87);
\draw (27.65,22.87)-- (39.89,22.87);
\draw (40,22.95)-- (40.,5.);
\draw (40.,5.)-- (54,5.);
\draw (6.91,14.8) node[anchor=north west] {$\bar{v}^n_{m-1}$};
\draw (30.63,23.6) node[anchor=north west] {$\check{v}_1$};
\draw (21.45,12) node[anchor=north west] {$\hat{v}$};
\draw (42.10,11) node[anchor=north west] {$\bar{v}^n_{m+1}$};
\begin{scriptsize}
\draw [fill=xdxdff] (20.,0.) circle (1.5pt);
\draw [fill=xdxdff] (40.,0.) circle (1.5pt);
\draw [fill=xdxdff] (27.65,0.) circle (1.5pt);
\end{scriptsize}
\end{tikzpicture}
\caption{$v$ component.}
\label{fig_ricostruzione_discontinuita_v}
\end{subfigure}
\caption{An example of a discontinuity reconstruction for the Riemann solver $\mathcal{RS}^\alpha_1$.}\label{fig_ricostruzione_discontinuita}
\end{figure}

We have tested our procedure with Matlab with several data.
\begin{remark}
In the definitions of $\mathcal{RS}^\alpha_1$ and $\mathcal{RS}^\alpha_2$, when the constraint is enforced, we have chosen arbitrarily to give respectively the values $\check{u}_1$ and $\check{u}_2$ to the solutions at $x=y(t)$. Therefore in all the figures that will follow, whenever the non-classical shock appears, the bus vertical position will be $(\check{\rho}_1,\check{v}_1)$ or $(\check{\rho}_2,\check{v}_2)$. In general the bus vertical position represents the value of the solution $(\rho,v)$ at the bus position.
\end{remark}
The next proposition ensures that the constants defined in (\ref{costanti_ricostruzione_discontinuita}) are equal when the initial datum of the Riemann problem is a non-classical shock. Moreover the non-classical shock is exactly captured.
\begin{prop}\label{posizione_delle_discontinuita_identica}
Fix $n\in \mathbb{N}$ and let us suppose that $y^n=x_{m-1/2}$. Let $d^{n,\rho}_m$ and $d^{n,z}_m$ be the two constants defined in (\ref{costanti_ricostruzione_discontinuita}). Consider the Riemann problem (\ref{ARZ_riemann_probelem_rho_y}) at the instant $t^n$ with initial datum
\begin{equation}\label{initial_non_classical_shock}
u(t^n,x)=\begin{cases}
\bar{u}^n_{m-1} & \text{if } x<x_{m-1/2},\\
\bar{u}^n_{m+1} & \text{if } x>x_{m-1/2}.
\end{cases}
\end{equation}
If $\bar{u}^n_{m-1}=\hat{u}$, $\bar{u}^n_{m+1}=\check{u}_1$ and there exists $\gamma\in[0,1]$ such that
\begin{equation}\label{Prop_posizione_della_discontinuita_ipotesi}
\bar{u}^n_m = \gamma \, \bar{u}^n_{m-1} + (1-\gamma) \, \bar{u}^n_{m+1},
\end{equation}
then
\begin{equation}
d^{n,\rho}_m=d^{n,z}_m= \gamma.
\end{equation}
Moreover, the reconstruction method cancels the numerical diffusion introduced by the averaging process (\ref{mean_process}); see Figure \ref{fig_nonclassical_shock_RS_1}.
\end{prop}
\begin{proof}
Since the condition (\ref{Prop_posizione_della_discontinuita_ipotesi}) holds, we have
$$\begin{pmatrix}
\bar{\rho}^n_m\\
\bar{z}^n_m
\end{pmatrix} = \begin{pmatrix}
\gamma \, \hat{\rho}+(1-\gamma ) \, \check{\rho}_1\\
\gamma \, \hat{z}+(1-\gamma) \, \check{z}_1
\end{pmatrix}.$$
Hence we find
\begin{equation*}
\begin{split}
& d^{n,\rho}_m=\dfrac{\bar{\rho}^n_m-\rho^n_{m,r}}{\rho^n_{m,l}-\rho^n_{m,r}}=\dfrac{\gamma \, \hat{\rho}+(1-\gamma) \, \check{\rho}_1-\check{\rho}_1}{\hat{\rho}-\check{\rho}_1}= \dfrac{\gamma(\hat{\rho}-\check{\rho}_1)}{\hat{\rho}-\check{\rho}_1}=\gamma \; \text{ and}\\
& d^{n,z}_m=\dfrac{\bar{z}^n_m-z^n_{m,r}}{z^n_{m,l}-z^n_{m,r}}=\dfrac{\gamma \, \hat{z}+(1-\gamma) \, \check{z}_1-\check{z}_1}{\hat{z}-\check{z}_1}= \dfrac{\gamma(\hat{z}-\check{z}_1)}{\hat{z}-\check{z}_1}=\gamma.
\end{split}
\end{equation*}
For the second part, we have only to prove that if the initial datum is the non-classical shock (\ref{initial_non_classical_shock}), for every $x \in [x_{m-1/2},x_{m+1/2})$ we have
$$u^n_\text{rec}(x) = \check{u}_1= u(t^n,x).$$
This equality holds if and only if $\gamma = 0$ and we observe that
\begin{equation*}
\begin{split}
&\gamma = 0 \Longleftrightarrow \bar{\rho}^n_m = \rho^n_{m,r} \Longleftrightarrow \bar{\rho}^n_m = \check{\rho}_1 \Longleftrightarrow \\
& \Longleftrightarrow \dfrac{1}{h}\int_{x_{m-1/2}}^{x_{m+1/2}}\rho(t^n,x) \, dx = \check{\rho}_1 \Longleftrightarrow \dfrac{1}{h}\int_{x_{m-1/2}}^{x_{m+1/2}} \check{\rho}_1\, dx = \check{\rho}_1.
\end{split}
\end{equation*}
The last condition is clearly true.
\end{proof}
\begin{remark}\label{non_classical_shock_for_RS_alfa_1}
Consider the pressure $p(\rho)=\rho$. An initial datum satisfying $(\rho^l,v^l)=(\hat{\rho},\hat{v})$ and $(\rho^r,v^r)=(\check{\rho}_1,\check{v}_1)$ can be obtained as follows.\\
The points $(\rho^l,v^l)$ and $(\rho^r,v^r)$ must satisfy the equations
$$\rho\,v =F_\alpha+\rho\,V_b \; \text{ and } \; v^l+\rho^l=v^r+\rho^r.$$
From the first equation we find
$$v^l=\dfrac{F_\alpha}{\rho}+V_b$$
and from the second we obtain
$$v^r=v^l+\rho^l-\rho^r.$$
Therefore $\rho^l$ and $\rho^r$ are the solutions to the equation
$$\rho\,(v^l+\rho^l-\rho)=F_\alpha+\rho\,V_b\Longleftrightarrow \rho^2 -\rho\,\left(\rho^l+\dfrac{F_\alpha}{\rho^l}\right)+F_\alpha=0.$$
Hence
$$\rho^r =\dfrac{F_\alpha}{\rho^l}$$
which implies
$$v^r=\rho^l+V_b.$$
Therefore, fixed a value $\rho^l$, we have $(\rho^l,v^l) =(\hat{\rho},\hat{v})$ and $(\rho^r,v^r)=(\check{\rho}_1,\check{v}_1)$ if
$$(\rho^l,v^l)=\left(\rho^l,\dfrac{F_\alpha}{\rho^l}+V_b\right) \; \text{ and }\; (\rho^r,v^r)=\left(\dfrac{F_\alpha}{\rho^l},\rho^l+V_b\right).$$
\end{remark}
\begin{figure}[hbtp]
\centering
\includegraphics[width=0.95\linewidth]{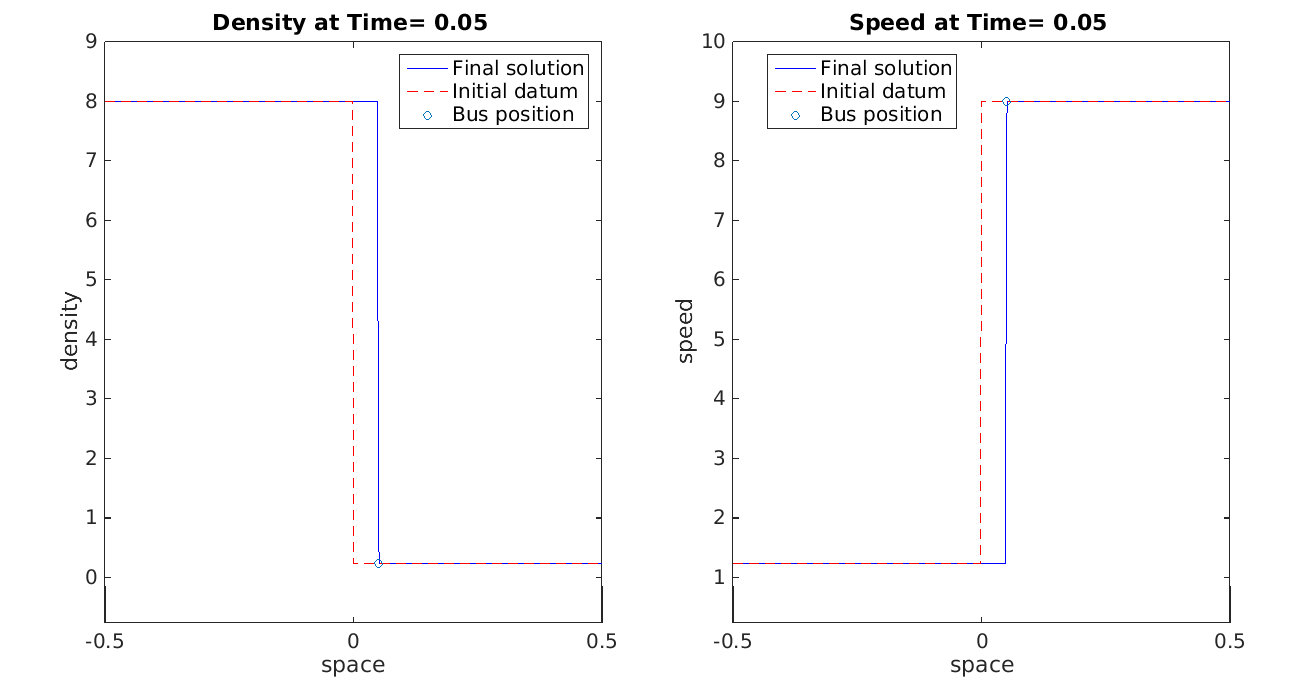}
\caption{Reconstruction of a non-classical shock with initial datum $(\rho^l,v^l)=(8,V_b+F_\alpha/8)$ and $(\rho^r,v^r)=(F_\alpha/8,8+V_b)$, obtained with the discontinuity reconstruction method for $\mathcal{RS}^\alpha_1$. By Remark \ref{non_classical_shock_for_RS_alfa_1}, the initial datum is a non-classical shock. The other parameters are $R_{\max}=15$, $V_b=1$, $y_0=0$ and $\alpha=0.25$. The pressure function is $p(\rho)=\rho$.}\label{fig_nonclassical_shock_RS_1}
\end{figure}
The described procedure (\ref{first_inequality}), (\ref{second_inequality}) to detect whether a non-classical shock appears, in some cases introduces undesirable oscillations. The following counterexample is one of these situations: the exact solution does not satisfy the constraint, while the numerical solution does (at some iterations); see Figure \ref{fig_counterexample_inequalities}.
\begin{exmp}\label{counterexample_inequalities}
Let us consider the Riemann problem (\ref{ARZ_riemann_probelem_rho_y}) with the constraint (\ref{vincolo}) for $x \in K:=[-1/2,1/2]$ and the following data:
\begin{itemize}
\item $\rho \to p(\rho) = \rho$ is the pressure function;
\item the constant for the CFL condition is $1/2$;
\item $\bar{V} = 3/2$ is the bus speed;
\item $y_0=0$ is the bus initial position;
\item $\alpha=0.4$ is the coefficient which gives the reduction rate in the road capacity caused by the bus;
\item $R = 15$ is the maximal density allowed on the road.
\end{itemize}
We take as initial datum for the Riemann problem the function
\begin{equation*}
(\rho_0,v_0)(x)=
\begin{cases}
(\rho^l,v^l)=(7,3) & \text{if } x\leq 0,\\
(\rho^r,v^r) =(6,4) & \text{if } x>0.
\end{cases}
\end{equation*}
First, we recall that
$$F_\alpha= \rho_\alpha^2\, p'(\rho_\alpha),$$
where $\rho_\alpha$ is the solution of the equation
$$p(\rho_\alpha)+\rho_\alpha\,p'(\rho_\alpha)=\omega_\alpha-\bar{V} \; \text{ and }\; \omega_\alpha=p(\alpha R).$$
Since $p'(\rho)=1$ and using the values of the data, we find
$$\omega_\alpha=6, \; \; \rho_\alpha = \dfrac{\omega_\alpha-\bar{V}}{2}= 2.25 \; \text{ and } \; F_\alpha = \rho_\alpha^2 = 5.0625.$$
Since $v^l+p(\rho^l)=v^r+p(\rho^r)$ and $\rho^l>\rho^r$, the standard Riemann solver links the two points $(\rho^l,v^l)$ and $(\rho^r,v^r)$ with a rarefaction. The right propagation speed of the rarefaction is negative, because
$$\lambda_1(\rho^r,v^r)=v^r-\rho^r\,p'(\rho^r)=4-6<0.$$
Hence the trace of $\mathcal{RS}((\rho^l,v^l),(\rho^r,v^r))$ in $\bar{V}$ is $(\rho^r,v^r)$.\\
Now, we note that
$$f_1(\mathcal{RS}(\rho^l,v^l),(\rho^r,v^r))(\bar{V}))=\rho^r v^r=24> 14.06=F_\alpha+\bar{V}\rho^r.$$
Therefore, the constraint is not satisfied by the exact classical solution and a non-classical shock appears in the solution given by $\mathcal{RS}^\alpha_1$.\\
Let us make explicitly one iteration of the algorithm described near the bus position. We will show that the numerical solution at time $t^1$ does not satisfy the first inequality
$$f_1(\bar{u}^1_m) > F_\alpha+ \bar{V} \bar{\rho}^1_m.$$
We divide the interval $K$ in 500 points, so that $h=0.002$ and we call $xh = \lbrace xh(i) \rbrace_{i=1}^{500}$ the mesh points. We take as initial datum for the simulation, the piecewise constant function
$$(\rho,v)(0,x)=\begin{cases}
(\rho^l,v^l) & \text{if } x\in[xh(i),xh(i+1)), \; xh(i)\leq 0 \; \text{ for } i = 1,...,500,\\
(\rho^r,v^r) & \text{otherwise},
\end{cases}$$
which is defined in the interval $[-0.5,0.5]$.
This choice is consistent with the fact that the value of the exact initial datum in $y_0=0$ is $(\rho^l,v^l)$. \\
We note that the value of $k$ that satisfies the CFL condition is
$$k=\dfrac{h}{2\times \lambda^0}= \dfrac{h}{8}\simeq 2.5 \times 10^{-4},$$
because
$$\lambda^0= \max\lbrace |\lambda_i(u^l)|,|\lambda_i(u^r)|\rbrace_{i=1,2}= \lambda_2(u^r) = 4.$$
Since $\bar{u}^0_m=u^l$, we have
$$f_1(\bar{u}^0_m)= 21>15.0625 = F_\alpha+\bar{V}\bar{\rho}^n_m.$$
Hence we have to check the second inequality. Since $\bar{u}^0_{m-1}=u^l$ and $\bar{u}^0_{m+1}=u^r$, reasoning as in the beginning of the example, we find
$$f_1(\mathcal{RS}(\bar{u}^0_{m-1},\bar{u}^0_{m+1})(\bar{V}))=f_1(u^r) = 24>14.0625 = F_\alpha+\bar{V}\bar{\rho}(\bar{u}^0_{m-1},\bar{u}^0_{m+1}).$$
Hence we apply the reconstruction procedure.\\
Solving the equation
$$\rho L_1(\rho,\rho^l,v^l)=F_\alpha+\bar{V} \rho,$$
we obtain
\begin{equation*}
\begin{split}
& (\hat{\rho},\hat{v})=((8.5+\sqrt{52})/2,(11.5-\sqrt{52})/2) \; \text{ and } \\
& (\check{\rho}_1,\check{v}_1) =((8.5-\sqrt{52})/2, (11.5+\sqrt{52})/2).
\end{split}
\end{equation*}
Moreover, $\hat{z}=\hat{\rho}(\hat{v}+p(\hat{\rho}))=\hat{\rho}(v^l+p(\rho^l))=10\times \hat{\rho}$ and $\check{z}_1 = \check{\rho}_1(\check{v}_1+p(\check{\rho}_1)) = 10\times \check{\rho}_1$, because $v^l+p(\rho^l)=v^l+\rho^l = 10$.\\
Since $\hat{\rho}>8.5>\rho^l$, the standard Riemann solver $\mathcal{RS}$ joins $\bar{u}^0_{m-1}=u^l$ with $\hat{u}$ with a shock propagating with speed
$$\lambda=\dfrac{\hat{\rho} \hat{v}-\rho^l v^l}{\hat{\rho}-\rho^l}\simeq -4.85.$$
This speed is negative, then
$$\mathcal{RS}(\bar{u}^0_{m-1},\hat{u})(\bar{V})= \hat{u}.$$
Therefore
$$F(\bar{u}^0_{m-1},\hat{u})=\begin{pmatrix}
\hat{\rho}\hat{v}\\
\hat{z}\hat{v}
\end{pmatrix}=
\begin{pmatrix}
(45.75+3\sqrt{52})/{4}\\
10\times \hat{\rho}\hat{v}
\end{pmatrix}.$$
The two values $d^{0,\rho}_m$ and $d^{0,z}_m$ are
\begin{equation*}
\begin{split}
& d^{0,\rho}_m=\dfrac{\bar{\rho}^0_m-\check{\rho}_1}{\hat{\rho}-\check{\rho}_1}= \dfrac{5.5+\sqrt{52}}{2\sqrt{52}} \; \text{ and }\\
& d^{0,z}_m=\dfrac{\bar{z}^0_m-\check{z}_1}{\hat{z}-\check{z}_1}= \dfrac{10\times (\bar{\rho}^0_m-\check{\rho}_1)}{10\times (\hat{\rho}-\check{\rho}_1)}= d^{0,\rho}_m.
\end{split}
\end{equation*}
Hence we have
$$\Delta t^\rho_m = \Delta t^z_m = h\dfrac{1-d^{0,\rho}_m}{\bar{V}}=h\dfrac{\sqrt{52}-5.5}{3\sqrt{52}}\simeq 1.58 \times 10^{-4}<k.$$
The reconstructed flux is
$$F_1(\bar{u}^0_m,\bar{u}^0_{m+1})=\dfrac{401.25\times \sqrt{52}-2028}{12\times \sqrt{52}}$$
and $F_2(\bar{u}^0_m,\bar{u}^0_{m+1}) = 10\times F_1(\bar{u}^0_m,\bar{u}^0_{m+1})$.\\
We are now ready to compute the solution at the new time step in the $m$-th cell. We find
\begin{equation*}
\begin{split}
\bar{\rho}^1_m & = \bar{\rho}^0_m-\dfrac{k}{h}\left( \dfrac{45.75+3\sqrt{52}}{4}- \dfrac{401.25\times \sqrt{52}-2028}{12\times \sqrt{52}}\right)= \dfrac{26+4.25\times \sqrt{52}}{\sqrt{52}}=\\
& = 4.25+\dfrac{\sqrt{52}}{2}=\hat{\rho}
\end{split}
\end{equation*}
and
$$\bar{v}^1_m = \dfrac{\bar{z}^1_m}{\bar{\rho}^1_m}-\bar{\rho}^1_m= \dfrac{11.5-\sqrt{52}}{2}=\hat{v},$$
because $\bar{z}^1_m =10 \times \bar{\rho}^1_m$.\\
Since $(\bar{\rho}^1_m,\bar{v}^1_m) = (\hat{\rho},\hat{v})$, at the second iteration the first inequality (\ref{first_inequality}) is not satisfied, because by definition
$$\hat{\rho}\hat{v}=F_\alpha+\bar{V}\hat{\rho}.$$
Even if we accept the equal, i.e.
$$ f_1(\bar{u}^n_m) \geq F_\alpha+ \bar{V} \bar{\rho}^n_m,$$
our procedure fails because Matlab makes a numerical error of order $10^{-15}$ for which the right term in (\ref{first_inequality}) results bigger then the left term.
\end{exmp}
In view of Example \ref{counterexample_inequalities}, we propose to remove the first condition (\ref{first_inequality}) and to keep only the inequality (\ref{second_inequality}) as necessary to start the reconstruction procedure.

\begin{figure} 
\centering
\begin{subfigure}[h]{\linewidth}
\includegraphics[width=\textwidth]{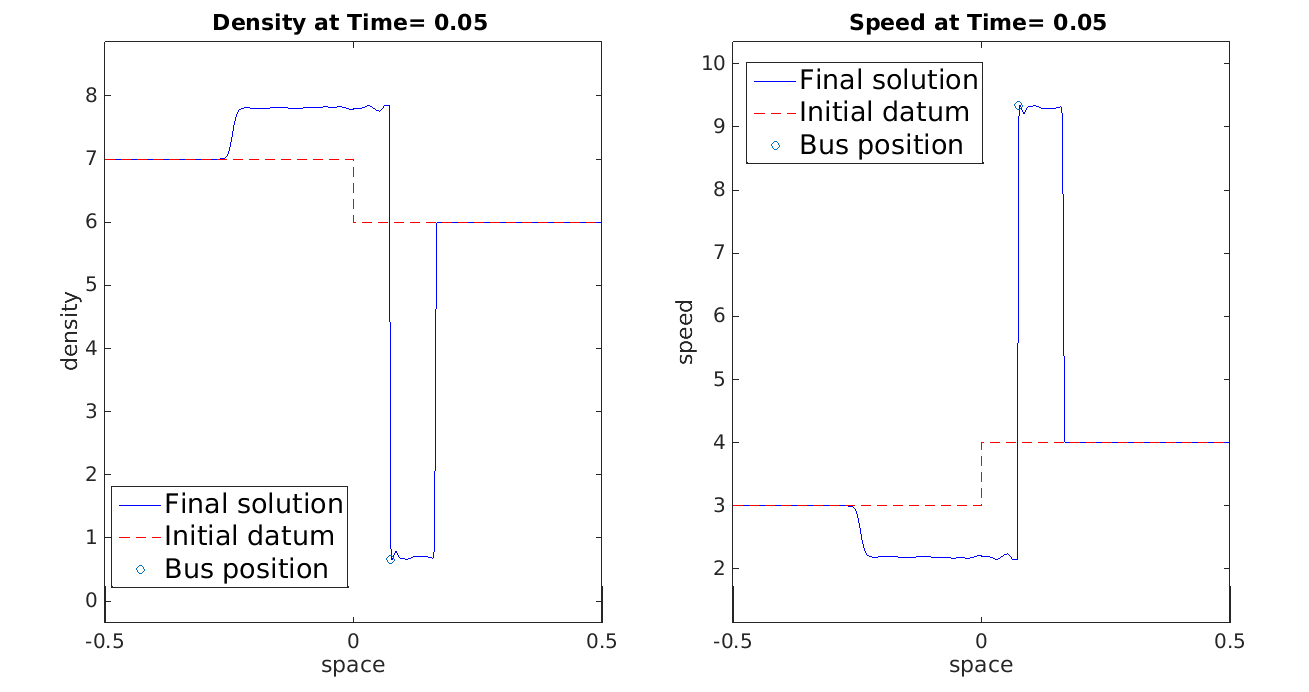}
\caption{The inequality (\ref{first_inequality}) has been required as necessary to the reconstruction procedure. Note that oscillations appear around the non-classical shock.}
\end{subfigure}
\\
\begin{subfigure}[h]{\linewidth}
\includegraphics[width=\textwidth]{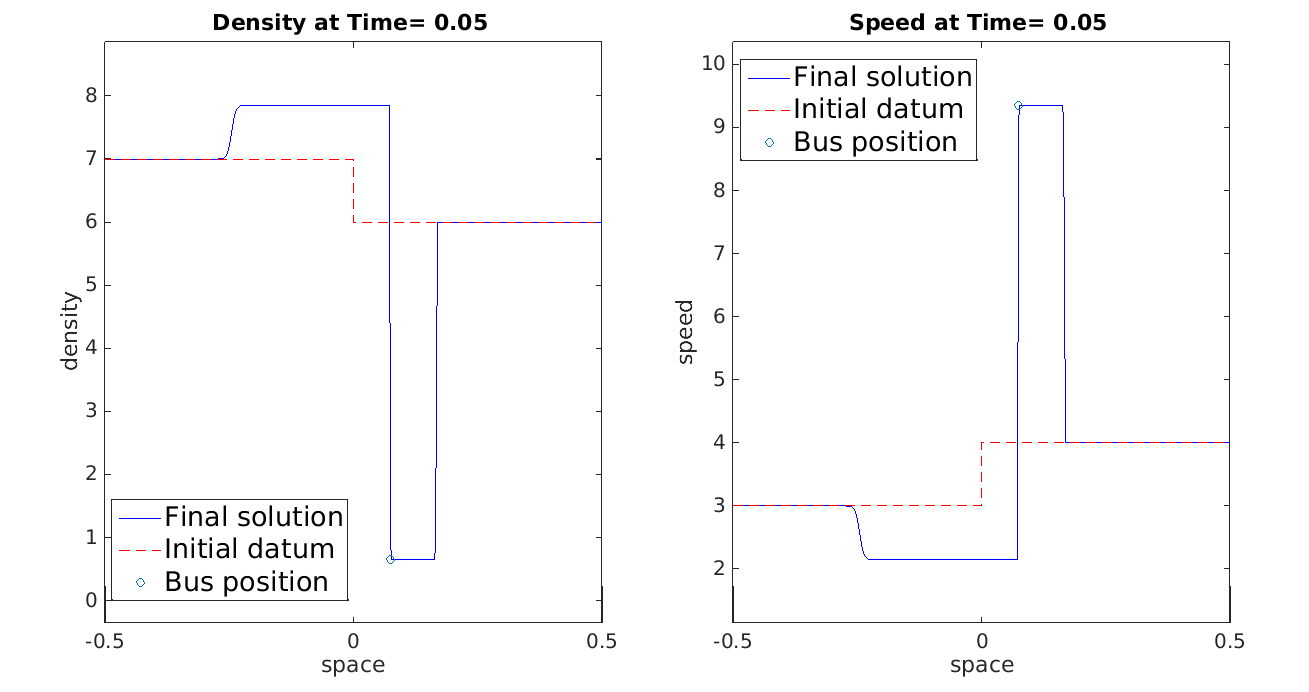}
\caption{The inequality (\ref{first_inequality}) has not been required as necessary to the reconstruction procedure.}
\end{subfigure}
\caption{Numerical solutions obtained with the data of Example \ref{counterexample_inequalities}. We can see the undesired oscillations produced in the case (a): the inequality (\ref{first_inequality}) is not satisfied at some iterations, so that the solution at these steps is obtained with the standard Godunov's scheme, which does not take into account the presence of the non-classical shock. In case (b) the oscillations disappears because the inequality (\ref{first_inequality}) has not been required as necessary to the reconstruction procedure.}\label{fig_counterexample_inequalities}
\end{figure}

\subsection{The bus and the vehicles do not influence each other}
If $\bar{V}\, \bar{\rho}(u^l,u^r)(\bar{V})<f_1(\mathcal{RS}(u^l,u^r)(\bar{V}))\leq F_\alpha+ \bar{V}\, \bar{\rho}(u^l,u^r)(\bar{V})$, then the bus and the vehicles do not influence each other. The corresponding numerical condition is
\begin{equation}\label{case_not_influence}
\bar{V}\bar{\rho}(\bar{u}^n_{m-1},\bar{u}^n_{m+1})(\bar{V})<f_1(\mathcal{RS}(\bar{u}^n_{m-1},\bar{u}^n_{m+1})(\bar{V})) \leq F_\alpha+\bar{V}\bar{\rho}(\bar{u}^n_{m-1},\bar{u}^n_{m+1})(\bar{V}).
\end{equation}
Hence, if condition (\ref{case_not_influence}) holds, the bus position at the time $t^n$ is $y^n=\bar{V}t^n$ and the solution at the new time step $t^{n+1}$ can be computed with the standard Godunov's method.

\subsection{The bus is influenced by the preceding vehicles}
Let us recall the model to describe the bus speed. Let $\lbrace x_i \rbrace _{i=1}^N$ be the bus stops and let $\delta$ be the space needed by the bus to stop starting from the maximal speed $V_b$. The profile of the bus velocity is given by a sufficiently regular function $V(x)$ such that
\begin{equation}\label{bus_speed_profile}
V(x)=\begin{cases}
V_b & \text{if } |x-x_i|>\delta \; \text{ for } \; i=1,...,N,\\
0   & \text{if } x=x_i\; \text{ for } \; i=1,...,N.
\end{cases}
\end{equation}
Now, let us suppose that the bus remains at each stop for a constant time $\tau\in \mathbb{R}^+$ and let $t_i=\inf\lbrace t\in \mathbb{R}^+\,:\, y(t)=x_i\rbrace $ for $i=1,...,N$ be the $i$-th stop instant. The bus speed without traffic is a function $\dot{y}_{F}:\mathbb{R}^+\to [0,V_b]$ defined by
\begin{equation}\label{bus_speed_free}
\dot{y}_F(t)=\begin{cases}
V(y(t)) & \text{if } t \notin [t_i,t_i+\tau) \;\text{ for every } \; i=1,...,N,\\
0 & \text{if } t\in[t_i,t_i+\tau) \;\text{ for every } \; i=1,...,N.
\end{cases}
\end{equation}
If we now introduce the traffic, the bus will travel with velocity $V(y(t))$ when there are no vehicles in front of him or when their speed $v(t,y(t)+)$ is higher than $V(y(t))$. Otherwise, it will adapt its velocity to the one of the traffic, namely
\begin{equation}\label{bus_speed}
\dot{y}(t)=\omega(y(t),v(t,y(t)+))=\begin{cases}
\dot{y}_F(t) & \text{if } \dot{y}_F(t)<v(t,y(t)+),\\
v(t,y(t)+) & \text{if } \dot{y}_F(t)\geq v(t,y(t)+).
\end{cases}
\end{equation}
Let us consider the situation of a bus far from the stops, i.e. $|y(t)-x_i|>\delta$ for every $i\in \lbrace 1,...,N\rbrace$. In this case the bus keeps the maximal speed allowed by the traffic, namely
\begin{equation}
\dot{y}(t)=\begin{cases}
V_b & \text{if } V_b \leq v(t,y(t)+),\\
v(t,y(t)+) & \text{otherwise}.
\end{cases}
\end{equation}
Fix $n\in\mathbb{N}$. Let $y^n=y(t^n)$ be the bus position at time $t^n$ and let us define the number $m\in \mathbb{Z}$ such that $y^n\in C_m= [x_{m-1/2},x_{m+1/2})$.\\
Let us suppose that
\begin{equation}\label{slow_traffic}
\bar{V}^n\, \bar{\rho}(\bar{u}^n_{m-1},\bar{u}^n_{m+1})(\bar{V}^n)\geq f_1(\mathcal{RS}(\bar{u}^n_{m-1},\bar{u}^n_{m+1})(\bar{V}^n)),
\end{equation}
where
\begin{equation}
\bar{V}^n=\dot{y}(t^n)=\begin{cases}
V_b & \text{if } V_b\leq \bar{v}^n_m,\\
\bar{v}^n_m & \text{if } V_b > \bar{v}^n_m.
\end{cases}
\end{equation}
The inequality (\ref{slow_traffic}) implies that
\begin{equation}
\bar{V}^n\geq \bar{v}(\bar{u}^n_{m-1},\bar{u}^n_{m+1})(\bar{V}^n),
\end{equation}
where we recall that $\bar{v}(\bar{u}^n_{m-1},\bar{u}^n_{m+1})$ is the $v$ component of the classical solution $$\mathcal{RS}(\bar{u}^n_{m-1},\bar{u}^n_{m+1})(\bar{V}^n).$$
Since the bus is travelling faster then the preceding vehicles, it has to adapt its speed to the traffic. It will keep this speed until the traffic will not change its velocity. This situation is described by an interaction between the bus trajectory and a wave coming from one of the local Riemann problems centred in $\lbrace x_{j+1/2} \rbrace_{j\in \mathbb{Z}}$. By the CFL condition, this can happen only for waves arisen in the Riemann problems centred in $x_{m-1/2}$ and $x_{m+1/2}$. Therefore we have to distinguish two cases:
\begin{enumerate}
\item[(i)] the bus trajectory interacts with a wave coming from the local Riemann problem centred in $x_{m+1/2}$;
\item[(ii)] the bus trajectory interacts with a wave coming from the local Riemann problem centred in $x_{m-1/2}$. 
\end{enumerate}
We adapt the algorithms introduced in \cite{bretti_piccoli, chalons_delle_monache_goatin}.

\subsubsection{Case (i)}
Let us consider the Riemann problem
\begin{equation}\label{Riemann_problem_centered_x_m+1/2}
\begin{cases}
\begin{cases}
\partial_t \rho + \partial_x(\rho v)=0,\\
\partial_t z + \partial_x(z v) = 0,
\end{cases}\\
(\rho,z)(t^n,x)=\begin{cases}
\bar{u}^n_{m} & \text{if } x\leq x_{m+1/2},\\
\bar{u}^n_{m+1} & \text{if } x> x_{m+1/2}.
\end{cases}
\end{cases}
\end{equation}
Let $u^\text{int}_{m+1/2}$ be the intermediate state of the classical solution $\mathcal{RS}(\bar{u}^n_m,\bar{u}^n_{m+1})(\lambda)$ and $(\rho^\text{int}_{m+1/2},v^\text{int}_{m+1/2})$ the corresponding point in the $(\rho,v)$ plane.\\
In the general case we have
$$\bar{v}^n_{m+1}+p(\bar{\rho}^n_{m+1})\neq \bar{v}^n_m +p(\bar{\rho}^n_m).$$
Hence we have to consider the intermediate state $u^\text{int}_{m+1/2}$ which is connected to $\bar{u}^n_m$ by a contact discontinuity. However, since the speed of the bus $\bar{V}^n=\min\lbrace \bar{v}^n_m,V_b\rbrace$ depends only on the speed of the traffic in front of him and ${v}^\text{int}_{m+1/2}=\bar{v}^n_m$, this contact discontinuity does not influence the bus trajectory.\\
Therefore we can suppose that $\bar{v}^n_{m+1}+p(\bar{\rho}^n_{m+1})=\bar{v}^n_m+p(\bar{\rho}^n_m)$. The general case is obtained substituting $u^\text{int}_{m+1/2}$ to $\bar{u}^n_m$.\\
\begin{enumerate}
\item Let us consider a shock centred in $(t^n,{x_{m+1/2}})$.\\
This case happens whenever $\bar{\rho}^n_m<\bar{\rho}^n_{m+1}$. The propagation speed $\lambda_{m+1/2}$ of the shock is given by the Rankine-Hugoniot condition:
$$\lambda_{m+1/2}=\dfrac{\bar{\rho}^n_{m+1} \, \bar{v}^n_{m+1}-\bar{\rho}^n_m\bar{v}^n_m}{\bar{\rho}^n_{m+1}-\bar{\rho}^n_m}.$$
\begin{remark}
Since $\bar{v}^n_m+p(\bar{\rho}^n_m)=\bar{v}^n_{m+1}+p(\bar{\rho}^n_{m+1})$, the condition $\bar{\rho}^n_m<\bar{\rho}^n_{m+1}$ is equivalent to
$$\bar{v}^n_m>\bar{v}^n_{m+1}.$$
Therefore after the interaction with the shock the bus travels slower than before. 
\end{remark}
Let $(t^*,x^*)$ be the interaction point between the shock and the bus trajectory. Solving in $t^*$ the equation
$$y^n+(t^*-t^n)\bar{V}^n=x_{m+1/2}+\lambda_{m+1/2}(t^*-t^n),$$
we find
\begin{equation}
t^*=\dfrac{x_{m+1/2}-y^n}{\bar{V}^n-\lambda_{m+1/2}}+t^n.
\end{equation}
If $t^*\geq k$, then no interaction between the car and the shock wave occurs within the interval $[t^n,t^{n+1})$. Otherwise we have
$$x^* = y^n+\bar{V}^n t^*$$
and the new speed of the bus is
$$\bar{V}^n_\text{new}=\min(\bar{v}^n_{m+1},V_b).$$
If no other interactions with waves centred in $x_{m-1/2}$ or in $x_{m+3/2}$ happen, the new position of the bus at time $t^{n+1}$ is
\begin{equation}
y^{n+1}=x^*+\bar{V}^n_\text{new}(k-t^*).
\end{equation}
\begin{remark}
If the shock has a positive propagation speed, then the bus crosses the cell $C_m$ before the interaction.
\end{remark}
\item Let us consider the case $\bar{\rho}^n_m\geq \bar{\rho}^n_{m+1}$ in which a rarefaction wave centred in $(t^n,x_{m+1/2})$ joins the states $\bar{u}^n_m$ and $u^n_{m+1}$ and it can interact with the bus trajectory.
Let $(\rho^\sigma,v^\sigma)$ for $\sigma\in [0,1]$ be a point of the rarefaction, i.e.
$$v^\sigma=\bar{v}^n_m+p(\bar{\rho}^n_m)-p(\rho^\sigma) \; \text{ and } \; \bar{\rho}^n_{m+1}\leq \rho^\sigma \leq \bar{\rho}^n_m.$$
The propagation speed of the rarefaction varies in the interval
$$[\lambda_1(\bar{\rho}^n_m,\bar{v}^n_m),\lambda_1(\bar{\rho}^n_{m+1},\bar{v}^n_{m+1})].$$
\begin{remark}
Since $v^\sigma+p(\rho^\sigma)=\bar{v}^n_m+p(\bar{\rho}^n_m)$, the condition $\rho^\sigma \leq \bar{\rho}^n_m$ is equivalent to $v^\sigma\geq \bar{v}^n_m$. Hence the speed of the traffic increases during the rarefaction.
\end{remark}
Let
$$R^\sigma := \lbrace (t,x) \in \mathbb{R}^+\times \mathbb{R}: x -x_{m+1/2} = \lambda_1(\rho^\sigma,v^\sigma)(t-t^n) \rbrace$$
be the line where the rarefaction centred in $(t^n,x_{m+1/2})$ and passing through $(\bar{\rho}^n_m,\bar{v}^n_m)$ takes the value $(\rho^\sigma,v^\sigma)$.\\
Let us call
$$\xi(t,x)= \dfrac{x-x_{m+1/2}}{t-t^n}.$$
Hence $(t,x)\in R^\sigma$ if and only if
$$\xi(t,x)= \lambda_1(\rho^\sigma,v^\sigma).$$
During the travel the bus takes the speed of the vehicles in front of him until their speed is lower then the maximal velocity of the bus, provided that an interaction between the wave and the bus trajectory occurs. This happens for all the points $(\rho^\sigma,v^\sigma)$ such that
\begin{equation}\label{necessary_condition_interaction_case_i}
\dot{y}(t) > \lambda_1(\rho^\sigma,v^\sigma),
\end{equation}
where $\dot{y}$ is the bus speed. Therefore
\begin{equation}\label{rarefaction_from_right_bus_speed}
\dot{y}(t)=\begin{cases}
\bar{V}^n= \min(V_b,\bar{v}^n_m) & \text{if } \xi(t,y(t)) \leq \lambda_1(\bar{\rho}^n_m,\bar{v}^n_m),\\
\min(v^\sigma,V_b) & \text{if } \lambda_1(\bar{\rho}^n_m,\bar{v}^n_m) < \xi(t,y(t)) < \lambda_1(\bar{\rho}^n_{m+1},\bar{v}^n_{m+1}),\\
\min(\bar{v}^n_{m+1},V_b) & \text{if } \xi(t,y(t)) \geq \lambda_1(\bar{\rho}^n_{m+1},\bar{v}^n_{m+1}).
\end{cases}
\end{equation}
\begin{obs}
We don't keep in account the necessary condition (\ref{necessary_condition_interaction_case_i}) for the interaction, because for every $\sigma \in [0,1]$ we have
\begin{itemize}
\item if $\min(V_b,v^\sigma)=v^\sigma$, then $V_b >\lambda_1(\rho^\sigma,v^\sigma)$, because $v^\sigma > \lambda_1(\rho^\sigma,v^\sigma)$. Therefore the condition (\ref{necessary_condition_interaction_case_i}) is satisfied and the interaction occurs;
\item if $\min(V_b,v^\sigma)=V_b$ and $V_b> \lambda_1(\rho^\sigma,v^\sigma)$, then the interaction occurs. In this case the law (\ref{rarefaction_from_right_bus_speed}) gives
$$\dot{y}(t) = V_b,$$
therefore the bus does not exceed its maximal speed;
\item if $V_b\leq \lambda_1(\rho^\sigma,v^\sigma)$, then the interaction does not occur and the bus travels with its maximal speed $V_b$, because the inequalities $v^\sigma>\lambda_1(\rho^\sigma,v^\sigma)$ and $ \lambda_1(\rho^\sigma,v^\sigma)\geq V_b$ imply $\min(V_b,v^\sigma)=V_b$.
\end{itemize}
Hence the bus speed is always well defined.
\end{obs}
Let $(t^*,x^*)$ be the first point of interaction between the bus and the rarefaction. Solving with respect to $t^*$ the equation
$$y^n+\bar{V}^n (t^*-t^n)=x_{m+1/2}+\lambda_1(\bar{\rho}^n_m,\bar{v}^n_m)(t^*-t^n),$$
we find
\begin{equation}
t^* = \dfrac{x_{m+1/2}-y^n}{\bar{V}^n-\lambda_1(\bar{\rho}^n_m,\bar{v}^n_m)}+t^n.
\end{equation}
Therefore, we also have
\begin{equation}
x^*=y^n+\bar{V}^n t^*.
\end{equation}
Fix $(t,x)\in R^\sigma$. Let us call $\xi :=\xi(t,x)$ and $\bar{\omega}^n_m=\bar{v}^n_m+p(\bar{\rho}^n_m)$. Then we have
\begin{equation*}
\begin{split}
& \begin{cases}
\lambda_1(\rho^\sigma,v^\sigma)= \xi\\
v^\sigma = \bar{\omega}^n_m -p(\rho^\sigma)
\end{cases} \Longrightarrow 
\begin{cases}
\bar{\omega}^n_m - p(\rho^\sigma) - \rho^\sigma p'(\rho^\sigma) = \xi\\
v^\sigma = \bar{\omega}^n_m -p(\rho^\sigma)
\end{cases} \Longrightarrow \\
& \begin{cases}
\bar{\omega}^n_m -\dfrac{d}{d\rho}(\rho p(\rho)) \Big|_{\rho=\rho^\sigma} = \xi\\
v^\sigma = \bar{\omega}^n_m -p(\rho^\sigma).
\end{cases}
\end{split}
\end{equation*}
Let us define the function
$$\rho \to \varphi(\rho) = \dfrac{d}{d\rho}(\rho p(\rho))$$
which admits an inverse function because it is strictly increasing by the strict convexity of $\rho \to \rho p(\rho)$.\\ 
The first equation of the previous system becomes
\begin{equation}\label{rarefaction_density}
\varphi(\rho^\sigma)=\bar{\omega}^n_m-\xi \Longrightarrow \rho^\sigma = \varphi^{-1}(\bar{\omega}^n_m-\xi).
\end{equation}
Using the equation (\ref{rarefaction_density}) in the second equation of the system we find
\begin{equation}\label{rarefaction_speed}
v^\sigma=\bar{\omega}^n_m-p(\varphi^{-1}(\bar{\omega}^n_m-\xi)).
\end{equation}
Our aim is to determine the bus trajectory along the rarefaction. Hence we give an explicit expression for the pressure function $p$. An usual choice (see \cite{Aw-Rascle}) is
\begin{equation}\label{explicit_pressure}
p(\rho) = \rho^\gamma \; \text{ for } \gamma\geq 1.
\end{equation}
For this function we have
$$ \varphi(\rho) = (\gamma+1)\rho^\gamma = (\gamma+1)\, p(\rho) \; \text{ and } \; \varphi^{-1}(\tau) = \sqrt[\gamma]{\dfrac{\tau}{\gamma+1}}.$$
Therefore
$$\rho^\sigma = \sqrt[\gamma]{\dfrac{\bar{\omega}^n_m-\xi}{\gamma+1}}\; \text{ and } \; v^\sigma= \dfrac{\gamma\,  \bar{\omega}^n_m+\xi}{\gamma+1}.$$
We can use the expression of $v^\sigma$ in the law (\ref{rarefaction_from_right_bus_speed}) to compute the speed of the bus when it is in $(t,y(t))\in R^\sigma$. In the case $v^\sigma<V_b$, we find the Cauchy problem
\begin{equation}\label{rarefaction_from_right_Cauchy_problem_for_bus_trajectory}
\begin{cases}
(\gamma+1)\, \dot{y}(t) = \gamma \, \bar{\omega}^n_m + \dfrac{y(t)-x_{m+1/2}}{t-t^n},\\
y(t^*)=x^*.
\end{cases}
\end{equation}
We are ready to compute the bus trajectory along the rarefaction.
\begin{prop}\label{prop_bus_trajectory}
If the bus interacts with a rarefaction wave centred in $(t^n,x_{m+1/2})$, then its trajectory is
\begin{equation}\label{rarefaction_from_right_bus_trajectory}
y(t)=x_{m+1/2}+\bar{\omega}^n_m(t-t^n)+C^*(t-t^n)^\frac{1}{\gamma+1},
\end{equation}
where
$$C^* = \dfrac{x^*-x_{m+1/2}-\bar{\omega}^n_m(t^*-t^n)}{(t^*-t^n)^\frac{1}{\gamma+1}}$$
which depends on the first point of interaction between the bus and the rarefaction wave, i.e.
\begin{equation}
t^* = \dfrac{x_{m+1/2}-y^n}{\bar{V}^n-\lambda_1(\bar{\rho}^n_m,\bar{v}^n_m)}+t^n \; \text{ and } \; x^*=y^n+\bar{V}^n t^*.
\end{equation}
\end{prop}
\begin{proof}
We have to solve the Cauchy problem (\ref{rarefaction_from_right_Cauchy_problem_for_bus_trajectory}). The equation
$$\dot{y}(t) -\dfrac{1}{\gamma+1}\dfrac{y(t)}{t-t^n}=\dfrac{\gamma\,\bar{\omega}^n_m}{\gamma+1}-\dfrac{1}{\gamma+1}\dfrac{x_{m+1/2}}{t-t^n}$$
is linear, then to solve it, we introduce the integration factor
\begin{equation*}
\begin{split}
\exp\left(-\dfrac{1}{\gamma+1}\int \dfrac{1}{t-t^n}\, dt\right) & = \exp\left(-\dfrac{1}{\gamma+1}\log(t-t^n)\right)=\\
& = \exp\left(\log\left((t-t^n)^{-\frac{1}{\gamma+1}}\right)\right)=\\
& = \dfrac{1}{(t-t^n)^{\frac{1}{\gamma+1}}}.
\end{split}
\end{equation*}
Multiplying this factor on both sides of the equation, we find
\begin{equation*}
\dfrac{\dot{y}(t)}{(t-t^n)^{\frac{1}{\gamma+1}}}-\dfrac{y(t)}{(\gamma+1)(t-t^n)^{1+\frac{1}{\gamma+1}}} = \dfrac{\gamma\,\bar{\omega}^n_m}{(\gamma+1)(t-t^n)^{\frac{1}{\gamma+1}}}-\dfrac{x_{m+1/2}}{(\gamma+1)(t-t^n)^{1+\frac{1}{\gamma+1}}}.
\end{equation*}
The left term is equal to
$$\dfrac{d}{dt}\left(\dfrac{y(t)}{(t-t^n)^{\frac{1}{\gamma+1}}}\right),$$
hence integrating both sides we find
\begin{equation*}
\begin{split}
\dfrac{y(t)}{(t-t^n)^{\frac{1}{\gamma+1}}} & = \dfrac{\gamma \, \bar{\omega}^n_m}{\gamma+1}\int (t-t^n)^{-\frac{1}{\gamma+1}} \, dt-\dfrac{x_{m+1/2}}{\gamma+1}\int(t-t^n)^{-\left(1+\frac{1}{\gamma+1}\right)}\, dt=\\
& = \bar{\omega}^n_m(t-t^n)^{\frac{\gamma}{\gamma+1}}+x_{m+1/2}(t-t^n)^{-\frac{1}{\gamma+1}}+C,
\end{split}
\end{equation*}
where $C$ is a constant.\\
Multiplying $(t-t^n)^{\frac{1}{\gamma+1}}$ on both sides, we find
$$y(t) = \bar{\omega}^n_m(t-t^n)+x_{m+1/2}+C(t-t^n)^{\frac{1}{\gamma+1}}.$$
Imposing the initial condition $y(t^*)=x^*$, we find the value of $C$ and the trajectory (\ref{rarefaction_from_right_bus_trajectory}).
\end{proof}
Deriving the equation (\ref{rarefaction_from_right_bus_trajectory}), we find
\begin{equation}\label{rarefaction_fom_right_explicit_bus_speed}
\dot{y}(t)=\bar{\omega}^n_m+\dfrac{C^*}{\gamma+1}(t-t^n)^{-\frac{\gamma}{\gamma+1}}.
\end{equation}
Let $(t^{**},x^{**})$ be the last point of interaction between the bus trajectory and the rarefaction wave. By the law (\ref{rarefaction_from_right_bus_speed}), the bus will take the speed $v^\sigma\in[\bar{v}^n_m,\bar{v}^n_{m+1}]$ of the vehicles in front of him, until this speed is lower then its maximal speed $V_b$. Hence, if we call
$$\bar{V}^n_\text{rar}:=\max\lbrace \min(V_b, v^\sigma): v^\sigma \in[\bar{v}^n_{m},\bar{v}^n_{m+1}] \rbrace= \min \lbrace V_b,\bar{v}^n_{m+1} \rbrace,$$ 
the point $(t^{**},x^{**})$ satisfies the equation
$$\dot{y}(t^{**}) = \bar{V}^n_\text{rar}.$$
Substituting in this equation the expression (\ref{rarefaction_fom_right_explicit_bus_speed}), we find
$$t^{**} = \left[(\gamma+1)\dfrac{\bar{V}^n_\text{rar}-\bar{\omega}^n_m}{C^*}\right]^{-\frac{\gamma+1}{\gamma}}$$
and therefore
$$x^{**} = y(t^{**}),$$
where $y(t)$ is given by (\ref{rarefaction_from_right_bus_trajectory}).
After the interaction the bus position is one of the following, provided that no more interactions with other waves happen.\\
If $t^*\geq k$, then no interaction between the car and the rarefaction wave occurs in the interval $[t^n,t^{n+1})$ and the bus position at time $t^{n+1}$ is
$$y^{n+1} = y^n+\bar{V}^n k.$$
If $t^*<k$, then we have to consider two cases: whether $t^{**} \geq k$ the bus position at time $t^{n+1}$ is $$y^{n+1} = y(t^{n+1}),$$
where $y(t)$ is the trajectory (\ref{rarefaction_from_right_bus_trajectory}); otherwise
$$y^{n+1} = x^{**} + (k-t^{**})\bar{V}^n_\text{rar}.$$
\end{enumerate}
In both cases of an interaction between the bus trajectory and a shock or a rarefaction wave centred in $(t^n,x^{m+1/2})$, the bus can cross to the following cell before the instant $t^{n+1}$. The new cell of the bus is $C_{m+1}=[x_{m+1/2},x_{m+3/2})$.\\
A wave centred in $x_{m+3/2}$ could interact with the bus trajectory before time $t^{n+1}$. The next proposition states that we can avoid this situation choosing a sufficiently strong CFL condition.

\begin{prop}
Let  $\lambda^n=\max\lbrace |\lambda_i(\bar{u}^n_j)|\, : \, i=1,2 \; \text{ and } \; j\in \mathbb{Z}\rbrace$ be the maximum of the eigenvalues $\lbrace\lambda_i\rbrace_{i=1,2}$ of the Jacobian matrix $Df$ of the flux function $f$.\\
No interactions occur between the bus trajectory and a wave centred in $(t^n,x_{m+{3/2}})$, provided that the following CFL condition holds:
\begin{equation}\label{strong_CFL_condiftion}
\left| \dfrac{k}{h}\lambda^n\right|\leq \dfrac{1}{2}.
\end{equation}
\end{prop}
\begin{proof}
Let us show that after an interaction between the bus trajectory and a shock or a rarefaction wave centred in $(t^n,x_{m+1/2})$, the bus speed $\bar{V}^n_{new}$ at the right edge $x_{m+1/2}$ of the $m$-th cell is almost equal to $\lambda_2(\bar{\rho}^n_m,\bar{v}^n_m)$ or $\lambda_2(\bar{\rho}^n_{m+1},\bar{v}^n_{m+1})$.\\ 
First, let us consider the case of a shock. When the bus reach $x_{m+1/2}$ the bus travels with speed
$$\bar{V}^n_\text{new}=\bar{v}^n_{m+1}=\lambda_2(\bar{\rho}^n_{m+1},\bar{v}^n_{m+1}) \; \text{ or } \; \bar{V}^n=\min \lbrace \bar{v}^n_m,V_b\rbrace \leq \lambda_2(\bar{\rho}^n_m,\bar{v}^n_m).$$
If a rarefaction meets the bus trajectory, the bus speed after the interaction is
$$\bar{V}^n_\text{new}=\bar{V}^n_\text{rar}:=\min ( V_b,v^\sigma),$$
where $v^\sigma$ belongs to the interval $[\bar{v}^n_m,\bar{v}^n_{m+1}]$. Since during the rarefaction the speed of the traffic increases, the bus speed at $x_{m+1/2}$ is at most $\min (V_b,\bar{v}^n_{m+1})$.\\
Whether $\bar{V}^n_\text{rar} = v^\sigma$, we have
$$\bar{V}^n_\text{rar}\leq \bar{v}^n_{m+1}=\lambda_2(\bar{\rho}^n_{m+1},\bar{v}^n_{m+1}),$$
since $v^\sigma\leq \bar{v}^n_{m+1}$. Otherwise $\bar{V}^n_\text{rar} = V_b$, but this happens when $V_b < v^\sigma$. Hence again we find
$$\bar{V}^n_\text{rar}< \lambda_2(\bar{\rho}^n_{m+1},\bar{v}^n_{m+1}).$$
Since $\lambda^n \geq \lambda_2(\bar{\rho}^n_{m+1},\bar{v}^n_{m+1})$, if the condition (\ref{strong_CFL_condiftion}) holds, we have
$$ k\bar{V}^n_\text{new}\leq  k \lambda_2(\bar{\rho}^n_{m+1},\bar{v}^n_{m+1}) \leq k | \lambda^n| \leq \dfrac{h}{2}.$$
Hence before the instant $t^{n+1}$ the bus could have covered at most half the length of a cell $C_{m+1/2}$.\\
By Proposition \ref{rarefaction_wave_properties}, the Lax-entropy condition (\ref{Lax_entropy_condition}) and the condition (\ref{strong_CFL_condiftion}), the same holds for every wave centred in $(t^n,x_{m+3/2})$. Then we have the thesis.
\end{proof}
We will refer to the condition (\ref{strong_CFL_condiftion}) as ``strong CFL condition''.

\subsubsection{Case (ii)}
Let us consider the Riemann problem
\begin{equation}\label{Riemann_problem_centered_x_m-1/2}
\begin{cases}
\begin{cases}
\partial_t \rho + \partial_x(\rho v)=0,\\
\partial_t z + \partial_x(z v) = 0,
\end{cases}\\
(\rho,z)(t^n,x)=\begin{cases}
\bar{u}^n_{m-1} & \text{if } x\leq x_{m-1/2},\\
\bar{u}^n_{m} & \text{if } x> x_{m-1/2}.
\end{cases}
\end{cases}
\end{equation}
Let $u^\text{int}_{m-1/2}$ be the intermediate state of the classical solution $\mathcal{RS}(\bar{u}^n_{m-1},\bar{u}^n_{m})(\lambda)$ and let $(\rho^\text{int}_{m-1/2},v^\text{int}_{m-1/2})$ be the corresponding point in the $(\rho,v)$ plane.\\
Since the bus speed changes according to the variations of the traffic velocity, we are not interested in the case of an interaction between the bus trajectory and a contact discontinuity, because the two states connected by the contact discontinuity have the same speed. Hence we are interested only in rarefaction and shock waves.\\
Let us suppose that
$$\bar{v}^n_m+p(\bar{\rho}^n_m)=\bar{v}^n_{m-1}+p(\bar{\rho}^n_{m-1}).$$
The general case (in which $\bar{v}^n_m+p(\bar{\rho}^n_m)\neq \bar{v}^n_{m-1}+p(\bar{\rho}^n_{m-1})$) is obtained simply replacing $u^\text{int}_{m-1/2}$ to $\bar{u}^n_{m}$.
A solution to the system (\ref{Riemann_problem_centered_x_m-1/2}) can interact with the trajectory of the bus only if its propagation speed is positive, because $y^n\geq x_{m-1/2}$.
\begin{enumerate}
\item Let us suppose that $\bar{\rho}^n_{m-1}< \bar{\rho}^n_m$, so that the wave coming from $x_{m-1/2}$ is a shock with propagation speed $\lambda_{m-1/2}$ given by the Rankine-Hugoniot condition, i.e.
\begin{equation}
\lambda_{m-1/2} = \dfrac{\bar{\rho}^n_m \bar{v}^n_m-\bar{\rho}^n_{m-1}\bar{v}^n_{m-1}}{\bar{\rho}^n_m-\bar{\rho}^n_{m-1}}.
\end{equation}
\begin{remark}
Since $\bar{v}^n_m+p(\bar{\rho}^n_m)=\bar{v}^n_{m-1}+p(\bar{\rho}^n_{m-1})$, the condition $\bar{\rho}^n_{m-1}< \bar{\rho}^n_m$ is equivalent to $\bar{v}^n_{m-1}>\bar{v}^n_m$. Hence after the interaction the bus speed has increased.
\end{remark}
Let us suppose that the bus speed $\bar{V}^n=\min(V_b,\bar{v}^n_m)$ at time $t^n$ is equal to $\bar{v}^n_m$. This means that $\bar{v}^n_m \leq V_b$. Since $y^n\geq x_{m+1/2}$, an interaction with the shock can happen if and only if
$$\bar{v}^n_m \leq \lambda_{m-1/2}.$$
This is absurd, indeed
\begin{equation*}
\begin{split}
& \bar{v}^n_m \leq \lambda_{m-1/2} \Longleftrightarrow \bar{v}^n_m \leq \dfrac{\bar{v}^n_m \bar{\rho}^n_m-\bar{v}^n_{m-1}\bar{\rho}^n_{m-1}}{\bar{\rho}^n_m-\bar{\rho}^n_{m-1}} \Longleftrightarrow \\
& \bar{\rho}^n_m\bar{v}^n_m-\bar{\rho}^n_{m-1}\bar{v}^n_{m}\leq \bar{v}^n_m \bar{\rho}^n_m-\bar{v}^n_{m-1}\bar{\rho}^n_{m-1} \Longleftrightarrow \bar{\rho}^n_{m-1}\bar{v}^n_m\geq \bar{\rho}^n_{m-1}\bar{v}^n_{m-1} \Longleftrightarrow \\ 
& \bar{v}^n_m\geq \bar{v}^n_{m-1}
\end{split}
\end{equation*}
and this is a contradiction with $\bar{\rho}^n_{m-1}< \bar{\rho}^n_m$.\\
On the other hand the case $\bar{V}^n = V_b$ happens when $V_b \leq \bar{v}^n_m$. Since $\bar{v}^n_m<\bar{v}^n_{m-1}$, we have
$$\bar{V}^n_\text{new} = V_b = \min(V_b,\bar{v}^n_{m-1}).$$
Hence in this case the bus speed before and after the interaction does not change and its value is $V_b$.\\
Therefore, provided that no other interactions with waves coming from $x_{m+1/2}$ happen, the bus position at time $t^{n+1}$ is
$$y^{n+1} = y^n+k \bar{V}^n.$$

\item Let us suppose that $\bar{\rho}^n_m \leq \bar{\rho}^n_{m-1}$. In this case $\bar{u}^n_{m-1}$ and $\bar{u}^n_m$ are connected by a rarefaction wave centred in $x_{m-1/2}$.\\
Let $(\rho^\sigma,v^\sigma)$ be a point of the rarefaction, i.e.
$$v^\sigma \in [\bar{v}^n_{m-1},\bar{v}^n_{m}],\; \; \rho^\sigma \in [\bar{\rho}^n_{m},\bar{\rho}^n_{m-1}]\; \text{ and } \; \bar{v}^n_{m-1}+p(\bar{\rho}^n_{m-1})=v^\sigma+p(\rho^\sigma).$$
\begin{remark}
Since $\bar{v}^n_{m-1}\leq v^\sigma \leq \bar{v}^n_{m}$, the traffic speed increases during the rarefaction.
\end{remark}
\begin{prop}
Let us suppose that
$$\bar{V}^n=\min \lbrace \bar{v}^n_m,V_b \rbrace < \lambda_1(\bar{\rho}^n_m,\bar{v}^n_m),$$
so that an interaction occurs between the bus and the rarefaction wave centred in $x_{m-1/2}$.\\
Then the bus speed before the interaction is $\bar{V}^n = V_b$ and the bus keeps its maximal speed in all the points of interaction.
\end{prop}
\begin{proof}
We have $\bar{V}^n=\bar{v}^n_m$ if and only if $\bar{v}^n_m \leq V_b$. Since we are supposing that $\bar{V}^n<\lambda_1(\bar{\rho}^n_m,\bar{v}^n_m)$, this case cannot happen or we would have
$$\bar{v}^n_m<\lambda_1(\bar{\rho}^n_m,\bar{v}^n_m)$$
and this is absurd.\\
Hence the only possible case is
$$\bar{V}^n=V_b$$
which is equivalent to $V_b\leq \bar{v}^n_m$.\\
Let us call
$$R^\sigma = \lbrace (t,x)\in\mathbb{R}^+ \times \mathbb{R} : x-x_{m-1/2} = \lambda(\rho^\sigma,v^\sigma)(t-t^n)\rbrace$$
and
$$\xi(t,x) = \dfrac{x-x_{m-1/2}}{t-t^n}.$$
We have $\xi(t,x) = \lambda_1(\rho^\sigma,v^\sigma)$ if and only if $(t,x)\in R^\sigma$.\\
Since $y^n \geq x_{m-1/2}$, the case $(t,y(t))\in R^\sigma$ is possible only if the propagation speed of the rarefaction speed in $(\rho^\sigma,v^\sigma)$ is higher then the bus speed, i.e.
$$\lambda_1(\rho^\sigma,v^\sigma)\geq V_b.$$
Let $(\rho^{V_b+},v^{V_b+})$ be the point such that
$$\lambda_1(\rho^{V_b+},v^{V_b+})=V_b+.$$
This point is the last point of the rarefaction which can interact with the bus. Since $v^{V_b+}>\lambda_1(\rho^{V_b+},v^{V_b+})$, when the interaction occurs the bus cannot take the speed of the vehicles but it keeps its maximal speed $V_b$. Moreover the bus interacts only with the points $(\rho^\sigma,v^\sigma)$ of the rarefaction wave such that
$$\lambda_1(\rho^{V_b+},v^{V_b+})\leq \lambda_1(\rho^\sigma,v^\sigma)\leq \lambda_1(\bar{\rho}^n_m,\bar{v}^n_m).$$
Since for all these points we have $v^\sigma> v^{V_b+}>V_b$, the bus keeps its maximal speed $V_b$ during all the interaction.
\end{proof}
Let $(t^*,x^*)$ be the first point of interaction between the bus and the rarefaction wave centred in $x_{m-1/2}$. Since $\lambda_1(\bar{\rho}^n_m,\bar{v}^n_m)=\bar{v}^n_m-\bar{\rho}^n_m \,p'(\bar{\rho}^n_m)$ is the right propagation speed of the rarefaction, we can solve the equation
$$x_{m-1/2}+\lambda_1(\bar{\rho}^n_m,\bar{v}^n_m)(t^*-t^n)=y^n+V_b (t^*-t^n)$$
with respect to $t^*$, finding
\begin{equation}
t^* = \dfrac{y^n-x_{m-1/2}}{\lambda_1(\bar{\rho}^n_m,\bar{v}^n_m)-V_b}+t^n \; \text{ and } \; x^* = y^n+V_bt^*.
\end{equation}
Whether $t^*\geq k$, no interactions between the bus and the wave occur within the interval $[t^n,t^{n+1})$. Therefore the bus position at time $t^{n+1}$ is
$$y^{n+1}=y^n+\bar{V}^n k.$$
Otherwise the bus position at time $t^{n+1}$ is
$$y^{n+1}= y^n+V_b k,$$
provided that no other interactions happen whit waves centred in $x_{m+1/2}$.
\end{enumerate}
The next proposition states that if the bus interacts with a wave coming from the Riemann problem centred in $x_{m-1/2}$, then it cannot interact with a wave centred in $x_{m+1/2}$ and vice versa, provided that the time step is small enough.
\begin{prop}\label{prop_only_left_or_right_interaction_can_happen}
Let us suppose that the strong CFL condition (\ref{strong_CFL_condiftion}) holds. Within the interval $[t^n,t^{n+1})$ the bus can interact only with a wave coming from the local Riemann problem (\ref{Riemann_problem_centered_x_m-1/2}) centred in $x_{m-1/2}$ or with a wave (shock or rarefaction) coming form the local Riemann problem (\ref{Riemann_problem_centered_x_m+1/2}) centred in $x_{m+1/2}$.
\end{prop}
\begin{proof}
Let us recall that the propagation speeds of two shocks centred respectively in $x_{m-1/2}$ and $x_{m+1/2}$ are respectively
$$\lambda_{m-1/2} =\dfrac{\rho^\text{int}_{m-1/2} \, v^\text{int}_{m-1/2}-\bar{\rho}^n_m\bar{v}^n_m}{\rho^\text{int}_{m-1/2}-\bar{\rho}^n_m} \; \text{ and }\;  \lambda_{m+1/2}=\dfrac{\rho^\text{int}_{m+1/2} \, v^\text{int}_{m+1/2}-\bar{\rho}^n_m\bar{v}^n_m}{\rho^\text{int}_{m+1/2}-\bar{\rho}^n_m},$$
where $({\rho}^\text{int}_{m\pm 1/2},{v}^\text{int}_{m\pm 1/2})$ are the intermediate states of the classical solution to the Riemann problems centred in $x_{m\pm 1/2}$.\\
Let us denote
$$\overline{\lambda}_{m-1/2} = \begin{cases}
\lambda_{m-1/2} & \text{if } \bar{\rho}^n_{m-1} < \bar{\rho}^n_m,\\
\lambda_1(\bar{\rho}^n_{m},\bar{v}^n_m) & \text{otherwise,}
\end{cases}$$
the (right) propagation speed of a wave centred in $x_{m-1/2}$ and
$$\overline{\lambda}_{m+1/2} = \begin{cases}
\lambda_{m+1/2} & \text{if } \bar{\rho}^n_{m} < \bar{\rho}^n_{m+1},\\
\lambda_1(\bar{\rho}^n_{m},\bar{v}^n_m) & \text{otherwise,}
\end{cases}$$
the (left) propagation speed of a wave centred in $x_{m+1/2}$.\\
The condition (\ref{strong_CFL_condiftion}) implies that waves coming from the local Riemann problems (\ref{Riemann_problem_centered_x_m+1/2}) and (\ref{Riemann_problem_centered_x_m-1/2}) cannot cover more space then $h/2$ before time $t^{n+1}$, where
$$h = x_{m+1/2}-x_{m-1/2}$$
is the (constant) space length of the cells. Hence we have
$$|\overline{\lambda}_{m \pm 1/2}|k\leq h/2.$$
Let us suppose that the bus interacts with a shock or a rarefaction wave centred in $x_{m-1/2}$ and let $(t^*,x^*)$ be the point of first interaction, i.e.
$$x^*= x_{m-1/2}+\overline{\lambda}_{m-1/2}t^* = y^n+V_bt^*,$$
because in this case the bus speed is $V_b$ on all the interval $[t^n,t^{n+1})$.
The bus and the wave can interact only if
$$t^*\leq k \; \text{ and } V_b \leq \overline{\lambda}_{m-1/2},$$
which imply
$$V_b(k-t^*) \leq \overline{\lambda}_{m-1/2}(k-t^*).$$
Since $\overline{\lambda}_{m-1/2}k\leq h/2$, we find
$$x^*-x_{m-1/2}=\overline{\lambda}_{m-1/2}t^*\leq \overline{\lambda}_{m-1/2}k \leq h/2.$$
Moreover we have
\begin{equation*}
\begin{split}
x^*+ V_b(k-t^*) & \leq x^*+\overline{\lambda}_{m-1/2}(k-t^*) = x_{m-1/2}+\overline{\lambda}_{m-1/2}t^* +\overline{\lambda}_{m-1/2}(k-t^*) =\\
& =  x_{m-1/2}+\overline{\lambda}_{m-1/2}k \leq x_{m-1/2}+\lambda^n k\leq x_{m-1/2}+h/2=\\
& = x_m.
\end{split}
\end{equation*}
Hence no interactions with a wave centred in $x_{m+1/2}$ can happen.\\
On the other hand, let us suppose that the bus interacts first with a wave centred in $x_{m+1/2}$ and let $(t^*,x^*)$ be the first point of interaction. Since $|\overline{\lambda}_{m+1/2}|k\leq h/2$, we must have $x^*\geq x_m = x_{m-1/2} +h/2$. Hence no interactions with waves centred in $x_{m-1/2}$ can happen after $t^*$.
\end{proof}
\begin{remark}
Let $t^*_{m-1/2}$ and $t^*_{m+1/2}$ be respectively the time of first interaction between the bus trajectory and a wave centred in $x_{m-1/2}$ and $x_{m+1/2}$. To recognize if the bus interacts with a wave centred in $x_{m-1/2}$ or with a wave centred in $x_{m+1/2}$, we compute both $t^*_{m-1/2}$ and $t^*_{m+1/2}$.\\
If $t^*_{m-1/2}\leq \min(t^*_{m+1/2},k)$, then Proposition \ref{prop_only_left_or_right_interaction_can_happen} ensures that within $[t^n,t^{n+1})$ the only interaction happens with the wave centred in $x_{m-1/2}$.\\
Similarly for the case $t^*_{m+1/2}\leq \min(t^*_{m-1/2}, k)$.
\end{remark}
\subsection{Example of a reconstructed bus trajectory}\label{section_bus_trajectory}
Fix the following data:
\begin{itemize}
\item the pressure function is $p(\rho)=\rho$;
\item the reduction rate in the road capacity due to the bus is $\alpha=0.5$;
\item the bus initial position is $y_0=-0.1$;
\item the bus maximal speed is $V_b=4$;
\item the maximal density of vehicles allowed on the road is $R_{\text{max}}=15$.
\end{itemize}
Let us consider a bus influenced by the previous vehicles. For example, consider the initial datum
\begin{equation}
(\rho,v)(0,x)=\begin{cases}
(\rho^l,v^l)=(9,1) &\text{if } x \leq 0,\\
(\rho^r,v^r)=(2,8) & \text{if } x>0.
\end{cases}
\end{equation}
Since $\rho^l >\rho^r$, the standard solution is a rarefaction wave with propagation speed ranging within the interval $[\lambda_1(\rho^l,v^l),\lambda_1(\rho^r,v^r)]=[-8,6]$; see Figures \ref{fig_bus_trajectory_flux_1} and \ref{fig_bus_trajectory_rho_v_1}.\\
The speed of the cars at $y_0$ is $v^l=1$ and it is lower than the bus maximal speed. Hence the bus has to adapt its speed to the one of the cars in front of it.\\
The bus will keep this velocity until, at some instant $t^*$, an interaction with the rarefaction wave centred in $x=0$ happens. At time $t^*$ the bus can accelerate to its maximal speed; see Figures \ref{fig_bus_trajectory_flux_2}, \ref{fig_bus_trajectory_flux_3}, \ref{fig_bus_trajectory_rho_v_2} and \ref{fig_bus_trajectory_bus_speed}. Then there is an interval $[t^*,t^{**}]$ in which the bus and the traffic have no influence on each other, until at time $t^{**}$ the constraint is enforced and the non-classical shock appears; see Figures \ref{fig_bus_trajectory_rho_v_3} and \ref{fig_bus_trajectory_rho_v_4}.\\
In Figures \ref{fig_bus_trajectory_bus_speed} and \ref{fig_bus_trajectory_space_time} is represented the evolution of the bus speed and the bus trajectory in the space-time diagram.
\begin{figure}[hbtp]
\centering
\includegraphics[width=0.93\linewidth]{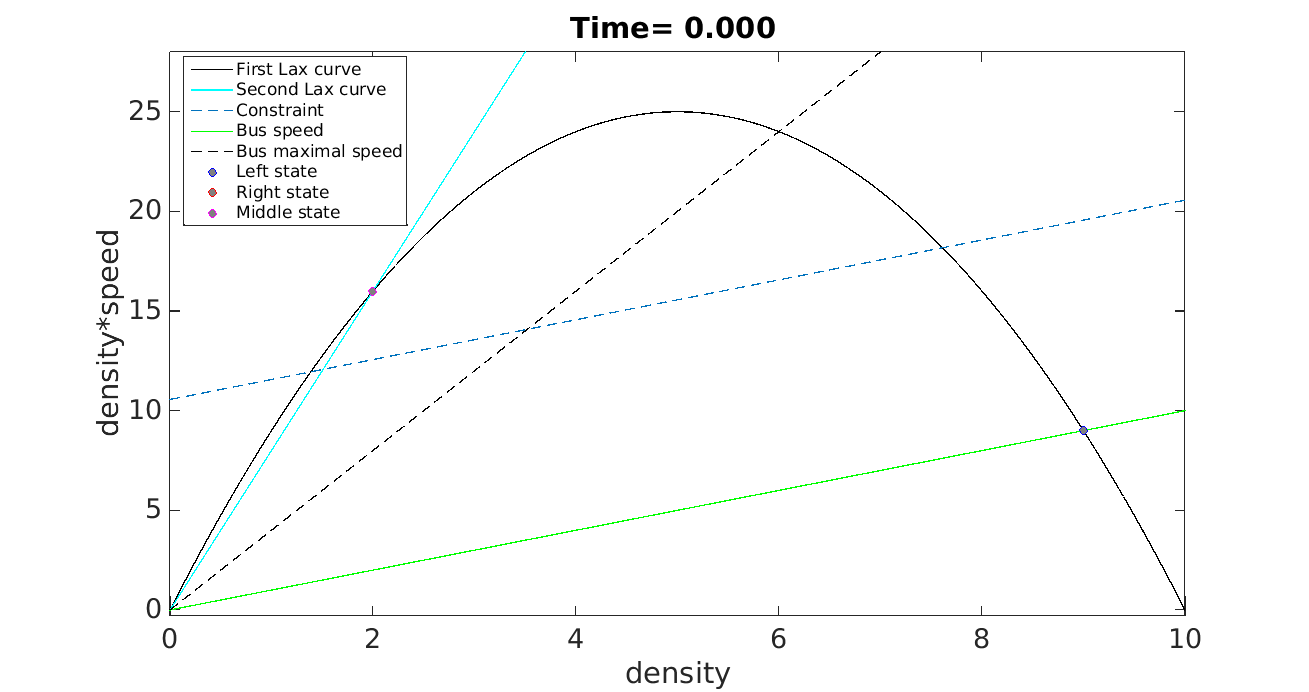}
\caption{Situation described in Subsection \ref{section_bus_trajectory}: initial configuration. The bus takes the speed of the cars in front of it.}\label{fig_bus_trajectory_flux_1}
\end{figure}
\begin{figure}[hbtp]
\centering
\includegraphics[width=0.93\linewidth]{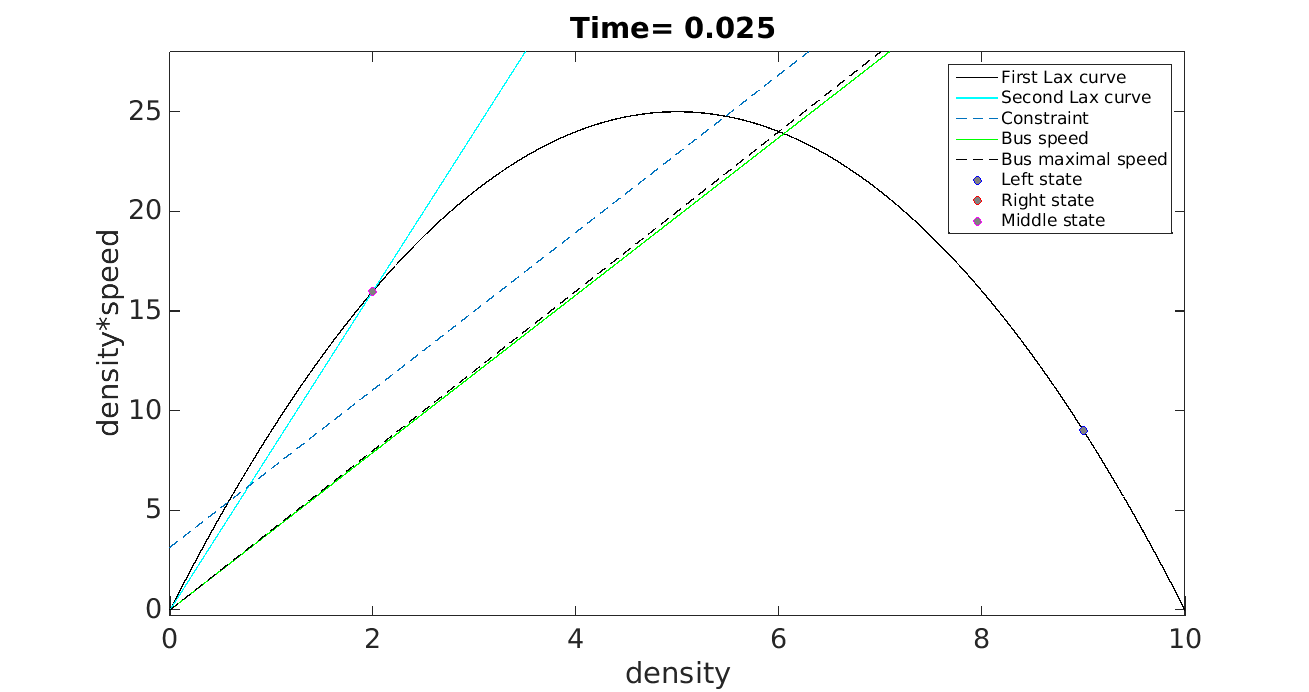}
\caption{Situation described in Subsection \ref{section_bus_trajectory}: the time $t$ is less then the instant $t^*$ for which the bus reaches its maximal speed. The bus is accelerating: it takes the speed of the cars in front of it.}\label{fig_bus_trajectory_flux_2}
\end{figure}
\begin{figure}[hbtp]
\centering
\includegraphics[width=0.93\linewidth]{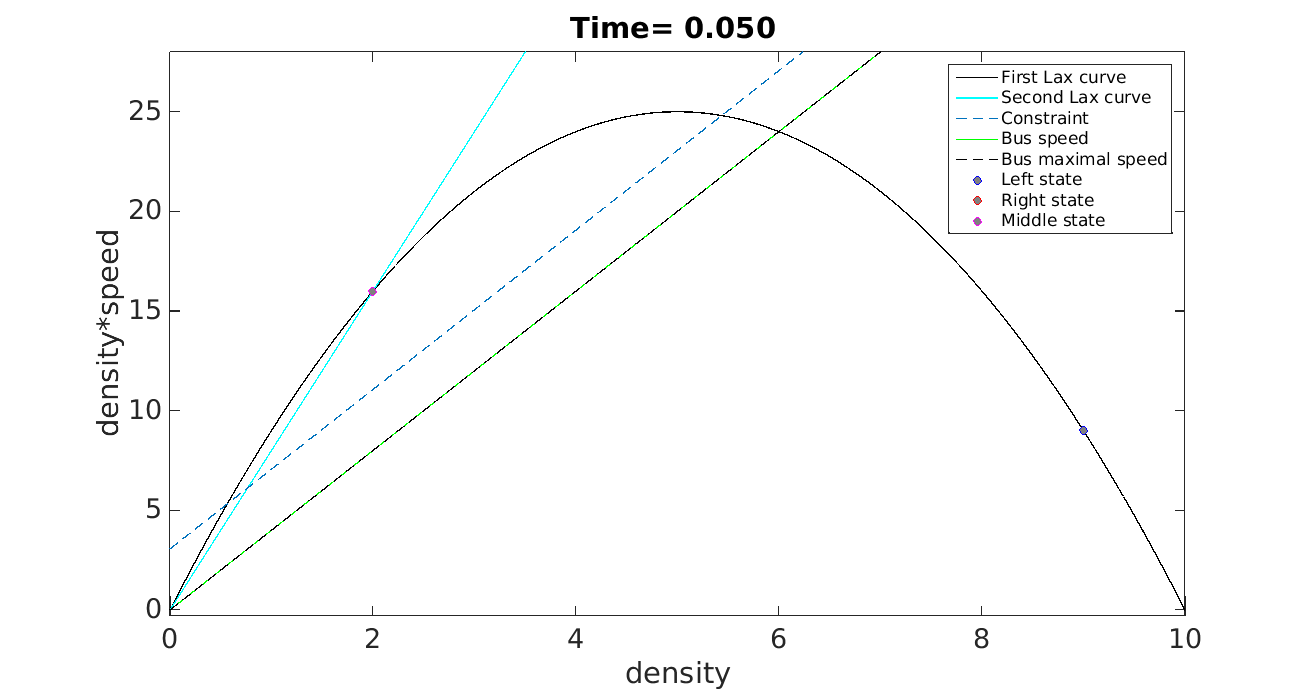}
\caption{Situation described in Subsection \ref{section_bus_trajectory}. The bus has reached its maximal speed.}\label{fig_bus_trajectory_flux_3}
\end{figure}
\begin{figure}[hbtp]
\centering
\includegraphics[width=0.9\linewidth]{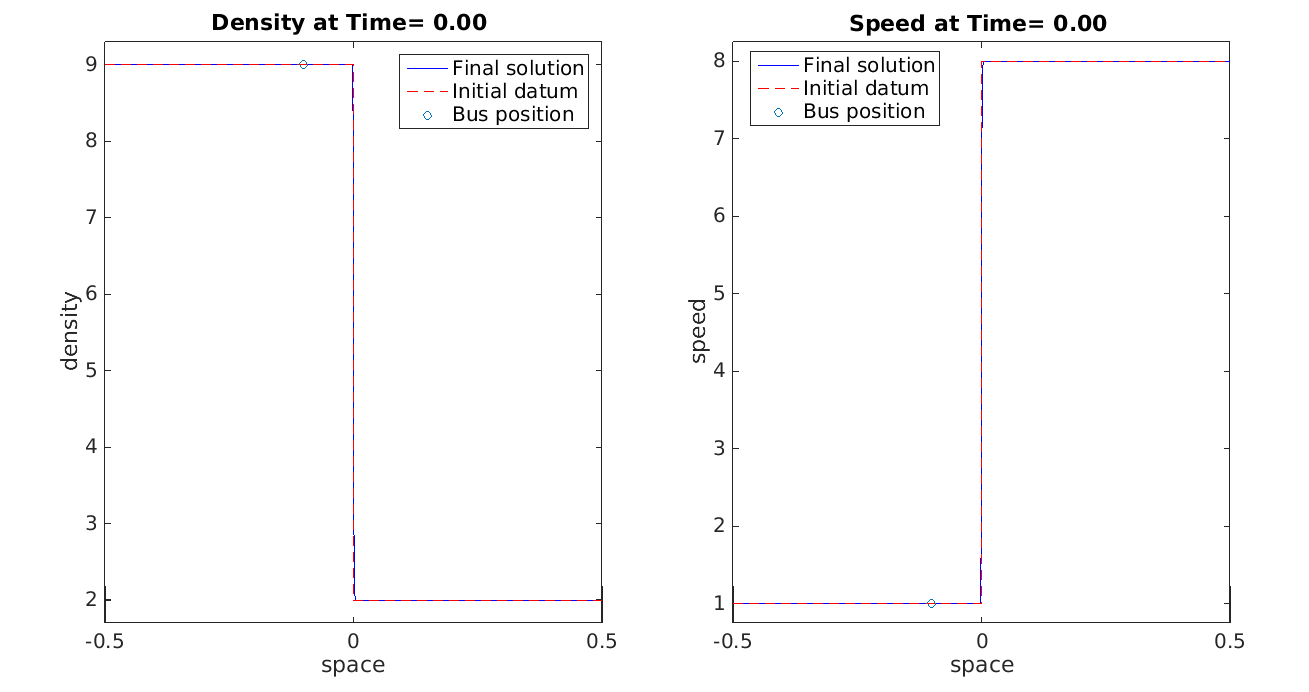}
\caption{Situation described in Subsection \ref{section_bus_trajectory}: initial datum.}\label{fig_bus_trajectory_rho_v_1}
\end{figure}
\begin{figure}[hbtp]
\centering
\includegraphics[width=0.9\linewidth]{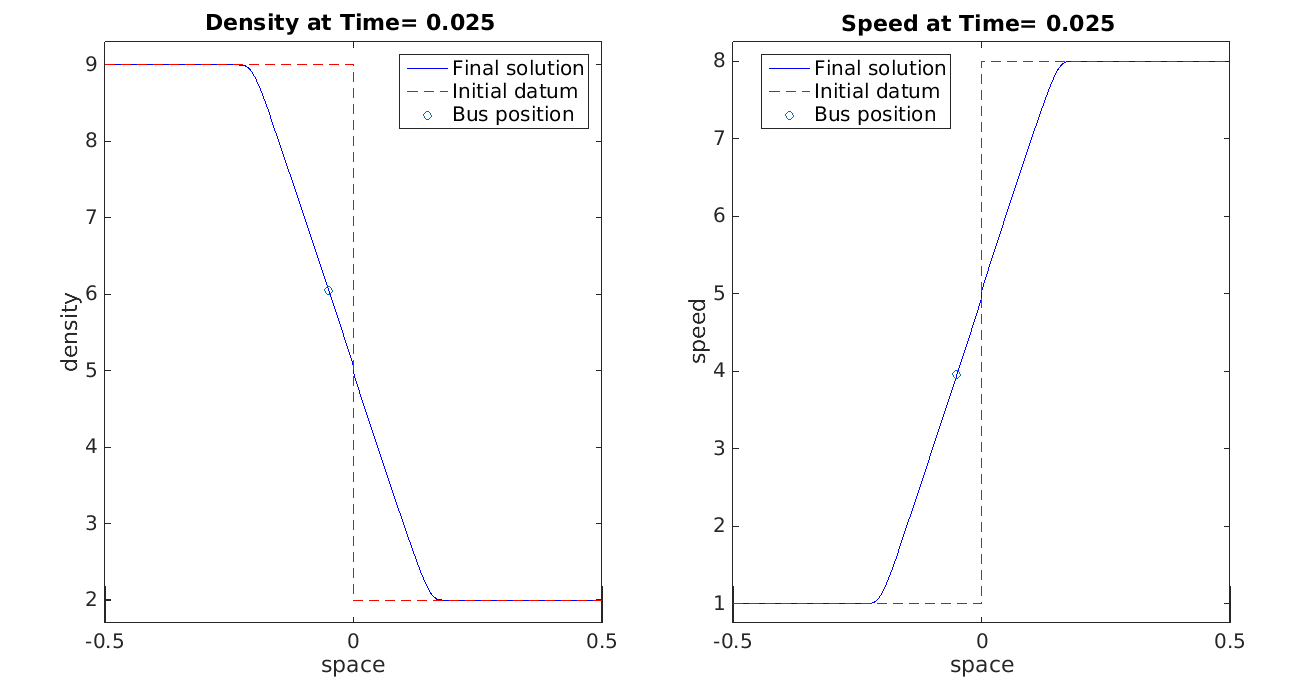}
\caption{Situation described in Subsection \ref{section_bus_trajectory}: the time $t$ is less then the instant $t^*$ for which the bus reaches its maximal speed. The solution given by $\mathcal{RS}^\alpha_1$ is classical.}\label{fig_bus_trajectory_rho_v_2}
\end{figure}
\begin{figure}[hbtp]
\centering
\includegraphics[width=0.9\linewidth]{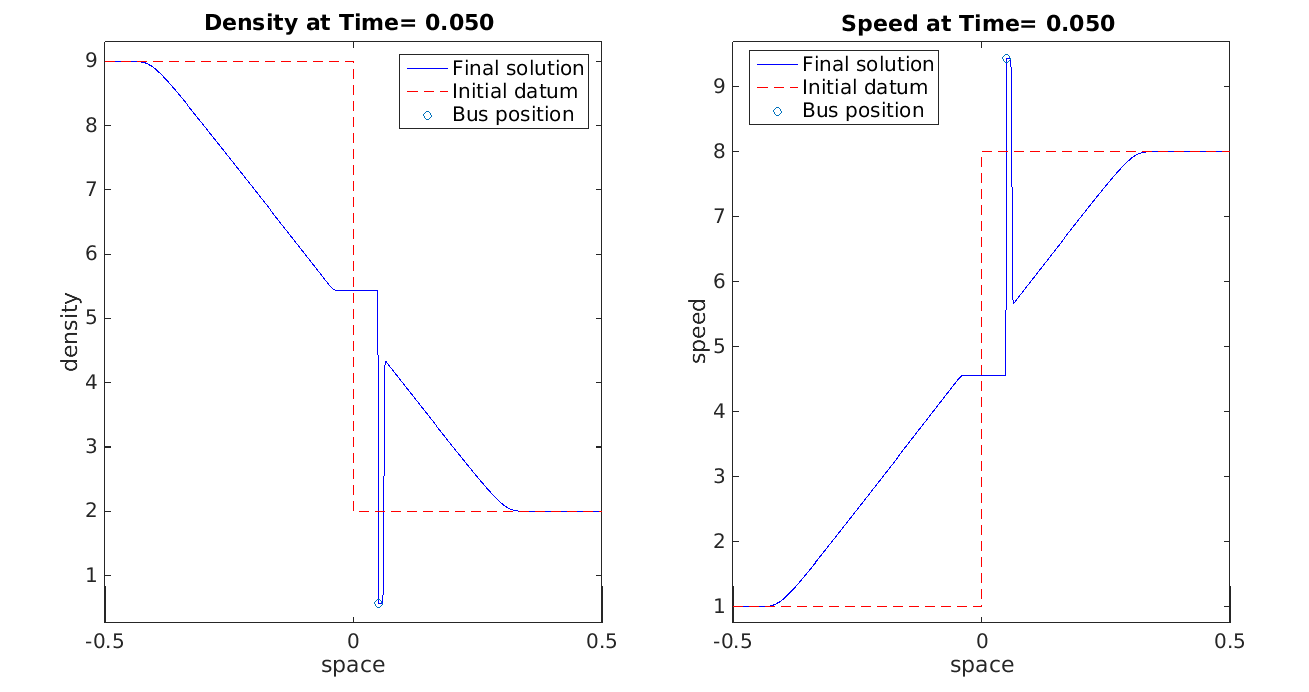}
\caption{Situation described in Subsection \ref{section_bus_trajectory}: the non-classical shock has appeared. }\label{fig_bus_trajectory_rho_v_3}
\end{figure}
\begin{figure}[hbtp]
\centering
\includegraphics[width=0.9\linewidth]{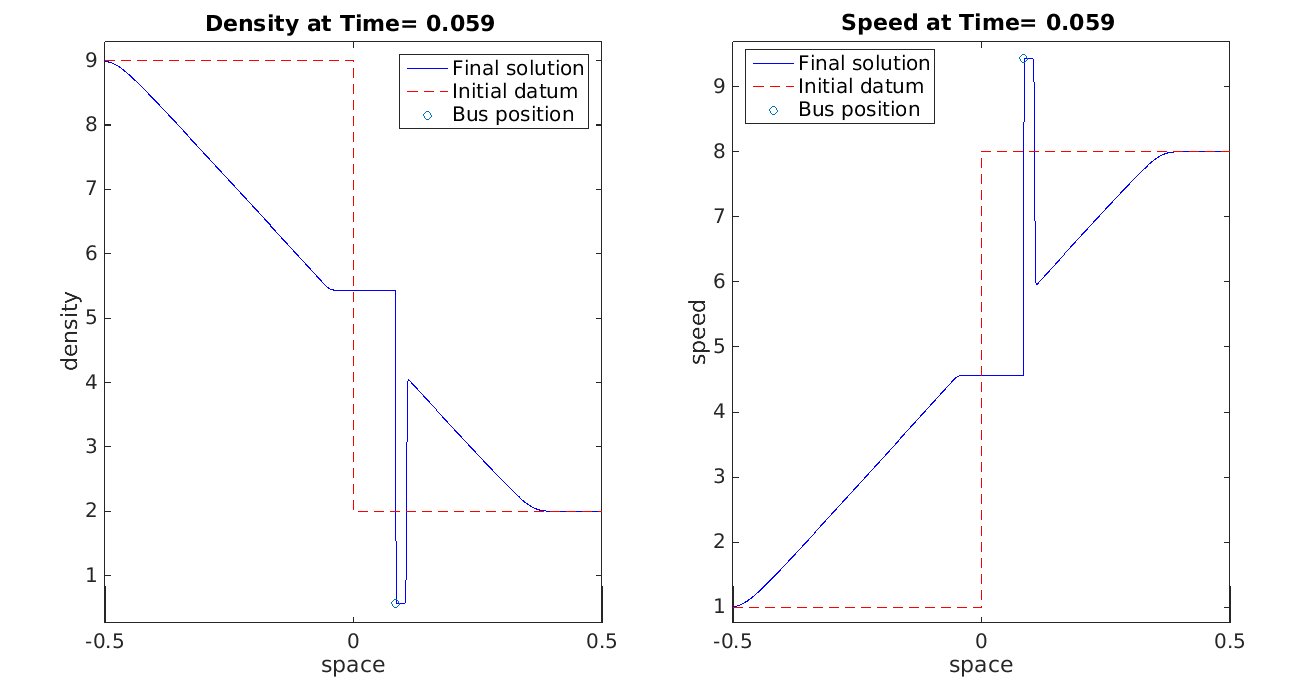}
\caption{Situation described in Subsection \ref{section_bus_trajectory}: the non-classical shock has appeared. }\label{fig_bus_trajectory_rho_v_4}
\end{figure}
\begin{figure}[hbtp]
\centering
\includegraphics[width=0.9\linewidth]{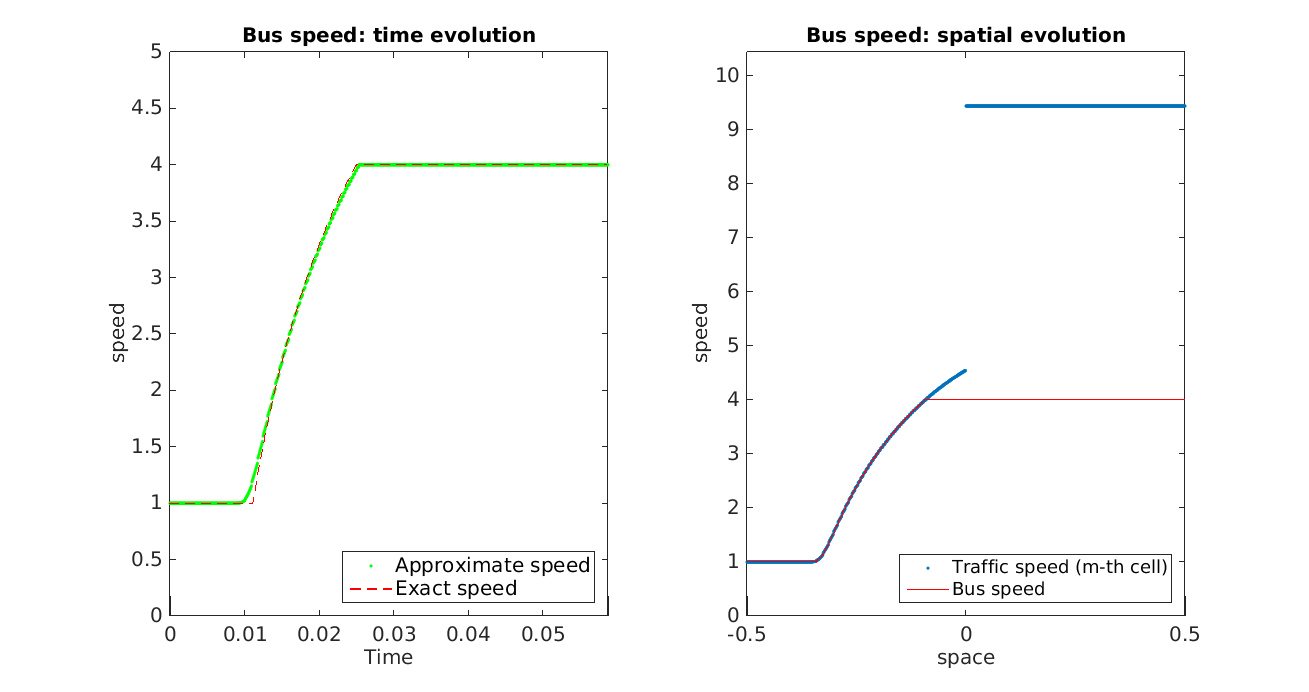}
\caption{Situation described in Subsection \ref{section_bus_trajectory}. The first plot is the time evolution of the bus speed. The bus accelerates until it reaches its maximal speed. The reconstructed bus speed coincides with the one expected.\\
The second plot is the spatial evolution of the bus speed and of the velocity of the traffic in front of the bus. We can see that the bus takes the speed of the cars in front of it until it reaches its maximal speed. When this happens the cars accelerate until the constraint is enforced, then they take the speed $\check{v}_1$.}\label{fig_bus_trajectory_bus_speed}
\end{figure}
\begin{figure}[hbtp]
\centering
\includegraphics[width=\linewidth]{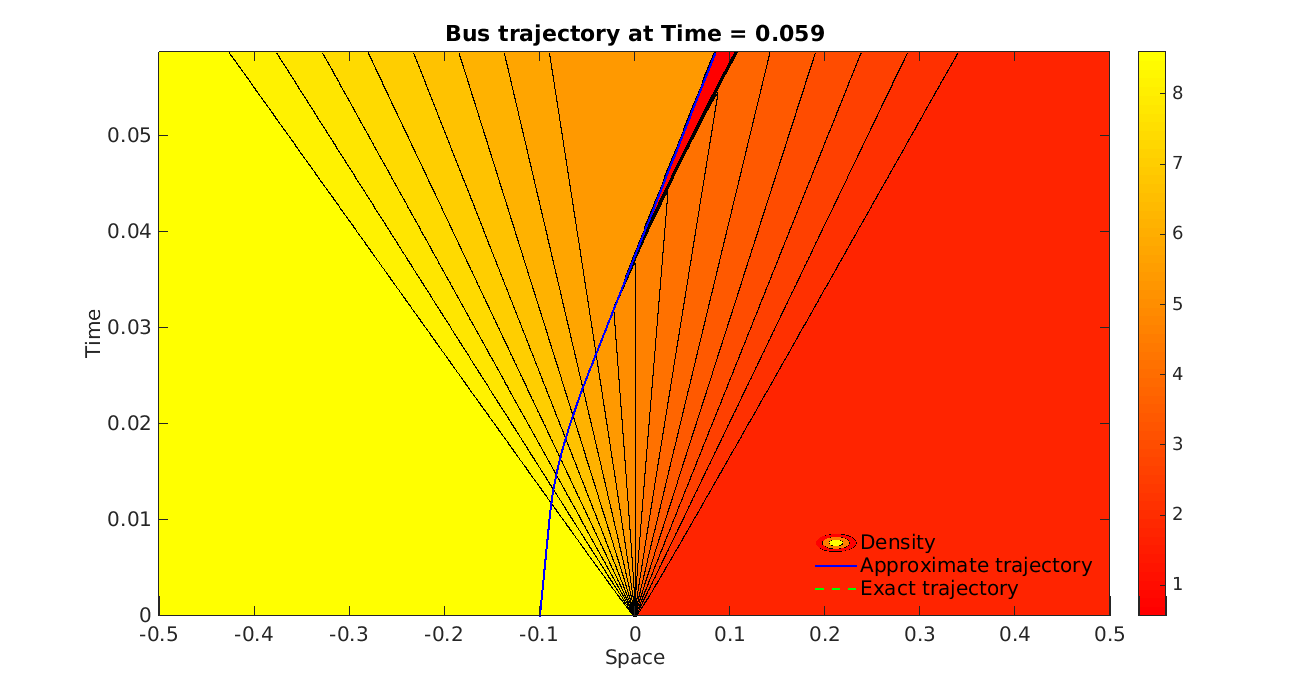}
\caption{Situation described in Subsection \ref{section_bus_trajectory}: the bus trajectory and the density in the $x-t$ plane. The bus trajectory is exactly captured.}\label{fig_bus_trajectory_space_time}
\end{figure}
\newpage
\section{Numerical methods for the Riemann solver $\mathcal{RS}^\alpha_2$}
Let $\bar{V}^n$ be the bus speed at time $t^n$. The Riemann solvers $\mathcal{RS}^\alpha_1$ and $\mathcal{RS}^\alpha_2$ give the same solution whenever the constraint is satisfied. Hence, if
$$f_1(\mathcal{RS}(\bar{u}^n_{m-1},\bar{u}^n_{m+1})(\bar{V}^n))\leq F_\alpha + \bar{V}^n\bar{\rho}(\bar{u}^n_{m-1},\bar{u}^n_{m+1})(\bar{V}^n),$$
we can apply to $\mathcal{RS}^\alpha_2$ the methods introduced for $\mathcal{RS}^\alpha_1$.\\
When the constraint is enforced, we propose two methods to capture the non-classical shock.\\
The following lemma gives a formula to compute the solution at the new time step on a cell of a non-uniform mesh, if the solution at time $t^n$ is known.
\begin{lemma}\label{Green_theorem_on_nonuniform_mesh}
Fix $n\in \mathbb{N}$. Let us consider two time steps $t^n$ and $t^{n+1} = t^n + k$ for a fixed constant $k \in \mathbb{R}^+$ and the $j$-th cell of a non-uniform spatial mesh $\lbrace x^n_{j+1/2} \rbrace_{j\in \mathbb{Z}}$ defined by
$$ x^n_{j+1/2} = x^n_{j-1/2} + h^n_j \; \text{ for } \; h^n_j \in \mathbb{R}^+ \; \text{ and }\; j\in\mathbb{Z}.$$
Let $\bar{u}^n_j$ and $\bar{u}^{n+1}_j$ be approximate solutions for the conservation law
$$\partial_t \, u + \partial_x \, [f(u)] =0$$
on the cell $C^i_j = [x^i_{j-1/2},x^i_{j+1/2})$ for $i=n,n+1$, at time $t^n$ and $t^{n+1}$ respectively. The following formula holds:
\begin{equation}\label{formula_conservation_imposed_nonuniform_mesh}
\begin{split}
h^{n+1}_{j}\bar{u}^{n+1}_j = h^n_j\bar{u}^n_{j}
-k & [ f(\bar{u}^n_{j+1/2}(\lambda^r))-\lambda^r\, \bar{u}^n_{j+1/2}(\lambda^r) +\\
& -f(\bar{u}^n_{j-1/2}(\lambda^l))+\lambda^l \, \bar{u}^n_{j-1/2}(\lambda^l)],
\end{split}
\end{equation}
where
\begin{equation*}
\begin{split}
& \lambda^l = \dfrac{x^{n+1}_{j-1/2}-x^n_{j-1/2}}{k}, \; \; \lambda^r = \dfrac{x^{n+1}_{j+1/2}-x^n_{j+1/2}}{k}, \; \; \bar{u}^n_{j-1/2}(\lambda^l) = \mathcal{RS}(\bar{u}^n_{j-1},\bar{u}^n_j)(\lambda^l) \; \text{ and}\\
& \bar{u}^n_{j+1/2}(\lambda^r) = \mathcal{RS}(\bar{u}^n_j,\bar{u}^n_{j+1})(\lambda^r).
\end{split}
\end{equation*}
See Figure \ref{fig_Green_theorem_on_nonuniform_mesh}.
\end{lemma}
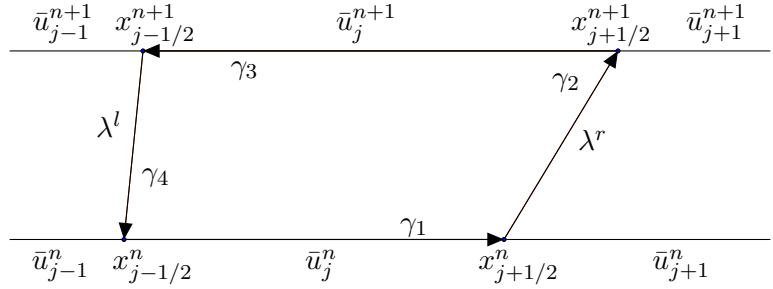
\begin{figure}[hbt]
\centering
\definecolor{zzttqq}{rgb}{0.6,0.2,0.}
\definecolor{qqqqff}{rgb}{0.,0.,1.}
\begin{tikzpicture}[scale =0.5, line cap=round,line join=round,>=triangle 45,x=1.0cm,y=1.0cm]
\clip(-4.,-3.) rectangle (17.,8.);
\draw [color=zzttqq] (0.5,5.)-- (0.,0.);
\draw [color=zzttqq] (0.,0.)-- (10.,0.);
\draw [color=zzttqq] (10.,0.)-- (13.,5.);
\draw [color=zzttqq] (13.,5.)-- (0.5,5.);
\draw [domain=-3.:17.] plot(\x,{(-150.-0.*\x)/-30.});
\draw [domain=-3.:17.] plot(\x,{(-0.-0.*\x)/30.});
\draw (-0.5,6.5) node[anchor=north west] {$x^{n+1}_{j-1/2}$};
\draw (11.5,6.5) node[anchor=north west] {$x^{n+1}_{j+1/2}$};
\draw (-0.57,0.) node[anchor=north west] {$x^n_{j-1/2}$};
\draw (9.06,0.) node[anchor=north west] {$x^n_{j+1/2}$};
\draw (4.50,0.) node[anchor=north west] {$\bar{u}^n_j$};
\draw (5.31,6.5) node[anchor=north west] {$\bar{u}^{n+1}_j$};
\draw (-2.7,0.) node[anchor=north west] {$\bar{u}^n_{j-1}$};
\draw (-2.67,6.5) node[anchor=north west] {$\bar{u}^{n+1}_{j-1}$};
\draw (14.52,6.5) node[anchor=north west] {$\bar{u}^{n+1}_{j+1}$};
\draw (13.64,0.) node[anchor=north west] {$\bar{u}^n_{j+1}$};
\draw (-1,3.66) node[anchor=north west] {$\lambda^l$};
\draw (11.68,3.17) node[anchor=north west] {$\lambda^r$};
\draw (7.,0.83) node[anchor=north west] {$\gamma_1$};
\draw (11.,4.68) node[anchor=north west] {$\gamma_2$};
\draw (2.54,5.) node[anchor=north west] {$\gamma_3$};
\draw (0.20,2.15) node[anchor=north west] {$\gamma_4$};
\draw [->] (0.,0.) -- (10.,0.);
\draw [->] (10.,0.) -- (13.,5.);
\draw [->] (13.,5.) -- (0.5,5.);
\draw [->] (0.5,5.) -- (0.,0.);
\begin{scriptsize}
\draw [fill=qqqqff] (0.,0.) circle (1.5pt);
\draw [fill=qqqqff] (10.,0.) circle (1.5pt);
\draw [fill=qqqqff] (13.,5.) circle (1.5pt);
\draw [fill=qqqqff] (0.5,5.) circle (1.5pt);
\draw [fill=qqqqff] (-10.,0.) circle (1.5pt);
\draw[color=qqqqff] (-4.16,9.71) node {$E$};
\draw [fill=qqqqff] (-10.,5.) circle (1.5pt);
\draw[color=qqqqff] (-4.16,9.71) node {$F$};
\draw [fill=qqqqff] (20.,0.) circle (1.5pt);
\draw[color=qqqqff] (-4.16,9.71) node {$G$};
\draw [fill=qqqqff] (20.,5.) circle (1.5pt);
\draw[color=qqqqff] (-4.16,9.71) node {$H$};
\end{scriptsize}
\end{tikzpicture}
\caption{Notations used in Lemma \ref{Green_theorem_on_nonuniform_mesh}.}\label{fig_Green_theorem_on_nonuniform_mesh}
\end{figure}
\begin{proof}
Integrating the conservation law $\partial_t u +\partial_x[f(u)] = 0$ over $C_{j}$ and applying Green's Theorem, we find
$$0=\int_{C_{j}} \partial_t u +\partial_x[f(u)] \, dt\, dx = \int_{\partial C_{j}} [u \, dx - f(u) \, dt] = \int_{\partial \, C_{j}} \begin{pmatrix}
-f(u)\\
u
\end{pmatrix}\cdot \begin{pmatrix}
dt\\
dx
\end{pmatrix}.$$
Let us split the boundary $\partial C_{j}$ of the cell $C_j$ in the four edges $\lbrace \gamma_i \rbrace_{i=1}^4$ parametrized in the $(t,x)$ plane as follows (see Figure \ref{fig_Green_theorem_on_nonuniform_mesh}).
\begin{itemize}
\item $\gamma_1(s) = \left(t^n, x^n_{j-1/2} + s \, h^n_j\right)$ for $s \in [0,1]$.
\item $\gamma_2(s) = \left(t^n + s\, k,x^n_{j+1/2}+s\left(x^{n+1}_{j+1/2}-x^n_{j+1/2}\right)\right)$ for $s \in [0,1]$.
\item $\gamma_3(s) = \left(t^{n+1}, x^{n+1}_{j+1/2} - s\, h^{n+1}_j\right)$ for $s\in [0,1]$.
\item $\gamma_4(s) = \left(t^{n+1} -s\, k, x^{n+1}_{j-1/2}+s\left(x^n_{j-1/2}-x^{n+1}_{j-1/2}\right)\right)$ for $s\in [0,1]$.
\end{itemize}
We have to integrate the vector field $(-f(u),u)$ on each edge.\\
For the first edge, we find
\begin{equation*}
\begin{split}
\int_{\gamma_1}  \begin{pmatrix}
-f(u)\\
u
\end{pmatrix}\cdot \begin{pmatrix}
dt\\
dx
\end{pmatrix} & = \int_0^1 \begin{pmatrix}
-f(\bar{u}^n_{j})\\
\bar{u}^n_{j}
\end{pmatrix} \cdot \gamma_1'(s)\, ds =\\
& = \int_0^1 \begin{pmatrix}
-f(\bar{u}^n_{j})\\
\bar{u}^n_{j}
\end{pmatrix} \cdot \begin{pmatrix}
0\\
h^n_j
\end{pmatrix} \, ds= \\
& = \, h^n_j\, \bar{u}^n_{j}.
\end{split}
\end{equation*}
Similarly 
\begin{equation*}
\begin{split}
\int_{\gamma_3}  \begin{pmatrix}
-f(u)\\
u
\end{pmatrix}\cdot \begin{pmatrix}
dt\\
dx
\end{pmatrix} & = \int_0^1 \begin{pmatrix}
-f(\bar{u}^{n+1}_{j})\\
\bar{u}^{n+1}_{j}
\end{pmatrix} \cdot \begin{pmatrix}
0\\
-h^{n+1}_j
\end{pmatrix} \, ds= \\
& = -h^{n+1}_j \, \bar{u}^{n+1}_{j}.
\end{split}
\end{equation*}
For the edge $\gamma_2$, we have
\begin{equation*}
\begin{split}
\int_{\gamma_2}  \begin{pmatrix}
-f(u)\\
u
\end{pmatrix}\cdot \begin{pmatrix}
dt\\
dx
\end{pmatrix} & = \int_0^1  \begin{pmatrix}
-f(\bar{u}^n_{j+1/2}(\lambda^r))\\
\bar{u}^n_{j+1/2}(\lambda^r)
\end{pmatrix}\cdot \begin{pmatrix}
k\\
\left(x^{n+1}_{j+1/2}-x^n_{j+1/2}\right)
\end{pmatrix} \, ds = \\
& = \left(x^{n+1}_{j+1/2}-x^n_{j+1/2} \right)\bar{u}^n_{j+1/2}(\lambda^r)-k \, f(\bar{u}^n_{j+1/2}(\lambda^r)) = \\
& = k \, \left[ \lambda^r \bar{u}^n_{j+1/2}(\lambda^r)-f(\bar{u}^n_{j+1/2}(\lambda^r))\right].
\end{split}
\end{equation*}
Finally
\begin{equation*}
\begin{split}
\int_{\gamma_4}  \begin{pmatrix}
-f(u)\\
u
\end{pmatrix}\cdot \begin{pmatrix}
dt\\
dx
\end{pmatrix} 
& = \int_0^1 \begin{pmatrix}
-f(\bar{u}^n_{j-1/2}(\lambda^l))\\
\bar{u}^n_{j-1/2}(\lambda^l)
\end{pmatrix}\cdot \begin{pmatrix}
-k\\
\left(x^n_{j-1/2} -x^{n+1}_{j-1/2} \right)
\end{pmatrix} \, ds = \\
& = k \left[f(\bar{u}^n_{j-1/2}(\lambda^l)) - \lambda^l \, \bar{u}^n_{j-1/2}(\lambda^l) \right]
\end{split}
\end{equation*}
Assembling the results, we find the equation (\ref{formula_conservation_imposed_nonuniform_mesh}).
\end{proof}
\subsection{First method: Discontinuity reconstruction}
The following method is similar to the discontinuity reconstruction procedure that we have used for the Riemann solver $\mathcal{RS}^\alpha_1$.\\
Fix $n\in\mathbb{N}$. Let us suppose that we have computed the piecewise constant approximate solution $\bar{u}^n$ at the time $t^n$ with the Godunov's method.\\
Let $y^n:=y(t^n)$ be the bus position at time $t^n$ and let us fix $m \in \mathbb{Z}$ such that $y^n \in C_m$. Let 
\begin{equation}
\bar{\rho}^n_m \; \text{ and } \; \bar{z}^n_m
\end{equation}
be respectively the $\rho$ and $z$ component of the approximate solution $\bar{u}^n$ in the $m$-th cell.\\
If the Riemann solver $\mathcal{RS}^\alpha_2$ does not give the classical solution, a non-classical shock appears in $x=y(t)$. Since the non-classical shock arises as the solution given by $\mathcal{RS}^\alpha_2$ to the Riemann problem with initial datum
\begin{equation}\label{initial_datum_discontinuity_reconstructio_RS_2}
u(0,x)=\begin{cases}
\bar{u}^n_{m-1} & \text{if } x\leq y^n,\\
\bar{u}^n_{m+1} & \text{if } x> y^n,
\end{cases}
\end{equation}
we will make a reconstruction of the discontinuity, if the inequality
\begin{equation}\label{second_inequality_RS_2}
f_1(\mathcal{RS}(\bar{u}^n_{m-1},\bar{u}^n_{m+1})(\bar{V}^n))>F_\alpha+\bar{V}^n\bar{\rho}(\bar{u}^n_{m-1},\bar{u}^n_{m+1})(\bar{V}^n)
\end{equation}
holds. In this case we modify the Godunov's scheme as follows.\\
We introduce in the $m$-th cell one left state $u^n_{m,l}=(\rho^n_{m,l},z^n_{m,l})$ and one right state $u^n_{m,r}=(\rho^n_{m,r},z^n_{m,r})$ defined by
$$u^n_{m,l}= \hat{u} \; \text{ and } \; u^n_{m,r}=\check{u}_2,$$
where $\hat{u}$ is defined in (\ref{def_u_hat_u_check_1}) and $\check{u}_2= (\check{\rho}_2,\check{z}_2)$ is given by
$$\check{\rho}_2 = \dfrac{F_\alpha}{\bar{v}^n_{m+1}-\bar{V}^n} \; \text{ and } \; \check{z}_2 = \check{\rho}_2(\check{v}_2+p(\check{\rho}_2)) \; \text{ with } \; \check{v}_2 = \bar{v}^n_{m+1}.$$
Then we replace the solution $\bar{u}^n_m$ obtained with the Godunov's method in the $m$-th cell, with the function $u^n_\text{rec}=(\rho^n_\text{rec},z^n_\text{rec})$ defined by
\begin{equation}
\begin{split}
&\rho^n_\text{rec}=\rho^n_{m,l}\mathbf{1}_{(x_{m-1/2},{x}^{\rho,n}_m)}+\rho^n_{m,r}\mathbf{1}_{({x}^{\rho,n}_m,x_{m+1/2})} \; \text{ and}\\
&z^n_\text{rec}=z^n_{m,l}\mathbf{1}_{(x_{m-1/2},{x}^{z,n}_m)}+z^n_{m,r}\mathbf{1}_{({x}^{z,n}_m,x_{m+1/2})},
\end{split}
\end{equation}
where we have used the two points
$${x}^{\rho,n}_m=x_{m-1/2}+h\, d^{n,\rho}_m \; \text{ and } \; {x}^{z,n}_m=x_{m-1/2}+h \, d^{n,z}_m$$
defined for two suitable constants $d^{n,\rho}_m$ and $d^{n,z}_m$ in $[0,1]$.\\
We require
\begin{equation}
\begin{split}
&\rho^n_{m,l}d^{n,\rho}_m+\rho^n_{m,r}(1-d^{n,\rho}_m)=\bar{\rho}^n_m \; \text{ and}\\
&z^n_{m,l}d^{n,z}_m+z^n_{m,r}(1-d^{n,z}_m)=\bar{z}^n_m.
\end{split}
\end{equation}
Solving these two equations for $d^{n,\rho}_m$ and $d^{n,z}_m$, we find
\begin{equation}\label{costanti_ricostruzione_discontinuita_RS2}
d^{n,\rho}_m=\dfrac{\bar{\rho}^n_m-\rho^n_{m,r}}{\rho^n_{m,l}-\rho^n_{m,r}} \; \text{ and } \; d^{n,z}_m=\dfrac{\bar{z}^n_m-z^n_{m,r}}{z^n_{m,l}-z^n_{m,r}}.
\end{equation}
Clearly, the conditions $d^{n,\rho}_m\in [0,1]$ and $d^{n,z}_m\in [0,1]$ are necessary to reconstruct the discontinuity in the cell $C_m$. These two constants are in general different.\\
We assume that the discontinuities propagate at the same speed of the non-classical shock. Therefore their positions at time $t^{n+1}$ are
\begin{equation}
x^{n+1,\rho}_m = x^{n,\rho}_m + \bar{V}^n\, k \; \text{ and } \; x^{n+1,z}_m = x^{n,z}_m + \bar{V}^n\,k.
\end{equation}
Let us introduce the matrix
$$\mathbf{d}= \begin{bmatrix}
d^{n,\rho}_m & 0\\
0 & d^{n,z}_m
\end{bmatrix}$$
and let
$$\mathbf{I}= \begin{bmatrix}
1 & 0\\
0 & 1
\end{bmatrix}$$
be the $2 \times 2$ identity matrix.\\
We have to distinguish the following cases.
\begin{enumerate}
\item[(i)] If $x^{n+1,\rho}_m < x_{m+1/2}$ and $x^{n+1,z}_m < x_{m+1/2}$, then we compute the average solution at the new time step in the $(m-1)$-th and in the $(m+1)$-th cell as:
\begin{equation}\label{disc_rec_RS2_case_i_cells_before_and_after}
\begin{split}
& \bar{u}^{n+1}_{m-1} = \bar{u}^n_{m-1} -\dfrac{k}{h} \big[F(\bar{u}^n_{m-1},\hat{u})-F(\bar{u}^n_{m-2}-\bar{u}^n_{m-1})\big] \; \text{ and }\\
& \bar{u}^{n+1}_{m+1} = \bar{u}^n_{m+1} -\dfrac{k}{h}\big[F(\bar{u}^n_{m+1},\bar{u}^n_{m+2})-F(\check{u}_2,\bar{u}^n_{m+1})\big].
\end{split}
\end{equation}
For the $m$-th cell, we set
\begin{equation}\label{discontinuity_recostruction_second_r_s_case_i_new_solution_in_m}
\bar{u}^{n+1}_m = \dfrac{1}{h}\left[ (h\, \mathbf{d}+k \, \bar{V}^n \mathbf{I}) \, \tilde{u}^{n+1}_{m,l} + (h(\mathbf{I}-\mathbf{d})-k \, \bar{V}^n \mathbf{I}) \, \tilde{u}^{n+1}_{m,r}\right],
\end{equation}
where $\tilde{u}^{n+1}_{m,l}$ and $\tilde{u}^{n+1}_{m,r}$ are given by the equation (\ref{formula_conservation_imposed_nonuniform_mesh}) applied to the two parts in which the $m$-th cell is divided by the travelling discontinuity (see Figure \ref{fig_disc_rec_sec_RS_case_1}), i.e.
\begin{equation}\label{discontinuity_reconstruction_second_riemann_solver_case_i}
\begin{split}
& (h\, \mathbf{d}+\bar{V}^n\,k \, \mathbf{I})\, \tilde{u}^{n+1}_{m,l} = h\, \mathbf{d}\, \hat{u} - k \left[f(\hat{u})-\bar{V}^n \, \hat{u} - F(\bar{u}^n_{m-1},\hat{u})\right] \; \text{ and}\\
& (h(\mathbf{I}-\mathbf{d})-\bar{V}^n \, k \, \mathbf{I})\, \tilde{u}^{n+1}_{m,r} = h(\mathbf{I}-\mathbf{d}) \, \check{u}_2 - k\left[F(\check{u}_2,\bar{u}^n_{m+1})-f(\check{u}_2)+\bar{V}^n \check{u}_2\right].
\end{split}
\end{equation}
\item[(ii)] If $x^{n+1,\rho}_m \geq x_{m+1/2}$ and $x^{n+1,z}_m \geq x_{m+1/2}$, then for the cell $(m-1)$ we set
\begin{equation}\label{disc_rec_RS2_case_ii_cell_before}
\bar{u}^{n+1}_{m-1} = \bar{u}^n_{m-1} -\dfrac{k}{h} \left[F(\bar{u}^n_{m-1},\hat{u})-F(\bar{u}^n_{m-2}-\bar{u}^n_{m-1})\right].
\end{equation}
Applying the formula (\ref{formula_conservation_imposed_nonuniform_mesh}) to the two parts in which the travelling discontinuity divides the cells $m$ and $(m+1)$ (see Figure \ref{fig_disc_rec_sec_RS_case_2}), we define $\tilde{u}^{n+1}_{m}$ and $\tilde{u}^{n+1}_{m+1}$ as the solutions to the equations
\begin{equation}\label{discontinuity_reconstruction_second_riemann_solver_case_ii}
\begin{split}
& (h \, \mathbf{d}+ k \, \bar{V}^n \, \mathbf{I}) \tilde{u}^{n+1}_{m} = h \, \mathbf{d} \, \hat{u} -k \left[ f(\hat{u})-\bar{V}^n \, \hat{u}-F(\bar{u}^n_{m-1},\hat{u})\right] \; \text{ and}\\
& \begin{split}(h(2\mathbf{I}-\mathbf{d}) - k\, \bar{V}^n \, \mathbf{I}) \, \tilde{u}^{n+1}_{m+1} & = h(\mathbf{I}-\mathbf{d}) \, \check{u}_2 + h\, \bar{u}^n_{m+1} + \\
& - k \left[ F(\bar{u}^n_{m+1},\bar{u}^n_{m+2})-f(\check{u}_2)+\bar{V}^n\,  \check{u}_2 \right]
\end{split}
\end{split}
\end{equation}
Then, in order to compute the average solutions $\bar{u}^{n+1}_{m}$ and $\bar{u}^{n+1}_{m+1}$ on the $m$-th and the $(m+1)$-th cell, we make an appropriate combination of $\tilde{u}^{n+1}_{m}$ and $\tilde{u}^{n+1}_{m+1}$, obtaining
\begin{equation}\label{disc_rec_RS2_case_ii_combinations}
\begin{split}
& \bar{u}^{n+1}_m = \tilde{u}^{n+1}_m \; \text{ and}\\
& \bar{u}^{n+1}_{m+1} = \dfrac{1}{h}\left[(h(2\mathbf{I}-\mathbf{d})-k\, \, \bar{V}^n\,  \mathbf{I})\, \tilde{u}^{n+1}_{m+1} + (h(\mathbf{d}-\mathbf{I}) + k\, \bar{V}^n \, \mathbf{I})\, \tilde{u}^{n+1}_m \right].
\end{split}
\end{equation}
\item[(iii)] If at time $t^{n+1}$ one component of the discontinuity is in the $m$-th cell and the other is in the $(m+1)$-th cell, then we apply respectively the case (i) to the former and the case (ii) to the latter.
\end{enumerate}
For every $j \notin \lbrace m-1, \, m, \, m+1 \rbrace$, we apply the standard Godunov's method to compute the solution $\bar{u}^{n+1}_j$ at the new time step.\\

\begin{figure}
\centering
\begin{subfigure}[h]{0.75\linewidth}
\centering
\definecolor{xdxdff}{rgb}{0.49019607843137253,0.49019607843137253,1.}
\definecolor{qqqqff}{rgb}{0.,0.,1.}
\begin{tikzpicture}[scale=0.45,line cap=round,line join=round,>=triangle 45,x=1.0cm,y=1.0cm]
\clip(-4.,-2.) rectangle (16.,7.);
\draw [domain=-4.:16.] plot(\x,{(--70.-0.*\x)/14.});
\draw [domain=-4.:16.] plot(\x,{(-0.-0.*\x)/14.});
\draw (0.,5.)-- (0.,0.);
\draw (7.,5.)-- (7.,0.);
\draw (14.,5.)-- (14.,0.);
\draw (2.2,0.)-- (3.82,5.);
\draw (-2,0.) node[anchor=north west] {$x_{m-1/2}$};
\draw (5.5,0.) node[anchor=north west] {$x_{m+1/2}$};
\draw (12.5,0.) node[anchor=north west] {$x_{m+3/2}$};
\draw (2.57,-0.68) node[anchor=north west] {$C_m$};
\draw (9.70,-0.50) node[anchor=north west] {$C_{m+1}$};
\draw (3.07,3.30) node[anchor=north west] {$\bar{V}^n$};
\draw (1.2,0.) node[anchor=north west] {$x^{n,\rho}_m$};
\draw (0.90,5.2) node[anchor=north west] {$\tilde{\rho}^{n+1}_{m,l}$};
\draw (4.67,5.20) node[anchor=north west] {$\tilde{\rho}^{n+1}_{m,r}$};
\draw (9.19,5.2) node[anchor=north west] {$\bar{\rho}^{n+1}_{m+1}$};
\draw (4.52,1.30) node[anchor=north west] {$\check{\rho}_2$};
\draw (0.8,1.23) node[anchor=north west] {$\hat{\rho}$};
\draw (-2.56,1.49) node[anchor=north west] {$\bar{\rho}^n_{m-1}$};
\draw (-2.78,5.2) node[anchor=north west] {$\bar{\rho}^{n+1}_{m-1}$};
\draw (9.70,1.41) node[anchor=north west] {$\bar{\rho}^n_{m+1}$};
\draw (2.93,6.55) node[anchor=north west] {$x^{n+1,\rho}_m$};
\begin{scriptsize}
\draw [fill=qqqqff] (0.,0.) circle (3pt);
\draw [fill=qqqqff] (14.,0.) circle (3pt);
\draw [fill=qqqqff] (7.,0.) circle (3pt);
\draw [fill=qqqqff] (0.,5.) circle (3pt);
\draw [fill=qqqqff] (7.,5.) circle (3pt);
\draw [fill=qqqqff] (14.,5.) circle (3pt);
\draw [fill=xdxdff] (2.2,0.) circle (3pt);
\draw [fill=xdxdff] (3.82,5.) circle (3pt);
\end{scriptsize}
\end{tikzpicture}
\caption{Case $x^{n+1,\rho}_m < x_{m+1/2}$ for the $\rho$ component. The values $\tilde{\rho}^{n+1}_{m,l}$ and $\tilde{\rho}^{n+1}_{m,r}$, respectively on the left and the right side of the discontinuity at time $t^{n+1}$, are computed imposing the conservation of the solution on both sides of the discontinuity. The approximate solution at time $t^{n+1}$ in the $m$-th cell is obtained averaging $\tilde{\rho}^{n+1}_{m,l}$ and $\tilde{\rho}^{n+1}_{m,r}$.}\label{fig_disc_rec_sec_RS_case_1}
\end{subfigure}
\\
\begin{subfigure}[h]{0.75\linewidth}
\centering
\definecolor{xdxdff}{rgb}{0.49019607843137253,0.49019607843137253,1.}
\definecolor{qqqqff}{rgb}{0.,0.,1.}
\begin{tikzpicture}[scale =0.45,line cap=round,line join=round,>=triangle 45,x=1.0cm,y=1.0cm]
\clip(-4.,-2.) rectangle (15.5,7.);
\draw [domain=-4.:15.5] plot(\x,{(--70.-0.*\x)/14.});
\draw [domain=-4.:15.5] plot(\x,{(-0.-0.*\x)/14.});
\draw (0.,5.)-- (0.,0.);
\draw (7.,5.)-- (7.,0.);
\draw (14.,5.)-- (14.,0.);
\draw (3.64,0.)-- (9.,5.);
\draw (-2,0.) node[anchor=north west] {$x_{m-1/2}$};
\draw (5.,0.) node[anchor=north west] {$x_{m+1/2}$};
\draw (12,0.) node[anchor=north west] {$x_{m+3/2}$};
\draw (1.5,-0.8) node[anchor=north west] {$C_m$};
\draw (9.99,-0.8) node[anchor=north west] {$C_{m+1}$};
\draw (4.5,3.14) node[anchor=north west] {$\bar{V}^n$};
\draw (2.5,0.) node[anchor=north west] {$x^{n,\rho}_m$};
\draw (2.44,5.1) node[anchor=north west] {$\tilde{\rho}^{n+1}_{m}$};
\draw (10.66,5.1) node[anchor=north west] {$\tilde{\rho}^{n+1}_{m+1}$};
\draw (5.12,1.26) node[anchor=north west] {$\check{\rho}_2$};
\draw (1.36,1.26) node[anchor=north west] {$\hat{\rho}$};
\draw (-3.35,1.23) node[anchor=north west] {$\bar{\rho}^n_{m-1}$};
\draw (-3.41,5.1) node[anchor=north west] {$\bar{\rho}^{n+1}_{m-1}$};
\draw (10.22,1.26) node[anchor=north west] {$\bar{\rho}^n_{m+1}$};
\draw (8.18,6.6) node[anchor=north west] {$x^{n+1,\rho}_m$};
\begin{scriptsize}
\draw [fill=qqqqff] (0.,0.) circle (3pt);
\draw [fill=qqqqff] (14.,0.) circle (3pt);
\draw [fill=qqqqff] (7.,0.) circle (3pt);
\draw [fill=qqqqff] (0.,5.) circle (3pt);
\draw [fill=qqqqff] (7.,5.) circle (3pt);
\draw [fill=qqqqff] (14.,5.) circle (3pt);
\draw [fill=xdxdff] (3.64,0.) circle (3pt);
\draw [fill=xdxdff] (9.,5.) circle (3pt);
\end{scriptsize}
\end{tikzpicture}
\caption{Case $x^{n+1,\rho}_m \geq x_{m+1/2}$. The values $\tilde{\rho}^{n+1}_{m}$ and $\tilde{\rho}^{n+1}_{m+1}$, respectively on the left and the right side of the discontinuity at time $t^{n+1}$, are computed imposing the conservation of the solution on both sides of the discontinuity. The approximate solution at time $t^{n+1}$ in the $m$-th and in the $(m+1)$-th cells are obtained with an appropriate combination of $\tilde{\rho}^{n+1}_{m}$ and $\tilde{\rho}^{n+1}_{m+1}$.}\label{fig_disc_rec_sec_RS_case_2}
\end{subfigure}
\caption{Representation of the reconstruction method. Analogous notations are used for the $z$ component.}
\end{figure}
The next proposition states that if the initial datum is a non-classical shock, then the solution given by the discontinuity reconstruction method is the non-classical shock itself.
\begin{prop}\label{prop_disc_rec_RS2_non_classical_shock}
Fix $n\in \mathbb{N}$ and let us suppose that $y^n = x_{m-1/2}$, where $y^n$ is the bus initial position. Let us consider the Riemann problem
\begin{equation}\label{RP_non_classical_shock}
\begin{cases}
\partial_t \, u + \partial_x \, [f(u)] = 0,\\
u(t^n,x) = \begin{cases}
\hat{u} & \text{if } x\leq y^n,\\
\check{u}_2 & \text{if } x>y^n.
\end{cases}
\end{cases}
\end{equation}
 If there exists $\gamma\in [0,1]$ such that
$$\bar{u}^n_m = \gamma\, \bar{u}^n_{m-1}+(1-\gamma) \, \bar{u}^n_{m+1} = \gamma \, \hat{u}+(1-\gamma) \, \check{u}_2,$$
then we have
$$d^{n,\rho}_m = d^{n,z}_m = \gamma.$$
Moreover at time $t^{n+1}$ the solution given by the discontinuity reconstruction procedure for $\mathcal{RS}^\alpha_2$ to (\ref{RP_non_classical_shock}) is
\begin{equation}
u(t^{n+1},x) = \begin{cases}
\hat{u} & \text{if } x \leq x_{j-1/2},\\
\tilde{u} & \text{if } x_{j-1/2}< x < x_{j+1/2},\\
\check{u}_2 & \text{if } x\geq x_{j+1/2},
\end{cases}
\end{equation}
where 
\begin{equation*}
j = \begin{cases}
m & \text{if } \; x^{n+1,\rho}_m<x_{m+1/2} \; \text{ and } \; x^{n+1,z}_m < x_{m+1/2},\\
m+1 & \text{otherwise},
\end{cases}
\end{equation*}
and where $\tilde{u}$ is a convex combination of $\hat{u}$ and $\check{u}_2$; see Figure \ref{fig_nonclassical_shock_RS_2}.
\end{prop}
\begin{proof}
By the definition (\ref{costanti_ricostruzione_discontinuita_RS2}) and the hypothesis, we have
\begin{equation*}
\begin{split}
d^{n,\rho}_m & = \dfrac{\tilde{\rho}-\check{\rho}_2}{\hat{\rho}-\check{\rho}_2}= \dfrac{\gamma \, \hat{\rho}+(1-\gamma) \, \check{\rho}_2-\check{\rho}_2}{\hat{\rho}-\check{\rho}_2} =\\
& = \dfrac{\gamma\, \hat{\rho}-\gamma \, \check{\rho}_2}{\hat{\rho}-\check{\rho}_2} = \gamma,
\end{split}
\end{equation*}
where $\tilde{\rho}$ is the $\rho$ component of $\tilde{u}$. Similarly for $d^{n,z}_m$.\\
In order to apply the Godunov's method, we want a piecewise constant initial datum. Therefore, we compute the average value of the solution in the cell $C_m$ obtaining
$$\bar{u}^n_m = \dfrac{1}{h}\left[(x_{m+1/2}-y^n) \, \check{u}_2 + (y^n-x_{m-1/2}) \, \hat{u} \right].$$
This is a convex combination of $\hat{u}$ and $\check{u}_2$. Therefore the first part of the proposition holds and we have
\begin{equation}\label{position_of_the_discontinuity_proof}
d^{n,\rho}_m = d^{n,z}_m = \gamma = \dfrac{y^n-x_{m-1/2}}{h}.
\end{equation}
If $x^{n+1,\rho}_m < x_{m+1/2}$ and $x^{n+1,z}_m < x_{m+1/2}$, then by the equations (\ref{disc_rec_RS2_case_i_cells_before_and_after}), we have
\begin{equation*}
\begin{split}
& \bar{u}^{n+1}_{m-1} = \hat{u}-\dfrac{k}{h}\left[F(\hat{u},\hat{u})-F(\hat{u},\hat{u})\right] = \hat{u} \; \text{ and}\\
& \bar{u}^{n+1}_{m+1} = \check{u}_2 - \dfrac{k}{h}\left[F(\check{u}_2,\check{u}_2)-F(\check{u}_2,\check{u}_2)\right] = \check{u}_2,
\end{split}
\end{equation*}
because $\bar{u}^n_j = \hat{u}$ for every $j \leq m-1$ and $\bar{u}^n_j = \check{u}_2$ for every $j \geq m+1$.\\
For the cell $m$, by the formula (\ref{discontinuity_reconstruction_second_riemann_solver_case_i}), we have
\begin{equation*}
\begin{split}
& \tilde{u}^{n+1}_{m,l} = \dfrac{h\, \gamma\,\hat{u} - k\left[f(\hat{u})-\bar{V}^n\, \hat{u}-F(\hat{u},\hat{u}) \right]}{h\, \gamma + \bar{V}^n \, k} = \hat{u} \; \text{ and}\\
& \tilde{u}^{n+1}_{m,r} = \dfrac{h(1-\gamma)\check{u}_2-k\,\left[F(\check{u}_2,\check{u}_2)-f(\check{u}_2)+\bar{V}^n\,\check{u}_2) \right]}{h(1-\gamma)-\bar{V}^n\,k}=\check{u}_2.
\end{split}
\end{equation*}
Hence the new solution in the $m$-th cell, by the equation (\ref{discontinuity_recostruction_second_r_s_case_i_new_solution_in_m}), is
$$\bar{u}^{n+1}_m = \dfrac{1}{h}\left[(h \, \gamma+k\bar{V}^n)\hat{u} + (h(1-\gamma)-k\,\bar{V}^n)\check{u}_2 \right]=\tilde{u}.$$
Since
$$\dfrac{h \, \gamma + k\, \bar{V}^n}{h}+ \dfrac{h(1-\gamma) -k\, \bar{V}^n}{h} = 1,$$
$\tilde{u}$ is a convex combination of $\hat{u}$ and $\check{u}_2$.\\
If $x^{n+1,\rho}_m \geq x_{m+1/2}$ and $x^{n+1,z}_m \geq x_{m+1/2}$, then by the formula (\ref{disc_rec_RS2_case_ii_cell_before}), we find
$$\bar{u}^{n+1}_{m-1} = \hat{u}-\dfrac{k}{h}\left[ F(\hat{u},\hat{u})-F(\hat{u},\hat{u})\right] = \hat{u}.$$
Applying the equations (\ref{discontinuity_reconstruction_second_riemann_solver_case_ii}), we obtain
\begin{equation*}
\begin{split}
& \tilde{u}^{n+1}_m = \dfrac{h\,\gamma\, \hat{u}-k\left[f(\hat{u})-\bar{V}^n \, \hat{u}-F(\hat{u},\hat{u})\right]}{h\,\gamma + k\, \bar{V}^n}= \hat{u} \; \text{ and}\\
& \tilde{u}^{n+1}_{m+1} = \dfrac{h(1-\gamma)\check{u}_2+h\, \check{u}_2- k \left[F(\check{u}_2,\check{u}_2)-f(\check{u}_2)+\bar{V}^n\, \check{u}_2 \right]}{h(2-\gamma)-k\, \bar{V}^n} = \check{u}_2.
\end{split}
\end{equation*}
Therefore, by (\ref{disc_rec_RS2_case_ii_combinations}), we have
\begin{equation*}
\begin{split}
& \bar{u}^{n+1}_m = \tilde{u}^{n+1}_m = \hat{u} \; \text{ and}\\
& \bar{u}^{n+1}_{m+1} = \dfrac{1}{h}\left[(h(2-\gamma)-k\, \bar{V}^n)\check{u}_2 + (h(\gamma-1)+k\,\bar{V}^n)\tilde{u}^{n+1}_m \right] = \tilde{u}.
\end{split}
\end{equation*}
Since
$$\dfrac{h(2-\gamma)-k\, \bar{V}^n}{h}+\dfrac{h(\gamma-1)+ k \, \bar{V}^n}{h} = 1,$$
$\tilde{u}$ is a convex combination of $\hat{u}$ and $\check{u}_2$.
\end{proof}
\begin{figure}[hbt]
\centering
\includegraphics[width=0.95\linewidth]{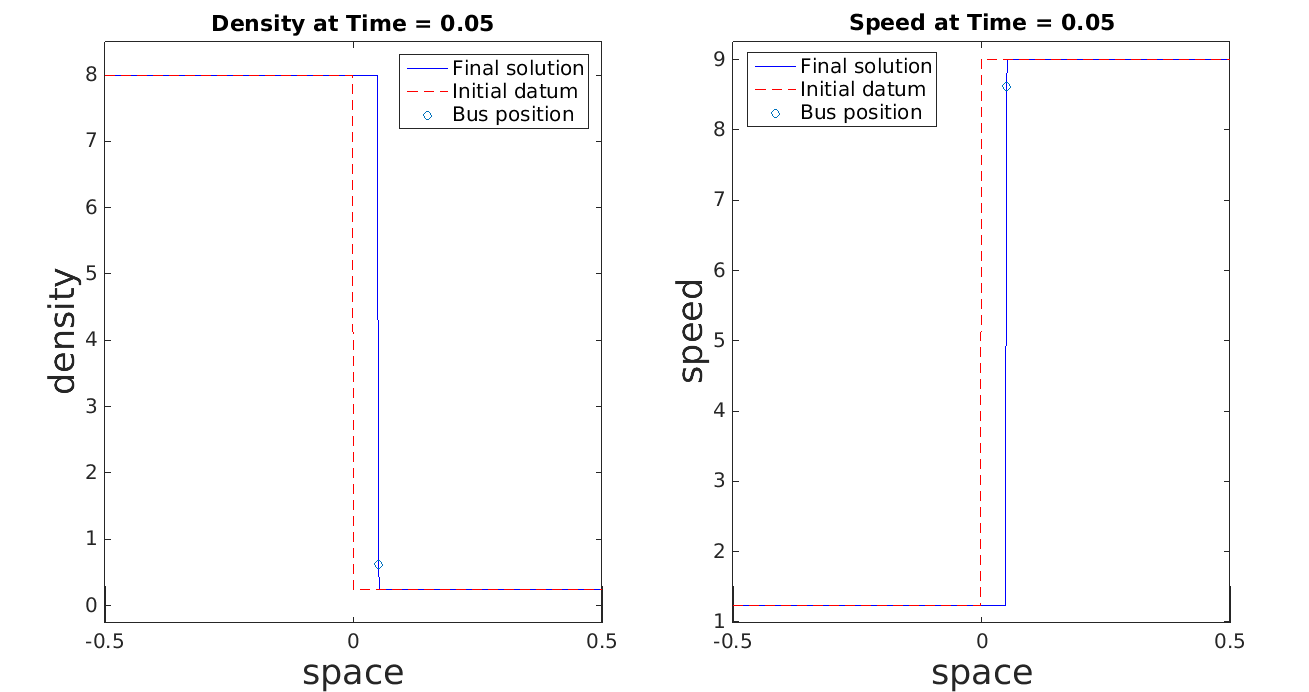}
\caption{Reconstruction of the non-classical shock with initial datum $(\rho^l,v^l)=(8,F_\alpha/8+V_b)$ and $(\rho^r,v^r)=(F_\alpha/8,V_b+8)$ obtained with the discontinuity reconstruction method for $\mathcal{RS}^\alpha_2$. By Remark \ref{non_classical_shock_for_RS_alfa_1}, this initial datum is a non-classical shock: $(\rho^l,v^l)=(\hat{\rho},\hat{v})$ and $(\rho^r,v^r)=(\check{\rho}_2,\check{v}_2)$ (in this particular case we have $(\check{\rho}_1,\check{v}_1)=(\check{\rho}_2,\check{v}_2)$). The other parameters are $R_{\max}=15$, $V_b=1$, $y_0=0$ and $\alpha=0.25$. The pressure function is $p(\rho)=\rho)$.\\
Note that the value of the density and of the velocity in the bus cell, is an average of the left and right state.}\label{fig_nonclassical_shock_RS_2}
\end{figure}
\begin{remark}
If we do not require $y^n = x_{m-1/2}$, whenever $y^n > x^{\rho,n}_m = x^{z,n}_m$ the bus crosses the right interface of the $m$-th cell before the discontinuity. This introduces an error in the solution, because at time $t^{n+1}$ in the $m$-th cell we should have $\hat{u}$, while we will have a convex combination of $\hat{u}$ and $\check{u}_2$.\\
On the contrary, if $y^n = x_{m-1/2}$, then cannot happen that at time $t^{n+1}$ the bus is in the $m$-th cell and the discontinuity in the $(m+1)$-th cell or vice versa. Indeed by the equation (\ref{position_of_the_discontinuity_proof}), we find that the discontinuity at time $t^n$ is in
$$x^{n,\rho}_m = x^{n,\rho}_z = x_{m-1/2}+h\, \dfrac{y^n-x_{m-1/2}}{h} = y^n.$$
This fact and the Proposition \ref{prop_disc_rec_RS2_non_classical_shock}, ensure that a non-classical shock is captured exactly with the reconstruction procedure, except for the cell in which the bus is at time $t^{n+1}$, where we have a convex combination of $\hat{u}$ and $\check{u}_2$.
\end{remark}
For a more general initial datum the right trace of the non-classical shock obtained with the discontinuity reconstruction procedure is overestimated; see Figure \ref{fig_example_overestimation_RS_2_conservation_imposed}.\\
This error is due to the fact that the program does not capture the exact right state, but it takes a point $(\tilde{\rho},\tilde{v})$ with a lower density and momentum and correspondingly, a higher velocity.\\
Moreover the Godunov's method is not effective to capture contact discontinuities.
\begin{figure}
\centering
\begin{subfigure}[h]{0.95\linewidth}
\includegraphics[width=\textwidth]{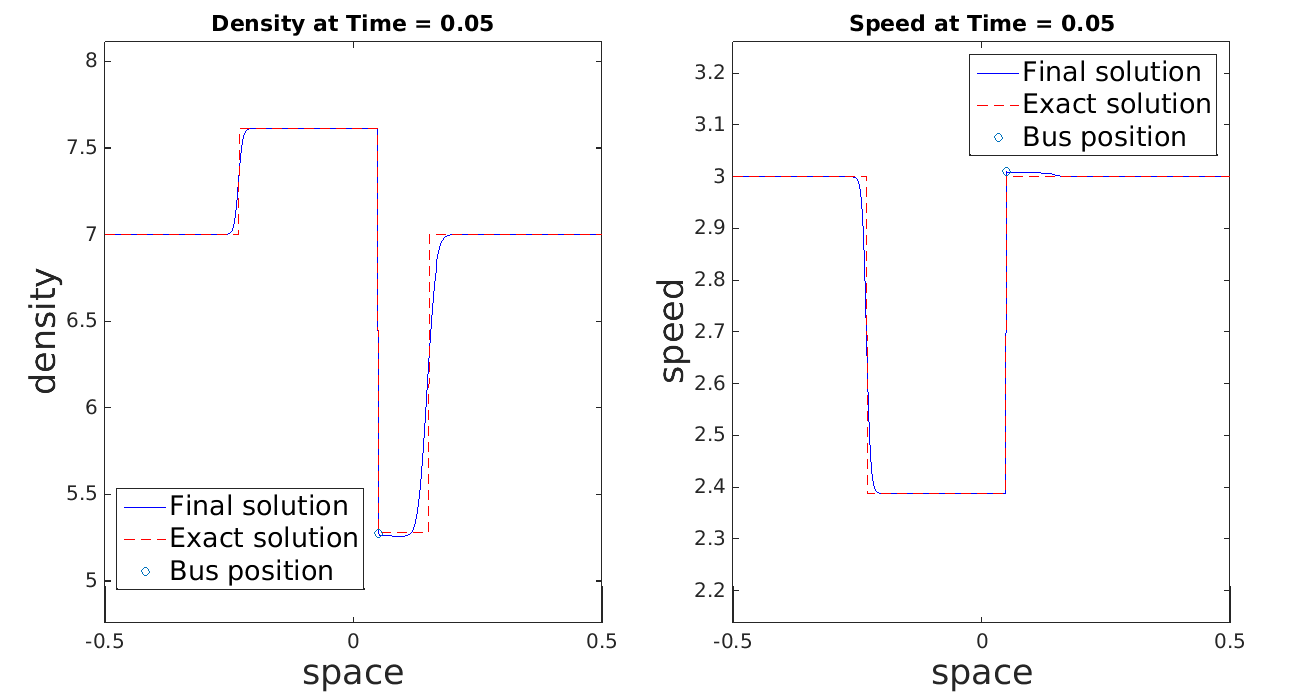}
\caption{The initial datum is constant $(\rho,v)(0,x) = (7,3)$, $\alpha= 0.5$, $y_0=0$ and the bus maximal speed is $V_b = 1$.}
\end{subfigure}
\\
\begin{subfigure}[h]{0.95\linewidth}
\includegraphics[width=\textwidth]{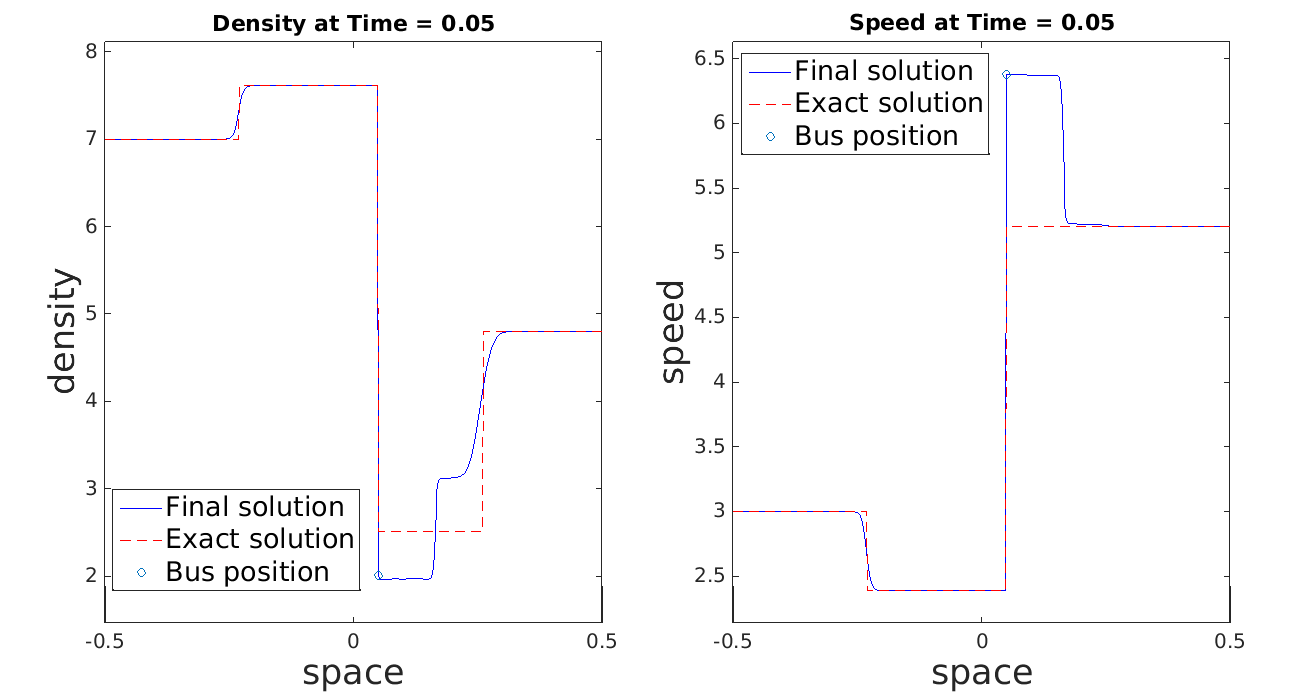}
\caption{The initial datum is $(\rho,v)(0,x) = (7,3)$ if $x<0$ and $(\rho,v) = (4.8,5.2)$ if $x>0$, $\alpha= 0.5$, $y_0=0$ and the bus maximal speed is $V_b = 1$.}
\end{subfigure}
\caption{Numerical solutions obtained with the reconstruction procedure proposed for $\mathcal{RS}^\alpha_2$ for $p(\rho) = \rho$. We can see that the right trace of the non-classical shock is overstimated.}\label{fig_example_overestimation_RS_2_conservation_imposed}
\end{figure}

\subsubsection{Discontinuity reconstruction with fixed value}
The right trace of the non-classical shock obtained with the discontinuity reconstruction procedure is overestimated; see Figure \ref{fig_example_overestimation_RS_2_conservation_imposed}.
Therefore, following \cite{garavello_goatin}, we propose to correct the method as follows.
\begin{enumerate}
\item[(i)] If $x^{n+1,z}_m < x_{m+1/2}$, then we fix at $\check{v}_2$ the value of the velocity in the sector $[x^{n,z}_m+s\, \bar{V}^n,x_{m+1/2})$ for $s\in [0,k]$. In particular for $s=k$ we have
$$\bar{v}^{n+1}_{m,r} = \bar{v}^n_{m+1}$$
and we modify the second component of $\tilde{u}^{n+1}_{m,r}$ as
$$\tilde{z}^{n+1}_{m,r} = \begin{cases}
\tilde{\rho}^{n+1}_{m,r}(\bar{v}^n_{m+1}+p(\tilde{\rho}^{n+1}_{m,r})) & \text{if } x^{n+1,m}_\rho < x_{m+1/2},\\
\tilde{\rho}^{n+1}_{m}(\bar{v}^n_{m+1}+p(\tilde{\rho}^{n+1}_{m})) & \text{if } x^{n+1,m}_\rho \geq x_{m+1/2},
\end{cases}$$
where $\tilde{\rho}^{n+1}_{m,r}$ and $\tilde{\rho}^{n+1}_{m}$ are the first component respectively of the vector defined in (\ref{discontinuity_reconstruction_second_riemann_solver_case_i}) and (\ref{discontinuity_reconstruction_second_riemann_solver_case_ii}). Then we use this value in the equation (\ref{discontinuity_recostruction_second_r_s_case_i_new_solution_in_m}).
\item[(ii)] If $x^{n+1,z}_m \geq x_{m+1/2}$, then we fix the value of the velocity in the cell $(m+1)$, i.e.
$$\bar{v}^{n+1}_{m+1} = \bar{v}^n_{m+1}$$
and we update the second component of $\tilde{u}^{n+1}_{m+1}$ as
$$\tilde{z}^{n+1}_{m+1} = \begin{cases}
\bar{\rho}^{n+1}_{m+1}(\bar{v}^n_{m+1}+p(\bar{\rho}^{n+1}_{m+1})) & \text{if } x^{n+1,m}_\rho < x_{m+1/2},\\
\tilde{\rho}^{n+1}_{m+1}(\bar{v}^n_{m+1}+p(\tilde{\rho}^{n+1}_{m+1})) & \text{if } x^{n+1,m}_\rho \geq x_{m+1/2},
\end{cases}$$
where $\tilde{\rho}^{n+1}_{m+1}$ is the first component of the vector defined in (\ref{discontinuity_reconstruction_second_riemann_solver_case_ii}).
\end{enumerate}
Unfortunately even with these corrections we do not obtain the desired results (see Figure \ref{fig_example_overestimation_RS_2_fixed_value}), although in some cases the results are better than the ones obtained with the discontinuity reconstruction procedure (see Figure \ref{fig_comparison_RS_2_fixed_value_discontinuity reconstruction}).

\begin{figure}
\centering
\begin{subfigure}[h]{0.88\linewidth}
\includegraphics[width=\textwidth]{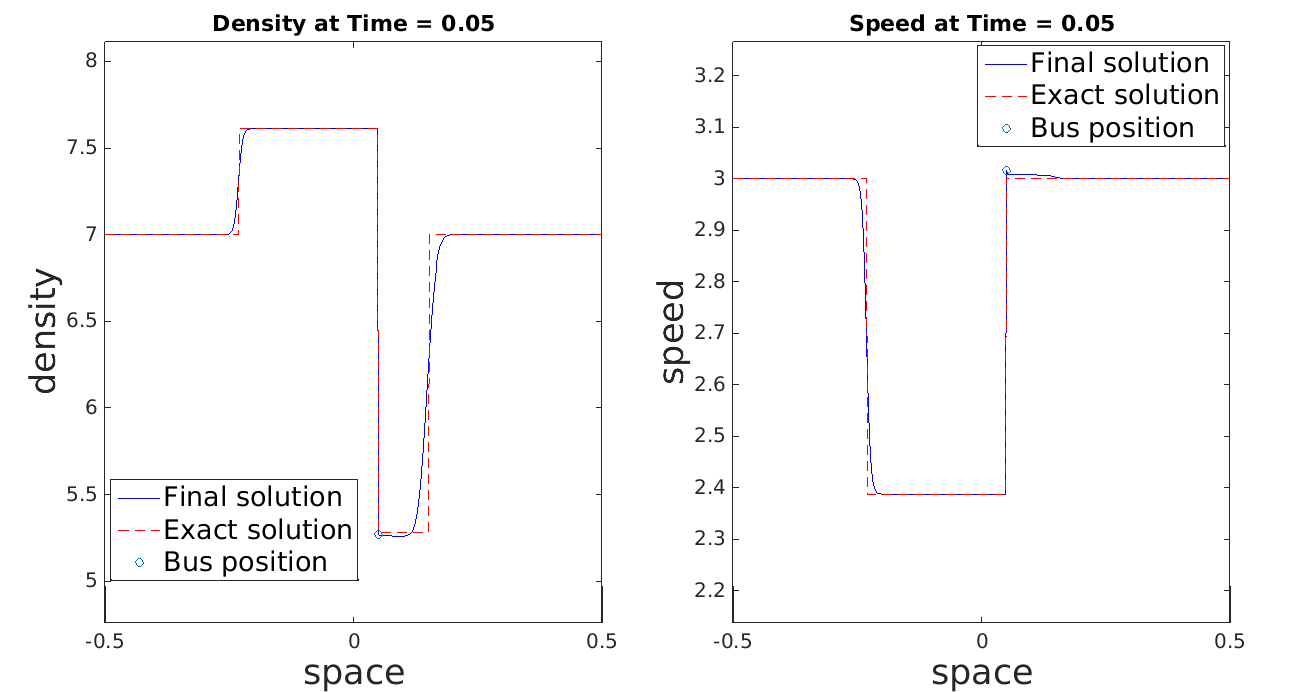}
\caption{The initial datum is constant $(\rho,v)(0,x) = (7,3)$, $F_\alpha = 10.5625$ and the bus maximal speed is $V_b = 1$.}
\end{subfigure}
\\
\begin{subfigure}[h]{0.88\linewidth}
\includegraphics[width=\textwidth]{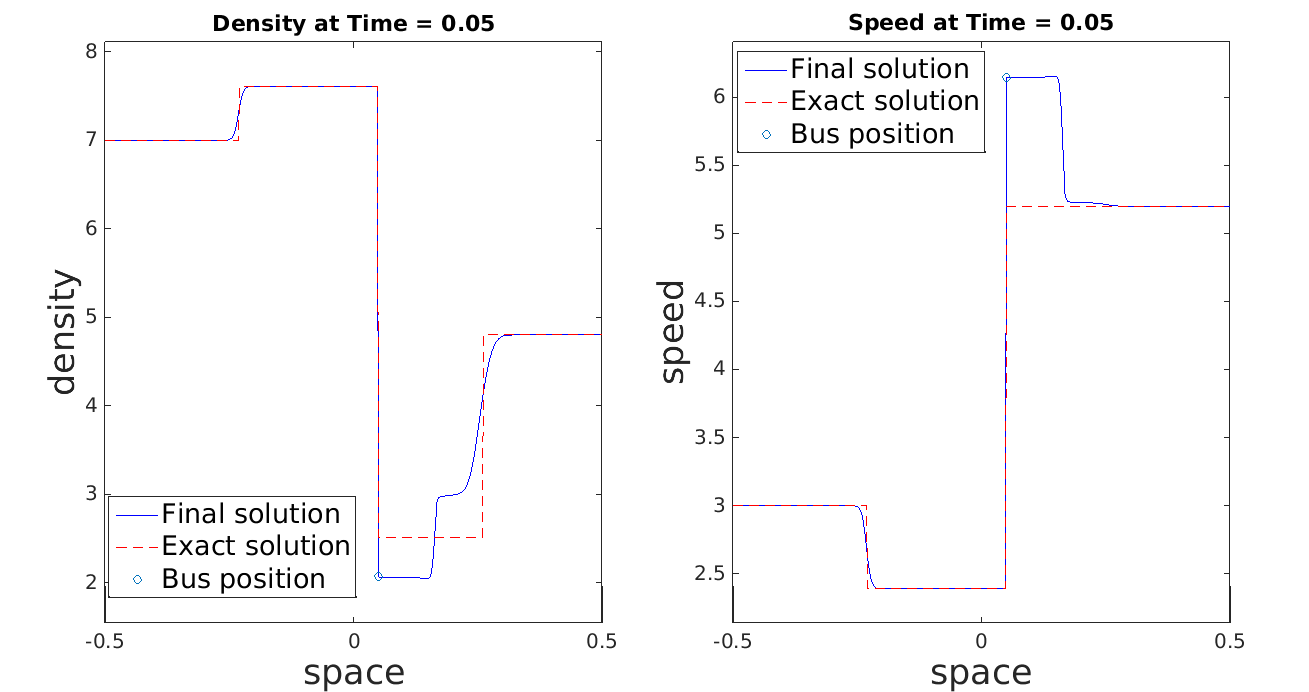}
\caption{The initial datum is $(\rho,v)(0,x) = (7,3)$ if $x<0$ and $(\rho,v) = (4.8,5.2)$ if $x>0$, $F_\alpha = 10.5625$ and the bus maximal speed is $V_b = 1$.}
\end{subfigure}
\caption{Solution obtained with the fixed value method proposed for $\mathcal{RS}^\alpha_2$: the value of the velocity on the right side of the non-classical shock has been fixed to $\check{v}_2$. We can see that the right trace of the non-classical shock is overestimated.}\label{fig_example_overestimation_RS_2_fixed_value}
\end{figure}

\begin{figure}
\centering
\begin{subfigure}[h]{1\linewidth}
\includegraphics[width=\textwidth]{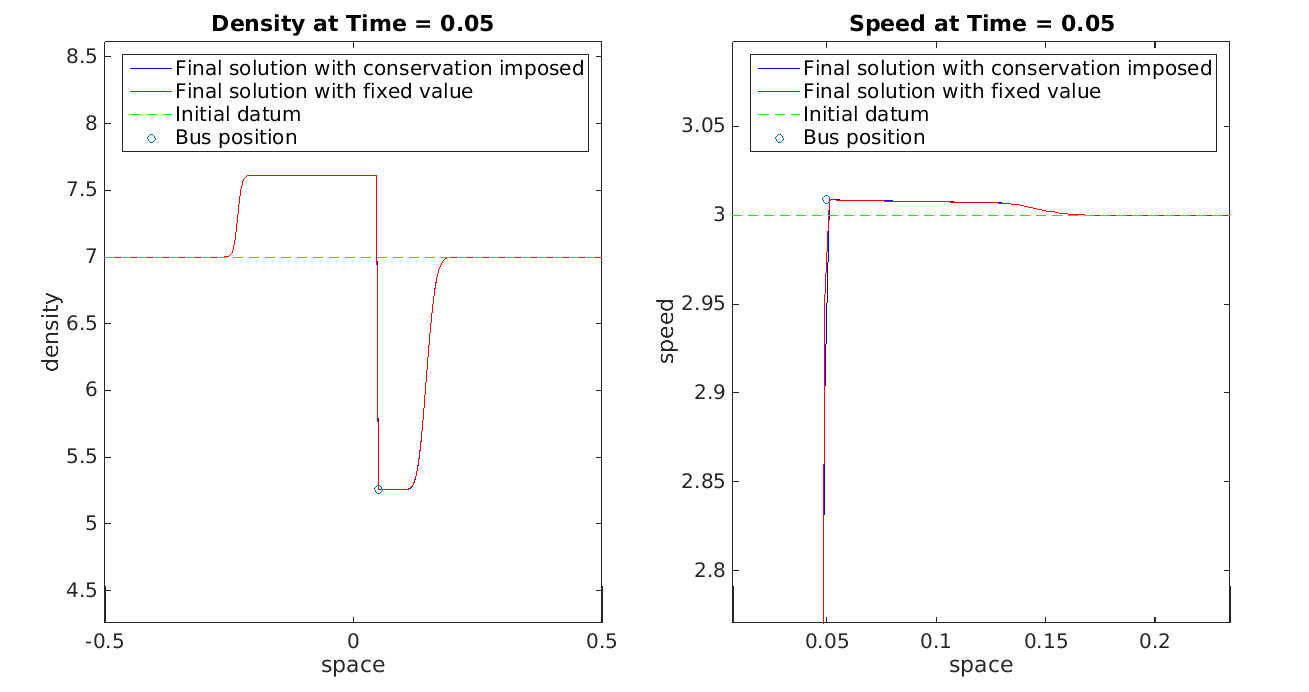}
\caption{The initial datum is constant $(\rho,v)(0,x) = (7,3)$, $F_\alpha = 10.5625$ and the bus maximal speed is $V_b = 1$.}
\end{subfigure}
\\
\begin{subfigure}[h]{1\linewidth}
\includegraphics[width=\textwidth]{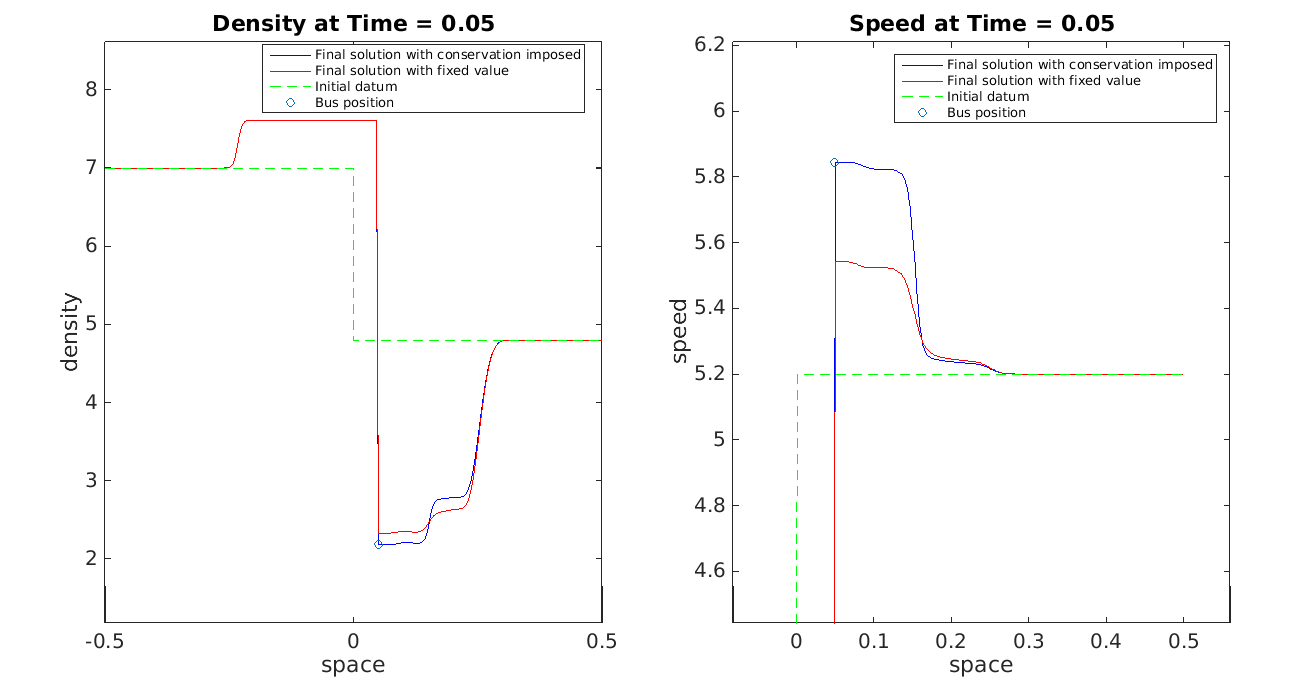}
\caption{The initial datum is $(\rho,v)(0,x) = (7,3)$ if $x<0$ and $(\rho,v) = (4.8,5.2)$ if $x>0$, $F_\alpha = 10.5625$ and the bus maximal speed is $V_b = 1$.}
\end{subfigure}
\caption{Comparison between the solutions obtained with the reconstruction procedure and the fixed value method for $p(\rho) = \rho$. In Figure (a) we see that the two solutions are almost equivalent. In Figure (b) the solution given by the fixed value method is better than the one obtained with the reconstruction procedure, although both the numerical solutions are clearly wrong, because the velocity after the non-classical shock should be constant ($v=v^r=5.2$).}\label{fig_comparison_RS_2_fixed_value_discontinuity reconstruction}
\end{figure}
\subsection{Second method: non-uniform mesh}
The second method is based on a non-uniform mesh: we shift the grid points locally around the bus. See \cite{delle_monache_goatin_mesh_mobile} for the scalar case.\\
Fix $n\in \mathbb{N}$. Let $\lbrace x^n_{j+1/2}\rbrace_{j\in \mathbb{Z}}$ be the mesh at time $t^n$. For every $j\in \mathbb{Z}$, the $j$-th cell is $C^n_j= [x^n_{j-1/2},x^n_{j+1/2})$ and its length is
$$h_j^n= x^n_{j+1/2}-x^n_{j-1/2}.$$
Let $\lbrace k^n \rbrace_{n\in \mathbb{N}}$ be the sequence of time increments, so that
$$t^{n+1}=t^n+k^n \; \text{ for every } \; n=1,2,....$$
These quantities can change at each time step.\\
Let $y^n=y(t^n)$ and $\bar{V}^n$ be respectively the bus position and the bus speed at time $t^n$ and let $m\in \mathbb{Z}$ be such that $y^n \in C^n_m$.\\
Let us suppose that at time $t^0$ the mesh is uniform, namely
$$x^0_{j+1/2}-x^0_{j-1/2} = h^0_j = h \; \text{ for every } \; j \in \mathbb{Z}.$$
\begin{remark}
We choose the initial mesh so that there exists $m \in \mathbb{Z}$ such that $y^0 = x^0_{m-1/2}$. The reason for this choice will be clarified later.
\end{remark}
The idea is to modify always only two cells near the bus and to restore the initial mesh far from the bus.\\
We distinguish two cases.
\begin{enumerate}
\item[(i)] If $ x^n_{m+1/2}-y^n > h/2$, then we introduce
\begin{equation}
x^\text{new}_{m-1/2} = y^n.
\end{equation}
Let us call
$$h^\text{new}_{m}:=x^{n}_{m+1/2}-x^\text{new}_{m-1/2}\; \text{ and } \; h^\text{new}_{m-1}= x^\text{new}_{m-1/2}-x^n_{m-3/2}$$
respectively the length of the new $m$-th cell and of the new $(m-1)$-th cell.
\begin{remark}
We have to adapt $k^n$ to the length of the new cells.
\end{remark}
Since we have modified the mesh, we have to recompute the average solution in the cells $C^\text{new}_{m-1}$ and $C^\text{new}_m$; see Figure \ref{fig_nonunif_mesh_case_a}. For the former, we find
\begin{equation*}
\bar{u}^\text{new}_{m-1} = \dfrac{1}{h^\text{new}_{m-1}}\left[(x^\text{new}_{m-1/2}-x^n_{m-1/2})\, \bar{u}^n_{m}+(x^n_{m-1/2}-x^n_{m-3/2}) \, \bar{u}^n_{m-1}  \right].
\end{equation*}
The average on the $m$-th cell remains unchanged, i.e.
$$\bar{u}^\text{new}_m = \bar{u}^n_m.$$
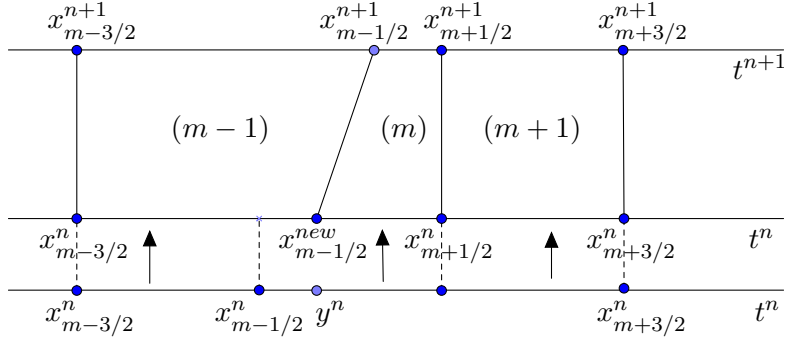
\begin{figure}[hbtp]
\centering
\definecolor{xdxdff}{rgb}{0.49019607843137253,0.49019607843137253,1.}
\definecolor{qqqqff}{rgb}{0.,0.,1.}
\begin{tikzpicture}[scale =0.6,line cap=round,line join=round,>=triangle 45,x=1.0cm,y=1.0cm]
\clip(-1.5,3.) rectangle (15.5,11.);
\draw [domain=-1.5:15.5] plot(\x,{(--70.69-0.*\x)/16.});
\draw [domain=-1.5:15.5] plot(\x,{(--96.-0.*\x)/16.});
\draw [domain=-1.5:15.5] plot(\x,{(--155.63-0.*\x)/16.});
\draw (5.01,4.4) node[anchor=north west] {$y^n$};
\draw (2.8,4.4) node[anchor=north west] {$x^n_{m-1/2}$};
\draw (-0.92,4.4) node[anchor=north west] {$x^n_{m-3/2}$};
\draw (11.18,4.4) node[anchor=north west] {$x^n_{m+3/2}$};
\draw [->] (1.6,4.57) -- (1.6,5.75);
\draw [->] (6.72,4.61) -- (6.70,5.77);
\draw [->] (10.41,4.68) -- (10.41,5.72);
\draw (14.64,4.4) node[anchor=north west] {$t^n$};
\draw (14.51,6.) node[anchor=north west] {$t^n$};
\draw (14.16,9.85) node[anchor=north west] {$t^{n+1}$};
\draw (-1.05,6.) node[anchor=north west] {$x^n_{m-3/2}$};
\draw (4.2,6.) node[anchor=north west] {$x^{new}_{m-1/2}$};
\draw (6.95,6.) node[anchor=north west] {$x^n_{m+1/2}$};
\draw (10.95,6.) node[anchor=north west] {$x^n_{m+3/2}$};
\draw (5.26,6.)-- (6.52,9.72);
\draw (8.,6.)-- (8.,9.72);
\draw (12.,6.)-- (11.98,9.72);
\draw (0.,6.)-- (0.,9.72);
\draw (1.84,8.5) node[anchor=north west] {$(m-1)$};
\draw (6.42,8.5) node[anchor=north west] {$(m)$};
\draw (8.68,8.5) node[anchor=north west] {$(m+1)$};
\draw (-0.87,11) node[anchor=north west] {$x^{n+1}_{m-3/2}$};
\draw (5.1,11) node[anchor=north west] {$x^{n+1}_{m-1/2}$};
\draw (7.38,11) node[anchor=north west] {$x^{n+1}_{m+1/2}$};
\draw (11.21,11) node[anchor=north west] {$x^{n+1}_{m+3/2}$};
\draw [dash pattern=on 2pt off 2pt] (0.,4.41)-- (0.,6.);
\draw [dash pattern=on 2pt off 2pt] (4.,4.42)-- (4.,6.);
\draw [dash pattern=on 2pt off 2pt] (8.,4.41)-- (8.,6.);
\draw [dash pattern=on 2pt off 2pt] (12.,4.46)-- (12.,6.);
\begin{scriptsize}
\draw [fill=qqqqff] (0.,4.41) circle (3pt);
\draw [fill=qqqqff] (16.,4.41) circle (3pt);
\draw [fill=qqqqff] (4.,4.42) circle (3pt);
\draw [fill=qqqqff] (8.,4.41) circle (3pt);
\draw [fill=qqqqff] (12.,4.46) circle (3pt);
\draw [fill=qqqqff] (0.,6.) circle (3pt);
\draw [fill=qqqqff] (5.26,6.) circle (3pt);
\draw [fill=qqqqff] (8.,6.) circle (3pt);
\draw [fill=qqqqff] (12.,6.) circle (3pt);
\draw [fill=qqqqff] (16.,6.) circle (3pt);
\draw [fill=qqqqff] (0.,9.72) circle (3pt);
\draw [fill=qqqqff] (8.,9.72) circle (3pt);
\draw [fill=qqqqff] (11.98,9.72) circle (3pt);
\draw [fill=qqqqff] (16.,9.72) circle (3pt);
\draw [fill=xdxdff] (5.26,4.41) circle (3pt);
\draw [fill=xdxdff] (6.52,9.72) circle (3pt);
\draw [color=xdxdff] (4.,6.)-- ++(-1.5pt,-1.5pt) -- ++(3.0pt,3.0pt) ++(-3.0pt,0) -- ++(3.0pt,-3.0pt);
\end{scriptsize}
\end{tikzpicture}
\caption{Representation of the nonuniform mesh for the case $ x^n_{m+1/2}-y^n > h/2$. From bottom to top: the initial mesh at time $t^n$, the modified mesh at time $t^n$ and the new mesh at time $t^{n+1}$ when the constraint is enforced.}\label{fig_nonunif_mesh_case_a}
\end{figure}
At this point the bus position coincides with the left edge of the $m$-th cell. Therefore the non-classical shock appears as the solution given by $\mathcal{RS}^\alpha_2$ to the Riemann problem
\begin{equation}\label{Riemann_problem_nonunif_mesh_case_1}
\begin{cases}
\partial_t u + \partial_x[f(u)] = 0,\\
u(0,x) = \begin{cases}
u^l = \bar{u}^\text{new}_{m-1} & \text{if } x\leq x^\text{new}_{m-1/2},\\
u^r = \bar{u}^\text{new}_m & \text{if } x>x^\text{new}_{m-1/2}.
\end{cases}
\end{cases}
\end{equation}
When
$$f_1(\mathcal{RS}(\bar{u}^\text{new}_{m-1},\bar{u}^\text{new}_{m})(\bar{V}^n)) \leq F_\alpha + \bar{V}^n \, \bar{\rho}(\bar{u}^\text{new}_{m-1},\bar{u}^\text{new}_{m})(\bar{V}^n),$$
we can apply the standard Godunov's method and compute the bus trajectory with the method introduced for the Riemann solver $\mathcal{RS}^\alpha_1$. At time $t^{n+1}$ we keep the modified mesh of time $t^n$.\\
When the constraint is not satisfied, the non-classical shock appears and propagates at the same speed of the bus. Therefore, in order to follow the non-classical shock, at time $t^{n+1}$ we move the left edge of the $m$-th cell in
\begin{equation}\label{nonuniform_mesh_case_a_time_t_n+1}
x^{n+1}_{m-1/2} = x^\text{new}_{m-1/2} +\bar{V}^n \, k^n.
\end{equation}
Since the $(m-1)$-th and the $m$-th cells change between the time steps $t^n$ and $t^{n+1}$, we have to apply Lemma \ref{Green_theorem_on_nonuniform_mesh} to compute the approximate solution $\bar{u}^{n+1}$; see Figure \ref{fig_nonunif_mesh_case_a}. We obtain
\begin{equation}\label{nonunif_mesh_case_i_new_solution}
\begin{split}
& \bar{u}^{n+1}_{m-1} = \dfrac{ h^\text{new}_{m-1} \bar{u}^\text{new}_{m-1} -k^n\left[f(\hat{u})-\bar{V}^n \, \hat{u}-f(\mathcal{RS}(\bar{u}^n_{m-2},\bar{u}^\text{new}_{m-1})(0))\right]}{h^\text{new}_{m-1}+\bar{V}^n\, k^n} \; \text{ and}\\
& \bar{u}^{n+1}_{m} = \dfrac{h^\text{new}_{m} \bar{u}^\text{new}_m -k^n \, \left[ f(\mathcal{RS}(\bar{u}^\text{new}_m,\bar{u}^n_{m+1})(0))-f(\check{u}_2)+\bar{V}^n \, \check{u}_2\right]}{h^\text{new}_m-\bar{V}^n \, k^n}.
\end{split}
\end{equation}
For the other cells we can use the standard Godunov's method.\\
\item[(ii)] If $ x^n_{m+1/2}-y^n \leq h/2$, then we shift the right interface $x^n_{m+1/2}$ of the $m$-th cell to $y^n$, i.e. we introduce
\begin{equation}\label{nonuniform_mesh_case_b_time_t_n_right_edge_of_Cm}
x^\text{new}_{m+1/2} = y^n
\end{equation}
and we restore the previous cells introducing
\begin{equation}\label{nonuniform_mesh_case_b_time_t_n_left_edge_of_Cm}
x^\text{new}_{m-1/2} = x^n_{m-3/2}+h.
\end{equation}
We move $x^n_{m+1/2}$ to $x^\text{new}_{m+1/2} = y^n$ and $x^n_{m-1/2}$ to $x^\text{new}_{m-1/2}$; see Figure \ref{fig_nonunif_mesh_case_b2}.\\
Let us call
$$h^\text{new}_{m}= x^\text{new}_{m+1/2}-x^\text{new}_{m-1/2} \; \text{ and } \; h^\text{new}_{m+1}=x^n_{m+3/2}-x^\text{new}_{m+1/2}$$
respectively the length of the new $m$-th cell and the length of the $(m+1)$-th cell.
\begin{remark}
Since the lengths of the cells have changed, we have to adapt $k^n$ to satisfy the CFL condition.
\end{remark}
\begin{figure}[ht]
\centering
\definecolor{xdxdff}{rgb}{0.49019607843137253,0.49019607843137253,1.}
\definecolor{qqqqff}{rgb}{0.,0.,1.}
\begin{tikzpicture}[scale=0.6,line cap=round,line join=round,>=triangle 45,x=1.0cm,y=1.0cm]
\clip(-1.5,3.5) rectangle (15.5,11);
\draw [domain=-1.5:15.5] plot(\x,{(--72.14-0.*\x)/16.});
\draw [domain=-1.5:15.5] plot(\x,{(--96.-0.*\x)/16.});
\draw [domain=-1.5:15.5] plot(\x,{(--143.70-0.*\x)/16.});
\draw (4.5,4.5) node[anchor=north west] {$x^n_{m-1/2}$};
\draw (-1.15,4.5) node[anchor=north west] {$x^n_{m-3/2}$};
\draw (7.41,4.5) node[anchor=north west] {$x^n_{m+1/2}$};
\draw (11.,4.5) node[anchor=north west] {$x^n_{m+3/2}$};
\draw [->] (1.98,4.62) -- (1.98,5.79);
\draw [->] (5.41,4.64) -- (5.41,5.75);
\draw [->] (10.56,4.77) -- (10.52,5.82);
\draw (14.60,6.) node[anchor=north west] {$t^n$};
\draw (14.57,4.5) node[anchor=north west] {$t^n$};
\draw (14.20,9.1) node[anchor=north west] {$t^{n+1}$};
\draw (-1.15,6.) node[anchor=north west] {$x^n_{m-3/2}$};
\draw (2.8,6.) node[anchor=north west] {$x^{new}_{m-1/2}$};
\draw (5.8,6.) node[anchor=north west] {$x^{new}_{m+1/2}$};
\draw (11.,6.) node[anchor=north west] {$x^n_{m+3/2}$};
\draw (4.,6.)-- (4.,8.98);
\draw (7.11,6.01)-- (8.43,8.99);
\draw (12.,6.)-- (12.,8.99);
\draw (0.,6.)-- (0.,8.98);
\draw (-1.15,10.10) node[anchor=north west] {$x^{n+1}_{m-3/2}$};
\draw (3.07,10.21) node[anchor=north west] {$x^{n+1}_{m-1/2}$};
\draw (7.50,10.21) node[anchor=north west] {$x^{n+1}_{m+1/2}$};
\draw (11.,10.26) node[anchor=north west] {$x^{n+1}_{m+3/2}$};
\draw (5.50,8.52) node[anchor=north west] {$(m)$};
\draw (0.80,8.59) node[anchor=north west] {$(m-1)$};
\draw (9.32,8.28) node[anchor=north west] {$(m+1)$};
\draw (6.6,4.5) node[anchor=north west] {$y^n$};
\draw [dash pattern=on 1pt off 1pt] (-0.01,4.50)-- (0.,6.);
\draw [dash pattern=on 1pt off 1pt] (6.05,4.50)-- (6.07,6.);
\draw [dash pattern=on 1pt off 1pt] (8.,4.52)-- (8.,6.);
\draw [dash pattern=on 1pt off 1pt] (12.,4.53)-- (12.,6.);
\begin{scriptsize}
\draw [fill=qqqqff] (-0.01,4.50) circle (3pt);
\draw [fill=qqqqff] (15.98,4.50) circle (3pt);
\draw [fill=qqqqff] (8.,4.52) circle (3pt);
\draw [fill=qqqqff] (12.,4.53) circle (3pt);
\draw [fill=qqqqff] (0.,6.) circle (3pt);
\draw [fill=qqqqff] (4.,6.) circle (3pt);
\draw [fill=qqqqff] (7.11,6.01091) circle (3pt);
\draw [fill=qqqqff] (12.,6.) circle (3pt);
\draw [fill=qqqqff] (16.,6.) circle (3pt);
\draw [fill=qqqqff] (0.,8.98) circle (3pt);
\draw [fill=qqqqff] (8.43,8.99) circle (3pt);
\draw [fill=qqqqff] (12.,8.99) circle (3pt);
\draw [fill=qqqqff] (16.,8.98) circle (3pt);
\draw [fill=xdxdff] (7.09,4.5) circle (3pt);
\draw [fill=xdxdff] (4.,8.98) circle (3pt);
\draw [color=xdxdff] (8.05,6.)-- ++(-3pt,-3pt) -- ++(3.0pt,3.0pt) ++(-3.0pt,0) -- ++(3.0pt,-3.0pt);
\draw [fill=xdxdff] (6.05,4.5) circle (3pt);
\draw [color=xdxdff] (6.1,6.)-- ++(-3pt,-3pt) -- ++(3.0pt,3.0pt) ++(-3.0pt,0) -- ++(3.0pt,-3.0pt);
\end{scriptsize}
\end{tikzpicture}
\caption{Representation of the nonuniform mesh for the case $ x^n_{m+1/2}-y^n \leq h/2$. From bottom to top: the initial mesh at time $t^n$, the modified mesh at time $t^n$ and the new mesh at time $t^{n+1}$ when the constraint is enforced.}\label{fig_nonunif_mesh_case_b2}
\end{figure}
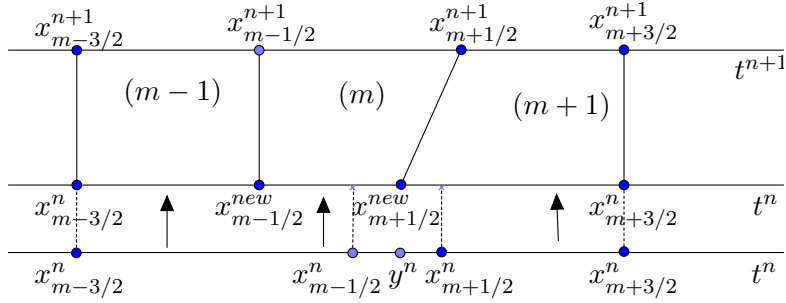

Since we split the $(m-1)$-th cell in two parts, the average solutions at time $t^n$ in the cells $C^n_{m-1}$ and $C^n_{m}$ are:
\begin{equation}\label{nonunif_mesh_modif_sol_time_t^n}
\begin{split}
& \bar{u}^\text{new}_{m-1} = \bar{u}^n_{m-1} \; \text{ and}\\
& \bar{u}^\text{new}_m = \dfrac{(x^\text{new}_{m+1/2}-x^n_{m-1/2})\,\bar{u}^n_m + (x^n_{m-1/2}-x^\text{new}_{m-1/2})\, \bar{u}^n_{m-1}}{h^\text{new}_m}.
\end{split}
\end{equation}
For the $(m+1)$-th cell the new average solution is
\begin{equation}\label{nonunif_mesh_modified_sol_time_t^n_2}
\bar{u}^\text{new}_{m+1} = \dfrac{(x^n_{m+1/2}-x^\text{new}_{m+1/2})\, \bar{u}^n_m + (x^n_{m+3/2}-x^n_{m+1/2})\, \bar{u}^n_{m+1}}{h^\text{new}_{m+1}}.
\end{equation}
At this point the bus is on the right interface of the $m$-th cell. Therefore we have to consider the Riemann problem
\begin{equation}\label{Riemann_problem_nonunif_mesh_case_2}
\begin{cases}
\partial_t u + \partial_x[f(u)] = 0,\\
u(0,x) = \begin{cases}
u^l = \bar{u}^\text{new}_{m} & \text{if } x\leq x^\text{new}_{m+1/2},\\
u^r = \bar{u}^\text{new}_{m+1} & \text{if } x>x^\text{new}_{m+1/2}.
\end{cases}
\end{cases}
\end{equation}
We apply the standard Godunov's method, when
$$f_1(\mathcal{RS}(\bar{u}^\text{new}_{m},\bar{u}^\text{new}_{m+1})(\bar{V}^n)) \leq F_\alpha + \bar{V}^n \, \bar{\rho}(\bar{u}^\text{new}_{m},\bar{u}^\text{new}_{m+1})(\bar{V}^n).$$
In this case we keep the modified mesh within $[t^n,t^{n+1}]$.
Otherwise the position of the right interface of the $m$-th cell at time $t^{n+1}$ is
\begin{equation}\label{nonuniform_mesh_case_b_time_t_n+1}
x^{n+1}_{m+1/2} = x^\text{new}_{m+1/2} +\bar{V}^n \, k^n.
\end{equation}
Since the constraint is violated, the non-classical shock appears. Therefore the new solution can be computed with the standard Godunov's method in all the cells except for the $m$-th and the $(m+1)$-th for which we have to apply Lemma \ref{Green_theorem_on_nonuniform_mesh}. We find
\begin{equation}\label{nonunif_mesh_case_ii_new_sol}
\begin{split}
& \bar{u}^{n+1}_{m} = \dfrac{h^\text{new}_{m} \, \bar{u}^\text{new}_m -k^n \, \left[ f(\hat{u})-\bar{V}^n\, \hat{u}-f(\mathcal{RS}(\bar{u}^n_{m-1},\bar{u}^\text{new}_m)(0)) \right]}{h^\text{new}_m+\bar{V}^n \, k^n} \; \text{ and}\\
& \bar{u}^{n+1}_{m+1} = \dfrac{ h^\text{new}_{m+1} \, \bar{u}^\text{new}_{m+1} -k^n\left[ f(\mathcal{RS}(\bar{u}^\text{new}_{m+1},\bar{u}^n_{m+2})(0))-f(\check{u}_2)+\bar{V}^n \, \check{u}_2 \right]}{h^\text{new}_{m+1}-\bar{V}^n \, k^n}.
\end{split}
\end{equation}
\end{enumerate}
Once we have concluded the previous processes, we compute the new value of $m$. The bus position coincides with the point $x^{n+1}_{m-1/2}$ of the mesh.\\
The next proposition states that this method allows to reconstruct exactly a non classical-shock when the initial datum is the non-classical shock itself; see Figure \ref{fig_non_classical_shock}.
\begin{prop}\label{prop_exact_reconstruction_of_non_classical_shock_with_non_uniform_mesh}
Fix $n \in \mathbb{N}$. Let us consider the Riemann problem (\ref{RP_non_classical_shock}) with initial datum
\begin{equation}
u(t^n,x) =\begin{cases}
\hat{u} & \text{if } x\leq y^n,\\
\check{u}_2 & \text{if } x> y^n.
\end{cases}
\end{equation}
\begin{enumerate}
\item[(i)] If $x^n_{m+1/2}-y^n >h/2$, then
$$\bar{u}^{n+1}_{m-1} = \hat{u} \; \text{ and } \; \bar{u}^{n+1}_m = \check{u}_2.$$
\item[(ii)] If $x^n_{m+1/2}-y^n \leq h/2$, then
$$\bar{u}^{n+1}_{m} = \hat{u} \; \text{ and } \; \bar{u}^{n+1}_{m+1} = \check{u}_2.$$
\end{enumerate}
\end{prop}
\begin{proof}
The mesh definition guarantees that the bus is always in $x^n_{m-1/2}$.\\
Since the initial datum is a non-classical shock, the solution given by $\mathcal{RS}^\alpha_2$ is the non-classical shock itself.\\
In case (i) we introduce $x^\text{new}_{m-1/2}=y^n$. The average solution at time $t^n$ is
$$\bar{u}^n_{j} = \hat{u}\; \text{ if }\; j\leq m-1 \; \text{ and } \; \bar{u}^n_m = \check{u}_2 \; \text{ if } \;j\geq m.$$ Therefore
\begin{equation*}
\begin{split}
& \mathcal{RS}(\bar{u}^n_{m-2},\bar{u}^n_{m-1})(0) = \mathcal{RS}(\hat{u},\hat{u})(0) = \hat{u} \; \text{ and}\\
& \mathcal{RS}(\bar{u}^n_{m},\bar{u}^n_{m+1})(0) = \mathcal{RS}(\check{u}_2,\check{u}_2)(0) = \check{u}_2.
\end{split}
\end{equation*}
Hence at time $t^{n+1}$, applying the formula (\ref{nonunif_mesh_case_i_new_solution}), we obtain
\begin{equation*}
\begin{split}
& \bar{u}^{n+1}_{m-1} = \dfrac{h^n_{m-1}\, \bar{u}^n_{m-1}-k^n \, \left(f(\hat{u})-\bar{V}^n\, \hat{u} - f(\hat{u}) \right)}{h^n_{m-1}+\bar{V}^n \, k^n} = \dfrac{h^n_{m-1}\, \hat{u} + \bar{V}^n \, k^n \,\hat{u}}{h^n_{m-1}+\bar{V}^n \, k^n} = \hat{u} \; \;  \text{ and}\\
& \bar{u}^{n+1}_m = \dfrac{h^n_m \, \bar{u}^n_m -k^n \, \left(f(\check{u}_2)-f(\check{u}_2) +\bar{V}^n \, k^n \right)}{h^n_m - \bar{V}^n\, k^n} = \dfrac{h^n_m \, \check{u}_2-\bar
V^n \,k^n\, \check{u}_2 }{h^n_m+ \bar{V}^n\, k^n} = \check{u}_2. 
\end{split}
\end{equation*}
In case (ii) we introduce the points $x^\text{new}_{m-1/2}$ and $x^\text{new}_{m+1/2}$. The new solution is $\bar{u}^n_{j} = \hat{u}$, if $j\leq m$ and $\bar{u}^n_m = \check{u}_2$, if $j\geq m+1$. Indeed the formulas (\ref{nonunif_mesh_modif_sol_time_t^n}) and (\ref{nonunif_mesh_modified_sol_time_t^n_2}) give
\begin{equation*}
\begin{split}
& \bar{u}^\text{new}_{m-1} = \bar{u}^n_{m-1} = \hat{u},\\
& \bar{u}^\text{new}_m = \bar{u}^n_{m-1} = \hat{u} \; \text{ and}\\
& \bar{u}^\text{new}_{m+1} = \dfrac{h^n_m\, \check{u}_2 + h\, \check{u}_2}{h+h^n_m}=\check{u}_2.
\end{split}
\end{equation*}
Moreover we have
$$\mathcal{RS}(\bar{u}^n_{m-1},\bar{u}^\text{new}_m)(0) = \hat{u} \; \text{ and } \; \mathcal{RS}(\bar{u}^\text{new}_{m+1},\bar{u}^n_{m+2})(0)= \check{u}_2.$$
Therefore at time $t^{n+1}$, applying the formula (\ref{nonunif_mesh_case_ii_new_sol}), we obtain
\begin{equation*}
\begin{split}
& \bar{u}^{n+1}_m= \dfrac{h^\text{new}_m\,\bar{u}^\text{new}_{m} -k^n\, \left(f(\hat{u})-\bar{V}^n\, \hat{u}-f(\hat{u}) \right)}{h^\text{new}_m + \bar{V}^n\,k^n} = \hat{u} \; \text{ and}\\
& \bar{u}^{n+1}_{m+1} = \dfrac{h^n_{m+1}\, \bar{u}^{n+1}_{m+1}-k^n\, \left(f(\check{u}_2)-f(\check{u}_2)+\bar{V}^n \, \check{u}_2 \right)}{h^\text{new}_{m+1}-\bar{V}^n \, k^n} = \check{u}_2.
\end{split}
\end{equation*}
\end{proof}
\begin{figure}[ht]
\centering
\includegraphics[width=1.0\linewidth]{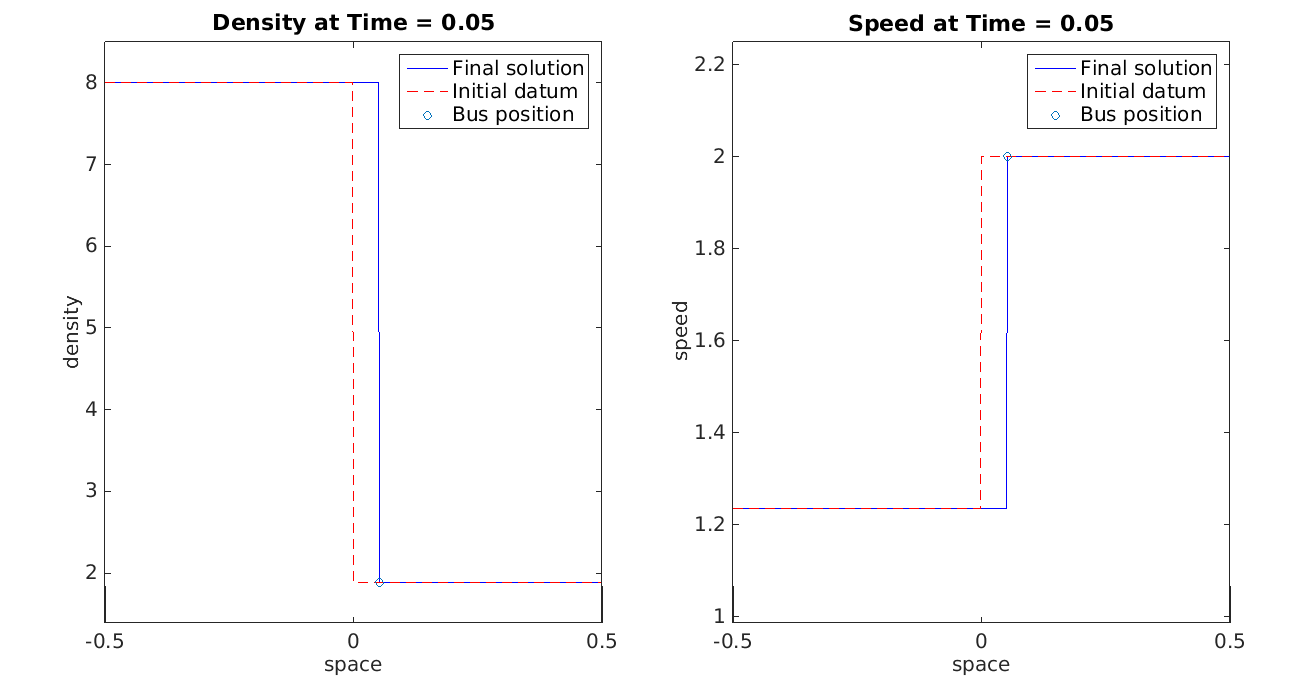}
\caption{Solution obtained with the nonuniform mesh method for $p(\rho) = \rho$. The initial datum is the non-classical shock  $(\rho,v)(0,x)=(8,V_b+F_\alpha/\rho^l)$ if $x \leq y_0$ and $(\rho,v)(0,x) = (F_\alpha/(v^r-V_b),2)$ if $x>y_0$, the reduction rate due to the presence of the bus is $\alpha = 0.25$, the bus initial position is $y_0=0$, and the bus maximal speed is $V_b = 1$.}\label{fig_non_classical_shock}
\end{figure}
\begin{obs}
Let us suppose that the initial position of the bus does not coincide with one of the points of the initial mesh, so that we have
$$x^0_{m-1/2} < y^0 < x^0_{m+1/2},$$
and let us consider the Riemann problem 
\begin{equation}\label{Riemann_problem_non-classical_shock}
\begin{cases}
\partial_t \, u + \partial_x\,[f(u)] = 0,\\
u(0,x) = \begin{cases}
\hat{u} & \text{if } x \leq y^0,\\
\check{u}_2 & \text{if } x>y^0.
\end{cases}
\end{cases}
\end{equation}
Since in the $m$-th cell the initial datum is not constant, we have to do an average in order to apply the nonuniform mesh method for $\mathcal{RS}^\alpha_2$. The average value is given by
\begin{equation*}
\tilde{u}^0_{m} = \dfrac{\left(y^0-x^0_{m-1/2}\right) \hat{u}+\left(x^0_{m+1/2}-y^0\right) \check{u}_2}{h^0_{m}}.
\end{equation*}
$\tilde{u}^0_m$ is a convex combination of $\hat{u}$ and $\check{u}_2$. Therefore, if $(\tilde{\rho}^0_m,\tilde{v}^0_m)$ are the non conservative components of $\tilde{u}^0_m$, then we have
$$\tilde{\rho}^0_m\, \tilde{v}^0_m = F_\alpha + \tilde{\rho}^0_m\, \bar{V}^n.$$\\
Let us suppose that $x^0_{m+1/2}-y^0>h/2$.\\
Applying the nonuniform mesh method, we shift the left side $x^0_{m-1/2}$ of the $m$-th cell to $y^0$, i.e. we introduce
\begin{equation}
x^\text{new}_{m-1/2} = y^0.
\end{equation}
The lengths of the new $m$-th and $(m-1)$-th cells are respectively
$$h^\text{new}_{m}:=x^{0}_{m+1/2}-x^\text{new}_{m-1/2}\; \text{ and } \; h^\text{new}_{m-1}= x^\text{new}_{m-1/2}-x^0_{m-3/2}$$
Since we have modified the mesh, we have to recompute the average solution. For the former we find
\begin{equation*}
\begin{split}
\bar{u}^\text{new}_{m-1} & = \dfrac{1}{h^\text{new}_{m-1}}\left[(x^\text{new}_{m-1/2}-x^n_{m-1/2})\, \bar{u}^n_{m}+(x^n_{m-1/2}-x^n_{m-3/2}) \, \bar{u}^n_{m-1}  \right] =\\
& = \dfrac{1}{h^\text{new}_{m-1}}\left[(x^\text{new}_{m-1/2}-x^n_{m-1/2})\, \tilde{u}+(x^n_{m-1/2}-x^n_{m-3/2}) \, \hat{u}  \right].
\end{split}
\end{equation*}
The average on the $m$-th cell remains unchanged, i.e.
$$\bar{u}^\text{new}_m = \tilde{u}.$$
Since $\bar{u}^\text{new}_{m-1}$ is in turn a convex combination of $\tilde{u}$ and $\hat{u}$, all the points $\hat{u}$, $\check{u}_2$, $\tilde{u}$ and $\bar{u}^\text{new}_{m-1}$ satisfy the constraint with the equal. Therefore the solution to the Riemann problem (\ref{Riemann_problem_nonunif_mesh_case_1}) with initial datum
$$
u(0,x) = \begin{cases}
\bar{u}^\text{new}_{m-1} & \text{if } x\geq 0,\\
\bar{u}^\text{new}_{m} = \tilde{u} & \text{if } x>0,
\end{cases}
$$
does not satisfy the constraint. Therefore we compute the solution at time $t^1$ with the equations (\ref{nonunif_mesh_case_i_new_solution}). The solutions $\mathcal{RS}(\hat{u},\bar{u}^\text{new}_{m-1})(0)$ and $\mathcal{RS}(\tilde{u},\check{u}_2)(0)$ introduce an error that does not allow to capture exactly the non-classical shock.\\
Similarly if $x^0_{m+1/2}-y^0\leq h/2$.
If we choose a mesh for which we have $y^0=x^0_{m-1/2}$, then we do not introduce the average value $\tilde{u}$. Hence the solution is exact.\\
\end{obs}
Unfortunately, for more general initial datum, even with this method the non-classical shock is not reconstructed correctly. Indeed, the right trace of the velocity is overestimated or underestimated (and correspondingly the density is underestimated or overestimated), because we should have $\check{v}_2 = v^r$; see Figure \ref{fig_example_solution_nonunif_mesh_ghost_cell_method}.\\
\begin{figure}
\centering
\begin{subfigure}[h]{0.9\linewidth}
\includegraphics[width=\textwidth]{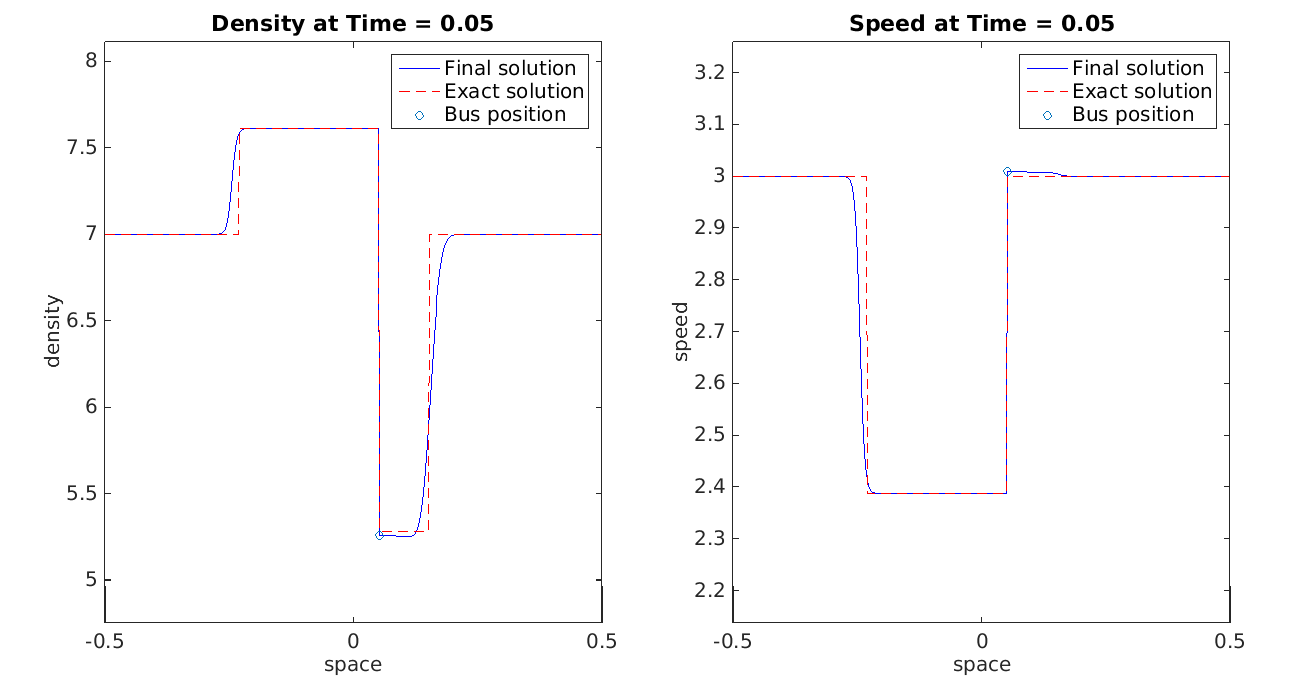}
\caption{Solution obtained for the constant initial datum $(\rho,v)(0,x) = (7,3)$, $\alpha = 0.5$, $y_0=0$ and maximal bus speed $V_b = 1$ (as in case (a) of Figure \ref{fig_example_overestimation_RS_2_conservation_imposed}).}\label{fig_example_solution_nonunif_mesh_ghost_cell_method_a}
\end{subfigure}
\\
\begin{subfigure}[h]{0.9\linewidth}
\includegraphics[width=\textwidth]{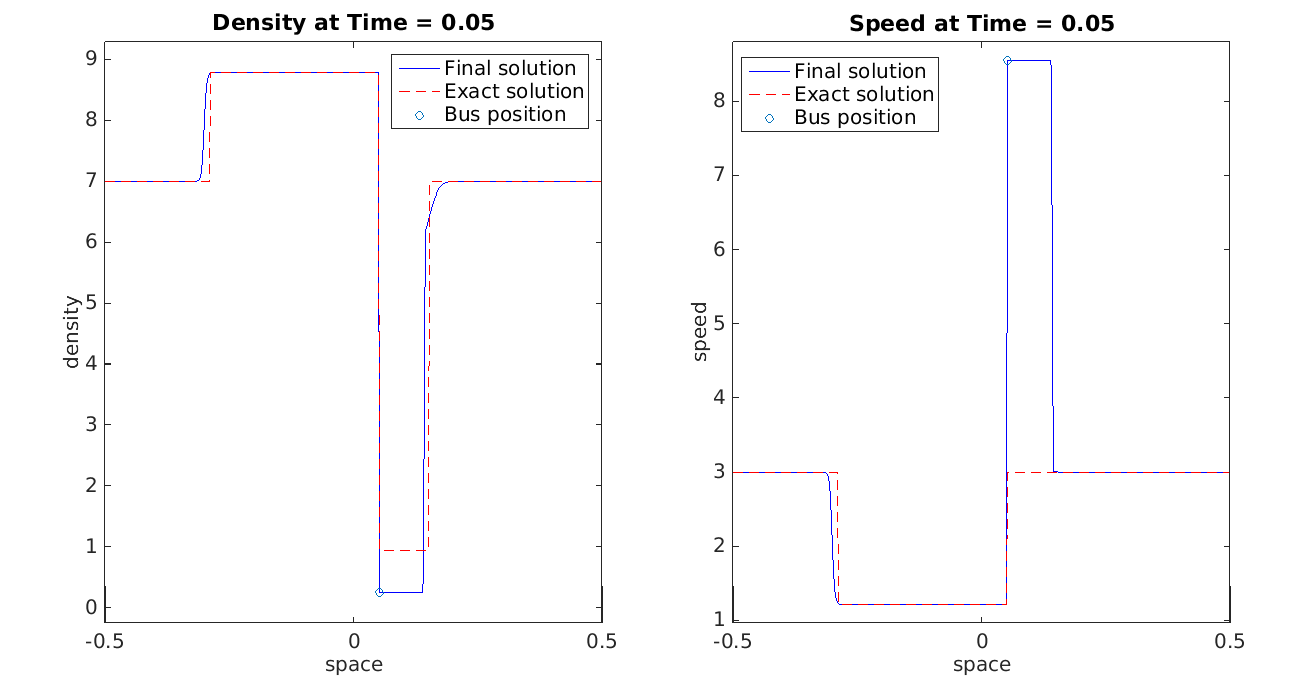} 
\caption{Solution obtained for the constant initial datum $(\rho,v)(0,x) = (7,3)$, $\alpha = 0.25$, $y_0$ and maximal bus speed $V_b = 1$.}\label{fig_example_solution_nonunif_mesh_ghost_cell_method_b}
\end{subfigure}
\caption{Example of solutions obtained with the nonuniform mesh method for $p(\rho) = \rho$. The right trace of the non-classical shock is overestimated. The result gets worse when the value of $F_\alpha$ is lower.}\label{fig_example_solution_nonunif_mesh_ghost_cell_method}
\end{figure}
\subsubsection{Non-uniform mesh with imposed value}
As we have done with the discontinuity reconstruction procedure, one idea to get a better result is to impose the value $\check{v}_2$ in the first cell after the bus.\\
In the case $x^n_{m+1/2}-y^n<h/2$, we do not modify the value of the solution in $C^n_m$ at time $t^n$. Therefore we have only to update the second conservative component of the solution at time $t^{n+1}$ as
\begin{equation}
\bar{z}^{n+1}_m = \bar{\rho}^{n+1}_m(\check{v}_2+p(\bar{\rho}^{n+1}_m)),
\end{equation}
where $\check{v}_2$ is the right trace of $\mathcal{RS}^\alpha_2$ for the Riemann problem (\ref{Riemann_problem_nonunif_mesh_case_1}), i.e.
$$\check{v}_2 = \bar{v}^n_m.$$
The case $x^n_{m+1/2}-y^n<h/2$ is more delicate. At time $t^n$ we recompute the average solution in the cell $C^n_{m+1}$ using the formula (\ref{nonunif_mesh_modified_sol_time_t^n_2}). Let us take $\check{v}_2$ as the right trace of $\mathcal{RS}^\alpha_2$ for the Riemann problem (\ref{Riemann_problem_nonunif_mesh_case_2}), i.e.
$$\check{v}_2 = \bar{v}^\text{new}_{m+1}.$$
If we simply update the second conservative component of the solution in the cell $C^n_{m+1}$ as
\begin{equation}\label{updating_the_solution_case_ii}
\bar{z}^{n+1}_{m+1} = \bar{\rho}^{n+1}_{m+1} (\check{v}_2+p(\bar{\rho}^{n+1}_{m+1}))
\end{equation}
we introduce an error in the solution.\\
Therefore, we propose to fix $\check{v}_2 = \bar{v}^n_m$, which is the value we want to preserve, then we do all the steps (\ref{nonunif_mesh_modif_sol_time_t^n}), (\ref{nonunif_mesh_modified_sol_time_t^n_2}) and (\ref{nonuniform_mesh_case_b_time_t_n+1}) and finally we update the second conservative component with the formula (\ref{updating_the_solution_case_ii}).\\
The result that we obtain is still imperfect, but the right trace of the solution is captured correctly at least in the first cell after the bus (which can be in turn $C_{m}$ or $C_{m+1}$). In the following cells a travelling oscillation appears. Its amplitude depends on the initial datum, the bus speed and the value of $\alpha$; see Figure \ref{fig_example_solution_nonunif_mesh_imposed value}.
\begin{figure}
\centering
\begin{subfigure}[h]{1\linewidth}
\includegraphics[width=1\textwidth]{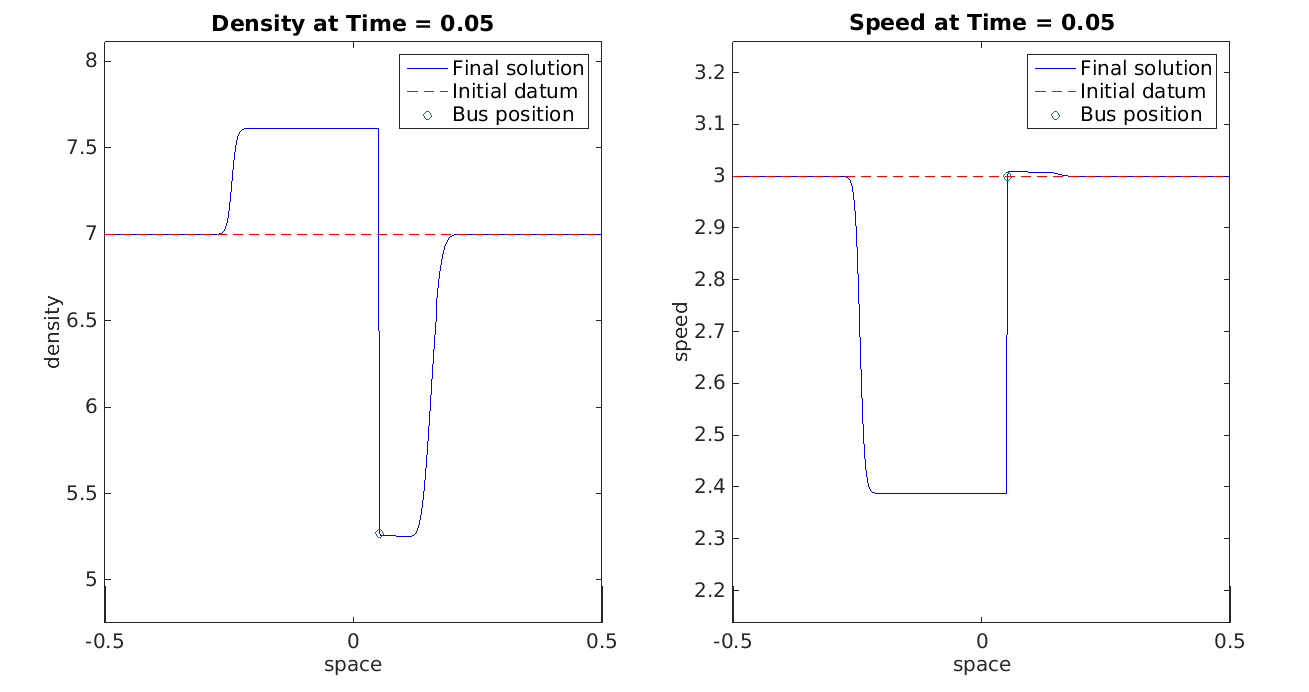}
\caption{Solution obtained for the constant initial datum $(\rho,v)(0,x) = (7,3)$, $\alpha = 0.5$, $y_0=0$ and maximal bus speed $V_b = 1$. The result is essentially equivalent to the one obtained in Figure \ref{fig_example_solution_nonunif_mesh_ghost_cell_method_a}.}
\end{subfigure}
\\
\begin{subfigure}[h]{1\linewidth}
\includegraphics[width=\textwidth]{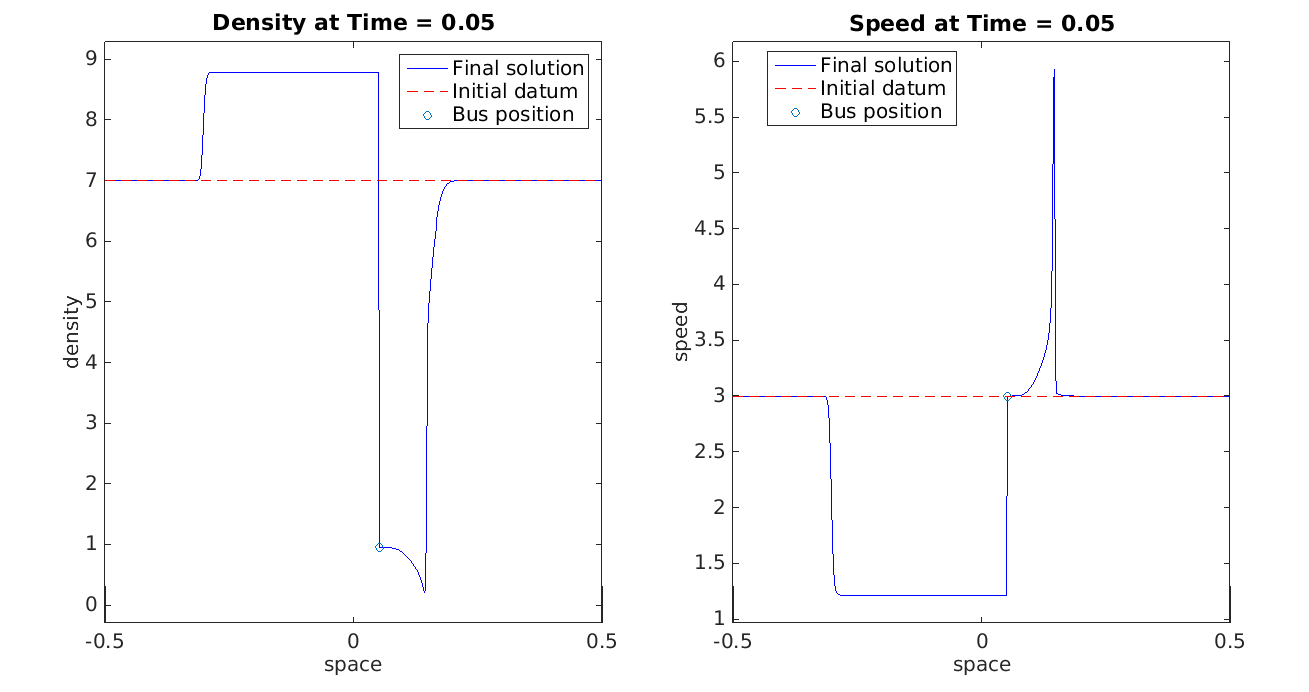} 
\caption{Solution obtained for the constant initial datum $(\rho,v)(0,x) = (7,3)$, $\alpha = 0.25$, $y_0=0$ and maximal bus speed $V_b = 1$. The result is better than the one obtained in Figure \ref{fig_example_solution_nonunif_mesh_ghost_cell_method_b}.}
\end{subfigure}
\caption{Example of solutions obtained with the nonuniform mesh method where we have imposed the desired value of the velocity in the cell after the bus for $p(\rho) = \rho$. The right trace is captured correctly at least in the first cell after the bus.}\label{fig_example_solution_nonunif_mesh_imposed value}
\end{figure}
\subsection{Mathematical details}
Let us fix the constants $v_1$, $v_2$, $w_1$ and $w_2$ in $\mathbb{R}^+$, such that $0<v_1 \leq v_2$, $0< w_1 \leq w_2$ and $v_2<w_2$. Let us suppose that the domain
$$\mathcal{D}_{v_1,v_2,w_1,w_2} = \lbrace (\rho,v)\in \mathbb{R}^+ \times \mathbb{R}^+: v_1 \leq v \leq v_2, \; w_1 \leq v+p(\rho) \leq w_2 \rbrace $$
is invariant for the Riemann solver $\mathcal{RS}^\alpha_2$. The following proposition states that, under an appropriate CFL condition, the nonuniform mesh that we have introduced in the previous section is well defined and that the time step $k^n$ does not converge to 0.
\begin{prop}\label{prop_well_def_of_nonunif_mesh}
The meshes given by (\ref{nonuniform_mesh_case_a_time_t_n+1}) for the case $ x^n_{m+1/2}-y^n > h/2$ or by (\ref{nonuniform_mesh_case_b_time_t_n_right_edge_of_Cm}), (\ref{nonuniform_mesh_case_b_time_t_n_left_edge_of_Cm}) and (\ref{nonuniform_mesh_case_b_time_t_n+1}) for the case $ x^n_{m+1/2}-y^n \leq h/2$ are well defined at every time step, i.e. for every $n\in \mathbb{N}$ we have
\begin{equation}\label{moving_mesh_well_def}
x^{n+1}_{m-1/2} < x^{n+1}_{m+1/2} \; \text{ and } \; x^{n+1}_{m+1/2}< x^{n+1}_{m+3/2},
\end{equation}
provided that the following CFL condition holds:
\begin{equation}\label{CFL_condition_nonuniform_mesh}
\left| k^n\lambda^n\right| = \dfrac{1}{2}\,\min_{j\in \mathbb{Z}}\, h^n_j \; \text{ for every } \; n\in\mathbb{N},
\end{equation}
where $\lambda^n=\max\lbrace |\lambda_i(\bar{u}^n_j)|\, : \, i=1,2 \; \text{ and } \; j\in \mathbb{Z}\rbrace$ and $\lambda_i$ is the $i$-th eigenvalue of the Jacobian matrix $Df$ of the flux function.\\
Moreover there exists a constant $k>0$ which depends only on the invariant domain $\mathcal{D}_{v_1,v_2,w_1,w_2}$ and such that
\begin{equation}\label{time_step_limited_from_below}
k^n \geq k \; \text{ for every } n\in \mathbb{N}.
\end{equation}
\end{prop}
\begin{proof}
Let us consider the case $x^n_{m+1/2}-y^n > h/2$.\\
When the constraint is satisfied we have
$$x^{n+1}_{m-1/2} = x^\text{new}_{m-1/2} = y^n < x^n_{m+1/2} = x^{n+1}_{m+1/2}.$$
Otherwise, by (\ref{nonuniform_mesh_case_a_time_t_n+1}), we find
$$x^{n+1}_{m-1/2} = x^\text{new}_{m-1/2} + \bar{V}^n \, k^n.$$
By the definition (\ref{bus_speed}), we have
$$\bar{V}^n = \min \lbrace V_b,\bar{v}^n_m\rbrace.$$
Hence
$$\bar{V}^n\leq \bar{v}^n_m = \lambda_2(\bar{\rho}^n_m,\bar{v}^n_m) \leq \lambda^n.$$
By the CFL condition (\ref{CFL_condition_nonuniform_mesh}), we find
\begin{equation*}
\begin{split}
x^{n+1}_{m-1/2} & \leq x^\text{new}_{m-1/2} + \lambda^n \, k^n = x^\text{new}_{m-1/2} + \dfrac{1}{2} \min_{j\in \mathbb{Z}} h^n_j\leq x^\text{new}_{m-1/2} + \dfrac{h}{2} = \\
& = y^n+\dfrac{h}{2} < x^{n}_{m+1/2} = x^{n+1}_{m+1/2},
\end{split}
\end{equation*}
because $x^n_{m+1/2}-y^n > h/2$ and $x^{n+1}_{j-1/2} = x^n_{j-1/2}$ for every $j \neq m$.\\
If $x^n_{m+1/2}-y^n \leq h/2$, then we shift the point $x^n_{m+1/2}$ to the bus position, namely we introduce
$$x^\text{new}_{m+1/2} = y^n.$$
When the constraint is satisfied, we have
$$x^{n+1}_{m+1/2} = x^\text{new}_{m+1/2} = y^n < x^{n}_{m+1/2} < x^{n+1}_{m+3/2},$$
because $x^{n+1}_{j+1/2} = x^n_{j+1/2}$ for every $j \neq m$.\\
Otherwise, proceeding as in the previous case, we find
$$x^{n+1}_{m+1/2} = x^\text{new}_{m+1/2} + \bar{V}^n\, k^n < y^n +\dfrac{h}{2} \leq x^{n}_{m+1/2}+\dfrac{h}{2} < x^{n+1}_{m+3/2}.$$
For the second part of the proposition, let us observe that for every $(\rho,v)\in \mathcal{D}_{v_1,v_2,w_1,w_2}$, we have $v_1\leq \lambda_2(\rho,v)\leq v_2$, because $\lambda_2(\rho,v) = v$. Moreover for every $(\rho_0,v_0) \in \mathcal{D}_{v_1,v_2,w_1,w_2}$, the function $\psi: \rho \to \rho\,L_1(\rho,\rho_0,v_0)$ is Lipschitz (see Lemma \ref{lemma_curve_lipschitz_1}). Then there exist two constants $L_1$ and $L_2$ depending only on the invariant domain and such that for every $(\rho,v)\in\mathcal{D}_{v_1,v_2,w_1,w_2}$ and $v+p(\rho) = v_0 + p(\rho_0)$, we have
$$L_1 \leq \psi'(\rho) \leq L_2.$$
Therefore $L_1 \leq \lambda_1(\rho,L_1(\rho,\rho_0,v_0)) \leq L_2$, because $\lambda_1(\rho,L_1(\rho,\rho_0,v_0)) = \psi'(\rho)$. Hence for every $n\in \mathbb{N}$, we have $\lambda^n \in [\min(L_1,v_1),\max(L_2,v_2)]$.\\
Fix $n\in\mathbb{N}$.
\begin{enumerate}
\item[(i)] If $x^n_{m+1/2}-y^n > h/2$, then, by the mesh definition, we have
$$ h^n_j = h \; \text{ for every } \; j \notin \lbrace m-1, m \rbrace,$$
because we modify only the cell $C^n_{m-1}$ and $C^n_m$.\\
For the $(m-1)$-th cell when the constraint is enforced, we find
$$h^\text{new}_{m-1} = x^\text{new}_{m-1/2}-x^n_{m-3/2} \geq x^n_{m-1/2}-x^n_{m-3/2} = h^n_m \geq h.$$
For the $m$-th cell, since $x^n_{m+1/2}-y^n > h/2$, we easily find
$$h^\text{new}_m = x^n_{m+1/2}-x^\text{new}_{m-1/2} = x^n_{m+1/2}-y^n >\dfrac{h}{2}.$$
\item[(ii)] If $x^n_{m+1/2}-y^n \leq h/2$, then we have
$$ h^n_j = h \; \text{ for every } \; j \notin \lbrace m, m+1 \rbrace.$$
For the cell $C^n_m$, we find
$$h^\text{new}_m = x^\text{new}_{m+1/2}-x^\text{new}_{m-1/2} = y^n - (x^n_{m+1/2}-h) \geq h - \dfrac{h}{2} = \dfrac{h}{2},$$
because $x^\text{new}_{m-1/2} =x^n_{m+1/2} -h$ and $x^n_{m+1/2}-y^n \leq h/2$.\\
For the $(m+1)$-th cell, we obtain
$$h^\text{new}_{m+1} = x^n_{m+3/2}-x^\text{new}_{m+1/2} = x^n_{m+3/2}-y^n\geq x^n_{m+3/2}-x^n_{m+1/2} = h.$$
\end{enumerate}
Therefore
$$h^n_j \geq \dfrac{h}{2} \; \text{ for every } \; j\in \mathbb{Z}.$$
By the CFL condition (\ref{CFL_condition_nonuniform_mesh}) we have
$$k^n = \dfrac{1}{2 \, \lambda^n}\min_{j\in \mathbb{Z}} h^n_j \geq \dfrac{h}{4 \, \max(L_2,v_2)}.$$ 
Taking $k := \dfrac{h}{4 \, \max(L_2,v_2)}$, we obtain the inequality (\ref{time_step_limited_from_below}).
\end{proof}
The next proposition guarantees that if the CFL condition (\ref{CFL_condition_nonuniform_mesh}) holds, then waves centred in different points of the mesh cannot interact within a time step.
\begin{prop}
Let $n\in \mathbb{N}$ be fixed. Waves centred in two neighbouring Riemann problems cannot interact, provided that the CFL condition (\ref{CFL_condition_nonuniform_mesh}) holds. Moreover no wave can cross the bus trajectory within a time step.
\end{prop}
\begin{proof}
Since, by Proposition \ref{rarefaction_wave_properties} and the Lax-entropy condition (\ref{Lax_entropy_condition}), each wave propagates at a speed lower than $\lambda^n$, the CFL condition (\ref{CFL_condition_nonuniform_mesh}) ensures that no interactions between waves centred in neighbouring Riemann problems can happen, because they can cover no more than half length of a cell within $[t^n,t^{n+1}]$.\\
In the case $x^n_{m+1/2}-y^n > h/2$, at time $t^n$ the bus is in $x^\text{new}_{m-1/2}$. Therefore only waves centred in $x^n_{m-3/2}$ or $x^n_{m+1/2}$ could cross the bus trajectory. Proceeding as in Proposition \ref{prop_well_def_of_nonunif_mesh}, we find
$$\bar{V}^n\leq \lambda^n.$$
Hence also the bus remains in the first half of the new $m$-th cell, indeed
$$y^n + k^n\, \bar{V}^n \leq y^n+\dfrac{1}{2}\min_{j\in \mathbb{Z}}h^n_{j} = x^n_{m-1/2}+\dfrac{1}{2}\min_{j\in \mathbb{Z}}h^n_{j}.$$
\end{proof}
\chapter{Existence of solutions to the Cauchy problem for the Riemann solver $\mathcal{RS}^q_2$}
In this chapter we are going to apply the wave-front tracking method to find a solution to the Cauchy problem for the Aw-Rascle-Zhang system with a fixed constraint and for an initial datum in $L^1$ and with bounded total variation.\\
We assume that the pressure function satisfies the following stronger conditions:
\begin{equation}\label{ipotesi_pressione_forti}
\begin{cases}
p(0)=0,\\
p'(\rho)>0 \, \text{ for every } \, \rho>0,\\
p''(\rho)\geq 0 \; \text{ for every } \; \rho\geq 0.
\end{cases}
\end{equation}
\begin{lemma}\label{convex_concave}
Fix $(\rho_0,v_0) \in \mathbb{R}^+ \times \mathbb{R}^+$. Under the hypotheses (\ref{ipotesi_pressione_forti}), the functions $\rho \to \rho \, p(\rho)$ and $\rho \to \rho\, L_1(\rho,\rho_0,v_0) = \rho \, (v_0+p(\rho_0)-p(\rho))$ are respectively strictly convex and strictly concave for every $\rho>0$.
\end{lemma}
\begin{proof}
We have
\begin{equation*}
\begin{split}
& \dfrac{d}{d\rho}\left(\rho\,p(\rho)\right) = p(\rho)+\rho\,p'(\rho) \Longrightarrow \\
& \dfrac{d^2}{d\rho^2}\left(\rho\, p(\rho) \right) = 2p'(\rho)+\rho \, p''(\rho).
\end{split}
\end{equation*}
The function $\rho \to \rho\,p(\rho)$ is strictly convex if and only if
$$\dfrac{d^2}{d\rho^2}\left( \rho\,p(\rho)\right) >0.$$
By the hypotheses (\ref{ipotesi_pressione_forti}) $p'(\rho)>0$ and $p''(\rho)\geq 0$ for every $\rho>0$. Hence we have
$$2p'(\rho)+\rho\,p''(\rho)>0.$$
Similarly for $\rho \to \rho\, L_1(\rho,\rho_0,v_0)$.
\end{proof}
Fix $(\rho^l,v^l)$ and $(\rho^r,v^r)$ in $\mathbb{R}^+\times \mathbb{R}^+$. Consider the Riemann problem for the ARZ system centred in $\bar{x}\in \mathbb{R}$
\begin{equation}\label{Riemann_problem}
\begin{cases}
\partial_t\, \rho + \partial_x \, (\rho v) = 0,\\
\partial_t\, \rho(v+p(\rho))+\partial_x\,[\rho\, v (v+p(\rho))]=0,\\
(\rho,v)(0,x) = \begin{cases}
(\rho^l,v^l) & \text{if } x\leq \bar{x},\\
(\rho^r,v^r) & \text{if } x> \bar{x}.
\end{cases}
\end{cases}
\end{equation}
Let us consider a fixed constraint on the first component of the flux at $x=0$, i.e.
\begin{equation}
f_1((\rho,v)(t,0))=\rho(t,0)\, v(t,0) \leq q \; \text{ for every } \; t\in \mathbb{R}^+,
\end{equation}
where $q \in \mathbb{R}^+$  is fixed.\\
The Riemann solver $\mathcal{RS}^q_2$ is defined as follows; see \cite{garavello_goatin}.\\
\begin{enumerate}
\item If $f_1(\mathcal{RS}((\rho^l,v^l),(\rho^r,v^r))(0))\leq 0$, then
$$\mathcal{RS}^q_2((\rho^l,v^l),(\rho^r,v^r))(\lambda)=\mathcal{RS}((\rho^l,v^l),(\rho^r,v^r))(\lambda) \; \text{ for every } \; \lambda \in \mathbb{R}.$$
\item If $f_1f_1(\mathcal{RS}((\rho^l,v^l),(\rho^r,v^r))(0))> 0$, then
$$\mathcal{RS}^q_2((\rho^l,v^l),(\rho^r,v^r))(\lambda)=\begin{cases}
\mathcal{RS}((\rho^l,v^l),(\hat{\rho},\hat{v}))(\lambda) & \text{if } \lambda <0,\\
\mathcal{RS}((\check{\rho}_2,\check{v}_2),(\rho^r,v^r))(\lambda) & \text{if } \lambda \geq 0.
\end{cases}
$$
\end{enumerate}
\begin{remark}
The Riemann solver $\mathcal{RS}^q_2$ coincides with the Riemann solver $\mathcal{RS}^\alpha_2$ for the moving constraint, when the bus speed $\bar{V}$ is zero and $F_\alpha=q$.\\
Therefore we can apply the results that we have obtained in the previous chapters for this special case.
\end{remark}
\section{Interaction estimates}
Let $v_1$, $v_2$, $w_1$ and $w_2$ be fixed constants such that $0 <v_1<v_2$, $0<w_1<w_2$ and $v_2<w_2$. Let us suppose that
$$ v_1+p\left(\dfrac{q}{v_1}\right)\geq w_2, \;\;\; v_2+p\left(\dfrac{q}{v_2}\right) \leq w_2 \; \text{ and } \; v+p\left(\dfrac{q}{v}\right)\geq w_1 \; \text{ for every } v \in [v_1,v_2],$$
and that there exists $\bar{v}\in [v_1,v_2]$ such that
$$\bar{v}+p\left(\dfrac{q}{\bar{v}}\right)<w_2.$$
By Theorem \ref{Proposizione_Dominio_invariante_R_alfa2_definitiva}, under these hypotheses the set
$$\mathcal{D}_{v_1,v_2,w_1,w_2}=\lbrace (\rho,v)\in \mathbb{R}^+\times \mathbb{R}^+ \; : \; v_1 \leq v \leq v_2, \, w_1 \leq v+p(\rho) \leq w_2 \rbrace$$
is an invariant domain for $\mathcal{RS}^q_2$.
\begin{mydef}\label{def_total_variation}
Let us consider a function
$$f:\mathbb{R}^+\times \mathbb{R}\to \mathbb{R},$$
$$(t,x)\to f(t,x).$$
Let us suppose that $f$ is piecewise constant in the second variable, so that for every $t>0$, there exists a sequence $\chi:=\lbrace x_k \rbrace_{k\in \mathbb{N}}\subset \mathbb{R}$ for which
$$f(t,x) = \sum_{i\in\mathbb{N}} f(t,x_i)\mathbf{1}_{(x_i,x_{i+1})},$$
where $\mathbf{1}_{I}$ is the characteristic function of the interval $I\in \mathbb{R}$. The Total Variation of $f$ at time $t$ is
$$TV_t(f) :=\sum_{i \in\mathbb{N}} |f(t,x^+_{i+1})-f(t,x^+_i)| \; \text{ where } \; x_k \in \chi \; \text{ for every }\; k \in \mathbb{N},$$
where
$$x^+_j:= \lim_{\varepsilon\to 0} (x_j+\varepsilon)\; \text{ for every } \; j \in \mathbb{N}.$$
\end{mydef}
Fix two instants $t^1$ and $t^2$ in $\mathbb{R}^+$ such that $t^2>t^1$. Let us denote
$$\Delta TV_{\tilde{t}} (\rho)=TV_{t^2}(\rho)-TV_{t^1}(\rho) \; \text{ and } \; \Delta TV_{\tilde{t}} (v)=TV_{t^2}(v)-TV_{t^1}(v)$$
respectively the difference of the total variation of the first and the second component of the solution to the Riemann problem (\ref{Riemann_problem}) after an interaction between two waves has happened at time $\tilde{t}\in [t^1,t^2]$. Let
$$\Delta_{\tilde{t}} \mathcal{N}=\mathcal{N}_{t^2}-\mathcal{N}_{t^1}$$
be the variation in the number of waves before and after the interaction.\\
The aim of this section is to give estimates on $\Delta TV_{\tilde{t}}(\rho)$, $\Delta TV_{\tilde{t}}(v)$ and $\Delta_{\tilde{t}}\mathcal{N}$ at each instant $\tilde{t}$ of interaction.\\
The following table contains a list of the possible interactions between a characteristic wave and the constraint.\\
\begin{table}[H]
\centering
\begin{tabular}{|l|c|c|}
\hline Wave type & Classical solution in $x=0$ & Propositions and Case\\
\hline Contact discontinuity & Yes & \ref{prop_contact_discontinuity_from_left_case_1}(i), \ref{prop_contact_discontinuity_from_left_case_2}(i) \\
\hline Contact discontinuity & No & \ref{prop_contact_discontinuity_from_left_case_1}(ii), \ref{prop_contact_discontinuity_from_left_case_2}(ii) \\
\hline Shock & Yes & \ref{prop_shock_from_right_case_1}, \ref{prop_shock_from_right_case_2}(i)\\
\hline Shock & No & \ref{prop_shock_from_right_case_2} (ii) \\
\hline Rarefaction wave & Yes & \ref{prop_rarefaction_from_right_case_1}(i)  \\
\hline Rarefaction wave & No & \ref{prop_rarefaction_from_right_case_1}(ii), \ref{prop_rarefaction_from_right_case_2} \\
\hline
\end{tabular}
\caption{List of the propositions and cases in which the interaction between a characteristic wave and the line $x=0$ is discussed.}
\end{table}
We need some preliminary results.\\
Fix a domain $\mathcal{D}_{v_1,w_2,w_1,w_2}$ invariant for $\mathcal{RS}^q_2$. Let us recall the definitions (\ref{def_rho_v_min_1}) and (\ref{def_rho_v_max_1}) of the points $(\rho_\text{min},v_\text{min})$ and $(\rho_\text{max},v_\text{max})$ which have respectively the minimal and the maximal density in the invariant domain, i.e.
$$\rho_{\min}\leq \rho \leq \rho_{\max} \; \text{ for every }\; (\rho,v) \in \mathcal{D}_{v_1,v_2,w_1,w_2}.$$
The point $(\rho_\text{min},v_\text{min})\in \mathcal{D}_{v_1,v_2,w_1,w_2}$ is the solution to the system
\begin{equation}\label{def_rho_v_min}
\begin{cases}
v+p(\rho)=w_1,\\
v=v_2.
\end{cases}
\end{equation}
Similarly, the point $(\rho_\text{max},v_\text{max})\in \mathcal{D}_{v_1,v_2,w_1,w_2}$ is the solution to the system
\begin{equation}\label{def_rho_v_max}
\begin{cases}
v+p(\rho)=w_2,\\
v=v_1.
\end{cases}
\end{equation}
See Figure \ref{fig_notations_interaction_estimates}.\\
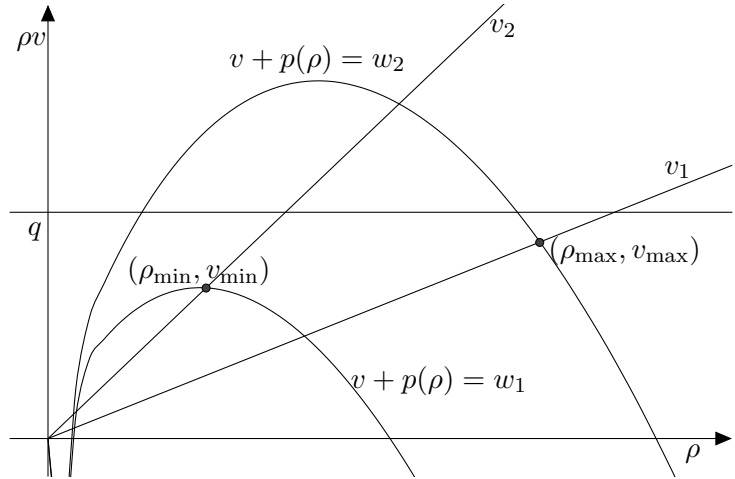
\begin{figure}[hbtp]
\centering
\definecolor{uququq}{rgb}{0.25098039215686274,0.25098039215686274,0.25098039215686274}
\begin{tikzpicture}[scale=0.5,line cap=round,line join=round,>=triangle 45,x=1.0cm,y=1.0cm]
\draw[->,color=black] (-1.,0.) -- (18.,0.);
\draw[->,color=black] (0.,-1.) -- (0.,11.5);
\clip(-1.,-1.) rectangle (18.,11.5);
\draw[smooth,samples=50,domain=0.001:18.0] plot(\x,{3.0*(\x)-(\x)^(1.5)});
\draw[smooth,samples=50,domain=0.001:18.0] plot(\x,{4.0*(\x)-(\x)^(1.5)});
\draw [domain=0.0:18.0] plot(\x,{(-0.--4.5*\x)/11.18});
\draw [domain=0.0:18.0] plot(\x,{(-0.--9.68*\x)/10.08});
\draw [domain=-1.:18.] plot(\x,{(--6.-0.*\x)/1.});
\draw (12.9,5.64) node[anchor=north west] {$(\rho_\text{max},v_\text{max})$};
\draw (1.81,5.1) node[anchor=north west] {$(\rho_\text{min},v_\text{min})$};
\draw (-0.8,6.) node[anchor=north west] {$q$};
\draw (11.35,11.37) node[anchor=north west] {$v_2$};
\draw (15.95,7.62) node[anchor=north west] {$v_1$};
\draw (4.50,10.7) node[anchor=north west] {$v+p(\rho)=w_2$};
\draw (7.7,2.20) node[anchor=north west] {$v+p(\rho) = w_1$};
\draw (-1.1,11.04) node[anchor=north west] {$\rho v$};
\draw (16.5,0.1) node[anchor=north west] {$\rho$};
\begin{scriptsize}
\draw [fill=uququq] (4.16,3.99) circle (3pt);
\draw [fill=uququq] (12.94,5.20) circle (3pt);
\end{scriptsize}
\end{tikzpicture}
\caption{Example of invariant domain for $\mathcal{RS}^q_2$ (the coloured area) and representation of the points $(\rho_\text{min},v_\text{min})$ and $(\rho_\text{max},v_\text{max})$ defined in (\ref{def_rho_v_min}) and (\ref{def_rho_v_max}).}\label{fig_notations_interaction_estimates}
\end{figure}

The next lemma characterizes the points of the invariant domain $\mathcal{D}_{v_1,v_2,w_1,w_2}$ where the Lax curves of the first family are decreasing.
\begin{lemma}\label{lemma_curve_decrescenti}
Let $(\rho_0,v_0)$ be a point of the domain $\mathcal{D}_{v_1,v_2,w_1,w_2}$. The function
$$\psi: \rho \rightarrow \rho L_1(\rho,\rho_0,v_0)$$
is strictly decreasing in the points of $\mathcal{D}_{v_1,v_2,w_1,w_2}$ if and only if
\begin{equation}\label{condizione_curve_decrescenti}
\lambda_1(\rho_\text{min},v_\text{min}) <0,
\end{equation}
where $(\rho_\text{min},v_\text{min})$ is the point defined in (\ref{def_rho_v_min}) and $\lambda_1(\rho,v) = v-\rho\,p'(\rho)$ is the first eigenvalue of the Jacobian matrix $Df$ of the flux function.
\end{lemma}
\begin{proof}
Let us call
$$\varphi(\rho) := \dfrac{d}{d\rho}(\rho p(\rho))= p(\rho)+\rho\, p'(\rho).$$
The inequality (\ref{condizione_curve_decrescenti}) is equivalent to
\begin{equation}\label{cond_curve_decrescenti_2}
\varphi(\rho_\text{min})>w_1,
\end{equation}
indeed
\begin{equation*}
\begin{split} \varphi(\rho_\text{min}) & = \dfrac{d}{d\rho}(\rho p(\rho))\Big|_{\rho=\rho_\text{min}}=\rho_\text{min}\, p'(\rho_\text{min})+p(\rho_\text{min})= \\
& = v_\text{min}-\lambda_1(\rho_\text{min},v_\text{min}) +p(\rho_\text{min}) = w_1 -\lambda_1(\rho_\text{min},v_\text{min}),
\end{split}
\end{equation*}
because $v_\text{min}+p(\rho_\text{min}) = w_1$. Then
\begin{equation*}
\lambda_1(\rho_\text{min},v_\text{min})<0 \Longleftrightarrow w_1-\lambda_1(\rho_\text{min},v_\text{min})>w_1 \Longleftrightarrow \varphi(\rho_\text{min}) >w_1.
\end{equation*}
Let $(\rho_0,v_0)$ be a point in $\mathcal{D}_{v_1,v_2,w_1,w_2}$. The function
$$\psi: \rho \rightarrow \rho L_1(\rho,\rho_0,v_0)=\rho(v_0+p(\rho_0)-p(\rho))$$
is strictly decreasing if and only if the inequality
\begin{equation}\label{dim_condizione_curve_decrescenti}
\dfrac{d}{d\rho}(\rho L_1(\rho,\rho_0,v_0))<0
\end{equation}
holds. We have
\begin{equation*}
\begin{split}
& \dfrac{d}{d\rho}(\rho L_1(\rho,\rho_0,v_0))<0 \Longleftrightarrow v_0+p(\rho_0)-p(\rho)-\rho p'(\rho) <0 \Longleftrightarrow\\
& p(\rho)+\rho p'(\rho) > v_0+p(\rho_0) \Longleftrightarrow \varphi(\rho)>v_0+p(\rho_0).
\end{split}
\end{equation*}
By Lemma \ref{convex_concave}, the function $\rho \rightarrow \rho p(\rho)$ is strictly convex, hence its derivative $\varphi$ is strictly increasing. Therefore, if the inequality (\ref{cond_curve_decrescenti_2}) holds, the condition (\ref{dim_condizione_curve_decrescenti}) is satisfied for every point $(\rho,v)\in \mathcal{D}_{v_1,v_2,w_1,w_2}$ such that $v+p(\rho) = v_0+p(\rho_0)$, because $\rho>\rho_\text{min}$ by Proposition \ref{prop_min_and_max_density_invariant_domain_RS_q_2}.\\
Vice versa, if the inequality (\ref{dim_condizione_curve_decrescenti}) holds for every $(\rho,v)\in\mathcal{D}_{v_1,v_2,w_1,w_2}$, then it holds in particular for $(\rho_0,v_0)=(\rho_\text{min},v_\text{min})$, i.e.
$$\dfrac{d}{d\rho}\rho\,L_1(\rho,\rho_\text{min},v_\text{min})<0 \;\text{ for every } \; (\rho,v)\in\mathcal{D}_{v_1,v_2,w_1,w_2}.$$
Since $(\rho_\text{min},v_\text{min})\in \mathcal{D}_{v_1,v_2,w_1,w_2}$, we have
$$\dfrac{d}{d\rho}\rho\,L_1(\rho,\rho_\text{min},v_\text{min})\Big|_{\rho=\rho_\text{min}}<0$$
which is equivalent to
$$\lambda_1(\rho_\text{min},v_\text{min}) <0,$$
because, by Proposition \ref{eigenvalue_as_slope}, we have
$$\dfrac{d}{d\rho}\rho\,L_1(\rho,\rho_\text{min},v_\text{min})\Big|_{\rho=\rho_\text{min}} = \lambda_1(\rho_\text{min},v_\text{min})$$.
\end{proof}
Let us recall two properties of concave functions.
\begin{lemma}\label{lemma_concave_functions_properties}
Let us fix a constant $q\in \mathbb{R}^+$ and a point $(\rho^\gamma,v^\gamma)$ in $\mathbb{R}^+\times \mathbb{R}^+$ such that
\begin{equation}
\rho^\gamma\,v^\gamma >q.
\end{equation}
Let us consider a point $(\rho^\delta,v^\delta)\in \mathbb{R}^+\times\mathbb{R}^+$ such that $v^\delta+p(\rho^\delta)=v^\gamma+p(\rho^\gamma)$.
Then the following statements hold.
\begin{enumerate}
\item[(i)] $\rho^\delta\, v^\delta \leq q$ if and only if
$$\rho^\delta\leq \check{\rho}_1 \; \text{ or } \; \rho^\delta \geq \hat{\rho}.$$
\item[(ii)] Suppose that $\rho^\delta\, v^\delta \leq q$. Then $\rho^\delta \leq \check{\rho}_1$ if and only if $\lambda_1(\rho^\delta,v^\delta)>0$, while $\rho^\delta \geq \hat{\rho}$ if and only if $\lambda_1(\rho^\delta,v^\delta)<0$.
\end{enumerate}
\end{lemma}
\begin{proof}
By Lemma \ref{esistenza_rho_hat_rho_check_1}, the hypothesis $\rho^\gamma\,v^\gamma >q$ ensures that the points $(\check{\rho}_1,\check{v}_1)$ and $(\hat{\rho},\hat{v})$ exist.\\
To prove (i) is sufficient to apply Lemma \ref{funzione_concava_e_retta} with $\bar{V}=0$ and $F_\alpha=q$.\\
To prove (ii), let us observe that by Lagrange theorem there exists $\rho^\text{sup}\in (\check{\rho}_1,\hat{\rho})$ such that
$$\lambda_1(\rho^\text{sup},L_1(\rho^\text{sup},\rho^\gamma,v^\gamma))=\dfrac{d}{d\rho}(\rho\,L_1(\rho,\rho^\gamma,v^\gamma))\Big|_{\rho=\rho^\text{sup}}=\dfrac{\hat{\rho}\,\hat{v}-\check{\rho}_1\,\check{v}_1}{\hat{\rho}-\check{\rho}_1}=0,$$
because $\hat{\rho}\, \hat{v}=\check{\rho}_1\,\check{v}_1=q$.\\
Assume $\rho^\delta \leq \check{\rho}_1$. By Proposition \ref{eigenvalue_as_slope}, the function $\rho\to \lambda_1(\rho,L_1(\rho,\rho^\gamma,v^\gamma)$ is strictly decreasing. Hence if $\rho^\alpha \leq \check{\rho}_1$, then
$$\lambda_1(\rho^\delta,v^\delta) \geq \lambda_1(\check{\rho}_1,\check{v}_1) > \lambda_1(\rho^\text{sup},L_1(\rho^\text{sup},\rho^\gamma,v^\gamma))=0.$$
Similarly for the case $\rho^\delta\geq \hat{\rho}$, we find $\lambda_1(\rho^\delta,v^\delta)<0$.\\
For the vice versa, since $\rho^\delta\, v^\delta \leq q$, we find
$$\rho^\delta\leq \check{\rho}_1 \; \text{ or } \; \rho^\delta \geq \hat{\rho}$$
by point (i). If $\lambda_1(\rho^\delta,v^\delta)>0$ and it was $\rho^\delta\geq \hat{\rho}$, then $\lambda_1(\hat{\rho},\hat{v}) \geq \lambda_1(\rho^\delta,v^\delta)>0$. This contradicts $\hat{\rho}>\rho^\text{sup}$.
\end{proof}
The next lemma gives a relation between the order of the densities of two points on the line $\rho v=q$ and the corresponding values of the Riemann invariant $w$.
\begin{lemma}\label{lemma_riemann_invariant_property}
Fix $q\in \mathbb{R}^+$ and let us suppose $\lambda_1(\rho_\text{min},v_\text{min})<0$. Fix $(\rho^\alpha,v^\alpha)\in \mathcal{D}_{v_1,v_2,w_1,w_2}$ and $(\rho^\beta,v^\beta)\in \mathcal{D}_{v_1,v_2,w_1,w_2}$ such that
$$\rho^\alpha \, v^\alpha=\rho^\beta\,v^\beta=q.$$
We have
$$v^\alpha+p(\rho^\alpha)\leq v^\beta+p(\rho^\beta) \; \text{ if and only if }  \; \rho^\alpha \leq \rho^\beta.$$
\end{lemma}
\begin{proof}
Let us call $(\rho^*,v^*)$ the solution to the system (see Figure \ref{fig_lemma_riemann_invariant_property})
$$\begin{cases}
v=v^\beta,\\
v+p(\rho)=v^\alpha+p(\rho^\alpha).
\end{cases}
$$
By Lemma \ref{lemma_concave_functions_properties}, we find $\rho^\alpha\leq\rho^*$, because $\lambda_1(\rho^\alpha,v^\alpha)<0$ by Lemma \ref{lemma_curve_decrescenti}.\\
Since $v^* =v^\beta$ and $v^*+p(\rho^*)=v^\alpha+p(\rho^\alpha)$, we have
$$v^\alpha+p(\rho^\alpha)=v^*+p(\rho^*)\leq v^\beta+p(\rho^\beta) \Longleftrightarrow p(\rho^*) \leq p(\rho^\beta)\Longleftrightarrow \rho^* \leq \rho^\beta,$$
because $\rho \to p(\rho)$ is increasing by the hypotheses (\ref{ipotesi_pressione_forti}). Therefore $\rho^*\,v^*\leq \rho^\beta\, v^\beta =q$.
\end{proof}
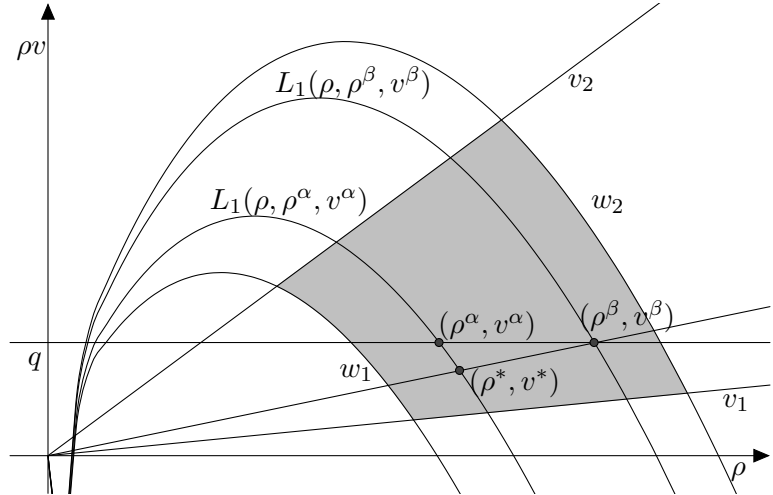
\begin{figure}[hbtp]
\centering
\definecolor{cqcqcq}{rgb}{0.7529411764705882,0.7529411764705882,0.7529411764705882}
\definecolor{uququq}{rgb}{0.25098039215686274,0.25098039215686274,0.25098039215686274}
\begin{tikzpicture}[scale=0.5,line cap=round,line join=round,>=triangle 45,x=1.0cm,y=1.0cm]
\draw[->,color=black] (-1.,0.) -- (19.,0.);
\draw[->,color=black] (0.,-1.) -- (0.,12.);
\clip(-1.,-1.) rectangle (19.,12.);
\draw[line width=0.pt,color=cqcqcq,fill=cqcqcq,fill opacity=1.0] {[smooth,samples=50,domain=6.02:9.62] plot(\x,{0-(-6.52/8.74)*\x-0.0/8.74})} -- (9.62,0.94) {[smooth,samples=50,domain=9.62:6.02] -- plot(\x,{3.2*\x-\x^(1.5)})} -- (6.02,4.49) -- cycle;
\draw[line width=0.pt,color=cqcqcq,fill=cqcqcq,fill opacity=1.0] {[smooth,samples=50,domain=9.6:11.93] plot(\x,{0-(-6.52/8.74)*\x-0.0/8.74})} -- (11.93,1.17) {[smooth,samples=50,domain=11.93:9.6] -- plot(\x,{0-(-1.50/15.22)*\x-0.0/15.22})} -- (9.6,7.16) -- cycle;
\draw[line width=0.pt,color=cqcqcq,fill=cqcqcq,fill opacity=1.0] {[smooth,samples=50,domain=11.9:16.82] plot(\x,{4.2*\x-\x^(1.5)})} -- (16.82,1.66) {[smooth,samples=50,domain=16.82:11.9] -- plot(\x,{0-(-1.50/15.22)*\x-0.0/15.22})} -- (11.9,8.92) -- cycle;
\draw[smooth,samples=50,domain=0.001:19.0] plot(\x,{4.0*(\x)-(\x)^(1.5)});
\draw[smooth,samples=50,domain=0.001:19.0] plot(\x,{3.5*(\x)-(\x)^(1.5)});
\draw [domain=-1.:19.] plot(\x,{(--3.-0.*\x)/1.});
\draw [domain=0.0:19.0] plot(\x,{(-0.--2.99*\x)/14.373863542464552});
\draw (-1.1,11.26) node[anchor=north west] {$\rho v$};
\draw (17.7,0.1) node[anchor=north west] {$\rho$};
\draw (10.00,4.10) node[anchor=north west] {$(\rho^\alpha,v^\alpha)$};
\draw (13.7,4.4) node[anchor=north west] {$(\rho^\beta,v^\beta)$};
\draw (10.8,2.6) node[anchor=north west] {$(\rho^*,v^*)$};
\draw (3.99,7.4) node[anchor=north west] {$L_1(\rho,\rho^\alpha,v^\alpha)$};
\draw (5.7,10.59) node[anchor=north west] {$L_1(\rho,\rho^\beta,v^\beta)$};
\draw [domain=0.0:19.0] plot(\x,{(-0.--1.50*\x)/15.22});
\draw [domain=0.0:19.0] plot(\x,{(-0.--6.52*\x)/8.74});
\draw[smooth,samples=50,domain=0.001:19.0] plot(\x,{3.2*(\x)-(\x)^(1.5)});
\draw[smooth,samples=50,domain=0.001:19.0] plot(\x,{4.2*(\x)-(\x)^(1.5)});
\draw (13.41,10.38) node[anchor=north west] {$v_2$};
\draw (17.48,1.9) node[anchor=north west] {$v_1$};
\draw (-0.8,3.) node[anchor=north west] {$q$};
\draw (14.02,7.16) node[anchor=north west] {$w_2$};
\draw (7.4,2.7) node[anchor=north west] {$w_1$};
\begin{scriptsize}
\draw [fill=uququq] (10.29,3.) circle (3pt);
\draw [fill=uququq] (14.37,3) circle (3pt);
\draw [fill=uququq] (10.83,2.26) circle (3pt);
\end{scriptsize}
\end{tikzpicture}
\caption{Notation used in the proof of Lemma \ref{lemma_riemann_invariant_property}. The coloured area is the invariant domain.}\label{fig_lemma_riemann_invariant_property}
\end{figure}
Next, we state some useful relations between the points of interaction between the Lax curves of the first and the second family and the line of the constraint.
\begin{lemma}\label{lemma_disuguaglianze_triangolo}
Assume that $\lambda_1(\rho_\text{min},v_\text{min})<0$ and fix $q\in \mathbb{R}^+$. Let us consider three points $(\rho^\alpha,v^\alpha)$, $(\rho^\beta,v^\beta)$ and $(\rho^\gamma,v^\gamma)$ in the invariant domain $\mathcal{D}_{v_1,v_2,w_1,w_2}$ such that
$$v^\beta+p(\rho^\beta) > v^\alpha+p(\rho^\alpha), \; \; v^\gamma=v^\alpha \; \text{ and } \; v^\beta+p(\rho^\beta)=v^\gamma+p(\rho^\gamma).$$
Moreover let us suppose that
$$\rho^\alpha\,v^\alpha =\rho^\beta \, v^\beta= q.$$
Then there exist three positive constants $c_1$, $c_2$ and $c_3$ depending only on $q$ and the invariant domain $\mathcal{D}_{v_1,v_2,w_1,w_2}$ for which:
\begin{enumerate}
\item[(i)] $\rho^\beta-\rho^\gamma\leq c_1\,(\rho^\gamma-\rho^\alpha)$;
\item[(ii)] $\rho^\gamma-\rho^\alpha\leq c_2\, (\rho^\beta-\rho^\gamma)$\;
\item[(iii)] $v^\gamma-v^\beta\leq c_3\,(\rho^\gamma-\rho^\alpha)$.
\end{enumerate}
\end{lemma}
\begin{figure}[ht]
\centering
\definecolor{xdxdff}{rgb}{0.49019607843137253,0.49019607843137253,1.}
\definecolor{uququq}{rgb}{0.25098039215686274,0.25098039215686274,0.25098039215686274}
\definecolor{qqqqff}{rgb}{0.,0.,1.}
\begin{tikzpicture}[scale =30,line cap=round,line join=round,>=triangle 45,x=1.0cm,y=1.0cm]
\clip(1.285,1.08) rectangle (1.57,1.21);
\draw[smooth,samples=50,domain=1.285:1.57] plot(\x,{3.0*(\x)-(\x)^(3.0)});
\draw[smooth,samples=50,domain=1.285:1.57] plot(\x,{2.0*(\x)-(\x)^(3.0)});
\draw (0.,1.087) -- (0.,1.21);
\draw (0.,1.087) -- (0.,1.21);
\draw [domain=1.285:1.57] plot(\x,{(--1.1-0.*\x)/1.});
\draw (0.,1.087) -- (0.,1.21);
\draw[smooth,samples=50,domain=1.285:1.57] plot(\x,{2.85*(\x)-(\x)^(3.0)});
\draw (1.37,1.19) node[anchor=north west] {$ (\rho^\gamma,v^\gamma) $};
\draw (1.40,1.10) node[anchor=north west] {$(\rho^\beta,v^\beta)$};
\draw (1.29,1.1) node[anchor=north west] {$(\rho^\alpha,v^\alpha)$};
\draw (1.45,1.10) node[anchor=north west] {$(\rho^{**},v^{**})$};
\draw (1.44,1.16) node[anchor=north west] {$k_2$};
\draw (1.42,1.14) node[anchor=north west] {$\mu$};
\draw (1.385,1.14) node[anchor=north west] {$v_2$};
\draw (1.34,1.10) node[anchor=north west] {$(\rho^*,v^*)$};
\draw (1.36,1.1)-- (1.42,1.18);
\draw (1.42,1.18)-- (1.47,1.1);
\draw (1.32,1.1) --(1.42,1.18);
\draw (1.52,1.10) node[anchor=north west] {$\rho v=q$};
\begin{scriptsize}
\draw [fill=qqqqff] (1.2,1.) circle (0.05pt);
\draw [fill=uququq] (0.,0.) circle (0.05pt);
\draw [fill=xdxdff] (1.42,1.18) circle (0.05pt);
\draw [fill=uququq] (1.22,0.61) circle (0.05pt);
\draw [fill=uququq] (1.,1.) circle (0.05pt);
\draw [fill=uququq] (1.58,0.79) circle (0.05pt);
%\draw [fill=uququq] (1.6,1.41) circle (0.05pt);
\draw [fill=uququq] (1.445,1.1) circle (0.05pt);
\draw [fill=uququq] (1.32,1.1) circle (0.05pt);
\draw [fill=xdxdff] (1.47,1.1) circle (0.05pt);
\draw [fill=xdxdff] (1.36,1.1) circle (0.05pt);
\end{scriptsize}
\end{tikzpicture}
\caption{Notations used in the proof of Lemma \ref{lemma_disuguaglianze_triangolo} for the inequality $\rho^\beta-\rho^\gamma\leq c_1\,(\rho^\gamma-\rho^\alpha)$.}\label{fig_proof_stime_caso_1}
\end{figure}
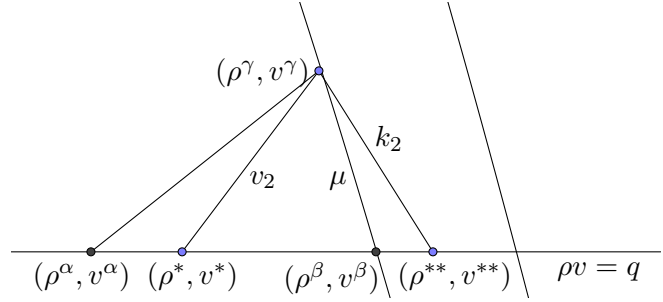
\begin{proof}
Since $v^\gamma+p(\rho^\gamma)=v^\beta+p(\rho^\beta)>v^\alpha+p(\rho^\alpha)$ and $v^\gamma=v^\alpha$, we find
$$\rho^\gamma>\rho^\alpha \Longrightarrow \rho^\gamma v^\gamma >\rho^\alpha v^\alpha=q.$$
By Lemma \ref{lemma_riemann_invariant_property} and since $\rho^\alpha v^\alpha=\rho^\beta v^\beta$, we obtain
$$\rho^\beta >\rho^\alpha.$$
Moreover since $\rho^\gamma\,v^\gamma>q=\rho^\beta\,v^\beta$ and $v^\beta+p(\rho^\beta)=v^\gamma+p(\rho^\gamma)$, by Lemma \ref{lemma_concave_functions_properties} we have
$$\rho^\beta>\rho^\gamma.$$
Let us denote
$$\epsilon= \rho^\gamma v^\gamma-q, \; \; k_2 = \lambda_1(\rho_\text{min},v_\text{min})\; \text{ and } \; k_1= \lambda_1(\rho_\text{max},v_\text{max}).$$
Let $(\rho^*,v^*)$ be the solution to the system
\begin{equation*}
\begin{cases}
\rho v=\rho^\gamma\, v^\gamma + v_2(\rho-\rho^\gamma),\\
\rho v=q,
\end{cases}
\end{equation*}
i.e. the intersection between the line passing through $(\rho^\gamma,v^\gamma)$ with slope $v_2$ and the line $\rho v =q$; see Figure \ref{fig_proof_stime_caso_1}. The definition implies:
\begin{equation}\label{caratteristiche_u_star}
\begin{split}
& \rho^\gamma v^\gamma = q-v_2 \, (\rho^*-\rho^\gamma)\Longrightarrow \\
& \rho^\gamma-\rho^*=\dfrac{\rho^\gamma\, v^\gamma-q}{v_2} = \dfrac{\epsilon}{v_2}\Longleftrightarrow v_2\,(\rho^\gamma-\rho^*)=\epsilon.
\end{split}
\end{equation}
In particular we obtain
$$\rho^\gamma >\rho^*,$$
because $\epsilon/v_2>0$. Moreover
$$\rho^*\geq \rho^\alpha,$$
because, by (\ref{caratteristiche_u_star}) and $\epsilon= \rho^\gamma \,v^\gamma -q = \rho^\gamma \, v^\alpha - \rho^\alpha \, v^\alpha$, we have
\begin{equation*}
v_2\geq v^\alpha \Rightarrow \dfrac{\epsilon}{v^\alpha}\geq\dfrac{\epsilon}{v_2} \Rightarrow \rho^\alpha=\rho^\gamma-\dfrac{\epsilon}{v^\alpha}\leq \rho^\gamma-\dfrac{\epsilon}{v_2}=\rho^*.
\end{equation*}
Therefore
\begin{equation}\label{rocheck}
v_2(\rho^\gamma-\rho^\alpha) \geq v_2(\rho^\gamma-\rho^*)=\epsilon,
\end{equation}
where we have used again the condition (\ref{caratteristiche_u_star}).\\
Let $(\rho^{**},v^{**})$ be the solution to the system
\begin{equation*}
\begin{cases}
\rho v=\rho^\gamma\, v^\gamma+k_2(\rho-\rho^\gamma),\\
\rho v=q,
\end{cases}
\end{equation*}
i.e. the intersection between the line passing through $(\rho^\gamma,v^\gamma)$ with slope $k_2$; see Figure \ref{fig_proof_stime_caso_1}.\\
By the definition we have
$$\rho^{**}=\rho^\gamma+\dfrac{\epsilon}{-k_2}.$$
Let $\mu=\dfrac{\rho^\gamma\, v^\gamma -q}{\rho^\gamma -\rho^\beta}$ be the slope of the line passing through the points $(\rho^\gamma, v^\gamma)$ and $(\rho^\beta,v^\beta)$. By Lemma \ref{lemma_curve_lipschitz_1}, we have
$$\mu \leq k_2<0.$$
Hence $\rho^\beta\leq \rho^{**}$, indeed
$$-\mu \geq -k_2 \Rightarrow \dfrac{\epsilon}{-k_2} \geq \dfrac{\epsilon}{-\mu}\Rightarrow \rho^{**}=\rho^\gamma+\dfrac{\epsilon}{-k_2} \geq \rho^\gamma+\dfrac{\epsilon}{-\mu}=\rho^\beta.$$
Therefore
\begin{equation}\label{rohat}
\epsilon = -k_2(\rho^{**}-\rho^\gamma)= -\mu (\rho^\beta-\rho^\gamma)\geq -k_2\, (\rho^\beta-\rho^\gamma).
\end{equation}
The inequalities (\ref{rocheck}) and (\ref{rohat}), imply
\begin{equation*}
\begin{split}
&-k_2(\rho^\beta-\rho^\gamma)\leq \epsilon \leq v_2(\rho^\gamma-\rho^\alpha) \Rightarrow\\
& \Rightarrow \rho^\beta-\rho^\gamma\leq c_1(\rho^\gamma-\rho^\alpha),
\end{split}
\end{equation*}
where $c_1=v_2/(-k_2) > 0$.\\
\begin{figure}[hbtp]
\centering
\definecolor{xdxdff}{rgb}{0.49019607843137253,0.49019607843137253,1.}
\definecolor{uququq}{rgb}{0.25098039215686274,0.25098039215686274,0.25098039215686274}
\definecolor{qqqqff}{rgb}{0.,0.,1.}
\begin{tikzpicture}[scale=30,line cap=round,line join=round,>=triangle 45,x=1.0cm,y=1.0cm]
\clip(1.245,1.08) rectangle (1.53,1.2);
\draw[smooth,samples=50,domain=1.245:1.53] plot(\x,{2.0*(\x)-(\x)^(3.0)});
\draw (0.,1.08) -- (0.,1.2);
\draw (0.,1.08) -- (0.,1.2);
\draw [domain=1.245:1.53] plot(\x,{(--1.1-0.*\x)/1.});
\draw (0.,1.08) -- (0.,1.2);
\draw[smooth,samples=50,domain=1.245:1.53] plot(\x,{2.85*(\x)-(\x)^(3.0)});
\draw (1.372,1.195) node[anchor=north west] {$(\rho^\gamma,v^\gamma)$};
\draw (1.445,1.101) node[anchor=north west] {$(\rho^\beta,v^\beta)$};
\draw (1.31,1.10) node[anchor=north west] {$(\rho^\alpha,v^\alpha)$};
\draw (1.401,1.10) node[anchor=north west] {$(\rho^{o},v^{o})$};
\draw (1.41,1.13) node[anchor=north west] {$k_1$};
\draw (1.435,1.14) node[anchor=north west] {$\mu$};
\draw (1.32,1.15) node[anchor=north west] {$v_1$};
\draw (1.24,1.10) node[anchor=north west] {$(\rho^{oo},v^{oo})$};
\draw (1.26,1.1)-- (1.42,1.18);
\draw (1.42,1.18)-- (1.43,1.1);
\draw (1.32,1.1) -- (1.42,1.18);
\draw (1.47,1.115) node[anchor=north west] {$\rho v=q$};
\begin{scriptsize}
\draw [fill=qqqqff] (1.2,1.) circle (0.05pt);
\draw [fill=uququq] (0.,0.) circle (0.05pt);
\draw [fill=xdxdff] (1.42,1.18) circle (0.05pt);
\draw [fill=uququq] (1.22,0.61) circle (0.05pt);
\draw [fill=uququq] (1.445,1.1) circle (0.05pt);
\draw [fill=uququq] (1.32,1.1) circle (0.05pt);
\draw [fill=xdxdff] (1.43,1.1) circle (0.05pt);
\draw [fill=xdxdff] (1.26,1.1) circle (0.05pt);
\end{scriptsize}
\end{tikzpicture}
\caption{Notations used in the proof of Lemma \ref{lemma_disuguaglianze_triangolo} for the inequality $(\rho^\gamma-\rho^\alpha) \leq c_2\,(\rho^\beta-\rho^\gamma)$.}\label{fig_dim_stime_interazione_caso_1_sec_dis}
\end{figure}
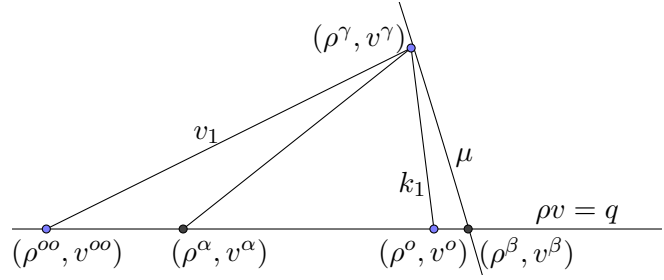

For the inequality
$$(\rho^\gamma-\rho^\alpha) \leq c_2\,(\rho^\beta-\rho^\gamma)$$
let us define the point $(\rho^o,v^o)$ as the solution to the system
\begin{equation*}
\begin{cases}
\rho v=\rho^\gamma v^\gamma+k_1(\rho-\rho^\gamma)\\
\rho v=q,
\end{cases}
\end{equation*}
i.e. $(\rho^o,v^o)$ is the intersection between the line $\rho v=q$ and the line passing through $(\rho^\gamma, v^\gamma)$ with slope $k_1$; see Figure \ref{fig_dim_stime_interazione_caso_1_sec_dis}.\\
By the definition, we have
$$ \rho^o = \rho^\gamma+\dfrac{\epsilon}{-k_1}>\rho^\gamma$$
because $-\epsilon/k_1>0$. Therefore
$$\rho^o-\rho^\gamma=\dfrac{\epsilon}{-k_1}\Rightarrow \epsilon = -k_1(\rho^o-\rho^\gamma).$$
Moreover, by Lemma \ref{lemma_curve_lipschitz_1}, we have $k_1\leq \mu=\epsilon/(\rho^\gamma-\rho^\beta)$. Hence
\begin{equation*}
\begin{split}
& -k_1 \geq -\mu \Rightarrow \dfrac{\epsilon}{-\mu} \geq \dfrac{\epsilon}{-k_1} \Rightarrow \rho^\beta-\rho^\gamma\geq \rho^o-\rho^\gamma \Rightarrow\\
&\Rightarrow -k_1(\rho^\beta-\rho^\gamma)\geq -k_1(\rho^o-\rho^\gamma)=\epsilon.
\end{split}
\end{equation*}
Let us now define the point $(\rho^{oo}, v^{oo})$ as the solution to the system
\begin{equation*}
\begin{cases}
\rho v=q\\
\rho v=\rho^\gamma v^\gamma+v_1(\rho-\rho^\gamma),
\end{cases}
\end{equation*}
i.e. the point of intersection between the line $\rho v=q$ and the line passing through $(\rho^\gamma, v^\gamma)$ with slope $v_1$; see Figure \ref{fig_dim_stime_interazione_caso_1_sec_dis}.\\
By the definition, we obtain
$$\rho^{oo}=\rho^\gamma-\dfrac{\epsilon}{v_1}.$$
Since $v^\alpha \geq v_1$, we find
$$\dfrac{\epsilon}{v_1} \geq \dfrac{\epsilon}{v^\alpha} \Rightarrow \rho^\gamma-\dfrac{\epsilon}{v^\alpha}\geq \rho^\gamma -\dfrac{\epsilon}{v_1} \Rightarrow \rho^\alpha \geq \rho^{oo}$$
because $\rho^\alpha=\rho^\gamma-\dfrac{\epsilon}{v^\alpha}$. Therefore $v_1(\rho^\gamma-\rho^\alpha) \leq v_1(\rho^\gamma-\rho^{oo})=\epsilon$.\\
Hence
\begin{equation*}
\begin{split}
& v_1(\rho^\gamma-\rho^\alpha) \leq  -k_1(\rho^\beta-\rho^\gamma) \Rightarrow \\
&\Rightarrow \rho^\gamma-\rho^\alpha \leq c_2(\rho^\beta-\rho^\gamma),
\end{split}
\end{equation*}
where $c_2=-k_1/v_1$.\\
For the last inequality, since $\rho^\alpha\, v^\alpha=\rho^\beta\, v^\beta=q$ and $v^\alpha=v^\gamma$, we find
$$v^\alpha-v^\beta=v^\alpha-\dfrac{q}{\rho^\beta}=v^\alpha\left(1-\dfrac{q}{\rho^\beta\, v^\alpha}\right)=\dfrac{v^\alpha}{\rho^\beta}\left(\rho^\beta-\dfrac{q}{v^\alpha}\right)=\dfrac{v^\alpha}{\rho^\beta}\,(\rho^\beta-\rho^\alpha).$$
Since $(\rho^\gamma,v^\gamma)$ and $(\rho^\alpha,v^\alpha)$ are in the invariant domain $\mathcal{D}_{v_1,v_2,w_1,w_2}$, we have $v^\alpha=v^\gamma\leq v_2$. Moreover $(\rho^\beta,v^\beta)\in \mathcal{D}_{v_1,v_2,w_1,w_2}$ implies $\rho^\beta\geq \rho_\text{min}$ by Proposition \ref{prop_min_and_max_density_invariant_domain_RS_q_2}. Therefore
$$\dfrac{v^\alpha}{\rho^\beta}\leq \dfrac{v_2}{\rho_\text{min}}.$$
Since $\rho^\alpha<\rho^\gamma <\rho^\beta$, we find
$$\rho^\beta-\rho^\alpha\leq (\rho^\beta-\rho^\gamma)+(\rho^\gamma-\rho^\alpha).$$
By the inequalities $(i)$ and $(ii)$, we find $(\rho^\beta-\rho^\gamma) \leq c_1\, (\rho^\gamma-\rho^\alpha)$ and
$$(\rho^\gamma-\rho^\alpha)\leq c_2\,(\rho^\beta-\rho^\gamma)\leq \dfrac{c_1}{c_2}\,(\rho^\gamma-\rho^\alpha).$$
Therefore
$$\rho^\beta-\rho^\alpha\leq c_1\,(\rho^\gamma-\rho^\alpha)+\dfrac{c_1}{c_2}(\rho^\gamma-\rho^\alpha)\leq c_3 \, (\rho^\gamma-\rho^\alpha),$$
where
$$c_3 =c_1\, \left(1+\dfrac{1}{c_2}\right).$$
\end{proof}
Finally, let us show that the pressure function in Lipschitz in every bounded interval of $\mathbb{R}$.
\begin{lemma}\label{pressure_Lipschitz}
Consider a set $[\rho_1,\rho_2]\subset \mathbb{R}$. The function $\rho\to p(\rho)$ is bi-Lipschitz for every $\rho \in [\rho_1,\rho_2]$, i.e. there exist two positive constants $C_1$ and $C_2$ for which
$$C_1 \leq |p'(\rho)|\leq C_2.$$
\end{lemma}
\begin{proof}
By the hypotheses (\ref{ipotesi_pressione_forti}), the function $\rho \to p(\rho)$ is convex. Therefore
$$p'(\rho_1)\leq |p'(\rho)| \leq p'(\rho_2) \; \text{ for every } \; \rho \in [\rho_1,\rho_2].$$
We have the thesis with
$$C_1:=p'(\rho_1) \; \text{ and }\; C_2:=p'(\rho_2).$$
\end{proof}
In the following we fix an invariant domain $\mathcal{D}_{v_1,v_2,w_1,w_2}$ for the Riemann solver $\mathcal{RS}^q_2$.
\subsection{A contact discontinuity interacts with $x=0$}
Fix $\bar{x}<0$ and three points $(\rho^l,v^l)$, $(\rho^k,v^k)$ and $(\rho^r,v^r)$ in the invariant domain $\mathcal{D}_{v_1,v_2,w_1,w_2}$ such that
$$v^l=v^k.$$
Assume that the hypotheses of Lemmas \ref{lemma_curve_decrescenti} and \ref{lemma_curve_lipschitz_1} are satisfied, so that for every $(\rho_0,v_0)\in \mathcal{D}_{v_1,v_2,w_1,w_2}$ the function $\rho \to \rho\,L_1(\rho,\rho_0,v_0)$ is strictly decreasing and bi-Lipschitz inside the domain $\mathcal{D}_{v_1,v_2,w_1,w_2}$.\\
Let us consider the Riemann problems (\ref{Riemann_problem}) centred in $\bar{x}$ and $x=0$ with initial data
\begin{equation}\label{initial_data_contact_discontinuity_form_left}
(\rho,v)(0,x) = \begin{cases}
(\rho^l,v^l) & \text{if } x\leq \bar{x},\\
(\rho^k,v^k) & \text{if } x> \bar{x}.
\end{cases}
\; \text{ and } \;
(\rho,v)(0,x)= \begin{cases}
(\rho^k,v^k) & \text{if } x\leq 0,\\
(\rho^r,v^r) & \text{if } x>0.
\end{cases}
\end{equation}
Since $v^l=v^k$, the solution to the Riemann problem centred in $\bar{x}$ is a contact discontinuity travelling with speed $v^k$.
\subsubsection{Case $(\rho^k,v^k)=(\rho^r,v^r)$}
Let us suppose that $(\rho^k,v^k)=(\rho^r,v^r)$ and that $(\rho^r,v^r)$ satisfies the constraint, i.e.
$$\rho^r\, v^r \leq q.$$
In this case the solution given by the Riemann solver $\mathcal{RS}^q_2$ to the Riemann problem centred in $x=0$ is classical (constant).\\
At time $\tilde{t}= \vert \bar{x} \vert /v^r$ the contact discontinuity reaches the constraint and we have to solve the new Riemann problem (\ref{Riemann_problem}) with initial datum
\begin{equation*}
(\rho,v)(\tilde{t},x) = \begin{cases}
(\rho^l,v^l) & \text{if } x\leq 0,\\
(\rho^r,v^r) & \text{if } x>0.
\end{cases}
\end{equation*}
Since $v^l = v^k =v^r$, the standard solution is a new contact discontinuity, so that
$$RS((\rho^l,v^l),(\rho^r,v^r)(0)=(\rho^l,v^l).$$
Hence if $\rho^l \, v^l \leq q$, then the constraint is satisfied and the solution is classical. In this case the total variation before and after the interaction remains unchanged. Otherwise the constraint is enforced and the non-classical shock appears in $x=0$; see Figure \ref{fig_interaction_1}.\\

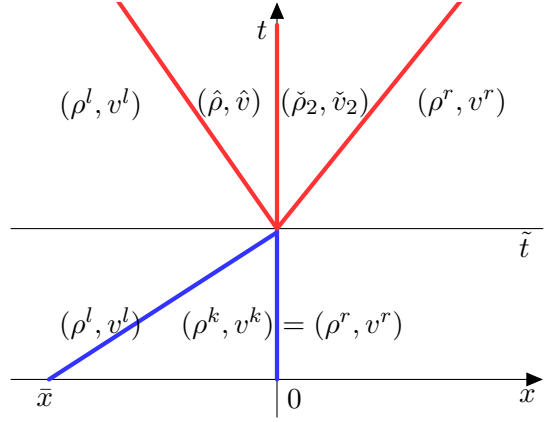
\begin{figure}[ht]
\centering
\definecolor{fftttt}{rgb}{1.,0.2,0.2}
\definecolor{ttttff}{rgb}{0.2,0.2,1.}
\begin{tikzpicture}[line cap=round,line join=round,>=triangle 45,x=1.0cm,y=1.0cm]
\draw[->,color=black] (-3.5,0.) -- (3.5,0.);
\draw[->,color=black] (0.,-0.5) -- (0.,5.);
\clip(-3.5,-0.5) rectangle (3.5,5.);
\draw [domain=-3.5:3.5] plot(\x,{(--2.-0.*\x)/1.});
\draw [line width=1.6pt,color=ttttff] (-3.,0.)-- (0.,1.95);
\draw [line width=1.6pt,color=fftttt,domain=-3.5:0.0] plot(\x,{(-7.32--5.28*\x)/-3.66});
\draw [line width=1.6pt,color=fftttt,domain=0.0:3.5] plot(\x,{(--6.88--4.3*\x)/3.44});
\draw (3.05,2.09) node[anchor=north west] {$\tilde{t}$};
\draw (-3,4.) node[anchor=north west] {$(\rho^l,v^l)$};
\draw (-3,1.09) node[anchor=north west] {$(\rho^l,v^l)$};
\draw (-0.1,4.) node[anchor=north west] {$(\check{\rho}_2,\check{v}_2)$};
\draw (-1.2,4.) node[anchor=north west] {$(\hat{\rho},\hat{v})$};
\draw (1.70,4) node[anchor=north west] {$(\rho^r,v^r)$};
\draw (-1.4,1.09) node[anchor=north west] {$(\rho^k,v^k)=(\rho^r,v^r)$};
\draw [line width=1.6pt,color=ttttff] (0.,2.)-- (0.,0.);
\draw [line width=1.6pt,color=fftttt] (0.,2.) -- (0.,4.7);
\draw (-0.4,4.88) node[anchor=north west] {$t$};
\draw (3.05,0.) node[anchor=north west] {$x$};
\draw (-3.3,0.) node[anchor=north west] {$\bar{x}$};
\draw (0.,0.) node[anchor=north west] {$0$};
\end{tikzpicture}
\caption{A contact discontinuity reaches the constraint and after the interaction the constraint is not satisfied.}\label{fig_interaction_1}
\end{figure}
\begin{figure}[hbtp]
\centering
\definecolor{cqcqcq}{rgb}{0.7529411764705882,0.7529411764705882,0.7529411764705882}
\definecolor{xdxdff}{rgb}{0.49019607843137253,0.49019607843137253,1.}
\definecolor{uququq}{rgb}{0.25098039215686274,0.25098039215686274,0.25098039215686274}
\definecolor{qqqqff}{rgb}{0.,0.,1.}
\begin{tikzpicture}[scale=0.5,line cap=round,line join=round,>=triangle 45,x=1.0cm,y=1.0cm]
\draw[->,color=black] (-1.5,0.) -- (18.,0.);
\draw[->,color=black] (0.,-1.5) -- (0.,11.5);
\clip(-1.5,-1.5) rectangle (18.,12);
\draw[line width=0.pt,color=cqcqcq,fill=cqcqcq,fill opacity=1.0] {[smooth,samples=50,domain=4.16:6.75] plot(\x,{0-(-9.68/10.08)*\x-0.0/10.08})} -- (6.75,2.71) {[smooth,samples=50,domain=6.75:4.16] -- plot(\x,{3.0*\x-\x^(1.5)})} -- (4.16,4) -- cycle;
\draw[line width=0.pt,color=cqcqcq,fill=cqcqcq,fill opacity=1.0] {[smooth,samples=50,domain=6.74:9.24] plot(\x,{0-(-9.68/10.08)*\x-0.0/10.08})} -- (9.24,3.71) {[smooth,samples=50,domain=9.24:6.74] -- plot(\x,{0-(-4.5/11.18)*\x-0.0/11.18})} -- (6.74,6.47) -- cycle;
\draw[line width=0.pt,color=cqcqcq,fill=cqcqcq,fill opacity=1.0] {[smooth,samples=50,domain=9.23:12.94] plot(\x,{4.0*\x-\x^(1.5)})} -- (12.94,5.20) {[smooth,samples=50,domain=12.94:9.23] -- plot(\x,{0-(-4.5/11.18)*\x-0.0/11.18})} -- (9.23,8.87) -- cycle;
\draw[smooth,samples=50,domain=0.001:18.0] plot(\x,{3.0*(\x)-(\x)^(1.5)});
\draw[smooth,samples=50,domain=0.001:18.0] plot(\x,{4.0*(\x)-(\x)^(1.5)});
\draw [domain=0.0:18.0] plot(\x,{(-0.--4.5*\x)/11.18});
\draw [domain=0.0:18.0] plot(\x,{(-0.--9.68*\x)/10.08});
\draw [domain=-3.:18.] plot(\x,{(--6.-0.*\x)/1.});
\draw (-0.8,6.2) node[anchor=north west] {$q$};
\draw (10.48,11.86) node[anchor=north west] {$v_2$};
\draw (15.94,7.69) node[anchor=north west] {$v_1$};
\draw (4.48,10.55) node[anchor=north west] {$v+p(\rho)=w_2$};
\draw (7.68,2.46) node[anchor=north west] {$v+p(\rho) = w_1$};
\draw (-1.2,11.) node[anchor=north west] {$\rho v$};
\draw (16.5,0.1) node[anchor=north west] {$\rho$};
\draw [domain=0.0:18.0] plot(\x,{(-0.--7.6*\x)/9.1});
\draw[smooth,samples=50,domain=0.001:18.0] plot(\x,{3.85*(\x)-(\x)^(1.5)});
\draw (6.7,6.1) node[anchor=north west] {$(\check{\rho_2},\check{v}_2)$};
\draw (9.3,6.1) node[anchor=north west] {$(\hat{\rho},\hat{v})$};
\draw (5.6,5.2) node[anchor=north west] {$(\rho^r,v^r)$};
\draw (9.23,8.3) node[anchor=north west] {$(\rho^l,v^l)$};
\begin{scriptsize}
\draw [fill=qqqqff] (9.1,7.6) circle (3pt);
\draw [fill=uququq] (7.18,6.) circle (3pt);
\draw [fill=xdxdff] (5.75,4.80) circle (3pt);
\draw [fill=uququq] (10.88,5.99) circle (3pt);
\end{scriptsize}
\end{tikzpicture}
\caption{Representation of the case $v^l=v^r$, $(\rho^k,v^k)=(\rho^r,v^r)$, $(\rho^r,v^r)$ satisfies the constraint and $(\rho^l,v^l)$ does not. The coloured area is the invariant domain.}\label{fig_interazioni_posizione_punti_1}
\end{figure}
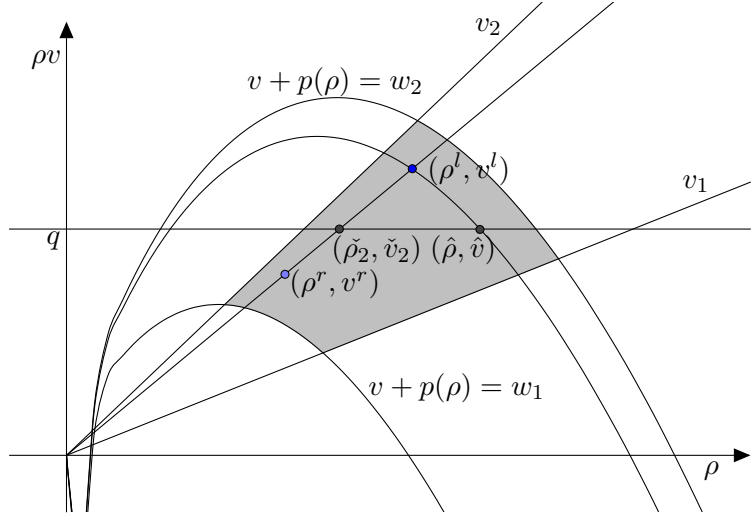
\begin{prop}\label{prop_contact_discontinuity_from_left_case_1}
Assume $\lambda_1(\rho_\text{min},v_\text{min})<0$. Let us consider two points $(\rho^l,v^l)$ and $(\rho^r,v^r)$ in the invariant domain $\mathcal{D}_{v_1,v_2,w_1,w_2}$ such that $\rho^r\,v^r \leq q$. Assume that a wave joining $(\rho^l,v^l)$ to $(\rho^r,v^r)$ with positive speed interacts with $x=0$ at time $\tilde{t}>0$.\\
Then the wave is a contact discontinuity, i.e. $v^r=v^l$.\\
Moreover the following statements hold.
\begin{enumerate}
\item[(i)] If $\rho^l\,v^l\leq q$, then for every $(t,x) \in \mathbb{R}^+\times \mathbb{R}$ we have
\begin{equation}
\mathcal{RS}^q_2((\rho^l,v^l),(\rho^r,v^r))\left(\frac{x}{t-\tilde{t}}\right)=\begin{cases}
(\rho^l,v^l) & \text{if }\;  \frac{x}{t-\tilde{t}}\leq v^l,\\
(\rho^r,v^r) & \text{if }\;  \frac{x}{t-\tilde{t}} >v^l.
\end{cases}
\end{equation}
Therefore 
$$\Delta TV_{\tilde{t}}\,(\rho)=\Delta TV_{\tilde{t}}(v)=0 \; \text{ and } \; \Delta_{\tilde{t}}\,\mathcal{N}=0.$$
In the coordinates $(v,w)$ of the Riemann invariants, we have
$$\Delta TV_{\tilde{t}}\,(v)=0 \; \text{ and } \; \Delta TV_{\tilde{t}}\, (w)=0.$$
\item[(ii)] If $\rho^l\,v^l >q$, then for every $(t,x) \in \mathbb{R}^+\times \mathbb{R}$ we have
\begin{equation}
\mathcal{RS}^q_2((\rho^l,v^l),(\rho^r,v^r))\left(\frac{x}{t-\tilde{t}}\right)=\begin{cases}
(\rho^l,v^l) & \text{if }\;  \frac{x}{t-\tilde{t}}\leq \lambda,\\
(\hat{\rho},\hat{v}) & \text{if }\;  \lambda<\frac{x}{t-\tilde{t}} \leq 0,\\
(\check{\rho}_2,\check{v}_2) & \text{if } \; 0\leq \frac{x}{t-\tilde{t}}\leq v^r,\\
(\rho^r,v^r) & \text{if } \; \frac{x}{t-\tilde{t}}>v^r,
\end{cases}
\end{equation}
where by the Rankine-Hugoniot condition 
$$\lambda=\dfrac{\rho^l\,v^l-\hat{\rho}\, \hat{v}}{\rho^l-\hat{\rho}}$$
is the speed of the shock joining $(\rho^l,v^l)$ to $(\hat{\rho},\hat{v})$. Furthermore
\begin{equation*}
\Delta TV_{\tilde{t}}(\rho) \leq C_1\, (\rho^l-\rho^r), \; \; \Delta TV_{\tilde{t}} (v) \leq C_2\,(\rho^l-\rho^r) \; \text{ and } \; \Delta_{\tilde{t}}\, \mathcal{N}=2,
\end{equation*}
where $C_1$ and $C_2$ are positive constants depending only on $q$ and $\mathcal{D}_{v_1,v_2,w_1,w_2}$.\\
In the coordinates $(v,w)$ of the Riemann invariants, we have
$$\Delta TV_{\tilde{t}}\,(v)\leq C_3\,|w^r-w^l| \; \text{ and } \; \Delta TV_{\tilde{t}}\, (w)=0,$$
where $C_3$ is a positive constant depending only on $q$ and $\mathcal{D}_{v_1,v_2,w_1,w_2}$.
\end{enumerate}
\end{prop}
\begin{proof}
By Lemma \ref{lemma_curve_decrescenti}, the waves of the first family have negative propagation speed. Therefore a wave with positive speed belongs to the second family and hence $v^r=v^l$.\\
In case $(i)$ the solution is classical and at time $\tilde{t}$ do not arise new waves. Then the total variation and the number of waves remain unchanged.\\
In case $(ii)$ the constraint is not satisfied and the non-classical shock appears.\\
We have (see Figure \ref{fig_interazioni_posizione_punti_1})
\begin{equation}\label{disuguaglianze_interazioni_posizione_punti_1}
\rho^r\leq \check{\rho}_2<\rho^l<\hat{\rho},
\end{equation}
indeed
$$\rho^r\, v^r \leq q =\check{\rho}_2\, \check{v}_2 \Rightarrow \rho^r\leq \check{\rho}_2,$$
because $v^r=\check{v}_2$.
Similarly $\check{\rho}_2 < \rho^l$, because $\rho^l \, v^l >q$ and $v^r=v^l$. The inequality $\rho^l<\hat{\rho}$ holds by Lemma \ref{lemma_concave_functions_properties}.\\
The condition $\rho^l<\hat{\rho}$ and Lemma \ref{lemma_curve_decrescenti}, imply that $\mathcal{RS}((\rho^l,v^l),(\hat{\rho},\hat{v}))$ is a shock of the first family with negative propagation speed.\\
By (\ref{disuguaglianze_interazioni_posizione_punti_1}), the solution is given by a shock joining $(\rho^l,v^l)$ to $(\hat{\rho},\hat{v})$, followed by the non-classical shock connecting $(\hat{\rho},\hat{v})$ to $(\check{\rho}_2,\check{v}_2)$ and finally a contact discontinuity joining $(\check{\rho}_2,\check{v}_2)$ to $(\rho^r,v^r)$. This implies
$$\Delta_{\tilde{t}}\,\mathcal{N}= 3-1 =2.$$
By the inequalities (\ref{disuguaglianze_interazioni_posizione_punti_1}), we obtain
$$\Delta TV_{\tilde{t}}(\rho)=|\rho^r-\check{\rho}_2 |+|\check{\rho}_2-\hat{\rho}| + |\hat{\rho}-\rho^l|-|\rho^r-\rho^l|=2(\hat{\rho}-\rho^l).$$
By the first inequality of Lemma \ref{lemma_disuguaglianze_triangolo} and $\rho^r\leq \check{\rho_2}$, we find
$$\hat{\rho}-\rho^l \leq c_1(\rho^l-\check{\rho}_2)\leq c_1\,(\rho^l-\rho^r),$$
where $c_1$ is a positive constant depending only on $q$ and the invariant domain. Therefore  
$$\Delta TV_{\tilde{t}}(\rho) \leq C_1\,(\rho^l-\rho^r),$$
where $C_1 = 2\,c_1$.\\
For the second component, we find
\begin{equation*}
\begin{split}
\Delta TV_{\tilde{t}}(v)& =|v^r-\check{v}_2 |+|\check{v}_2-\hat{v}| + |\hat{v}-v^l|-|v^r-v^k|-|v^k-v^l| =\\
& = 0+\check{v}_2-\hat{v}+v^l-\hat{v}-0=2(\check{v}_2-\hat{v}),
\end{split}
\end{equation*}
because $\check{v}_2 = v^r = v^l = v^k$ and $\check{v}_2 > \hat{v}$.\\
By the third inequality in Lemma \ref{lemma_disuguaglianze_triangolo}, we obtain
$$\Delta TV_{\tilde{t}}(v)\leq 2\,c_3\,(\rho^l-\check{\rho}_2)\leq C_2\,(\rho^l-\rho^r),$$
where $c_3$ depends only on $q$ and the invariant domain and $C_2 = 2\, c_3$.\\
By Lemma \ref{pressure_Lipschitz}, there exists a positive constant $M$ depending only on the invariant domain for which
$$\rho^l-\rho^r \leq M\,(p(\rho^l)-p(\rho^r)).$$
Since $v^l=v^r$, we find
$$\rho^l-\rho^r\leq M\, (v^l+p(\rho^l)-v^r-p(\rho^r)) =M\,(w^l-w^r).$$
Therefore 
$$\Delta TV_{\tilde{t}}(v)\leq C_2\,M\,(w^l-w^r).$$
For the Riemann invariant $w$, we have:
\begin{itemize}
\item $\hat{w}=w^l$;
\item $w^r\leq \check{w}_2$, because $\rho^r\leq \check{\rho}_2$ and $v^r=\check{v}_2$;
\item $\check{w}_2\leq \hat{w}$, because $\check{\rho}_2 \leq \hat{\rho}$ and by Lemma \ref{lemma_riemann_invariant_property};
\item $w^r \leq w^l$, because $w^r\leq \check{w}_2\leq \hat{w}=w^l$.
\end{itemize}
Hence we obtain
\begin{equation*}
\begin{split}
\Delta TV_{\tilde{t}}(w)&=|w^r-\check{w}_2|+|\check{w}_2-\hat{w}|+|\hat{w}-w^l|-|w^r-w^l|=\\
& = \check{w}_2 -w^r+\hat{w}-\check{w}_2-w^l+w^r=0.
\end{split}
\end{equation*}
\end{proof}
\subsubsection{Case $(\rho^k,v^k)\neq (\rho^r,v^r)$}
The points $(\rho^k,v^k)$ and $(\rho^r,v^r)$ can be joined by a wave with zero propagation speed if:
\begin{itemize}
\item the wave is a classical shock;
\item the wave is a non-classical shock.
\end{itemize}
The following proposition states that the case of a classical shock is not possible in the invariant domain.
\begin{prop}\label{prop_classical_shock_not_possible}
Let us suppose that $\lambda_1(\rho_\text{min},v_\text{min})<0$. The solution to the Riemann problem (\ref{Riemann_problem}) with initial datum
$$(\rho,v)(0,x)= \begin{cases}
(\rho^k,v^k) & \text{if } x\leq 0,\\
(\rho^r,v^r) & \text{if } x>0,
\end{cases}
$$
cannot be a classical shock with zero speed.
\end{prop}
\begin{proof}
A classical shock with zero speed appears when
$$\rho^k<\rho^r, \; \; \; v^k = L_1(\rho^k,\rho^r,v^r) \; \text{ and } \; \lambda = \dfrac{\rho^k\, v^k- \rho^r\,v^r}{\rho^k-\rho^r}= 0 \Longleftrightarrow \rho^r\, v^r = \rho^k\, v^k,$$
where $\lambda$ is the propagation speed of the shock obtained with the Rankine-Hugoniot condition; see Figure \ref{fig_stime_int_caso_1_shock}.
Let $\rho^\text{sup}$ be the point of maximum of the function $\rho \to \rho\, L_1(\rho,\rho^k,v^k)$, i.e.
$$\lambda_1(\rho^\text{sup},L_1(\rho^\text{sup},\rho^k,v^k))=0.$$
Applying Lagrange theorem as in the proof of Lemma \ref{lemma_concave_functions_properties}, we obtain
$$\rho^k<\rho^\text{sup}<\rho^r.$$
By Lemma \ref{eigenvalue_as_slope}, the function $\rho \to \lambda_1(\rho,L_1(\rho,\rho^k,v^k))$ is strictly decreasing. Hence we obtain
$$\lambda_1(\rho^k,v^k)>\lambda_1(\rho^\text{sup},L_1(\rho^\text{sup},\rho^k,v^k))=0$$
which is a contradiction of the hypothesis $(\rho^k,v^k)\in \mathcal{D}_{v_1,v_2,w_1,w_2}$, because by Lemma \ref{lemma_curve_decrescenti} we have
$$\lambda_1(\rho,v) <0 \; \text{ for every } \; (\rho,v) \in \mathcal{D}_{v_1,v_2,w_1,w_2}.$$
\end{proof}
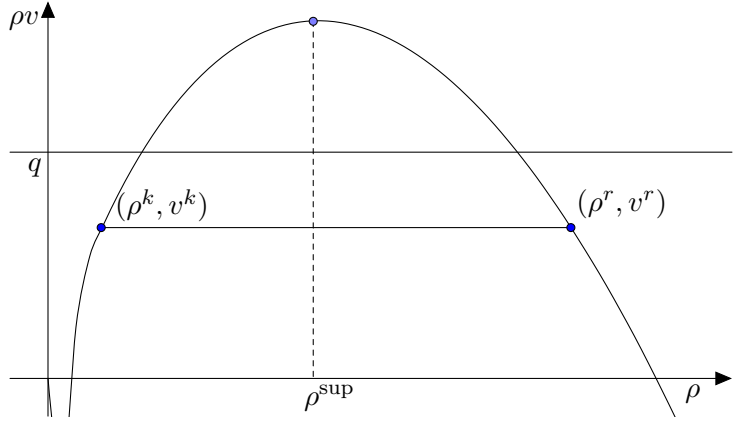
\begin{figure}[hbtp]
\centering
\definecolor{xdxdff}{rgb}{0.49019607843137253,0.49019607843137253,1.}
\definecolor{qqqqff}{rgb}{0.,0.,1.}
\begin{tikzpicture}[scale=0.5,line cap=round,line join=round,>=triangle 45,x=1.0cm,y=1.0cm]
\draw[->,color=black] (-1.,0.) -- (18.,0.);
\draw[->,color=black] (0.,-1.) -- (0.,10.);
\clip(-1.,-1.) rectangle (18.,10.);
\draw[smooth,samples=50,domain=0.001:18.0] plot(\x,{4.0*(\x)-(\x)^(1.5)});
\draw [domain=-1.:18.] plot(\x,{(--6.-0.*\x)/1.});
\draw (1.4,4.)-- (13.76,4.);
\draw [dash pattern=on 2pt off 2pt] (6.98,9.47)-- (6.98,0.);
\draw (1.5,5.3) node[anchor=north west] {$(\rho^k,v^k)$};
\draw (13.64,5.3) node[anchor=north west] {$(\rho^r,v^r)$};
\draw (6.47,0.1) node[anchor=north west] {$\rho^\text{sup}$};
\draw (-1.3,10.) node[anchor=north west] {$\rho v$};
\draw (16.5,0.1) node[anchor=north west] {$\rho$};
\draw (-0.8,6.1) node[anchor=north west] {$q$};
\begin{scriptsize}
\draw [fill=qqqqff] (1.4,4.) circle (3pt);
\draw [fill=qqqqff] (13.76,4.) circle (3pt);
\draw [fill=xdxdff] (6.98,9.47) circle (3pt);
\end{scriptsize}
\end{tikzpicture}
\caption{Representation of a shock with zero propagation speed joining the points $(\rho^k,v^k)$ and $(\rho^r,v^r)$. The density $\rho^\text{sup}$ is the point of maximum of the function $\rho \to \rho \, L_1(\rho,\rho^k,v^k)$.}\label{fig_stime_int_caso_1_shock}
\end{figure}
Next, we give the necessary conditions to have a non-classical shock in $x=0$.
\begin{prop}\label{when_non_classical_shock}
A non-classical shock appears as the the solution given by $\mathcal{RS}^q_2$ to the Riemann problem with initial datum
\begin{equation}\label{initial_datum_non_classical_shock}
(\rho,v)(0,x)= \begin{cases}
(\rho^k,v^k) & \text{if } x\leq 0,\\
(\rho^r,v^r) & \text{if } x>0,
\end{cases}
\end{equation}
when $\rho^r<\rho^k$ and $\rho^k\,v^k = \rho^r\,v^r=q$. Moreover in this case
$$(\rho^k,v^k)=(\hat{\rho},\hat{v}) \; \text{ and } \; (\rho^r,v^r)=(\check{\rho}_2,\check{v}_2).$$
If $\rho^k<\rho^r$ the solution $\mathcal{RS}^q_2((\rho^k,v^k),(\rho^r,v^r))$ is classical. 
\end{prop}
\begin{proof}
Let $(\rho^m,v^m)$ be the middle state of the classical solution to the Riemann problem with initial datum (\ref{initial_datum_non_classical_shock}). Since $v^m=v^r$ and $v^k<v^r$ (because $\rho^k\,v^k=\rho^r\,v^r=q$ and $\rho^r<\rho^k$), we find
$$p(\rho^m)=p(\rho^k)+v^k-v^m<p(\rho^k) \Longrightarrow \rho^m< \rho^k.$$
Hence by Lemma \ref{lemma_concave_functions_properties}, we have
$$\rho^m\,v^m >q.$$
The classical solution to the problem is a rarefaction joining $(\rho^k,v^k)$ to $(\rho^m,v^m)$ and with negative propagation speeds, followed by a contact discontinuity connecting $(\rho^m,v^m)$ to $(\rho^r,v^r)$ and with positive speed. Therefore
$$RS((\rho^k,v^k),(\rho^r,v^r))(0)=(\rho^m,v^m)$$
and the solution given by $\mathcal{RS}^q_2$ is the non-classical shock. Since $\rho^k\,v^k=q$, we could have
$$(\rho^k,v^k)=(\check{\rho}_1,\check{v}_1)\; \text{ or }\; (\rho^k,v^k)=(\hat{\rho},\hat{v}).$$
The first case is not possible, because by Lemma \ref{lemma_concave_functions_properties} we have $\lambda_1(\check{\rho}_1,\check{v}_1)>0$, while $\lambda_1(\rho^k,v^k)$ is negative because $(\rho^k,v^k)$ is in $\mathcal{D}_{v_1,v_2,w_1,w_2}$. Therefore we have
$$(\rho^k,v^k)=(\hat{\rho},\hat{v}) \; \text{ and }\; (\rho^r,v^r)=(\check{\rho}_2,\check{v}_2)$$
and the solution given by $\mathcal{RS}^q_2$ to the Riemann problem with initial datum (\ref{initial_datum_non_classical_shock}) is constant, i.e.
\begin{equation*}
(\rho,v)(t,x)=
\begin{cases}
(\rho^k,v^k)=(\hat{\rho},\hat{v}) & \text{if } x/t \leq 0,\\
(\rho^r,v^r)=(\check{\rho}_2,\check{v}_2) &\text{if } x/t>0.
\end{cases}
\end{equation*}
On the contrary, if $\rho^r>\rho^k$, then $v^m=v^r<v^k$. Therefore
$$v^r-v^k<0 \Longrightarrow p(\rho^m)=p(\rho^k)+v^k-v^r >p(\rho^k)\Longrightarrow \rho^m >\rho^k.$$
By Lemma \ref{lemma_concave_functions_properties}, the middle state $(\rho^m,v^m)$ satisfies the constraint and the non-classical shock does not appear.
\end{proof}
\begin{prop}\label{prop_contact_discontinuity_from_left_case_2}
Assume $\lambda_1(\rho_\text{min},v_\text{min})<0$. Fix $(\rho^l,v^l)$, $(\rho^k,v^k)$ and $(\rho^r,v^r)$ in the domain $\mathcal{D}_{v_1,v_2,w_1,w_2}$ such that
$$v^l=v^k, \; \; \; v^r+p(\rho^r)<v^k+p(\rho^k)\; \text{ and }\; \rho^k\,v^k = \rho^r\,v^r=q.$$
Assume that a wave joining $(\rho^l,v^l)$ to $(\rho^k,v^k)$ with positive propagation speed interacts with $x=0$ at time $\tilde{t}>0$. Let $(\rho^m,v^m)$ be the middle state for the classical solution to the Riemann problem with initial datum
\begin{equation}\label{initial_datum_t_tilde_contact_discontinuity}
(\rho,v)(\tilde{t},x) = \begin{cases}
(\rho^l,v^l) & \text{if } x\leq 0,\\
(\rho^r,v^r) & \text{if } x>0.
\end{cases}
\end{equation}
\begin{itemize}
\item[(i)] Let $(\rho^\sigma,v^\sigma)$ for $\sigma \in [0,1]$ be a point satisfying $\rho^\sigma\in [\rho^m,\rho^l]$ and $v^\sigma= v^l+p(\rho^l)-p(\rho)$. If $\rho^m\,v^m\leq q$, then
\begin{equation*}
\mathcal{RS}^q_2((\rho^l,v^l),(\rho^r,v^r))\left(\frac{x}{t-\tilde{t}}\right) = \begin{cases}
\vspace{0.1cm}
(\rho^l,v^l) & \text{if }\; \lambda_1(\rho^l,v^l)\leq \frac{x}{t-\tilde{t}},\\
\vspace{0.1cm}
(\rho^\sigma,v^\sigma) &\text{if }\; \frac{x}{t-\tilde{t}} = \lambda_1(\rho^\sigma,v^\sigma) \;\text{ for } \; \sigma \in (0,1),\\
\vspace{0.1cm}
(\rho^m,v^m) & \text{if } \; \lambda_1(\rho^m,v^m)\leq \frac{x}{t-\tilde{t}}\leq v^r,\\
(\rho^r,v^r) & \text{if } \; \frac{x}{t-\tilde{t}}>v^r.
\end{cases}
\end{equation*}
Moreover
$$\Delta TV_{\tilde{t}}(\rho) \leq 0, \; \; \; \Delta TV_{\tilde{t}}(v) =0 \; \text{ and } \; \Delta_{\tilde{t}} \mathcal{N}=0.$$
For the Riemann invariant $w$ we have
$$\Delta TV_{\tilde{t}}(w)\leq 0.$$
\item[(ii)] If $\rho^m\, v^m>q$, then
\begin{equation*}
\mathcal{RS}^q_2((\rho^l,v^l),(\rho^r,v^r))\left(\frac{x}{t-\tilde{t}}\right) = \begin{cases}
\vspace{0.1cm}
(\rho^l,v^l) & \text{if }\; \frac{x}{t-\tilde{t}}\leq \lambda,\\
\vspace{0.1cm}
(\hat{\rho},\hat{v}) & \text{if } \; \lambda\leq \frac{x}{t-\tilde{t}}\leq 0,\\
(\rho^r,v^r) & \text{if } \; \frac{x}{t-\tilde{t}}>0,
\end{cases}
\end{equation*}
where by the Rankine-Hugoniot condition 
$$\lambda=\dfrac{\rho^l\,v^l-\hat{\rho}\, \hat{v}}{\rho^l-\hat{\rho}}$$
is the speed of the shock joining $(\rho^l,v^l)$ to $(\hat{\rho},\hat{v})$. Furthermore
\begin{equation*}
\Delta TV_{\tilde{t}}(\rho) \leq C_1\, (\rho^l-\rho^k), \; \; \Delta TV_{\tilde{t}} (v) \leq C_2\,(\rho^l-\rho^k) \; \text{ and } \; \Delta_{\tilde{t}}\, \mathcal{N}=0,
\end{equation*}
where $C_1$ and $C_2$ are positive constants depending only on $q$ and $\mathcal{D}_{v_1,v_2,w_1,w_2}$.\\
In the coordinates $(v,w)$ of the Riemann invariants, we have
$$\Delta TV_{\tilde{t}}(v)=C_3\,|w^l-w^k|\; \text{ and } \; \Delta TV_{\tilde{t}}(w)\leq 0,$$
where $C_3$ is a positive constant depending only on the invariant domain and $q$.
\end{itemize}
\end{prop}
\begin{proof}
The hypothesis $w^r=v^r+p(\rho^r)<v^k+p(\rho^k)=w^k$ and Lemma \ref{lemma_riemann_invariant_property} imply $\rho^k>\rho^r$.\\
We have to distinguish different cases.\\
\textit{Case (i):} $\rho^l \leq \rho^k$ and $\rho^m\,v^m\leq q$; see Figure \ref{fig_stime_cont_disc_vincolo_soddisfatto}.\\
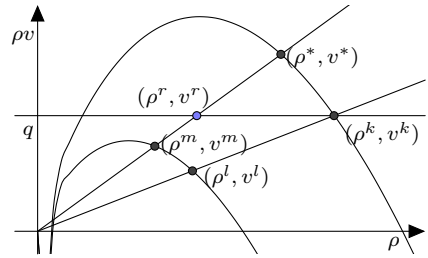
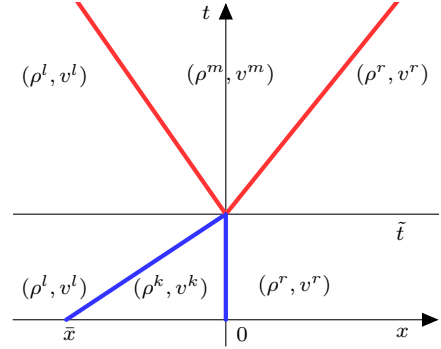
\begin{figure}[hbtp]
\centering
\begin{subfigure}[hbtp]{0.47\linewidth}
\definecolor{uququq}{rgb}{0.25098039215686274,0.25098039215686274,0.25098039215686274}
\definecolor{xdxdff}{rgb}{0.49019607843137253,0.49019607843137253,1.}
\begin{tikzpicture}[scale=0.3,line cap=round,line join=round,>=triangle 45,x=1.0cm,y=1.0cm]
\draw[->,color=black] (-1.,0.) -- (17.,0.);
\draw[->,color=black] (0.,-1.) -- (0.,10.);
\clip(-1.5,-1.) rectangle (17.,10.);
\draw[smooth,samples=50,domain=0.001:17.0] plot(\x,{4.0*(\x)-(\x)^(1.5)});
\draw[smooth,samples=50,domain=0.001:17.0] plot(\x,{3.0*(\x)-(\x)^(1.5)});
\draw [domain=-1.:17.] plot(\x,{(--5.110-0.*\x)/1.});
\draw [domain=0.0:17.0] plot(\x,{(-0.--5.11*\x)/13.00});
\draw [domain=0.0:17.0] plot(\x,{(-0.--5.11*\x)/6.98});
\begin{scriptsize}
\draw (6.9,3.3) node[anchor=north west] {$(\rho^l,v^l)$};
\draw (13.1,5.2) node[anchor=north west] {$(\rho^k,v^k)$};
\draw (4.,6.73) node[anchor=north west] {$(\rho^r,v^r)$};
\draw (5.,4.66) node[anchor=north west] {$(\rho^m,v^m)$};
\draw (10.5,8.5) node[anchor=north west] {$(\rho^*,v^*)$};
\draw (-1.5,9.2) node[anchor=north west] {$\rho v$};
\draw (15.,0.1) node[anchor=north west] {$\rho$};
\draw (-1,5) node[anchor=north west] {$q$};
\end{scriptsize}
\begin{scriptsize}
\draw [fill=xdxdff] (6.98,5.11) circle (5pt);
\draw [fill=uququq] (13.00,5.11) circle (5pt);
\draw [fill=uququq] (5.14,3.76) circle (5pt);
\draw [fill=uququq] (6.79,2.67) circle (5pt);
\draw [fill=uququq] (10.67,7.82) circle (5pt);
\end{scriptsize}
\end{tikzpicture}
\caption{Notations used in the case $\rho^l \leq \rho^k$ and $\rho^m\,v^m\leq q$.}
\end{subfigure}
\quad
\begin{subfigure}[hbtp]{0.47\linewidth}
\definecolor{fftttt}{rgb}{1.,0.2,0.2}
\definecolor{ttttff}{rgb}{0.2,0.2,1.}
\begin{tikzpicture}[scale=0.7,line cap=round,line join=round,>=triangle 45,x=1.0cm,y=1.0cm]
\draw[->,color=black] (-4.,0.) -- (4.,0.);
\draw[->,color=black] (0.,-0.5) -- (0.,6.);
\clip(-4.,-0.5) rectangle (4.,6.);
\draw [domain=-4.:4.] plot(\x,{(--2.-0.*\x)/1.});
\draw [line width=1.6pt,color=ttttff] (-3.,0.)-- (0.,2.);
\draw [line width=1.6pt,color=fftttt,domain=-4.0:0.0] plot(\x,{(-7.32--5.28*\x)/-3.66});
\draw [line width=1.6pt,color=fftttt,domain=0.0:4.0] plot(\x,{(--6.88--4.3*\x)/3.44});
\draw [line width=1.6pt,color=ttttff] (0.,2.)-- (0.,0.);
\begin{scriptsize}
\draw (3.05,2.) node[anchor=north west] {$\tilde{t}$};
\draw (-4,5) node[anchor=north west] {$(\rho^l,v^l)$};
\draw (-4,1.) node[anchor=north west] {$(\rho^l,v^l)$};
\draw (2.29,5) node[anchor=north west] {$(\rho^r,v^r)$};
\draw (0.45,1.) node[anchor=north west] {$(\rho^r,v^r)$};
\draw (-1.9,1.) node[anchor=north west] {$(\rho^k,v^k)$};
\draw (-0.6,6.09) node[anchor=north west] {$t$};
\draw (3.05,0.) node[anchor=north west] {$x$};
\draw (-3.21,0.) node[anchor=north west] {$\bar{x}$};
\draw (0.05,0.) node[anchor=north west] {$0$};
\draw (-0.9,5) node[anchor=north west] {$(\rho^m,v^m)$};
\end{scriptsize}
\end{tikzpicture}
\caption{The solution after the interaction is classical.}
\end{subfigure}
\caption{Interaction between a contact discontinuity centred in  $x=\bar{x}$ and a non-classical shock centred in $x=0$. In the represented case, after the interaction the constraint is satisfied.}\label{fig_stime_cont_disc_vincolo_soddisfatto}
\end{figure}

We have:
\begin{itemize}
\item $\rho^r \geq \rho^m$, because $\rho^m\,v^m\leq q =\rho^r\,v^r$ and $v^m=v^r$;
\item $\rho^m \leq \rho^l$, because $v^k=v^l<v^m=v^r$ and $\rho^k>\rho^r$.
\end{itemize}
Hence
\begin{equation*}
\begin{split}
\Delta TV_{\tilde{t}}(\rho)& =|\rho^r-\rho^m|+|\rho^m-\rho^l|-|\rho^r-\rho^k|-|\rho^k-\rho^l|=\\
& = \rho^r-\rho^m+\rho^l-\rho^m-\rho^k+\rho^r-\rho^k+\rho^l=\\
& = 2\,(\rho^r-\rho^m+\rho^l-\rho^k).
\end{split}
\end{equation*}
We \textit{claim} that
$$\rho^r-\rho^m\leq \rho^k-\rho^l.$$
We postpone the proof of this claim.\\
Hence $\Delta TV_{\tilde{t}}(\rho)\leq 0$.\\
For the $v$ component we find
\begin{equation*}
\begin{split}
\Delta TV_{\tilde{t}}(v) & =|v^r-v^m|+|v^m-v^l|-|v^r-v^k|-|v^k-v^l|=\\
& = 0+v^m-v^l-v^r+v^k+0= 0,
\end{split}
\end{equation*}
because $v^r=v^m$ and $v^l=v^k$.\\
For the Riemann invariant $w$, since $w^m=w^l$ and by the triangular inequality, we find
\begin{equation*}
\begin{split}
\Delta TV_{\tilde{t}}(w) & =|w^r-w^m|+|w^m-w^l|-|w^r-w^k|-|w^k-w^l| =\\
& = |w^r-w^l|-|w^r-w^k|-|w^k-w^l| \leq 0.
\end{split}
\end{equation*}
\textit{Case (ii):} $\rho^l \leq \rho^k$ and $\rho^m\,v^m >q$; see Figure \ref{fig_stime_cont_disc_vincolo_non_soddisfatto}.\\
\begin{figure}[hbtp]
\centering
\begin{subfigure}[hbtp]{0.47\linewidth}
\definecolor{uququq}{rgb}{0.25098039215686274,0.25098039215686274,0.25098039215686274}
\definecolor{xdxdff}{rgb}{0.49019607843137253,0.49019607843137253,1.}
\begin{tikzpicture}[scale=0.3,line cap=round,line join=round,>=triangle 45,x=1.0cm,y=1.0cm]
\draw[->,color=black] (-1.,0.) -- (17.5,0.);
\draw[->,color=black] (0.,-1.) -- (0.,10.);
\clip(-1.5,-1.) rectangle (17.5,10.);
\draw[smooth,samples=50,domain=0.001:17.5] plot(\x,{4.0*(\x)-(\x)^(1.5)});
\draw[smooth,samples=50,domain=0.001:17.5] plot(\x,{3.5*(\x)-(\x)^(1.5)});
\draw [domain=-1.:17.5] plot(\x,{(--4.-0.*\x)/1.});
\draw [domain=0.0:17.5] plot(\x,{(-0.--4.00*\x)/13.75});
\draw [domain=0.0:17.5] plot(\x,{(-0.--4.*\x)/6.});
\begin{scriptsize}
\draw (10.3,3.3) node[anchor=north west] {$(\rho^l,v^l)$};
\draw (13.8,4.2) node[anchor=north west] {$(\rho^k,v^k)$};
\draw (5.5,4.1) node[anchor=north west] {$(\rho^r,v^r)$};
\draw (7.8,6.3) node[anchor=north west] {$(\rho^m,v^m)$};
\draw (-1.5,9.) node[anchor=north west] {$\rho v$};
\draw (15.5,0.1) node[anchor=north west] {$\rho$};
\draw (-1.,4.2) node[anchor=north west] {$q$};
\draw (9.,5.3) node[anchor=north west] {$(\hat{\rho},\hat{v})$};
\end{scriptsize}
\begin{scriptsize}
\draw [fill=xdxdff] (6.,4.) circle (5pt);
\draw [fill=uququq] (13.75,4.00) circle (5pt);
\draw [fill=xdxdff] (10.29,2.99) circle (5pt);
\draw [fill=uququq] (8.02,5.35) circle (5pt);
\draw [fill=uququq] (9.47,3.99) circle (5pt);
\end{scriptsize}
\end{tikzpicture}
\caption{Notations used in the case $\rho^l \leq \rho^k$ and $\rho^m\,v^m> q$. Here $(\rho^r,v^r)=(\check{\rho}_2,\check{v}_2)$.}
\end{subfigure}
\quad
\begin{subfigure}[hbtp]{0.47\linewidth}
\definecolor{fftttt}{rgb}{1.,0.2,0.2}
\definecolor{ttttff}{rgb}{0.2,0.2,1.}
\begin{tikzpicture}[scale=0.7,line cap=round,line join=round,>=triangle 45,x=1.0cm,y=1.0cm]
\draw[->,color=black] (-4.,0.) -- (3.5,0.);
\draw[->,color=black] (0.,-0.5) -- (0.,5.5);
\clip(-4.,-0.5) rectangle (3.5,5.5);
\draw [domain=-4.:3.5] plot(\x,{(--2.-0.*\x)/1.});
\draw [line width=1.6pt,color=ttttff] (-3.,0.)-- (0.,2.);
\draw [line width=1.6pt,color=fftttt,domain=-4.0:0.0] plot(\x,{(-7.32--5.28*\x)/-3.66});
\draw [line width=1.6pt,color=ttttff] (0.,2.)-- (0.,0.);
\draw [line width=2.pt,color=fftttt] (0.,2.)-- (0.,5.);
\begin{scriptsize}
\draw (2.6,2.) node[anchor=north west] {$\tilde{t}$};
\draw (-3.6,4.7) node[anchor=north west] {$(\rho^l,v^l)$};
\draw (-3.9,1.) node[anchor=north west] {$(\rho^l,v^l)$};
\draw (0.,4.7) node[anchor=north west] {$(\rho^r,v^r)=(\check{\rho}_2,\check{v}_2)$};
\draw (0.76,1.) node[anchor=north west] {$(\rho^r,v^r)$};
\draw (-1.7,1.) node[anchor=north west] {$(\rho^k,v^k)$};
\draw (-0.5,5.2) node[anchor=north west] {$t$};
\draw (2.6,0.) node[anchor=north west] {$x$};
\draw (-3.12,0) node[anchor=north west] {$\bar{x}$};
\draw (0.04,0.) node[anchor=north west] {$0$};
\draw (-1.35,4.7) node[anchor=north west] {$(\hat{\rho},\hat{v})$};
\end{scriptsize}
\end{tikzpicture}
\caption{After the interaction the non-classical shock appears.}
\end{subfigure}
\caption{Interaction between a contact discontinuity centred in  $x=\bar{x}$ and a non-classical shock centred in $x=0$. In the represented case, after the interaction the constraint is not satisfied.}\label{fig_stime_cont_disc_vincolo_non_soddisfatto}
\end{figure}
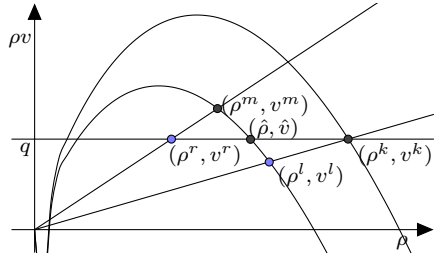
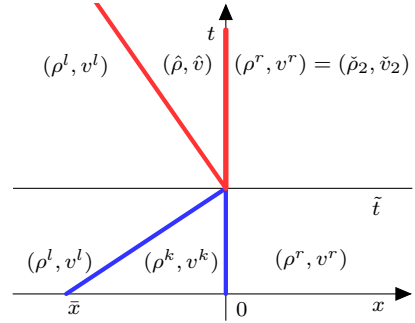

At time $\tilde{t}=|\bar{x}|/v^r$ the solution given by $\mathcal{RS}^q_2$ to the Riemann problem with initial datum (\ref{initial_datum_t_tilde_contact_discontinuity}) in non-classical. Therefore the states $(\hat{\rho},\hat{v})$ and $(\check{\rho}_2,\check{v}_2)$ appear. Since $\rho^r\,v^r=q$, we have
$$(\check{\rho}_2,\check{v}_2)=(\rho^r,v^r).$$
We have:
\begin{itemize}
\item $\rho^r <\hat{\rho}$, because: by Lemma \ref{lemma_concave_functions_properties}, the condition $\rho^m\,v^m> q$ implies $\rho^m <\hat{\rho}$ and $v^m=v^r$ implies $\rho^m\,v^m>q=\rho^r\,v^r \Rightarrow \rho^m >\rho^r$;
\item $\rho^l > \hat{\rho}$, applying Lemma \ref{lemma_concave_functions_properties} to $(\rho^l,v^l)$ (which satisfies $\rho^l\,v^l \leq q$).
\end{itemize}
Therefore
\begin{equation*}
\begin{split}
\Delta TV (\rho)& =|\rho^r-\hat{\rho}|+|\hat{\rho}-\rho^l|-|\rho^r-\rho^k|-|\rho^k-\rho^l|=\\
& = \hat{\rho}-\rho^r+\rho^l-\hat{\rho}+\rho^r-\rho^k+\rho^l-\rho^k=\\
& = 2(\rho^l-\rho^k)\leq 0.
\end{split}
\end{equation*}
For the second component, we find
\begin{equation*}
\begin{split}
\Delta TV_{\tilde{t}}(v)& =|v^r-\hat{v}|+|\hat{v}-v^l|-|v^r-v^k|-|v^k-v^l|=\\
&= v^r-\hat{v}+\hat{v}-v^l-v^r+v^k+0 =0,
\end{split}
\end{equation*}
because $v^l=v^k$.\\
For the Riemann invariant $w$, we have:
\begin{itemize}
\item $w^r<w^k$ and $w^l=\hat{w}$ respectively by hypothesis and definition;
\item $w^r<\hat{w}$, because $\rho^r<\hat{\rho}$ and by Lemma \ref{lemma_riemann_invariant_property};
\item $w^k\geq w^l$, because $v^k=v^l$ and $\rho^k \geq \rho^l$.
\end{itemize}
Therefore
\begin{equation*}
\begin{split}
\Delta TV_{\tilde{t}}(w) & = |w^r-\hat{w}|+|\hat{w}-w^l|-|w^r-w^k|-|w^k-w^l| =\\
& = \hat{w}-w^r -w^k+w^r-w^k+w^l =\\
& = 2\, (w^l-w^k) \leq 0.
\end{split}
\end{equation*}
\textit{Case (iii):} $\rho^l>\rho^k$; see Figure \ref{fig_stime_cont_disc_vincolo_non_soddisfatto_caso_3}.\\
\begin{figure}[hbtp]
\centering
\begin{subfigure}[hbtp]{0.47\linewidth}
\definecolor{uququq}{rgb}{0.25098039215686274,0.25098039215686274,0.25098039215686274}
\definecolor{xdxdff}{rgb}{0.49019607843137253,0.49019607843137253,1.}
\begin{tikzpicture}[scale=0.3,line cap=round,line join=round,>=triangle 45,x=1.0cm,y=1.0cm]
\draw[->,color=black] (-1.,0.) -- (17.,0.);
\draw[->,color=black] (0.,-0.5) -- (0.,10.);
\clip(-1.5,-1) rectangle (17.,10.);
\draw[smooth,samples=50,domain=0.001:17.0] plot(\x,{4.0*(\x)-(\x)^(1.5)});
\draw[smooth,samples=50,domain=0.001:17.0] plot(\x,{3.5*(\x)-(\x)^(1.5)});
\draw [domain=-1.:17.] plot(\x,{(--4.-0.*\x)/1.});
\draw [domain=0.0:17.0] plot(\x,{(-0.--4.*\x)/5.72});
\draw [domain=0.0:17.0] plot(\x,{(-0.--4.00*\x)/9.47});
\begin{scriptsize}
\draw (9.5,4.2) node[anchor=north west] {$(\rho^k,v^k)$};
\draw (12.6,6.1) node[anchor=north west] {$(\rho^l,v^l)$};
\draw (4.8,4.1) node[anchor=north west] {$(\rho^r,v^r)$};
\draw (11.03,8.21) node[anchor=north west] {$(\rho^m,v^m)$};
\draw (-1.5,9.) node[anchor=north west] {$\rho v$};
\draw (15.,0.1) node[anchor=north west] {$\rho$};
\draw (-1,4.) node[anchor=north west] {$q$};
\draw (13.66,4.1) node[anchor=north west] {$(\hat{\rho},\hat{v})$};
\end{scriptsize}
\begin{scriptsize}
\draw [fill=xdxdff] (5.72,4.) circle (5pt);
\draw [fill=uququq] (9.47,4.00) circle (5pt);
\draw [fill=uququq] (12.80,5.40) circle (5pt);
\draw [fill=uququq] (13.75,4.00) circle (5pt);
\draw [fill=uququq] (10.89,7.61) circle (5pt);
\end{scriptsize}
\end{tikzpicture}
\caption{Notations used in the case $\rho^l > \rho^k$. In this case we always have $\rho^m\,v^m> q$ and $(\rho^r,v^r)=(\check{\rho}_2,\check{v}_2)$.}
\end{subfigure}
\quad
\begin{subfigure}[hbtp]{0.47\linewidth}
\definecolor{fftttt}{rgb}{1.,0.2,0.2}
\definecolor{ttttff}{rgb}{0.2,0.2,1.}
\begin{tikzpicture}[scale=0.7,line cap=round,line join=round,>=triangle 45,x=1.0cm,y=1.0cm]
\draw[->,color=black] (-4.,0.) -- (3.5,0.);
\draw[->,color=black] (0.,-0.5) -- (0.,5.5);
\clip(-4.,-0.5) rectangle (3.5,5.5);
\draw [domain=-4.:3.5] plot(\x,{(--2.-0.*\x)/1.});
\draw [line width=1.6pt,color=ttttff] (-3.,0.)-- (0.,2.);
\draw [line width=1.6pt,color=fftttt,domain=-4.0:0.0] plot(\x,{(-7.32--5.28*\x)/-3.66});
\draw [line width=2.pt,color=fftttt] (0.,2.)-- (0.,5.);
\draw [line width=1.6pt,color=ttttff] (0.,2.)-- (0.,0.);
\begin{scriptsize}
\draw (2.6,2.) node[anchor=north west] {$\tilde{t}$};
\draw (-3.39,4.5) node[anchor=north west] {$(\rho^l,v^l)$};
\draw (-3.79,1.) node[anchor=north west] {$(\rho^l,v^l)$};
\draw (0.1,4.5) node[anchor=north west] {$(\rho^r,v^r)=(\check{\rho}_2,\check{v}_2)$};
\draw (0.76,1.) node[anchor=north west] {$(\rho^r,v^r)$};
\draw (-1.6,1.) node[anchor=north west] {$(\rho^k,v^k)$};
\draw (-0.5,5.1) node[anchor=north west] {$t$};
\draw (2.6,0.0) node[anchor=north west] {$x$};
\draw (-3.12,0.) node[anchor=north west] {$\bar{x}$};
\draw (0.04,0.) node[anchor=north west] {$0$};
\draw (-1.35,4.5) node[anchor=north west] {$(\hat{\rho},\hat{v})$};
\end{scriptsize}
\end{tikzpicture}
\caption{After the interaction the non-classical shock appears.}
\end{subfigure}
\caption{Interaction between a contact discontinuity centred in  $x=\bar{x}$ and a non-classical shock centred in $x=0$. After the interaction the constraint is not satisfied.}\label{fig_stime_cont_disc_vincolo_non_soddisfatto_caso_3}
\end{figure}
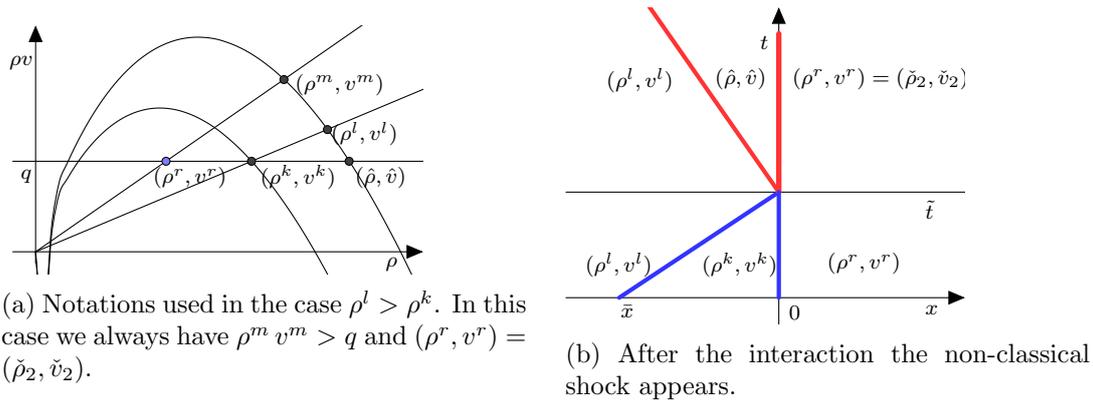

Since $\rho^l >\rho^k$ and $v^l=v^k$, we have $\rho^l\,v^l>q=\rho^k\,v^k$. Moreover, since $v^k+p(\rho^k)>v^r+p(\rho^r)$, by Lemma \ref{lemma_riemann_invariant_property}, we have $\rho^k>\rho^r$, which implies $v^k<v^r$ because $q=\rho^k\,v^k=\rho^r\,v^r$. Therefore for the middle state $(\rho^m,v^m)$ we find
\begin{equation*}
\begin{split}
& v^m=v^r>v^k=v^l \Longrightarrow p(\rho^l)=p(\rho^m)+v^m-v^l >p(\rho^m) \Longrightarrow\\
& \Longrightarrow \rho^l >\rho^m \Longrightarrow \rho^m\,v^m >\rho^l \, v^l >q,
\end{split}
\end{equation*}
by Lemma \ref{lemma_concave_functions_properties} applied to $(\rho^m,v^m)$ and $(\rho^l,v^l)$. Hence after the interaction the non-classical shock appears and as in the previous case, $(\check{\rho}_2,\check{v}_2)=(\rho^r,v^r)$.\\
We have:
\begin{itemize}
\item $\rho^l<\hat{\rho}$, by Lemma \ref{lemma_concave_functions_properties} applied to $(\rho^l,v^l)$ (which satisfies $\rho^l \, v^l >q$);
\item $\hat{\rho}>\rho^r$, because $\hat{\rho}>\rho^l >\rho^k >\rho^r$.
\end{itemize}
Hence
\begin{equation*}
\begin{split}
\Delta TV_{\tilde{t}}(\rho) & =|\rho^r-\hat{\rho}|+|\hat{\rho}-\rho^l|-|\rho^r-\rho^k|-|\rho^k-\rho^l|=\\
& = \hat{\rho}-\rho^r +\hat{\rho}-\rho^l -\rho^k+\rho^r-\rho^l+\rho^k=\\
& = 2(\hat{\rho}-\rho^l)\leq C_1\,(\rho^l-\rho^k),
\end{split}
\end{equation*}
by the first inequality of Lemma \ref{lemma_disuguaglianze_triangolo}.\\
For the second component, we have
\begin{equation*}
\begin{split}
\Delta TV_{\tilde{t}}(v)&= |v^r-\hat{v}|+|\hat{v}-v^l|-|v^r-v^k|-|v^k-v^l|=\\
& = v^r-\hat{v}+v^l-\hat{v}-v^r+v^k-0=\\
& =2\,(v^k-\hat{v}) \leq C_2\,(\rho^l-\rho^k),
\end{split}
\end{equation*} 
by the third inequality of Lemma \ref{lemma_disuguaglianze_triangolo}.\\
By Lemma \ref{pressure_Lipschitz} and since $v^l=v^k$, there exists a constant $M$ depending only on the invariant domain for which
$$\rho^l-\rho^k\leq M\,(p(\rho^l)-p(\rho^k))=M\,(v^l+p(\rho^l)-v^k-p(\rho^k))=M\,(w^l-w^k).$$
Therefore
$$\Delta TV_{\tilde{t}}(v) \leq C_2\,M\,(w^l-w^k).$$
For the Riemann invariant $w$ we have:
\begin{itemize}
\item $w^r<w^k$ and $w^l=\hat{w}$ respectively by hypothesis and definition;
\item $w^k< w^l$, because $v^k=v^l$ and $\rho^k < \rho^l$.
\end{itemize}
Therefore
\begin{equation*}
\begin{split}
\Delta TV_{\tilde{t}}(w) & = |w^r-\hat{w}|+|\hat{w}-w^l|-|w^r-w^k|-|w^k-w^l|=\\
& = \hat{w}-w^r-w^k+w^r-w^l+w^k=0.
\end{split}
\end{equation*}

\textit{Proof of the claim.}\\
\textit{Claim:} If $v^r+p(\rho^r)<v^k+p(\rho^k)$, $\rho^l \leq \rho^k$ and $\rho^m\,v^m\leq q$, then
$$\rho^r-\rho^m\leq \rho^k-\rho^l.$$
\textit{Proof.} Let us call $(\rho^*,v^*)$ the solution to the system (see Figure \ref{fig_stime_cont_disc_vincolo_soddisfatto})
$$
\begin{cases}
v=v^r,\\
v+p(\rho)=v^k+p(\rho^k).
\end{cases}
$$
Let $w^\alpha$ and $w^\beta$ be two positive constants such that $w^\alpha< w^\beta$ and let $(\rho^\alpha,v^\alpha)$ and $(\rho^\beta,v^\beta)$ be two points such that $v^\alpha+p(\rho^\alpha)=w^\alpha$, $v^\beta+p(\rho^\beta)=w^\beta$ and $v^\alpha=v^\beta=v$ for some $v$. We have
$$\rho^\alpha =p^{-1}(w^\alpha-v) \; \text{ and }\; \rho^\beta = p^{-1}(w^\beta-v).$$
Consider the function
$$v \to \varphi(v)=p^{-1}(w^\beta-v)-p^{-1}(w^\alpha-v).$$
We can write
$$\rho^\beta-\rho^\alpha=\varphi(v^\beta).$$
We have
$$w^\alpha<w^\beta \Longrightarrow w^\alpha -v <w^\beta-v.$$
The function $\rho \to p(\rho)$ is convex, hence
$$p'(p^{-1}(w^\alpha-v))\leq p'(p^{-1}(w^\beta-v)) \Longrightarrow \dfrac{1}{p'(p^{-1}(w^\beta-v))}-\dfrac{1}{p'(p^{-1}(w^\alpha-v))}\leq 0,$$
because the functions $\rho \to p'(\rho)$ and $\gamma \to p^{-1}(\gamma)$ are increasing. Since
$$\dfrac{d}{d v}(p^{-1}(v))=\dfrac{1}{p'(\rho(v))},$$
we find
$$\dfrac{d}{dv}\varphi(v)=-\dfrac{1}{p'(p^{-1}(w^\beta-v))}+\dfrac{1}{p'(p^{-1}(w^\alpha-v))}\leq 0,$$
which means that the function $v\to \varphi(v)$ is decreasing.\\
Taking $w^\alpha=v^l+p(\rho^l)$ and $w^\beta = v^k+p(\rho^k)$, since $v^l= v^k\leq v^r=v^*$, we obtain
$$\varphi(v^l)\geq \varphi(v^*) \Longleftrightarrow \rho^* -\rho^m \leq \rho^k-\rho^l.$$
Since $v^*=v^r$ and $v^*+p(\rho^*)=v^k+p(\rho^k)>v^r+p(\rho^r)$, we have $\rho^*>\rho^r$ and then
$$\rho^*-\rho^m>	\rho^r-\rho^m \Longrightarrow \rho^k-\rho^l > \rho^r-\rho^l.$$
\end{proof}
\subsection{A shock interacts with $x=0$}
Fix $\bar{x}>0$. Let us consider the Riemann problems (\ref{Riemann_problem}) centred in $\bar{x}$ and $x=0$ with initial data 
\begin{equation}\label{stime_interazione_shock_da_destra_dato iniziale}
(\rho,v)(0,x) = \begin{cases}
(\rho^l,v^l) & \text{if } x\leq 0,\\
(\rho^k,v^k) & \text{if } x> 0.
\end{cases}
\; \text{ and } \;
(\rho,v)(0,x)= \begin{cases}
(\rho^k,v^k) & \text{if } x\leq \bar{x},\\
(\rho^r,v^r) & \text{if } x>\bar{x}.
\end{cases}
\end{equation}
Let us assume that the solution given by $\mathcal{RS}^q_2$ to the Riemann problem centred in $x=0$ is a wave propagating with  zero speed and that
\begin{equation}\label{stime_shock_da_destra_condizione_shock}
\rho^k<\rho^r\; \text{ and } \; v^r=L_1(\rho^r,\rho^k,v^k),
\end{equation}
so that the standard solution to the Riemann problem centred in $\bar{x}$ is a shock. Under the hypothesis $\lambda_1(\rho_\text{min},v_\text{min})<0$, the propagation speed of the shock centred in $\bar{x}$ is negative, namely
$$\lambda=\dfrac{\rho^r\,v^r -\rho^k \,v^k}{\rho^r-\rho^k}<0.$$
Hence at time $\tilde{t}= -\bar{x}/\lambda$, the shock reaches the line $x=0$ and we have to solve the new Riemann problem (\ref{Riemann_problem}) with initial datum
$$(\rho,v)(\tilde{t},x)=\begin{cases}
(\rho^l,v^l) &\text{if } x\leq 0,\\
(\rho^r,v^r) &\text{if } x>0.
\end{cases}$$
\subsubsection{Case $(\rho^l,v^l)=(\rho^k,v^k)$}
\begin{prop}\label{prop_shock_from_right_case_1}
Assume $\lambda_1(\rho_\text{min},v_\text{min})<0$. Fix $(\rho^l,v^l)$ and $(\rho^r,v^r)$ in the domain $\mathcal{D}_{v_1,v_2,w_1,w_2}$ such that
$$\rho^l <\rho^r.$$
Assume that a wave joining $(\rho^l,v^l)$ to $(\rho^r,v^r)$ with negative speed interacts with $x=0$ at time $\tilde{t}>0$.\\
Therefore the wave is a classical shock. Moreover
$$\mathcal{RS}^q_2((\rho^l,v^l),(\rho^r,v^r))\left(\dfrac{x}{t-\tilde{t}} \right)=\begin{cases}
(\rho^l,v^l) & \text{if } \frac{x}{t-\tilde{t}}\leq \lambda,\\
(\rho^r,v^r) & \text{if } \frac{x}{t-\tilde{t}}>\lambda,
\end{cases}$$
where $\lambda$ is the speed of the shock given by the Rankine-Hugoniot condition.\\
Finally
$$\Delta TV_{\tilde{t}} (\rho)=\Delta TV_{\tilde{t}} (v)=\Delta TV_{\tilde{t}}(w)=0, \; \text{ and } \; \Delta_{\tilde{t}}\, \mathcal{N}=0.$$
\end{prop}
\begin{proof}
Since the solution $\mathcal{RS}^q_2((\rho^l,v^l),(\rho^k,v^k))$ is classical (constant), we must have $\rho^l\,v^l\leq q$, otherwise the non-classical shock would appear. Since
$$\rho^l=\rho^k<\rho^r, \;\;\; v^r=L_1(\rho^r,\rho^l,v^l)$$
and the function $\rho \to \rho\, L_1(\rho,\rho^l,v^l)$ is strictly decreasing in the invariant domain, we have
$$\rho^r\,v^r =\rho^r \, L_1(\rho^r,\rho^l,v^l) \leq \rho^l\,v^l <q,$$
which means that the constraint is satisfied also by $(\rho^r,v^r)$.\\
The classical solution to the Riemann problem at time $\tilde{t}$ is a shock centred in $(\tilde{t},0)$ and with propagation speed $\lambda$. Therefore the classical solution in $x=0$ for $t>\tilde{t}$ is $(\rho^r,v^r)$ and consequently the solution given by $\mathcal{RS}^q_2$ is classical. Hence the number of waves and the total variation before and after the interaction does not change, i.e.
$$\Delta TV_{\tilde{t}}(\rho)=0, \; \; \; \Delta TV_{\tilde{t}}(v)=0, \; \; \; \Delta TV_{\tilde{t}}(w)=0 \; \text{ and } \; \Delta_{\tilde{t}}\, \mathcal{N}=0.$$
\end{proof}
\subsubsection{Case $(\rho^l,v^l)\neq (\rho^k,v^k)$}
We have to distinguish two cases:
\begin{itemize}
\item the wave centred in $x=0$ is a classical shock;
\item the wave centred in $x=0$ is a non-classical shock.
\end{itemize}
The case of a classical shock is not possible by Proposition \ref{prop_classical_shock_not_possible} applied to the Riemann problem (\ref{Riemann_problem}) with initial datum
$$(\rho,v)(0,x)=
\begin{cases}
(\rho^l,v^l) & \text{if } x\leq 0,\\
(\rho^k,v^k) & \text{if } x>0.
\end{cases}
$$
By Proposition \ref{when_non_classical_shock}, the non-classical shock appears when $\rho^k<\rho^l$ and $\rho^l\,v^l=\rho^k\,v^k=q$. In this case we have
\begin{equation}\label{stime_interazione_shock_non_classico}
(\rho^l,v^l)=(\hat{\rho},\hat{v}) \; \text{ and } \; (\rho^k,v^k)=(\check{\rho}_2,\check{v}_2).
\end{equation}
\begin{prop}\label{prop_shock_from_right_case_2}
Assume $\lambda_1(\rho_\text{min},v_\text{min})<0$. Fix $(\rho^l,v^l)$, $(\rho^k,v^k)$ and $(\rho^r,v^r)$ in the domain $\mathcal{D}_{v_1,v_2,w_1,w_2}$ such that
$$v^k+p(\rho^k)<v^l+p(\rho^l), \; \; \; \rho^l\,v^l=\rho^k\,v^k=q \; \text{ and } \; \rho^k<\rho^r.$$
Assume that a wave joining $(\rho^k,v^k)$ to $(\rho^r,v^r)$ with negative propagation speed interacts with $x=0$ at time $\tilde{t}>0$. Therefore the wave with negative speed is a shock. Furthermore, let us call $(\rho^m,v^m)$ the middle state of the classical solution to the Riemann problem with initial datum
$$
(\rho,v)(\tilde{t},x)=\begin{cases}
(\rho^l,v^l) & \text{if } x\leq 0,\\
(\rho^r,v^r) & \text{if } x>0.
\end{cases}
$$
\begin{itemize}
\item[(i)] If $\rho^m\, v^m \leq q$, then
\begin{equation*}
\mathcal{RS}^q_2((\rho^l,v^l),(\rho^r,v^r))\left(\frac{x}{t-\tilde{t}}\right) = \begin{cases}
(\rho^l,v^l) & \text{if } \; \frac{x}{t-\tilde{t}}\leq \lambda,\\
(\rho^m,v^m) & \text{if } \; \lambda < \frac{x}{t-\tilde{t}}\leq v^r,\\
(\rho^r,v^r) & \text{if } \; \frac{x}{t-\tilde{t}}>v^r,
\end{cases}
\end{equation*}
where by the Rankine-Hugoniot condition
$$\lambda=\dfrac{\rho^m\,v^m-\rho^l\,v^l}{\rho^m-\rho^l}$$
is the propagation speed of the shock joining $(\rho^l,v^l)$ to $(\rho^m,v^m)$.\\
Moreover
$$\Delta TV_{\tilde{t}}(\rho) <0,\; \; \; \Delta TV_{\tilde{t}}(v) <0 \; \text{ and } \; \Delta_{\tilde{t}}\, \mathcal{N}=1.$$
For the Riemann invariant $w$, we have
$$\Delta TV_{\tilde{t}}(w)=0.$$
\item[(ii)] If $\rho^m\, v^m>q$, then
\begin{equation*}
\mathcal{RS}^q_2((\rho^l,v^l),(\rho^r,v^r))\left(\frac{x}{t-\tilde{t}}\right) = \begin{cases}
(\rho^l,v^l) & \text{if } \; \frac{x}{t-\tilde{t}}\leq 0,\\
(\check{\rho}_2,\check{v}_2) & \text{if } \; 0 < \frac{x}{t-\tilde{t}}\leq v^r,\\
(\rho^r,v^r) & \text{if } \; \frac{x}{t-\tilde{t}}>v^r.
\end{cases}
\end{equation*}
Moreover
$$\Delta TV_{\tilde{t}}(\rho) <0,\;\;\; \Delta TV_{\tilde{t}}(v)<0 \; \text{ and } \; \Delta_{\tilde{t}} \, \mathcal{N} = 0.$$
For the Riemann invariant $w$, we have
$$\Delta TV_{\tilde{t}}(w)=0.$$
\end{itemize}
\end{prop}
\begin{proof}
The hypothesis $v^k+p(\rho^k)<v^l+p(\rho^l)$ implies that $\rho^k<\rho^l$ by Lemma \ref{lemma_riemann_invariant_property}.\\
After the interaction the solution $\mathcal{RS}^q_2((\rho^l,v^l),(\rho^r,v^r))$ can be classical or non-classical.\\
We have to distinguish two cases.\\
\textit{Case (i):} $\rho^m\,v^m \leq q$; see Figure \ref{fig_shock_da_destra_vincolo_soddisfatto}.\\
\begin{figure}[hbtp]
\centering
\begin{subfigure}[h]{0.48\linewidth}
\definecolor{xdxdff}{rgb}{0.49019607843137253,0.49019607843137253,1.}
\definecolor{uququq}{rgb}{0.25098039215686274,0.25098039215686274,0.25098039215686274}
\begin{tikzpicture}[scale=0.35,line cap=round,line join=round,>=triangle 45,x=1.0cm,y=1.0cm]
\draw[->,color=black] (-1.,0.) -- (16.,0.);
\draw[->,color=black] (0.,-0.5) -- (0.,9.);
\clip(-1.5,-1) rectangle (16.,9.);
\draw [domain=-1.:16.] plot(\x,{(--6.-0.*\x)/1.});
\draw (12.331,4.16)-- (9.32,4.16);
\draw[smooth,samples=50,domain=0.001:16.0] plot(\x,{3.85*(\x)-(\x)^(1.5)});
\draw [dash pattern=on 2pt off 2pt] (11.58,5.17)-- (11.58,4.16);
\draw [domain=0.0:16.0] plot(\x,{(-0.--5.17*\x)/11.58});
\draw[smooth,samples=50,domain=0.001:16.0] plot(\x,{3.5*(\x)-(\x)^(1.5)});
\draw [domain=0.0:16.0] plot(\x,{(-0.--6*\x)/6.95});
\begin{scriptsize}
\draw (-0.8,6.) node[anchor=north west] {$q$};
\draw (-1.5,8.) node[anchor=north west] {$\rho v$};
\draw (14,0.1) node[anchor=north west] {$\rho$};
\draw (10.4,7.4) node[anchor=north west] {$(\rho^l,v^l)$};
\draw (6.2,7.4) node[anchor=north west] {$(\rho^k,v^k)$};
\draw (6.,4.90) node[anchor=north west] {$(\rho^r,v^r)$};
\draw (11.5,5.8) node[anchor=north west] {$(\rho^m,v^m)$};
\draw (12.2,4.66) node[anchor=north west] {$(\rho^*,v^*)$};
\end{scriptsize}
\begin{scriptsize}
\draw [fill=uququq] (10.88,6) circle (5pt);
\draw [fill=xdxdff] (11.58,5.17) circle (5pt);
\draw [fill=uququq] (9.32,4.16) circle (5pt);
\draw [fill=uququq] (6.95,6) circle (5pt);
\draw [fill=uququq] (12.33,4.16) circle (5pt);
\draw [fill=uququq] (11.58,4.16) circle (5pt);
\end{scriptsize}
\end{tikzpicture}
\caption{Notations used in the case $\rho^m\,v^m \leq q$. }\label{fig_shock_da_destra_vincolo_soddisfatto_a}
\end{subfigure}
\quad
\begin{subfigure}[h]{0.48\linewidth}
\definecolor{fftttt}{rgb}{1.,0.2,0.2}
\definecolor{ttttff}{rgb}{0.2,0.2,1.}
\begin{tikzpicture}[scale=0.7,line cap=round,line join=round,>=triangle 45,x=1.0cm,y=1.0cm]
\draw[->,color=black] (-3.5,0.) -- (4.,0.);
\draw[->,color=black] (0.,-0.5) -- (0.,5.5);
\clip(-3.5,-0.5) rectangle (4.,5.5);
\draw [dash pattern=on 3pt off 3pt,domain=-3.5:4.] plot(\x,{(--2.-0.*\x)/1.});
\draw [line width=1.6pt,color=ttttff] (3.,0.)-- (0.,2.);
\draw [line width=1.6pt,color=fftttt,domain=-3.5:0.0] plot(\x,{(-7.32--5.28*\x)/-3.66});
\draw [line width=1.6pt,color=fftttt,domain=0.0:4.0] plot(\x,{(--6.88--4.3*\x)/3.44});
\draw [line width=1.6pt,color=ttttff] (0.,2.)-- (0.,0.);
\begin{scriptsize}
\draw (3.4,2.) node[anchor=north west] {$\tilde{t}$};
\draw (0.05,1) node[anchor=north west] {$(\rho^k,v^k)$};
\draw (-1.58,1.) node[anchor=north west] {$(\rho^l,v^l)$};
\draw (2.3,4.6) node[anchor=north west] {$(\rho^r,v^r)$};
\draw (2.3,1.) node[anchor=north west] {$(\rho^r,v^r)$};
\draw (-0.9,4.6) node[anchor=north west] {$(\rho^m,v^m)$};
\draw (-0.4,5.2) node[anchor=north west] {$t$};
\draw (3.4,0.) node[anchor=north west] {$x$};
\draw (2.72,0.) node[anchor=north west] {$\bar{x}$};
\draw (0.05,0.) node[anchor=north west] {$0$};
\draw (-3.25,4.6) node[anchor=north west] {$(\rho^l,v^l)$};
\end{scriptsize}
\end{tikzpicture}
\caption{The solution after the interaction is classical.}\label{fig_shock_da_destra_vincolo_soddisfatto_b}
\end{subfigure}
\caption{Interaction between a non-classical shock centred in $x=0$ and a classical shock centred in $x=\bar{x}$. In the represented case, after the interaction the constraint is satisfied.}\label{fig_shock_da_destra_vincolo_soddisfatto}
\end{figure}
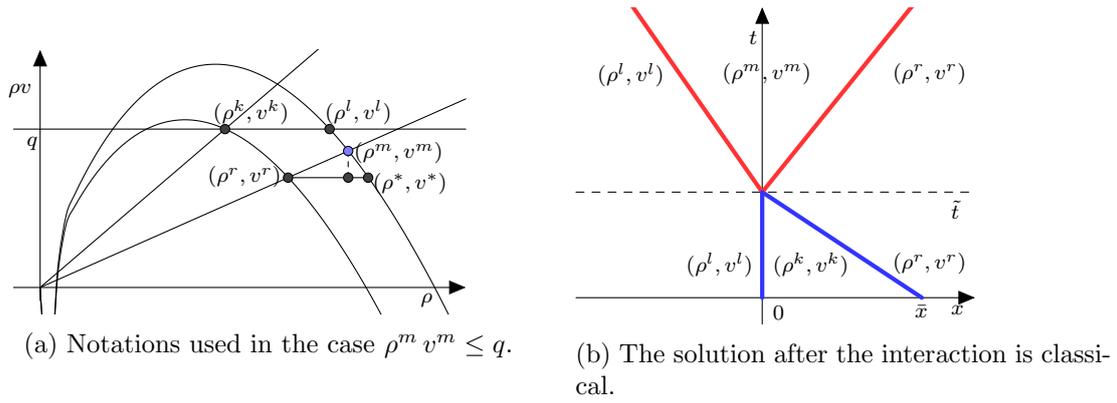

The solution $RS^q_2((\rho^l,v^l),(\rho^r,v^r))$ is a shock joining $(\rho^l,v^l)$ to $(\rho^m,v^m)$ followed by a contact discontinuity connecting $(\rho^m,v^m)$ to $(\rho^r,v^r)$. Since the shock and the contact discontinuity have respectively negative and positive propagation speed, we have
$$\mathcal{RS}((\rho^l,v^l),(\rho^r,v^r))(0)= (\rho^m,v^m)$$
and hence the solution given by $\mathcal{RS}^q_2$ after the interaction is classical.\\
We find:
\begin{itemize}
\item $\rho^m>\rho^r$, because the conditions $v^m=v^r$ and
$$v^m+p(\rho^m)=v^l+p(\rho^l)>v^k+p(\rho^k) =v^r+p(\rho^r),$$
imply $\rho^m>\rho^r$.
\item $\rho^m\geq\rho^l$, applying Lemma \ref{lemma_concave_functions_properties} to $(\rho^l,v^l)$ and $(\rho^m,v^m)$ (which satisfy $\rho^m\,v^m \leq q=\rho^l\,v^l$).
\end{itemize}
Hence
\begin{equation*}
\begin{split}
\Delta TV_{\tilde{t}} (\rho) &=|\rho^r-\rho^m|+|\rho^m-\rho^l|-|\rho^r-\rho^k|-|\rho^k-\rho^l|=\\
&= \rho^m-\rho^r+\rho^m-\rho^l-\rho^r+\rho^k-\rho^l+\rho^k=\\
&= 2\, (\rho^m-\rho^r+\rho^k-\rho^l).
\end{split}
\end{equation*}
We \textit{claim} that
$$\rho^m-\rho^r<\rho^l-\rho^k.$$
We postpone the proof of this claim.\\
Therefore $\Delta TV_{\tilde{t}}(\rho) < 0$.\\
Since $v^r=v^m$, for the second component we have
\begin{equation*}
\begin{split}
\Delta TV_{\tilde{t}}(v) & = |v^r-v^m|+|v^m-v^l|-|v^r-v^k|-|v^k-v^l|=\\
& = 0+v^l-v^m-v^k+v^r-v^k+v^l=\\
& = 2\,(v^l-v^k)< 0,
\end{split}
\end{equation*}
because the conditions $\rho^l\,v^l=\rho^k\,v^k=q$ and $\rho^l>\rho^k$ imply $v^l<v^k$.\\
For the Riemann invariant $w$, since $w^m=w^l$ and $w^r=w^k$, we find
\begin{equation*}
\begin{split}
\Delta TV_{\tilde{t}}(w)& =|w^r-w^m|+|w^m-w^l|-|w^r-w^k|-|w^k-w^l|=\\
& = |w^r-w^m|-|w^k-w^l|=0.
\end{split}
\end{equation*}
\textit{Case (ii):} $\rho^m\,v^m >q$; see Figure \ref{fig_shock_da_destra_vincolo_non_soddisfatto}.\\
\begin{figure}[hbtp]
\centering
\begin{subfigure}[h]{0.48\linewidth}
\definecolor{xdxdff}{rgb}{0.49019607843137253,0.49019607843137253,1.}
\definecolor{uququq}{rgb}{0.25098039215686274,0.25098039215686274,0.25098039215686274}
\begin{tikzpicture}[scale=0.35,line cap=round,line join=round,>=triangle 45,x=1.0cm,y=1.0cm]
\draw[->,color=black] (-1.,0.) -- (16.,0.);
\draw[->,color=black] (0.,-0.5) -- (0.,9.);
\clip(-1.5,-1) rectangle (16.,10.);
\draw [domain=-1.:16.] plot(\x,{(--6.-0.*\x)/1.});
\draw[smooth,samples=50,domain=0.001:16.0] plot(\x,{3.85*(\x)-(\x)^(1.5)});
\draw [domain=0.0:16.0] plot(\x,{(-0.--6.98*\x)/9.87});
\draw[smooth,samples=50,domain=0.001:16.0] plot(\x,{3.5*(\x)-(\x)^(1.5)});
\draw [domain=0.0:16.0] plot(\x,{(-0.--5.99*\x)/6.95});
\begin{scriptsize}
\draw (-0.8,6.) node[anchor=north west] {$q$};
\draw (-1.5,8.3) node[anchor=north west] {$\rho v$};
\draw (14,0.1) node[anchor=north west] {$\rho$};
\draw (10.7,7.2) node[anchor=north west] {$(\rho^l,v^l)$};
\draw (4.5,7.5) node[anchor=north west] {$(\rho^k,v^k)$};
\draw (5.5,5.5) node[anchor=north west] {$(\rho^r,v^r)$};
\draw (9.3,8.4) node[anchor=north west] {$(\rho^m,v^m)$};
\draw (7.9,6.) node[anchor=north west] {$(\check{\rho}_2,\check{v}_2)$};
\end{scriptsize}
\begin{scriptsize}
\draw [fill=uququq] (10.88,5.99) circle (5pt);
\draw [fill=xdxdff] (9.87,6.98) circle (5pt);
\draw [fill=uququq] (7.80,5.51) circle (5pt);
\draw [fill=uququq] (6.95,5.99) circle (5pt);
\draw [fill=uququq] (8.48,6.) circle (5pt);
\end{scriptsize}
\end{tikzpicture}
\caption{Notations used in the case $\rho^m v^m >q$.}
\end{subfigure}
\quad
\begin{subfigure}[h]{0.48\linewidth}
\definecolor{fftttt}{rgb}{1.,0.2,0.2}
\definecolor{qqqqff}{rgb}{0.,0.,1.}
\definecolor{ttttff}{rgb}{0.2,0.2,1.}
\begin{tikzpicture}[scale=0.7,line cap=round,line join=round,>=triangle 45,x=1.0cm,y=1.0cm]
\draw[->,color=black] (-4.,0.) -- (4.,0.);
\draw[->,color=black] (0.,-0.5) -- (0.,5.5);
\clip(-4.,-0.5) rectangle (4.,5.5);
\draw [dash pattern=on 3pt off 3pt,domain=-4.:4.] plot(\x,{(--2.-0.*\x)/1.});
\draw [line width=1.6pt,color=ttttff] (3.,0.)-- (0.,2.);
\draw [line width=1.6pt,color=fftttt] (0.,2.) -- (0.,4.8);
\draw [line width=1.6pt,color=fftttt,domain=0.0:4.0] plot(\x,{(--6.88--4.3*\x)/3.44});
\draw [line width=1.6pt,color=ttttff] (0.,2.)-- (0.,0.);
\begin{scriptsize}
\draw (-0.5,2.) node[anchor=north west] {$\tilde{t}$};
\draw (0.05,1.) node[anchor=north west] {$(\rho^k,v^k)$};
\draw (-1.76,1.) node[anchor=north west] {$(\rho^l,v^l)$};
\draw (2.4,4.49) node[anchor=north west] {$(\rho^r,v^r)$};
\draw (2.4,1.) node[anchor=north west] {$(\rho^r,v^r)$};
\draw (0.,4.54) node[anchor=north west] {$(\check{\rho}_2,\check{v}_2)$};
\draw (-0.4,5.3) node[anchor=north west] {$t$};
\draw (3.4,-0.1) node[anchor=north west] {$x$};
\draw (2.6,0.) node[anchor=north west] {$\bar{x}$};
\draw (0.05,0.) node[anchor=north west] {$0$};
\draw (-3.37,4.52) node[anchor=north west] {$(\rho^l,v^l)=(\hat{\rho},\hat{v})$};
\end{scriptsize}
\begin{scriptsize}
\draw [fill=qqqqff] (0.,7.28) circle (1.5pt);
\draw[color=qqqqff] (-5.60,6.69) node {$C$};
\end{scriptsize}
\end{tikzpicture}
\caption{After the interaction the non-classical shock appears.}
\end{subfigure}
\caption{Interaction between a non-classical shock centred in $x=0$ and a classical shock centred in $x=\bar{x}$. In the represented case, after the interaction the constraint is not satisfied.}\label{fig_shock_da_destra_vincolo_non_soddisfatto}
\end{figure}
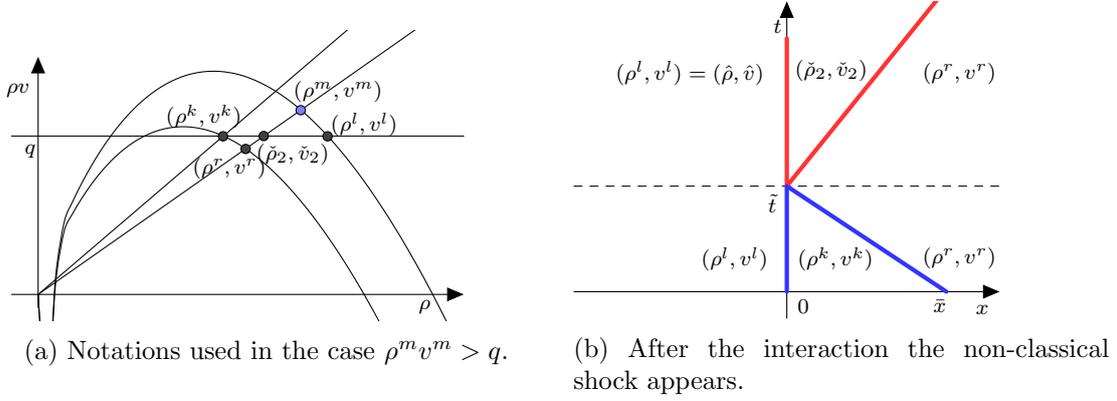

If the constraint is not satisfied the non-classical shock appears. Since $\rho^l\,v^l=q$, we find $(\rho^l,v^l)=(\hat{\rho},\hat{v})$.\\
We have:
\begin{itemize}
\item $\rho^r<\check{\rho}_2$, because $v^r=\check{v}_2$ and $\rho^r\,v^r<q=\check{\rho}_2\,\check{v}_2$;
\item $\rho^l>\check{\rho}_2$, indeed: applying Lemma \ref{lemma_concave_functions_properties} to $(\rho^m,v^m)$ (which satisfies $\rho^m\,v^m >q=\rho^l\,v^l$), we find $\rho^m < \rho^l$ and since $\check{v}_2=v^m$ and $\rho^m\,v^m >q=\check{\rho}_2\,\check{v}_2$, we have $\rho^m >\check{\rho}_2$.
\end{itemize}
Hence
\begin{equation*}
\begin{split}
\Delta TV_{\tilde{t}}(\rho) & = |\rho^r-\check{\rho}_2|+|\check{\rho}_2-\rho^l|-|\rho^r-\rho^k|-|\rho^k-\rho^l|=\\
& = \check{\rho}_2-\rho^r+\rho^l-\check{\rho}_2-\rho^r+\rho^k-\rho^l+\rho^k=\\
& = 2\, (\rho^k-\rho^r)< 0,
\end{split}
\end{equation*}
because by the hypothesis we have $\rho^r>\rho^k$.\\
For the second component, we find
\begin{equation*}
\begin{split}
\Delta TV_{\tilde{t}}(v)& =|v^r-\check{v}_2|+|\check{v}_2-v^l|-|v^r-v^k|-|v^k-v^l|=\\
&= 0+\check{v}_2-v^l-v^k+v^r-v^k+v^l=\\
& = 2\,(v^r-v^k)<0,
\end{split}
\end{equation*}
because the conditions $v^r+p(\rho^r)=v^k+p(\rho^k)$ and $\rho^r>\rho^k$ imply $\check{v}_2=v^r<v^k$.\\
For the Riemann invariant $w$, we have:
\begin{itemize}
\item $w^k<w^l$ by hypothesis;
\item $w^l>\check{w}_2$, because $\rho^l>\check{\rho}_2$ and by Lemma \ref{lemma_riemann_invariant_property};
\item $w^r<\check{w}_2$, because $\rho^r<\check{\rho}_2$ and $v^r=\check{v}_2$.
\end{itemize}
Hence we obtain:
\begin{equation*}
\begin{split}
\Delta TV_{\tilde{t}}(w) & =|w^r-\check{w}_2|+|\check{w}_2-w^l|-|w^r-w^k|-|w^k-w^l|=\\
& = \check{w}_2-w^r+w^l-\check{w}_2-w^l+w^k=0.
\end{split}
\end{equation*}

\textit{Proof of the claim.}\\
\textit{Claim:} If $v^k+p(\rho^k)<v^l+p(\rho^l)$, $\rho^r<\rho^m$ and $\rho^m>\rho^l$, then
$$\rho^m-\rho^r<\rho^l-\rho^k.$$
\textit{Proof.} Let us call $(\rho^*,v^*)\in \mathcal{D}_{v_1,v_2,w_1,w_2}$ the solution to the system (see Figure \ref{fig_shock_da_destra_vincolo_soddisfatto})
$$
\begin{cases}
\rho v= \rho^r\,v^r,\\
v+p(\rho)=v^m+p(\rho^m).
\end{cases}
$$
Fix $\pi\in \mathbb{R}^+$ and consider two points $(\rho^\alpha,v^\alpha)$ and $(\rho^\beta,v^\beta)$ such that
$$\rho^\alpha\,v^\alpha =\rho^\beta\, v^\beta = \pi, \; \; v^\alpha=L_1(\rho^\alpha,\rho^l,v^l) \; \text{ and } \; v^\beta= L_1(\rho^\beta,\rho^k,v^k).$$
Since $v^l+p(\rho^l)>v^k+p(\rho^k)$, by Lemma \ref{lemma_riemann_invariant_property} we obtain $\rho^\beta<\rho^\alpha$. Moreover $v^\beta >v^\alpha$, because $\rho^\alpha\,v^\alpha = \rho^\beta\, v^\beta =\pi$.\\
Consider the function
$$\pi \to d_\pi(\rho^\alpha,\rho^\beta)=\rho^\alpha-\rho^\beta=\dfrac{\pi}{v^\alpha}-\dfrac{\pi}{v^\beta}=\pi\left(\dfrac{1}{v^\alpha}-\dfrac{1}{v^\beta}\right).$$
We observe that
$$\dfrac{d}{d\pi}d_\pi(\rho^\alpha,\rho^\beta)=\dfrac{1}{v^\alpha}-\dfrac{1}{v^\beta}>0,$$
because $v^\beta >v^\alpha$. Hence the function $\pi \to d_\pi(\rho^\alpha,\rho^\beta)$ is strictly increasing.\\
Applying this result to $(\rho^\alpha,v^\alpha)=(\rho^*,v^*)$ and $(\rho^\beta,v^\beta)=(\rho^r,v^r)$ and then to $(\rho^\alpha,v^\alpha)=(\rho^l,v^l)$ and $(\rho^\beta,v^\beta)=(\rho^k,v^k)$, we obtain
$$\rho^*-\rho^r\leq \rho^l-\rho^k,$$
because $\rho^k\,v^k = q\geq \rho^r\,v^r$. Since $\rho^r< \rho^m$ and $v^r=v^m$, we have
$$\rho^r\,v^r <\rho^m\,v^m.$$
Therefore $\rho^*>\rho^m$ by Lemma \ref{lemma_concave_functions_properties}. Hence
$$\rho^m-\rho^r<\rho^*-\rho^r \leq \rho^l-\rho^k.$$
\end{proof}
\subsection{A rarefaction wave reaches the constraint}
Fix $\bar{x}>0$ and let us consider the Riemann problem (\ref{Riemann_problem}) with initial datum
$$
(\rho,v)(0,x)= \begin{cases}
(\rho^k,v^k) & \text{if } x\leq \bar{x},\\
(\rho^r,v^r) & \text{if } x>\bar{x}.
\end{cases}
$$
Suppose that $(\rho^k,v^k)$ and $(\rho^r,v^r)$ are connected by a rarefaction wave, i.e.
$$\rho^k \geq \rho^r \; \text{ and } \; v^r=v^k+p(\rho^k)-p(\rho^r),$$
and let $\sigma \to (\rho^\sigma,v^\sigma)$ for $\sigma \in [0,1]$ be a parametrization of the points of the rarefaction.\\
Fix $\delta>0$ and let us define the number $N=N(\delta)\in \mathbb{N}$ as
$$ N(\delta)=\Big\lfloor\dfrac{1}{\delta} \Big\rfloor,$$
where $\lfloor r \rfloor\in \mathbb{Z}$ is the integer part of the number $r$, i.e. the closest integer less or equal than $r$.
For $i=0,...,N(\delta)$, let us define the numbers
$$\sigma_i=\dfrac{i}{N}\in [0,1]$$
and the points
\begin{equation}\label{points_of_the_rarefaction_fan}
\lbrace (\rho^{\sigma_i},v^{\sigma_i})\rbrace_{i=0}^N\subset \lbrace (\rho^\sigma,v^\sigma)\rbrace_{\sigma \in [0,1]}.
\end{equation}
\begin{mydef}\label{def_rarefaction_fan}
A rarefaction fan for the rarefaction wave $\mathcal{RS}((\rho^k,v^k),(\rho^r,v^r))$ centred in $\bar{x}$ is (see Figure \ref{fig_rarefaction_fan})
\begin{equation}\label{rarefaction_fan}
(\rho,v)(t,x)=\begin{cases}
(\rho^k,v^k) & \text{if } \; \dfrac{x-\bar{x}}{t}\leq \lambda_1(\rho^k,v^k),\\
(\rho^{\sigma_i},v^{\sigma_i}) & \text{if }\;  \lambda_1(\rho^{\sigma_{i-1}},v^{\sigma_{i-1}})<\dfrac{x-\bar{x}}{t}\leq \lambda_1(\rho^{\sigma_i},v^{\sigma_i}) \; \text{ for } \, i=1,...,N(\delta),\\
(\rho^r,v^r) & \text{if } \; \dfrac{x-\bar{x}}{t}>\lambda_1(\rho^r,v^r).
\end{cases}
\end{equation}
\end{mydef}
\begin{figure}[hbtp]
\centering
\begin{subfigure}[hbtp]{0.48\linewidth}
\definecolor{uququq}{rgb}{0.25098039215686274,0.25098039215686274,0.25098039215686274}
\definecolor{xdxdff}{rgb}{0.49019607843137253,0.49019607843137253,1.}
\definecolor{qqqqff}{rgb}{0.,0.,1.}
\begin{tikzpicture}[scale=0.3,line cap=round,line join=round,>=triangle 45,x=1.0cm,y=1.0cm]
\draw[->,color=black] (-1.,0.) -- (17.,0.);
\draw[->,color=black] (0.,-0.5) -- (0.,10.);
\clip(-1.5,-1) rectangle (18.5,11.);
\draw[smooth,samples=50,domain=0.001:17.0] plot(\x,{4.0*(\x)-(\x)^(1.5)});
\begin{scriptsize}
\draw (7.8,10.7) node[anchor=north west] {$(\rho^r,v^r)=(\rho^{\sigma_N},v^{\sigma_N})$};
\draw (-1.5,9.) node[anchor=north west] {$\rho v$};
\draw (15.,0.1) node[anchor=north west] {$\rho$};
\draw (5.6,2.48) node[anchor=north west] {$(\rho^k,v^k)=(\rho^{\sigma_0},v^{\sigma_0})$};
\draw (13.6,5) node[anchor=north west] {$(\rho^{\sigma_1},v^{\sigma_1})$};
\draw (11.9,7.4) node[anchor=north west] {$(\rho^{\sigma_i},v^{\sigma_i})$};
\draw (9.85,9.4) node[anchor=north west] {$(\rho^{\sigma_{N-1}},v^{\sigma_{N-1}})$};
\end{scriptsize}
\begin{scriptsize}
\draw [fill=qqqqff] (9.,9.) circle (5pt);
\draw [fill=xdxdff] (14.00,3.60) circle (5pt);
\draw [fill=uququq] (10.,8.37) circle (5pt);
\draw [fill=uququq] (11.,7.51) circle (5pt);
\draw [fill=uququq] (12.,6.43) circle (5pt);
\draw [fill=uququq] (13.,5.12) circle (5pt);
\draw [fill=uququq] (15.,1.90) circle (5pt);
\end{scriptsize}
\end{tikzpicture}
\end{subfigure}
\quad
\begin{subfigure}[hbtp]{0.48\linewidth}
\definecolor{xdxdff}{rgb}{0.49019607843137253,0.49019607843137253,1.}
\begin{tikzpicture}[scale=0.95,line cap=round,line join=round,>=triangle 45,x=1.0cm,y=1.0cm]
\draw[->,color=black] (-0.2,0.) -- (5.,0.);
\draw[->,color=black] (0.,-0.2) -- (0.,2.7);
\clip(-1,-1) rectangle (5.3,2.7);
\draw (4.,0.)-- (-5.62,4.86);
\draw (4.,0.)-- (-2,6.00);
\draw (4.,0.)-- (0.45,6);
\draw (4.,0.)-- (1.,5.86);
\draw (-3.22,5.98)-- (4.,0.);
\draw [domain=-0.2:4.0] plot(\x,{(-12.--3.*\x)/-0.5});
\begin{scriptsize}
\draw (1.1,1.8) node[anchor=north west] {$u^{\sigma_1}$};
\draw (3.8,0.) node[anchor=north west] {$\bar{x}$};
\draw (0.5,0.6) node[anchor=north west] {$u^k=u^{\sigma_0}$};
\draw (1.9,2.54) node[anchor=north west] {$u^{\sigma_i}$};
\draw (3.65,2.4) node[anchor=north west] {$u^r=u^{\sigma_N}$};
\draw (-0.3,2.5) node[anchor=north west] {$t$};
\draw (4.4,0.) node[anchor=north west] {$x$};
\draw (2.6,2.8) node[anchor=north west] {$u^{\sigma_{N-1}}$};
\end{scriptsize}
\begin{scriptsize}
\draw [fill=xdxdff] (4.,0.) circle (1.5pt);
\end{scriptsize}
\end{tikzpicture}
\end{subfigure}
\caption{Representation of a rarefaction fan.}\label{fig_rarefaction_fan}
\end{figure}
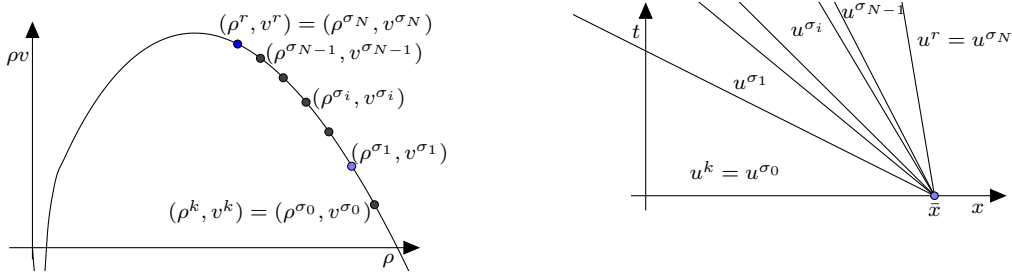
The next proposition states that the presence of a rarefaction fan does not have any influence on the total variation, but only the left and the right states of the initial datum count.
\begin{prop}\label{prop_rarefaction_fan_properties}
Let $t\in\mathbb{R}^+$ be a fixed instant. Fix $(\rho^k,v^k)$ and $(\rho^r,v^r)$ in the invariant domain $\mathcal{D}_{v_1,v_2,w_1,w_2}$ such that
$$\rho^k \geq \rho^r \; \text{ and } \; v^r=v^k+p(\rho^k)-p(\rho^r).$$
Consider the rarefaction fan (\ref{rarefaction_fan}). For every $i=1,...,N$, we have
$$\rho^{\sigma_{i-1}}\geq \rho^{\sigma_i}\; \text{ and } \; v^{\sigma_{i-1}}\leq v^{\sigma_i}.$$
Moreover for every $t\geq 0$ the total variation of the rarefaction fan is
\begin{equation}
TV_t(\rho)=\rho^k-\rho^r \; \text{ and } \; TV_t(v)= v^r-v^k.
\end{equation}
\end{prop}
\begin{proof}
We have
$$\rho^{\sigma_{i-1}}\geq \rho^{\sigma_i},$$
because $\sigma_{i-1}\leq \sigma_i$ and the function $\sigma \to \rho^\sigma$ is decreasing, because its derivative is the first component of the eigenvalue $r_1$ of the flux function which is $-1$. Moreover, since $v^{\sigma_{i-1}}+p(\rho^{\sigma_{i-1}})=v^k+p(\rho^k)$, we obtain:
$$v^{\sigma_{i}}=v^k+p(\rho^k)-p(\rho^{\sigma_{i}})=v^{\sigma_{i-1}}+p(\rho^{\sigma_{i-1}})-p(\rho^{\sigma_{i}})\geq v^{\sigma_{i-1}},$$
because $p(\rho^{\sigma_{i-1}})\geq p(\rho^{\sigma_{i}})$.\\
Therefore
\begin{equation*}
\begin{split}
TV_t(\rho) & = \sum_{j=1}^N |\rho^{\sigma_j}-\rho^{\sigma_{j-1}}|=\sum_{j=1}^N \left(\rho^{\sigma_{j-1}}-\rho^{\sigma_j} \right)=\\
&= \rho^{\sigma_0}-\rho^{\sigma_1}+\rho^{\sigma_1}-\rho^{\sigma_2}+...-\rho^{\sigma_N}=\\
&= \rho^{\sigma_0}-\rho^{\sigma_N}=\rho^k-\rho^r.
\end{split}
\end{equation*}
Similarly for the velocity.
\end{proof}
\begin{remark}\label{remark_rarefaction_fan_as_shocks}
By Proposition \ref{prop_rarefaction_fan_properties}, we can approximate a rarefaction wave with a rarefaction fan without affecting the total variation and we can consider each discontinuity of the fan as a single (non-classical) shock which is prolonged forward in time (rarefaction fans appear only at $t=0$ or $x=0$).
\end{remark}
Fix $\bar{x}>0$. Let us consider the Riemann problems (\ref{Riemann_problem}) centred in $\bar{x}$ and $x=0$ with initial data 
\begin{equation}\label{stime_interazione_rarefazione_da_destra_dato iniziale}
(\rho,v)(0,x) = \begin{cases}
(\rho^l,v^l) & \text{if } x\leq 0,\\
(\rho^k,v^k) & \text{if } x> 0.
\end{cases}
\; \text{ and } \;
(\rho,v)(0,x)= \begin{cases}
(\rho^k,v^k) & \text{if } x\leq \bar{x},\\
(\rho^r,v^r) & \text{if } x>\bar{x}.
\end{cases}
\end{equation}
Let us suppose that $(\rho^l,v^l)$ and $(\rho^k,v^k)$ are joined by a wave with zero propagation speed and that
$(\rho^k,v^k)$ and $(\rho^r,v^r)$ are connected by a non-classical shock of the first family with negative speed, i.e.
$$\rho^k \geq \rho^r \; \text{ and } \; v^r=v^k+p(\rho^k)-p(\rho^r).$$
The speed of the shock wave is given by the Rankine-Hugoniot condition in order to preserve conservation.
\subsubsection{Case $(\rho^l,v^l)=(\rho^k,v^k)$}
\begin{prop}\label{prop_rarefaction_from_right_case_1}
Assume $\lambda_1(\rho_\text{min},v_\text{min})<0$. Fix $(\rho^l,v^l)$ and $(\rho^r,v^r)$ in the domain $\mathcal{D}_{v_1,v_2,w_1,w_2}$ and let $N \in \mathbb{N}$ be the number of wave-fronts of a rarefaction fan.\\
Assume that $\rho^r< \rho^l$ and that a wave joining $(\rho^l,v^l)$ to $(\rho^r,v^r)$ with negative propagation speed interacts with $x=0$ at time $\tilde{t}>0$.\\
Then the wave is a rarefaction, i.e. $v^r=L_1(\rho^r,\rho^l,v^l)$.\\
Moreover the following statements hold.
\begin{itemize}
\item[(i)] If $\rho^r\,v^r \leq q$, then for every $(t,x)\in \mathbb{R}^+\times \mathbb{R}$ we have
\begin{equation*}
\begin{split}
\mathcal{RS}^q_2 & ((\rho^l,v^l),(\rho^r,v^r))\left(\dfrac{x}{t-\tilde{t}}\right) =\\
& = \begin{cases}
(\rho^l,v^l) & \text{if } \frac{x}{t-\tilde{t}}\leq \lambda_1(\rho^l,v^l),\\
(\rho^{\sigma_{i}},v^{\sigma_{i}}) & \text{if } \lambda_1(\rho^{\sigma{i-1}},v^{\sigma_{i-1}})< \frac{x}{t-\tilde{t}}<\lambda_1(\rho^{\sigma_i},v^{\sigma_i})\;\text{ for } \; i=1,...,N,\\
(\rho^r,v^r) & \text{if } \frac{x}{t-\tilde{t}}>\lambda_1(\rho^r,v^r), 
\end{cases}
\end{split}
\end{equation*}
where $\lbrace (\rho^{\sigma_i},v^{\sigma_i}) \rbrace_{i=1}^N$ are the points defined in (\ref{points_of_the_rarefaction_fan}). Therefore
$$ \Delta TV_{\tilde{t}}(\rho)=\Delta TV_{\tilde{t}}(v)=0 \; \text{ and } \; \Delta_{\tilde{t}} \mathcal{N}=N-1.$$
For the Riemann invariant $w$, we have
$$\Delta TV_{\tilde{t}}(w)=0.$$
\item[(ii)] If $\rho^r\,v^r >q$, then
\begin{equation*}
\begin{split}
\mathcal{RS}^q_2 & ((\rho^l,v^l),(\rho^r,v^r))\left(\dfrac{x}{t-\tilde{t}}\right) =\\
& = \begin{cases}
(\rho^l,v^l) & \text{if } \frac{x}{t-\tilde{t}}\leq \lambda_1(\rho^l,v^l),\\
(\rho^{\sigma_i},v^{\sigma_i}) & \text{if } \lambda_1(\rho^{\sigma_{i-1}},v^{\sigma_{i-1}})<\frac{x}{t-\tilde{t}}\leq \lambda_1(\rho^{\sigma_i},v^{\sigma_i})\; \text{ for } \; i=1,...,N,\\
(\hat{\rho},\hat{v}) & \text{if } \lambda_1(\hat{\rho},\hat{v}) < \frac{x}{t-\tilde{t}}\leq 0,\\
(\check{\rho}_2,\check{v}_2) & \text{if } 0<\frac{x}{t-\tilde{t}}\leq v^r,\\
(\rho^r,v^r) & \text{if } \frac{x}{t-\tilde{t}}>v^r,
\end{cases}
\end{split}
\end{equation*}
where $\lbrace (\rho^{\sigma_i},v^{\sigma_i}) \rbrace_{i=1}^N$ are the points defined in (\ref{points_of_the_rarefaction_fan}) for the rarefaction connecting $(\rho^l,v^l)$ to $(\hat{\rho},\hat{v})$. Furthermore
$$\Delta TV_{\tilde{t}}(\rho)\leq C_1\,|\rho^l-\rho^r|,\;\; \Delta TV_{\tilde{t}}(v)=0 \; \text{ and } \; \Delta_{\tilde{t}}\, \mathcal{N}=N+1,$$
where $C_1$ is a constant depending only by $q$ and the domain $\mathcal{D}_{v_1,v_2,w_1,w_2}$.\\
For the Riemann invariant $w$, we have
$$\Delta TV_{\tilde{t}}(w)\leq C_2\,|v^r-v^l|,$$
where $C_2$ is a constant depending only by $q$ and the domain $\mathcal{D}_{v_1,v_2,w_1,w_2}$.
\end{itemize}
\end{prop}
\begin{proof}
Let us observe that $\rho^l\,v^l \leq q $, otherwise the non-classical shock would appear as the solution given by $\mathcal{RS}^q_2$ to the Riemann problem with constant initial datum $(\rho^l,v^l)$.\\
The hypotheses $\rho^l> \rho^r$ ensure that the wave with negative speed is a rarefaction which we approximate with a rarefaction fan. By Remark \ref{remark_rarefaction_fan_as_shocks}, we can consider each discontinuity of the rarefaction fan as a single (non-classical) shock.\\
\begin{figure}[hbtp]
\centering
\begin{subfigure}[hbtp]{0.48\linewidth}
\definecolor{xdxdff}{rgb}{0.49019607843137253,0.49019607843137253,1.}
\begin{tikzpicture}[scale=0.3,line cap=round,line join=round,>=triangle 45,x=1.0cm,y=1.0cm]
\draw[->,color=black] (-1.,0.) -- (17.,0.);
\draw[->,color=black] (0.,-0.5) -- (0.,10.);
\clip(-1.,-1) rectangle (18.,10.);
\draw[smooth,samples=50,domain=0.001:17.0] plot(\x,{4.0*(\x)-(\x)^(1.5)});
\draw [domain=-1.:17.] plot(\x,{(--5.-0.*\x)/1.});
\draw [domain=0.0:17.0] plot(\x,{(-0.--4.28*\x)/13.5712});
\begin{scriptsize}
\draw (-0.9,5.2) node[anchor=north west] {$q$};
\draw (-1.4,9) node[anchor=north west] {$\rho v$};
\draw (15.,0.1) node[anchor=north west] {$\rho$};
\draw (10.1,3.3) node[anchor=north west] {$(\rho^l,v^l)=(\rho^k,v^k)$};
\draw (13.6,4.70) node[anchor=north west] {$(\rho^r,v^r)$};
\end{scriptsize}
\begin{scriptsize}
\draw [fill=xdxdff] (13.5712,4.28) circle (5pt);
\draw [fill=xdxdff] (14.5,2.78) circle (5pt);
\end{scriptsize}
\end{tikzpicture}
\caption{Position of the points in the plane $(\rho,\rho\,v)$.}
\end{subfigure}
\quad
\begin{subfigure}[hbtp]{0.48\linewidth}
\begin{tikzpicture}[scale=0.4,line cap=round,line join=round,>=triangle 45,x=1.0cm,y=1.0cm]
\draw[->,color=black] (-6.5,0.) -- (6.,0.);
\draw[->,color=black] (0.,-0.5) -- (0.,8.5);
\clip(-7,-0.5) rectangle (6.,8.5);
\draw (4.,0.)-- (0.,4.);
\draw [domain=-6.5:6.] plot(\x,{(--4.-0.*\x)/1.});
\draw [domain=-6.5:0.0] plot(\x,{(-15.84--4.6*\x)/-3.96});
\draw [domain=-6.5:0.0] plot(\x,{(-19.04--4.76*\x)/-4.76});
\draw [domain=-6.5:0.0] plot(\x,{(-12.4--4.18*\x)/-3.1});
\draw [domain=-6.5:0.0] plot(\x,{(-11.76--4.9*\x)/-2.94});
\draw [domain=-6.5:0.0] plot(\x,{(-9.04--4.68*\x)/-2.26});
\begin{scriptsize}
\draw (-7,2.8) node[anchor=north west] {$(\rho^l,v^l)=(\rho^k,v^k)$};
\draw (2.5,2.8) node[anchor=north west] {$(\rho^r,v^r)$};
\draw (-7,5.4) node[anchor=north west] {$(\rho^l,v^l)=(\rho^k,v^k)$};
\draw (2.5,5.4) node[anchor=north west] {$(\rho^r,v^r)$};
\draw (-0.7,8.) node[anchor=north west] {$t$};
\draw (4.6,4.) node[anchor=north west] {$\tilde{t}$};
\draw (4.5,0.1) node[anchor=north west] {$x$};
\end{scriptsize}
\end{tikzpicture}
\caption{After the interaction a new rarefaction wave (approximated with a rarefaction fan) appears.}
\end{subfigure}
\caption{Interaction between a wave with negative speed and $x=0$ for the case $(\rho^l,v^l)=(\rho^k,v^k)$ and $\rho^l > \rho^r$. After the interaction the constraint is satisfied.}\label{fig_rarefaction_from_right_caso_1}
\end{figure}
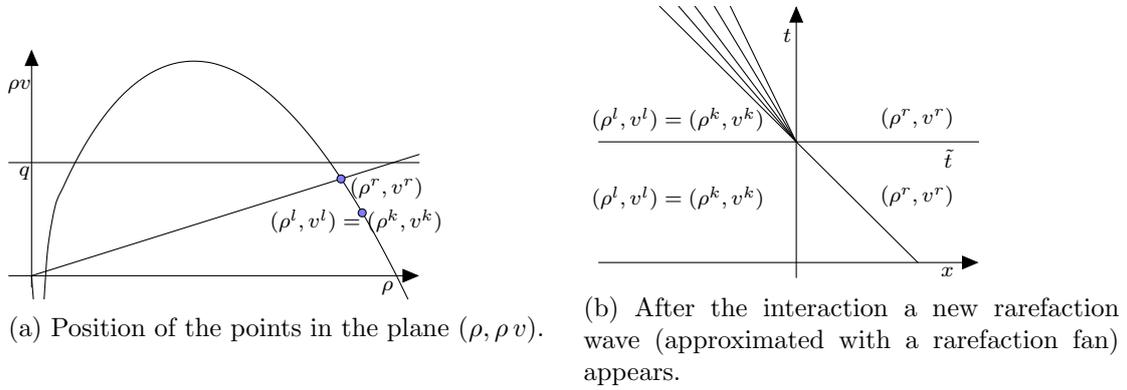

In case $(i)$ the constraint is satisfied (see Figure \ref{fig_rarefaction_from_right_caso_1}), indeed the solution to the Riemann problem with initial datum
$$(\rho,v)(\tilde{t},x) = \begin{cases}
(\rho^l,v^l) & \text{if } x\leq 0,\\
(\rho^r,v^r) & \text{if } x>0,
\end{cases}
$$
is a rarefaction wave which can be approximated by the rarefaction fan (\ref{rarefaction_fan}). Therefore it has negative speeds that vary between $\lambda_1(\rho^l,v^l)$ and $\lambda_1(\rho^r,v^r)$ and
$$\mathcal{RS}^q_2((\rho^l,v^l),(\rho^r,v^r))(0)=(\rho^r,v^r)$$
which satisfies the constraint by hypothesis.\\
Therefore
\begin{equation*}
\begin{split}
&\Delta TV_{\tilde{t}}(\rho)= |\rho^r-\rho^l|-|\rho^r-\rho^l| =0,\\
& \Delta TV_{\tilde{t}}(v) = |v^r-v^l|-|v^r-v^l|=0,\\
& \Delta TV_{\tilde{t}}(w) =|w^r-w^l|-|w^r-w^l|=0 \; \text{ and}\\
& \Delta_{\tilde{t}}\, \mathcal{N}=N-1.
\end{split}
\end{equation*}
In case $(ii)$ the constraint is not satisfied (see Figure \ref{fig_rarefaction_from_right_caso_2}).\\
\begin{figure}[hbtp]
\centering
\begin{subfigure}[hbtp]{0.48\linewidth}
\definecolor{uququq}{rgb}{0.25098039215686274,0.25098039215686274,0.25098039215686274}
\definecolor{xdxdff}{rgb}{0.49019607843137253,0.49019607843137253,1.}
\begin{tikzpicture}[scale =0.3,line cap=round,line join=round,>=triangle 45,x=1.0cm,y=1.0cm]
\draw[->,color=black] (-1.,0.) -- (17.,0.);
\draw[->,color=black] (0.,-0.7) -- (0.,10.);
\clip(-1.5,-1) rectangle (17.,10.);
\draw[smooth,samples=50,domain=0.001:17.0] plot(\x,{4.0*(\x)-(\x)^(1.5)});
\draw [domain=-1.:17.] plot(\x,{(--5.-0.*\x)/1.});
\draw [domain=0.0:17.0] plot(\x,{(-0.--8.17*\x)/10.26});
\begin{scriptsize}
\draw (-0.9,5.) node[anchor=north west] {$q$};
\draw (-1.5,9.3) node[anchor=north west] {$\rho v$};
\draw (15,0.1) node[anchor=north west] {$\rho$};
\draw (6.5,3.34) node[anchor=north west] {$(\rho^l,v^l)=(\rho^k,v^k)$};
\draw (10.3,9) node[anchor=north west] {$(\rho^r,v^r)$};
\draw (6.,5.) node[anchor=north west] {$(\check{\rho}_2,\check{v}_2)$};
\draw (10.5,5) node[anchor=north west] {$(\hat{\rho},\hat{v})$};
\end{scriptsize}
\begin{scriptsize}
\draw [fill=xdxdff] (10.26,8.17) circle (5pt);
\draw [fill=xdxdff] (14.5,2.78) circle (5pt);
\draw [fill=uququq] (6.27,5.) circle (5pt);
\draw [fill=uququq] (13.09,5) circle (5pt);
\end{scriptsize}
\end{tikzpicture}
\caption{Position of the points in the plane $(\rho,\rho\,v)$.}
\end{subfigure}
\quad
\begin{subfigure}[hbtp]{0.48\linewidth}
\begin{tikzpicture}[scale=0.4,line cap=round,line join=round,>=triangle 45,x=1.0cm,y=1.0cm]
\draw[->,color=black] (-6.,0.) -- (6.5,0.);
\draw[->,color=black] (0.,-0.7) -- (0.,8.5);
\clip(-6.,-0.7) rectangle (6.5,8.5);
\draw (4.,0.)-- (0.,4.);
\draw [domain=-6.:6.5] plot(\x,{(--4.-0.*\x)/1.});
\draw [domain=-6.0:0.0] plot(\x,{(-15.84--4.6*\x)/-3.96});
\draw [domain=-6.0:0.0] plot(\x,{(-19.04--4.76*\x)/-4.76});
\draw [domain=-6.0:0.0] plot(\x,{(-12.4--4.18*\x)/-3.1});
\draw [domain=-6.0:0.0] plot(\x,{(-11.76--4.9*\x)/-2.94});
%\draw [domain=-6.0:0.0] plot(\x,{(-9.04--4.68*\x)/-2.26});
\draw [domain=0.0:6.5] plot(\x,{--10.33/2.58--4.22*\x/4.});
\begin{scriptsize}
\draw (-4.80,2.5) node[anchor=north west] {$(\rho^l,v^l)=(\rho^k,v^k)$};
\draw (2.72,2.5) node[anchor=north west] {$(\rho^r,v^r)$};
\draw (-4.8,6.2) node[anchor=north west] {$(\rho^l,v^l)$};
\draw (2.72,6.2) node[anchor=north west] {$(\rho^r,v^r)$};
\draw (-0.7,8.2) node[anchor=north west] {$t$};
\draw (5.2,4.) node[anchor=north west] {$\tilde{t}$};
\draw (5.2,0.) node[anchor=north west] {$x$};
\draw (-1.9,7.5) node[anchor=north west] {$(\hat{\rho},\hat{v})$};
\draw (0.,7.5) node[anchor=north west] {$(\check{\rho}_2,\check{v}_2)$};
\end{scriptsize}
\end{tikzpicture}
\caption{After the interaction a new rarefaction joining $(\rho^l,v^l)$ and $(\hat{\rho},\hat{v})$ and the non-classical shock appear.}
\end{subfigure}
\caption{Interaction between a wave with negative speed and $x=0$ for the case $(\rho^l,v^l)=(\rho^k,v^k)$ and $\rho^l > \rho^r$. After the interaction the constraint is not satisfied.}\label{fig_rarefaction_from_right_caso_2}
\end{figure}

By Lemma \ref{lemma_concave_functions_properties} applied to $(\rho^l,v^l)$ (which satisfies $\rho^l\,v^l \leq q$), we find $\hat{\rho}\leq\rho^l$. Then $\mathcal{RS}^q_2$ connects $(\rho^l,v^l)$ to $(\hat{\rho},\hat{v})$ with a rarefaction wave which can be approximated with a rarefaction fan. Therefore the difference in the number of waves before and after the interaction is 
$$\Delta_{\tilde{t}} \, \mathcal{N} = (N+2)-1 = N+1.$$
We have:
\begin{itemize}
\item $\check{\rho}_2<\rho^r$, because $\rho^r\,v^r>q=\check{\rho}_2\,\check{v}_2$ and $v^r=\check{v}_2$;
\item $\check{\rho}_2<\hat{\rho}$, because $v^r+p(\rho^r)=\hat{v}+p(\hat{\rho})=v^l+p(\rho^l)$ and, since $\rho^r\,v^r>q$, by Lemma \ref{lemma_concave_functions_properties} we have $\rho^r<\hat{\rho}$. Therefore $\check{\rho}_2<\rho^r<\hat{\rho}$.\end{itemize}
Hence we find
\begin{equation*}
\begin{split}
\Delta TV_{\tilde{t}}(\rho) & = |\rho^r-\check{\rho}_2|+|\check{\rho}_2-\hat{\rho}|+|\hat{\rho}-\rho^l|-|\rho^r-\rho^l|=\\
& = \rho^r-\check{\rho}_2+\hat{\rho}-\check{\rho}_2+\rho^l-\hat{\rho}-\rho^l+\rho^r=\\
& = 2\,(\rho^r-\check{\rho}_2) \leq C\,(\rho^l-\rho^r),
\end{split}
\end{equation*}
where we have applied Lemma \ref{lemma_disuguaglianze_triangolo}.\\
For the second component we have
\begin{equation*}
\begin{split}
\Delta TV_{\tilde{t}} (v) &= |v^r-\check{v}_2|+|\check{v}_2-\hat{v}|+|\hat{v}-v^l|-|v^r-v^l|=\\
& = 0+ \check{v}_2-\hat{v}+\hat{v}-v^l-v^r+v^l=\\
& = 0,
\end{split}
\end{equation*}
because $v^r=\check{v}_2$, $\hat{v}<\check{v}_2$ (because $\hat{\rho}>\check{\rho}_2$) and $v^r>\hat{v}\geq v^l$ (because $\rho^r<\hat{\rho}\leq\rho^l$ and $v^r+p(\rho^r)=\hat{v}+p(\hat{\rho})=v^l+p(\rho^l)$).\\
For the Riemann invariant $w$, we have:
\begin{itemize}
\item $w^l=w^r=\hat{w}$;
\item $w^r>\check{w}_2$, because $\rho^r>\check{\rho}_2$, $v^r=\check{v}_2$;
\item $\check{w}_2<w^l$, by Lemma \ref{lemma_riemann_invariant_property} and the inequalities $\rho^l<\hat{\rho}<\check{\rho}_2$.
\end{itemize} 
Therefore
\begin{equation*}
\begin{split}
\Delta_{\tilde{t}} TV (w) & =|w^r-\check{w}_2|+|\check{w}_2-\hat{w}|-|\hat{w}-w^l|-|w^r-w^l|=\\
& = w^r-\check{w}_2+\hat{w}-\check{w}_2=2\,(w^r-\check{w}_2)=\\
& = 2\,(v^r+p(\rho^r)-\check{v}_2-p(\check{\rho}_2)) \leq\\
& \leq 2\,M\,(\rho^r-\check{\rho}_2),
\end{split}
\end{equation*}
where $M$ is a positive constant depending on the invariant domain, because the function $\rho \to p(\rho)$ is bi-Lipschitz by Lemma \ref{pressure_Lipschitz}.\\
By Lemma \ref{lemma_disuguaglianze_triangolo}, there exists a constant $c$ depending on the invariant domain, for which
$$\rho^r-\check{\rho}_2\leq\dfrac{1}{c}(v^r-\hat{v})\leq \dfrac{1}{c}(v^r-v^l),$$
because $\rho^l>\hat{\rho}$ if and only if $v^l<\hat{v}$.\\
Therefore
$$\Delta TV_{\tilde{t}}(w)\leq C \, (v^r-v^l),$$
where $C$ is a positive constant depending only on the invariant domain.
\end{proof}
\subsubsection{Case $(\rho^l,v^l)\neq (\rho^k,v^k)$}
\begin{prop}\label{prop_rarefaction_from_right_case_2}
Assume $\lambda_1(\rho_\text{min},v_\text{min})<0$. Fix $(\rho^l,v^l)$, $(\rho^k,v^k)$ and $(\rho^r,v^r)$ in the invariant domain $\mathcal{D}_{v_1,v_2,w_1,w_2}$ such that
$$v^k+p(\rho^k)<v^l+p(\rho^l).$$
Assume that $\rho^r < \rho^k$ and that a wave joining $(\rho^k,v^k)$ to $(\rho^r,v^r)$ with negative speed interacts in $x=0$ at time $\tilde{t}>0$ with a wave with zero speed which connects $(\rho^l,v^l)$ to $(\rho^k,v^k)$.\\
Then the wave with negative speed is a rarefaction wave and the wave with zero speed is a non-classical shock. Moreover
$$\mathcal{RS}^q_2((\rho^l,v^l),(\rho^r,v^r))\left(\dfrac{x}{t-\tilde{t}}\right) = \begin{cases}
(\rho^l,v^l) & \text{if } \frac{x}{t-\tilde{t}}\leq 0,\\
(\check{\rho}_2,\check{v}_2) & \text{if } 0<\frac{x}{t-\tilde{t}}\leq v^r,\\
(\rho^r,v^r) & \text{if } \frac{x}{t-\tilde{t}}>v^r.
\end{cases}
$$
Finally
$$\Delta TV_{\tilde{t}} (\rho) \leq C_1 \, (\rho^k-\rho^r), \; \; \Delta TV_{\tilde{t}}(v) = 0 \; \text{ and } \; \Delta_{\tilde{t}}\, \mathcal{N}=0,$$
where $C_1$ is a constant depending only by $q$ and the domain $\mathcal{D}_{v_1,v_2,w_1,w_2}$.\\
For the Riemann invariant $w$, we have
$$\Delta TV_{\tilde{t}}(w)\leq C_2\,|v^r-v^l|,$$
where $C_2$ is a constant depending only by $q$ and the domain $\mathcal{D}_{v_1,v_2,w_1,w_2}$.
\end{prop}
\begin{proof}
Proposition \ref{when_non_classical_shock} ensures that $\rho^l\,v^l=\rho^k\,v^k=q$. The hypothesis $v^k+p(\rho^k)<v^l+p(\rho^l)$ and Lemma \ref{lemma_riemann_invariant_property} imply that $\rho^k<\rho^l$. By Propositions \ref{prop_classical_shock_not_possible} and \ref{when_non_classical_shock} the wave with zero speed is a non-classical shock. Since $\rho^r< \rho^k$, the wave with negative speed is a rarefaction. By Remark \ref{remark_rarefaction_fan_as_shocks}, we can consider it as a single shock.\\
Let us call $(\rho^m,v^m)$ the middle state of the classical solution to the Riemann problem with initial datum
$$
(\rho,v)(\tilde{t},x)=\begin{cases}
(\rho^l,v^l) & \text{if } x\leq 0,\\
(\rho^r,v^r) & \text{if } x>0.
\end{cases}
$$
Since $v^r+p(\rho^r)=v^k+p(\rho^k)$, $\rho^k\,v^k=q$ and $\rho^r< \rho^k$, by Lemma \ref{lemma_concave_functions_properties} we find $\rho^r\, v^r >q$. Moreover $\rho^m\,v^m >\rho^r\,v^r>q$, because the conditions $v^m=v^r$ and
$$v^m+p(\rho^m)=v^l+p(\rho^l) >v^k+p(\rho^k) = v^r+p(\rho^r)$$
imply $\rho^m>\rho^r$.\\
Therefore the constraint is not satisfied and the solution $\mathcal{RS}^q_2((\rho^l,v^l),(\rho^r,v^r))(0)$ is the non-classical shock; see Figure \ref{fig_rarefaction_from_right_caso_3}.\\
\begin{figure}[hbtp]
\centering
\begin{subfigure}[hbtp]{0.48\linewidth}
\centering
\definecolor{uququq}{rgb}{0.25098039215686274,0.25098039215686274,0.25098039215686274}
\definecolor{xdxdff}{rgb}{0.49019607843137253,0.49019607843137253,1.}
\begin{tikzpicture}[scale=0.35,line cap=round,line join=round,>=triangle 45,x=1.0cm,y=1.0cm]
\draw[->,color=black] (-1.,0.) -- (17.,0.);
\draw[->,color=black] (0.,-0.5) -- (0.,10.);
\clip(-1.5,-1) rectangle (18.,10.);
\draw[smooth,samples=50,domain=0.001:17.0] plot(\x,{4.0*(\x)-(\x)^(1.5)});
\draw [domain=-1.:17.] plot(\x,{(--5.-0.*\x)/1.});
\draw [domain=0.0:17.0] plot(\x,{(-0.--8.43*\x)/9.92});
\draw[smooth,samples=50,domain=0.001:17.0] plot(\x,{3.7*(\x)-(\x)^(1.5)});
\begin{scriptsize}
\draw (-0.8,5.) node[anchor=north west] {$q$};
\draw (-1.4,9.1) node[anchor=north west] {$\rho v$};
\draw (15,0.1) node[anchor=north west] {$\rho$};
\draw (10.,9.05) node[anchor=north west] {$(\rho^m,v^m)$};
\draw (5.5,5.14) node[anchor=north west] {$(\check{\rho}_2,\check{v}_2)$};
\draw (12.,6.5) node[anchor=north west] {$(\rho^l,v^l)=(\hat{\rho},\hat{v})$};
\draw (9.7,5.14) node[anchor=north west] {$(\rho^k,v^k)$};
\draw (8.1,7.7) node[anchor=north west] {$(\rho^r,v^r)$};
\end{scriptsize}
\begin{scriptsize}
\draw [fill=xdxdff] (9.92,8.43) circle (5pt);
\draw [fill=uququq] (5.87,5.) circle (5pt);
\draw [fill=uququq] (13.09,5) circle (5pt);
\draw [fill=uququq] (10.34,5.00) circle (5pt);
\draw [fill=uququq] (8.12,6.9) circle (5pt);
\end{scriptsize}
\end{tikzpicture}
\caption{Position of the points in the plane $(\rho,\rho\,v)$.}
\end{subfigure}
\quad
\begin{subfigure}[hbtp]{0.48\linewidth}
\centering
\definecolor{fftttt}{rgb}{1.,0.2,0.2}
\definecolor{qqqqff}{rgb}{0.,0.,1.}
\definecolor{ttttff}{rgb}{0.2,0.2,1.}
\begin{tikzpicture}[scale=0.6,line cap=round,line join=round,>=triangle 45,x=1.0cm,y=1.0cm]
\draw[->,color=black] (-4.,0.) -- (4.,0.);
\draw[->,color=black] (0.,-0.5) -- (0.,5.5);
\clip(-4.,-0.5) rectangle (4.,5.5);
\draw [dash pattern=on 3pt off 3pt,domain=-4.:4.] plot(\x,{(--2.-0.*\x)/1.});
\draw [line width=1.6pt,color=ttttff] (3.,0.)-- (0.,2.);
\draw [line width=1.6pt,color=fftttt] (0.,2.) -- (0.,5.);
\draw [line width=1.6pt,color=fftttt,domain=0.0:4.0] plot(\x,{(--6.88--4.3*\x)/3.44});
\draw [line width=1.6pt,color=ttttff] (0.,2.)-- (0.,0.);
\begin{scriptsize}
\draw (-0.6,2.) node[anchor=north west] {$\tilde{t}$};
\draw (-0.1,1.1) node[anchor=north west] {$(\rho^k,v^k)$};
\draw (-2,1.1) node[anchor=north west] {$(\rho^l,v^l)$};
\draw (2.1,4.8) node[anchor=north west] {$(\rho^r,v^r)$};
\draw (2.1,1.1) node[anchor=north west] {$(\rho^r,v^r)$};
\draw (-0.1,4.8) node[anchor=north west] {$(\check{\rho}_2,\check{v}_2)$};
\draw (-0.6,5.30) node[anchor=north west] {$t$};
\draw (3.1,-0.1) node[anchor=north west] {$x$};
\draw (2.6,0.) node[anchor=north west] {$\bar{x}$};
\draw (0.,0.) node[anchor=north west] {$0$};
\draw (-3.5,4.8) node[anchor=north west] {$(\rho^l,v^l)=(\hat{\rho},\hat{v})$};
\end{scriptsize}
\begin{scriptsize}
\draw [fill=qqqqff] (0.,7.28) circle (5pt);
\end{scriptsize}
\end{tikzpicture}
\caption{After the interaction the non-classical shock appears.}
\end{subfigure}
\caption{Interaction between a wave with negative speed and $x=0$ for the case $(\rho^l,v^l)\neq(\rho^k,v^k)$ and $\rho^k> \rho^r$. After the interaction the constraint is not satisfied.}\label{fig_rarefaction_from_right_caso_3}
\end{figure}
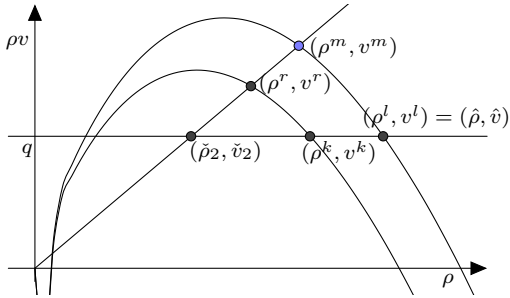
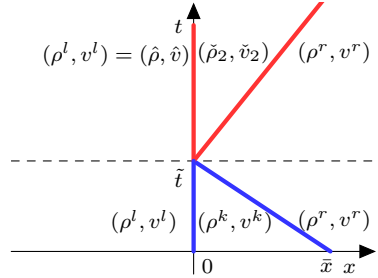

Since $\rho^l\,v^l=q$ the solution for $x\leq 0$ is constant. Therefore
$$\Delta _{\tilde{t}} \, \mathcal{N}=2-2=0.$$
We have:
\begin{itemize}
\item $\rho^r >\check{\rho}_2$, because $\rho^r\,v^r >q= \check{\rho}_2\, \check{v}_2$ and $v^r=\check{v}_2$;
\item $\check{\rho}_2 <\rho^l$, because $\rho^l >\rho^k >\rho^r >\check{\rho}_2$.
\end{itemize}
Therefore
\begin{equation*}
\begin{split}
\Delta TV_{\tilde{t}}(\rho) & =|\rho^r-\check{\rho}_2|+|\check{\rho}_2-\rho^l|-|\rho^r-\rho^k|-|\rho^k-\rho^l|=\\
& = \rho^r-\check{\rho}_2+\rho^l-\check{\rho}_2-\rho^k+\rho^r-\rho^l+\rho^k=\\
&= 2\,(\rho^r-\check{\rho}_2)\leq C\,(\rho^k-\rho^r),
\end{split}
\end{equation*}
by Lemma \ref{lemma_disuguaglianze_triangolo}.\\
For the second component, we obtain
\begin{equation*}
\begin{split}
\Delta TV_{\tilde{t}} (v) & = |v^r-\check{v}_2|+|\check{v}_2-v^l|-|v^r-v^k|-|v^k-v^l|=\\
&= 0+\check{v}_2-v^l-v^r+v^k-v^k+v^l=0,
\end{split}
\end{equation*}
because $\check{v}_2=v^r$.\\
For the Riemann invariant $w$, we have $w^r=w^k$ and $w^k<w^l$ by hypothesis. Moreover $\check{w}_2<w^r$, because $\rho^r>\check{\rho}_2$ and $v^r=\check{v}_2$. Therefore
\begin{equation*}
\begin{split}
\Delta TV_{\tilde{t}} (w) & =|w^r-\check{w}_2|+|\check{w}_2-\hat{w}|-|w^r-w^k|-|w^k-w^l|=\\
& = w^r-\check{w}_2+w^l-\check{w}_2-w^l+w^k=2\,(w^r-\check{w}_2)=\\
& = 2\,(v^r+p(\rho^r)-\check{v}_2-p(\check{\rho}_2)) \leq\\
& \leq 2\,M\,(\rho^r-\check{\rho}_2),
\end{split}
\end{equation*}
where $M$ is a positive constant depending on the invariant domain, because the function $\rho \to p(\rho)$ is bi-Lipschitz by Lemma \ref{pressure_Lipschitz}.\\
By Lemma \ref{lemma_disuguaglianze_triangolo}, there exists a constant $c$ depending on the invariant domain, for which
$$\rho^r-\check{\rho}_2\leq\dfrac{1}{c}(v^r-\hat{v}) = \dfrac{1}{c}(v^r-v^l),$$
because $(\rho^l,v^l)=(\hat{\rho},\hat{v})$.\\
Therefore
$$\Delta TV_{\tilde{t}}(w)\leq C \, (v^r-v^l),$$
where $C$ is a positive constant depending only on the invariant domain.
\end{proof}
\subsection{Other interactions}
Let us now consider the case of two waves interacting at time $\tilde{t}$, ``far'' from the line $x=0$.
\subsubsection{Two waves with positive speed}
The next proposition states that two waves with positive speed cannot interact, provided that the functions $\rho \to \rho\, L_1(\rho,\rho_0,v_0)$ passing through some point $(\rho_0,v_0)$ in $\mathbb{R}^+\times \mathbb{R}^+$ are decreasing in the invariant domain $\mathcal{D}_{v_1,v_2,w_1,w_2}$.
\begin{prop}\label{prop_interacting_waves_pos_speed}
Assume $\lambda_1(\rho_\text{min},v_\text{min})<0$. Fix three points $(\rho^l,v^l)$, $(\rho^k,v^k)$ and $(\rho^r,v^r)$ in the invariant domain $\mathcal{D}_{v_1,v_2,w_1,w_2}$ and suppose that two waves with positive speed join respectively $(\rho^l,v^l)$ to $(\rho^k,v^k)$ and $(\rho^k,v^k)$ to $(\rho^r,v^r)$. Then:
\begin{enumerate}
\item[(i)] the waves are contact discontinuities;
\item[(ii)] the waves do not interact.
\end{enumerate}
\end{prop}
\begin{proof}
By Lemma \ref{lemma_curve_decrescenti}, the hypothesis $\lambda_1(\rho_\text{min},v_\text{min})<0$ implies that waves with positive speed are of the second family, i.e. contact discontinuities. Hence $v^l=v^k=v^r$.\\
The wave joining $(\rho^l,v^l)$ to $(\rho^k,v^k)$ has propagation speed $v^k$ and the wave connecting $(\rho^k,v^k)$ to $(\rho^r,v^r)$ has speed $v^r$; see Figure \ref{fig_onde_positive_parallele}. Therefore no interactions happen between the waves.
\end{proof}
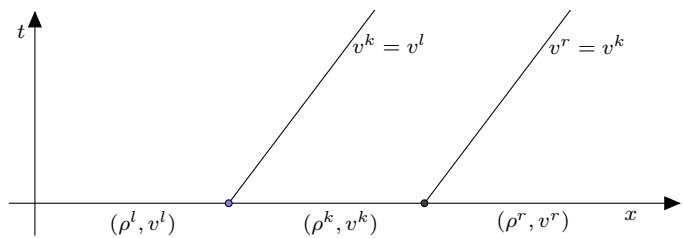
\begin{figure}[hbtp]
\centering
\definecolor{uququq}{rgb}{0.25098039215686274,0.25098039215686274,0.25098039215686274}
\definecolor{xdxdff}{rgb}{0.49019607843137253,0.49019607843137253,1.}
\begin{tikzpicture}[scale = 0.85, line cap=round,line join=round,>=triangle 45,x=1.0cm,y=1.0cm]
\draw[->,color=black] (-0.4,0.) -- (10.,0.);
\draw[->,color=black] (0.,-0.5) -- (0.,3.);
\clip(-0.4,-0.5) rectangle (10.,3.);
\draw [domain=3.0:10.0] plot(\x,{(-12.--4.*\x)/3.});
\draw [domain=6.02:10.0] plot(\x,{(-13.266--2.2*\x)/1.65});
\begin{scriptsize}
\draw (-0.4,2.9) node[anchor=north west] {$t$};
\draw (9.,0.) node[anchor=north west] {$x$};
\draw (1.03,0.) node[anchor=north west] {$(\rho^l,v^l)$};
\draw (4.02,0) node[anchor=north west] {$(\rho^k,v^k)$};
\draw (7.02,0.) node[anchor=north west] {$(\rho^r,v^r)$};
\draw (4.79,2.76) node[anchor=north west] {$v^k=v^l$};
\draw (7.83,2.75) node[anchor=north west] {$v^r=v^k$};
\end{scriptsize}
\begin{scriptsize}
\draw [fill=xdxdff] (3.,0.) circle (1.5pt);
\draw [fill=uququq] (6.03,0.) circle (1.5pt);
\end{scriptsize}
\end{tikzpicture}
\caption{Under the hypotheses $\lambda_1(\rho_\text{min},v_\text{min})<0$, two waves with positive speed cannot interact, because they travel with the same speed.}\label{fig_onde_positive_parallele}
\end{figure}
\subsubsection{Two waves with negative speed}
The next proposition shows that two waves with negative speed can interact. After the interaction a new wave with negative speed arises and the total variation decreases.
\begin{prop}\label{prop_interacting_waves_with_neg_speed}
Fix three points $(\rho^l,v^l)$, $(\rho^k,v^k)$ and $(\rho^r,v^r)$ in the invariant domain $\mathcal{D}_{v_1,v_2,w_1,w_2}$ and suppose that two waves with negative speed connect respectively $(\rho^l,v^l)$ to $(\rho^k,v^k)$ and $(\rho^k,v^k)$ to $(\rho^r,v^r)$. If at time $\tilde{t}>0$ the waves interact at some point $\bar{x}\in \mathbb{R}$, then:
\begin{enumerate}
\item[(i)] after the interaction only a new wave of the first family appears;
\item[(ii)] the total variation decreases or remains constant:
$$\Delta TV_{\tilde{t}}(\rho) \leq 0, \;\;\; \Delta TV_{\tilde{t}}(v) \leq 0 \; \text{ and } \; \Delta TV_{\tilde{t}} (w)=0;$$
\item[(iii)] the number of waves before and after the interaction decreases:
$$\Delta_{\tilde{t}} \mathcal{N}= -1.$$
\end{enumerate} 
\end{prop}
\begin{proof}
Waves with negative speed are of the first family, i.e. shock or rarefaction waves. Therefore
$$v^l+p(\rho^l)=v^k+p(\rho^k)=v^r+p(\rho^r).$$
At time $\tilde{t}$ we have to solve the Riemann problem with initial datum
$$(\rho,v)(\tilde{t},x)=
\begin{cases}
(\rho^l,v^l) & \text{if } x\leq \bar{x},\\
(\rho^r,v^r) & \text{if } x >\bar{x}.
\end{cases}
$$
Since $v^l+p(\rho^l)=v^r+p(\rho^r)$, the solution is a new wave of the first family. If the solution is a shock, then the number of waves decreases:
$$\Delta_{\tilde{t}} \mathcal{N}=-1.$$
If the solution is a rarefaction wave, by Remark \ref{remark_rarefaction_fan_as_shocks}, we approximate it with a single shock. Hence
$$\Delta_{\tilde{t}} \mathcal{N}=-1.$$
Therefore we find (see Figure \ref{fig_onde_vel_neg})
$$\Delta TV_{\tilde{t}}(\rho)=|\rho^r-\rho^l|-|\rho^r-\rho^k|-|\rho^k-\rho^l| \; \text{ and}$$
$$\Delta TV_{\tilde{t}}(v) =|v^r-v^l|-|v^r-v^k|-|v^k-v^l|.$$
By the triangular inequality, we have
$$|\rho^r-\rho^l| \leq |\rho^r-\rho^k|+|\rho^k-\rho^l|.$$
Hence
$$\Delta TV_{\tilde{t}}(\rho) \leq 0.$$
Similarly for the velocity.\\
For the Riemann invariant $w$, we have $w^r=w^k=w^l$. Hence
$$\Delta TV_{\tilde{t}}(w)=|w^r-w^l|-|w^r-w^k|-|w^k-w^l|=0.$$
\end{proof}
\begin{figure}[hbtp]
\centering
\definecolor{qqqqff}{rgb}{0.,0.,1.}
\definecolor{xdxdff}{rgb}{0.49019607843137253,0.49019607843137253,1.}
\begin{tikzpicture}[line cap=round,line join=round,>=triangle 45,x=1.0cm,y=1.0cm]
\draw[->,color=black] (-0.5,0.) -- (11.,0.);
\draw[->,color=black] (0.,-0.5) -- (0.,5.);
\clip(-0.5,-0.5) rectangle (11.,5.);
\draw [domain=-0.5:11.] plot(\x,{(--2.32-0.*\x)/1.});
\draw (5.55,0.)-- (3.25,2.32);
\draw (10.,0.)-- (3.25,2.32);
\draw (3.25,2.32)-- (1.21,7.37);
\draw [dash pattern=on 1pt off 1pt] (3.25,2.32)-- (3.2562,0.);
\draw (1.51,1.35) node[anchor=north west] {$(\rho^l,v^l)$};
\draw (4.84,1.35) node[anchor=north west] {$(\rho^k,v^k)$};
\draw (7.83,1.35) node[anchor=north west] {$(\rho^r,v^r)$};
\draw (-0.5,4.9) node[anchor=north west] {$t$};
\draw (10.3,0.) node[anchor=north west] {$x$};
\draw (7.73,4.45) node[anchor=north west] {$(\rho^r,v^r)$};
\draw (0.93,4.45) node[anchor=north west] {$(\rho^l,v^l)$};
\draw (-0.5,2.3) node[anchor=north west] {$\tilde{t}$};
\draw (3.04,0.) node[anchor=north west] {$\bar{x}$};
\begin{scriptsize}
\draw [fill=xdxdff] (5.55,0.) circle (1.5pt);
\draw [fill=xdxdff] (3.25,2.32) circle (1.5pt);
\draw [fill=xdxdff] (10.,0.) circle (1.5pt);
\draw [fill=qqqqff] (1.21,7.37) circle (1.5pt);
\draw[color=qqqqff] (-1.83,5.61) node {$D$};
\end{scriptsize}
\end{tikzpicture}
\caption{Two waves with negative speed interact at time $\tilde{t}$. After the interaction a new wave with negative speed appears.}\label{fig_onde_vel_neg}
\end{figure}
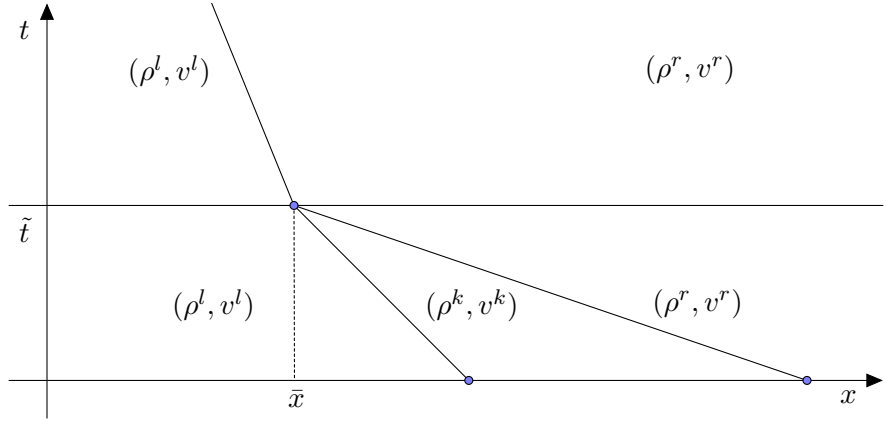
\subsubsection{A wave has negative speed and a wave has positive speed}
\begin{prop}\label{prop_interacting_wave_neg_speed_pos_speed}
Assume $\lambda_1(\rho_{\text{min}},v_\text{min})<0$. Fix three points $(\rho^l,v^l)$, $(\rho^k,v^k)$ and $(\rho^r,v^r)$ in the invariant domain $\mathcal{D}_{v_1,v_2,w_1,w_2}$. Assume that a wave with positive speed joins $(\rho^l,v^l)$ to $(\rho^k,v^k)$ and a wave with negative speed connects $(\rho^k,v^k)$ to $(\rho^r,v^r)$ and that the two waves interacts at time $\tilde{t}>0$ at a point $\bar{x}\in \mathbb{R}$. Let $(\rho^m,v^m)$ be the middle state of the classical solution to the Riemann problem with initial datum
\begin{equation}\label{initial_datum_prop_onda_neg_onda_pos}
(\rho,v)(\tilde{t},x)=\begin{cases}
(\rho^l,v^l) & \text{if } x\leq \bar{x},\\
(\rho^r,v^r) & \text{if } x> \bar{x}.
\end{cases}
\end{equation}
Then:
\begin{enumerate}
\item[(i)] the solution after the interaction is given by a wave of the first family joining $(\rho^l,v^l)$ to $(\rho^m,v^m)$ and a wave of the second family connecting $(\rho^m,v^m)$ to $(\rho^r,v^r)$ (the middle state always appears);
\item[(ii)] for the total variation we have: 
$$\Delta TV_{\tilde{t}}(\rho)\leq C\, |v^r-v^l| \; \; \; \Delta TV_{\tilde{t}}(v) \leq 0 \; \text{ and } \; \Delta TV_{\tilde{t}}(w)=0,$$
where $C$ is a positive constant depending only on the invariant domain $\mathcal{D}_{v_1,v_2,w_1,w_2}$;
\item[(iii)] the number of waves before and after the interaction remains unchanged:
$$\Delta_{\tilde{t}} \mathcal{N}= 0.$$
\end{enumerate}
\end{prop}
\begin{proof}
By Lemma \ref{lemma_curve_decrescenti}, the hypothesis $\lambda_1(\rho_\text{min},v_\text{min})<0$ implies that waves of the first family, i.e. shock or rarefaction waves, have negative speed. Hence waves with positive speed are of the second family, i.e. contact discontinuities. Therefore
$$v^l=v^k \; \text{ and } \; v^k+p(\rho^k)=v^r+p(\rho^r).$$
We have
$$v^l+p(\rho^l) \neq v^r+p(\rho^r),$$
otherwise it would be $v^l+p(\rho^l)=v^k+p(\rho^k)$, which implies $(\rho^l,v^l)=(\rho^k,v^k)$. Moreover, we have
$$v^l \neq v^r,$$
otherwise we find $v^k=v^l=v^r$, which implies that a wave of the second should join $(\rho^k,v^k)$ to $(\rho^r,v^r)$. Hence the middle state $(\rho^m,v^m)$ appears in the solution to the Riemann problem with initial datum (\ref{initial_datum_prop_onda_neg_onda_pos}) and by the definition of $(\rho^m,v^m)$, it is connected to $(\rho^l,v^l)$ and $(\rho^r,v^r)$ respectively by a wave of the first and of the second family; see Figure \ref{fig_onda_neg_onda_pos}.\\
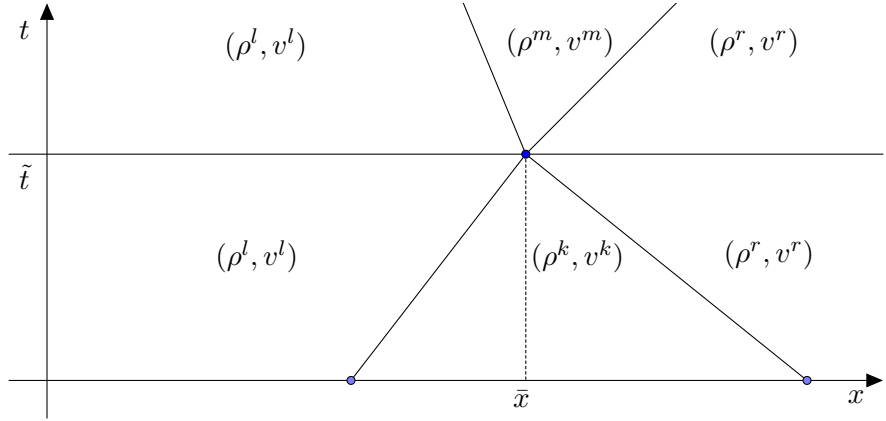
\begin{figure}[hbtp]
\centering
\definecolor{qqqqff}{rgb}{0.,0.,1.}
\definecolor{xdxdff}{rgb}{0.49019607843137253,0.49019607843137253,1.}
\begin{tikzpicture}[line cap=round,line join=round,>=triangle 45,x=1.0cm,y=1.0cm]
\draw[->,color=black] (-0.5,0.) -- (11.,0.);
\draw[->,color=black] (0.,-0.5) -- (0.,5.);
\clip(-0.5,-0.5) rectangle (11.,5.);
\draw (4.,0.)-- (6.3,3.);
\draw (10.,0.)-- (6.3,3.);
\draw [domain=-0.5:11.] plot(\x,{(--3.-0.*\x)/1.});
\draw (6.3,3.)-- (3.92,8.8);
\draw (6.3,3.)-- (11.36,8.12);
\draw (2.08,2.) node[anchor=north west] {$(\rho^l,v^l)$};
\draw (6.23,2.) node[anchor=north west] {$(\rho^k,v^k)$};
\draw (8.77,2) node[anchor=north west] {$(\rho^r,v^r)$};
\draw (2.20,4.8) node[anchor=north west] {$(\rho^l,v^l)$};
\draw (8.57,4.8) node[anchor=north west] {$(\rho^r,v^r)$};
\draw (5.9,4.8) node[anchor=north west] {$(\rho^m,v^m)$};
\draw (-0.5,4.9) node[anchor=north west] {$t$};
\draw (-0.5,3.) node[anchor=north west] {$\tilde{t}$};
\draw (10.4,0.) node[anchor=north west] {$x$};
\draw [dash pattern=on 1pt off 1pt] (6.3,3.)-- (6.3,0.);
\draw (6.,0.) node[anchor=north west] {$\bar{x}$};
\begin{scriptsize}
\draw [fill=xdxdff] (4.,0.) circle (1.5pt);
\draw [fill=qqqqff] (6.3,3.) circle (1.5pt);
\draw [fill=xdxdff] (10.,0.) circle (1.5pt);
\draw [fill=qqqqff] (3.92,8.8) circle (1.5pt);
\draw[color=qqqqff] (-0.79,6.52) node {$D$};
\draw [fill=qqqqff] (11.36,8.12) circle (1.5pt);
\draw[color=qqqqff] (-0.79,6.52) node {$E$};
\end{scriptsize}
\end{tikzpicture}
\caption{One wave with positive speed interacts with another one with negative speed. After the interaction two new waves appear: one has negative speed and one has positive speed.}\label{fig_onda_neg_onda_pos}
\end{figure}

If the wave joining $(\rho^l,v^l)$ to $(\rho^m,v^m)$ is a shock, then the number of waves remains unchanged. If the wave is a rarefaction, we approximate it with a single non-classical shock. Hence
$$\Delta_{\tilde{t}}\mathcal{N}=0.$$
For the speed we find:
\begin{equation*}
\begin{split}
\Delta TV_{\tilde{t}}(v) & =|v^r-v^m|+|v^m-v^l|-|v^r-v^k|-|v^k-v^l|=\\
& = 0+|v^r-v^l|-|v^r-v^l|-0=\\
& = 0,
\end{split}
\end{equation*}
because $v^r=v^m$ and $v^k=v^l$.\\
For the Riemann invariant $w$, since $w^r=w^k$ and $w^m=w^l$, we find
\begin{equation*}
\begin{split}
\Delta TV_{\tilde{t}}(w)& = |w^r-w^m|+|w^m-w^l|-|w^r-w^k|-|w^k-w^l|=\\
& = |w^r-w^m|-|w^k-w^l|=\\
& = |w^r-w^l|-|w^r-w^l|=0.
\end{split}
\end{equation*}
For the density we have
$$\Delta TV_{\tilde{t}}(\rho)=|\rho^r-\rho^m|+|\rho^m-\rho^l|-|\rho^r-\rho^k|-|\rho^k-\rho^l|.$$
We have to distinguish different cases.\\
First, suppose that
$$v^l+p(\rho^l)< v^k+p(\rho^k).$$
If $v^r < v^k$, then (see Figure \ref{dim_onda_pos_onda_neg_caso_1})
$$\Delta TV_{\tilde{t}}(\rho)=\rho^r-\rho^m+\rho^m-\rho^l-\rho^r+\rho^k-\rho^k+\rho^l=0,$$
because:
\begin{itemize}
\item $v^m=v^r$ and $v^r+p(\rho^r)=v^k+p(\rho^k)> v^l+p(\rho^l)=v^m+p(\rho^m)$ imply $\rho^r > \rho^m$;
\item $v^m+p(\rho^m)=v^l+p(\rho^l)$ and $v^m=v^r < v^k=v^l$ imply $\rho^m > \rho^l$;
\item $v^k+p(\rho^k)=v^r+p(\rho^r)$ and $v^r < v^k$ imply $\rho^r > \rho^k$;
\item $v^k=v^l$ and $v^k+p(\rho^k)> v^l+p(\rho^l)$ imply $\rho^k > \rho^l$.
\end{itemize}
If $v^r >v^k$, then (see Figure \ref{dim_onda_pos_onda_neg_caso_2})
$$\Delta TV_{\tilde{t}}(\rho)=\rho^r-\rho^m+\rho^l-\rho^m-\rho^k+\rho^r-\rho^k+\rho^l=2\,(\rho^r-\rho^m+\rho^l-\rho^k).$$
We claim (\textit{Claim (i)}) that
$$\rho^r-\rho^m\leq C|v^r-v^l|+|\rho^k-\rho^l|=C(v^r-v^l)+(\rho^k-\rho^l),$$
where $C$ is a constant depending only on the invariant domain $\mathcal{D}_{v_1,v_2,w_1,w_2}$. We postpone the proof of this claim. Hence
$$\Delta TV_{\tilde{t}}(\rho)\leq 2\,C\,(v^r-v^l).$$

\begin{figure}[hbtp]
\centering
\begin{subfigure}[hbtp]{0.48\linewidth}
\centering
\definecolor{uququq}{rgb}{0.25098039215686274,0.25098039215686274,0.25098039215686274}
\begin{tikzpicture}[scale=0.35,line cap=round,line join=round,>=triangle 45,x=1.0cm,y=1.0cm]
\draw[->,color=black] (-1.,0.) -- (17.,0.);
\draw[->,color=black] (0.,-0.5) -- (0.,10.);
\clip(-1.5,-1) rectangle (17.,10.);
\draw[smooth,samples=50,domain=0.001:17.0] plot(\x,{4.0*(\x)-(\x)^(1.5)});
\draw[smooth,samples=50,domain=0.001:17.0] plot(\x,{3.5*(\x)-(\x)^(1.5)});
\draw [domain=0.0:17.0] plot(\x,{(-0.--7.3*\x)/7.74});
\draw [domain=0.0:17.0] plot(\x,{(-0.--4.82*\x)/10.34});
\begin{scriptsize}
\draw (-1.4,9.35) node[anchor=north west] {$\rho v$};
\draw (15,0.1) node[anchor=north west] {$\rho$};
\draw (9.4,9.6) node[anchor=north west] {$(\rho^k,v^k)$};
\draw (12.5,6.4) node[anchor=north west] {$(\rho^r,v^r)$};
\draw (9.25,4.92) node[anchor=north west] {$(\rho^m,v^m)$};
\draw (6.6,7.) node[anchor=north west] {$(\rho^l,v^l)$};
\end{scriptsize}
\begin{scriptsize}
\draw [fill=uququq] (6.53,6.16) circle (5pt);
\draw [fill=uququq] (9.34,8.81) circle (5pt);
\draw [fill=uququq] (12.48,5.82) circle (5pt);
\draw [fill=uququq] (9.20,4.29) circle (5pt);
\end{scriptsize}
\end{tikzpicture}
\caption{Case $v^k> v^r$.}\label{dim_onda_pos_onda_neg_caso_1}
\end{subfigure}
\quad
\begin{subfigure}[hbtp]{0.48\linewidth}
\centering
\definecolor{uququq}{rgb}{0.25098039215686274,0.25098039215686274,0.25098039215686274}
\begin{tikzpicture}[scale=0.35,line cap=round,line join=round,>=triangle 45,x=1.0cm,y=1.0cm]
\draw[->,color=black] (-1.,0.) -- (17.,0.);
\draw[->,color=black] (0.,-0.5) -- (0.,10.);
\clip(-1.5,-1) rectangle (17.,10.);
\draw[smooth,samples=50,domain=0.001:17.0] plot(\x,{4.0*(\x)-(\x)^(1.5)});
\draw[smooth,samples=50,domain=0.001:17.0] plot(\x,{3.5*(\x)-(\x)^(1.5)});
\draw [domain=0.0:17.0] plot(\x,{(-0.--7.3*\x)/7.74});
\draw [domain=0.0:17.0] plot(\x,{(-0.--4.82*\x)/10.34});
\begin{scriptsize}
\draw (-1.4,9.3) node[anchor=north west] {$\rho v$};
\draw (15,0.1) node[anchor=north west] {$\rho$};
\draw (12.45,6.45) node[anchor=north west] {$(\rho^k,v^k)$};
\draw (9.45,9.6) node[anchor=north west] {$(\rho^r,v^r)$};
\draw (6.5,7) node[anchor=north west] {$(\rho^m,v^m)$};
\draw (9.35,4.86) node[anchor=north west] {$(\rho^l,v^l)$};
\end{scriptsize}
\begin{scriptsize}
\draw [fill=uququq] (6.53,6.16) circle (5pt);
\draw [fill=uququq] (9.34,8.81) circle (5pt);
\draw [fill=uququq] (12.48,5.82) circle (5pt);
\draw [fill=uququq] (9.20,4.29) circle (5pt);
\end{scriptsize}
\end{tikzpicture}
\caption{Case $v^k<v^r$.}\label{dim_onda_pos_onda_neg_caso_2}
\end{subfigure}
\caption{Interaction between a wave with negative speed and a wave with positive speed: flux diagram. Case $v^l+p(\rho^l)< v^k+p(\rho^k)$.}
\end{figure}
Now, suppose that
$$v^l+p(\rho^l)> v^k+p(\rho^k).$$
If $v^r> v^l$, then (see Figure \ref{dim_onda_pos_onda_neg_caso_3})
$$\Delta TV_{\tilde{t}}(\rho)=\rho^m-\rho^r+\rho^l-\rho^m-\rho^k+\rho^r-\rho^l+\rho^k=0,$$
because:
\begin{itemize}
\item $v^m=v^r$ and $v^m+p(\rho^m)=v^l+p(\rho^l)> v^k+p(\rho^k)=v^r+p(\rho^r)$ imply $\rho^m>\rho^r$;
\item $v^l<v^r=v^m$ and $v^m+p(\rho^m)=v^l+p(\rho^l)$ imply $\rho^l>\rho^m$;
\item $v^k=v^l<v^r$ and $v^r+p(\rho^r)=v^k+p(\rho^k)$ imply $\rho^k>\rho^r$;
\item $v^l=v^k$ and $v^l+p(\rho^l)>v^k+p(\rho^k)$ imply $\rho^l>\rho^k$.
\end{itemize}
If $v^r <v^l$, then (see Figure \ref{dim_onda_pos_onda_neg_caso_4})
$$\Delta TV_{\tilde{t}}(\rho)=\rho^m-\rho^r+\rho^m-\rho^l-\rho^r+\rho^k-\rho^l+\rho^k=2\, (\rho^m-\rho^r+\rho^k-\rho^l).$$
We claim (\textit{Claim (ii)}) that
$$\rho^l-\rho^k \geq \rho^m-\rho^r.$$
We postpone the proof of this claim.\\
Hence
$$\Delta TV_{\tilde{t}}(\rho)\leq 0.$$
\begin{figure}[hbtp]
\centering
\begin{subfigure}[hbtp]{0.48\linewidth}
\definecolor{uququq}{rgb}{0.25098039215686274,0.25098039215686274,0.25098039215686274}
\begin{tikzpicture}[scale=0.35,line cap=round,line join=round,>=triangle 45,x=1.0cm,y=1.0cm]
\draw[->,color=black] (-1.,0.) -- (17.,0.);
\draw[->,color=black] (0.,-0.5) -- (0.,10.);
\clip(-1.5,-1) rectangle (17.,10.);
\draw[smooth,samples=50,domain=0.001:17.0] plot(\x,{4.0*(\x)-(\x)^(1.5)});
\draw[smooth,samples=50,domain=0.001:17.0] plot(\x,{3.5*(\x)-(\x)^(1.5)});
\draw [domain=0.0:17.0] plot(\x,{(-0.--7.3*\x)/7.74});
\draw [domain=0.0:17.0] plot(\x,{(-0.--4.82*\x)/10.34});
\begin{scriptsize}
\draw (-1.4,9.1) node[anchor=north west] {$\rho v$};
\draw (15,0.10) node[anchor=north west] {$\rho$};
\draw (9.40,4.76) node[anchor=north west] {$(\rho^k,v^k)$};
\draw (6.7,7.) node[anchor=north west] {$(\rho^r,v^r)$};
\draw (9.49,9.45) node[anchor=north west] {$(\rho^m,v^m)$};
\draw (12.61,6.28) node[anchor=north west] {$(\rho^l,v^l)$};
\end{scriptsize}
\begin{scriptsize}
\draw [fill=uququq] (6.53,6.16) circle (5pt);
\draw [fill=uququq] (9.34,8.81) circle (5pt);
\draw [fill=uququq] (12.48,5.82) circle (5pt);
\draw [fill=uququq] (9.20,4.29) circle (5pt);
\end{scriptsize}
\end{tikzpicture}
\caption{Case $v^k< v^r$.}\label{dim_onda_pos_onda_neg_caso_3}
\end{subfigure}
\quad
\begin{subfigure}[hbtp]{0.48\linewidth}
\definecolor{uququq}{rgb}{0.25098039215686274,0.25098039215686274,0.25098039215686274}
\begin{tikzpicture}[scale=0.35,line cap=round,line join=round,>=triangle 45,x=1.0cm,y=1.0cm]
\draw[->,color=black] (-1.,0.) -- (17.,0.);
\draw[->,color=black] (0.,-0.5) -- (0.,10.);
\clip(-1.5,-1) rectangle (17.,10.);
\draw[smooth,samples=50,domain=0.001:17.0] plot(\x,{4.0*(\x)-(\x)^(1.5)});
\draw[smooth,samples=50,domain=0.001:17.0] plot(\x,{3.5*(\x)-(\x)^(1.5)});
\draw [domain=0.0:17.0] plot(\x,{(-0.--7.3*\x)/7.74});
\draw [domain=0.0:17.0] plot(\x,{(-0.--4.82*\x)/10.34});
\begin{scriptsize}
\draw (-1.4,9.1) node[anchor=north west] {$\rho v$};
\draw (15,0.1) node[anchor=north west] {$\rho$};
\draw (6.7,6.95) node[anchor=north west] {$(\rho^k,v^k)$};
\draw (9.38,4.70) node[anchor=north west] {$(\rho^r,v^r)$};
\draw (12.60,6.27) node[anchor=north west] {$(\rho^m,v^m)$};
\draw (9.4,9.55) node[anchor=north west] {$(\rho^l,v^l)$};
\end{scriptsize}
\begin{scriptsize}
\draw [fill=uququq] (6.53,6.16) circle (5pt);
\draw [fill=uququq] (9.34,8.81) circle (5pt);
\draw [fill=uququq] (12.48,5.82) circle (5pt);
\draw [fill=uququq] (9.20,4.29) circle (5pt);
\end{scriptsize}
\end{tikzpicture}
\caption{Case $v^k>v^r$.}\label{dim_onda_pos_onda_neg_caso_4}
\end{subfigure}
\caption{Interaction between a wave with negative speed and a wave with positive speed: flux diagram. Case $v^l+p(\rho^l)> v^k+p(\rho^k)$.}
\end{figure}
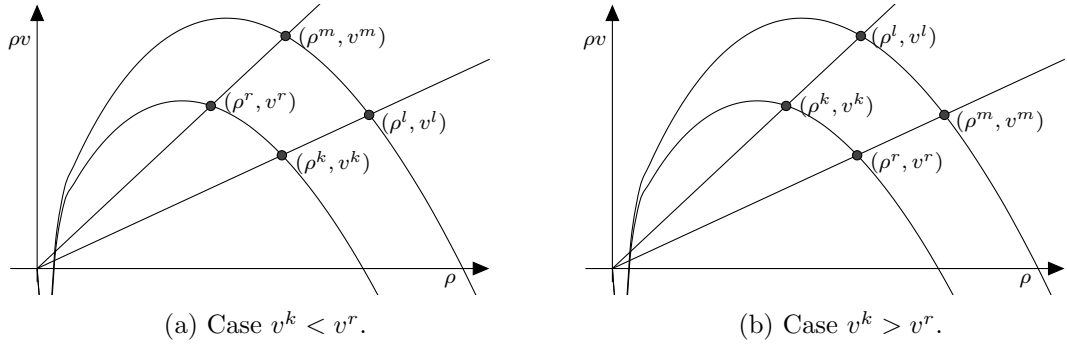

\textit{{Proof of the claims.}}\\

\textit{Claim (i)}: If $w^r=v^r+p(\rho^r)> v^l+p(\rho^l)=w^l$ and $v^r>v^k=v^l$, then
$$\rho^r-\rho^m\leq C|v^r-v^l|+|\rho^k-\rho^l|.,$$
where $C$ is a constant depending only on the invariant domain $\mathcal{D}_{v_1,v_2,w_1,w_2}$.\\
\textit{Proof.} Observe that
$$\rho^r-\rho^m=p^{-1}(w^r-v^r)-p^{-1}(w^l-v^m),$$
because $w^m=w^l$. Similarly
$$\rho^k-\rho^l=p^{-1}(w^r-v^k)-p^{-1}(w^l-v^l).$$
Since $v^l=v^k$ and $v^m=v^r$, we find
$$\rho^r-\rho^m=p^{-1}(w^r-v^r)-p^{-1}(w^l-v^r)\; \text{ and } \; \rho^k-\rho^l=p^{-1}(w^r-v^l)-p^{-1}(w^l-v^l).$$
Consider the function
$$v \to \varphi(v)=p^{-1}(w^r-v)-p^{-1}(w^l-v).$$
We can write
$$\rho^r-\rho^m=\varphi(v^r) \; \text{ and } \; \rho^k-\rho^l=\varphi(v^l).$$
We find
$$\varphi'(v)=\dfrac{1}{p'(p^{-1}(w^l-v))}-\dfrac{1}{p'(p^{-1}(w^r-v))}.$$
Suppose that $\varphi$ is Lipschitz continuous and let $C$ be the Lipschitz constant, i.e.
$$|\varphi'(v)|\leq C \; \text{ for every } \; v\in [v_1,v_2].$$
Hence we find
$$|\varphi(v^r)-\varphi(v^l)|\leq C|v^r-v^l|,$$
which, by the triangular inequality, implies
\begin{equation*}
\begin{split}
|\rho^r-\rho^m| & = |\rho^r-\rho^m+\rho^k-\rho^k+\rho^l-\rho^l|\leq \\
& \leq |\rho^r-\rho^m-\rho^k+\rho^l|+|\rho^k-\rho^l|\leq |\varphi(v^r)-\varphi(v^l)|+|\rho^k-\rho^l|\leq\\
& \leq C\,|v^r-v^l|+|\rho^k-\rho^l|.
\end{split}
\end{equation*}
We have only to show that $\varphi$ is Lipschitz continuous.\\
Since $w^r-v>w^l-v$, $w^l >w_1$ and $p^{-1}$ and $p'$ are increasing functions, we find
\begin{equation*}
\begin{split}
|\varphi'(v)| & \leq \Big|\dfrac{1}{p'(p^{-1}(w^l-v))}\Big|+\Big|\dfrac{1}{p'(p^{-1}(w^r-v))}\Big| \leq\\
& \leq 2\, \Big| \dfrac{1}{p'(p^{-1}(w^l-v))}\Big| \leq 2\, \sup_{v\in[v_1,v_2]}\Big|\dfrac{1}{p'(p^{-1}(w^l-v))}\Big|\leq \\
& \leq \dfrac{2}{\inf_{v\in[v_1,v_2]}p'(p^{-1}(w^l-v))}\leq \\
& \leq \dfrac{2}{p'(p^{-1}(w_1-v_2))}=:C.
\end{split}
\end{equation*}
Finally, let us recall the Definition \ref{def_rho_v_min} of the point $(\rho_{\min},v_\text{min})$, which is the solution to the system
$$
\begin{cases}
v+p(\rho)=w_1,\\
v=v_2.
\end{cases}
$$
We have $p(\rho_\text{min})=w_1-v_2$, which implies $w_1-v_2>0$. Hence the constant $C$ depends only on the invariant domain $\mathcal{D}_{v_1,v_2,w_1,w_2}$, is positive and is finite.\\

\textit{Claim (ii)}: If $w^r=v^r+p(\rho^r)<v^l+p(\rho^l)=w^l$ and $v^l>v^r$, then
$$\rho^l-\rho^k \geq \rho^m-\rho^r.$$
\textit{Proof.} Observe that
$$\rho^l-\rho^k=p^{-1}(w^l-v^l)-p^{-1}(w^r-v^l),$$
because $v^l=v^k$ and $w^r=w^k$. Similarly
$$\rho^m-\rho^r=p^{-1}(w^l-v^r)-p^{-1}(w^r-v^r).$$
Hence, if we show that the function
$$v \to \psi(v)=p^{-1}(w^l-v)-p^{-1}(w^r-v)$$
is increasing, we find the thesis, because $v^l>v^r$.\\
We have
$$\psi'(v)=-\dfrac{1}{p'(p^{-1}(w^l-v))}+\dfrac{1}{p'(p^{-1}(w^r-v))}.$$
Since $w^l>w^r$, we obtain:
$$w^l-v >w^r-v \Longrightarrow p^{-1}(w^l-v)\geq p^{-1}(w^r-v) \Longrightarrow \dfrac{1}{p'(p^{-1}(w^r-v))}\geq \dfrac{1}{p'(p^{-1}(w^l-v))}.$$
Therefore $\psi$ is increasing.
\end{proof}
\section{Wave-Front Tracking}
Let us denote $BV(\mathbb{R}, \mathbb{R}^d)$ the set of the functions with bounded total variation defined in $\mathbb{R}$ and with values in $\mathbb{R}^d$. In this section we are going to find a solution to the Cauchy problem
\begin{equation}\label{existence_Cauchy_problem_general}
\begin{cases}
\partial_t \rho+\partial_x(\rho v)=0,\\
\partial_t(\rho(v+p(\rho))+\partial_x[\rho v(v+p(\rho))]=0,\\
\rho(t,0)\,v(t,0)\leq q,\\
(\rho,v)(0,x)=(\rho_0,v_0)(x),
\end{cases}
\end{equation}
where $q\in \mathbb{R}^+$ is a fixed constant and $(\rho_0,v_0)$ is a function in $L^1(\mathbb{R}) \cap BV(\mathbb{R})$.\\
\subsection{The number of waves and interactions}
In Table \ref{table_number_of_waves} are summarized the results obtained in the previous section on the difference $\Delta_{\tilde{t}}\mathcal{N}$ of the number of waves before and after an interaction between two waves, happened at time $\tilde{t}$.
\begin{table}[H]
\centering
\begin{tabular}{|l|c|c|c|}
\hline Position & Interacting waves & $\Delta_{\tilde{t}}\mathcal{N}$ & Propositions\\
\hline $x=0$ & Contact discontinuity and $x=0$ & $\leq 2$ & \ref{prop_contact_discontinuity_from_left_case_1}, \ref{prop_contact_discontinuity_from_left_case_2} \\
\hline $x=0$ & Shock and $x=0$ & $\leq 1$ & \ref{prop_shock_from_right_case_1}, \ref{prop_shock_from_right_case_2}\\
\hline $x=0$ & Rarefaction wave and $x=0$ & $\leq N+1$ & \ref{prop_rarefaction_from_right_case_1}, \ref{prop_rarefaction_from_right_case_2}  \\
\hline $x\neq 0$ & Two waves with positive speed & Never occurs & \ref{prop_interacting_waves_pos_speed}\\
\hline $x\neq 0$ & Two waves with negative speed & $\leq 0$ & \ref{prop_interacting_waves_with_neg_speed}\\
\hline $x \neq 0$ & One wave has speed $>0$ and one $<0$ & $=0$ & \ref{prop_interacting_wave_neg_speed_pos_speed}\\
\hline
\end{tabular}
\caption{List of the difference in the number of waves before and after an interaction.}\label{table_number_of_waves}
\end{table}
The next proposition gives two estimates of the total number of waves and the total number of interactions that appear solving the sequence of Riemann problems for a piecewise constant initial datum, with a finite number of discontinuities.
\begin{prop}\label{prop_number_of_waves_and_interactions}
Assume $\lambda_1(\rho_\text{min},v_\text{min})<0$. Consider the Cauchy problem (\ref{existence_Cauchy_problem_general}) and assume that the initial datum $(\rho_0,v_0)$ is piecewise constant with a finite number of discontinuities. Let $k_1\in \mathbb{N}$ and $k_2\in \mathbb{N}$ be respectively the number of discontinuities for $x<0$ and for $x>0$. Let $N$ be the number of discontinuities of a rarefaction fan.
\begin{enumerate}
\item The total number $\mathcal{N}$ of waves that appear solving each Riemann problem centred in a discontinuity of the initial datum and at a point of interaction between two waves or between a wave and the line $x=0$, with the Riemann solvers $\mathcal{RS}^q_2$ in $x=0$ and $\mathcal{RS}$ in $x\neq 0$, satisfies
$$\mathcal{N}\leq N+2+(N+1)(k_1+k_2)+2\,k_1+N\,(N+1)\,k_2.$$
\item The total number $\mathcal{I}$ of interactions between the waves that appear solving each Riemann problem centred in a discontinuity of the initial datum and at a point of interaction between two waves or between one wave and the line $x=0$ with the Riemann solvers $\mathcal{RS}^q_2$ in $x=0$ and $\mathcal{RS}$ in $x\neq 0$, satisfies
\begin{equation*}
\begin{split}
\mathcal{I} \leq & k_1+N\,k_2+N^2[(k_1+1)^2+2\,k_2^2]+N[k_1^2+k_2^2+(k_1+1)\,k_2]+\\
& +(k_1+N^2\,k_2)(2k_1+N^2\,k_2+N\,(k_1+1)].
\end{split}
\end{equation*}
\end{enumerate}
\end{prop}
\begin{proof}
First, let us consider the total number of waves.\\
The number of waves that appear at each Riemann problem at $t=0$ is at most $N+2$. Indeed, let $x_\nu$ be a discontinuity far from $x=0$ and let $(\rho_0^{\nu-},v_0^{\nu-})$ and $(\rho_0^{\nu+},v_0^{\nu+})$ be respectively the left and right value of the initial datum at $x=x_\nu$. Let $(\rho_0^{m,\nu},v_0^{m,\nu})$ be the middle state of the classical solution to the Riemann problem with initial datum
$$
(\rho,v)(0,x)=\begin{cases}
(\rho_0^{\nu-},v_0^{\nu-}) & \text{if } x\leq x_{\nu},\\
(\rho_0^{\nu+},v_0^{\nu+}) & \text{if } x> x_{\nu}.
\end{cases}
$$
The highest number of wave-fronts centred in $x=x_\nu$ appears when a rarefaction fan joins $(\rho_0^{\nu-},v_0^{\nu-})$ to $(\rho_0^{m,\nu},v_0^{m,\nu})$, followed by a contact discontinuity connecting $(\rho_0^{m,\nu},v_0^{m,\nu})$ to $(\rho_0^{\nu+},v_0^{\nu+})$. In this case the total number of waves is
$$\mathcal{N}_{x_\nu}=N+1.$$
At $x=0$, if the solution given by $\mathcal{RS}^q_2$ coincides with the classical one and if the initial datum has a discontinuity, we can repeat the same calculation. Otherwise we have the maximum number of waves when a rarefaction fan joins $(\rho_0^{0-},v_0^{0-})$ to $(\hat{\rho},\hat{v})$ and a contact discontinuity joins $(\check{\rho}_2,\check{v}_2)$ to $(\rho_0^{0+},v_0^{0+})$. In this case the number of waves is
$$\mathcal{N}_0= N+2,$$
because the non-classical shock between $(\hat{\rho},\hat{v})$ and $(\check{\rho}_2,\check{v}_2)$ appears.\\
The number of waves increases only when a contact discontinuity or a rarefaction wave reaches the line $x=0$ and the increment is respectively of $2$ and $N+1$ waves; see Table \ref{table_number_of_waves}.
Since $\lambda_1(\rho_\text{min},v_\text{min})<0$, the waves of the first family have negative speed. Therefore, only the contact discontinuities centred in the points $x_{\nu}<0$ can reach the constraint. Their contribution to the increment of the total number of waves is $2\, k_1$. On the contrary, only waves of the first family can reach the constraint from the right. Their maximal number is $k_2\,N$, because at each discontinuity $x_\nu>0$ can arise $N$ (non-classical) shocks with negative speed due to the presence of a rarefaction fan. When these shocks meet the constraint, at most $N+1$ new waves arise.\\
Therefore the total number of waves is at most
$$N+2+(N+1)(k_1+k_2)+2\,k_1+N\,(N+1)\,k_2.$$
Now, let us turn our attention to the number of interactions.\\
Waves of the second family do not interact together; see Table \ref{table_number_of_waves}. Therefore we have to consider only the interactions between waves of the first family, between waves of the first and the second family and between waves and $x=0$. In the first case the number of waves decreases and in the second the number remains unchanged; see Table \ref{table_number_of_waves}. For $t>0$, new waves can arise only as the result of an interaction between a wave of the first or of the second family and the line $x=0$.\\
At most $k_1$ contact discontinuities and $N\, k_2$ waves of the first family can reach the constraint respectively from left and from right.\\
In the region $x<0$ the possible interactions are the following.
\begin{itemize}
\item \textit{Interactions between waves of the first family}: the number of waves of the first family in the considered region at time $t=0^+$ is less than
$$N\,k_1+N=N(k_1+1),$$
because from each of the $k_1$ discontinuities of the initial datum and at $x=0$ there can be a rarefaction fan. If all these waves interact together, the number of interactions is
$$[N(k_1+1)]^2.$$
When a wave of the first family coming from a point of discontinuity $x_\nu>0$ reaches the constraint, a number $N$ of new waves of the first family can arise. Similarly when a contact discontinuity reaches the constraint from left a new wave of the first family can appear. If all these waves interact with each other and with the waves of the first family centred in the discontinuities $x_\nu<0$, the number of interactions is
$$(k_1+N^2\, k_2)(k_1+N^2\, k_2+N(k_1+1)).$$
\item \textit{Interactions between waves of the first and the second family}: at each discontinuity $x_\nu<0$ of the initial datum arise at most $k_1$ waves of the second family and $N\,k_1$ waves of the first family. If all these waves interact, the number of interactions is
$$k_1\, Nk_1.$$
From $x=0$ come at most $k_1+N^2k_2$ waves of the first family which can interact with the $k_1$ waves of the second family arisen in $x_\nu<0$. Therefore the number of interactions between these waves is less than
$$k_1(k_1+N^2\,k_2).$$
\end{itemize}
In the region $x>0$ the possible interactions are the following.
\begin{itemize}
\item \textit{Interactions between waves of the first family}: in the region $x>0$ there are at most $N\,k_2$ waves of the first family. Therefore the maximum number of interactions is
$$(N\,k_2)^2.$$
\item \textit{Interactions between waves of the first and the second family}: there are $k_2$ contact discontinuities coming from the discontinuities $x_\nu>0$ of the initial datum which can interact with the $N k_2$ waves of the first family of the rarefaction fans. The total number of interactions between these waves is less than
$$N\, k_2^2.$$
The number of waves of the second family that can arise at $x=0$ is $1+k_1+N\,k_2$. If all these waves interact with waves of the first family centred in the points $x_\nu>0$, then the number of possible interactions is
$$(k_1+1+N\,k_2)\,N\,k_2.$$ 
\end{itemize}
Summing all the contributions and reordering the terms, we find
\begin{equation*}
\begin{split}
\mathcal{I} \leq & k_1+N\,k_2+N^2[(k_1+1)^2+2\,k_2^2]+N[k_1^2+k_2^2+(k_1+1)\,k_2]+\\
& +(k_1+N^2\,k_2)(2k_1+N^2\,k_2+N\,(k_1+1)].
\end{split}
\end{equation*}
\end{proof}
\subsection{The total variation}
In Tables \ref{table_total_variation} and \ref{table_total_variation_v_w} are summarized the results obtained in the previous section on the differences of the total variation respectively in the $(\rho,v)$ and in the $(v,w)$ plane, before and after an interaction  happened at time $\tilde{t}$. The points $(\rho^l,v^l)$ and $(\rho^r,v^r)$ are respectively the left and the right state of the Riemann problem which originates the contact discontinuity or the rarefaction wave at $t=0$, for the case $x=0$ and for the Riemann problem in $\tilde{t}$ in the case $x\neq 0$. The positive constant $C$ depends only on the invariant domain $\mathcal{D}_{v_1,v_2,w_1,w_2}$.
\begin{table}[H]
\centering
\begin{tabular}{|l|c|c|c|}
\hline Position & Interacting waves & $\Delta TV_{\tilde{t}}(\rho)\leq $ & $\Delta TV_{\tilde{t}}(v)\leq $\\
\hline $x=0$ & Contact discontinuity and $x=0$ & $C|\rho^r-\rho^l|$ & $C|\rho^r-\rho^l|$ \\
\hline $x=0$ & Shock and $x=0$ & $0$ & $0$\\
\hline $x=0$ & Rarefaction wave and $x=0$ & $C|\rho^r-\rho^l|$ & $0$  \\
\hline $x\neq 0$ & Two waves with positive speed & Never occurs & Never occurs\\
\hline $x\neq 0$ & Two waves with negative speed & $0$ & $0$\\
\hline $x \neq 0$ & One wave has speed $>0$ and one $<0$ & $C|v^r-v^l|$ & $0$\\
\hline
\end{tabular}
\caption{List of the difference of the total variation before and after an interaction in the $(\rho,v)$ plane. The points $(\rho^l,v^l)$ and $(\rho^r,v^r)$ are respectively the left and the right state of the Riemann problem which originates the contact discontinuity or the rarefaction wave at $t=0$, for the case $x=0$ and for the Riemann problem in $\tilde{t}$ in the case $x\neq 0$. See Table \ref{table_number_of_waves} for a list of the propositions in which each case is discussed.}\label{table_total_variation}
\end{table}
\begin{table}[H]
\centering
\begin{tabular}{|l|c|c|c|}
\hline Position & Interacting waves & $\Delta TV_{\tilde{t}}(v)\leq $ & $\Delta TV_{\tilde{t}}(w)\leq $\\
\hline $x=0$ & Contact discontinuity and $x=0$ & $C|w^r-w^l|$ & $0$ \\
\hline $x=0$ & Shock and $x=0$ & $0$ & $0$\\
\hline $x=0$ & Rarefaction wave and $x=0$ & $0$ & $C|v^r-v^l|$  \\
\hline $x\neq 0$ & Two waves with positive speed & Never occurs & Never occurs\\
\hline $x\neq 0$ & Two waves with negative speed & $0$ & $0$\\
\hline $x \neq 0$ & One wave has speed $>0$ and one $<0$ & $0$ & $0$\\
\hline
\end{tabular}
\caption{List of the difference of the total variation before and after an interaction in the $(v,w)$ plane. The points $(\rho^l,v^l)$ and $(\rho^r,v^r)$ are respectively the left and the right state of the Riemann problem which originates the contact discontinuity or the rarefaction wave at $t=0$, for the case $x=0$ and for the Riemann problem in $\tilde{t}$ in the case $x\neq 0$. See Table \ref{table_number_of_waves} for a list of the propositions in which each case is discussed.}\label{table_total_variation_v_w}
\end{table}
The next theorem states that a function with bounded variation can be approximated with a piecewise constant function having a finite number of discontinuities. See \cite{bressan} for the proof.
\begin{theorem}\label{theorem_piecewise_constant_approxim_of_a_bv_function}
Let $u$ be a function in $BV(\mathbb{R},\mathbb{R}^d)$. For every $\varepsilon>0$, there exists a piecewise constant function $\tilde{u}$ in $BV(\mathbb{R},\mathbb{R}^d)$ having a finite number of discontinuities and such that
$$TV(\tilde{u})\leq TV(u) \; \text{ and } \; \parallel u -\tilde{u} \parallel_\infty <\varepsilon.$$
\end{theorem}
\begin{cor}\label{cor_piecewise_constant_approxim_of_a_bv_function}
Let $u$ be a function in $BV(\mathbb{R},\mathbb{R}^d)$. There exists a sequence of piecewise constant functions $\lbrace
u_\nu \rbrace_\nu \in \mathbb{N}$ satisfying:
\begin{itemize}
\item[(i)] for every $\nu$, the function $u_\nu$ has a finite number of discontinuities;
\item[(ii)] for every $\nu$, we have
$$TV(u_\nu)\leq TV(u);$$
\item[(iii)] $u_\nu \to u$ when $\nu \to +\infty$ in $L^1_\text{loc}(\mathbb{R})$.
\end{itemize}
\end{cor}
\begin{proof}
For the points (i) and (ii) we have only to apply Theorem \ref{theorem_piecewise_constant_approxim_of_a_bv_function} with
$$\varepsilon=\dfrac{1}{\nu}.$$
To prove point (iii), let us consider a compact $K\subset \mathbb{R}$. We find
$$\int_K |u_\nu(x)-u(x)|\,dx\leq \int_K \parallel u_\nu-u \parallel_{\infty} \, dx \leq \dfrac{1}{\nu}|K| \xrightarrow{\nu \to +\infty} 0.$$
\end{proof}
Consider the Cauchy problem (\ref{existence_Cauchy_problem_general}). The initial datum $(\rho_0,v_0)$ has bounded variation, hence applying Corollary \ref{cor_piecewise_constant_approxim_of_a_bv_function} we find a sequence $(\rho_{0,\nu},v_{0,\nu})$ of piecewise constant functions with bounded variation, with a finite number of discontinuities and such that
\begin{equation}\label{total_var_of_the_pc_initial_datum}
TV(\rho_{0,\nu})+TV(v_{0,\nu})\leq TV(\rho_0)+TV(v_0) \; \text{ for every } \; \nu \in \mathbb{N} \; \text{ and}
\end{equation}
\begin{equation}\label{convergence_of_the_pc_initial_datum}
\lim_{\nu\to +\infty}\parallel (\rho_{0,\nu},v_{0,\nu})-(\rho_0,v_0)\parallel_{L^1_\text{loc}(\mathbb{R})}= 0.
\end{equation}
Let $k_1^\nu$ and $k_2^\nu$ be the finite numbers of discontinuities of the function $(\rho_{0,\nu},v_{0,\nu})$ respectively before and after $x=0$. Let us define the sequence of the points of discontinuity of $(\rho_{0,\nu},v_{0,\nu})$ as
\begin{equation}\label{points_of_discontinuity_of_the_approximant_initial_datum}
\lbrace x_i^\nu \rbrace_{i=-k_1^\nu}^{k_2^\nu}, \; \; x_i^\nu < x_{i+1}^\nu \; \text{ for every } i \in \mathbb{N},
\end{equation}
where $x_0^\nu=0$ (even if $x=0$ is not a point of discontinuity of $(\rho_{0,\nu},v_{0,\nu})$).
Therefore there exists a sequence of points $\lbrace(\rho^i_{0,\nu},v^i_{0,\nu})\rbrace_{i=-k_1^\nu-1}^{k_2^\nu}$ for which we can write
\begin{equation}\label{piecewise_constant_initial_datum}
\begin{split}
(\rho_{0,\nu},v_{0,\nu}) & =(\rho^{-k_1^\nu-1}_{0,\nu},v^{-k_1^\nu-1}_{0,\nu})\mathbf{1}_{(-\infty,x^\nu_{-k_1^\nu})}+\sum_{i=-k_1^\nu}^{k_2^\nu-1}(\rho_{0,\nu}^i,v_{0,\nu}^i)\mathbf{1}_{[x^\nu_i,x^\nu_{i+1})}+\\
&+(\rho_{0,\nu}^{k_2^\nu},v_{0,\nu}^{k_2^\nu})\mathbf{1}_{[x^\nu_{k_2^\nu},+\infty)}.
\end{split}
\end{equation}
For a fixed $\nu$ the function $(\rho_{0,\nu},v_{0,\nu})$ is piecewise constant. Hence solving the sequence of Riemann problems centred in the points $\lbrace x_n^\nu \rbrace_{n=-k_1^\nu}^{k_2^\nu}$ with $\mathcal{RS}$ at $x\neq 0$ and $\mathcal{RS}^q_2$ at $x=0$, we obtain a solution to the Cauchy problem with initial datum $(\rho_{0,\nu},v_{0,\nu})$ until an interaction happens at time $t^1$. Solving the new Riemann problem at the point of interaction, with either $\mathcal{RS}$ at $x\neq 0$ or $\mathcal{RS}^q_2$ at $x=0$, we find a solution until there is a new interaction at time $t^2$. Repeating the process we obtain a solution $(\rho_\nu,v_\nu)$ to the Cauchy problem (\ref{existence_Cauchy_problem_general}) with initial datum $(\rho_{0,\nu},v_{0,\nu})$. We call this solution ``wave-front tracking approximation''.\\
Proposition \ref{prop_number_of_waves_and_interactions} ensures that the number of waves and interactions that appear with this procedure is finite for every $\nu$ and depends on the number of points of discontinuity of the initial datum and of a rarefaction fan. Hence the function $(\rho_\nu,v_\nu)$ is defined for every $t>0$.\\
Let us denote
$$\Gamma_\nu^w(t)=TV_t(w_\nu) \; \text{ and } \; \Gamma_\nu^v(t)=TV_t(v_\nu)$$
the functions which give the total variation respectively of $w_\nu(t,\cdot)=v_\nu(t,\cdot)+p(\rho_\nu(t,\cdot))$ and $v_\nu=v_\nu(\cdot,t)$ at time $t$ and
$$\Gamma^w_0=TV(w_0)=TV(v_0+p(\rho_0)) \; \text{ and } \; \Gamma^v_0=TV(v_0).$$
The next proposition gives an uniform estimate of the total variation for the $v$ and $w$ components of the functions of the sequence $\lbrace (\rho_\nu,v_\nu)\rbrace_{\nu \in \mathbb{N}}$.
\begin{prop}\label{prop_uniform_bounds_TV}
Assume $\lambda_1(\rho_\text{min},v_\text{min})<0$. Let $(\rho_\nu,v_\nu)$ be a wave-front tracking approximation for the Cauchy problem (\ref{existence_Cauchy_problem_general}) with initial datum $(\rho_{0,\nu},v_{0,\nu})$ satisfying the conditions (\ref{total_var_of_the_pc_initial_datum}) and (\ref{convergence_of_the_pc_initial_datum}). Let  $\xi^\nu=\lbrace x_i^\nu \rbrace_{i=-k_1^\nu}^{k_2^\nu}$ be the sequence of the points of discontinuity of the initial datum $(\rho_{0,\nu},v_{0,\nu})$ defined in (\ref{points_of_discontinuity_of_the_approximant_initial_datum}) and let
$$\lbrace (\rho^i_{0,\nu},v^i_{0,\nu}) \rbrace_{i=-k_1^\nu-1}^{k_2^\nu}\subset \mathcal{D}_{v_1,v_2,w_1,w_2}$$
be the sequence of values of the initial datum. There exists a positive constant $M$ such that
$$\Gamma_\nu^v(t)\leq M \; \text{ and } \; \Gamma_\nu^w(t)\leq M\; \text{ for every } \; \nu \in \mathbb{N} \; \text{ and } \; t>0.$$
\end{prop}
\begin{proof}
Let us consider an interaction between a wave of the first family joining a point $(\rho^l,v^l)$ to a point $(\rho^k,v^k)$ and a wave of the second family connecting $(\rho^k,v^k)$ to $(\rho^r,v^r)$. By Proposition \ref{prop_interacting_wave_neg_speed_pos_speed}, after the interaction the solution is given by a new wave of the first family and a new wave of the second family. Consider the value of the Riemann invariant $w$ on the left and on the right of the wave of the second family, respectively $w^l$ and $w^k=w^r$. After the interaction the middle state $(\rho^m,v^m)$ appears. The left and the right values of the Riemann invariant $w$ on the left and on the right of the new contact discontinuity remain $w^l$ and $w^r$, because $w^m=w^l$.\\
Similarly, since $v^m=v^r$, the left and the right velocity of the wave of the first family before and after the interaction are always $v^l$ and $v^r$.\\ 
If the interaction is between two waves of the first family joining respectively $(\rho^l,v^l)$ to $(\rho^k,v^k)$ and $(\rho^k,v^k)$ to $(\rho^r,v^r)$, only a new wave of the first family appears and the left and right velocity are $v^l$ and $v^r$. Therefore the values of the left and right speed change. However, since
$$\Delta TV_{\tilde{t}}(v) \leq 0\; \text{ and } \; \Delta TV_{\tilde{t}}(w)\leq 0,$$
without loss of generality, we can assume that the left and the right traces of the speed of a wave of the first family are constant.\\ 
If $p_\nu \leq k_1^\nu$ is the number of contact discontinuities that arise from the Riemann problem with initial datum
$$
(\rho,v)(0,x) =\begin{cases}
(\rho^{i-1}_{0,\nu},v^{i-1}_{0,\nu}) & \text{if } x\leq x^\nu_{i},\\
(\rho^{i}_{0,\nu},v^{i}_{0,\nu}) & \text{if } x>x^\nu_{i},
\end{cases}
$$ 
the number of waves which reach the constraint remains $p_\nu$. We can assign to the wave of the second family centred in $x^\nu_i$ the number $i$.\\
The $v$ component of the total variation changes only when a contact discontinuity reaches the constraint (see Table \ref{table_total_variation_v_w}) and the contribution is
$$C\,|w^l-w^r|,$$
where $C$ is a constant which depends only on the invariant domain $\mathcal{D}_{v_1,v_2,w_1,w_2}$ and $(\rho^l,v^l)$ and $(\rho^r,v^r)$ are the left and the right states of the Riemann problem from which the wave is born.\\
Since the left and the right values of the Riemann invariant $w$ remain constant, when the $i$-th contact discontinuity interacts with the line $x=0$, the left and the right traces of the component $w$ of the wave are respectively $w_{\nu,0}^{i-1}$ and $w_{\nu,0}^{i}$.\\
Assume that $p_\nu=k_1^\nu$ and let $t^i$, for $i\in \lbrace 1,...,k_1^\nu\rbrace$, be the instant in which the $i$-th contact discontinuity reaches the constraint. Since $x_i^\nu<x_{i+1}^\nu$ and waves of the second family do not interact together, these instants increase with $i$, i.e.
$$t^{i+1}>t^i.$$
Therefore we find
\begin{equation*}
\begin{split}
\Gamma^v_\nu(t^{k_1^\nu}+) & \leq \Gamma^v_\nu(t^{k_1^\nu}-)+C|w_{0,\nu}^{-k_1^\nu-1}-w_{0,\nu}^{-k_1^\nu}| = \\
& =\Gamma^v_\nu(t^{k_1^\nu-1}+)+C|w_{0,\nu}^{-k_1^\nu-1}-w_{0,\nu}^{-k_1^\nu}| \leq \\
& \leq \Gamma^v_\nu(t^{k_1^\nu-1}-)+C(|w_{0,\nu}^{-k_1^\nu}-w_{0,\nu}^{-k_1^\nu+1}| +|w_{0,\nu}^{-k_1^\nu-1}-w_{0,\nu}^{-k_1^\nu}|) \leq \\
& \leq \cdots \leq \Gamma^v_\nu(0)+C\sum_{j=-k_1^\nu}^{-1}|w_{0,\nu}^{j-1}-w_{0,\nu}^{j}|\leq \\
& \leq \Gamma^v_\nu(0)+C\, \Gamma^w_\nu(0).
\end{split}
\end{equation*} 
By Corollary \ref{cor_piecewise_constant_approxim_of_a_bv_function}, we have
$$\Gamma^v_\nu(0)\leq \Gamma^v_0+\Gamma^w_0\; \text{ and } \; \Gamma^w_\nu(0)\leq \Gamma^v_0+\Gamma^w_0.$$
After $t^{k_1^\nu}+$ the $v$ component of the total variation remains constant. Hence for every $t>0$, we find
\begin{equation*}
\Gamma^v_\nu(t)\leq (C+1)(\Gamma^v_0+\Gamma^w_0)=:M.
\end{equation*}
Similarly, the $w$ component of the total variation increases only when a discontinuity of a rarefaction fan born at $t=0$ reaches the constraint. These waves can interact with $x=0$ only if they are centred in the positive points of discontinuity of the initial datum $(\rho_{0,\nu},v_{0,\nu})$, because they have negative speed. Let $q_\nu$ be their number. We have
$$q_\nu\leq N\, k_2^\nu,$$
where $N$ is the number of discontinuities of a rarefaction fan. In the region $x>0$ this number remains constant or decreases. By Proposition \ref{prop_rarefaction_fan_properties}, when we look to the total variation, we can consider a rarefaction fan as a single non-classical shock.\\
Assume that the number $q_\nu$ remains constant and let us assign to each wave a number $i\in \lbrace 1,...,q_\nu\rbrace$. When at time $t^i$ the $i$-th wave reaches the line $x=0$, the contribution to the $w$ component of the total variation is at most
$$C|v^r-v^l|,$$
where $C$ is a constant depending only on the invariant domain and $(\rho^l,v^l)$ and $(\rho^r,v^r)$ are the left and the right states of the Riemann problem from which the wave is born.\\
The contribution of the $i$-th wave to the $w$ component of the total variation is at most
$$C\,|v^{i-1}_{0,\nu}-v^i_{0,\nu}|.$$
Therefore, with a computation similar to the one for the $v$ component, we find
\begin{equation*}
\begin{split}
\Gamma^w_\nu(t^{k_2^\nu}+) & \leq \Gamma^w_\nu(t^{k_2^\nu}-)+C\,|v^{k_2^\nu+1}_{0,\nu}-v^{k_2^\nu}_{0,\nu}|\leq \\
& \leq \cdots \leq \Gamma^w_\nu(0)+C\,\Gamma^v_\nu(0)\leq \\
& \leq (C+1)(\Gamma^w_0+\Gamma^v_0)=M.
\end{split}
\end{equation*}
\end{proof}
\begin{cor}
Under the same assumptions of Proposition \ref{prop_uniform_bounds_TV}, there exists a positive constant $N$ depending only on the invariant domain $\mathcal{D}_{v_1,v_2,w_1,w_2}$ for which
$$\Gamma_\nu^\rho(t)\leq N\; \text{ for every } \; \nu \in \mathbb{N} \; \text{ and } \; t>0.$$
\end{cor}
\begin{proof}
By Lemma \ref{pressure_Lipschitz}, the pressure function is bi-Lipschitz continuous and the Lipschitz constant $C$ depends only on the invariant domain. Therefore, there exists a sequence $\lbrace (\rho_j,v_j)\rbrace_{j\in I}$, for which
\begin{equation*}
\begin{split}
\Gamma^\rho_\nu(t) & =\sum_{j\in I}|\rho_{j+1}-\rho_j|\leq C\,\sum_{j\in I}|p(\rho_{j+1})-p(\rho_j)|= \\
&= C\, \sum_{j\in I}|v_{j+1}+p(\rho_{j+1})-v_j-p(\rho_j)-v_{j+1}+v_{j}|\leq\\
& \leq C\, \sum_{j\in I} (|w_{j+1}-w_j|+|v_{j+1}-v_{j}|)\leq\\
& \leq C\, (TV^w_\nu(t)+TV^v_\nu(t)).
\end{split}
\end{equation*}
We find the thesis applying Proposition \ref{prop_uniform_bounds_TV}.
\end{proof}

\subsection{Existence of solutions to the constrained Cauchy problem}
The next theorem gives a compactness property of the functions with bounded variation. See \cite{bressan} for the proof.
\begin{theorem}\label{theorem_helly_tempo}
Consider a sequence of functions $\lbrace u_\nu \rbrace_{\nu\in\mathbb{N}}$ defined in $\mathbb{R}^+\times \mathbb{R}$ and valued in $\mathbb{R}^n$. Assume that there exist two positive constants $C$ and $M$, for which
$$TV(u_\nu(t,\cdot))\leq C\; \text{ and } \; |u_\nu(t,x)| \leq M \; \text{ for every } \; \nu\in \mathbb{N}, \; \; (t,x) \in \mathbb{R}^+\times \mathbb{R}.$$
Moreover suppose that there exists another constant $L$ for which for every $t$ and $s$ in $\mathbb{R}^+$, the following inequality holds:
$$\int_{-\infty}^{+\infty}|u_\nu(t,x)-u_\nu(s,x)|\, dx \leq L\,|t-s|.$$
Then there exists a subsequence $\lbrace u_\mu \rbrace_{\mu}$ which converges to some function $u$ in $L^1_{loc}(\mathbb{R}^+\times \mathbb{R},\mathbb{R}^n)$. The limit function $u$ satisfies
$$\int_{-\infty}^{+\infty}|u(t,x)-u(s,x)|\, dx \leq L\,|t-s| \; \text{ for every } \; t,s \geq 0.$$
The point values of the limit function $u$ can be uniquely determined by requiring that
$$u(t,x)=u(t,x+):=\lim_{y \to x+}u(t,y) \; \text{ for all } \; t,x.$$
In this case
$$TV(u(t,\cdot))\leq C \; \text{ and } \; |u(t,x)|\leq M \; \text{ for every } \; t,x.$$
\end{theorem}
The next lemma states that a sequence of weak solutions to a system of conservation laws is compact. See \cite{bressan}.
\begin{lemma}\label{lemma_convergence_weak_solutions}
Let $\lbrace u_\nu \rbrace_\nu$ be a sequence of weak solutions of the system of conservation laws
\begin{equation}\label{existence_system_of_conservation_laws}
\partial_t u + \partial_x f(u)=0.
\end{equation}
Assume that there exists a function $u$ such that
$$\lim_{\nu \to +\infty}\parallel u_\nu-u\parallel_{L^1_\text{loc}(\mathbb{R})}=0 \; \text{ and } \; \lim_{\nu \to +\infty}\parallel f(u_\nu)-f(u)\parallel_{L^1_\text{loc}(\mathbb{R})}=0.$$
Then $u$ is a weak solution to the system (\ref{existence_system_of_conservation_laws}).
\end{lemma}
\begin{proof}
Let $\phi$ be a function in $\mathcal{C}^1_c(\mathbb{R}^+\times \mathbb{R})$ and let $\Omega \subset \mathbb{R}^+\times \mathbb{R}$ be its support. For every $\nu \in \mathbb{N}$ the function $u_\nu$ is a weak solution to (\ref{existence_system_of_conservation_laws}), hence $u_\nu$ and $f(u_\nu)$ are in $L^1_\text{loc}(\mathbb{R}^+\times \mathbb{R})$ and
$$\int_{\mathbb{R}^+\times\mathbb{R}}\left(\partial_t\phi \, u_\nu +\partial_x \phi \, f(u_\nu) \right) \,dx\,dt=0.$$
Since $u_\nu \to u$ in $L^1_\text{loc}$ when $\nu \to +\infty$, we can extract a subsequence, which we still denote $u_\nu$, such that for a.e. $(t,x) \in \Omega$, we have
$$\partial_t \phi (t,x) \, u_\nu (t,x) \xrightarrow{\nu \to +\infty} \partial_t \phi (t,x) \, u (t,x).$$
Since $f$ is continuous, on the same subsequence we have
$$\partial_t \phi (t,x) \, f(u_\nu (t,x)) \xrightarrow{\nu \to +\infty} \partial_t \phi (t,x) \, f(u (t,x)).$$
The limits $\partial_t \phi \,u$ and $\partial_x\phi \, f(u)$ are in $L^1_\text{loc}$, because $\phi$ has compact support, $\partial_t \phi$ is continuous (hence limited on $\Omega$) and $u$ and $f(u)$ are in $L^1_\text{loc}$.\\
Finally, let us define
$$c_1=\sup_{\nu \in \mathbb{N}}\parallel u_\nu \parallel_{L^1_\text{loc}}\; \text{ and } \; c_2=\sup_{\nu \in \mathbb{N}}\parallel f(u_\nu) \parallel_{L^1_\text{loc}}.$$
We have:
$$|\partial_t\phi \, u_\nu +\partial_x \phi \, f(u_\nu)|\leq |\partial_t\phi \, u_\nu| +|\partial_x \phi \, f(u_\nu)|$$
and for every $\nu$ we find:
$$|\partial_t\phi \, u_\nu|\leq c_1 \cdot \sup_{(t,x)\in \Omega} |\partial_t \phi|\; \text{ and } \; |\partial_x \phi \, f(u_\nu)|\leq c_2 \cdot \sup_{(t,x)\in \Omega} |\partial_x \phi|.$$
Since the right terms of these inequalities are integrable, we can apply Lebesgue's Theorem and we have the thesis.
\end{proof}
We are now ready to prove that the constrained Cauchy problem (\ref{existence_Cauchy_problem_general}) admits a solution.
\begin{theorem}
Assume $\lambda_1(\rho_{\min},v_{\min})<0$. Let $(\rho_0,v_0)\subset \mathcal{D}_{v_1,v_2,w_1,w_2}$ be a fixed function. The constrained Cauchy problem (\ref{existence_Cauchy_problem_general}) has a weak solution $(\rho,v)$ defined for every $t \geq 0$.
\end{theorem}
\begin{proof}
By Corollary \ref{cor_piecewise_constant_approxim_of_a_bv_function}, there exists a sequence $(\rho_{0,\nu},v_{0,\nu})$ of piecewise constant functions with bounded variation, with a finite number of discontinuities and satisfying (\ref{total_var_of_the_pc_initial_datum}) and (\ref{convergence_of_the_pc_initial_datum}). For every $\nu$ we can define a wave-front tracking approximation $(\rho_\nu,v_\nu)$ which is a weak solution of the ARZ system by construction.\\
For every $t>0$ and for every $\nu\in \mathbb{N}$, there exists constant $M>0$ such that
$$TV_t(\rho_\nu)\leq M \; \text{ and } \; TV_t(v_\nu)\leq M.$$
Since the initial datum $(\rho_0,v_0)$ is in  $\mathcal{D}_{v_1,v_2,w_1,w_2}$, there exists a positive constant for which for every $\nu$, we have
$$|(\rho_\nu,v_\nu)(t,x)|\leq N \;\text{ for every } \; (t,x) \in \mathbb{R}^+ \times \mathbb{R}.$$
Moreover for two instants $t>s$, denoting
$$\bar{\lambda}=\sup \lbrace |\lambda_i(\rho,v)|: i=1,2, \; (\rho,v)\in \mathcal{D}_{v_1,v_2,w_1,w_2}\rbrace,$$
we find
\begin{equation*}
\begin{split}
& \int_{-\infty}^{+\infty} |\rho_\nu(t,x)-\rho_\nu(s,x)| \, dx \leq (t-s)\, \bar{\lambda} \sup_{\tau \in [s,t]} TV_\tau(\rho) \; \text{ and}\\
& \int_{-\infty}^{+\infty} |v_\nu(t,x)-v_\nu(s,x)| \, dx \leq (t-s)\, \bar{\lambda} \sup_{\tau \in [s,t]} TV_\tau(v).
\end{split}
\end{equation*}
Indeed, consider for example the $\rho$ component. If we assume that the function $\rho_\nu$ has a single discontinuity between two states $\rho^l$ and $\rho^r$, having propagation speed $\lambda$, we find 
$$\int_\mathbb{R}|\rho_\nu(t,x)-\rho_\nu(s,x)| \, dx = (t-s)\, \lambda |\rho^r-\rho^l|,$$
i.e. the integral is equal to the product between the space covered by the discontinuity within the interval $[s,t]$ and the size of the jump. If there are more discontinuities and we assume that in the interval $[s,t]$ does not happen any interaction, then we can repeat the same idea to each discontinuity and we find
$$\int_\mathbb{R}|\rho_\nu(t,x)-\rho_\nu(s,x)| \, dx \leq (t-s)\, \bar{\lambda} \, TV_t(\rho).$$
Finally, if an interaction happens at time $\tilde{t}\in [s,t]$, we find
\begin{equation*}
\begin{split}
\int_\mathbb{R}|\rho_\nu(t,x)-\rho_\nu(s,x)| \, dx & \leq\int_\mathbb{R}(|\rho_\nu(t,x)-\rho_\nu(\tilde{t},x)|+|\rho_\nu(\tilde{t},x)-\rho_\nu(s,x)|) \, dx \leq\\
& \leq (t-s)\, \bar{\lambda}\, (TV_{\tilde{t}}(\rho)+TV_t(\rho))\leq \\
& \leq 2\,(t-s)\, \bar{\lambda} \sup_{\tau \in [s,t]} TV_\tau(\rho).
\end{split}
\end{equation*}
Similarly if there are more interactions and for the $v$ component.\\
Therefore there exists a constant $L>0$, for which
\begin{equation*}
\begin{split}
& \int_{-\infty}^{+\infty} |\rho_\nu(t,x)-\rho_\nu(s,x)| \, dx \leq L\,|t-s| \; \text{ and}\\
& \int_{-\infty}^{+\infty} |v_\nu(t,x)-v_\nu(s,x)| \, dx \leq L\,|t-s|.
\end{split}
\end{equation*}
Hence we can apply Theorem \ref{theorem_helly_tempo}: for every $t\geq 0$, there exists a subsequence of $(\rho_\nu,v_\nu)$ and a function $(\rho,v)$ such that
$$(\rho_\nu,v_\nu)(t,x) \xrightarrow{\nu \to +\infty} (\rho,v)(t,x) \; \text{ for every } \; x \in \mathbb{R},$$
$$TV_t(\rho)\leq M, \; \; \; TV_t(v) \leq M \; \text{ and } \; |(\rho,v)(t,x)|\leq N \; \text{ for every } \; (t,x)\in \mathbb{R}^+\times \mathbb{R}.$$
Since $f$ is continuous, even $f(\rho,v)$ is bounded. Hence $(\rho,v)$ and $f(\rho,v)$ are in $L^1_\text{loc}$. Applying Lemma \ref{lemma_convergence_weak_solutions}, $(\rho,v)$ is a weak solution to (\ref{existence_system_of_conservation_laws}).\\
Finally, we have to show that for a.e. $x\in \mathbb{R}$, we have
\begin{equation}\label{existence_proof_initial_datum}
(\rho,v)(0,x)=(\rho_0,v_0)(x).
\end{equation}
Since at $x=0$ we use the Riemann solver $\mathcal{RS}^q_2$, for almost every $t>0$ the left and the right traces of the solution at $x=0$ satisfy the constraint:
$$\rho(t,0\pm)\,v(t,0\pm)\leq q, $$
where
$$(\rho,v)(t,0+):=\lim_{x \to 0+} (\rho,v)(t,x) \; \text{ and } \; (\rho,v)(t,0-):=\lim_{x\to 0-} (\rho,v)(t,x)$$
are respectively the right and left traces of the solution at $x=0$.\\
Fix $\varepsilon>0$. There exist two numbers $\tilde{N}$ and $\tilde{t}$ such that for every $\nu >\tilde{N}$ and $t\in [0,\tilde{t})$, we have
\begin{equation*}
\begin{split}
& a(\nu,t):= \parallel (\rho,v)(t,\cdot)-(\rho_\nu,v_\nu)(t,\cdot) \parallel_{L^1_\text{loc}}<\varepsilon,\\
& b(\nu,t):=\parallel (\rho_\nu,v_\nu)(t,\cdot)-(\rho_{0,\nu},v_{0,\nu})(\cdot) \parallel_{L^1_\text{loc}}<\varepsilon \; \text{ and}\\
&c(\nu):= \parallel (\rho_{0,\nu},v_{0,\nu})(\cdot)-(\rho_0,v_0)(\cdot) \parallel_{L^1_\text{loc}}<\varepsilon.
\end{split}
\end{equation*}
Fix $\nu>\tilde{N}$. By the triangular inequality, for every $t\in [0,\tilde{t})$ we obtain
\begin{equation*}
\begin{split}
\parallel (\rho,v)(t,\cdot)-(\rho_0,v_0)(\cdot) \parallel_{L^1_\text{loc}} & \leq a(\nu,t)+b(\nu,t)+c(\nu) <\\
& <3\, \varepsilon.
\end{split}
\end{equation*}
Hence the condition (\ref{existence_proof_initial_datum}) holds.
\end{proof}

\bibliography{mybib}{}
\addcontentsline{toc}{chapter}{Bibliography}
\bibliographystyle{plain}

\end{document}